\documentclass[reqno]{amsart}

\usepackage{macros}

\toggletrue{strat}

\numberwithin{equation}{subsection}

\counterwithin{theorem}{subsection}

\makeatletter
\def\enumfix{%
\if@inlabel
 \noindent \par\nobreak\vskip-\topsep\hrule\@height\z@
\fi}

\let\olditemize\itemize
\def\itemize{\enumfix\olditemize}

\makeatother

\usepackage{calligra}

\begin{document}

\title{Stratified noncommutative geometry}

\author{David Ayala, Aaron Mazel-Gee, and Nick Rozenblyum}

\date{\today}

\begin{abstract}
We introduce a theory of stratifications of noncommutative stacks (i.e.\! presentable stable $\infty$-categories), and we prove a reconstruction theorem that expresses them in terms of their strata and gluing data.
This reconstruction theorem is compatible with symmetric monoidal structures, and with more general operadic structures such as $\EE_n$-monoidal structures.
We also provide a suite of fundamental operations for constructing new stratifications from old ones: restriction, pullback, quotient, pushforward, and refinement. Moreover, we establish a dual form of reconstruction; this is closely related to Verdier duality and reflection functors, and gives a categorification of M\"obius inversion.

Our main application is to equivariant stable homotopy theory: for any compact Lie group $G$, we give a symmetric monoidal stratification of genuine $G$-spectra. In the case that $G$ is finite, this expresses genuine $G$-spectra in terms of their geometric fixedpoints (as homotopy-equivariant spectra) and gluing data therebetween (which are given by proper Tate constructions).

We also prove an adelic reconstruction theorem;
this applies not just to ordinary schemes but in the more general context of tensor-triangular geometry, where we obtain
a symmetric monoidal stratification over the Balmer spectrum.  We discuss the particular example of chromatic homotopy theory.
\end{abstract}


\maketitle

\setcounter{tocdepth}{2}
\tableofcontents
\setcounter{tocdepth}{2}

\setcounter{section}{-1}

\section{Introduction}
\label{section.intro}

\startcontents[sections]


\subsection{Overview}
\label{subsection.intro.overview}

In this work, we develop a theory of \bit{stratified noncommutative stacks}. We take the term \textit{noncommutative stack} to mean a presentable stable $\infty$-category, as explained in \Cref{rmk.history.of.terminology}.\footnote{Our results apply equally well to pretriangulated dg-categories admitting all direct sums (or more precisely, to their underlying $\kk$-linear presentable stable $\infty$-categories).} We suggestively refer to the objects of a noncommutative stack as its \textit{quasicoherent sheaves}.\footnote{In particular, an ordinary scheme or stack $X$ has an underlying noncommutative stack $\QC(X)$, its presentable stable $\infty$-category of quasicoherent sheaves.} Our novel contribution is a theory of \textit{stratifications}.\footnote{This builds on work of Glasman and others, as described in \Cref{subsection.relation.with.lit}.} In short, a stratification of a noncommutative stack $\cX$ is a filtration by noncommutative substacks $\{ \cZ_p \}_{p \in \pos}$ indexed by a poset $\pos$ that satisfies certain natural geometrically-inspired conditions; for each $p \in \pos$, the \bit{$p\th$ stratum} of the stratification is the associated-graded $\cX_p := \cZ_p / \cZ_{^< p}$.\footnote{As we explain in \S\S\ref{subsection.intro.strat.sch}-\ref{subsection.intro.stratn.nc.stacks}, a stratification of an ordinary scheme $X$ determines a stratification of $\QC(X)$ via set-theoretic support on closures of strata, whose strata are closely related to those of $X$. (On the other hand, in general not all stratifications of $\QC(X)$ arise from stratifications of $X$.)}

The primary purpose of stratifications is that they provide \bit{reconstruction theorems},
in a way that can be summarized informally as follows.\footnote{Our terminology for the two parts of \Cref{slogan.cosm} is inspired by the ``macrocosm/microcosm principle'', which asserts e.g.\! that it is precisely a monoidal structure on a category that enables one to speak of algebra objects in that category.  In the present situation, macrocosm reconstruction for the noncommutative stack $\cX$ enables microcosm reconstruction for each quasicoherent sheaf $\cF \in \cX$.  This is a familiar phenomenon from classical sheaf theory: categories of globally-defined sheaves can be reconstructed from categories of locally-defined sheaves, and so globally-defined sheaves can be reconstructed from locally-defined sheaves.}
\begin{slogan}
\label{slogan.cosm}
Let $\cX$ be a noncommutative stack equipped with a stratification over a poset $\pos$.
\begin{enumerate}
\item {\bf macrocosm:} The noncommutative stack 
$\cX$ can be reconstructed from the strata
\[
\{ \cX_p \subseteq \cX \}_{p \in \pos}
\]
along with gluing data between them.
\item {\bf microcosm:} Each quasicoherent sheaf 
$\cF \in \cX$ can be reconstructed from its geometric localizations
\[
\{ \Phi_p( \cF ) \in \cX_p \}_{p \in \pos}
\]
along with gluing data between them.
\end{enumerate}
\end{slogan}


\noindent The simplest interesting example of a stratification is when $\pos = \{ 0 < 1 \}$: in this case we recover the data of a \textit{recollement} (which we review for the reader's convenience in \Cref{subsection.intro.closed.open.decomps.and.recollements}).\footnote{The French word \textit{recollement} translates to ``regluing''.}

Our main application is a symmetric monoidal reconstruction theorem for genuine $G$-spectra, which has particularly simple strata.\footnote{At the microcosm level, this presents a genuine $G$-spectrum in terms of its geometric fixedpoints (as opposed to its presentation in terms of its categorical fixedpoints as a spectral Mackey functor \cite{GM-gen,Bar-Mack}).} The eager reader may turn directly to \Cref{subsection.examples.of.SpgG} to see specific examples of this reconstruction theorem in action:
\begin{itemize}
\item genuine $G$-spectra where $G$ is one of the cyclic groups $\Cyclic_p$, $\Cyclic_{p^2}$, and $\Cyclic_{pq}$ (for distinct primes $p$ and $q$) or the symmetric group $\Symm_3$, and
\item {proper}-genuine $\TT$-spectra, where $\TT$ denotes the circle group.
\end{itemize}
In \cite{AMR-cyclo}, we build on this last example to provide a symmetric monoidal reconstruction theorem for cyclotomic spectra.  This improves on the foundational work \cite{NS} of Nikolaus--Scholze, in that it applies to all cyclotomic spectra (instead of only eventually-connective ones) and specifies its canonical symmetric monoidal structure.  In particular, it provides a universal mapping-in property at the level of objects, which we use to obtain the cyclotomic trace map
\[
\K
\longra
\TC
\]
from algebraic K-theory to topological cyclic homology in \cite{AMR-trace}. In a different direction, in \cite{AMR-mackey} we apply our reconstruction theorem to compute the $\Cyclic_{p^n}$-equivariant cohomology of a point for any odd prime $p$.




We also set up an $\cO$-monoidal enhancement of our theory, where $\cO$ denotes any $\infty$-operad satisfying certain mild conditions (e.g.\! $\EE_n$ for $1 \leq n \leq \infty$); this accounts for the symmetric monoidality of our reconstruction theorem for genuine $G$-spectra.  In this vein, we make contact with the world of tensor-triangular geometry, by showing that under mild hypotheses a presentably symmetric monoidal stable $\infty$-category admits a canonical \textit{adelic stratification}, which is a symmetric monoidal stratification over the specialization poset of its Balmer spectrum.  The adelic stratification of $\Mod_\ZZ$ recovers the classical \textit{arithmetic fracture square}, which is the natural pullback square for any $M \in \Mod_\ZZ$ that is indicated in \Cref{figure.first.instance.of.arithmetic.fracture.square}.
\begin{figure}[h]
\begin{equation}
\label{intro.arithmetic.fracture.square}
\begin{tikzcd}[row sep=1.5cm]
M
\arrow{r}
\arrow{d}
&
\QQ \otimes_\ZZ M
\arrow{d}
\\
{\displaystyle \prod_{p \textup{ prime}} M^\wedge_p}
\arrow{r}
&
\QQ \otimes_\ZZ \left( 
{\displaystyle \prod_{p \textup{ prime}} M^\wedge_p}
\right)
\end{tikzcd}
\end{equation}
\caption{The arithmetic fracture square is a natural pullback square that reconstructs any $M \in \Mod_\ZZ$ from its rationalization, its $p$-completions, and gluing data between them.}
\label{figure.first.instance.of.arithmetic.fracture.square}
\end{figure}
More generally, for any scheme $X$ satisfying mild finiteness hypotheses, the adelic stratification of $\QC(X)$ leads to an \textit{adelic reconstruction theorem}, which bears a close relationship to existing such formalisms of Beilinson and others. Moreover, the \textit{chromatic stratification} of the $\infty$-category $\Spectra$ of spectra organizes the fundamental objects of chromatic homotopy theory and recovers integral (i.e.\! not $p$-local) and higher-dimensional variants of the chromatic fracture square, as described in \Cref{ex.chromatic.stratn.of.spectra}.\footnote{This is closely related to its adelic stratification, which is described in \Cref{ex.adelic.stratn.of.spectra}.}

In a different direction, we introduce the theory of \bit{reflection}. This affords a dual form of reconstruction; applied to $\Mod_\ZZ$, this yields the \textit{reflected arithmetic fracture square}, which is the natural pushout square for any $M \in \Mod_\ZZ$ that is indicated in \Cref{figure.reflected.arithmetic.fracture.square}.\footnote{This particular example can be seen as a consequence of Greenlees--May duality (or even of local duality for $\Spec(\ZZ)$).}
\begin{figure}[h]
\[ \begin{tikzcd}[row sep=1.5cm]
\ulhom_{\Mod_\ZZ}
\left(
\QQ,
\left(
{\displaystyle
\bigoplus_{p \textup{ prime}}
M^\tors_p
}
\right)
\right)
\arrow{r}
\arrow{d}
&
\ulhom_{\Mod_\ZZ}(\QQ,M)
\arrow{d}
\\
{\displaystyle
\bigoplus_{p \textup{ prime}}
M_p^\tors
}
\arrow{r}
&
M
\end{tikzcd} \]
\caption{The reflected arithmetic fracture square is a natural pushout square that reconstructs any $M \in \Mod_\ZZ$ from its corationalization, its $p$-torsionifications, and gluing data between them.\footnotemark}
\label{figure.reflected.arithmetic.fracture.square}
\end{figure}
\footnotetext{We write $M_p^\tors := \fib ( M \ra M \otimes_\ZZ \ZZ[p^{-1}])$ for the $p$-torsionification of $M$, in analogy with the notation $M^\wedge_p$ for its $p$-completion.}
In particular, we establish a precise relationship between the gluing data and the reflected gluing data. In the case of $\Mod_\ZZ$, this specializes to the remarkable equivalence
\[
\QQ \otimes_\ZZ
\left(
\prod_p M^\wedge_p
\right)
\simeq
\Sigma
\ulhom_{\Mod_\ZZ}
\left(
\QQ
,
\left( \bigoplus_p M^\tors_p \right)
\right)
~;
\]
taking $M = \ZZ$, this gives an equivalence $\AA_\fin \simeq \ulhom_{\Mod_\ZZ}(\QQ , \QQ/\ZZ)$, where $\AA_\fin$ denotes the ring of finite adeles. Specialized to the poset $\pos = [1]$, reflection recovers the theory of reflection functors (which explains our choice of terminology). More generally, it gives a categorification of the M\"obius inversion formula and is closely related to Verdier duality.

We give a detailed overview of our work in \Cref{subsection.intro.detailed.overview}, which begins with some recollections and motivation. Our main theorems (which are stated more precisely therein) may be summarized as follows.
\begin{itemize}

\item \Cref{intro.thm.cosms} is our \bit{reconstruction theorem} for stratified noncommutative stacks, a precise articulation of \Cref{slogan.cosm}.  In fact, it provides a universal mapping-in property -- that is, a limit-type description -- both at the macrocosm level (for noncommutative stacks) and at the microcosm level (for their quasicoherent sheaves).

\item \Cref{intro.thm.fund.opns} provides a suite of \bit{fundamental operations} for constructing new stratifications from old ones: restriction, pullback, quotient, pushforward, and refinement. 

\item \Cref{intro.thm.O.mon.reconstrn} is our \bit{$\cO$-monoidal reconstruction theorem}, an enhancement of \Cref{intro.thm.cosms}.  At the macrocosm level, this provides universal mapping-in properties for presentably $\cO$-monoidal stable $\infty$-categories as such.

\item \Cref{intro.thm.balmer} establishes the symmetric monoidal \bit{adelic stratification} of a presentably symmetric monoidal stable $\infty$-category satisfying mild finiteness hypotheses over (the specialization poset of) its Balmer spectrum.

\item \Cref{intro.thm.gen.G.spt} establishes the symmetric monoidal \bit{geometric stratification} of the presentably symmetric monoidal stable $\infty$-category $\Spectra^{\gen G}$ of genuine $G$-spectra, where $G$ is any compact Lie group.  This has the following features:
\begin{itemize}
\item
its strata are the presentably symmetric monoidal stable $\infty$-categories
\[
\Spectra^{\htpy \Weyl(H)} := \Fun(\sB \Weyl(H) , \Spectra)
\]
of homotopy $\Weyl(H)$-spectra, where $H$ is a closed subgroup of $G$ and $\Weyl(H)$ denotes its Weyl group;
\item its geometric localization functors are the geometric fixedpoints functors
\[
\Spectra^{\gen G}
\xra{\Phi^H}
\Spectra^{\htpy \Weyl(H)}
~;
\]
and
\item its gluing functors are given by a version of the Tate construction.
\end{itemize}
As explained in \Cref{rmk.SpgG.nearly.commutative}, this provides a sense in which genuine $G$-spectra are the quasicoherent sheaves on a ``nearly commutative'' stack.

\item \Cref{intro.thm.reflection} establishes the theory of \bit{reflection}, which affords a dual form of reconstruction for stratified noncommutative stacks.

\end{itemize}

\noindent In \Cref{subsection.intro.detailed.overview} we also discuss a number of additional applications of our work: constructible sheaves; categorified M\"obius inversion; naive $G$-spectra; t-structures; and additive and localizing invariants.

\begin{remark}
\label{rmk.history.of.terminology}
The philosophy of noncommutative algebraic geometry can be traced back to Gabriel's thesis \cite{Gabriel-thesis}, in which he proved that one can reconstruct a scheme from its abelian category of quasicoherent sheaves.  Following this, Manin proposed that arbitrary abelian categories might therefore be thought of as categories of quasicoherent sheaves on ``noncommutative schemes'' \cite[\S 12.6]{Manin-qgrpsandncgeom}. 
This proposal has since been developed further by many authors, notably Rosenberg \cite{Rosenberg-SpecAbCat,Rosenberg-ncscheme}, as well as Kontsevich--Rosenberg \cite{KontRos-nc} and Kontsevich--Soibelman \cite{KontSoi-ncgeom} from a more derived perspective. Our usage of the term ``noncommutative stack'' to mean a presentable stable $\infty$-category is inspired by this trajectory.
\end{remark}

\subsection{Relations with existing literature}
\label{subsection.relation.with.lit}

A number of distinct narrative threads converge in the present work, some of which we discuss here.  However, the literature is vast, and we make no attempt to be comprehensive.

\subsubsection{Recollements and semiorthogonal decompositions}

Stratifications admit a rich history: 
they generalize recollements (which are stratifications over $[1]$) and more generally semiorthogonal decompositions (which are stratifications over $[n]$).\footnote{In the present discussion we do not distinguish between the small and presentable settings.} Recollements were originally introduced by Beilinson--Bernstein--Deligne in their study of perverse sheaves \cite{BBD-perv}.  A fruitful source of semiorthogonal decompositions is exceptional collections; this technique first appeared in Beilinson's calculation of the derived category of $\PP^n$ \cite{Beil-linalg}, and was pursued more systematically by Bondal--Kapranov \cite{BondKap-reconstrn}.  Semiorthogonal decompositions continue to be a highly active area of research, especially in connection with algebraic geometry; see e.g.\! \cite{Kuz-SODinAG} for more in this direction.

\subsubsection{Adelic reconstruction}

As explained in \Cref{subsection.intro.adelic}, given a scheme $X$, our work provides a decomposition of $\QC(X)$ in adelic terms; this generalizes the arithmetic fracture square \Cref{intro.arithmetic.fracture.square}, which corresponds to the case that $X=\Spec(\ZZ)$.  This is quite similar to prior adelic reconstruction results in the literature, e.g.\! \cite{Parshin-ad,Beil-ad,Huber-ad,Groechenig-ad,HPV-gluing}.  However, there is a subtle difference, even in the case of $X = \Spec(\ZZ)$: we recover the arithmetic fracture square \Cref{intro.arithmetic.fracture.square} for all $\ZZ$-modules $M$, despite the fact that two of its terms don't commute with filtered colimits in the variable $M$. In the specific context of tensor-triangulated geometry, \cite{BalchGreen-adelic} provides a symmetric monoidal macrocosm-type reconstruction theorem.

\subsubsection{Chromatic homotopy theory}

Reconstruction has long been a guiding principle in homotopy theory, going back to Sullivan's influential lecture notes \cite{Sullivan-MIT}.  The chromatic approach to stable homotopy theory grew out of Ravenel's work \cite{Rav-loc} and the resulting nilpotence and periodicity theorems of Devinatz--Hopkins--Smith \cite{DHS-nilp,HS-nilptwo}, along with the extensive axiomatic treatment of Hovey--Palmieri--Strickland \cite{HPS} -- all pointing to the chromatic fracture squares as essential from the perspective of reconstruction.  More recently, higher-dimensional chromatic fracture cubes for $p$-local spectra -- and indeed, corresponding macrocosm reconstruction theorems -- appear e.g.\! in \cite[Examples 3.14 and 3.31]{Saul-strat} and \cite{ACB-chromfrac}.

\subsubsection{Reconstruction for genuine $G$-spectra} 

The idea that genuine $G$-spectra can be expressed in terms of their geometric fixedpoints stems from the work of Greenlees and May; see in particular \cite{Greenlees-thesis,GreenMay-Tate}.  There is also much work on similar expressions of rational $G$-spectra (which are simpler because the relevant Tate constructions vanish rationally), notably the reconstruction results of Greenlees--Shipley \cite{GreenShip}.  More recent works in this direction include \cite{MNN,Saul-strat}; see also \cite[Remark II.4.8]{NS}.

\subsubsection{Glasman's theory of stratifications}
\label{subsubsection.compare.with.saul}

Theorems \ref{intro.thm.cosms} \and \ref{intro.thm.gen.G.spt} are directly inspired by Glasman's paper \cite{Saul-strat}, as we now explain.


In \cite[Definition 3.5]{Saul-strat}, Glasman introduces a notion of a stratification of a stable $\infty$-category (not assumed to be presentable). 
His definition is phrased in terms of the strata (in the sense of \Cref{defn.Cth.stratum.and.geometric.localizn}) for all convex subsets $\sC \subseteq \pos$ of the stratifying poset.  He proves a reconstruction result for his stratifications \cite[Theorem 3.21]{Saul-strat}, and for any finite group $G$ he provides a stratification of the $\infty$-category $\Spectra^{\gen G}$ genuine $G$-spectra over the 
poset $\pos_G$ of conjugacy classes of subgroups of $G$ \cite[Proposition 3.18]{Saul-strat}.

By contrast, we work primarily in the setting of presentable stable $\infty$-categories. This enables us to give a relatively simple definition of a stratification, in terms of closed subcategories indexed by the poset $\pos$ itself (rather than by its poset of convex subsets): we recover the strata as presentable quotients. (These notions are summarized in \Cref{subsection.intro.stratn.nc.stacks}.) On the other hand, using this we also provide a theory of stratifications of stable $\infty$-categories (see \Cref{subsection.stable.stratns}).\footnote{As a matter of convenience, we restrict our attention to idempotent-complete stable $\infty$-categories.} This effectively recovers Glasman's theory of stratifications, and offers a substantial refinement of his reconstruction theorem as well (which is a version of our microcosm reconstruction).

\subsection{Outline}
\label{subsection.linear.outline}

This work is organized as follows.
\begin{itemize}

\item[\Cref{subsection.intro.detailed.overview}:] We give a detailed overview of our work, and explain a number of fundamental examples and applications.

\item[\Cref{section.strat}:] We introduce closed subcategories and stratifications.  We prove that the macrocosm reconstruction theorem (\Cref{intro.thm.cosms}\Cref{intro.main.thm.macrocosm}) follows from the metacosm reconstruction theorem (\Cref{intro.thm.cosms}\Cref{intro.main.thm.metacosm}).

\item[\Cref{section.fund.opns}:] We establish our fundamental operations on stratifications (\Cref{intro.thm.fund.opns}). We accomplish this by studying the phenomenon of \textit{alignment}.

\item[\Cref{section.O.mon.reconstrn.thm}:] We introduce $\cO$-monoidal stratifications and prove the $\cO$-monoidal reconstruction theorem (\Cref{intro.thm.O.mon.reconstrn}).  We also establish the adelic stratification (\Cref{intro.thm.balmer}), which we unpack in the setting of chromatic homotopy theory.

\item[\Cref{section.genuine}:] We review the $\infty$-category of genuine $G$-spectra and establish its geometric stratification (\Cref{intro.thm.gen.G.spt}).  We record a few facts about its gluing functors, which are essentially given by proper Tate constructions.  Using these facts, we unpack a number of examples of reconstruction for genuine $G$-spectra. We also give a formula for categorical fixedpoints in terms of the geometric stratification, and we explain how this interacts with restriction and transfer.

\item[\Cref{section.reconstrn}:] We prove the metacosm reconstruction theorem (\Cref{intro.thm.cosms}\Cref{intro.main.thm.metacosm}).

\item[\Cref{section.variations}:] We prove a number of variants of the metacosm reconstruction theorem, notably our dual form of reconstruction and the theory of reflection (\Cref{intro.thm.reflection}).

\item[\Cref{section.lax.actions.and.limits}:] We review the theory of lax modules and lax limits, and record a number of results that we need.  This material is used systematically throughout the main body of the work, but this usage is confined to proofs (rather than assertions) to the greatest extent possible.

\item[\Cref{section.inftytwocats}:] We establish the necessary background regarding $(\infty,2)$-categories, particularly the theory of lax functors and natural transformations as well as the theory of adjunctions. This material primarily supports \Cref{section.lax.actions.and.limits}.

\end{itemize}

\subsection{Notation and conventions}
\label{subsection.notation.and.conventions}

\begin{enumerate}

\item \catconventions \inftytwoconventions  

\item \functorconventions

\item \circconventions

\item \kanextnconventions

\item \spacescatsspectraconventions

\item \Prconventions \Promegaconventions


\item \fibrationconventions \dualfibnsconvention \LFibRFibconvention \projectionfromproductconvention

\item \efibconventions



\end{enumerate}

\subsection{Acknowledgments}
\label{subsection.acknowledgments}

This work builds upon a broad array of mathematics developed over the past few decades; so, we owe a substantial intellectual debt to the community at large, all of the individual contributors among which it would be impossible to name here.  Nevertheless, we would like to highlight the works of Saul Glasman \cite{Saul-strat} and of Akhil Mathew, Niko Naumann, and Justin Noel \cite{MNN}, which were particularly influential to us. We also extend our gratitude to an anonymous referee for a helpful report.

\acksupport \ Additionally, all three authors gratefully acknowledge the superb working conditions 
provided by the Mathematical Sciences Research Institute (which is supported by NSF award 1440140), where DA and AMG were in residence and NR was a visitor during the Spring 2020 semester.

\stopcontents[sections]

\section{Detailed overview and fundamental examples}
\label{subsection.intro.detailed.overview}

In this section, we give an informal overview of our work. In addition to giving somewhat more precise statements of our main theorems (which we only informally described in \Cref{subsection.intro.overview}), we place our work within a broader mathematical narrative and collect key examples and applications.

In contrast with the present section, the main body of the work (i.e.\! all the material beyond \Cref{subsection.intro.detailed.overview}) is almost entirely devoted to proofs of the main theorems.\footnote{However, our specific examples of reconstruction for genuine $G$-spectra are collected in \Cref{subsection.examples.of.SpgG}, and we defer a discussion of the chromatic and adelic stratifications of spectra to Examples \ref{ex.chromatic.stratn.of.spectra} \and \ref{ex.adelic.stratn.of.spectra}.} (So for instance, we will not revisit any discussion of sheaves.)

\begin{local}
Throughout this section, we fix a scheme $X$,\footnote{More precisely, in order to simplify our exposition, we tacitly assume that our scheme $X$ is finite-dimensional and noetherian. The utility of these assumptions is explained in Footnotes \ref{fn.why.qcqs}, \ref{fn.why.noetherian}, \and \ref{fn.why.fdim}.}
a noncommutative stack $\cX$ (i.e.\! a presentable stable $\infty$-category), and a poset $\pos$.
\end{local}

This section is organized as follows.
\begin{itemize}

\item[\Cref{subsection.intro.closed.open.decomps.and.recollements}:] We recall the notion of a recollement of $\cX$ and the fact that a closed-open decomposition of $X$ determines a recollement of $\QC(X)$.

\item[\Cref{subsection.intro.strat.sch}:] We generalize closed-open decompositions of $X$ to stratifications of $X$.

\item[\Cref{subsection.intro.stratn.nc.stacks}:] We define stratifications of $\cX$ and state our main reconstruction theorem (\Cref{intro.thm.cosms}).  We also explain how a stratification of $X$ determines a stratification of $\QC(X)$; in retrospect, \Cref{subsection.intro.closed.open.decomps.and.recollements} describes the special case of this phenomenon when $\pos = [1]$.

\item[\Cref{subsection.intro.fund.operations}:] To address certain subtleties arising in \Cref{intro.thm.cosms}, we indicate our fundamental operations on stratifications (\Cref{intro.thm.fund.opns}).

\item[\Cref{subsection.intro.O.mon.stratns}:] We describe our theory of $\cO$-monoidal stratifications and state our $\cO$-monoidal reconstruction theorem (\Cref{intro.thm.O.mon.reconstrn}).

\item[\Cref{subsection.intro.adelic}:] We begin by describing the adelic stratification of $\QC(X)$.  We unpack in detail the example of $X = \Spec(\ZZ)$, which nicely illustrates essentially all of the material surveyed up to this point, and which ultimately recovers the arithmetic fracture square \Cref{intro.arithmetic.fracture.square}.  We conclude by generalizing adelic stratifications to the setting of tensor-triangular geometry (\Cref{intro.thm.balmer}).

\item[\Cref{subsection.intro.eq.spt}:] We describe the geometric stratification of genuine $G$-spectra (\Cref{intro.thm.gen.G.spt}).

\item[\Cref{subsection.cbl}:] Given a $\pos$-stratified topological space, we obtain stratifications over $\pos^\op$ of its $\infty$-categories of sheaves, constructible sheaves, and $\pos$-constructible sheaves.

\item[\Cref{subsection.naive.G.spectra.stratn}:] As a special case of a general construction, we obtain a stratification of naive $G$-spectra, which is closely related to the geometric stratification of genuine $G$-spectra.

\item[\Cref{subsection.verdier}:] We explain the theory of reflection (\Cref{intro.thm.reflection}) and indicate a number of examples, notably its close relationship with Verdier duality.

\item[\Cref{subsection.t.structures}:] We explain how to use stratifications to build t-structures.

\item[\Cref{subsection.add.loc.invts}:] We explain the relationship between stratifications and additive and localizing invariants (such as (resp.\! connective and nonconnective) algebraic K-theory).

\end{itemize}




\subsection{Closed-open decompositions and recollements}
\label{subsection.intro.closed.open.decomps.and.recollements}

We begin by recalling the theory of recollements (in the context of presentable stable $\infty$-categories).

\begin{definition}
\label{defn.recollement.in.intro}
A \bit{recollement} of the noncommutative stack (i.e.\! presentable stable $\infty$-category) $\cX$ is a diagram
\begin{equation}
\label{recollement.in.intro}
\begin{tikzcd}[column sep=1.5cm]
\cZ
\arrow[hook, bend left=45]{r}[description]{i_L}
\arrow[leftarrow]{r}[transform canvas={yshift=0.1cm}]{\bot}[swap,transform canvas={yshift=-0.1cm}]{\bot}[description]{\yo}
\arrow[bend right=45, hook]{r}[description]{i_R}
&
\cX
\arrow[bend left=45]{r}[description]{p_L}
\arrow[hookleftarrow]{r}[transform canvas={yshift=0.1cm}]{\bot}[swap,transform canvas={yshift=-0.1cm}]{\bot}[description]{\nu}
\arrow[bend right=45]{r}[description]{p_R}
&
\cU
\end{tikzcd}
\end{equation}
of adjunctions among presentable stable $\infty$-categories such that there are equalities
\begin{equation}
\label{equalities.in.defn.of.recollement}
\im(i_L) = \ker(p_L)
~,
\qquad
\im(\nu) = \ker(\yo)
~,
\qquad
\text{and}
\qquad
\im(i_R) = \ker(p_R)
\end{equation}
among full subcategories of $\cX$.\footnote{We have chosen the notation ``$\yo$'' because this is the restricted Yoneda functor (with respect to the inclusion $i_L$), and the notation ``$\nu$'' because this is the inclusion of the full subcategory of objects whose restricted Yoneda functors are null.}
\end{definition}

Given a recollement \Cref{recollement.in.intro} of $\cX$, it is not hard to check that for each $\cF \in \cX$ we obtain a canonical pullback square
\begin{equation}
\label{microcosm.pullback.square.in.intro}
\begin{tikzcd}[row sep=1.5cm, column sep=2cm]
\cF
\arrow{r}{\eta_{p_L \adj \nu}(\cF)}
\arrow{d}[swap]{\eta_{y \adj i_R}(\cF)}
&
\nu p_L \cF
\arrow{d}{\nu p_L ( \eta_{y \adj i_R}(\cF))}
\\
i_R y \cF
\arrow{r}[swap]{\eta_{p_L \adj \nu}(i_R y \cF)}
&
\nu p_L i_R y \cF
\end{tikzcd}~.\footnote{Indeed, taking fibers of the vertical morphisms reduces us to the case where $\cF \in \ker(\yo) = \im(\nu) \subseteq \cX$, in which case the claim is immediate.}
\end{equation}
Hence, the object $\cF \in \cX$ is recorded by the lower right cospan.  However, to record the object $\cF \in \cX$ we may actually record less data than this cospan: its lower morphism is the unit of the adjunction $p_L \adj \nu$, and so is canonically determined by its source $i_R y \cF \in \cX$.  Noting further that the functors $i_R$ and $\nu$ are fully faithful, we find that the object $\cF \in \cX$ can be reconstructed from the data of the object $y \cF \in \cZ$, the object $p_L \cF \in \cU$, and the morphism
\begin{equation}
\label{gluing.morphism.in.cU.for.cF}
p_L \cF \xra{p_L(\eta_{y \adj i_R}(\cF))} p_L i_R y \cF
\end{equation}
in $\cU$.  This observation forms the basis of an equivalence
\begin{equation}
\label{macrocosm.equivalence.for.recollement.in.intro}
\cX
\xlongra{\sim}
\lim^\rlax \left( \cZ \xra{p_L i_R} \cU \right)
:=
\lim \left( \begin{tikzcd}
&
\Fun([1],\cU)
\arrow{d}{t}
\\
\cZ
\arrow{r}[swap]{p_L i_R}
&
\cU
\end{tikzcd} \right)
~,\footnote{Right-lax limits will be explained further in Remarks \ref{rmk.lax.limits.have.structure.maps} \and \ref{rmk.intro.strictification.of.rlax.lim}.}
\end{equation}
which is given by the formula
\[
\cF
\longmapsto
\left( \begin{tikzcd}
&
\Cref{gluing.morphism.in.cU.for.cF}
\arrow[maps to]{d}
\\
y\cF
\arrow[maps to]{r}
&
p_L i_R y \cF
\end{tikzcd} \right)
\]
and whose inverse reconstructs each object $\cF \in \cX$ as the pullback \Cref{microcosm.pullback.square.in.intro}.

This situation is a prototypical instance of \Cref{slogan.cosm}, as well as a special case of \Cref{intro.thm.cosms} below: the equivalence \Cref{macrocosm.equivalence.for.recollement.in.intro} is a macrocosm reconstruction of the noncommutative stack $\cX$, and the pullback square \Cref{microcosm.pullback.square.in.intro} determines a microcosm reconstruction of the quasicoherent sheaf $\cF \in \cX$.

We have the following fundamental source of recollements.

\begin{example}
\label{ex.of.closed.open.decomp.giving.recollement.in.intro}
Suppose we are given a closed-open decomposition of our scheme $X$ as in the diagram
\begin{equation}
\label{closed.open.decomp.of.scheme.in.intro}
\begin{tikzcd}[column sep=0.75cm, ampersand replacement=\&]
Z
\arrow[hook]{rr}[swap]{\sf closed}{i}
\arrow[hook]{rd}
\&[-0.25cm]
\&[-0.25cm]
X
\arrow[hookleftarrow]{rr}{j}[swap]{\sf open}
\&
\&
U
\\
\&
X^\wedge_Z
\arrow[hook]{ru}[swap, sloped]{\ihat}
\end{tikzcd}~,
\end{equation}
in which we have additionally included the formal completion $X^\wedge_Z$ of $X$ along $Z$.  Then, we have a recollement
\begin{equation}
\label{specific.recollement.for.closed.open.decomp.of.scheme}
\begin{tikzcd}[column sep=1.5cm, row sep=1.5cm, ampersand replacement=\&]
\QC_Z(X)
\arrow[hook, transform canvas={yshift=0.9ex}]{r}
\arrow[leftarrow, transform canvas={yshift=-0.9ex}]{r}[yshift=-0.2ex]{\bot}
\arrow[leftrightarrow]{d}[sloped, anchor=north]{\sim}
\&
\QC(X)
\arrow[bend left=30]{r}[description]{j^*}
\arrow[hookleftarrow]{r}[transform canvas={yshift=0.15cm}]{\bot}[swap,transform canvas={yshift=-0.15cm}]{\bot}[description]{j_*}
\arrow[bend right=30]{r}
\&
\QC(U)
\\
\QC(X^\wedge_Z)
\arrow[leftarrow, transform canvas={xshift=-0.6ex, yshift=0.7ex}]{ru}[description, sloped, pos=0.35]{\ihat^*}
\arrow[hook, transform canvas={xshift=0.6ex, yshift=-0.7ex}]{ur}[sloped, yshift=-0.2ex]{\bot}[description, sloped, pos=0.70]{\ihat_*}
\end{tikzcd}~,
\end{equation}
in which
\begin{itemize}
\item $\QC_Z(X) := \ker(j^*) \subseteq \QC(X)$ denotes the full subcategory of those quasicoherent sheaves on $X$ that are set-theoretically supported on $Z$,
\item the left vertical equivalence is that between $\ms{I}_Z$-torsion and $\ms{I}_Z$-complete quasicoherent sheaves of $\ms{O}_X$-modules,\footnote{This equivalence is recorded e.g.\! as \cite[Proposition 7.1.3]{GR-dgindsch}; see also \cite{GreenMay-dual,DwyGreen-comptors}.  (Note that it is not generally t-exact, and so is an inherently derived phenomenon.)} and
\item the triangle commutes.\footnote{\label{fn.why.qcqs}The existence of the recollement \Cref{specific.recollement.for.closed.open.decomp.of.scheme} is guaranteed by the assumption that $X$ is qcqs: namely, this guarantees that the functor $j_*$ preserves colimits.}
\end{itemize}
\end{example}

\begin{warning}
In the situation of \Cref{ex.of.closed.open.decomp.giving.recollement.in.intro}, the full subcategory $\QC_Z(X) \subseteq \QC(X)$ is generated under colimits by the image of the pushforward functor
\[ \QC(Z) \xlongra{i_*} \QC(X)~, \]
but this latter functor is \textit{not} generally fully faithful.\footnote{On the other hand, it is not hard to recover the closed subset $Z \subseteq X$ from the data of the full subcategory $\QC_Z(X) \subseteq \QC(X)$.}
\end{warning}

\subsection{Stratified schemes}
\label{subsection.intro.strat.sch}

We now generalize the notion of a closed-open decomposition of $X$.  Evidently, the closed-open decomposition \Cref{closed.open.decomp.of.scheme.in.intro} of $X$ is entirely determined by the closed subset $Z \subseteq X$.  Let us write $\Cls_X$ for the poset of closed subsets of $X$ ordered by inclusion.

\begin{definition}
\label{defn.intro.comm.stratn}
A \bit{stratification} of the scheme $X$ over the poset $\pos$ is a functor
\begin{equation}
\label{intro.fixed.stratn.of.scheme}
\begin{tikzcd}[row sep=0cm]
\pos
\arrow{r}{Z_\bullet}
&
\Cls_X
\\
\rotatebox{90}{$\in$}
&
\rotatebox{90}{$\in$}
\\
p
\arrow[maps to]{r}
&
Z_p
\end{tikzcd}
\end{equation}
satisfying the following conditions:
\begin{enumerate}

\item[] {\bf generation:} $X = \bigcup_{p \in \pos} Z_p$;

\item[] {\bf stratification:} for any $p,q \in \pos$, we have
\[ Z_p \cap Z_q = \bigcup_{r \leq p \textup{ and }r \leq q} Z_r ~. \]

\end{enumerate}
\end{definition}

\begin{example}
\label{ex.intro.closed.subscheme.defines.stratn.over.brax.one}
Suppose that $\pos = [1] = \{ 0 \ra 1 \}$.  Then, a stratification of $X$ over $\pos$ is equivalent data to that of a closed subset $Z := Z_0 \subseteq X$.
\end{example}

\begin{example}
\label{ex.stratn.of.scheme.over.totally.ordered.poset}
Generalizing \Cref{ex.intro.closed.subscheme.defines.stratn.over.brax.one}, suppose that the poset $\pos$ is in fact a totally ordered set.  Then, any functor \Cref{intro.fixed.stratn.of.scheme} satisfies the stratification condition.  If $\pos$ contains a maximal element, then the functor \Cref{intro.fixed.stratn.of.scheme} satisfies the generation condition (and hence defines a stratification of $X$ over $\pos$) if and only if the maximal element $X \in \Cls_X$ lies in its image.
\end{example}

\begin{example}
\label{ex.stratn.of.scheme.over.set}
Let $S$ be a set, and suppose that $X \ra S$ is a morphism to $S$ considered as a discrete scheme (i.e.\! an $S$-indexed coproduct of copies of $\Spec(\kk)$).  Then, taking preimages defines a stratification
\[
S
\longra
\Cls_X
\]
(where $S$ is considered as a discrete poset).
\end{example}

\begin{example}
\label{ex.intro.stratn.of.affine.plane}
Suppose that $X = \AA^2 = \Spec(\kk[x,y])$ is the affine plane.  Choose any $a,b \in \kk^\times$, and consider the three full subposets
\[
\begin{tikzcd}
&
V(x)
\\
V(y)
\arrow[hook]{r}
&
\AA^2
\arrow[hookleftarrow]{u}
\end{tikzcd}
~,
\qquad
\begin{tikzcd}
V(x,y)
\arrow[hook]{r}
&
V(x)
\\
V(y)
\arrow[hook]{r}
\arrow[hookleftarrow]{u}
&
\AA^2
\arrow[hookleftarrow]{u}
\end{tikzcd}
~,
\qquad
\text{and}
\qquad
\begin{tikzcd}
V(x,y)
\arrow[hook]{r}
&
V(x)
\\
V(y)
\arrow[hook]{r}
\arrow[hookleftarrow]{u}
&
\AA^2
\arrow[hookleftarrow]{u}
\\
V(x-a,y-b)
\arrow[hook]{ru}
\end{tikzcd}
 \]
of $\Cls_{\AA^2}$: all three contain $\AA^2$, the first contains the two coordinate axes, the second additionally contains the origin $(0,0)$, and the third additionally contains the point $(a,b)$.\footnote{Here, $V(I) \in \Cls_{\AA^2}$ denotes the vanishing locus of an ideal $I \subseteq \kk[x,y]$.}  The first satisfies the generation condition but not the stratification condition, while the latter two define stratifications of $\AA^2$.
\end{example}

\begin{definition}
For each element $p \in \pos$, the \bit{$p\th$ stratum} of the stratification \Cref{intro.fixed.stratn.of.scheme} is the locally closed subset
\[
X_p
:=
\left( Z_p \left\backslash \bigcup_{q < p} Z_q \right. \right)
\]
of $X$.
\end{definition}

Altogether, the inclusions of the strata of the stratification \Cref{intro.fixed.stratn.of.scheme} assemble into a morphism
\begin{equation}
\label{morphism.from.disjoint.union.of.strata.of.stratified.scheme}
\coprod_{p \in \pos} X_p
\longra
X
~.
\end{equation}
For the stratifications described in Examples \ref{ex.intro.closed.subscheme.defines.stratn.over.brax.one}, \ref{ex.stratn.of.scheme.over.set}, \and \ref{ex.intro.stratn.of.affine.plane}, the morphism \Cref{morphism.from.disjoint.union.of.strata.of.stratified.scheme} defines a bijection on underlying sets.
In fact, for any stratification \Cref{intro.fixed.stratn.of.scheme}, the morphism \Cref{morphism.from.disjoint.union.of.strata.of.stratified.scheme} defines an injection on underlying sets: this is a consequence of the stratification condition.  However, it does not always define a surjection: for instance, the constant functor
\[
\NN^\op
:=
\{ 1 \ra 2 \ra 3 \ra \cdots \}^\op
\xra{\const_X}
\Cls_X
\]
defines a stratification (as a special instance of \Cref{ex.stratn.of.scheme.over.totally.ordered.poset}) whose strata are all empty, so that in this case the morphism \Cref{morphism.from.disjoint.union.of.strata.of.stratified.scheme} is not surjective unless $X$ itself is empty.  In fact, it is not hard to see that this counterexample is prototypical, in the sense that the morphism \Cref{morphism.from.disjoint.union.of.strata.of.stratified.scheme} is guaranteed to be surjective precisely when the poset $\pos$ is artinian (i.e.\! every decreasing sequence eventually stabilizes).

Of course, in order to reconstruct $X$ not just as a set but as a scheme, one would need to keep track of not just the strata $\{ X_p \}_{p \in \pos}$ but also gluing data between them.  \Cref{intro.thm.cosms} below enacts this idea in the noncommutative setting.  In parallel with the commutative situation just described, such reconstruction will depend on certain finiteness properties of the poset $\pos$.

\subsection{Stratified noncommutative stacks}
\label{subsection.intro.stratn.nc.stacks}

We now introduce our theory of stratified noncommutative stacks, which is closely patterned after the theory of stratified schemes.

\begin{definition}
\label{defn.closed.nc.substack.intro}
A \bit{closed noncommutative substack} of the noncommutative stack $\cX$ is a full presentable stable subcategory $\cZ \subseteq \cX$ whose inclusion extends to a diagram
\[ \begin{tikzcd}[column sep=1.5cm]
\cZ
\arrow[hook, bend left]{r}
\arrow[dashed,leftarrow]{r}[transform canvas={yshift=0.05cm}]{\bot}[swap,transform canvas={yshift=-0.05cm}]{\bot}
\arrow[dashed,bend right, hook]{r}
&
\cX
\end{tikzcd} \]
of adjoint functors.\footnote{If the right adjoint $\cX \ra \cZ$ admits its own right adjoint, the latter will automatically be fully faithful.  (In general, if a functor $F$ has adjoints $F^L \adj F \adj F^R$, then $F^L$ is fully faithful if and only if $F^R$ is: this follows from the composite adjunction $FF^L \adj FF^R$, in which one adjoint is naturally equivalent to the identity functor if and only if the other is.)}  We write $\Cls_\cX$ for the poset of closed noncommutative substacks of $\cX$ ordered by inclusion.
\end{definition}

\noindent Of course, our terminology is motivated by the fact that a closed subset $Z \subseteq X$ determines a closed noncommutative substack $\QC_Z(X) \subseteq \QC(X)$, as indicated in \Cref{ex.of.closed.open.decomp.giving.recollement.in.intro}.  This construction defines a functor
\[
\Cls_X
\xra{\QC_{(-)}(X)}
\Cls_{\QC(X)}
~.
\]

\begin{definition}
\label{defn.intro.nc.stratn}
A \bit{stratification} of the noncommutative stack (i.e.\! presentable stable $\infty$-category) $\cX$ over the poset $\pos$ is a functor
\begin{equation}
\label{intro.fixed.stratn.of.nc.stack}
\begin{tikzcd}[row sep=0cm]
\pos
\arrow{r}{\cZ_\bullet}
&
\Cls_\cX
\\
\rotatebox{90}{$\in$}
&
\rotatebox{90}{$\in$}
\\
p
\arrow[maps to]{r}
&
\cZ_p
\end{tikzcd}
\end{equation}
satisfying the following conditions:
\begin{enumerate}

\item[] {\bf generation:} $\cX = \bigcup_{p \in \pos} \cZ_p$;

\item[] {\bf stratification:} for any $p,q \in \pos$, there exists a factorization
\[ \begin{tikzcd}
{\displaystyle \bigcup_{r \leq p \textup{ and }r \leq q} \cZ_r}
\arrow[hook]{r}
&
\cZ_p
\\
\cZ_q
\arrow[hook]{r}
\arrow[dashed]{u}
&
\cX
\arrow{u}
\end{tikzcd}
~.
\]

\end{enumerate}
Here, the union symbol $\bigcup$ denotes the colimit (i.e.\! least upper bound) in the poset $\Cls_\cX$.\footnote{In fact, colimits in $\Cls_\cX$ always exist and are quite straightforward to compute; see \Cref{obs.closed.subcats.closed.under.colimit.closure}.} In this situation, we may also say that $\cX$ is \bit{$\pos$-stratified}.
\end{definition}

\begin{remark}
Given a stratification \Cref{intro.fixed.stratn.of.nc.stack} of $\cX$, the commutative square
\[ \begin{tikzcd}
{\displaystyle \bigcup_{r \leq p \textup{ and }r \leq q} \cZ_r}
\arrow[hook]{r}
&
\cZ_p
\\
\cZ_q
\arrow[hook]{r}
\arrow[hookleftarrow]{u}
&
\cX
\arrow[hookleftarrow]{u}
\end{tikzcd} \]
of defining fully faithful inclusions is in fact a pullback.\footnote{This follows from \Cref{closed.subcats.are.mutually.aligned}; see \Cref{defn.aligned}.}  Thus, the stratification condition of \Cref{defn.intro.nc.stratn} is a close cousin of the stratification condition of \Cref{defn.intro.comm.stratn}.
\end{remark}

\begin{example}
Suppose that $\pos = \{a,b\}$ is a two-element set, considered as a discrete poset.  A stratification
\[ \{ a , b \} \xlongra{\cZ_\bullet} \Cls_\cX \]
is the data of a pair of closed noncommutative substacks $\cZ_a,\cZ_b \in \cX$ such that $\cZ_a \cup \cZ_b = \cX$ and such that the composites
\[
\cZ_a
\longhookra
\cX
\longra
\cZ_b
\qquad
\text{and}
\qquad
\cZ_b
\longhookra
\cX
\longra
\cZ_a
\]
are both zero.  It follows immediately that we have an adjoint equivalence
\[ \begin{tikzcd}[column sep=1.5cm]
\cZ_a \times \cZ_b
\arrow[hook, transform canvas={yshift=0.9ex}]{r}{i_L \oplus i_L}
\arrow[leftarrow, transform canvas={yshift=-0.9ex}]{r}[yshift=-0.0ex]{\sim}[swap]{(y,y)}
&
\cX
\end{tikzcd}~; \]
in other words, a stratification of the noncommutative stack $\cX$ over $\{ a, b \}$ is nothing other than a decomposition of $\cX$ as the product of two closed noncommutative substacks.\footnote{Conversely, any product decomposition $\cX \simeq \cZ_a \times \cZ_b$ by full stable subcategories is necessarily by closed noncommutative substacks.}  More generally, for any set $S$ considered as a discrete poset, a stratification of $\cX$ over $S$ is the data of a product decomposition $\cX \simeq \prod_{s \in S} \cZ_s$ by full stable subcategories.\footnote{This may be compared with \Cref{ex.stratn.of.scheme.over.set}; note that the functor $\QC$ takes disjoint unions to products.}
\end{example}

\begin{definition}
\label{defn.intro.strata.and.geometric.localization.adjunction}
For each element $p \in \pos$, the \bit{$p\th$ stratum} of the stratification \Cref{intro.fixed.stratn.of.nc.stack} of $\cX$ is the presentable quotient
\[
\cX_p
:=
\left( \cZ_p \left/ \bigcup_{q < p} \cZ_q \right. \right)
~,
\]
which essentially by definition participates in the recollement
\[ \begin{tikzcd}[column sep=1.5cm]
{\displaystyle\bigcup_{q < p}}
&[-1.75cm]
\cZ_q
\arrow[hook, bend left=45]{r}[description]{i_L}
\arrow[leftarrow]{r}[transform canvas={yshift=0.1cm}]{\bot}[swap,transform canvas={yshift=-0.1cm}]{\bot}[description]{\yo}
\arrow[bend right=45, hook]{r}[description]{i_R}
&
\cZ_p
\arrow[bend left=45]{r}[description]{p_L}
\arrow[hookleftarrow]{r}[transform canvas={yshift=0.1cm}]{\bot}[swap,transform canvas={yshift=-0.1cm}]{\bot}[description]{\nu}
\arrow[bend right=45]{r}[description]{p_R}
&
\cX_p
\end{tikzcd}~. \]
Hence, we obtain a composite adjunction
\[
\begin{tikzcd}[column sep=1.5cm]
\Phi_p
:
\cX
\arrow[transform canvas={yshift=0.9ex}]{r}{y}
\arrow[hookleftarrow, transform canvas={yshift=-0.9ex}]{r}[yshift=-0.2ex]{\bot}[swap]{i_R}
&
\cZ_p
\arrow[transform canvas={yshift=0.9ex}]{r}{p_L}
\arrow[hookleftarrow, transform canvas={yshift=-0.9ex}]{r}[yshift=-0.2ex]{\bot}[swap]{\nu}
&
\cX_p
:
\rho^p
\end{tikzcd}~,
\]
whose left adjoint $\Phi_p$ we refer to as the \bit{$p\th$ geometric localization functor} of the stratification \Cref{intro.fixed.stratn.of.nc.stack}.
\end{definition}

\begin{example}
\label{ex.intro.induced.stratn.of.qcoh}
A stratification \Cref{intro.fixed.stratn.of.scheme} of the scheme $X$ determines a stratification
\begin{equation}
\label{intro.induced.stratn.of.QCoh}
\begin{tikzcd}[row sep=0cm, column sep=1.5cm]
\pos
\arrow{r}{Z_\bullet}
&
\Cls_X
\arrow{r}{\QC_{(-)}(X)}
&
\Cls_{\QC(X)}
\\
\rotatebox{90}{$\in$}
&
&
\rotatebox{90}{$\in$}
\\
p
\arrow[maps to]{rr}
&
&
\QC_{Z_p}(X)
\end{tikzcd}
\end{equation}
of its underlying noncommutative stack $\QC(X)$.\footnote{\label{fn.why.noetherian}Without hypotheses on the scheme $X$, suppose that the functor $\Cls_X \xra{\QC_{(-)}(X)} \Cls_{\QC(X)}$ exists, as guaranteed e.g.\! by the assumption that $X$ is qcqs (recall \Cref{fn.why.qcqs}). Then, the composite functor \Cref{intro.induced.stratn.of.QCoh} automatically satisfies stratification condition. The assumption that $X$ is noetherian guarantees that it also satisfies the generation condition. For an example where the generation condition fails, see \Cref{rmk.counterex.to.generation.condition.for.adelic}.}  Given any element $p \in \pos$, let us choose a factorization
\[ \begin{tikzcd}
X_p
\arrow[hook]{rr}[swap]{\textup{\sf locally closed}}
\arrow[dashed, hook]{rd}[swap, sloped]{\textup{\sf closed}}
&
&
X
\\
&
U_p
\arrow[dashed, hook]{ru}[sloped, swap]{\textup{\sf open}}
\end{tikzcd}~. \]
Then, the $p\th$ stratum of the stratification \Cref{intro.induced.stratn.of.QCoh} can be identified as $\QC_{X_p}(U_p) \simeq \QC((U_p)^\wedge_{X_p})$ (recall the equivalence of \Cref{ex.of.closed.open.decomp.giving.recollement.in.intro}), and thereafter its $p\th$ geometric localization functor can be identified as the composite
\[
\Phi_p
:
\QC(X)
\xra{p_L}
\QC(U_p)
\xlongra{y}
\QC((U_p)^\wedge_{X_p})
~.\footnote{This identification follows from Lemmas \ref{closed.subcats.are.mutually.aligned} \and \ref{lemma.all.about.aligned.subcats}\Cref{part.alignment.lemma.induced.map.on.quotients}\Cref{subpart.alignment.lemma.yo.and.pL.commutativity}.}
\]
\end{example}

In parallel with \Cref{ex.intro.closed.subscheme.defines.stratn.over.brax.one}, a stratification of $\cX$ over $[1]$ is simply the data of a closed noncommutative substack $\cZ := \cZ_0 \subseteq \cX$.  This necessarily extends to a recollement \Cref{recollement.in.intro}, and indeed the strata of this stratification are simply
\[
\cX_0
:=
\cZ/ 0
\simeq
\cZ
\qquad
\text{and}
\qquad
\cX_1
:=
\cX / \cZ
\simeq
\cU
~.
\]
Moreover, as we have seen in \Cref{subsection.intro.closed.open.decomps.and.recollements}, the gluing datum necessary for reconstructing $\cX$ from these strata is the composite functor
\begin{equation}
\label{extension.data.for.recollement.as.stratn.over.brax.one}
\begin{tikzcd}[row sep=0cm, column sep=1.5cm]
\cZ
\arrow{r}{p_L i_R}
&
\cU
\\
\rotatebox{90}{$=$}
&
\rotatebox{90}{$=$}
\\
\cX_0
\arrow{r}[swap]{\Phi_1 \rho^0}
&
\cX_1
\end{tikzcd}~.\footnote{Thus, the recollement \Cref{recollement.in.intro} may be thought of as a sort of categorified extension sequence, which is classified by the data of the functor \Cref{extension.data.for.recollement.as.stratn.over.brax.one}.  This analogy will be amplified in \Cref{rmk.table.of.analogies.between.stratns.and.filtrns}.}
\end{equation}


This suggests the following general construction: given a stratification \Cref{intro.fixed.stratn.of.nc.stack} of $\cX$ over an arbitrary poset $\pos$, for each morphism $p \ra q$ in $\pos$ we have an associated \bit{gluing functor}
\[ \Gamma^p_q : \cX_p \xlonghookra{\rho^p} \cX \xra{\Phi_q} \cX_q  \]
between the corresponding strata.\footnote{When both sources and targets appear in our notation, we put the source as a superscript and the target as a subscript (so e.g.\! we have $\Ar(\cC)^{|c} \simeq \cC_{c/}$ and $\Ar(\cC)_{|c} \simeq \cC_{/c}$ for any $\infty$-category $\cC$; these conventions are motivated by the fact that $\hom_\cC(-,-)$ is contravariant in the source and covariant in the target). Moreover, we have chosen to use a subscript in the notation $\Phi_q$ and a superscript in the notation $\rho^p$ in order to maintain consistency with the notation $\Gamma^p_q$.} Given a composable sequence $p \ra q \ra r$ in $\pos$, the associated gluing functors generally do not strictly compose: rather, they fit into a lax-commutative triangle
\[ \begin{tikzcd}[row sep=1.5cm, column sep=0.75cm]
&
\cX_q
\arrow{rd}[sloped]{\Gamma^q_r}
\\
\cX_p
\arrow{ru}[sloped]{\Gamma^p_q}
\arrow{rr}[transform canvas={yshift=0.6cm}]{\rotatebox{90}{$\Rightarrow$}}[swap]{\Gamma^p_r}
&
&
\cX_r
\end{tikzcd}~, \]
whose natural transformation arises from the unit of the adjunction $\Phi_q \adj \rho^q$.\footnote{For instance, in the situation and notation of \Cref{ex.intro.induced.stratn.of.qcoh}, the lax-commutative triangle
\[ \begin{tikzcd}[ampersand replacement=\&, row sep=0.7cm, column sep=-0.5cm]
\&
\&
\QC((U_q)^\wedge_{X_q})
\arrow[hook]{rd}[sloped]{\rho^q}
\\
\&
\QC(X)
\arrow{ru}[sloped]{\Phi_q}[xshift=0.6cm, yshift=-1cm]{\Uparrow}
\&
\&
\QC(X)
\arrow{rd}[sloped]{\Phi_r}
\\
\QC((U_p)^\wedge_{X_p})
\arrow[hook]{ru}[sloped]{\rho^p}
\arrow[hook]{rr}[swap]{\rho^p}
\&
\&
\QC(X)
\arrow{rr}[swap]{\Phi_r}
\&
\&
\QC((U_r)^\wedge_{X_r})
\end{tikzcd} \]
records the difference between push-pull operations either directly from $(U_p)^\wedge_{X_p}$ to $(U_r)^\wedge_{X_r}$ or passing intermediately through $(U_q)^\wedge_{X_q}$.} An elaboration of this observation reveals that the gluing functors assemble into a \textit{left-lax} functor
\begin{equation}
\label{intro.gluing.diagram.as.llax.functor.to.PrSt}
\begin{tikzcd}[column sep=1.5cm]
\pos
\arrow{r}[description, yshift=-0.05cm]{\llax}{\GD(\cX)}
&
\PrSt
\end{tikzcd}
\end{equation}
to the $(\infty,2)$-category $\PrSt$ of presentable stable $\infty$-categories and accessible exact functors between them, which we refer to as the \bit{gluing diagram} of the stratification and denote by $\GD(\cX)$ (see \Cref{defn.gluing.diagram}).

We can now state our first main theorem, which provides sufficient conditions for the reconstruction of a stratified noncommutative stack $\cX$ from its gluing diagram $\GD(\cX)$.  As foreshadowed at the end of \Cref{subsection.intro.strat.sch}, such reconstruction may be obstructed by certain convergence issues, which we precisely codify (see \Cref{rmk.sharpness.of.reconstrn.thm}).  In order to highlight its recursive structure, we state the theorem succinctly before explaining its terms.

\begin{maintheorem}[Theorems \ref{metacosm.thm} \and \ref{macrocosm.thm}]
\label{intro.thm.cosms}
Let $\pos$ be a poset.
\begin{enumerate}
\item\label{intro.main.thm.metacosm} {\bf metacosm:} The gluing diagram functor is the left adjoint in an adjunction
\begin{equation}
\label{intro.mainthm.metacosm.adjn}
\begin{tikzcd}[column sep=1.5cm]
\Strat_\pos
\arrow[transform canvas={yshift=0.9ex}]{r}{\GD}
\arrow[hookleftarrow, transform canvas={yshift=-0.9ex}]{r}[yshift=-0.2ex]{\bot}[swap]{\limrlaxfam}
&
\LMod^{\rlax,L}_{\llax.\pos}(\PrSt)
\end{tikzcd}
~.
\end{equation}

\item\label{intro.main.thm.macrocosm} {\bf macrocosm:} For each $\pos$-stratified noncommutative stack $\cX \in \Strat_\pos$, the unit of the adjunction \Cref{intro.mainthm.metacosm.adjn} determines the left adjoint in an adjunction
\begin{equation}
\label{intro.mainthm.macrocosm.adjn}
\begin{tikzcd}[column sep=1.5cm]
\cX
\arrow[transform canvas={yshift=0.9ex}]{r}{\gd}
\arrow[leftarrow, transform canvas={yshift=-0.9ex}]{r}[yshift=-0.2ex]{\bot}[swap]{\lim_{\sd(\pos)}}
&
\Glue (\cX) := \limrlaxP(\GD(\cX))
\end{tikzcd}
~.
\end{equation}

\item\label{intro.main.thm.microcosm} {\bf microcosm:} For each quasicoherent sheaf $\cF \in \cX \in \Strat_\pos$ on a $\pos$-stratified noncommutative stack, the unit of the adjunction \Cref{intro.mainthm.macrocosm.adjn} is a morphism
\begin{equation}
\label{intro.mainthm.microcosm.morphism}
\cF
\longra
\glue ( \cF ) := \lim_{\sd(\pos)} ( \gd ( \cF ) )
\end{equation}
in $\cX$.

\item\label{intro.main.thm.nanocosm} {\bf nanocosm:} For each quasicoherent sheaf $\cE \in \cX \in \Strat_\pos$ on a $\pos$-stratified noncommutative stack, applying $\ulhom_\cX(\cE,-)$ to the morphism \Cref{intro.mainthm.microcosm.morphism} determines a morphism
\begin{equation}
\label{intro.mainthm.nanocosm.morphism}
\ulhom_\cX(\cE,\cF)
\longra
\lim_{([n] \xra{\varphi} \pos) \in \sd(\pos)}
\left( \ulhom_{\cX_{\varphi(n)}} ( \Phi_{\varphi(n)}\cE , \Gamma_\varphi \Phi_{\varphi(0)} \cF ) \right)
\end{equation}
in $\Spectra$.

\end{enumerate}
Moreover, if the poset $\pos$ is down-finite,\footnote{A poset $\pos$ is called \textit{down-finite} if for each element $p \in \pos$, its down-closure $(^\leq p) := \{ q \in \pos : q \leq p \}$ is finite.} then the metacosm adjunction \Cref{intro.mainthm.metacosm.adjn} 
-- and hence the macrocosm adjunction \Cref{intro.mainthm.macrocosm.adjn}, and hence the microcosm morphism \Cref{intro.mainthm.microcosm.morphism}, and hence the nanocosm morphism \Cref{intro.mainthm.nanocosm.morphism} --
is an equivalence.\footnote{Of course, the implied implications are irreversible: respectively, it is possible for \Cref{intro.mainthm.macrocosm.adjn}, \Cref{intro.mainthm.microcosm.morphism}, or \Cref{intro.mainthm.nanocosm.morphism} to be an equivalence even if \Cref{intro.mainthm.metacosm.adjn}, \Cref{intro.mainthm.macrocosm.adjn}, or \Cref{intro.mainthm.microcosm.morphism} is not. (On the other hand, fixing some $\cX \in \Strat_\pos$, if for every $\cF \in \cX \in \Strat_\pos$ the microcosm morphism \Cref{intro.mainthm.microcosm.morphism} is an equivalence, then the macrocosm adjunction \Cref{intro.mainthm.macrocosm.adjn} is an equivalence. Likewise, fixing some $\cF \in \cX \in \Strat_\pos$, if for every $\cE \in \cX \in \Strat_\pos$ the nanocosm morphism \Cref{intro.mainthm.nanocosm.morphism} is an equivalence, then the microcosm morphism \Cref{intro.mainthm.microcosm.morphism} is an equivalence. See also \Cref{rmk.sharpness.of.reconstrn.thm}.)}
\end{maintheorem}

\noindent The various expressions appearing in \Cref{intro.thm.cosms} have the following meaning.
\begin{enumerate}

\item[] {\bf metacosm:} We write
\begin{itemize}
\item
$\Strat_\pos$ for the $\infty$-category of $\pos$-stratified noncommutative stacks,
\item
$\LMod^{\rlax,L}_{\llax.\pos}(\PrSt)$ for a certain $\infty$-category whose objects are left-lax left $\pos$-modules in $\PrSt$ (i.e.\! left-lax functors from $\pos$ to $\PrSt$),
\item
$\GD$ for the (\textit{macrocosm}) \textit{gluing diagram} functor (taking a $\pos$-stratified noncommutative stack to its gluing diagram), and
\item
$\limrlaxfam$ for a certain ``parametrized right-lax limit'' functor.
\end{itemize}
We say that $\cX \in \Strat_\pos$ is \textit{convergent} if it lies in the image of the right adjoint of the metacosm adjunction \Cref{intro.mainthm.metacosm.adjn}, or equivalently if its macrocosm adjunction \Cref{intro.mainthm.macrocosm.adjn} is an equivalence.

\item[] {\bf macrocosm:} We refer to $\Glue (\cX) := \limrlaxP ( \GD ( \cX ))$ as the \textit{reglued noncommutative stack} of $\cX \in \Strat_\pos$; this is the underlying object of the $\pos$-stratified noncommutative stack $\limrlaxfam(\GD(\cX)) \in \Strat_\pos$.  It may be identified as a full subcategory
\begin{equation}
\label{intro.inclusion.of.Glue.X.as.full.subcat.of.Fun.sd.P.X}
\Glue (\cX) \subseteq \Fun ( \sd(\pos) , \cX )
\end{equation}
of the $\infty$-category of functors to $\cX$ from the subdivision of $\pos$,\footnote{The \textit{subdivision} of the poset $\pos$ is the full subcategory $\sd(\pos) \subseteq \bDelta_{/\pos}$ (which is in fact a poset) on the conservative (or equivalently injective) functors $[n] \ra \pos$.} through which the notation $\lim_{\sd(\pos)}$ acquires meaning.  We write
$\gd$ for the (\textit{microcosm}) \textit{gluing diagram} functor (taking a quasicoherent sheaf to its gluing diagram).  For each $p \in \pos$, the $p\th$ geometric localization functor appears as the factored composite
\[ \begin{tikzcd}
\cX
\arrow{r}{\gd}
\arrow[dashed, bend right=10]{rrd}[sloped, swap]{\Phi_p}
&
\Glue (\cX)
\arrow{r}{\ev_p}
\arrow[dashed]{rd}
&
\cX
\\
&
&
\cX_p
\arrow[hook]{u}[swap]{\rho^p}
\end{tikzcd}
~.\footnote{More generally, for any $([n] \xra{\varphi} \pos) \in \sd(\pos)$, the composite $\cX \xra{\gd} \Glue(\cX) \xra{\ev_\varphi} \cX$ is the composite $\rho^{\varphi(n)} \Phi_{\varphi(n)} \cdots \rho^{\varphi(0)} \Phi_{\varphi(0)}$.}
\]
In particular, the gluing diagram $\gd(\cF) \in \Glue (\cX)$ indeed consists of the geometric localizations $\{ \Phi_p (\cF) \in \cX_p \}_{p \in \pos}$ along with gluing data between them.

\item[] {\bf microcosm:} We refer to $\glue (\cF) := \lim_{\sd(\pos)}(\gd(\cF))$ as the \textit{reglued quasicoherent sheaf} of $\cF \in \cX$.  We say that $\cF \in \cX$ is \textit{convergent} if its microcosm morphism \Cref{intro.mainthm.microcosm.morphism} is an equivalence.

\item[] {\bf nanocosm:} We write $\ulhom$ to denoted enriched hom, here meaning the hom-spectra of a stable $\infty$-category. Moreover, for any $([n] \xra{\varphi} \pos) \in \sd(\pos)$ we write
\[
\Gamma_\varphi
:=
\Gamma^{\varphi(n-1)}_{\varphi(n)} \cdots \Gamma^{\varphi(0)}_{\varphi(1)}
\]
for brevity.\footnote{The nanocosm morphism \Cref{intro.mainthm.nanocosm.morphism} is described in detail in \Cref{rmk.nanocosm}.}

\end{enumerate}
When the nanocosm morphism \Cref{intro.mainthm.nanocosm.morphism} is an equivalence, it may be viewed as affording a description of the ``(generalized) elements'' of a sheaf (i.e.\! morphisms into it) entirely in terms of compatible local elements (i.e.\! morphisms in the strata of the stratification).\footnote{For instance, via \Cref{intro.thm.gen.G.spt} this provides an explicit formula for the categorical fixedpoints of a genuine $G$-spectrum (where $G$ is a finite group) as a finite limit of spectra that are defined in terms of its geometric fixedpoints (see \Cref{subsection.categorical.fixedpoints}).}




\begin{remark}
\label{rmk.lax.limits.have.structure.maps}
Fix an $\infty$-category $\cB \in \Cat$.  Given a functor
\[
\cB
\xlongra{F}
\Cat
\]
to the $(\infty,2)$-category $\Cat$ of $\infty$-categories, an object of its \textit{limit} may be thought of informally as a system of the following data:
\begin{itemize}
\item for each object $b \in \cB$, an object $e_b \in F(b)$;
\item for each morphism $b_0 \xra{\gamma} b_1$ in $\cB$, an equivalence
\begin{equation}
\label{equivce.in.defn.of.limit.of.cats}
\begin{tikzcd}
F(\gamma)(e_{b_0})
\arrow[leftrightarrow]{r}{\sigma_\gamma}[swap]{\sim}
&
e_{b_1}
\end{tikzcd}
\end{equation} in $F(b_1)$;
\item for each composable pair of morphisms $b_0 \xra{\gamma} b_1 \xra{\delta} b_2$ in $\cB$, a commutative square
\begin{equation}
\label{comm.square.of.equivces.in.defn.of.limit.of.diagram.in.Cat}
\begin{tikzcd}[row sep=1.5cm]
e_{b_2}
\arrow[leftrightarrow]{r}{\sigma_\delta}[swap]{\sim}
\arrow[leftrightarrow]{d}[swap]{\sigma_{\delta \gamma}}[sloped, anchor=south]{\sim}
&
F(\delta)(e_{b_1})
\arrow[leftrightarrow]{d}{F(\delta)(\sigma_\gamma)}[sloped, anchor=north]{\sim}
\\
F(\delta \gamma)(e_{b_0})
\arrow[leftrightarrow]{r}[swap]{\sim}
&
F(\delta)(F(\gamma)(e_{b_0}))
\end{tikzcd}~,
\end{equation}
in which the lower equivalence follows from the fact that $F$ is a functor (and so respects composition of morphisms up to canonical equivalence);
\item higher coherence data.
\end{itemize}
By contrast, to describe an object of its \textit{left-lax limit} or of its \textit{right-lax limit}, we must replace the equivalences \Cref{equivce.in.defn.of.limit.of.cats} with morphisms
\[
F(\gamma)(e_{b_0}) \longra e_{b_1}
\qquad
\text{or}
\qquad
F(\gamma)(e_{b_0}) \longla e_{b_1}~,
\]
respectively.\footnote{Thus, the limit is a full subcategory of both the left-lax limit and the right-lax limit.}  More general definitions apply when the functor $F$ is itself only left- or right-lax, such as the gluing diagram \Cref{intro.gluing.diagram.as.llax.functor.to.PrSt} (whose right-lax limit defines the reglued noncommutative stack $\Glue(\cX)$); for instance, in the commutative square \Cref{comm.square.of.equivces.in.defn.of.limit.of.diagram.in.Cat}, the lower morphism will in general no longer be an equivalence.

We provide a comprehensive treatment of these notions in \Cref{section.lax.actions.and.limits}.  In particular, we unpack the general definition of the right-lax limit of a left-lax left $[2]$-module (the simplest nontrivial case) in \Cref{example.limits.of.lax.actions.when.laxness.disagrees}\Cref{example.limits.of.lax.actions.when.laxness.disagrees.rlax.lim.of.llax.action}. We also describe the following notions in \Cref{ex.gluing.stuff.over.brax.two} for a noncommutative stack $\cX$ stratified over the poset $\pos = [2]$: the gluing diagram $\GD(\cX)$, the reglued noncommutative stack $\Glue(\cX) := \lim^\rlax_{\llax.\pos}(\GD(\cX))$, the gluing diagram functor $\cX \xra{\gd} \Glue(\cX)$, and its right adjoint $\Glue(\cX) \xra{\lim_{\sd(\pos)}} \cX$.
\end{remark}

\begin{remark}
\label{rmk.intro.strictification.of.rlax.lim}
We show as \Cref{lem.strictification} that the right-lax limit of a left-lax functor
\[
\begin{tikzcd}[column sep=1.5cm]
\pos
\arrow{r}[description, yshift=-0.05cm]{\llax}{F}
&
\Cat
\end{tikzcd}
\]
may be computed as the strict (i.e.\! ordinary) limit of a certain strict (i.e.\! ordinary) functor
\[
\sd(\pos)
\xra{\Strict(F)}
\Cat
\]
constructed therefrom:
\begin{equation}
\label{intro.equivce.between.rlax.lim.of.llax.functor.and.lim.of.strictification}
\lim^\rlax \left( 
\begin{tikzcd}[column sep=1.5cm]
\pos
\arrow{r}[description, yshift=-0.05cm]{\llax}{F}
&
\Cat
\end{tikzcd}
\right)
\simeq
\lim \left(
\sd(\pos)
\xra{\Strict(F)}
\Cat
\right)
~.\footnote{This equivalence generalizes the identification appearing in the equivalence \Cref{macrocosm.equivalence.for.recollement.in.intro}, which is an instance of the equivalence \Cref{intro.equivce.between.rlax.lim.of.llax.functor.and.lim.of.strictification} in the case that $\pos = [1]$.}
\end{equation}
In addition to its technical utility, this result allows for a more uniform perspective on the metacosm adjunction \Cref{intro.mainthm.metacosm.adjn} and the macrocosm adjunction \Cref{intro.mainthm.macrocosm.adjn}: in both, the right adjoint is computed by taking (strict) limits over $\sd(\pos)$.
\end{remark}

\begin{remark}
\label{rmk.sharpness.of.reconstrn.thm}
\Cref{intro.thm.cosms} is sharp in the sense that the metacosm adjunction \Cref{intro.mainthm.metacosm.adjn} fails to be an equivalence whenever $\pos$ is not down-finite.\footnote{Hence, convergence is analogous at the metacosm level to the down-finiteness of $\pos$.}  Equivalently (as the forgetful functor $\Strat_\pos \ra \PrLSt$ is conservative), when $\pos$ is not down-finite then there exists a $\pos$-stratified noncommutative stack whose macrocosm adjunction \Cref{intro.mainthm.macrocosm.adjn} is not an equivalence.  Using (both the content and terminology of) Theorems \ref{intro.thm.fund.opns} \and \ref{intro.thm.balmer} as well as \Cref{ex.intro.arithmetic}, such a $\pos$-stratified noncommutative stack may be constructed by choosing an injective functor
\[
\pos_\ZZ
\longhookra
\pos
\]
from the specialization poset of $\Spec(\ZZ)$ and taking the pushforward of the adelic stratification of $\Mod_\ZZ$ along it.\footnote{Indeed, a poset is down-finite precisely when it admits no injective functors from $\pos_\ZZ$.}
It is not hard to see that passing to this new stratification of $\Mod_\ZZ$ yields an equivalent macrocosm adjunction, which 
is not an equivalence.
\end{remark}

\begin{remark}
\label{rmk.artinian.conservativity}
The requirement that $\pos$ be down-finite is strictly stronger than the requirement that it be artinian.  Indeed, the specialization poset of $\Spec(\ZZ)$ (depicted in diagram \Cref{speczn.poset.of.Mod.Z}) is artinian but not down-finite.  In fact, it is not hard to see that the assumption that $\pos$ be artinian guarantees that the functor
\begin{equation}
\label{projection.to.fibers.is.it.conservative}
\cX
\xra{(\Phi_p)_{p \in \pos}}
\prod_{p \in \pos} \cX_p
\end{equation}
is conservative; this is directly analogous to the guaranteed surjectivity on underlying sets of the morphism \Cref{morphism.from.disjoint.union.of.strata.of.stratified.scheme} under that same assumption. 
From this perspective, the further assumption that $\pos$ be down-finite may be seen as assuring that the gluing data suffice to recover the noncommutative stack structure of $\cX$.
\end{remark}

\begin{example}[the Goodwillie--Taylor stratification]
\label{ex.goo}
Goodwillie calculus leads to a stratification over a nonartinian poset in which the functor \Cref{projection.to.fibers.is.it.conservative} generally fails to be conservative, as we now explain.\footnote{A version of this stratification appears in work of Glasman \cite{Saul-strat,Saul-Goo} (see \Cref{subsubsection.compare.with.saul}).}  Specifically, we construct a stratification of $\cX := \Fun(\cI,\cY)$, where $\cY$ is any presentable stable $\infty$-category and $\cI$ is any $\infty$-category that admits finite colimits and has a terminal object.  First of all, by \cite[Theorem 6.1.1.10]{LurieHA} (see also \cite{GooCalc3}), for any $n \geq 0$ the inclusion of the full subcategory of $n$-excisive functors is the right adjoint in an adjunction
\[ \begin{tikzcd}[column sep=1.5cm]
\Fun(\cI,\cY)
\arrow[transform canvas={yshift=0.9ex}]{r}{P_n}
\arrow[hookleftarrow, transform canvas={yshift=-0.9ex}]{r}[yshift=-0.2ex]{\bot}
&
\Exc^n(\cI,\cY)
\end{tikzcd}~, \]
whose left adjoint $P_n$ carries a functor to its $n$-excisive approximation.  This inclusion commutes with colimits, and so admits a right adjoint of its own.  We trivially extend this to the case that $n=-1$ by declaring that $\Exc^{-1}(\cI,\cY) := \{ 0 \} \subseteq \Fun(\cI,\cY)$, i.e.\! that only the constant functor at the zero object is $(-1)$-excisive.  Hence, we obtain a stratification
\begin{equation}
\label{goodwillie.stratn}
\begin{tikzcd}[row sep=0cm]
(\ZZ_{\geq -1})^\op
\arrow{r}
&
\Cls_{\Fun(\cI,\cY)}
\\
\rotatebox{90}{$\in$}
&
\rotatebox{90}{$\in$}
\\
n
\arrow[maps to]{r}
&
\ker(P_n)
\end{tikzcd}
\end{equation}
(recall \Cref{ex.stratn.of.scheme.over.totally.ordered.poset}).  Following the same reasoning as is laid out in \Cref{ex.chromatic.stratn.of.spectra} (in the case of $p$-local spectra), we find that the macrocosm adjunction of the stratification \Cref{goodwillie.stratn} may be identified as the adjunction
\[ \begin{tikzcd}[column sep=1.5cm]
\Fun(\cI,\cY)
\arrow[transform canvas={yshift=0.9ex}]{r}
\arrow[leftarrow, transform canvas={yshift=-0.9ex}]{r}[yshift=-0.2ex]{\bot}
&
\Exc^\infty(\cI,\cY)
\end{tikzcd}~, \]
whose unit morphism at a functor $F \in \Fun(\cI,\cY)$ is the canonical morphism
\[
F
\longra
P_\infty F := \lim\left( \cdots \longra P_2 F \longra P_1 F \longra P_0 F \longra P_{-1} F \simeq 0 \right)
\]
to the limit of its Goodwillie--Taylor tower, which is not generally an equivalence.
\end{example}

\begin{remark}[filtrations from stratifications]
\label{rmk.filtrations.from.stratns}
Fix a $\pos$-stratified noncommutative stack $\cX \in \Strat_\pos$, and assume for simplicity that $\pos$ is down-finite.\footnote{Without hypotheses on $\pos$, the filtrations that we define below exist but convergence is more subtle.} For each $p \in \pos$, we have the corresponding recollement \Cref{recollement.in.intro} with $\cZ = \cZ_p$. Writing
\[
\fil^L_p := i_L y
\qquad
\text{and}
\qquad
\fil^p_R := i_R y
~,
\]
we obtain canonical embeddings with retracts
\[ \begin{tikzcd}
\Fun(\pos,\cX)
\arrow[hookleftarrow]{r}{\fil^L_\bullet}
\arrow[bend right]{r}[swap, pos=0.4]{\colim_\pos}
&
\cX
\arrow[hook]{r}{\fil_R^\bullet}
\arrow[leftarrow, bend right]{r}[swap, pos=0.6]{\lim_{\pos^\op}}
&
\Fun(\pos^\op,\cX)
\end{tikzcd}
~;
\]
in particular, every object of $\cX$ obtains a natural ascending $\pos$-filtration as well as a natural descending $\pos^\op$-filtration.\footnote{We use subscripts for ascending filtrations and superscripts for descending filtrations. (This handedness dictates the corresponding convention for associated graded components as either total cofibers or total fibers (\Cref{defn.tcofib.and.tfib}).)} In order to proceed, we introduce the composite adjunction
\[
\begin{tikzcd}[column sep=1.5cm]
\lambda^p
:
\cX_p
\arrow[hook, transform canvas={yshift=0.9ex}]{r}{\nu}
\arrow[leftarrow, transform canvas={yshift=-0.9ex}]{r}[yshift=-0.2ex]{\bot}[swap]{p_R}
&
\cZ_p
\arrow[hook, transform canvas={yshift=0.9ex}]{r}{i_L}
\arrow[leftarrow, transform canvas={yshift=-0.9ex}]{r}[yshift=-0.2ex]{\bot}[swap]{y}
&
\cX
:
\Psi_p
\end{tikzcd}
\]
for each $p \in \pos$, whose right adjoint $\Psi_p$ we refer to as the \bit{$p\th$ reflected geometric localization functor}.\footnote{The functors $\lambda^p$ and $\Psi_p$ play central roles in the theory of reflection (dual to those of $\rho^p$ and $\Phi_p$), which is introduced in \Cref{subsection.verdier}.} Using this notation, for each $p \in \pos$ we identify the $p\th$ associated graded components of these filtrations as total co/fibers (\Cref{defn.tcofib.and.tfib}), namely
\[
\gr^L_p
:=
\gr_p ( \fil^L_\bullet )
:=
\tcofib_{(^\leq p)}(\fil^L_\bullet)
\simeq
\lambda^p \Phi_p
\qquad
\text{and}
\qquad
\gr_R^p
:=
\gr^{p^\circ} ( \fil_R^\bullet )
:=
\tfib_{(^\leq p)^\op}(\fil_R^\bullet)
\simeq
\rho^p \Psi_p
~.
\]
In fact, we also have canonical filtrations corresponding to upwards-closures (rather than downwards-closures) of elements of $\pos$.\footnote{These latter filtrations are more natural in the context of stratified topological spaces, where a distinguished role is played by the closures of strata; see \Cref{rmk.easy.spectral.sequences.for.coh.of.strat.spcs}.} Namely, for each $p \in \pos$, we have the corresponding recollement \Cref{recollement.in.intro} with $\cZ = \bigcup_{q \not\geq p} \cZ_q$. Writing
\[
\fil_L^p
:=
\nu p_L
\qquad
\text{and}
\qquad
\fil^R_p
:=
\nu p_R
~,
\]
we similarly obtain canonical embeddings with retracts
\[ \begin{tikzcd}
\Fun(\pos,\cX)
\arrow[hookleftarrow]{r}{\fil_L^\bullet}
\arrow[bend right]{r}[swap, pos=0.4]{\lim_\pos}
&
\cX
\arrow[hook]{r}{\fil^R_\bullet}
\arrow[leftarrow, bend right]{r}[swap, pos=0.6]{\colim_{\pos^\op}}
&
\Fun(\pos^\op,\cX)
\end{tikzcd}
~.
\]
These have the same associated graded components as the above two filtrations: we have
\[
\gr_L^p
:=
\gr^p(\fil_L^\bullet)
:=
\tfib_{(^\geq p)}(\fil_L^\bullet)
\simeq
\lambda^p \Phi_p
\qquad
\text{and}
\qquad
\gr^R_p
:=
\gr_p(\fil^R_\bullet)
:=
\tcofib_{(^\geq p)^\op}(\fil^R_\bullet)
\simeq
\rho^p \Psi_p
\]
(i.e.\! $\gr_L^p \simeq \gr^L_p$ and $\gr^R_p \simeq \gr_R^p$). Moreover, the latter two filtrations can be obtained from the former two by the formulas
\[
\fil_L^p \cF
\simeq
\cofib
\Big(
\colim_{q \in (^{\not\geq} p)} \fil^L_q \cF
\longra
\cF
\Big)
\qquad
\text{and}
\qquad
\fil^R_p \cF
\simeq
\fib
\Big(
\cF
\longra
\lim_{q^\circ \in (^{\not\geq} p)^\op} \fil_R^q \cF
\Big)
~.
\]
\end{remark}

\begin{remark}[spectral sequences from stratifications]
\label{rmk.spectral.sequences.from.stratns}
Fix a $\pos$-stratified noncommutative stack $\cX \in \Strat_\pos$, and assume that $\pos$ is down-finite. Suppose that we are additionally given the following data:
\begin{itemize}

\item a conservative functor $\pos \xra{d} \ZZ$;

\item an exact functor $\cX \xra{H} \cV$ to a stable $\infty$-category $\cV$ equipped with a t-structure (e.g.\! $\cX \xra{\ulhom_\cX(\cE,-)} \Spectra$ for some $\cE \in \cX$).

\end{itemize}
Then, we obtain a composite
\[
\cX
\xra{\fil^L_\bullet}
\Fun(\pos,\cX)
\xlongra{H}
\Fun(\pos,\cV)
\xlongra{d_!}
\Fun(\ZZ,\cV)
~,
\]
i.e.\! a natural assignment of a filtered object in $\cV$ for each object of $\cX$. In particular, for each $\cF \in \cX$ we obtain a spectral sequence (see e.g.\! \cite[\S 1.2.2]{LurieHA}), which runs
\begin{equation}
\label{usual.filtration.spectral.sequence}
E^1_{s,t}
=
\bigoplus_{p \in d^{-1}(s)}
\pi_{s+t} ( H ( \gr^L_p \cF ) )
\Longrightarrow
\pi_{s+t} ( H ( \cF ) )
~.\footnote{This is guaranteed to converge e.g.\! if $\pos$ is finite or if $H$ is colimit-preserving.}
\end{equation}
Dually, using $\fil^R_\bullet$ in place of $\fil^L_\bullet$, we obtain a spectral sequence running
\begin{equation}
\label{usual.reflected.filtration.spectral.sequence}
E^1_{s,t}
=
\bigoplus_{p \in d^{-1}(-s)}
\pi_{s+t} ( H ( \gr^R_p \cF ) )
\Longrightarrow
\pi_{s+t} ( H ( \cF ) )
~.
\end{equation}
For the descending filtrations $\fil_L^\bullet$ and $\fil_R^\bullet$, we obtain essentially the same spectral sequences using $d_*$ instead of $d_!$, but with the sums replaced by products.
\end{remark}

\begin{remark}
\label{rmk.strictness}
More than being convergent, a $\pos$-stratified noncommutative stack $\cX$ or a quasicoherent sheaf thereon $\cF \in \cX$ may be \bit{strict}. In the former case, this is the condition that $\cX \in \Strat_\pos$ is convergent and moreover its gluing diagram \Cref{intro.gluing.diagram.as.llax.functor.to.PrSt} is a strict (as opposed to left-lax) functor. In both cases, strictness affords a simplified reconstruction theorem.  Moreover, $\cX \in \Strat_\pos$ is strict if and only if all of its objects are strict.\footnote{So, strictness is analogous at the metacosm level to the condition that the depth of the poset $\pos$ is at most $1$. One may likewise contemplate strictness at the nanocosm level, and (in a sense that is evident from the discussion of \Cref{subsection.strict.objects}) the object $\cF \in \cX$ is strict if and only if the pair $(\cE,\cF)$ is strict for all $\cE \in \cX$.}  We study strict objects in \Cref{subsection.strict.objects} and strict stratifications in \Cref{subsection.strict.stratns}.
\end{remark}

\begin{remark}
\label{rmk.extended.intro.variants.of.metacosm}
We establish a number of variations on the metacosm reconstruction of \Cref{intro.thm.cosms}\Cref{intro.main.thm.metacosm}, which we briefly describe here.
\begin{enumerate}

\item\label{rmk.extended.intro.part.stable.stratns}

We establish a theory of stratifications of stable $\infty$-categories (that are assumed to be idempotent-complete but not necessarily presentable), which we refer to as \textit{stable stratifications}. We provide a metacosm equivalence
\[
\begin{tikzcd}[column sep=1.5cm]
\strat_\pos
\arrow[transform canvas={yshift=0.9ex}]{r}{\GD}
\arrow[leftarrow, transform canvas={yshift=-0.9ex}]{r}[yshift=-0.0ex]{\sim}[swap]{\limrlaxfam}
&
\LMod^\rlax_{\llax.\pos}(\St^\idem)
\end{tikzcd}
\]
for stable stratifications as \Cref{thm.stable.metacosm} (under the assumption that $\pos$ is finite).

\item\label{rmk.extended.intro.part.strict.reconstrn}

We specialize both metacosm equivalences to \textit{strict} morphisms among (resp.\! stable) stratifications, which correspond with strict (as opposed to right-lax) morphisms among (the suitable sorts of) left-lax left $\pos$-modules: we establish equivalences
\[
\begin{tikzcd}[column sep=1.5cm]
\Strat_\pos^\strict
\arrow[transform canvas={yshift=0.9ex}]{r}{\GD}
\arrow[leftarrow, transform canvas={yshift=-0.9ex}]{r}[yshift=-0.0ex]{\sim}[swap]{\limrlaxfam}
&
\LMod^{L}_{\llax.\pos}(\PrSt)
\end{tikzcd}
\qquad
\text{and}
\qquad
\begin{tikzcd}[column sep=1.5cm]
\strat_\pos^\strict
\arrow[transform canvas={yshift=0.9ex}]{r}{\GD}
\arrow[leftarrow, transform canvas={yshift=-0.9ex}]{r}[yshift=-0.0ex]{\sim}[swap]{\limrlaxfam}
&
\LMod_{\llax.\pos}(\St^\idem)
\end{tikzcd}
\]
as \Cref{thm.strict.metacosm} (the former when $\pos$ is down-finite, the latter when $\pos$ is finite).

\item We establish the theory of \bit{reflection} for (resp.\! stable) stratifications, which we describe in \Cref{subsection.verdier}. This affords a dual form of reconstruction, which is desirable for reconstructing constructible sheaves (stratifications of which are discussed in \Cref{subsection.cbl}).

\end{enumerate}
\end{remark}

\subsection{Fundamental operations on stratified noncommutative stacks}
\label{subsection.intro.fund.operations}

The right adjoint in the metacosm adjunction \Cref{intro.mainthm.metacosm.adjn} may be viewed as the inclusion of the full subcategory of \textit{convergent} stratifications.  From this point of view, \Cref{intro.thm.cosms} (and its sharpness indicated in \Cref{rmk.sharpness.of.reconstrn.thm}) may be read as the assertion that all stratifications over $\pos$ are convergent if and only if $\pos$ is down-finite.

We view the possibility of nonconvergence not as a bug, but rather as an essential feature.  For example, the adelic stratifications guaranteed by \Cref{intro.thm.balmer} below are utterly fundamental and must constitute valid examples under any reasonable definition, and yet they do not generally converge.  And \Cref{ex.goo} provides compelling further evidence that nonconvergent stratifications should be considered as a common phenomenon indeed.

Of course, nonconvergent stratifications are not so useful on their own.  In order to extract convergent stratifications from nonconvergent ones (and as a key ingredient in the proof of \Cref{intro.thm.cosms}), we therefore establish a \textit{pushforward} operation for stratifications.  Its utility is illustrated in \Cref{ex.intro.arithmetic} below, where we show that a certain pushforward of the (nonconvergent) adelic stratification of $\Mod_\ZZ$ gives a (necessarily convergent) stratification over $[1]$ whose microcosm reconstruction theorem (i.e.\! the pullback square \Cref{microcosm.pullback.square.in.intro}) recovers the arithmetic fracture square \Cref{intro.arithmetic.fracture.square}.

In fact, pushforward is but one in a suite of \bit{fundamental operations} that we provide for constructing new stratifications from old ones.  We indicate their general structure here, and refer the reader to \Cref{subsection.structure.theory} for precise definitions and statements.

\begin{maintheorem}[\Cref{obs.restricted.stratn.over.D}, \Cref{prop.pullback.stratn}, \Cref{prop.quotient.stratn}, \Cref{prop.pushfwd.stratn}, and \Cref{prop.refined.stratn}]
\label{intro.thm.fund.opns}
Let $\pos$ be a poset and let $\cX \in \Strat_\pos$ be a $\pos$-stratified noncommutative stack.
\begin{enumerate}
\item\label{intro.main.thm.operations.restriction} {\bf restriction:} For any down-closed subset $\sD \subseteq \pos$, there is a \bit{restricted stratification} of $\cZ_\sD := \bigcup_{p \in \sD} \cZ_p$ over $\sD$.
\item\label{intro.main.thm.operations.pullback} {\bf pullback:} For any noncommutative stack $\w{\cX}$ equipped with a quotient functor $\w{\cX} \ra \cX$ by a closed noncommutative substack, there is a \bit{pullback stratification} of $\w{\cX}$ over $\pos$ (assuming that $\pos$ is nonempty).
\item\label{intro.main.thm.operations.quotient} {\bf quotient:} For any down-closed subset $\sD \subseteq \pos$, there is a \bit{quotient stratification} of $\cX/\cZ_\sD$ over $\pos \backslash \sD$.
\item\label{intro.main.thm.operations.pushfwd} {\bf pushforward:} 
For any functor $\pos \ra \sQ$ between posets, there is a \bit{pushforward stratification} of $\cX$ over $\posQ$.
\item\label{intro.main.thm.operations.refinement} {\bf refinement:} For any stratification of each stratum $\cX_p$ over a poset $\sR_p$, there is a \bit{refined stratification} of $\cX$ over the wreath product poset $\pos \wr \sR_\bullet$.
\end{enumerate}
\end{maintheorem}

\begin{remark}
Towards proving \Cref{intro.thm.fund.opns}, in \Cref{subsection.fund.opns.on.aligned.subcats} we introduce and study the notion of \bit{alignment} between closed subcategories. This does not seem to have a direct analog in point-set topology (or even in $\infty$-topos theory): in the $\infty$-category of sheaves on a topological space, closed subcategories associated to open subsets are automatically aligned (see \Cref{subsection.cbl}).
One manifestation of this idea is that alignment affords excision- and Mayer--Vietoris-type gluing formulas for closed subcategories.

Given a stratification, all of the closed subcategories that it determines (i.e.\! its values and colimits thereof) are automatically mutually aligned. Our results regarding alignment collectively streamline the arguments that comprise the proof of \Cref{intro.thm.fund.opns}. At the same time, the notion of alignment allows us to obtain generalizations of parts \Cref{intro.main.thm.operations.restriction} \and \Cref{intro.main.thm.operations.quotient} of \Cref{intro.thm.fund.opns} (see \Cref{prop.restricted.stratn} (and \Cref{rmk.restricted.stratn.recovers.that.for.down.closed}) for the former).
\end{remark}

\subsection{$\cO$-monoidal stratifications}
\label{subsection.intro.O.mon.stratns}

One attractive feature of our definition of a stratification is that it generalizes quite straightforwardly to the case of a presentably $\cO$-monoidal stable $\infty$-category $\cR$ (i.e.\! an $\cO$-algebra in the symmetric monoidal $\infty$-category $(\PrLSt,\otimes,\Spectra)$), as we now describe.

First of all, an \textit{ideal} of $\cR$ is a full presentable stable subcategory $\cI \subseteq \cR$ which is contagious under the $\cO$-monoidal structure, and a \textit{closed ideal} is a closed subcategory which is an ideal in a compatible way (\Cref{defn.closed.ideal.subcat}).  Closed ideals form a full subposet $\Idl_\cR \subseteq \Cls_\cR$, and an \bit{$\cO$-monoidal stratification} of $\cR$ is simply a stratification that factors through this subposet.

\begin{example}
\label{ex.intro.support.defines.closed.ideal}
For any closed subset $Z \in \Cls_X$, the corresponding closed subcategory $\QC_Z(X) \in \Cls_{\QC(X)}$ is a closed ideal subcategory.
\end{example}

We have the following macrocosm $\cO$-monoidal reconstruction theorem.

\begin{maintheorem}[\Cref{thm.s.m.reconstrn}]
\label{intro.thm.O.mon.reconstrn}
Let $\cO$ be an $\infty$-operad satisfying the conditions of \Cref{notation.for.operads.etc}\Cref{item.notation.for.operad} (e.g.\! $\EE_n$ for $1\leq n \leq \infty$), and suppose that $\cR$ is a presentably $\cO$-monoidal stable $\infty$-category equipped with an $\cO$-monoidal stratification
\[
\pos
\xlongra{\cI_\bullet}
\Idl_\cR
\]
over a poset $\pos$.
Then, the strata of the stratification inherit canonical $\cO$-monoidal structures, the gluing functors become canonically right-laxly $\cO$-monoidal, and these assemble into an $\cO$-monoidal gluing diagram $\GD^\otimes(\cR)$ that lifts the gluing diagram $\GD(\cR)$, in such a way that we have a canonical identification
\[ \begin{tikzcd}[row sep=0cm, column sep=1.5cm]
&[-1.8cm]
\Alg_\cO(\Cat)
\arrow{r}{\fgt}
&
\Cat
&[-1.7cm]
\\
&
\rotatebox{90}{$\in$}
&
\rotatebox{90}{$\in$}
\\
\limrlaxP ( \GD^\otimes(\cR) )
=:
&
\Glue^\otimes(\cR)
\arrow[maps to]{r}
&
\Glue(\cR)
&
:=
\limrlaxP ( \GD(\cR ) )
\end{tikzcd}~. \]
Moreover, the adjunction
\begin{equation}
\label{intro.thm.O.mon.reconstrn.underlying.adjn}
\begin{tikzcd}[column sep=1.5cm]
\cR
\arrow[transform canvas={yshift=0.9ex}]{r}{\gd}
\arrow[leftarrow, transform canvas={yshift=-0.9ex}]{r}[yshift=-0.2ex]{\bot}[swap]{\lim_{\sd(\pos)}}
&
\Glue(\cR)
\end{tikzcd}
\end{equation}
between $\infty$-categories of \Cref{intro.thm.cosms}\Cref{intro.main.thm.macrocosm} admits a canonical enhancement to an adjunction
\begin{equation}
\label{intro.thm.O.mon.reconstrn.O.monoidal.adjn}
\begin{tikzcd}[column sep=1.5cm]
\cR
\arrow[transform canvas={yshift=0.9ex}]{r}{\gd^\otimes}
\arrow[leftarrow, transform canvas={yshift=-0.9ex}]{r}[yshift=-0.2ex]{\bot}[swap]{\lim_{\sd(\pos)}^\otimes}
&
\Glue^\otimes(\cR)
\end{tikzcd}
\end{equation}
between $\cO$-monoidal $\infty$-categories, whose left adjoint is $\cO$-monoidal and whose right adjoint is right-laxly $\cO$-monoidal.
In particular, if the adjunction \Cref{intro.thm.O.mon.reconstrn.underlying.adjn} is an equivalence between $\infty$-categories (e.g.\! as guaranteed by \Cref{intro.thm.cosms} in the case that $\pos$ is down-finite), then the adjunction \Cref{intro.thm.O.mon.reconstrn.O.monoidal.adjn} is an equivalence between $\cO$-monoidal $\infty$-categories.
\end{maintheorem}

\begin{remark}
Given two $\cO$-monoidal $\infty$-categories, a \textit{right-laxly $\cO$-monoidal functor} between them is a functor between their underlying $\infty$-categories that preserves the $\cO$-monoidal structures up to certain (generally noninvertible) comparison morphisms.
For example, a right-laxly monoidal functor
\[
(\cC,\otimes_\cC,\uno_\cC)
\xlongra{F}
(\cD,\otimes_\cD,\uno_\cD)
\]
between monoidal $\infty$-categories involves the data of natural comparison morphisms
\[
\uno_\cD
\longra
F(\uno_\cC)
\qquad
\text{and}
\qquad
F(X) \otimes_\cD F(Y)
\longra
F(X \otimes_\cC Y)
~.
\]
This and related notions are reviewed in \Cref{subsection.O.monoidal.infty.cats}.
\end{remark}

\begin{remark}
Although we are confident in the existence of a metacosm $\cO$-monoidal reconstruction theorem, we state \Cref{intro.thm.O.mon.reconstrn} at the macrocosm level only.
\end{remark}

\begin{remark}
\label{rmk.fund.opns.for.O.mon.stratns}
It is immediate from \Cref{obs.Idl.in.Cls.stable.under.colims} that the fundamental operations described in \Cref{intro.thm.fund.opns} admit direct analogs for $\cO$-monoidal stratifications.
\end{remark}

\begin{remark}
\label{rmk.image.of.ideals.and.stratns}
Closed ideals in $\cR$ are equivalent data to \textit{central co/augmented idempotent objects} in $\cR$ (see \Cref{defn.idempotents.and.centrality} \and \Cref{prop.closed.ideal.subcats.are.central.idempotent}). It follows that a morphism $\cR \ra \cR'$ in $\Alg_\cO(\PrLSt)$ determines a functor
\begin{equation}
\label{functor.on.closed.ideals}
\Idl_\cR
\longra
\Idl_{\cR'}
~.
\end{equation}
Moreover, by \Cref{obs.yo.comm.for.s.m.stratns}, postcomposition with the functor \Cref{functor.on.closed.ideals} carries $\cO$-monoidal stratifications of $\cR$ to $\cO$-monoidal stratifications of $\cR'$.
\end{remark}

\begin{remark}
\label{rmk.stratn.of.module.cat}
Let $\cR$ be a presentably monoidal stable $\infty$-category equipped with a monoidal stratification
\[
\pos
\xlongra{\cI_\bullet}
\Idl_\cR
~.
\]
For any left $\cR$-module $\cM \in \LMod_\cR(\PrLSt)$, we immediately obtain a stratification
\[
\hspace{3cm}
\begin{tikzcd}[row sep=0cm]
\pos
\arrow{r}
&
\Cls_\cM
\\
\rotatebox{90}{$\in$}
&
\rotatebox{90}{$\in$}
\\
p
\arrow[maps to]{r}
&
\cI_p \otimes_\cR \cM
&[-1cm]
\simeq \LMod_{\uno_{\cI_p}}(\cM)
\end{tikzcd} \]
of $\cM$ over $\pos$ (using \Cref{notation.unit.object.for.I.D}).\footnote{This appears to be closely related to Elias--Hogancamp's theory of categorical diagonalization \cite{ElHog-diag}.}
\end{remark}

\begin{remark}
\label{rmk.table.of.analogies.between.stratns.and.filtrns}
Our work posits a system of analogies between classical algebra and categorified algebra, which is indicated in \Cref{table.of.categorification.analogies}.\footnote{Colimits categorify addition, $\cO$-monoidal structures categorify multiplication, and the distributivity of $\cO$-monoidal structures over colimits categorifies the distributivity of multiplication over addition.  The analogy between presentable stable $\infty$-categories and abelian groups is further evinced e.g.\! by the fact that given compact objects $X,Y \in \cX^\omega$, the sequence
\[
\brax{X}
\longhookra
\brax{X,Y}
\longrsurj
\brax{Y}
\]
(using \Cref{notn.clsd.subcat.gend.by.cpct.objs}) is exact if and only if $X$ and $Y$ are ``linearly independent'', i.e.\! $\ulhom_\cX(X,Y) \simeq 0$.  (A noncompact object of $\cX$ might be thought of as categorifying a nonconvergent infinite sum in an abelian group.)}
\begin{figure}[h]
\begin{tabular}{c||c}
classical algebra
&
categorified algebra
\\
\hline \hline
abelian group (or spectrum)
&
presentable stable $\infty$-category
\\
\hline
$\cO$-ring (spectrum)
&
presentably $\cO$-monoidal stable $\infty$-category
\\
\hline
filtration
&
stratification
\\
\hline
filtered pieces
&
$\{ \cZ_p \}_{p \in \pos}$
\\
\hline
associated graded pieces
&
$\{ \cX_p \}_{p \in \pos}$
\\
\hline
extension data
&
gluing diagram
\end{tabular}
\caption{This table lays out a system of analogies between classical algebra and categorified algebra.
\label{table.of.categorification.analogies}
}
\end{figure}
\end{remark}

\subsection{Adelic reconstruction}
\label{subsection.intro.adelic}

We now return to our scheme $X$. 
Let us write $\pos_X$ for the specialization poset of its underlying topological space: it has the same underlying set, and its relation is defined so that $x \leq y$ if and only if $x \in \ol{y}$.  Then, the closure functor
\[
\begin{tikzcd}[row sep=0cm]
\pos_X
\arrow{r}{\ol{(-)}}
&
\Cls_X
\\
\rotatebox{90}{$\in$}
&
\rotatebox{90}{$\in$}
\\
x
\arrow[maps to]{r}
&
\ol{x}
\end{tikzcd}
\]
defines a stratification of $X$.  Upgrading \Cref{ex.intro.induced.stratn.of.qcoh} via \Cref{ex.intro.support.defines.closed.ideal}, we obtain a symmetric monoidal stratification
\begin{equation}
\label{adelic.stratn.of.scheme}
\begin{tikzcd}[row sep=0cm, column sep=1.5cm]
\pos_X
\arrow{r}{\ol{(-)}}
&
\Cls_X
\arrow{r}{\QC_{(-)}(X)}
&
\Idl_{\QC(X)}
\\
\rotatebox{90}{$\in$}
&
&
\rotatebox{90}{$\in$}
\\
x
\arrow[maps to]{rr}
&
&
\QC_{\ol{x}}(X)
\end{tikzcd}
\end{equation}
of its underlying noncommutative stack $\QC(X)$, which we refer to as its \bit{adelic stratification}.  For each $x \in \pos_X$, the $x\th$ stratum of this symmetric monoidal stratification is
\[
\ker \left( \QC(X^\wedge_{\ol{x}}) \longra \prod_{y < x} \QC(X^\wedge_{\ol{y}}) \right)
~.\footnote{When the subset $(^< x) := (\ol{x} \backslash x ) \subseteq X$ is closed, the $x\th$ stratum may be identified more simply as $\QC( ( X \backslash (^< x) )^\wedge_x)$.}
\]

In general, the poset $\pos_X$ will not be down-finite, and so the adelic stratification of $\QC(X)$ is not guaranteed to converge.  However, writing $d := \dim(X)$ for the dimension of $X$,\footnote{\label{fn.why.fdim}Of course, it is here that we use that our scheme $X$ is finite-dimensional.} we may take the pushforward of the adelic stratification \Cref{adelic.stratn.of.scheme} along the \textit{dimension} functor
\[
\pos_X
\xra{\dim}
[d]
~;
\]
as $[d]$ is finite and hence down-finite, the pushforward symmetric monoidal stratification is guaranteed to converge.  Moreover, as the the fibers of the dimension functor are discrete, the strata of the pushforward symmetric monoidal stratification will simply be products of strata of the adelic stratification.  We illustrate this maneuver in the following fundamental example.

\begin{example}[the adelic stratification of $\ZZ$-modules]
\label{ex.intro.arithmetic}
Suppose that $X = \Spec(\ZZ)$.  The specialization poset of this affine scheme (which is the opposite of the poset of prime ideals of $\ZZ$) is given by
\begin{equation}
\label{speczn.poset.of.Mod.Z}
\pos_\ZZ
:=
\pos_{\Spec(\ZZ)}
=
\left(
\begin{tikzcd}
&
&
(0)
\\
(2)
\arrow{rru}
&
(3)
\arrow{ru}
&
(5)
\arrow{u}
&
\cdots
\arrow{lu}
\end{tikzcd}
\right)
~.
\end{equation}
Then, its adelic stratification
\begin{equation}
\label{adelic.stratification.of.Mod.Z}
\begin{tikzcd}[row sep=0cm]
\pos_\ZZ
\arrow{r}{\cI_\bullet}
&
\Idl_{\Mod_\ZZ}
\\
\rotatebox{90}{$\in$}
&
\rotatebox{90}{$\in$}
\\
\mf{p}
\arrow[maps to]{r}
&
\cI_\mf{p}
\end{tikzcd}
\end{equation}
is described by the formulas
\[
\cI_{(0)}
=
\Mod_\ZZ
\qquad
\text{and}
\qquad
\cI_{(p)}
=
\Mod_\ZZ^{(p)\textup{-torsion}}
~,
\]
i.e.\! it selects the diagram
\[
\begin{tikzcd}
&
&
\Mod_\ZZ
\arrow[hookleftarrow]{d}
\arrow[hookleftarrow]{rd}
\\
\Mod_\ZZ^{(2)\textup{-torsion}}
\arrow[hook]{rru}
&
\Mod_\ZZ^{(3)\textup{-torsion}}
\arrow[hook]{ru}
&
\Mod_\ZZ^{(5)\textup{-torsion}}
&
\cdots
\end{tikzcd}
\]
of closed ideal subcategories of $\Mod_\ZZ$.

We now identify the strata and geometric localization adjunctions, as follows. First of all, recalling \Cref{ex.of.closed.open.decomp.giving.recollement.in.intro}, we identify the $(p)\th$ strata and geometric localization adjunctions as 
\[
\begin{tikzcd}[column sep=2cm, row sep=2cm, ampersand replacement=\&]
\&[-3.2cm]
\cI_{(p)}
:=
\&[-2.3cm]
\Mod_\ZZ^{(p)\textup{-torsion}}
\arrow[hook, transform canvas={yshift=0.9ex}]{r}
\arrow[leftarrow, transform canvas={yshift=-0.9ex}]{r}[yshift=-0.2ex]{\bot}
\arrow[leftrightarrow]{d}[sloped, anchor=north]{\sim}
\&
\Mod_\ZZ
\&[-2.2cm]
=:
\cR
\\
\cR_{(p)}
:=
\&
\&
\Mod_{\ZZ^\wedge_p}^{(p)\textup{-complete}}
\arrow[leftarrow, transform canvas={xshift=-0.6ex, yshift=0.7ex}]{ru}[sloped]{\Phi_{(p)} = \ZZ^\wedge_p \widehat{\otimes}_\ZZ (-) }
\arrow[hook, transform canvas={xshift=0.6ex, yshift=-0.7ex}]{ur}[sloped, yshift=-0.2ex]{\bot}[sloped, swap]{\rho^{(p)} = \fgt}
\end{tikzcd}
~.\footnote{We distinguish between the equivalent $\infty$-categories $\cI_{(p)}$ and $\cR_{(p)}$ according to their inclusions into $\cR$.}
\]
Then, we identify the $(0)\th$ stratum as
\[
\cR_{(0)}
:=
\left( \cI_{(0)} \left/ \underset{p \textup{ prime}}{\dcup} \cI_{(p)} \right. \right)
:=
\left( \Mod_\ZZ \left/ \underset{p \textup{ prime}}{\dcup} \Mod_\ZZ^{(p)\textup{-torsion}} \right. \right)
\simeq
\Mod_\QQ
~,
\]
and the $(0)\th$ geometric localization adjunction as
\[ \begin{tikzcd}[column sep=3cm, ampersand replacement=\&]
\cR
:=
\&[-3.2cm]
\Mod_\ZZ
\arrow[transform canvas={yshift=0.9ex}]{r}{\Phi_{(0)} = \QQ \otimes_\ZZ (-)}
\arrow[hookleftarrow, transform canvas={yshift=-0.9ex}]{r}[yshift=-0.2ex]{\bot}[swap]{\rho^{(0)} = \fgt}
\&
\Mod_\QQ
\&[-3.2cm]
\simeq
\cR_{(0)}
\end{tikzcd}
~.
\]

From here, we see that the symmetric monoidal gluing diagram $\GD^\otimes(\Mod_\ZZ)$ of the adelic stratification \Cref{adelic.stratification.of.Mod.Z} is the diagram
\[
\begin{tikzcd}
&
&
\Mod_\QQ
\arrow[leftarrow]{d}
\arrow[leftarrow]{rd}
\\
\Mod_{\ZZ^\wedge_2}^{(2)\textup{-complete}}
\arrow{rru}
&
\Mod_{\ZZ^\wedge_3}^{(3)\textup{-complete}}
\arrow{ru}
&
\Mod_{\ZZ^\wedge_5}^{(5)\textup{-complete}}
&
\cdots
\end{tikzcd}
\]
of presentably symmetric monoidal stable $\infty$-categories, in which all gluing functors are given by rationalization.\footnote{The poset $\pos_\ZZ$ has no nondegenerate composite morphisms, and so the functor $\GD^\otimes(\Mod_\ZZ)$ is in fact a strict (instead of left-lax) functor.} We may now identify the reglued symmetric monoidal $\infty$-category
\[
\Glue^\otimes(\Mod_\ZZ)
:=
\lim^\rlax_{\pos_\ZZ}(\GD^\otimes(\Mod_\ZZ))
\]
as consisting of tuples of data
\begin{equation}
\label{tuples.in.glued.s.m.cat.for.adelic.stratn.of.Mod.Z}
\left(
~
M_0 \in \Mod_\QQ
~,~
\left(
~
M_p \in \Mod_{\ZZ^\wedge_p}^{(p)\textup{-complete}}
~,~
\begin{tikzcd}
M_0
\arrow{d}
\\
\QQ \otimes_\ZZ M_p
\end{tikzcd}
~
\right)_{p \textup{ prime}}
~
\right)
~,
\end{equation}
equipped with the componentwise symmetric monoidal structure.  This brings us to the symmetric monoidal macrocosm adjunction
\begin{equation}
\label{s.m.macrocosm.adjn.for.Mod.Z}
\begin{tikzcd}[column sep=1.5cm]
\Mod_\ZZ
\arrow[transform canvas={yshift=0.9ex}]{r}{\gd^\otimes}
\arrow[leftarrow, transform canvas={yshift=-0.9ex}]{r}[yshift=-0.2ex]{\bot}[swap]{\lim_{\sd(\pos_\ZZ)}^\otimes}
&
\Glue^\otimes(\Mod_\ZZ)
\end{tikzcd}~,
\end{equation}
whose left adjoint $\gd^\otimes$ takes $M \in \Mod_\ZZ$ to the evident tuple \Cref{tuples.in.glued.s.m.cat.for.adelic.stratn.of.Mod.Z} in which $M_0 := \QQ \otimes_\ZZ M$ and $M_p := M^\wedge_p := \ZZ^\wedge_p \widehat{\otimes}_\ZZ M$ and 
whose right adjoint takes the tuple \Cref{tuples.in.glued.s.m.cat.for.adelic.stratn.of.Mod.Z} to the evident object
\begin{equation}
\label{limit.over.sd.P.Z}
\lim
\left(
\begin{tikzcd}[column sep=0.5cm]
&
&
&
&
M_0
\arrow{lld}
\arrow{ld}
\arrow{d}
\arrow{rd}
\\
&
&
\QQ \otimes_\ZZ M_2
&
\QQ \otimes_\ZZ M_3
&
\QQ \otimes_\ZZ M_5
&
\cdots
\\
M_2
\arrow{rru}
&
&
M_3
\arrow{ru}
&
&
M_5
\arrow{u}
&
&
\cdots
\arrow{lu}
\end{tikzcd}
\right)
\in
\Mod_\ZZ
~.\footnote{The right adjoint $\lim_{\sd(\pos_\ZZ)}^\otimes$ of the symmetric monoidal macrocosm adjunction \Cref{s.m.macrocosm.adjn.for.Mod.Z} is only right-laxly (instead of strictly) symmetric monoidal, as the functors $\rho^{(p)}$ are only right-laxly symmetric monoidal.}
\end{equation}

We can now witness the failure of convergence of the adelic stratification \Cref{adelic.stratification.of.Mod.Z}.  Reorganizing the limit \Cref{limit.over.sd.P.Z} as the pullback
\[
\lim
\left(
\begin{tikzcd}
&
M_0
\arrow{d}
\\
{\displaystyle \prod_{p \textup{ prime}} M_p}
\arrow{r}
&
{\displaystyle \prod_{p \textup{ prime}} (\QQ \otimes_\ZZ M_p)}
\end{tikzcd}
\right)
\in
\Mod_\ZZ
~,
\]
we find that the unit of the adjunction \Cref{s.m.macrocosm.adjn.for.Mod.Z} at an object $M \in \Mod_\ZZ$ is a morphism
\begin{equation}
\label{unit.morphism.at.a.Z.module.for.s.m.macrocosm.adjn}
M
\longra
\lim
\left(
\begin{tikzcd}
&
\QQ \otimes_\ZZ M
\arrow{d}
\\
{\displaystyle \prod_{p \textup{ prime}} M^\wedge_p}
\arrow{r}
&
{\displaystyle \prod_{p \textup{ prime}} (\QQ \otimes_\ZZ M^\wedge_p)}
\end{tikzcd}
\right)
~.
\end{equation}
Recalling the equivalence
\[
M
\xlongra{\sim}
\lim
\left(
\begin{tikzcd}
&
\QQ \otimes_\ZZ M
\arrow{d}
\\
{\displaystyle \prod_{p \textup{ prime}} M^\wedge_p}
\arrow{r}
&
{\displaystyle \QQ \otimes_\ZZ \left( \prod_{p \textup{ prime}}  M^\wedge_p \right) }
\end{tikzcd}
\right)
\]
resulting from the arithmetic fracture square \Cref{intro.arithmetic.fracture.square}, we see that the unit morphism \Cref{unit.morphism.at.a.Z.module.for.s.m.macrocosm.adjn} is not generally an equivalence, because the rationalization functor $\QQ \otimes_\ZZ (-)$ does not commute with infinite products.\footnote{More precisely, the morphism \Cref{unit.morphism.at.a.Z.module.for.s.m.macrocosm.adjn} between pullbacks arises from a natural transformation between cospans which is an equivalence on the two source terms and is induced by the universal property of the product on the common target term.}  For instance, consider the abelian group
\[
M
:=
\bigoplus_{p \textup{ prime}} \ZZ/p
~:
\]
for each prime number $p$ we have $M^\wedge_p \simeq \ZZ/p$, and the morphism
\[
\QQ \otimes_\ZZ
\left( \prod_{p \textup{ prime}} \ZZ/p \right)
\longra
\prod_{p \textup{ prime}}
(\QQ \otimes_\ZZ \ZZ/p)
\simeq
0
\]
is not an equivalence.  Note that this failure of convergence does not contradict \Cref{intro.thm.cosms}, as the poset $\pos_\ZZ$ is not down-finite (because the closure of the generic point $(0) \in \Spec(\ZZ)$ is infinite).

In order to rectify this failure of convergence, we apply \Cref{intro.thm.fund.opns}: more precisely, we take the pushforward of the adelic stratification along the dimension functor
\[
\pos_\ZZ
\xra[(p) \mapsto 0]{(0) \mapsto 1}
[1]
~.
\]
Recalling \Cref{rmk.fund.opns.for.O.mon.stratns}, we see that this yields a symmetric monoidal stratification of $\Mod_\ZZ$ over $[1]$,\footnote{Note that this stratification of $\Mod_\ZZ \simeq \QC(\Spec(\ZZ))$ does not arise from a stratification of $\Spec(\ZZ)$, as the subset $(\Spec(\ZZ) \backslash \{ (0) \}) \subseteq \Spec(\ZZ)$ is not closed.} which determines a symmetric monoidal recollement
\begin{equation}
\label{s.m.recollement.of.Mod.Z}
\begin{tikzcd}[column sep=2cm, row sep=2cm, ampersand replacement=\&]
{\displaystyle \prod_{p \textup{ prime}} \Mod_\ZZ^{(p)\textup{-torsion}}}
\arrow[hook, transform canvas={yshift=0.9ex}]{r}
\arrow[leftarrow, transform canvas={yshift=-0.9ex}]{r}[yshift=-0.2ex]{\bot}
\arrow[leftrightarrow]{d}[sloped, anchor=north]{\sim}
\&
\Mod_\ZZ
\arrow[bend left=30]{r}{\Phi_1 = \Phi_{(0)}}
\arrow[hookleftarrow]{r}[transform canvas={yshift=0.15cm}]{\bot}[swap,transform canvas={yshift=-0.15cm}]{\bot}[description]{\rho^1 = \rho^{(0)}}
\arrow[bend right=30]{r}
\&
\Mod_\QQ
\\
{\displaystyle \prod_{p \textup{ prime}} \Mod_{\ZZ^\wedge_p}^{(p)\textup{-complete}}}
\arrow[leftarrow, transform canvas={xshift=-0.6ex, yshift=0.7ex}]{ru}[sloped]{\Phi_0 = \left( \Phi_{(p)} \right)_{p \textup{ prime}}  }
\arrow[hook, transform canvas={xshift=0.6ex, yshift=-0.7ex}]{ur}[sloped, yshift=-0.2ex]{\bot}[sloped, swap]{\rho^0 = \prod_{p \textup{ prime}} \rho^{(p)} }
\end{tikzcd}
\end{equation}
(in the sense that for $i \in [1]$ the left adjoints $\Phi_i$ are symmetric monoidal and their right adjoints $\rho^i$ are right-laxly symmetric monoidal).  Combining Theorems \ref{intro.thm.cosms} \and \ref{intro.thm.O.mon.reconstrn}, we obtain a macrocosm equivalence
\begin{equation}
\label{adjoint.equivalence.giving.arithmetic.square}
\begin{tikzcd}[column sep=1.5cm]
\Mod_\ZZ
\arrow[transform canvas={yshift=0.9ex}]{r}{\gd^\otimes}
\arrow[leftarrow, transform canvas={yshift=-0.9ex}]{r}[yshift=-0.0ex]{\sim}[swap]{\lim^\otimes_{\sd([1])}}
&
\lim^\rlax \left(
{\displaystyle \prod_{p \textup{ prime}} \Mod_{\ZZ^\wedge_p}^{(p)\textup{-complete}}}
\xra{\Phi_1 \rho^0}
\Mod_\QQ
\right)
\end{tikzcd}
\end{equation}
between presentably symmetric monoidal stable $\infty$-categories.  For each $M \in \Mod_\ZZ$, the unit of the adjoint equivalence \Cref{adjoint.equivalence.giving.arithmetic.square} recovers the microcosm equivalence
\[
M
\xlongra{\sim}
\lim \left(
\begin{tikzcd}
&
\QQ \otimes_\ZZ M
\arrow{d}
\\
{\displaystyle \prod_{p \textup{ prime}} M^\wedge_p}
\arrow{r}
&
{\displaystyle \QQ \otimes_\ZZ \left( \prod_{p \textup{ prime}}  M^\wedge_p \right) }
\end{tikzcd} \right)
~,
\]
i.e.\! the arithmetic fracture square \Cref{intro.arithmetic.fracture.square}.
\end{example}

We generalize the preceding discussion to the setting of tensor-triangular geometry as follows.

\begin{maintheorem}[\Cref{thm.s.m.stratn.over.balmer.spectrum}]
\label{intro.thm.balmer}
Let $\cR$ be a presentably symmetric monoidal stable $\infty$-category, and assume that $\cR$ is rigidly-compactly generated (\Cref{defn.rigidly.cpctly.gend}).  Then, there is a canonical functor
\begin{equation}
\label{adelic.stratn.intro}
\pos_\cR
\longra
\Idl_\cR
\end{equation}
from the specialization poset $\pos_\cR$ of $\Spec(\cR^\omega)$ (i.e.\! the poset of thick prime ideal subcategories of $\cR^\omega$ ordered by inclusion), which is defined in terms of supports.  The functor \Cref{adelic.stratn.intro} satisfies the stratification condition.  So, it defines a symmetric monoidal stratification assuming that it also satisfies the generation condition.
\end{maintheorem}

\noindent We refer to such a symmetric monoidal stratification \Cref{adelic.stratn.intro} as the \bit{adelic stratification} of $\cR$.  We unpack the adelic stratification of $\cR = \Spectra$ as \Cref{ex.adelic.stratn.of.spectra}.

\begin{remark}
The functor \Cref{adelic.stratn.intro} automatically satisfies the generation condition (and so defines a symmetric monoidal stratification) whenever the topological space $\Spec(\cR^\omega)$ has finitely many irreducible components.\footnote{See \Cref{rmk.counterex.to.generation.condition.for.adelic} for an example where it fails.} This holds for example in the case that $\Spec(\cR^\omega)$ is noetherian, which also implies that its specialization poset is down-finite.
\end{remark}


\begin{remark}
Adelic stratifications bring an exciting perspective to tensor-triangular geometry, which seems worthy of further investigation; this is discussed further in \Cref{rmk.stratns.helps.tt.geometry}.
\end{remark}

\subsection{The geometric stratification of genuine $G$-spectra}
\label{subsection.intro.eq.spt}

Let $G$ be a compact Lie group.  As a matter of notation and perspective, we write $\BBGG$ for the noncommutative stack whose quasicoherent sheaves are genuine $G$-spectra:
\[ \QC(\BBGG) := \Spectra^{\gen G}~. \]
We also introduce the following notation.
\begin{itemize}

\item We write $\pos_G$ for the poset of conjugacy classes of closed subgroups of $G$ (ordered by subconjugacy).

\item For any element $H \in \pos_G$, we write $\Weyl(H)  := \Normzer(H)/H$ for its Weyl group (the quotient by it of its normalizer in $G$).\footnote{More invariantly, one can also describe $\Weyl(H)$ as the compact Lie group of $G$-equivariant automorphisms of $G/H$.}

\item We write $\Spectra^{\htpy G} := \Fun(\BG,\Spectra)$ for the $\infty$-category of homotopy $G$-spectra.\footnote{In addition to nicely paralleling the notation $\Spectra^{\gen G}$, the notation $\Spectra^{\htpy G}$ is consistent: this is the homotopy fixedpoints of the trivial $G$-action on the $\infty$-category $\Spectra$.}
\end{itemize}

\begin{maintheorem}[\Cref{thm.geom.stratn.of.SpgG}]
\label{intro.thm.gen.G.spt}
The noncommutative stack $\BBGG$ admits a canonical symmetric monoidal stratification over $\pos_G$, with the following features.
\begin{enumerate}
\item Its stratum corresponding to an element $H \in \pos_G$ is the commutative stack $\sB \Weyl(H)$ (i.e.\! the presentable stable $\infty$-category $\Spectra^{\htpy \Weyl(H)} \simeq \QC(\sB \Weyl(H))$ of homotopy $\Weyl(H)$-spectra).
\item The geometric localization functors are given by geometric fixedpoints:
\[
\QC(\BBGG)
:=
\Spectra^{\gen G}
\xra{\Phi^H}
\Spectra^{\htpy \Weyl(H)}
\simeq
\QC(\sB \Weyl(H)) ~.\footnote{Our notation $\Phi_p$ for the $p\th$ geometric localization functor, and indeed the terminology itself, are motivated by the example of geometric fixedpoints. However, we use the notation $\Phi^H$ instead of $\Phi_H$ in order to adhere to standard conventions in equivariant homotopy theory.} \]
\item For any morphism $H \ra K$ in $\pos_G$, the associated gluing functor
\[
\QC(\sB \Weyl(H))
\simeq
\Spectra^{\htpy \Weyl(H)}
\xra{\Gamma^H_K}
\Spectra^{\htpy \Weyl(K)}
\simeq
\QC(\sB \Weyl(K))
\]
is given by a version of the Tate construction.\footnote{In the case that $G$ is abelian, the gluing functor associated to a morphism $H \ra K$ in $\pos_G$ is the \textit{proper} Tate construction
\[ \Spectra^{\htpy (G/H)} \xra{(-)^{\tate (K/H)}} \Spectra^{\htpy (G/K)} ~, \]
which quotients by norms from all proper subgroups (rather than just the trivial subgroup, as in the usual Tate construction).  When $G$ is not abelian, the corresponding description of the gluing functors is slightly more elaborate (see \Cref{rmk.nonabelian.monodromy}).}
\end{enumerate}
\end{maintheorem}

\noindent In order to emphasize its relationship with the geometric fixedpoints functors, we refer to the symmetric monoidal stratification of \Cref{intro.thm.gen.G.spt} as the \bit{geometric stratification} of $\Spectra^{\gen G}$.

\begin{remark}
\label{remark.stratn.of.SpgG.prob.not.convergent}
In the case that the poset $\pos_G$ is down-finite, it follows from Theorems \ref{intro.thm.gen.G.spt} \and \ref{intro.thm.cosms} that a genuine $G$-spectrum
\[ E \in \Spectra^{\gen G} \]
is equivalent data to its geometric fixedpoints spectra
\[ \left\{ \Phi^H(E) \in \Spectra^{\htpy \Weyl(H)} \right\}_{H \in \pos_G} \]
(considered as homotopy $\Weyl(H)$-spectra) along with gluing data among these; \Cref{intro.thm.O.mon.reconstrn} guarantees that this equivalence is moreover compatible with symmetric monoidal structures.

However, the poset $\pos_G$ is down-finite if and only if the compact Lie group $G$ is in fact a finite group.  We do not know whether the geometric stratification of $\Spectra^{\gen G}$ is convergent in the case that $G$ is positive-dimensional, but we see no reason to expect it to be so.\footnote{On the other hand, the poset $\pos_G$ is always artinian; applying \Cref{rmk.artinian.conservativity} to the geometric stratification of $\Spectra^{\gen G}$ recovers the ``geometric fixedpoints Whitehead theorem'' \cite[\S 1.6]{Greenlees-thesis}.}  In any case, its pushforward to any down-finite poset produces a symmetric monoidal reconstruction theorem for genuine $G$-spectra.  For instance, writing $d := \dim(G)$ we may take its pushforward along the \textit{dimension} functor
\[
\pos_G
\xra{\dim}
[d]
~;
\]
we note that its fibers are down-finite, so in principle this may lead to a fuller understanding of $\Spectra^{\gen G}$ in the case that $G$ is positive-dimensional.

Another symmetric monoidal reconstruction theorem resulting from Theorems \ref{intro.thm.gen.G.spt}, \ref{intro.thm.cosms}, \and \ref{intro.thm.O.mon.reconstrn} (\and \ref{intro.thm.fund.opns}\ref{intro.main.thm.operations.restriction}) is unpacked as \Cref{stratn.of.Sp.g.T.and.Sp.g.proper.T}: writing $\TT$ for the circle group, the geometric stratification of the noncommutative stack $\Spectra^{\gen \TT}$ of genuine $\TT$-spectra over the poset $\pos_\TT \cong (\Ndiv)^\rcone$ (which is not down-finite) restricts to a symmetric monoidal stratification of the noncommutative stack $\Spectra^{\gen^\proper \TT}$ of \textit{proper}-genuine $\TT$-spectra over the poset $\Ndiv$ (which is down-finite).  We use the resulting symmetric monoidal reconstruction theorem to study cyclotomic spectra (and their symmetric monoidal structure) in \cite{AMR-cyclo}.
\end{remark}

\begin{remark}
\label{rmk.SpgG.nearly.commutative}
As indicated by our formulation of \Cref{intro.thm.gen.G.spt}, we view it as providing a sense in which $\BBGG$ is a ``nearly commutative'' stack.\footnote{As a nice coincidence, this also gives a second meaning to the terminology ``geometric stratification''.}  Indeed, its strata are commutative stacks and its gluing functors are right-laxly symmetric monoidal, just as would be the case for a stratified commutative stack.  However, its gluing functors do not appear to be of commutative origin.  This is already apparent in the simplest nontrivial case, where $G = \Cyclic_p$ is the cyclic group of order $p$. In this situation, the geometric stratification of $\Spectra^{\gen \Cyclic_p}$ amounts to a symmetric monoidal recollement, whose gluing functor is the Tate construction
\begin{equation}
\label{Cp.tate.constrn.in.intro.when.remarking.nc.gluing.data}
\Spectra^{\htpy \Cyclic_p}
\xra{(-)^{{\sf t} \Cyclic_p}}
\Spectra
\end{equation}
(as is unpacked further in \Cref{example.genuine.Cp.spectra}), and there does not appear to be a natural example of a commutative (spectral) stack $X$ equipped with a closed-open decomposition \Cref{closed.open.decomp.of.scheme.in.intro} such that $\QC(X^\wedge_Z) \simeq \Spectra^{\htpy \Cyclic_p}$, $\QC(U) \simeq \Spectra$, and the gluing functor
\[
\QC(X^\wedge_Z)
\xra{j^* \ihat_*}
\QC(U)
\]
coincides with the Tate construction \Cref{Cp.tate.constrn.in.intro.when.remarking.nc.gluing.data}.
\end{remark}

\subsection{Stratified topological spaces and constructible sheaves}
\label{subsection.cbl}


We have discussed how the general theory of stratifications applies in the context of quasicoherent sheaves over a scheme (recall \Cref{ex.intro.induced.stratn.of.qcoh}).  In fact, it applies in other sheaf-theoretic contexts as well, as we now explain.

Let $T$ be a topological space, and suppose that
\[ \begin{tikzcd}[column sep=1.5cm]
U
\arrow[hook]{r}[swap]{\sf open}{j}
&
T
\arrow[hookleftarrow]{r}{i}[swap]{\sf closed}
&
Z
\end{tikzcd} \]
is a closed-open decomposition of $T$ (note that the placement is reversed from that of \Cref{ex.of.closed.open.decomp.giving.recollement.in.intro}).  Then, we obtain a recollement
\begin{equation}
\label{recollement.for.sheaves.on.topological.spaces}
\begin{tikzcd}[column sep=1.5cm]
\Shv(U)
\arrow[hook, bend left=45]{r}[description]{j_!}
\arrow[leftarrow]{r}[transform canvas={yshift=0.2cm}]{\bot}[swap,transform canvas={yshift=-0.2cm}]{\bot}[description]{j^! = j^*}
\arrow[bend right=45, hook]{r}[description]{j_*}
&
\Shv(T)
\arrow[bend left=45]{r}[description]{i^*}
\arrow[hookleftarrow]{r}[transform canvas={yshift=0.2cm}]{\bot}[swap,transform canvas={yshift=-0.2cm}]{\bot}[description]{i_* = i_!}
\arrow[bend right=45]{r}[description]{i^!}
&
\Shv(Z)
\end{tikzcd}
\end{equation}
among presentable stable $\infty$-categories of sheaves valued in a presentable stable $\infty$-category (which we omit from our notation).

This may be upgraded as follows.  A \bit{stratification} of $T$ over the poset $\pos$ is a continuous function
\begin{equation}
\label{stratn.of.top.spc}
T
\xlongra{f}
\pos
~,
\end{equation}
where we consider $\pos$ as a topological space via the poset topology on its underlying set (in which the closed subsets are precisely the down-closed subsets).  This determines a functor
\[
\begin{tikzcd}[row sep=0cm]
\pos^\op
\arrow{r}{U_\bullet}
&
\Open_T
\\
\rotatebox{90}{$\in$}
&
\rotatebox{90}{$\in$}
\\
p^\circ
\arrow[maps to]{r}
&
U_p
&[-1.4cm]
:= f^{-1}(^\geq p)
\end{tikzcd}
\]
that satisfies the evident analog of \Cref{defn.intro.comm.stratn}: we have $T = \bigcup_{p \in \pos} U_p$, and for any $p,q \in \pos$ we have $U_p \cap U_q = \bigcup_{p \leq r \textup{ and } q \leq r} U_r$.  From this we obtain a stratification of $\Shv(T)$ over $\pos^\op$, namely the composite
\begin{equation}
\label{stratn.of.shvs.on.a.top.spc}
\begin{tikzcd}[row sep=0cm, column sep=1.5cm]
\pos^\op
\arrow{r}{U_\bullet}
&
\Open_T
\arrow{r}{\Shv}
&
\Cls_{\Shv(T)}
\\
\rotatebox{90}{$\in$}
&
&
\rotatebox{90}{$\in$}
\\
p^\circ
\arrow[maps to]{rr}
&
&
\Shv(U_p)
\end{tikzcd}~.
\end{equation}
For each $p \in \pos$, let us write
\[
T_p := f^{-1}(p)
\xlonghookra{\sigma_p}
T
\]
for the inclusion of the $p\th$ stratum of the stratification \Cref{stratn.of.top.spc} (a locally closed subset).  Then, the $(p^\circ)\th$ stratum of the stratification \Cref{stratn.of.shvs.on.a.top.spc} is $\Shv(T_p)$, and its gluing functor with respect to a morphism $p^\circ \ra q^\circ$ in $\pos^\op$ is the composite
\[
\Shv(T_p)
\xhookra{(\sigma_p)_*}
\Shv(T)
\xra{(\sigma_q)^*}
\Shv(T_q)
~.
\]

Analogous stratifications exist for constructible sheaves. More precisely, the stratification \Cref{stratn.of.shvs.on.a.top.spc} restricts to stratifications of the full subcategories
\[
\Shv^{\pos\textup{-}\cbl}(T)
\subseteq
\Shv^\cbl(T)
\subseteq
\Shv(T)
\]
of $\pos$-constructible sheaves and of constructible sheaves. More generally, for any functor $\posQ \ra \pos$ among posets and any refinement
\[ \begin{tikzcd}[row sep=0cm]
&
\posQ
\arrow{dd}
\\
T
\arrow[dashed]{ru}
\arrow{rd}[swap, sloped]{f}
\\
&
\pos
\end{tikzcd} \]
of the stratification \Cref{stratn.of.top.spc}, the stratification \Cref{stratn.of.shvs.on.a.top.spc} restricts to a $\pos^\op$-stratification of the full subcategory
\[
\Shv^{\posQ\textup{-}\cbl}(T) \subseteq \Shv^\cbl(T)
\]
of $\posQ$-constructible sheaves.\footnote{Alternatively, this stratification may be obtained by taking the pushforward (in the sense of \Cref{intro.thm.fund.opns}) of the $\posQ^\op$-stratification of $\Shv^{\posQ\textup{-}\cbl}(T)$ along the functor $\posQ^\op \ra \pos^\op$.}

\begin{remark}
\label{rmk.easy.spectral.sequences.for.coh.of.strat.spcs}
Assume that $\pos^\op$ is down-finite, and fix a conservative functor $\pos \xra{d} \ZZ$ (e.g.\! the dimension function of strata). Choose any sheaf $\cF \in \Shv(T)$, and fix an exact functor $\Shv(T) \xra{H} \cV$ where $\cV$ has a t-structure (e.g.\! cohomology or cohomology with compact support).

\begin{enumerate}

\item

The stratification \Cref{stratn.of.shvs.on.a.top.spc} determines spectral sequences for the cohomology of $\cF$ in terms of its cohomologies over strata. Indeed, by \Cref{rmk.spectral.sequences.from.stratns}, we obtain spectral sequences
\[
E^1_{s,t}
=
\bigoplus_{r \in d^{-1}(-s)}
\pi_{s+t} ( H ( ( \sigma_r)_! (\sigma_r)^*(\cF) ) )
\Longrightarrow
\pi_{s+t} ( H ( \cF ) )
\]
and
\[
E^1_{s,t}
=
\bigoplus_{r \in d^{-1}(s)}
\pi_{s+t} ( H ( ( \sigma_r)_* (\sigma_r)^!(\cF) ) )
\Longrightarrow
\pi_{s+t} ( H ( \cF ) )
~.\footnote{Note that these two spectral sequences are related by Verdier duality (see \Cref{ex.reflection.and.Verdier}).}
\]

\item\label{item.describe.filtrations.for.sheaves.on.a.top.spc}

Let us describe the four filtrations of $\id_{\Shv(T)}$ that arise from applying \Cref{rmk.filtrations.from.stratns} to the stratification \Cref{stratn.of.shvs.on.a.top.spc}. For each $p \in \pos$, let us denote by
\[ \begin{tikzcd}[column sep=1.5cm]
U_p
:=
f^{-1}(^\geq p)
\arrow[hook]{r}[swap]{\sf open}{j_p}
&
T
\arrow[hookleftarrow]{r}{i_p}[swap]{\sf closed}
&
f^{-1}(^\leq p)
=
\ol{T_p}
=:
Z_p
\end{tikzcd} \]
the corresponding open and closed subsets (note that these are \textit{not} generally complements). Then, for any $p \in \pos$ we have
\[
\fil^L_p
\simeq
(j_p)_! (j_p)^*
~,
\qquad
\fil_R^p
\simeq
(j_p)_* (j_p)^*
~,
\qquad
\fil_L^p
\simeq
(i_p)_* (i_p)^*
~,
\qquad
\text{and}
\qquad
\fil^R_p
\simeq
(i_p)_* (i_p)^!
~.
\]

\item Using part \Cref{item.describe.filtrations.for.sheaves.on.a.top.spc}, we describe the spectral sequences obtained by applying the spectral sequence \Cref{Mobius.spectral.seq.for.Fun.P.V} discussed in \Cref{rmk.spectral.sequence.from.mobius.inversion}, which arises from categorified M\"{o}bius inversion (\Cref{ex.categorified.mobius.inversion}). Applied to the filtration $\fil_L^\bullet$, we obtain a spectral sequence
\[
E^1_{s,t}
=
\bigoplus_{r \in d^{-1}(-s) \cap (^\leq p)}
\pi_{s+t} ( M^r_p \tensoring H( ( i_r )_* (i_r)^* ( \cF ) ) )
\Longrightarrow
\pi_{s+t} ( H ( (\sigma_p)_! (\sigma_p)^* ( \cF ) ) )
~;
\]
taking $H$ to be compactly-supported cohomology, we recover the spectral sequence of \cite[Theorem 1.1]{Petersen-ss}, which computes compactly-supported cohomology over the $p\th$ stratum in terms of those over closures of strata. Next, applied to the filtration $\fil^L_\bullet$, we obtain a spectral sequence
\[
E^1_{s,t}
=
\bigoplus_{r \in d^{-1}(-s) \cap (^\leq p)}
\pi_{s+t}( M^r_p \tensoring H( (j_r)_! (j_r)^* ( \cF) ) )
\Longrightarrow
\pi_{s+t} ( H ( (\sigma_p)_! (\sigma_p)^* ( \cF ) ) )
\]
with the same abutment but different $E^1$ page. Finally, applied to the filtrations $\fil_R^\bullet$ and $\fil^R_\bullet$, we obtain spectral sequences that are Verdier dual to these two (again see \Cref{ex.reflection.and.Verdier}).

\end{enumerate}
\end{remark}




\begin{remark}
Under suitable hypotheses, the $\pos^\op$-stratification of $\Shv^{\pos\textup{-}\cbl}(T)$ admits a completely algebraic description; see \Cref{ex.tamely.conical.stratd.top.space}.
\end{remark}



\begin{remark}
If we consider sheaves valued in a presentably $\cO$-monoidal stable $\infty$-category, the stratification \Cref{stratn.of.shvs.on.a.top.spc} becomes an $\cO$-monoidal stratification.\footnote{In particular, the stratification \Cref{stratn.of.shvs.on.a.top.spc} for an arbitrary target is recovered from the case of $\Spectra$ through \Cref{rmk.stratn.of.module.cat}.}
\end{remark}

\begin{remark}
The stratification \Cref{stratn.of.shvs.on.a.top.spc} of sheaves on a topological space generalizes to a stratification of sheaves on an $\infty$-topos, using the theory of stratified $\infty$-topoi developed by Barwick--Glasman--Haine \cite{BGH-exodromy}.\footnote{In the case of a presheaf $\infty$-topos, this may also be recovered as an instance of the stratification \Cref{stratn.of.Fun.T.V} below.}
\end{remark}



\subsection{Functors to a poset and naive $G$-spectra}
\label{subsection.naive.G.spectra.stratn}

Let $G$ be a compact Lie group. The $\infty$-category of genuine $G$-spectra admits a variant, the $\infty$-category
\[
\Spectra^{\naive G}
:=
\Fun(\Orb_G^\op,\Spectra)
\]
of \bit{naive $G$-spectra}, i.e.\! of spectral presheaves on the orbit $\infty$-category of $G$.\footnote{The terminology ``naive'' stems from the fact that, whereas the $\infty$-category
$\Spectra^{\gen G}$ of genuine $G$-spectra is obtained from the $\infty$-category $\Spaces^{\gen G}_*$ of pointed genuine $G$-spaces by inverting all representation spheres under the smash product, the $\infty$-category $\Spectra^{\naive G}$ of naive $G$-spectra is obtained from $\Spaces^{\gen G}_*$ by inverting merely the spheres with trivial $G$ action under the smash product (i.e.\! by stabilizing).} Naive $G$-spectra provide a natural context for computing (generalized) Bredon co/homology, as well as for understanding genuine $G$-suspension spectra (see e.g.\! \cite[\S \ref{mackey:section.fxn.reps.to.Pic}]{AMR-mackey}) via the factorization
\[ \begin{tikzcd}
\Spaces^{\gen G}_*
\arrow{rr}{\Sigma^\infty_G}
\arrow{rd}[sloped, swap]{\Sigma^\infty}
&
&
\Spectra^{\gen G}
\\
&
\Spectra^{\naive G}
\arrow[dashed]{ru}[sloped, swap]{\Psi}
\end{tikzcd} \]
(which results from the universal property of stabilization).

The $\infty$-category of naive $G$-spectra admits a stratification closely related to the geometric stratification of the $\infty$-category of genuine $G$-spectra of \Cref{intro.thm.gen.G.spt}. In fact, this arises as a special instance of a more general source of stratifications; we return to naive $G$-spectra in \Cref{ex.naive.G.spectra}.

Fix a presentable stable $\infty$-category $\cV$, as well as an $\infty$-category $\cT$ equipped with a functor
\[
\cT
\longra
\pos
~.
\]
Then, we obtain a stratification of the presentable stable $\infty$-category
\[
\Fun(\cT,\cV)
\]
over the poset $\pos^\op$ according to the formula
\begin{equation}
\label{stratn.of.Fun.T.V}
\begin{tikzcd}[row sep=0cm]
\pos^\op
\arrow{r}
&
\Cls_{\Fun(\cT,\cV)}
\\
\rotatebox{90}{$\in$}
&
\rotatebox{90}{$\in$}
\\
p^\circ
\arrow[maps to]{r}
&
\Fun ( \cT_{^\geq p} , \cV )
\end{tikzcd}
~,\footnote{Here and throughout, for any subposet $\sQ \subseteq \pos$ we write $\cT_\sQ := \cT \times_\pos \sQ$ for the fiber product; for any element $p \in \pos$ we simply write $\cT_p := \cT_{\{ p \}}$.}
\end{equation}
where we consider
\[
\Fun ( \cT_{^\geq p} , \cV )
\subseteq
\Fun ( \cT , \cV )
\]
as a closed subcategory via left Kan extension (which is simply extension by zero); its right adjoint is restriction, the right adjoint to which is right Kan extension.

The following features of the stratification \Cref{stratn.of.Fun.T.V} are easily verified.
\begin{enumerate}

\item\label{feature.strata.of.naive.stratn}

 For each $p^\circ \in \pos^\op$, the $(p^\circ)\th$ stratum of the stratification \Cref{stratn.of.Fun.T.V} is
\[
\Fun ( \cT_p , \cV)
~.
\]

\item\label{feature.monodromy.of.naive.stratn}

For any morphism $q^\circ \ra p^\circ$ in $\pos^\op$, the corresponding gluing functor of the stratification \Cref{stratn.of.Fun.T.V} is given as follows. First of all, we define the \textit{$q\th$ stratum of the link of $\cT_p$ in $\cT$} as the limit in the diagram
\[ \begin{tikzcd}
\Link_{\cT_p}(\cT)_q
\arrow{rr}{t}
\arrow{dd}[swap]{s}
\arrow[hook]{rd}
&
&
\cT_q
\\
&
\Ar(\cT)
\arrow{r}[swap]{t}
\arrow{d}{s}
&
\cT
\arrow[hookleftarrow]{u}
\\
\cT_p
\arrow[hook]{r}
&
\cT
\end{tikzcd}
~.
\]
Then, the corresponding gluing functor is the composite
\[
\Gamma^{q^\circ}_{p^\circ}
:
\Fun(\cT_q,\cV)
\xlongra{t^*}
\Fun(\Link_{\cT_p}(\cT)_q,\cV)
\xlongra{s_*}
\Fun(\cT_p,\cV)
\]
of pullback along $t$ followed by right Kan extension along $s$.

\item If the functor $\cT \ra \pos$ is an exponentiable fibration, then the stratification \Cref{stratn.of.Fun.T.V} is strict, i.e.\! the gluing functors strictly compose. Specifically, exponentiability guarantees that links glue: for instance, given any composite $p \ra q \ra r$ in $\pos$, we have an equivalence
\[
\Link_{\cT_p}(\cT)_r
\simeq
\Link_{\cT_p}(\cT)_q
\otimes_{\cT_q}
\Link_{\cT_q}(\cT)_r
\]
(expressing the $r\th$ stratum of the link of $\cT_p$ in $\cT$ as a coend over $\cT_q$). In this case, the gluing diagram
\[
\pos^\op
\xra{\GD(\Fun(\cT,\cV))}
\PrSt
\]
is simply the unstraightening of the cartesian fibration
\[ \begin{tikzcd}
\Fun^\rel_{/\pos} ( \cT , \ul{\cV} )
\arrow{d}
\\
\pos
\end{tikzcd}
~.
\]

\item If $\cV$ is presentably $\cO$-monoidal, then $\Fun(\cT,\cV)$ is presentably $\cO$-monoidal via the pointwise $\cO$-monoidal structure, and with respect to this the stratification \Cref{stratn.of.Fun.T.V} is an $\cO$-monoidal stratification.

\end{enumerate}

\begin{example}
\label{ex.tamely.conical.stratd.top.space}
Let us say that a stratified topological space $T \ra \pos$ is \textit{tamely conical} if the topological space $T$ is paracompact and locally of singular shape and moreover its stratification is conical.\footnote{These are the conditions under which \cite[Theorem A.9.3]{LurieHA} applies.} In this case, if we take
\[
\cT
:=
\Exit(T)
\longra
\pos
\]
to be the \bit{exit-path $\infty$-category} of $T$ equipped with its canonical functor to $\pos$, then the $\pos^\op$-stratification \Cref{stratn.of.Fun.T.V} recovers that of the $\infty$-category $\Shv^{\pos\textup{-}\cbl}(T)$ of $\pos$-constructible sheaves on $T$ obtained in \Cref{subsection.cbl}. In this case, for each $p^\circ \in \pos^\op$, the $(p^\circ)^\th$ stratum is the presentable stable $\infty$-category $\Loc(T_p)$ of local systems on the $p\th$ stratum (according to \Cref{feature.strata.of.naive.stratn}), and the gluing functors are governed by spaces of exiting paths (as described in \Cref{feature.monodromy.of.naive.stratn}).
\end{example}

\begin{remark}
A converse to \Cref{ex.tamely.conical.stratd.top.space} is provided by \cite{Haine-strat}: whenever the functor $\cT \ra \pos$ is conservative (i.e.\! whenever its fibers are $\infty$-groupoids), there exists a $\pos$-stratified topological space $T \ra \pos$ and an equivalence $\cT \simeq \Exit(T)$ in $\Cat_{/\pos}$.\footnote{For instance, this applies to the functor $\Orb_G^\op \ra \pos_G^\op$ considered in \Cref{ex.naive.G.spectra}.}
\end{remark}

\begin{example}[a stratification of naive $G$-spectra]
\label{ex.naive.G.spectra}
Taking
\[
(\cT \longra \pos)
:=
(\Orb_G^\op \longra \pos_G^\op)
\qquad
\text{and}
\qquad
\cV
:=
\Spectra
~,
\]
the stratification \Cref{stratn.of.Fun.T.V} specializes to a stratification
\begin{equation}
\label{stratn.of.Sp.nG}
\begin{tikzcd}[row sep=0cm]
\pos_G
\arrow{r}
&
\Cls_{\Spectra^{\naive G}}
\\
\rotatebox{90}{$\in$}
&
\rotatebox{90}{$\in$}
\\
H
\arrow[maps to]{r}
&
\Fun ( (\Orb_G^\op)_{^\geq H} , \Spectra )
\end{tikzcd}
\end{equation}
of the presentable stable $\infty$-category of naive $G$-spectra. The above features of the stratification \Cref{stratn.of.Fun.T.V} bear upon the stratification \Cref{stratn.of.Sp.nG} as follows.
\begin{enumerate}

\item For each $H \in \pos_G$, the $H\th$ stratum of the stratification \Cref{stratn.of.Sp.nG} is
\[
\Spectra^{\htpy \Weyl(H)}
~.
\]

\item For any nonidentity morphism $H < K$ in $\pos_G$, the corresponding gluing functor of the stratification \Cref{stratn.of.Sp.nG} is given by pullback followed by right Kan extension along the span
\begin{equation}
\label{span.between.B.of.Weyl.groups}
\begin{tikzcd}
( (G/K)^H )_{\htpy (\Weyl(K) \times \Weyl(H))}
\\[-0.7cm]
\rotatebox{90}{$\simeq$}
\\[-0.7cm]
\hom_{\Orb_G} ( G/H , G/K )_{\htpy (\Weyl(K) \times \Weyl(H))}
\\[-0.7cm]
\rotatebox{90}{$\simeq$}
\\[-0.7cm]
\hom_{\Orb_G^\op}((G/K)^\circ,(G/H)^\circ)_{\htpy (\Weyl(H) \times \Weyl(K))}
\arrow{r}
\arrow{d}
&
\sB \Weyl(H)
\\
\sB \Weyl(K)
\end{tikzcd}
~.\footnote{At the level of path components, we have an identification
\[
\pi_0 \left( ( (G/K)^H )_{\htpy (\Weyl(K) \times \Weyl(H))} \right)
\cong
\Weyl(K) ~ \backslash ~ (G/K)^H \slash ~ \Weyl(H)
\]
with the set of double cosets.}
\end{equation}
In the case that $G$ is abelian, the span \Cref{span.between.B.of.Weyl.groups} reduces to the span
\[ \begin{tikzcd}
\sB G/H
\arrow{r}{\sim}
\arrow{d}
&
\sB G/H
\\
\sB G/K
\end{tikzcd}~, \]
so that the gluing functor is given by the homotopy $(K/H)$-fixedpoints functor
\[
\Spectra^{\htpy (G/H)}
\simeq
\Fun ( \sB (G/H) , \Spectra )
\xra{(-)^{\htpy (K/H)}}
\Fun ( \sB (G/K) , \Spectra )
\simeq
\Spectra^{\htpy (G/K)}
~.
\]

\item In the case that $G$ is abelian, the functor
\[
\Orb_G^\op
\longra
\pos_G^\op
\]
is a right fibration (and in particular an exponentiable fibration). Hence, the stratification \Cref{stratn.of.Sp.nG} is strict (corresponding to the fact that homotopy fixedpoints strictly compose).

\item As $\Spectra$ is presentably symmetric monoidal, $\Spectra^{\naive G}$ is presentably symmetric monoidal as well. With respect to this structure, the stratification \Cref{stratn.of.Sp.nG} is a symmetric monoidal stratification.

\end{enumerate}
\end{example}

\begin{remark}
It is not hard to see that the functor
\[
\Spectra^{\naive G}
\xlongra{\Psi}
\Spectra^{\gen G}
\]
defines a morphism in $\Strat_{\pos_G}$ (see \Cref{defn.Strat.P}), where the source is equipped with the stratification \Cref{stratn.of.Sp.nG} and the target is equipped with the geometric stratification of \Cref{intro.thm.gen.G.spt}.\footnote{This may be verified as follows (see \Cref{defn.geometric.prestratn.of.gen.G.spt} for the geometric stratification of $\Spectra^{\gen G}$). It is clear that the $i_L$ inclusions commute. It remains to show that the $y$ projections also commute. For this, let us denote the stratifications by
\[
\pos_G
\xra{\cZ^\naive_\bullet}
\Cls_{\Spectra^{\naive G}}
\qquad
\text{and}
\qquad
\pos_G
\xra{\cZ^\gen_\bullet}
\Cls_{\Spectra^{\gen G}}
~.
\]
Then, we observe that for any $H \in \pos_G$ there are conservative factorizations
\[
\begin{tikzcd}[ampersand replacement=\&]
\Spectra^{\naive G}
\arrow{rr}{\Res^G_H}
\arrow{rd}[sloped, swap]{y}
\&
\&
\Spectra^{\naive H}
\\
\&
\cZ^\naive_H
\arrow[dashed]{ru}
\end{tikzcd}
\qquad
\text{and}
\qquad
\begin{tikzcd}[ampersand replacement=\&]
\Spectra^{\gen G}
\arrow{rr}{\Res^G_H}
\arrow{rd}[sloped, swap]{y}
\&
\&
\Spectra^{\gen H}
\\
\&
\cZ^\gen_H
\arrow[dashed]{ru}
\end{tikzcd}
~.
\]
Hence, the fact that the $y$ projections commute follows from the commutativity of the square
\[ \begin{tikzcd}[ampersand replacement=\&]
\Spectra^{\naive G}
\arrow{r}{\Psi}
\arrow{d}[swap]{\Res^G_H}
\&
\Spectra^{\gen G}
\arrow{d}{\Res^G_H}
\\
\Spectra^{\naive H}
\arrow{r}[swap]{\Psi}
\&
\Spectra^{\gen H}
\end{tikzcd}
~.
\]
}
In fact, considering it as a morphism in $\CAlg(\PrLSt)$, the geometric stratification of $\Spectra^{\gen G}$ may be seen as arising from the stratification \Cref{stratn.of.Sp.nG} of $\Spectra^{\naive G}$ via \Cref{rmk.image.of.ideals.and.stratns}.
\end{remark}

\subsection{Reflection, M\"obius inversion, and Verdier duality}
\label{subsection.verdier}

In \Cref{subsection.cbl}, we established a stratification of (($\pos$-)constructible) sheaves over a $\pos$-stratified topological space. Applying \Cref{intro.thm.cosms}\Cref{intro.main.thm.macrocosm} (in the case that $\pos$ is down-finite), one obtains a reconstruction theorem for such sheaves that involves $*$-push/pull functors (e.g.\! the composite $p_L i_R = i^* j_*$ in the recollement \Cref{recollement.for.sheaves.on.topological.spaces}). On the other hand, particularly in the context of constructible sheaves, it is desirable to instead reconstruct sheaves using $!$-push/pull functors (e.g.\! the composite $p_R i_L = i^! j_!$ in the recollement \Cref{recollement.for.sheaves.on.topological.spaces}).

We establish a means of passing between these two dual reconstruction patterns, as we describe presently.\footnote{This applies primarily in the case that $\pos$ is down-finite, but see also \Cref{rmk.finite.intervals.but.not.down.finite}.} We refer to this theory as \bit{reflection}, since in the case that $\pos = [1]$ it recovers the theory of reflection functors (see \Cref{rmk.connect.with.reflection.functors}). We use this to give a categorification of the M\"obius inversion formula in \Cref{ex.categorified.mobius.inversion}, and we explain a close connection with Verdier duality in \Cref{ex.reflection.and.Verdier}.

Fix a stratification $\pos \xra{\cZ_\bullet} \cX$. Let us recall its \textit{gluing diagram} from \Cref{subsection.intro.stratn.nc.stacks}: this is a left-lax functor
\[
\begin{tikzcd}[column sep=1.5cm]
\pos
\arrow{r}[description, yshift=-0.05cm]{\llax}{\GD(\cX)}
&
\PrSt
\end{tikzcd}
\]
that carries each morphism $p \ra q$ in $\pos$ to the \textit{gluing functor}
\[ \Gamma^p_q : \cX_p \xlonghookra{\rho^p} \cX \xra{\Phi_q} \cX_q~, \]
which is built from the composite \textit{geometric localization} adjunctions
\[ \begin{tikzcd}[column sep=1.5cm]
\Phi_p
:
\cX
\arrow[transform canvas={yshift=0.9ex}]{r}{y}
\arrow[hookleftarrow, transform canvas={yshift=-0.9ex}]{r}[yshift=-0.2ex]{\bot}[swap]{i_R}
&
\cZ_p
\arrow[transform canvas={yshift=0.9ex}]{r}{p_L}
\arrow[hookleftarrow, transform canvas={yshift=-0.9ex}]{r}[yshift=-0.2ex]{\bot}[swap]{\nu}
&
\cX_p
:
\rho^p
\end{tikzcd}
\]
for all $p \in \pos$. By contrast, if we instead begin with the composite \bit{reflected geometric localization} adjunctions
\[ \begin{tikzcd}[column sep=1.5cm]
\lambda^p
:
\cX_p
\arrow[hook, transform canvas={yshift=0.9ex}]{r}{\nu}
\arrow[leftarrow, transform canvas={yshift=-0.9ex}]{r}[yshift=-0.2ex]{\bot}[swap]{p_R}
&
\cZ_p
\arrow[hook, transform canvas={yshift=0.9ex}]{r}{i_L}
\arrow[leftarrow, transform canvas={yshift=-0.9ex}]{r}[yshift=-0.2ex]{\bot}[swap]{y}
&
\cX
:
\Psi_p
\end{tikzcd}
\]
for all $p \in \pos$ (first introduced in \Cref{rmk.filtrations.from.stratns}), we obtain for each morphism $p \ra q$ in $\pos$ the \bit{reflected gluing functor}
\[
\widecheck{\Gamma}^p_q
:
\cX_p
\xlonghookra{\lambda^p}
\cX
\xra{\Psi_q}
\cX_q
~,
\]
and these assemble into the \bit{reflected gluing diagram} of the stratification: a \textit{right}-lax functor
\[
\begin{tikzcd}[column sep=1.5cm]
\pos
\arrow{r}[description, yshift=-0.05cm]{\rlax}{\widecheck{\GD}(\cX)}
&
\PrSt
\end{tikzcd}
~.
\]

\begin{maintheorem}[\Cref{cor.reflection.for.presentable.stratns}]
\label{intro.thm.reflection}
Let $\pos$ be a down-finite poset.
\begin{enumerate}

\item\label{intro.reflection.thm.metacosm} {\bf metacosm:} The reflected gluing diagram functor is an equivalence, as indicated in the canonical commutative diagram
\begin{equation}
\label{metacosm.reflection.equivalence.in.intro.thm}
\begin{tikzcd}[column sep=1.5cm, row sep=1.5cm]
&
{\displaystyle \prod_{p \in \pos} \PrSt}
\\
\LMod_{\rlax.\pos}^L(\PrSt)
\arrow[yshift=0.9ex]{r}{\limllaxfam}
\arrow[leftarrow, yshift=-0.9ex]{r}{\sim}[swap]{\widecheck{\GD}}
\arrow{ru}[sloped]{(\ev_p)_{p \in \pos}}
\arrow{rd}[sloped, swap]{\lim^\llax_{\rlax.\pos}}
&
\Strat_\pos^\strict
\arrow[yshift=0.9ex]{r}[sloped]{\GD}
\arrow[leftarrow, yshift=-0.9ex]{r}{\sim}[sloped, swap]{\limrlaxfam}
\arrow{d}{\fgt}
\arrow{u}[swap]{((-)_p)_{p \in \pos}}
&
\LMod_{\llax.\pos}^L(\PrSt)
\arrow{lu}[sloped]{(\ev_p)_{p \in \pos}}
\arrow{ld}[sloped, swap]{\lim^\rlax_{\llax.\pos}}
\\
&
\PrSt
\end{tikzcd}
~.\footnote{Of course, $\LMod^L_{\rlax.\pos}(\PrSt)$ denotes a certain $\infty$-category whose objects are right-lax functors from $\pos$ to $\PrSt$, and the notation $\limllaxfam$ denotes a certain ``parametrized left-lax limit'' functor. Here we must restrict to the subcategory $\Strat_\pos^\strict \subseteq \Strat_\pos$ of strict morphisms (as introduced in \Cref{rmk.extended.intro.variants.of.metacosm}\Cref{rmk.extended.intro.part.strict.reconstrn}): there is an implicit laxness in our definition of $\Strat_\pos$ that is compatible with the gluing diagram functor $\GD$ but not with the reflected gluing diagram functor $\widecheck{\GD}$. (Dually, one could instead allow for laxness that is compatible with $\widecheck{\GD}$ but not with $\GD$.)}
\end{equation}

\item\label{intro.reflection.thm.macrocosm} {\bf macrocosm:} For each $\pos$-stratified noncommutative stack $\cX \in \Strat_\pos^\strict$, the equivalence on the left in diagram \Cref{metacosm.reflection.equivalence.in.intro.thm} determines an equivalence
\begin{equation}
\label{macrocosm.reflection.equivalence.in.intro.thm}
\begin{tikzcd}[column sep=2cm]
\Fun(\sd(\pos)^\op,\cX)
\supseteq
\lim^\llax_{\rlax.\pos}(\widecheck{\GD}(\cX))
=:
\widecheck{\Glue}(\cX)
\arrow[yshift=0.9ex]{r}{\colim_{\sd(\pos)^\op}}
\arrow[yshift=-0.9ex, leftarrow]{r}{\sim}[swap]{\widecheck{\gd}}
&
\cX
~.
\end{tikzcd}
\end{equation}

\item\label{intro.reflection.thm.microcosm} {\bf microcosm:} For each quasicoherent sheaf $\cF \in \cX \in \Strat_\pos^\strict$ on a $\pos$-stratified noncommutative stack, the equivalence \Cref{macrocosm.reflection.equivalence.in.intro.thm} determines an equivalence
\begin{equation}
\label{microcosm.reflection.equivalence.in.intro.thm}
\colim_{\sd(\pos)^\op}(\widecheck{\gd}(\cF))
=:
\widecheck{\glue}(\cF)
\xlongra{\sim}
\cF
\end{equation}
in $\cX$.

\item\label{intro.reflection.thm.nanocosm} {\bf nanocosm:} For each quasicoherent sheaf $\cE \in \cX \in \Strat_\pos^\strict$ on a $\pos$-stratified noncommutative stack, applying $\ulhom_\cX(-,\cE)$ to the equivalence \Cref{microcosm.reflection.equivalence.in.intro.thm} determines an equivalence
\[
\lim_{([n] \xra{\varphi} \pos) \in \sd(\pos)}
\left(
\ulhom_{\cX_{\varphi(n)}}
(
\widecheck{\Gamma}_\varphi \Psi_{\varphi(0)} \cF
,
\Psi_{\varphi(0)} \cE
)
\right)
\xlongla{\sim}
\ulhom_\cX(\cF,\cE)
~.\footnote{Given an element $([n] \xra{\varphi} \pos) \in \sd(\pos)$, in parallel with the notation $\Gamma_\varphi := \Gamma^{\varphi(n-1)}_{\varphi(n)} \cdots \Gamma^{\varphi(0)}_{\varphi(1)}$ we write $\widecheck{\Gamma}_\varphi := \widecheck{\Gamma}^{\varphi(n-1)}_{\varphi(n)} \cdots \widecheck{\Gamma}^{\varphi(0)}_{\varphi(1)}$.}
\]

\end{enumerate}
\end{maintheorem}

\noindent In particular, under the assumption that $\pos$ is down-finite, Theorems \ref{intro.thm.reflection}\Cref{intro.reflection.thm.macrocosm} \and \ref{intro.thm.cosms}\Cref{intro.main.thm.macrocosm} provide dual macrocosm equivalences
\[
\lim^\llax_{\rlax.\pos}(\widecheck{\GD}(\cX))
=:
\widecheck{\Glue}(\cX)
\underset{\sim}{\xlongla{\widecheck{\gd}}}
\cX
\underset{\sim}{\xlongra{\gd}}
\Glue(\cX)
:=
\lim^\rlax_{\llax.\pos}(\GD(\cX))
\]
in $\PrSt$ for each $\pos$-stratified noncommutative stack $\cX \in \Strat_\pos^\strict$. On the other hand, omitting any reference to stratifications, \Cref{intro.thm.reflection}\Cref{intro.reflection.thm.metacosm} provides a canonical commutative diagram
\begin{equation}
\label{reflection.without.Strat}
\begin{tikzcd}[column sep=1.5cm, row sep=1.5cm]
&
{\displaystyle \prod_{p \in \pos} \PrSt}
\\
\LMod_{\rlax.\pos}^L(\PrSt)
\arrow{ru}[sloped]{(\ev_p)_{p \in \pos}}
\arrow{rd}[sloped, swap]{\lim^\llax_{\rlax.\pos}}
\arrow[leftarrow]{rr}{\widecheck{(-)}}[swap]{\sim}
&
&
\LMod_{\llax.\pos}^L(\PrSt)
\arrow{lu}[sloped]{(\ev_p)_{p \in \pos}}
\arrow{ld}[sloped, swap]{\lim^\rlax_{\llax.\pos}}
\\
&
\PrSt
\end{tikzcd}
\end{equation}
for any down-finite poset $\pos$.\footnote{More systematic notation would allow for the horizontal arrow in diagram \Cref{reflection.without.Strat} to point in both directions. We have written it in this way in order to maintain consistency, so that for any $\cX \in \Strat_\pos^\strict$ we have a canonical equivalence $\widecheck{\GD(\cX)} \simeq \widecheck{\GD}(\cX)$.} We refer to the equivalence
\[
\LMod^L_{\rlax.\pos}(\PrSt)
\xla[\sim]{\widecheck{(-)}}
\LMod^L_{\llax.\pos}(\PrSt)
\]
of diagram \Cref{reflection.without.Strat} as \bit{reflection}. In fact, we prove this equivalence more generally for posets whose intervals are finite (see \Cref{cor.reflection.for.finite.intervals}). Moreover, we give a direct formula for the reflected gluing functors in terms of gluing functors and reversely, as total co/fibers (\Cref{defn.tcofib.and.tfib}): for any nonidentity morphism $p < q$ in $\pos$ we have canonical equivalences
\begin{equation}
\label{gluing.and.reflected.gluing.functors.in.terms.of.each.other.in.general}
\widecheck{\Gamma}^p_q
\simeq
\tfib_{\varphi \in \sd(\pos)^{|p}_{|q}} \Sigma^{-1} \Gamma_\varphi
\qquad
\text{and}
\qquad
\Gamma^p_q
\simeq
\tcofib_{\varphi^\circ \in (\sd(\pos)^{|p}_{|q})^\op} \Sigma \widecheck{\Gamma}_\varphi
\end{equation}
in $\Fun(\cX_p,\cX_q)$ (see \Cref{prop.Gamma.check.as.tfib.of.Gammas.and.Gamma.as.tcofib.of.Gamma.checks} (and \Cref{notation.source.and.or.target.restricted.sd})). Note that if $\pos_{p//q} \cong [n]$ for some $n \geq 1$, then $\sd(\pos)^{|p}_{|q}$ is an $(n-1)$-cube; in particular, if $p < q$ admits no nontrivial factorizations then the equivalences \Cref{gluing.and.reflected.gluing.functors.in.terms.of.each.other.in.general} reduce to the equivalent equivalences
\begin{equation}
\label{gluing.and.reflected.gluing.functors.in.terms.of.each.other.for.consecutive}
\widecheck{\Gamma}^p_q \simeq \Sigma^{-1} \Gamma^p_q
\qquad
\text{and}
\qquad
\Gamma^p_q \simeq \Sigma \widecheck{\Gamma}^p_q
~.
\end{equation}

\begin{remark}
\label{rmk.reflection.concretely}
Observe that a closed subcategory
\[
\cZ
\xlonghookra{i_L}
\cX
\]
determines a closed subcategory
\[
\cZ^\op
\xlonghookra{i_R^\op}
\cX^\op
~,
\]
which we refer to as its \bit{reflected closed subcategory}.\footnote{Here and throughout this subsection, whenever we mention opposites of presentable stable $\infty$-categories we are implicitly referring to the theory of \textit{stable} stratifications (as introduced in \Cref{rmk.extended.intro.variants.of.metacosm}\Cref{rmk.extended.intro.part.stable.stratns}); we generally omit this distinction from the present discussion in order not to clutter our exposition.} In concrete terms, \Cref{intro.thm.reflection}\Cref{intro.reflection.thm.metacosm} may be interpreted as saying that given a stratification of $\cX$ over a down-finite poset $\pos$, passage to reflected closed subcategories determines a stratification of $\cX^\op$ over $\pos$, which we refer to as its \bit{reflected stratification}: writing $\cX^\refl$ for $\cX^\op$ equipped with its reflected stratification, we have an equivalence
\[
\GD(\cX^\refl)
\simeq
\widecheck{\GD}(\cX)^\op
~.\footnote{Indeed, our proof of \Cref{intro.thm.reflection} (which we establish as \Cref{cor.reflection.for.presentable.stratns}) is based on the analogous result for stable stratifications (\Cref{thm.reflection.for.stable.stratns}), the main ingredient in the proof of which is the reflected stable stratification (\Cref{prop.stratn.of.opposite}).}
\]
\end{remark}

\begin{example}
\label{ex.verdier.over.brax.one}
We unpack parts \Cref{intro.reflection.thm.metacosm} \and \Cref{intro.reflection.thm.macrocosm} of \Cref{intro.thm.reflection} in the case that $\pos = [1]$. First of all, we have identifications
\[
\begin{tikzcd}
\LMod_{\rlax.[1]}
\arrow{rr}{\sim}
\arrow{rd}[sloped, swap]{\lim^\llax_{\rlax.[1]}}
&
&
\coCart_{[1]}
\arrow{ld}[sloped, swap]{\Gamma}
\\
&
\Cat
\end{tikzcd}
\qquad
\text{and}
\qquad
\begin{tikzcd}
\LMod_{\llax.[1]}
\arrow{rr}{\sim}
\arrow{rd}[sloped, swap]{\lim^\rlax_{\llax.[1]}}
&
&
\Cart_{[1]^\op}
\arrow{ld}[sloped, swap]{\Gamma}
\\
&
\Cat
\end{tikzcd}
~.
\]
Let us denote by
\[
\begin{tikzcd}
\LMod^L_{\rlax.[1]}(\PrSt)
\arrow[dashed]{r}{\sim}
\arrow[hook]{d}
&
\coCart_{[1]}^L(\PrSt)
\arrow[hook]{d}
\\
\LMod_{\rlax.[1]}
\arrow{r}[swap]{\sim}
&
\coCart_{[1]}
\end{tikzcd}
\qquad
\text{and}
\qquad
\begin{tikzcd}
\LMod^L_{\llax.[1]}(\PrSt)
\arrow[dashed]{r}{\sim}
\arrow[hook]{d}
&
\Cart_{[1]^\op}^L(\PrSt)
\arrow[hook]{d}
\\
\LMod_{\llax.[1]}
\arrow{r}[swap]{\sim}
&
\Cart_{[1]^\op}
\end{tikzcd}
\]
the indicated corresponding subcategories. Now, for any recollement \Cref{recollement.in.intro} we have a canonical equivalence $p_R i_L \simeq \Sigma^{-1} p_L i_R$ (a special case of the equivalences \Cref{gluing.and.reflected.gluing.functors.in.terms.of.each.other.for.consecutive}). It follows that the commutative diagram \Cref{reflection.without.Strat} specializes to a commutative diagram
\begin{equation}
\label{reflection.without.Strat.for.brax.one}
\begin{tikzcd}[column sep=1.5cm, row sep=1.5cm]
&
{\displaystyle \prod_{p \in [1]} \PrSt}
\\
\coCart_{[1]}^L(\PrSt)
\arrow{ru}[sloped]{(\ev_p)_{p \in [1]}}
\arrow{rd}[sloped, swap]{\Gamma}
\arrow[leftarrow]{rr}{\widecheck{(-)}}[swap]{\sim}
&
&
\Cart_{[1]^\op}^L(\PrSt)
\arrow{lu}[sloped]{(\ev_p)_{p \in [1]^\op}}
\arrow{ld}[sloped, swap]{\Gamma}
\\
&
\PrSt
\end{tikzcd}
\end{equation}
in which the equivalence $\widecheck{(-)}$ carries the cartesian unstraightening of a functor $\cX_0 \xra{F} \cX_1$ to the cocartesian unstraightening of the functor $\cX_0 \xra{\Sigma^{-1} F} \cX_1$. Thereafter, the commutativity of the lower triangle in diagram \Cref{reflection.without.Strat.for.brax.one} records the equivalence
\begin{equation}
\label{equivce.between.rlax.lim.and.llax.lim.in.brax.one.case}
\begin{tikzcd}[row sep=0cm]
\lim^\llax_{\rlax.[1]} ( \cX_0 \xra{\Sigma^{-1}F} \cX_1 )
\arrow[leftarrow]{r}{\sim}
&
\lim^\rlax_{\llax.[1]}(\cX_0 \xra{F} \cX_1)
\\
\rotatebox{90}{$\in$}
&
\rotatebox{90}{$\in$}
\\
( Z \longmapsto \Sigma^{-1}F(Z) \longra \fib(\alpha) )
&
( Z \longmapsto F(Z) \xlongla{\alpha} U )
\arrow[maps to]{l}
\end{tikzcd}
\end{equation}
in $\PrLSt$.
\end{example}

\begin{remark}
\label{rmk.connect.with.reflection.functors}
\Cref{ex.verdier.over.brax.one} is closely related to the theory of reflection functors \cite{BGP-refl}. 
Indeed, we recover \cite[Theorem 2.3]{DJW-refl} as follows. Fix a finite poset $\posQ$ equipped with a conservative functor $\posQ \ra [1]$. Additionally fix a functor $\posQ \ra \Cat$, and let us respectively denote by
\[
\cE^+
\longra
\posQ
\qquad
\text{and}
\qquad
\cE^-
\longra
\posQ^\op
\]
its cocartesian and cartesian unstraightenings. These data determine composite functors
\[
\cE^+
\longra
\posQ
\longra
[1]
\qquad
\text{and}
\qquad
(\cE^-)^\op
\longra
\posQ
\longra
[1]
~.
\]
Fix a presentable stable $\infty$-category $\cV$. On the one hand, the functor $\cE^+ \ra [1]$ determines a stratification of
\[
\Fun ( \cE^+ , \cV)
\]
over $[1]^\op$ as in \Cref{subsection.naive.G.spectra.stratn}. On the other hand, the functor $(\cE^-)^\op \ra [1]$ similarly determines a stratification of
\[
\Fun ( (\cE^-)^\op , \cV^\op )
\simeq
\Fun ( \cE^- , \cV)^\op
\]
over $[1]^\op$, which by \Cref{intro.thm.reflection}\Cref{intro.reflection.thm.metacosm} (as interpreted via \Cref{rmk.reflection.concretely}) determines a stratification of
\[
\Fun ( \cE^- , \cV)
\]
over $[1]^\op$.
Unwinding the definitions, we find that the gluing diagram $\GD ( \Fun ( \cE^+ , \cV) )$ of the former as well as the reflected gluing diagram $\widecheck{\GD}( \Fun ( \cE^- , \cV) )$ of the latter both record the composite functor
\[
F
:
\prod_{q \in \posQ_1} \Fun ( \cE_q , \cV )
\longra
\prod_{\alpha \in \Gamma ( \posQ \da [1])} \Fun ( \cE_{\alpha(0)} , \cV )
\xra{
\cF_\bullet
\longmapsto
\left(
{
\prod_{\{\alpha : \alpha(0) = p \}} \cF_\alpha
}
\right)_{p \in \posQ_0}
}
\prod_{p \in \posQ_0} \Fun ( \cE_p , \cV )
~.
\]
Hence, applying \Cref{intro.thm.reflection}\Cref{intro.reflection.thm.metacosm} (and the equivalence \Cref{equivce.between.rlax.lim.and.llax.lim.in.brax.one.case} of \Cref{ex.verdier.over.brax.one} combined with the equivalence $F \simeq \Sigma^{-1} F$ in $\Fun([1],\PrSt)$), we obtain the composite equivalence
\[
\Fun ( \cE^+ , \cV )
\simeq
\lim^\rlax_{\llax.[1]^\op} ( \GD ( \Fun ( \cE^+ , \cV ) ) )
\simeq
\lim^\llax_{\rlax.[1]^\op} ( \widecheck{\GD} ( \Fun ( \cE^- , \cV ) ) )
\simeq
\Fun ( \cE^- , \cV )
~.
\]
\end{remark}

\begin{example}[categorified M\"obius inversion]
\label{ex.categorified.mobius.inversion}
Given a down-finite poset $\pos$ and a presentable stable $\infty$-category $\cV$, the presentable stable $\infty$-category
\[
\cX := \Fun(\pos,\cV)
\]
of $\pos$-filtered objects in $\cV$ admits a stratification
\begin{equation}
\label{stratn.of.Fun.P.V.for.Mobius}
\begin{tikzcd}[row sep=0cm]
\pos
\arrow{r}
&
\Cls_\cX
\\
\rotatebox{90}{$\in$}
&
\rotatebox{90}{$\in$}
\\
p
\arrow[maps to]{r}
&
\Fun((^\leq p),\cV)
\end{tikzcd}
~,
\end{equation}
where we consider
\[
\Fun((^\leq p),\cV)
\subseteq
\Fun(\pos,\cV)
=:
\cX
\]
as a closed subcategory via left Kan extension.\footnote{Beware that this is not an instance of the stratification \Cref{stratn.of.Fun.T.V} considered in \Cref{subsection.naive.G.spectra.stratn}.} Unwinding the definitions, for each $p \in \pos$ we obtain an identification
\[
\begin{tikzcd}[column sep=1.5cm]
\cX
\arrow[hookleftarrow, bend left=45]{r}[description]{i_L}
\arrow{r}[transform canvas={yshift=0.1cm}]{\bot}[swap,transform canvas={yshift=-0.1cm}]{\bot}[description]{y}
\arrow[hookleftarrow, bend right=45]{r}[description]{i_R}
&
\cZ_p
\arrow[bend left=45]{r}[description]{p_L}
\arrow[hookleftarrow]{r}[transform canvas={yshift=0.1cm}]{\bot}[swap,transform canvas={yshift=-0.1cm}]{\bot}[description]{\nu}
\arrow[bend right=45]{r}[description]{p_R}
&
\cX_p
\end{tikzcd}
\qquad
\simeq
\qquad
\begin{tikzcd}
\Fun(\pos,\cV)
\arrow[hookleftarrow, bend left=30]{r}[pos=0.48]{\lkan}
\arrow{r}[transform canvas={yshift=0.15cm}]{\bot}[swap,transform canvas={yshift=-0.15cm}]{\bot}[description]{\res}
\arrow[hookleftarrow, bend right=30]{r}[swap, pos=0.48]{\rkan}
&
\Fun((^\leq p),\cV)
\arrow[bend left=30]{r}[pos=0.45]{\tcofib_{(^\leq p)}}
\arrow[hookleftarrow]{r}[, pos=0.3, transform canvas={yshift=0.15cm}]{\bot}[swap, pos=0.3, transform canvas={yshift=-0.15cm}]{\bot}[description, pos=0.3]{\delta_p}
\arrow[bend right=30]{r}[swap, pos=0.45]{\ev_p}
&[0.5cm]
\cV
\end{tikzcd}
~,
\]
where $\delta_p$ denotes the ``Dirac delta'' functor at $p \in (^\leq p)$ (and the right Kan extension is simply extension by zero). In particular, the reflected gluing diagram is the constant locally cartesian fibration
\[
\widecheck{\GD}(\cX)
\simeq
\cV \times \pos^\op
\xlongra{\pr}
\pos^\op
~,
\]
and the $p\th$ geometric localization and $p\th$ reflected geometric localization functors are respectively the $p\th$ associated graded and $p\th$ filtered components:
\[
\Phi_p(V_\bullet)
:=
p_L y (V_\bullet)
\simeq
\gr_p(V_\bullet)
:=
\tcofib_{(^\leq p)}(V_\bullet)
\qquad
\text{and}
\qquad
\Psi_p(V_\bullet)
:=
p_R y(V_\bullet)
\simeq
\fil_p(V_\bullet)
:=
V_p
~.\footnote{In what follows, we use either or both of these possible notations, depending on our desired emphasis.}
\]

In this situation, \Cref{prop.Gamma.check.as.tfib.of.Gammas.and.Gamma.as.tcofib.of.Gamma.checks} yields a \bit{categorified M\"obius inversion formula} that expresses the $p\th$ associated graded component $\gr_p(V_\bullet) \simeq \Phi_p(V_\bullet)$ in terms of the filtered components $\fil_q(V_\bullet) \simeq \Psi_q(V_\bullet) \simeq V_q$ for various $q \in (^\leq p)$, as we now explain. First of all, the stratification \Cref{stratn.of.Fun.P.V.for.Mobius} endows each object $V_\bullet \in \cX$ with a (descending) $\pos^\op$-filtration
\[
\fil_R^\bullet(V_\bullet)
\in
\Fun(\pos^\op,\cX)
\]
(recall \Cref{rmk.filtrations.from.stratns}), and for each $r^\circ \in \pos^\op$ its $(r^\circ)\th$ associated graded component is
\[
\gr_R^r(V_\bullet)
\simeq
\rho^r ( \Psi_r(V_\bullet))
\simeq
\rho^r(V_r)
\simeq
\delta_r(V_r)
\in
\cX
~.
\]
Applying the functor $\cX \xra{\Phi_p} \cX_p \simeq \cV$, we obtain a $\pos^\op$-filtration
\[
\fil_R^\bullet(\Phi_p(V_\bullet))
:=
\Phi_p(\fil_R^\bullet(V_\bullet))
\in
\Fun(\pos^\op,\cV)
\]
of $\Phi_p(V_\bullet) \simeq \gr_p(V_\bullet) \in \cV$, whose $(r^\circ)\th$ associated graded component is the object
\[
\Phi_p ( \gr_R^r ( V_\bullet ) )
\simeq
\Phi_p ( \rho^r \Psi_r ( V_\bullet ) )
\simeq
\Gamma^r_p(V_r)
\in
\cV
\]
(which is zero whenever $r \not\leq p$, see \Cref{rmk.gluing.fctrs.even.if.no.relation.in.poset}). Applying \Cref{prop.Gamma.check.as.tfib.of.Gammas.and.Gamma.as.tcofib.of.Gamma.checks}\Cref{part.Gamma.as.tcofib.of.Gamma.checks}, for any $r \leq p$ we obtain an identification
\begin{equation}
\label{identify.gluing.functors.for.mobius.inversion}
\Gamma^r_p(V_r)
\simeq
M^r_p
\tensoring
V_r
\end{equation}
in $\cV$, where $M^r_p \in \Spaces^\fin_*$ denotes the finite pointed space
\[
M^r_p
:=
\left\{ \begin{array}{ll}
S^0
~,
&
r=p
\\
\Sigma^2 | \pos_{r//p} \backslash \{r,p\} |
~,
&
r<p
\end{array} \right.
~.\footnote{In the case that $\pos_{r//p} = \{r < p \}$, we have $M^r_p := \Sigma^2(\es) \simeq S^1$.}\footnote{One could also adopt the convention that $M^r_p := \pt$ in the case that $r \not\leq p$.}
\]
Note that the reduced Euler characteristic
\[
\ol{\chi} (M^r_p)
:=
\chi(\Sigma^\infty M^r_p)
\in
\sK_0(\Spectra^\fin)
\cong
\ZZ
\]
is the value $\mu_\pos(r,p) = \mu_{\pos^\op}(p^\circ,r^\circ) \in \ZZ$ of the M\"obius function.

Now, assume that $\cV$ is compactly generated. In this case, we have two inclusions
\[ \begin{tikzcd}
\sK_0(\cX^\omega)
\arrow[hook, yshift=0.9ex]{r}{\widecheck{i}}
\arrow[hook, yshift=-0.9ex]{r}[swap]{i}
&
\hom_\Set(\pos^\delta,\sK_0(\cV^\omega))
\end{tikzcd} \]
of abelian groups as the subgroup of finitely-supported functions, given by the two gluing diagrams:
\[
\widecheck{i}([V_\bullet])(p)
:=
[\Psi_p(V_\bullet)]
=
[\fil_p(V_\bullet)]
=
[V_p]
\qquad
\text{and}
\qquad
i([V_\bullet])(p)
:=
[\Phi_p(V_\bullet)]
=
[\gr_p(V_\bullet)]
~.
\]
Now, using the equivalences \Cref{identify.gluing.functors.for.mobius.inversion}, we obtain the M\"obius inversion formula for $\pos$ (valued in the abelian group $\sK_0(\cV^\omega)$):
\[
\widecheck{i}([V_\bullet])(p)
=
\sum_{r \in (^\leq p)}
i([V_\bullet])(r)
\qquad
\text{and}
\qquad
i([V_\bullet])(p)
=
\sum_{r \in (^\leq p)}
\ol{\chi}(M^r_p)
\cdot
\widecheck{i}([V_\bullet])(r)
=
\sum_{r \in (^\leq p)}
\mu_\pos(r,p)
\cdot
\widecheck{i}([V_\bullet])(r)
~.
\]
\end{example}

\begin{remark}
\label{rmk.spectral.sequence.from.mobius.inversion}
In the context of \Cref{ex.categorified.mobius.inversion}, let us fix a conservative functor $\pos \xra{d} \ZZ$ and an element $p \in \pos$, and let us assume that $\cV$ is equipped with a t-structure.
Then, we obtain two spectral sequences by applying \Cref{rmk.spectral.sequences.from.stratns} to the restricted stratification over the poset $(^\leq p)$: taking $H = \Psi_p = \fil_p = (-)_p$ we obtain a spectral sequence
\begin{equation}
\label{trivial.Mobius.spectral.seq.for.Fun.P.V}
E_{s,t}^1
=
\bigoplus_{r \in d^{-1}(s) \cap (^\leq p)}
\pi_{s+t} ( \gr_r (V_\bullet) )
\Longrightarrow
\pi_{s+t}(V_p)
~,
\end{equation}
while taking $H = \Phi_p = \gr_p$ we obtain a spectral sequence
\begin{equation}
\label{Mobius.spectral.seq.for.Fun.P.V}
E_{s,t}^1
=
\bigoplus_{r \in d^{-1}(-s) \cap (^\leq p)}
\pi_{s+t} ( M^r_p \tensoring V_r )
\Longrightarrow
\pi_{s+t}(\gr_p(V_\bullet))
~.\footnote{The other two spectral sequences that can be constructed in this way (applying \Cref{usual.filtration.spectral.sequence} to $H = \Phi_p$ or \Cref{usual.reflected.filtration.spectral.sequence} to $H = \Psi_p$) collapse immediately.}
\end{equation}
Note that given any $\pos$-stratified noncommutative stack $\cX \in \Strat_\pos$ and any exact functor $\cX \xra{H'} \cV$, we can apply these spectral sequences to (the value under $H'$ of) any of the four filtrations discussed in \Cref{rmk.filtrations.from.stratns}.\footnote{Note that applying the spectral sequence \Cref{trivial.Mobius.spectral.seq.for.Fun.P.V} in this way simply gives the spectral sequences \Cref{usual.filtration.spectral.sequence} \and \Cref{usual.reflected.filtration.spectral.sequence} of \Cref{rmk.spectral.sequences.from.stratns}.}
\end{remark}

\begin{example}[filtered objects and chain complexes]
\label{ex.reflection.and.dold.kan}
Let us specialize \Cref{ex.categorified.mobius.inversion} to the case that our down-finite poset is $\pos = \ZZ_{\geq 0}$. In this case, the gluing diagram is given by
\[
\GD(\cX)
\simeq
\left(
\begin{tikzcd}
\cV
\arrow{r}{\Sigma}[swap]{\sim}
\arrow[bend left=50]{rr}{0}[swap, yshift=-0.15cm]{\Downarrow}
&
\cV
\arrow{r}{\Sigma}[swap]{\sim}
\arrow[bend right=50]{rr}[yshift=0.15cm]{\Uparrow}[swap]{0}
&
\cV
\arrow{r}{\Sigma}[swap]{\sim}
&
\cdots
\arrow{r}{\Sigma}[swap]{\sim}
\arrow[bend left=50]{rr}{0}[swap, yshift=-0.15cm]{\Downarrow}
&
\cV
\arrow{r}{\Sigma}[swap]{\sim}
\arrow[bend right=50]{rr}[yshift=0.15cm]{\Uparrow}[swap]{0}
&
\cV
\arrow{r}{\Sigma}[swap]{\sim}
&
\cdots
\end{tikzcd}
\right)
\in
\LMod^L_{\llax.\ZZ_{\geq 0}}(\PrSt)
~,
\]
and its right-lax limit is the $\infty$-category
\[
\Glue(\cX)
:=
\lim^\rlax_{\llax.\ZZ_{\geq 0}}(\GD(\cX))
\simeq
\Ch_{\geq 0}(\cV)
\]
of \bit{chain complexes} in $\cV$ concentrated in nonnegative degrees.\footnote{Informally, an object of $\Ch_{\geq 0}(\cV)$ may be thought of as a functor $(\ZZ_{\geq 0})^\op \ra \cV$ equipped with a coherent system of nullhomotopies for its $i$-fold composites for all $i \geq 2$. (These are equivalent to gapped objects in $\cV$ (see \cite[Definition 1.2.2.2 and Remark 1.2.2.3]{LurieHA}).)} \Cref{intro.thm.cosms}\Cref{intro.main.thm.macrocosm} grants an equivalence
\[
\Fun ( \ZZ_{\geq 0} , \cV)
\simeq
\Ch_{\geq 0}(\cV)
~,
\]
which is closely related to Lurie's Dold--Kan correspondence for stable $\infty$-categories; more precisely, it recovers a version of \cite[Lemma 1.2.2.4]{LurieHA}.
\end{example}

\begin{warning}
As illustrated by \Cref{ex.reflection.and.dold.kan}, reflection does \textit{not} preserve the property of being a strict (as opposed to lax) left $\pos$-module.
\end{warning}

\begin{example}[Verdier duality]
\label{ex.reflection.and.Verdier}
Let $T$ be a locally compact Hausdorff topological space equipped with a stratification $T \ra \pos$. Assume for simplicity that $\pos$ is finite, and choose any presentable stable $\infty$-category $\cV$.

\begin{enumerate}

\item\label{reflection.and.Verdier.for.sheaves}

Recall that Verdier duality \cite[Theorem 5.5.5.1]{LurieHA} asserts an equivalence
\begin{equation}
\label{Verdier.duality.for.sheaves}
\begin{tikzcd}
\Shv_\cV(T)^\op
\arrow[leftrightarrow]{r}{\DD_T}[swap]{\sim}
&
\Shv_{\cV^\op}(T)
\end{tikzcd}
~.
\end{equation}
On the one hand, by \Cref{subsection.cbl} we have a canonical stratification of
\[
\Shv_\cV(T)
\]
over $\pos^\op$, which by \Cref{intro.thm.reflection}\Cref{intro.reflection.thm.metacosm} (as interpreted via \Cref{rmk.reflection.concretely}) determines a stratification of
\[
\Shv_\cV(T)^\op
\]
over $\pos^\op$. On the other hand, we similarly have a canonical stratification of
\[
\Shv_{\cV^\op}(T)
\]
over $\pos^\op$. It is not hard to see that the equivalence \Cref{Verdier.duality.for.sheaves} respects these $\pos^\op$-stratifications.\footnote{This follows from the general fact that Verdier duality is compatible with open embeddings, in the sense that for any open subset $U \xhookra{j} T$ we have a commutative diagram
\[ \begin{tikzcd}[ampersand replacement=\&]
\Shv_\cV(T)^\op
\arrow[leftrightarrow]{r}{\DD_T}[swap]{\sim}
\&
\Shv_{\cV^\op}(T)
\\
\Shv_\cV(U)^\op
\arrow[hook]{u}{(j_*)^\op}
\arrow[leftrightarrow]{r}{\sim}[swap]{\DD_U}
\&
\Shv_{\cV^\op}(U)
\arrow[hook]{u}[swap]{j_!}
\end{tikzcd}
~.
\]
} In particular, it interchanges the induced filtrations of \Cref{rmk.filtrations.from.stratns}:
\[
\fil_R^\bullet
\simeq
\DD_T
\circ
\fil^L_\bullet
\circ
\DD_T
\qquad
\text{and}
\qquad
\fil^R_\bullet
\simeq
\DD_T
\circ
\fil_L^\bullet
\circ
\DD_T
~.
\]

\item\label{reflection.and.Verdier.for.cbl.sheaves}

Suppose that the dualizing complex $\omega_T \in \Shv_\cV(T)$ is $\pos$-constructible. Then, the Verdier duality equivalence \Cref{Verdier.duality.for.sheaves} extends to a commutative square
\begin{equation}
\label{Verdier.duality.for.sheaves.and.cbl.sheaves}
\begin{tikzcd}[row sep=1.5cm, column sep=1.5cm]
\Shv_\cV(T)^\op
\arrow[leftrightarrow]{r}{\DD_T}[swap]{\sim}
&
\Shv_{\cV^\op}(T)
\\
\Shv_\cV^{\pos\textup{-}\cbl}(T)^\op
\arrow[hook]{u}
\arrow[dashed, leftrightarrow]{r}[swap]{\DD_T^{\pos\textup{-}\cbl}}{\sim}
&
\Shv_{\cV^\op}^{\pos\textup{-}\cbl}(T)
\arrow[hook]{u}
\end{tikzcd}
~.
\end{equation}
The lower two terms in diagram \Cref{Verdier.duality.for.sheaves.and.cbl.sheaves} inherit $\pos^\op$-stratifications from the upper two terms, as in \Cref{subsection.cbl}, such that the entire diagram \Cref{Verdier.duality.for.sheaves.and.cbl.sheaves} respects $\pos^\op$-stratifications.

\end{enumerate}
\end{example}

\begin{remark}
In the situation of \Cref{ex.reflection.and.Verdier}\Cref{reflection.and.Verdier.for.cbl.sheaves}, suppose further that $T \ra \pos$ is tamely conical (as in \Cref{ex.tamely.conical.stratd.top.space}). Then, the lower equivalence of diagram \Cref{Verdier.duality.for.sheaves.and.cbl.sheaves} extends to a commutative square
\begin{equation}
\label{Verdier.duality.for.cbl.sheaves.and.Fun.from.Exit}
\begin{tikzcd}[row sep=1.5cm, column sep=1.5cm]
\Shv_\cV^{\pos\textup{-}\cbl}(T)^\op
\arrow[leftrightarrow]{r}{\DD_T^{\pos\textup{-}\cbl}}[swap]{\sim}
\arrow[leftrightarrow]{d}[sloped, anchor=north]{\sim}
&
\Shv_{\cV^\op}^{\pos\textup{-}\cbl}(T)
\arrow[leftrightarrow]{d}[sloped, anchor=south]{\sim}
\\
\Fun ( \Exit(T) , \cV)^\op
\arrow[dashed, leftrightarrow]{r}[swap]{\sim}
&
\Fun ( \Exit(T) , \cV^\op )
\end{tikzcd}
\end{equation}
of equivalences. The lower two terms in diagram \Cref{Verdier.duality.for.cbl.sheaves.and.Fun.from.Exit} inherit $\pos^\op$-stratifications from the functor $\Exit(T) \ra \pos$ as in \Cref{ex.tamely.conical.stratd.top.space}, and it is not hard to see that the entire diagram \Cref{Verdier.duality.for.cbl.sheaves.and.Fun.from.Exit} respects $\pos^\op$-stratifications.
\end{remark}


\subsection{t-structures}
\label{subsection.t.structures}


As we now describe, stratifications give a method for constructing new t-structures from old ones in the spirit of the construction of perverse sheaves \cite{BBD-perv}; applied to the geometric stratification of $\Spectra^{\gen G}$ of \Cref{intro.thm.gen.G.spt} for a finite group $G$, this technique can also be used to obtain the slice filtration \cite{HillYarnall-slice}.

Let $\cZ_\bullet$ be a stratification of $\cX$ over $\pos$. Suppose that each stratum $\cX_p$ is endowed with a t-structure. Then, by \cite[Proposition 1.4.4.11]{LurieHA} we obtain a t-structure on $\cX$, whose connective objects are precisely those that are taken to connective objects by all geometric localization functors $\cX \xra{\Phi_p} \cX_p$, i.e.\! the composites
\begin{equation}
\label{composites.pL.y.for.t.structures}
\cX
\xlongra{y}
\cZ_p
\xra{p_L}
\cX_p
~.
\end{equation}

Suppose that the functors \Cref{composites.pL.y.for.t.structures} are jointly conservative, e.g.\! as guaranteed by $\pos$ being artinian (recall \Cref{rmk.artinian.conservativity}). Then, this t-structure becomes particularly computable: we can \textit{also} explicitly describe its coconnective objects. Namely, they are precisely those that are taken to coconnective objects by all of the composites
\begin{equation}
\label{composites.pR.y.for.t.structures}
\cX
\xlongra{y}
\cZ_p
\xra{p_R}
\cX_p
~.
\end{equation}
We may see this as follows. Given any down-closed subset $\sD \subseteq \pos$, let us write
\[
\cZ_\sD
:=
\bigcup_{p \in \sD} \cZ_p
\qquad
\text{and}
\qquad
\cX_\sD
:=
\cX / \cZ_\sD
~.
\]
Then, from \Cref{intro.thm.fund.opns} we obtain
\begin{itemize}

\item a restricted stratification of $\cZ_\sD$ over $\sD$, whose $p\th$ stratum is $\cX_p$ for all $p \in \sD$, as well as

\item a quotient stratification of $\cX_\sD$ over $\pos \backslash \sD$, whose $p\th$ stratum is $\cX_p$ for all $p \in \pos \backslash \sD$.

\end{itemize}
Hence, $\cZ_\sD$ and $\cX_\sD$ both inherit t-structures, such that in the recollement
\[ \begin{tikzcd}[column sep=1.5cm]
\cZ_\sD
\arrow[hook, bend left=45]{r}[description]{i_L}
\arrow[leftarrow]{r}[transform canvas={yshift=0.1cm}]{\bot}[swap,transform canvas={yshift=-0.1cm}]{\bot}[description]{\yo}
\arrow[bend right=45, hook]{r}[description]{i_R}
&
\cX
\arrow[bend left=45]{r}[description]{p_L}
\arrow[hookleftarrow]{r}[transform canvas={yshift=0.1cm}]{\bot}[swap,transform canvas={yshift=-0.1cm}]{\bot}[description]{\nu}
\arrow[bend right=45]{r}[description]{p_R}
&
\cX_\sD
\end{tikzcd}~, \]
the functors $y$ and $\nu$ are t-exact, their left adjoints $i_L$ and $p_L$ are right t-exact (i.e.\! preserve connective objects), and their right adjoints $i_R$ and $p_R$ are left t-exact (i.e.\! preserve coconnective objects).\footnote{Indeed, $y$ and $p_L$ are right t-exact by definition, while $i_L$ and $\nu$ are right t-exact by inspection.} It follows that the functors \Cref{composites.pR.y.for.t.structures} preserve coconnective objects, and the same argument as that for the functors \Cref{composites.pL.y.for.t.structures} proves that they too are jointly conservative.


\subsection{Additive and localizing invariants}
\label{subsection.add.loc.invts}

We discuss the interaction of stratifications with additive and localizing invariants \cite{BGT-K}.

Recall that the ind-completion functor on small stable idempotent-complete $\infty$-categories factors as an equivalence
\[ \begin{tikzcd}
\St^\idem
\arrow{rr}{\Ind}
\arrow[dashed]{rd}[sloped, swap]{\sim}
&
&
\PrLSt
\\
&
\PrLomegaSt
\arrow[hook]{ru}
\end{tikzcd} \]
onto the subcategory
\begin{itemize}
\item whose objects are the compactly generated stable $\infty$-categories and
\item whose morphisms are those functors that preserve both colimits and compact objects.
\end{itemize}
In fact, for every morphism $\cC \xra{F} \cD$ in $\St^\idem$, the right adjoint
\[ \begin{tikzcd}[column sep=2cm]
\Ind(\cC)
\arrow[transform canvas={yshift=0.9ex}]{r}{\Ind(F) := F_!}
\arrow[dashed, leftarrow, transform canvas={yshift=-0.9ex}]{r}[yshift=-0.2ex]{\bot}[swap]{F^*}
&
\Ind(\cD)
\end{tikzcd} \]
is automatically colimit-preserving, and it preserves compact objects if and only if $F$ itself admits a right adjoint.  Hence, the composite functor
\[ \begin{tikzcd}
\St
\arrow{rr}{\Ind}
\arrow{rd}[swap, sloped]{(-)^\idem}
&
&
\PrLSt
\\
&
\St^\idem
\arrow{ru}[swap, sloped]{\Ind}
\end{tikzcd} \]
carries
\begin{enumerate}
\item exact sequences to recollements,
\item split-exact sequences to recollements in which $i_R$ preserves colimits, and
\item stable recollements (i.e.\! recollements among stable $\infty$-categories (\Cref{defn.stable.recollement})) to recollements in which $i_R$ preserves both colimits and compact objects.
\end{enumerate}

Fix a stable $\infty$-category $\cC \in \St$. We say that a full stable subcategory of $\cC$ is
\begin{enumerate}
\item \bit{thick} if it is idempotent-complete (relative to $\cC$),
\item \bit{split} if it is thick and its inclusion admits a right adjoint, and
\item \bit{closed} if it is split and the right adjoint to its inclusion admits a further right adjoint.
\end{enumerate}
With the evident notation, we then have a sequence of fully faithful functors
\[
\clssub_\cC
\longhookra
\splitsub_\cC
\longhookra
\thicksub_\cC
\xlonghookra{\Ind}
\Cls_{\Ind(\cC)}
\]
among posets, and we may define three sorts of stratifications of $\cC$ as stratifications of $\Ind(\cC)$ that factor accordingly.

\begin{remark}
A convergent stratification of $\Ind(\cC)$ gives, in particular, a means of reconstructing its full subcategory $\cC \subseteq \Ind(\cC)$. However, this is somewhat unsatisfying, as it will not generally reconstruct $\cC$ in terms of subcategories thereof: neither the geometric localization functors nor the gluing functors for the stratification of $\Ind(\cC)$ need preserve compact objects. On the other hand, given a stratification
\[
\pos
\longra
\clssub_\cC
\]
(as defined just above), the geometric localization functors and gluing functors of the composite stratification
\[
\pos
\longra
\clssub_\cC
\xlonghookra{\Ind}
\Cls_{\Ind(\cC)}
\]
do preserve compact objects. Indeed, these are precisely the stable stratifications introduced in \Cref{rmk.extended.intro.variants.of.metacosm}\Cref{rmk.extended.intro.part.stable.stratns} (under the assumption that $\cC$ is idempotent-complete), and the metacosm reconstruction theorem indicated there expresses $\cC$ entirely in terms of subcategories thereof.
\end{remark}

Now, recall that for a presentable stable $\infty$-category $\cV$, a $\cV$-valued \bit{additive} (resp.\! \bit{localizing}) \bit{invariant} is a functor
\[
\St
\longra
\cV
\]
that
\begin{itemize}
\item preserves zero objects and filtered colimits,
\item inverts Morita equivalences (i.e.\! factors through $\St \xra{(-)^\idem} \St^\idem$), and
\item carries split-exact (resp.\! exact) sequences to co/fiber sequences;
\end{itemize}
key examples include algebraic K-theory (the universal additive invariant), nonconnective algebraic K-theory (the universal localizing invariant), and topological Hochschild homology (a localizing invariant).  It follows that additive invariants carry
\begin{enumerate}
\setcounter{enumi}{1}
\item split-exact sequences to split co/fiber sequences and
\item recollements to doubly-split co/fiber sequences (i.e.\! co/fiber sequences equipped with two splittings),
\end{enumerate}
while localizing invariants carry
\begin{enumerate}
\item exact sequences to co/fiber sequences.
\item split-exact sequences to split co/fiber sequences, and
\item recollements to doubly-split co/fiber sequences.
\end{enumerate}

Putting these observations together, we find that a stratification of $\cC$ in each of the senses above determines a corresponding structure on the value at $\cC$ of any additive and/or localizing invariant. 
For instance, given an additive invariant
\[
\St
\xlongra{F}
\cV
~,
\]
a convergent stratification
\[
\pos
\xra{\cZ_\bullet}
\splitsub_\cC
\]
determines a direct sum decomposition
\[
F(\cC)
\simeq
\bigoplus_{p \in \pos}
F(\cC_p)
~,
\]
where $\cC_p$ denotes the $p\th$ stratum of the stratification: the stable quotient of $\cZ_p$ by $\brax{ \cZ_q }_{q<p}^\thick$ (using \Cref{notation.for.thick.closure}). Similarly, a stratification
\[
\pos
\longra
\thicksub_\cC
\]
induces a $\pos$-filtration (as studied e.g.\! in \Cref{ex.categorified.mobius.inversion}) on the value at $\cC$ of any localizing invariant (e.g.\! nonconnective algebraic K-theory).





\section{Stratified noncommutative geometry}
\label{section.strat}

In this section, we introduce the theory of stratified noncommutative geometry.  From here onwards, for simplicity we revert to standard categorical terminology, in particular opting for the term ``presentable stable $\infty$-category'' over the term ``noncommutative stack'' employed in \S\S\ref{section.intro}-\ref{subsection.intro.detailed.overview}.

This section is organized as follows.
\begin{itemize}

\item[\Cref{subsection.posets}:] We collect some notation and terminology regarding posets.

\item[\Cref{subsection.recollements}:] We prove the macrocosm reconstruction theorem (\Cref{intro.thm.cosms}\Cref{intro.main.thm.macrocosm}) for recollements, i.e.\! stratifications over $[1]$.

\item[\Cref{subsection.closed.subcats}:] We study the basic features of closed subcategories (called ``closed noncommutative substacks'' in \Cref{section.intro}).

\item[\Cref{subsections.stratns}:] We recall the definition of a stratification and related notions.

\item[\Cref{subsection.reconstrn.thm.for.stratns}:] We prove the macrocosm reconstruction theorem (\Cref{intro.thm.cosms}\Cref{intro.main.thm.macrocosm}) as \Cref{macrocosm.thm}.  This follows easily from the metacosm reconstruction theorem (\Cref{intro.thm.cosms}\Cref{intro.main.thm.metacosm}), which we prove in \Cref{section.reconstrn}.  We also explain the entire theory in the particular case of stratifications over $[2]$ as \Cref{ex.gluing.stuff.over.brax.two}.

\item[\Cref{subsection.microcosm.and.nanocosm.morphisms}:] We explain the microcosm and nanocosm morphisms (over an arbitrary poset).

\item[\Cref{subsection.strict.objects}:] We explain the theory of strict objects in stratified presentable stable $\infty$-categories.

\end{itemize}

\begin{local}
In this section, we fix a presentable stable $\infty$-category $\cX$ and a poset $\pos$.
\end{local}

\subsection{Posets}
\label{subsection.posets}

In this subsection, we collect some basic notation, terminology, and facts regarding posets.

\begin{definition}
\label{defn.convex.subset}
A \bit{convex subset} of $\pos$ is a full subposet $\sC \subseteq \pos$ satisfying the condition that if $p,r \in \sC$ and $p \leq q \leq r$ in $\pos$ then also $q \in \sC$.  We write $\Conv_\pos$ for the poset of convex subsets of $\pos$ ordered by inclusion.
\end{definition}

\begin{notation}
For any element $p \in \pos$, we simply write $p \in \Conv_\pos$ (rather than $\{ p \}$) for the corresponding singleton convex subset of $\pos$ that it defines.
\end{notation}

\begin{definition}
A \bit{down-closed subset} of $\pos$ is a full subposet $\sD \subseteq \pos$ satisfying the condition that if $q \in \sD$ and $p \leq q$ then also $p \in \sD$.  We write $\Down_\pos$ for the poset of down-closed subsets of $\pos$ ordered by inclusion.
\end{definition}

\begin{observation}
There is a containment $\Down_\pos \subseteq \Conv_\pos$: a down-closed subset of $\pos$ is automatically convex.
\end{observation}

\begin{notation}
Choose any $\sC \in \Conv_\pos$.
\begin{enumerate}
\item We write
\[ ^\leq \sC := \{ p \in \pos : p \leq q \textup{ for some } q \in \sC \} \in \Down_\pos \]
for the down-closure of $\sC$ in $\pos$.
\item We write
\[ ^< \sC := ( ^\leq \sC ) \backslash \sC \in \Down_\pos \]
for the down-closed subset of $\pos$ obtained by removing the elements of $\sC$ from $^\leq \sC$.
\end{enumerate}
We also write $^{\not\leq} \sC := \pos \backslash (^\leq \sC)$ and $^{\not<} \sC := \pos \backslash (^< \sC)$.
\end{notation}

\begin{definition}
We say that the poset $\pos$ is \bit{down-finite} if for every $p \in \pos$ the subset $(^\leq p) \subseteq \pos$ is finite.
\end{definition}

\begin{definition}
We say that the poset $\pos$ is \bit{artinian} if it admits no injective (or equivalently conservative) functors from $\NN^\op$.
\end{definition}

\begin{definition}
An \bit{interval} in a poset $\pos$ is a subset of the form $\pos_{p//q} \subseteq \pos$.
\end{definition}

\begin{notation}
Given a functor $\pos \ra \posQ$ between posets, for any subset $\sS \subseteq \posQ$ we write $\pos_{\sS} \subseteq \pos$ for its preimage.
\end{notation}

\begin{observation}
\label{obs.colimits.in.posets.are.unions}
For any surjective functor $\cI \xra{G} \cJ$ among $\infty$-categories and any functor $\cJ \xra{F} \pos$, we have a canonical identification $\colim_{\cI}(F G) \simeq \colim_\cJ(F)$.\footnote{In particular, each colimit exists if and only if the other does.}  We use this fact without further comment.
\end{observation}

\begin{remark}
\Cref{obs.colimits.in.posets.are.unions} may be articulated informally as the assertion that colimits in posets are all simply unions (taking $\cI$ to be a set).
\end{remark}

\subsection{Recollements}
\label{subsection.recollements}

In this subsection, we record the (simple and classical) macrocosm reconstruction theorem for recollements (\Cref{defn.recollement.in.intro}).

\begin{lemma}
\label{lem.reconstrn.for.recollement}
Given a recollement \Cref{recollement.in.intro}, the canonical functor
\[
\cX
\longra
\lim^\rlax \left( \cZ \xra{p_L i_R} \cU \right)
:=
\left\{ \left(
Z \in \cZ
~,~
U \in \cU
~,~
\begin{tikzcd}
U
\arrow{d}
\\
p_L i_R Z
\end{tikzcd}
\right)
\right\}
\]
given by the association
\[
X
\longmapsto
(yX \longmapsto p_L i_R yX \longla p_L X)
:=
\left(
\yo X \in \cZ
~,~
p_L X \in \cU
~,~
\begin{tikzcd}
p_L X
\arrow{d}
\\
p_L i_R \yo X
\end{tikzcd}
\right)
\]
is an equivalence.
\end{lemma}

\begin{proof}
We claim that this functor has an inverse, given by the association
\[
\lim \left(
\begin{tikzcd}
&
\nu U
\arrow{d}
\\
i_R Z
\arrow{r}
&
\nu p_L i_R Z
\end{tikzcd}
\right)
\longmapsfrom
\left(
Z \in \cZ
~,~
U \in \cU
~,~
\begin{tikzcd}
U
\arrow{d}
\\
p_L i_R Z
\end{tikzcd}
\right)
~.
\]
Indeed, the composite endofunctor of $\lim^\rlax \left( \cZ \xra{p_L i_R} \cU \right)$ is immediately seen to be the identity.  To see that the composite endofunctor of $\cX$ is also the identity, it suffices to check that for any $X \in \cX$ the commutative square
\begin{equation}
\label{square.that.must.be.pb.in.proof.of.macrocosm.for.recollement}
\begin{tikzcd}[ampersand replacement=\&]
X
\arrow{r}
\arrow{d}
\&
\nu p_L X
\arrow{d}
\\
i_R \yo X
\arrow{r}
\&
\nu p_L i_R \yo X
\end{tikzcd}
\end{equation}
is a pullback square.  As a result of the equality $\im(\nu) = \ker(\yo)$, the fibers of the horizontal morphisms in the commutative square \Cref{square.that.must.be.pb.in.proof.of.macrocosm.for.recollement} are equivalent.
\end{proof}

\begin{warning}
Recollements play a central role in our work.  We generally use the notations of diagram \Cref{recollement.in.intro} for the various functors involved (e.g.\! $i_L$ or $p_R$), unless there is more pertinent notation in a particular context (such as in our study of genuine $G$-spectra). For simplicity and readability we do not decorate these symbols further, so that in a single expression (e.g.\! a composite functor) these various symbols may be referring to \textit{different} recollements -- some of which may not even have been explicitly indicated.  
We hope that the meanings of these functors are always made clear by the context.
\end{warning}

\subsection{Closed subcategories}
\label{subsection.closed.subcats}

In this subsection, we study some basic properties of closed subcategories (a.k.a. closed noncommutative substacks).

\begin{definition}
\label{defn.clsd.subcat}
For simplicity, here we use the term \bit{closed subcategory} of $\cX$ in place of the term ``closed noncommutative substack'' of $\cX$  (\Cref{defn.closed.nc.substack.intro}).  We write $\Cls_\cX$ for the poset of closed subcategories of $\cX$ ordered by inclusion.
\end{definition}

\begin{example}
\label{ex.cpct.objs.gen.clsd.subcat}
Given a set $\{ K_s \in \cX^\omega \}_{s \in S}$ of compact objects of $\cX$, the full stable subcategory
that they generate under colimits is a closed subcategory of $\cX$: the restricted Yoneda embedding commutes with filtered colimits, and hence admits a further right adjoint.
\end{example}

\begin{notation}
\label{notn.clsd.subcat.gend.by.cpct.objs}
In the situation of \Cref{ex.cpct.objs.gen.clsd.subcat}, we write
\[ \brax{K_s }_{s \in S} \in \Cls_\cX \]
for the closed subcategory of $\cX$ generated by the objects $\{K_s \in \cX^\omega \}_{s \in S}$.
\end{notation}

\begin{notation}
Given a full presentable stable subcategory $\cZ \subseteq \cX$, we write
\[ \cZ^\bot := \{ U \in \cX : \ulhom_\cX(Z,U) \simeq 0 \textup{ for all } Z \in \cZ \} \subseteq \cX \]
for its right-orthogonal subcategory.
\end{notation}

\begin{observation}
\label{obs.right.orthogonal.of.full.presentable.stable.subcat.is.presentable.quotient}
A full presentable stable subcategory $\cZ \subseteq \cX$ determines a diagram
\begin{equation}
\label{quotient.by.a.full.presentable.stable.subcategory}
\begin{tikzcd}[column sep=1.5cm]
\cZ
\arrow[hook, yshift=0.9ex]{r}{i}
\arrow[leftarrow, yshift=-0.9ex]{r}[yshift=-0.2ex]{\bot}[swap]{i^R}
&
\cX
\arrow[yshift=0.9ex]{r}{j^L}
\arrow[hookleftarrow, yshift=-0.9ex]{r}[yshift=-0.2ex]{\bot}[swap]{j}
&
\cZ^\bot
\end{tikzcd}
\end{equation}
in $\Cat$, in which the functors $i$ and $j$ are the defining fully faithful inclusions and the functor $j^L$ is determined by the formula
\[
j j^L
\simeq
\cofib
\left(
i i^R
\xlongra{\varepsilon}
\id_\cX
\right)
~.
\]
Moreover, the commutative square
\[ \begin{tikzcd}
\cZ
\arrow[hook]{r}{i}
\arrow{d}
&
\cX
\arrow{d}{j^L}
\\
0
\arrow[hook]{r}
&
\cZ^\bot
\end{tikzcd} \]
is a pushout square in $\PrLSt$: given a morphism
\[
\cX
\xlongra{F}
\cY
\]
in $\PrLSt$ such that $Fi \simeq 0$, we obtain a colimit-preserving factorization
\[ \begin{tikzcd}
\cX
\arrow{r}{j^L}
\arrow{rd}[sloped, swap]{F}
&
\cZ^\perp
\arrow[dashed]{d}{Fj}
\\
&
\cY
\end{tikzcd}~. \]
\end{observation}

\begin{definition}
\label{defn.presentable.quotient}
In light of \Cref{obs.right.orthogonal.of.full.presentable.stable.subcat.is.presentable.quotient}, given a full presentable stable subcategory $\cZ \subseteq \cX$, we write
\[
\cX / \cZ
:=
\cZ^\bot
\]
for its right-orthogonal subcategory and refer to it as the \bit{presentable quotient} of $\cX$ by $\cZ$.
\end{definition}

\begin{observation}
\label{obs.clsd.subcat.gives.recollement}
In the special case of \Cref{obs.right.orthogonal.of.full.presentable.stable.subcat.is.presentable.quotient} where $\cZ \in \Cls_\cX$ is a closed subcategory, diagram \Cref{quotient.by.a.full.presentable.stable.subcategory} (lies in $\PrLSt$ and therefore) extends to a recollement \Cref{recollement.in.intro} in which
\[
i_L := i
~,
\qquad
y := i^R
~,
\qquad
\cU := \cZ^\perp =: \cX / \cZ
~,
\qquad
p_L := j^L
~,
\qquad
\text{and}
\qquad
\nu := j
~;
\]
the functors $p_L$ and $p_R$ are respectively determined by the formulas
\[
\nu p_L
\simeq
\cofib
\left(
i_L y
\xlongra{\varepsilon}
\id_\cX
\right)
\qquad
\text{and}
\qquad
\nu p_R
\simeq
\fib
\left(
\id_\cX
\xlongra{\eta}
i_R y
\right)
~.
\]
Conversely, any recollement \Cref{recollement.in.intro} arises in this way: the functor $i_L$ is the inclusion of a closed subcategory and the functor $\nu$ is the inclusion of its right-orthogonal subcategory.
\end{observation}

\begin{observation}
Inclusions of closed subcategories are stable under composition.  Also, if $\cZ,\cY \in \Cls_\cX$ with $\cZ \subseteq \cY \subseteq \cX$, then $\cZ \in \Cls_\cY$.  We use these facts implicitly without further comment.\footnote{These facts are amplified in \Cref{subsection.fund.opns.on.aligned.subcats}.}
\end{observation}

\begin{observation}
\label{obs.closed.subcats.closed.under.colimit.closure}
Let $\cZ \subseteq \cX$ be a full stable subcategory that is closed under colimits.  Then, $\cZ$ is a closed subcategory of $\cX$ if and only if its right-orthogonal subcategory $\cZ^\bot \subseteq \cX$ is also closed under colimits.  It follows that for any set $\{ \cZ_s \in \Cls_\cX \}_{s \in S}$ of closed subcategories of $\cX$, the full stable subcategory of $\cX$ that they generate under colimits is also a closed subcategory of $\cX$.\footnote{To show this, writing $\cZ \subseteq \cX$ for the full stable subcategory generated under colimits by the subcategories $\{ \cZ_s \}_{s \in S}$, it suffices to show that $\cZ^\bot = \bigcap_{s \in S} ( (\cZ_s)^\bot )$.  It is immediate that $\cZ^\bot \subseteq \bigcap_{s \in S} ( (\cZ_s)^\bot)$.  To verify the inclusion $\cZ^\bot \supseteq \bigcap_{s \in S} ((\cZ_s)^\bot)$, we observe that this intersection of subcategories of $\cX$ may be computed as a limit in $\PrR$, and therefore its inclusion into $\cX$ admits a left adjoint that evidently annihilates all objects of $\cZ$.}  We use this fact implicitly without further comment.
\end{observation}

\begin{notation}
\label{notn.clsd.subcat.gend.by.clsd.subcats}
Concordantly with \Cref{notn.clsd.subcat.gend.by.cpct.objs}, given a set $\{ \cZ_s \in \Cls_\cX \}_{s \in S}$ of closed subcategories of $\cX$, we write
\[ \brax{ \cZ_s }_{s \in S} \in \Cls_\cX \]
for the closed subcategory of $\cX$ that they generate under colimits, i.e.\! the colimit of the functor $S \xra{\cZ_\bullet} \Cls_\cX$.\footnote{In \Cref{subsection.intro.detailed.overview}, this was written as $\bigcup_{s \in S} \cZ_s \simeq \colim(S \xra{\cZ_\bullet} \Cls_\cX)$, so as to highlight the analogy with the union of closed subsets of a scheme. Outside of that section, we use the notation $\brax{\cZ_s}_{s \in S}$ because it is more compact.}
\end{notation}

\begin{remark}
Closed subcategories of presentable stable $\infty$-category behave much like closed subsets of a topological space, but they are not completely analogous.  For instance, increasing unions in the poset $\Cls_\cX$ commute with the forgetful functor to $\PrLSt$, whereas increasing unions in the poset of closed subsets of a topological space do not generally commute with the forgetful functor to topological spaces.
\end{remark}

\subsection{Stratifications}
\label{subsections.stratns}

In this subsection, we recall the definitions (originally given in \Cref{subsection.intro.stratn.nc.stacks}) of a stratification and of its strata.

\begin{definition}
A \bit{prestratification} of $\cX$ over $\pos$ is a functor
\[
\begin{tikzcd}[row sep=0cm]
\pos
\arrow{r}{\cZ_\bullet}
&
\Cls_\cX
\\
\rotatebox{90}{$\in$}
&
\rotatebox{90}{$\in$}
\\
p
\arrow[maps to]{r}
&
\cZ_p
\end{tikzcd}
\]
such that $\cX = \brax{\cZ_p}_{p \in \pos}$.
\end{definition}

\begin{notation}
Given a prestratification $\cZ_\bullet$ of $\cX$ over $\pos$, for any $\sD \in \Down_\pos$ we write
\[ \cZ_\sD := \brax{ \cZ_p }_{p \in \sD} \in \Cls_\cX~. \]
Note that $\cZ_{^\leq p} = \cZ_p$; we use the latter notation for simplicity.  Note too that $\cZ_\es = 0$.
\end{notation}

\begin{definition}
\label{defn.stratn}
A prestratification $\cZ_\bullet$ of $\cX$ over $\pos$ is a \bit{stratification} if it satisfies the following \bit{stratification condition}: for any $p,q \in \pos$, there exists a factorization
\[ \begin{tikzcd}
\cZ_{(^\leq p ) \cap (^\leq q)}
\arrow[hook]{r}{i_L}
&
\cZ_p
\\
\cZ_q
\arrow[hook]{r}[swap]{i_L}
\arrow[dashed]{u}
&
\cX
\arrow{u}[swap]{y}
\end{tikzcd}
~.
\]
\end{definition}


\begin{remark}
In the stratification condition, the upper functor $i_L$ is a monomorphism (in fact it is the inclusion of a closed subcategory, as indicated by the notation), and so if there exists a factorization then it is unique.  Moreover, if the stratification condition holds, then its factorization is necessarily the right adjoint
\[ \begin{tikzcd}[column sep=1.5cm]
\cZ_{(^\leq p) \cap (^\leq q)}
\arrow[hook, transform canvas={yshift=0.9ex}]{r}{i_L}
\arrow[dashed, leftarrow, transform canvas={yshift=-0.9ex}]{r}[yshift=-0.2ex]{\bot}[swap]{y}
&
\cZ_q
\end{tikzcd}~; \]
this follows from \Cref{lem.equivalent.characterizations.of.alignment}.
\end{remark}

\begin{observation}
\label{obs.condn.star.vacuous.if.P.totally.ordered}
The stratification condition is automatic if $p \leq q$ or if $q \leq p$.  In particular, in the case that the poset $\pos$ is totally ordered, every prestratification of $\cX$ over $\pos$ is a stratification.
\end{observation}

\begin{definition}
\label{defn.Cth.stratum.and.geometric.localizn}
Suppose that $\cZ_\bullet$ is a prestratification of $\cX$ over $\pos$, and suppose that $\sC \in \Conv_\pos$.
\begin{enumerate}

\item

The \bit{$\sC\th$ stratum} of the prestratification is the presentable quotient
\[ \cX_\sC :=
 \cZ_{^\leq \sC} \left\slash \cZ_{^< \sC} \right.
~.
\]

\item The \bit{$\sC\th$ geometric localization functor} is the left adjoint in the composite adjunction
\[
\begin{tikzcd}[column sep=1.5cm]
\Phi_\sC
:
\cX
\arrow[transform canvas={yshift=0.9ex}]{r}{y}
\arrow[hookleftarrow, transform canvas={yshift=-0.9ex}]{r}[yshift=-0.2ex]{\bot}[swap]{i_R}
&
\cZ_\sC
\arrow[transform canvas={yshift=0.9ex}]{r}{p_L}
\arrow[hookleftarrow, transform canvas={yshift=-0.9ex}]{r}[yshift=-0.2ex]{\bot}[swap]{\nu}
&
\cX_\sC
:
\rho^\sC
\end{tikzcd}
~.
\]
\end{enumerate}
We also write
\[
L_\sC
:
\cX
\xra{\Phi_\sC}
\cX_\sC
\xlonghookra{\rho^\sC}
\cX
\]
for the composite endofunctor, and we write
\[
\id_\cX
\xra{\eta_\sC}
L_\sC
\]
for the unit morphism.
\end{definition}

\begin{remark}
Considering an element $p \in \pos$ as a convex subset of $\pos$, \Cref{defn.Cth.stratum.and.geometric.localizn} specializes to \Cref{defn.intro.strata.and.geometric.localization.adjunction} of the $p\th$ stratum of a stratification.
\end{remark}

\begin{remark}
For any $\sD \in \Down_\pos \subseteq \Conv_\pos$, the functor $\cZ_\sD \xra{p_L} \cX_\sD$ is an equivalence.  We use both of these notations, depending on the context: we use the notation $\cZ_\sD$ when we mean to consider this as a subcategory of $\cX$ via the inclusion $i_L$, while we use the notation $\cX_\sD$ when we mean to consider this as a subcategory of $\cX$ via the inclusion $\rho^\sD$ (which coincides with $i_R$ in this special case).
\end{remark}

\subsection{The macrocosm reconstruction theorem}
\label{subsection.reconstrn.thm.for.stratns}

This subsection is centered around the macrocosm reconstruction theorem (\Cref{macrocosm.thm}), which we prove using the metacosm reconstruction theorem (which is itself proved in \Cref{section.reconstrn}).  We unpack the entire theory in the case that $\pos = [2]$ in \Cref{ex.gluing.stuff.over.brax.two}.

\begin{local}
In this subsection, we fix a stratification $\cZ_\bullet$ of $\cX$ over $\pos$.
\end{local}

\begin{remark}
We use the language of \textit{modules} to discuss certain definitions and constructions.  This is explained in detail in \Cref{section.lax.actions.and.limits}.  In the interest of keeping the main body of this work relatively self-contained, we summarize the essential points here.
\begin{itemize}
\item By a left/right module over an $\infty$-category, we mean a co/cartesian fibration over it, or equivalently a functor from it(s opposite) to $\Cat$.
\item These modules become \textit{lax} when our fibrations are only \textit{locally} co/cartesian, which (definitionally) correspond to left/right-lax functors to $\Cat$.
\item One can take the strict, left-lax, or right-lax limit of any module (regardless of whether that module is itself strict or left/right-lax).
\item The specific construction that is relevant for us here is the right-lax limit of a left-lax module; the precise definition (in our case of interest) is recalled in \Cref{rmk.recall.defn.of.rlax.lim.of.llax.action}.
\end{itemize}
\end{remark}

\begin{definition}
For any $p,q \in \pos$, the corresponding \bit{gluing functor} is the composite
\[
\Gamma^p_q
:
\cX_p
\xlonghookra{\rho^p}
\cX
\xra{\Phi_q}
\cX_q
~.
\]
\end{definition}

\begin{remark}
\label{rmk.gluing.fctrs.even.if.no.relation.in.poset}
By the stratification condition, the functor $\cX_p \xra{\Gamma^p_q} \cX_q$ is zero whenever $p \not\leq q$.
\end{remark}

\begin{notation}
We define the full subcategory
\[
\GD(\cX)
:=
\{ (X,p) \in \cX \times \pos
:
X \in \cX_p
\}
\subseteq
\cX \times \pos
~,
\]
which we consider as an object of $\Cat_{/\pos}$.
\end{notation}

\begin{observation}
\label{obs.GD.is.a.llax.P.mod}
The functor
\[
\GD(\cX)
\longra
\pos
\]
is a locally cocartesian fibration, whose monodromy functor over each morphism $p \ra q$ in $\pos$ is the gluing functor
\[
\cX_p
\xra{\Gamma^p_q}
\cX_q
~.
\]
We therefore consider it as defining a left-lax left $\pos$-module
\[
\GD(\cX)
\in
\LMod_{\llax.\pos}
:=
\loc.\coCart_\pos
~.
\]
\end{observation}

\begin{definition}
\label{defn.gluing.diagram}
We refer to the left-lax left $\pos$-module
\[
\GD(\cX)
\in
\LMod_{\llax.\pos}
\]
of \Cref{obs.GD.is.a.llax.P.mod} as the \bit{gluing diagram} of the stratification.
\end{definition}

\begin{definition}
The \bit{glued $\infty$-category} of the stratification is the right-lax limit
\[
\Glue(\cX)
:=
\lim^\rlax_{\llax.\pos} ( \GD(\cX) )
~.
\]
\end{definition}

\begin{remark}
\label{rmk.recall.defn.of.rlax.lim.of.llax.action}
For the reader's convenience, we unpack the definition of the glued $\infty$-category $\Glue(\cX)$. First of all, we write $\sd(\pos)$ for the \textit{subdivision} of $\pos$: the poset of finite nonempty linearly ordered subsets of $\pos$ (\Cref{defn.subdivision.of.a.poset}). Moreover, the functor
\[
\sd(\pos)
\xra{\max}
\pos
\]
carrying each subset to its maximal element is a locally cocartesian fibration (\Cref{t4}\Cref{t4.part.one}), with nontrivial cocartesian monodromy functors given by adjoining new maximal elements. Then, by \Cref{prop.rlax.llaxB.via.sd} we have an identification
\[
\Glue(\cX)
:=
\lim^\rlax_{\llax.\pos} \left( \GD(\cX) \right)
\simeq
\Fun^\cocart_{/\pos} \left( \sd(\pos) , \GD(\cX) \right)
~;
\]
that is, the glued $\infty$-category $\Glue(\cX)$ is equivalent to that of morphisms
\begin{equation}
\label{typical.object.in.rlax.lim.of.gluing.diagram}
\begin{tikzcd}
\sd(\pos)
\arrow[dashed]{rr}
\arrow{rd}[sloped, swap]{\max}
&
&
\GD(\cX)
\arrow{ld}
\\
&
\pos
\end{tikzcd}
\end{equation}
in $\loc.\coCart_\pos$ (i.e.\! functors over $\pos$ that preserve cocartesian morphisms thereover).
\end{remark}

\begin{observation}
\label{obs.glued.cat.is.subcat.of.fctrs.from.sdP}
We can consider the glued $\infty$-category as a full subcategory
\[
\Glue(\cX)
\subseteq
\Fun(\sd(\pos),\cX)
\]
via the composite fully faithful embedding
\[
\hspace{-1cm}
\Glue(\cX)
:=
\lim^\rlax_{\llax.\pos}(\GD(\cX))
\simeq
\Fun^\cocart_{/\pos}(\sd(\pos),\GD(\cX))
\xlonghookra{\ff}
\Fun_{/\pos}(\sd(\pos),\GD(\cX))
\xlonghookra{\ff}
\Fun_{/\pos}(\sd(\pos),\ul{\cX})
\simeq
\Fun(\sd(\pos),\cX)
~.
\]
Explicitly, its image consists of those functors
\[
\sd(\pos)
\xlongra{F}
\cX
\]
satisfying the following conditions.
\begin{enumerate}

\item\label{item.require.value.of.fctr.sdP.to.X.to.land.in.max.subcat}

For every $([n] \xra{\varphi} \pos) \in \sd(\pos)$, we have
\[
F(\varphi)
\in
\rho^{\max(\varphi)}(\cX_{\max(\varphi)})
\subseteq
\cX
~.
\]

\item\label{item.require.cocart.morphisms.in.sdP.go.to.local.equivces}

It carries each morphism in $\sd(\pos)$ of the form
\[ \begin{tikzcd}
{[n]}
\arrow[hook]{rr}{i \longmapsto i}
\arrow[hook]{rd}[sloped, swap]{\varphi}
&
&
{[n+1]}
\\
&
\pos
\arrow[hookleftarrow]{ru}[sloped, swap]{\psi}
\end{tikzcd} \]
(which are precisely the cocartesian morphisms with respect to the locally cocartesian fibration $\sd(\pos) \xra{\max} \pos$) to a morphism
\begin{equation}
\label{morphism.from.F.varphi.to.F.psi.coming.from.fctr.sdP.to.X}
F(\varphi)
\longra
F(\psi)
\end{equation}
in $\cX$ that becomes an equivalence after applying the functor $\cX \xra{\Phi_{\max(\psi)}} \cX_{\max(\psi)}$.\footnote{Assuming condition \Cref{item.require.value.of.fctr.sdP.to.X.to.land.in.max.subcat}, condition \Cref{item.require.cocart.morphisms.in.sdP.go.to.local.equivces} is equivalent to requiring that the morphism \Cref{morphism.from.F.varphi.to.F.psi.coming.from.fctr.sdP.to.X} witnesses $F(\psi)$ as the $L_{\max(\psi)}$-localization of $F(\varphi)$.}

\end{enumerate}
We use these facts without further comment.
\end{observation}

\begin{notation}
We write
\[
\lim_{\sd(\pos)}
:
\Glue(\cX)
\longhookra
\Fun(\sd(\pos),\cX)
\xra{\lim_{\sd(\pos)}}
\cX
\]
for the composite.
\end{notation}

\begin{observation}
\label{obs.formula.for.mocrocosm.gluing.functor.without.using.Lphi.notation}
The defining inclusion
\[
\ul{\cX}
\xlonghookla{\ff}
\GD(\cX)
\]
is a morphism in $\LMod^\llax_{\llax.\pos} := \Cat_{\loc.\cocart/\pos}$ (though not generally in $\LMod_{\llax.\pos} := \loc.\coCart_\pos$). Over each object $p \in \pos$, this is the right adjoint in the adjunction
\[ \begin{tikzcd}[column sep=1.5cm]
\cX
\arrow[transform canvas={yshift=0.9ex}]{r}{\Phi_p}
\arrow[hookleftarrow, transform canvas={yshift=-0.9ex}]{r}[yshift=-0.2ex]{\bot}[swap]{\rho^p}
&
\cX_p
\end{tikzcd}
~. \]
By \Cref{lemma.ptwise.radjt.has.ptwise.ladjt}, the left adjoints $\Phi_p$ assemble into a morphism
\[
\const(\cX)
:=
\ul{\cX}
\longra
\GD(\cX)
\]
in $\LMod^\rlax_{\llax.\pos}$.  Through the definitional adjunction
\[
\begin{tikzcd}[column sep=1.5cm]
\Cat
\arrow[transform canvas={yshift=0.9ex}]{r}{\const}
\arrow[leftarrow, transform canvas={yshift=-0.9ex}]{r}[yshift=-0.2ex]{\bot}[swap]{\lim^\rlax_{\llax.\pos}}
&
\LMod^\rlax_{\llax.\pos}
\end{tikzcd}
~,
\]
this corresponds to a functor
\begin{equation}
\label{microcosm.gluing.functor.before.naming.it}
\cX
\longra
\lim^\rlax_{\llax.\pos}(\GD(\cX))
~.
\end{equation}
In terms of \Cref{obs.glued.cat.is.subcat.of.fctrs.from.sdP}, the functor \Cref{microcosm.gluing.functor.before.naming.it} is given by the formula
\begin{equation}
\label{formula.for.mocrocosm.gluing.functor.without.using.Lphi.notation}
X
\longmapsto
\left(
([n] \xra{\varphi} \pos)
\longmapsto
\rho_{\varphi(n)} \Phi_{\varphi(n)} \cdots \rho_{\varphi(0)} \Phi_{\varphi(0)} X
\right)
~.
\end{equation}
\end{observation}

\begin{definition}
We refer to the functor \Cref{microcosm.gluing.functor.before.naming.it} as the (\bit{microcosm}) \bit{gluing diagram functor}, and we denote it by
\[
\cX
\xlongra{\gd}
\Glue(\cX)
:=
\lim^\rlax_{\llax.\pos}(\GD(\cX))
~.
\]
\end{definition}

\begin{theorem}
\label{macrocosm.thm}
There is a canonical adjunction
\begin{equation}
\label{adjn.in.reconstrn.thm}
\begin{tikzcd}[column sep=1.5cm]
\cX
\arrow[transform canvas={yshift=0.9ex}]{r}{\gd}
\arrow[leftarrow, transform canvas={yshift=-0.9ex}]{r}[yshift=-0.2ex]{\bot}[swap]{\lim_{\sd(\pos)}}
&
\Glue(\cX)
\end{tikzcd}
~,
\end{equation}
which is an equivalence whenever $\pos$ is down-finite.
\end{theorem}

\begin{proof}
By \Cref{metacosm.thm}, the functor $\cX \xra{\gd} \Glue(\cX)$ defines a morphism in $\PrLSt$ (being the image under the forgetful functor $\Strat_\pos \ra \PrLSt$ of the unit of the adjunction \Cref{adjn.of.metacosm.thm}) and is an equivalence whenever $\pos$ is down-finite.  The identification of its right adjoint is contained in the proof of \Cref{metacosm.thm}.
\end{proof}

\begin{definition}
\label{defn.convergent.stratn}
We say that the stratification of $\cX$ over $\pos$ is \bit{convergent} if the adjunction \Cref{adjn.in.reconstrn.thm} is an equivalence.
\end{definition}

\begin{example}
\label{ex.gluing.stuff.over.brax.two}
Suppose that $\pos = [2]$.
\begin{enumerate}

\item

The gluing diagram of the stratification is the lax-commutative triangle
\[
\GD(\cX)
=
\left(
\begin{tikzcd}[row sep=1.5cm]
&
\cX_1
\arrow{rd}[sloped]{\Gamma^1_2}
\\
\cX_0
\arrow{ru}[sloped]{\Gamma^0_1}
\arrow{rr}[transform canvas={xshift=0.9ex, yshift=0.6cm}]{{\rotatebox{90}{$\Rightarrow$}}\eta_1}[swap]{\Gamma^0_2}
&
&
\cX_2
\end{tikzcd}
\right)
~,
\]
in which the natural transformation is the composite
\[
\eta_1
:
\Gamma^0_2 := \Phi_2 \rho^0 
\simeq
\Phi_2 \id_{\cX} \rho^0
\xra{\eta_1}
\Phi_2 \rho^1 \Phi_1 \rho^0
=:
\Gamma^1_2
\Gamma^0_1
~. \]

\item

An object of the glued $\infty$-category $\Glue(\cX)$ amounts to the data of the form
\begin{equation}
\label{object.of.glued.cat.when.poset.is.brax.two}
\begin{tikzcd}[row sep=1.25cm]
&
X_2
\arrow{rr}{\gamma^1_2}
\arrow{dd}[swap]{\gamma^0_2}
&
&
\Gamma^1_2(X_1)
\arrow{dd}{\Gamma^1_2(\gamma^0_1)}
\\
&
&
X_1
\arrow[maps to]{ru}
\\
&
\Gamma^0_2(X_0)
\arrow{rr}[swap, pos=0.3]{\eta_1}
&
&
\Gamma^1_2(\Gamma^0_1(X_0))
\\
X_0
\arrow[maps to]{rr}
\arrow[maps to]{ru}
&
&
\Gamma^0_1(X_0)
\arrow[crossing over, leftarrow]{uu}[pos=0.7]{\gamma^0_1}
\arrow[maps to]{ru}
\end{tikzcd}
~,\footnote{The locally cocartesian fibration $\sd([2]) \xra{\max} [2]$ is illustrated in \Cref{sd.of.brackets.2}.}
\end{equation}
where $X_i \in \cX_i$ for all $i \in [2]$.  One may think of the morphisms $\gamma^i_j$ as \textit{gluing morphisms} (i.e.\! 1-cubes) for this object of $\Glue(\cX)$, and of the commutative square in $\cX_2$ as higher-dimensional gluing data, namely a \textit{gluing square} $\gamma_{012}$.\footnote{The notation $X_j \xra{\gamma^i_j} \Gamma^i_j(X_i)$ for the gluing morphisms is chosen so as to parallel the notation $\cX_i \xra{\Gamma^i_j} \cX_j$ for the gluing functors.  More generally, each conservative functor $[n] \xra{p_\bullet} \pos$ determines a gluing $n$-cube $\gamma_{p_0,\ldots,p_n}$ that is part of the data of an object of the glued $\infty$-category.}

\item

Given an object $X \in \cX$, its gluing diagram is the object $\gd(X) = \Cref{object.of.glued.cat.when.poset.is.brax.two} \in \Glue(\cX)$ in which
\begin{itemize}
\item
$X_i = \Phi_i(X)$ for all $0 \leq i \leq 2$,
\item
the gluing morphism $\gamma^i_j$ is the unit morphism
\[
X_j
:=
\Phi_j(X)
\xra{\eta_i}
\Phi_j(\rho^i(\Phi_i(X)))
=:
\Gamma^i_j(X_i)
\]
for all $0 \leq i < j \leq 2$, and
\item the commutativity of the gluing square $\gamma_{012}$ follows from the commutativity of the square
\[ \begin{tikzcd}[row sep=1.5cm]
\Phi_2
\arrow{r}{\eta_1}
\arrow{d}[swap]{\eta_0}
&
\Phi_2 \rho^1 \Phi_1
\arrow{d}{\eta_0}
\\
\Phi_2 \rho^0 \Phi_0
\arrow{r}[swap]{\eta_1}
&
\Phi_2 \rho^1 \Phi_1 \rho^0 \Phi_0
\end{tikzcd} \]
in $\Fun(\cX,\cX_2)$.
\end{itemize}

\item\label{item.microcosm.in.example.of.brax.two}

Because $\pos=[2]$ is finite and hence down-finite, \Cref{macrocosm.thm} guarantees that each $X \in \cX$ is the limit of its gluing diagram: the equivalence
\[
X
\xlongra{\sim}
\lim_{\sd(\pos)}(\gd(X))
\]
amounts to the limit diagram
\[ \begin{tikzcd}[row sep=1.25cm, column sep=0.5cm]
&
L_2(X)
\arrow{rr}
\arrow{dd}
&
&
L_2(L_1(X))
\arrow{dd}
\\
X
\arrow[crossing over]{rr}
\arrow{ru}
\arrow{dd}
&
&
L_1(X)
\arrow{ru}
\\
&
L_2(L_0(X))
\arrow{rr}
&
&
L_2(L_1(L_0(X)))
\\
L_0(X)
\arrow{rr}
\arrow{ru}
&
&
L_1(L_0(X))
\arrow{ru}
\arrow[leftarrow, crossing over]{uu}
\end{tikzcd} ~. \]

\end{enumerate}
\end{example}

\subsection{The microcosm and nanocosm morphisms}
\label{subsection.microcosm.and.nanocosm.morphisms}

In this subsection, we discuss the microcosm and nanocosm morphisms. In particular, we give a detailed description of the nanocosm morphism in \Cref{rmk.nanocosm}.

\begin{notation}
\label{notn.iterated.Gamma.and.L.for.elt.of.sdP}
For any $([n] \xra{\varphi} \pos) \in \sd(\pos)$, we write
\[
\Gamma_\varphi
:=
\Gamma^{\varphi(n-1)}_{\varphi(n)} \cdots \Gamma^{\varphi(0)}_{\varphi(1)}
\qquad
\text{and}
\qquad
L_\varphi
:=
L_{\varphi(n)} \cdots L_{\varphi(0)}
~.
\]
\end{notation}

\begin{observation}
By definition, for any $([n] \xra{\varphi} \pos) \in \sd(\pos)$ the functors $\Gamma_\varphi$ and $L_\varphi$ participate in the commutative diagram
\[ \begin{tikzcd}
\cX
\arrow{r}{L_\varphi}
\arrow{d}[swap]{\Phi_{\varphi(0)}}
&
\cX
\\
\cX_{\varphi(0)}
\arrow{r}[swap]{\Gamma_\varphi}
&
\cX_{\varphi(n)}
\arrow[hook]{u}[swap]{\rho^{\varphi(n)}}
\end{tikzcd}~. \]
We use this fact without further comment.
\end{observation}

\begin{observation}
\label{obs.functoriality.of.L.phi.in.the.variable.phi}
By definition, the functors $L_\varphi$ for $\varphi \in \sd(\pos)$ are the values of a functor
\begin{equation}
\label{L.bullet.as.a.functor.to.coaugmented.endofunctors.of.X}
\sd(\pos)
\xra{L_\bullet}
\Fun^\ex(\cX,\cX)_{\id_\cX/}
~,
\end{equation}
namely the adjunct of the composite
\[
\cX
\xlongra{\gd}
\Glue(\cX)
\longhookra
\Fun(\sd(\pos),\cX)
\]
equipped with its coaugmentation given by the unit of the adjunction \Cref{adjn.in.reconstrn.thm} of \Cref{macrocosm.thm}.
\end{observation}

\begin{remark}
\label{rmk.shorter.version.of.formula.for.mocrocosm.gluing.functor.using.Lphi.notation}
Using \Cref{notn.iterated.Gamma.and.L.for.elt.of.sdP}, the formula \Cref{formula.for.mocrocosm.gluing.functor.without.using.Lphi.notation} for the composite
\[
\cX
\xlongra{\gd}
\Glue(\cX)
:=
\lim^\rlax_{\llax.\pos}(\GD(\cX))
\longhookra
\Fun(\sd(\pos),\cX)
\]
can be expressed more compactly as
\[
X
\longmapsto
\left(
\varphi
\longmapsto
L_\varphi X
\right)
~.
\]
\end{remark}

\begin{observation}
\label{obs.assemble.Gamma.phi.functorially.in.phi}
For each nonidentity morphism $p < q$ in $\pos$, the functors $\cX_p \xra{\Gamma_\varphi} \cX_q$ for $\varphi \in \sd(\pos)^{|p}_{|q}$ are the values of a factorization
\[ \begin{tikzcd}[row sep=1.5cm, column sep=1.5cm]
\sd(\pos)^{|p}_{|q}
\arrow[dashed]{rr}{\Gamma_\bullet}
&
&
\Fun^\ex(\cX_p,\cX_q)
\\
\sd(\pos)
\arrow{r}[swap]{L_\bullet}
\arrow[hookleftarrow]{u}
&
\Fun^\ex(\cX,\cX)
\arrow{r}[swap]{- \circ \rho^p}
&
\Fun^\ex(\cX_p,\cX)
\arrow[hookleftarrow]{u}[swap]{\rho^q \circ -}
\end{tikzcd}
~.
\]
\end{observation}

\needspace{2\baselineskip}
\begin{definition}
Fix any object $X \in \cX$.
\begin{enumerate}

\item

We define the \bit{reglued object} of $X$ to be
\[
\glue(X)
:=
\lim_{\sd(\pos)} ( \gd(X))
\in
\cX
~.
\]

\item

We define the \bit{microcosm morphism} of $X$ to be the unit morphism
\[
X
\longra
\glue(X)
\]
in $\cX$ of the adjunction \Cref{adjn.in.reconstrn.thm} of \Cref{macrocosm.thm}.

\item

For any $Y \in \cX$, we define the corresponding \bit{nanocosm morphism} to be the composite morphism
\begin{align*}
\ulhom_\cX(Y,X)
& \longra
\ulhom_\cX(Y,\glue(X))
\\
& \simeq
\lim_{\varphi \in \sd(\pos)} \ulhom_\cX(Y,L_\varphi(X))
\\
& \simeq
\lim_{([n] \xra{\varphi} \pos) \in \sd(\pos)}
\left( \ulhom_{\cX_{\varphi(n)}} ( \Phi_{\varphi(n)}Y , \Gamma_\varphi \Phi_{\varphi(0)} X ) \right)
\end{align*}
obtained by applying $\ulhom_\cX(Y,-)$ to the microcosm morphism of $X$.

\end{enumerate}
\end{definition}

\begin{remark}
\label{rmk.nanocosm}
For any objects $X,Y \in \cX$, we unpack the nanocosm morphism
\[
\ulhom_\cX(Y,X)
\longra
\lim_{([n] \xra{\varphi} \pos) \in \sd(\pos)}
\left( \ulhom_{\cX_{\varphi(n)}} ( \Phi_{\varphi(n)}Y , \Gamma_\varphi \Phi_{\varphi(0)} X ) \right)
\]
as follows. First of all, postcomposing with the canonical morphism to the $([n] \xra{\varphi} \pos)\th$ constituent of the limit, we obtain the composite
\[
\ulhom_\cX(Y,X)
\longra
\ulhom_\cX(Y,L_\varphi X)
\simeq
\ulhom_{\cX_{\varphi(n)}} ( \Phi_{\varphi(n)} Y , \Gamma_\varphi \Phi_{\varphi(0)} X)
~.
\]
The functoriality of the diagram
\begin{equation}
\label{sd.P.indexed.diagram.of.spectra}
\sd(\pos)
\xra{([n] \xra{\varphi} \pos) \longmapsto \ulhom_{\cX_{\varphi(n)}} ( \Phi_{\varphi(n)} Y , \Gamma_\varphi \Phi_{\varphi(0)} X)}
\Spectra
\end{equation}
may be described informally as follows.  Observe that every morphism in $\sd(\pos)$ factors as a composite of morphisms whose images under the forgetful functor $\sd(\pos) \ra \bDelta$ are all coface maps $[n] \xra{\delta^i} [n+1]$ (for some $n \geq 0$ and some $0 \leq i \leq n+1$), so it suffices to describe the functoriality of the diagram \Cref{sd.P.indexed.diagram.of.spectra} on such morphisms.  So, let us fix a morphism
\begin{equation}
\label{coface.map.in.sd.P}
\begin{tikzcd}
{[n]}
\arrow[hook]{rr}{\delta^i}
\arrow[hook]{rd}[swap, sloped]{\varphi}
&
&
{[n+1]}
\\
&
\pos
\arrow[hookleftarrow]{ru}[swap, sloped]{\w{\varphi}}
\end{tikzcd}
\end{equation}
in $\sd(\pos)$, and describe the morphism
\begin{equation}
\label{value.on.a.morphism.of.sd.P.indexed.diagram.of.spectra}
\ulhom_{\cX_{\varphi(n)}}(\Phi_{\varphi(n)} Y , \Gamma_\varphi \Phi_{\varphi(0)} X)
\longra
\ulhom_{\cX_{\w{\varphi}(n+1)}}(\Phi_{\w{\varphi}(n+1)} Y , \Gamma_{\w{\varphi}} \Phi_{\w{\varphi}(0)} X)
\end{equation}
in $\Spectra$ which is the image of the morphism \Cref{coface.map.in.sd.P} under the functor \Cref{sd.P.indexed.diagram.of.spectra}.
\begin{itemize}

\item If $i=0$, then $\varphi(n)= \w{\varphi}(n+1)$ and the morphism \Cref{value.on.a.morphism.of.sd.P.indexed.diagram.of.spectra} is obtained by postcomposition with the morphism
\[ \begin{tikzcd}[row sep=1.5cm]
&[-7.1cm]
\Gamma_\varphi \Phi_{\varphi(0)} X
:=
\Gamma^{\varphi(n-1)}_{\varphi(n)} \cdots \Gamma^{\varphi(0)}_{\varphi(1)} \Phi_{\varphi(0)} X
=
&[-1.2cm]
\Gamma^{\w{\varphi}(n)}_{\w{\varphi}(n+1)} \cdots \Gamma^{\w{\varphi}(1)}_{\w{\varphi}(2)} \Phi_{\w{\varphi}(1)} X
\arrow{d}{\eta_{\w{\varphi}(0)}}
\\
\Gamma_{\w{\varphi}} \Phi_{\w{\varphi}(0)} X :=
&
&
\Gamma^{\w{\varphi}(n)}_{\w{\varphi}(n+1)} \cdots \Gamma^{\w{\varphi}(1)}_{\w{\varphi}(2)} \Phi_{\w{\varphi}(1)} L_{\w{\varphi}(0)} X
\end{tikzcd} \]
in $\cX_{\varphi(n)} = \cX_{\w{\varphi}(n+1)}$.

\item If $1 \leq i \leq n$, then $\varphi(n) = \w{\varphi}(n+1)$ and the morphism \Cref{value.on.a.morphism.of.sd.P.indexed.diagram.of.spectra} is obtained by postcomposition with the morphism
\[ \begin{tikzcd}[row sep=1.5cm]
&[-6.7cm]
\Gamma_\varphi \Phi_{\varphi(0)} X
:=
\Gamma^{\varphi(n-1)}_{\varphi(n)} \cdots \Gamma^{\varphi(0)}_{\varphi(1)} \Phi_{\varphi(0)} X
=
&[-1.5cm]
\Gamma^{\w{\varphi}(n)}_{\w{\varphi}(n+1)} \cdots \Gamma^{\w{\varphi}(i-1)}_{\w{\varphi}(i+1)} \cdots \Gamma^{\w{\varphi}(0)}_{\w{\varphi}(1)} \Phi_{\w{\varphi}(0)} X
\arrow{d}{\eta_{\w{\varphi}(i)}}
\\
\Gamma_{\w{\varphi}} \Phi_{\w{\varphi}(0)} X
:=
&
&
\Gamma^{\w{\varphi}(n)}_{\w{\varphi}(n+1)} \cdots \Gamma^{\w{\varphi}(i)}_{\w{\varphi}(i+1)} \Gamma^{\w{\varphi}(i-1)}_{\w{\varphi}(i)} \cdots \Gamma^{\w{\varphi}(0)}_{\w{\varphi}(1)} \Phi_{\w{\varphi}(0)} X
\end{tikzcd} \]
in $\cX_{\varphi(n)} = \cX_{\w{\varphi}(n+1)}$.

\item If $i=n+1$, then the morphism \Cref{value.on.a.morphism.of.sd.P.indexed.diagram.of.spectra} is the composite
\[ \begin{tikzcd}[row sep=1.5cm]
\ulhom_{\cX_{\varphi(n)}}(\Phi_{\varphi(n)} Y , \Gamma_\varphi \Phi_{\varphi(0)} X) =
&[-2.3cm]
\ulhom_{\cX_{\w{\varphi}(n)}}(\Phi_{\w{\varphi}(n)} Y , \Gamma_{\w{\varphi}|_{[n]}} \Phi_{\w{\varphi}(0)} X)
\arrow{d}{\Gamma^{\w{\varphi}(n)}_{\w{\varphi}(n+1)}}
\\
&
\ulhom_{\cX_{\w{\varphi}(n+1)}}( \Gamma^{\w{\varphi}(n)}_{\w{\varphi}(n+1)} \Phi_{\w{\varphi}(n)} Y , \Gamma^{\w{\varphi}(n)}_{\w{\varphi}(n+1)} \Gamma_{\w{\varphi}|_{[n]}} \Phi_{\w{\varphi}(0)} X)
\\[-1.5cm]
&
\rotatebox{90}{$=$}
\\[-1.5cm]
&
\ulhom_{\cX_{\w{\varphi}(n+1)}}( \Gamma^{\w{\varphi}(n)}_{\w{\varphi}(n+1)} \Phi_{\w{\varphi}(n)} Y , \Gamma_{\w{\varphi}} \Phi_{\w{\varphi}(0)} X)
\arrow{d}{\eta_{\w{\varphi}(n)}}
\\
&
\ulhom_{\cX_{\w{\varphi}(n+1)}}( \Phi_{\w{\varphi}(n+1)} Y , \Gamma_{\w{\varphi}} \Phi_{\w{\varphi}(0)} X)
\end{tikzcd} \]
in which the first morphism is obtained by applying the functor
\[
\cX_{\w{\varphi}(n)}
\xra{\Gamma^{\w{\varphi}(n)}_{\w{\varphi}(n+1)}}
\cX_{\w{\varphi}(n+1)}
\]
and the second morphism is obtained by precomposing with the morphism
\[
\Phi_{\w{\varphi}(n+1)} Y
\xra{\eta_{\w{\varphi}(n)}}
\Phi_{\w{\varphi}(n+1)} L_{\w{\varphi}(n)} Y
=
\Gamma^{\w{\varphi}(n)}_{\w{\varphi}(n+1)} \Phi_{\w{\varphi}(n)} Y
~.
\]
\end{itemize}
\end{remark}

\subsection{Strict objects}
\label{subsection.strict.objects}

In this brief subsection, we lay out the general theory of strict objects.

\begin{local}
In this subsection, we fix a stratification $\cZ_\bullet$ of $\cX$ over $\pos$.
\end{local}

\needspace{2\baselineskip}
\begin{definition}
\label{defn.strict.objects}
\begin{enumerate}
\item[]

\item We say that $X \in \cX$ is \bit{convergent} if its microcosm morphism
\[
X
\longra
\glue(X)
:=
\lim_{\sd(\pos)} ( \gd(X) )
\]
is an equivalence.

\item We say that
\[
F \in
\Glue(\cX)
\subseteq
\Fun(\sd(\pos),\cX)
\]
is \bit{strict} if it carries every isominmax morphism in $\sd(\pos)$ (\Cref{defn.iso.min.and.or.max}) to an equivalence in $\cX$.

\item\label{item.defn.of.strict.object} We say that $X \in \cX$ is \bit{strict} if it is convergent and moreover its gluing diagram $\gd(X) \in \Glue(\cX)$ is strict.

\end{enumerate}
\end{definition}

\begin{lemma}
\label{lem.sd.P.localizes.onto.TwAr.P}
The functor
\[ \begin{tikzcd}[row sep=0cm]
\sd(\pos)
\arrow{r}{(\min \ra \max)}
&
\TwAr(\pos)
\\
\rotatebox{90}{$\in$}
&
\rotatebox{90}{$\in$}
\\
([n] \xra{\varphi} \pos)
\arrow[maps to]{r}
&
(\varphi(0) \ra \varphi(n))
\end{tikzcd} \]
witnesses $\TwAr(\pos)$ as the localization of $\sd(\pos)$ with respect to the isominmax morphisms.
\end{lemma}

\begin{proof}
Let us write $\bW \subseteq \sd(\pos)$ for the subcategory on the isominmax morphisms, and for any $\cK \in \Cat$ let us write $\Fun(\cK,\sd(\pos))^\bW \subseteq \Fun(\cK,\sd(\pos))$ for the subcategory on the natural transformations that are componentwise in $\bW$. By \cite[Theorem 3.8]{AMG-rnerves}, it suffices to show that for every $n \geq 0$ the evident factorization
\[ \begin{tikzcd}[column sep=2.5cm]
\Fun([n],\sd(\pos))
\arrow{r}{\Fun([n],(\min \ra \max))}
&
\Fun([n],\TwAr(\pos))
\\
\Fun([n],\sd(\pos))^\bW
\arrow[hook]{u}
\arrow[dashed]{r}
&
\hom_\Cat([n],\TwAr(\pos))
\arrow[hook]{u}
\end{tikzcd} \]
is an $\infty$-groupoid completion, which follows from the observation that it admits a fully faithful left adjoint.
\end{proof}

\begin{observation}
\label{obs.strict.microcosm.gluing.diagram.iff.factors.from.sd.to.TwAr}
By \Cref{lem.sd.P.localizes.onto.TwAr.P}, an object $F \in \Glue(\cX) \subseteq \Fun(\sd(\pos),\cX)$ is strict if and only if it admits a factorization
\[ \begin{tikzcd}
\sd(\pos)
\arrow{r}{F}
\arrow{d}[swap]{(\min \ra \max)}
&
\cX
\\
\TwAr(\pos)
\arrow[dashed]{ru}[sloped, swap]{F}
\end{tikzcd} \]
(for which we use the same notation), in which case because localizations are initial we have a canonical equivalence
\[
\lim_{\sd(\pos)}(F)
\xlongla{\sim}
\lim_{\TwAr(\pos)}(F)
~.
\]
In particular, if $X \in \cX$ is strict, then we have a canonical equivalence
\[
X
\xlongra{\sim}
\lim_{\TwAr(\pos)}(\gd(X))
\]
and for any $Y \in \cX$ the nanocosm morphism reduces to an equivalence
\[
\ulhom_\cX(Y,X)
\xlongra{\sim}
\lim_{(p \ra q) \in \TwAr(\pos)} \ulhom_{\cX_q} ( \Phi_q Y , \Gamma^p_q \Phi_p X )
~.
\]
\end{observation}

\section{Fundamental operations}
\label{section.fund.opns}

In this section, we establish our fundamental operations on stratifications.  Towards this end, we first study certain fundamental operations on closed subcategories.  In particular, we introduce and study the notion of one closed subcategory being \textit{aligned} with another.  The notion of alignment allows us to state our fundamental operations on stratifications in greater generality than is done in \Cref{subsection.intro.detailed.overview} as \Cref{intro.thm.fund.opns}, while at the same time streamlining their proofs.  The assertions of \Cref{intro.thm.fund.opns} are recovered as a consequence of the fact that any two closed subcategories determined by a stratification are mutually aligned (\Cref{closed.subcats.are.mutually.aligned}).

This section is organized as follows.
\begin{itemize}

\item[\Cref{subsection.alignment}:] We introduce the notion of alignment and study its basic consequences.

\item[\Cref{subsection.fund.opns.on.aligned.subcats}:] We establish a number of fundamental operations on aligned subcategories.

\item[\Cref{subsection.glue.aligned}:] We establish excision- and Mayer--Vietoris-type gluing results for closed subcategories in the presence of alignment.

\item[\Cref{subsection.structure.theory}:] We prove our suite of fundamental operations on stratifications.

\end{itemize}

\begin{local}
In this section, we fix a presentable stable $\infty$-category $\cX$.
\end{local}

\subsection{Alignment}
\label{subsection.alignment}

In this subsection, we introduce the notion of alignment between closed subcategories and study its basic consequences. We also give an alternative characterization of alignment as \Cref{lem.equivalent.characterizations.of.alignment}.

\begin{local}
In this subsection, we fix two closed subcategories $\cY,\cZ \in \Cls_\cX$.
\end{local}

\begin{definition}
\label{defn.aligned}
We say that $\cZ$ is \bit{aligned} with $\cY$ if there exists a factorization
\[ \begin{tikzcd}
\cY \cap \cZ
\arrow[hook]{r}
&
\cY
\\
\cZ
\arrow[dashed]{u}
\arrow[hook]{r}[swap]{i_L}
&
\cX
\arrow{u}[swap]{y}
\end{tikzcd}
\]
through the intersection (with both the intersection and the factorization considered in $\Cat$).  To indicate that $\cZ$ is aligned with $\cY$, we write either $\cZ \algnd \cY$ or $\cY \algndfrom \cZ$.  We say that $\cY$ and $\cZ$ are \bit{mutually aligned} if $\cY$ is aligned with $\cZ$ and $\cZ$ is aligned with $\cY$, and in this case we write $\cY \mutalgnd \cZ$.  We write
\[ \begin{tikzcd}
&[-1.1cm]
\Cls_\cX^{\mutalgnd \cY}
\arrow[hook]{r}
\arrow[hook]{d}
&
\Cls_\cX^{\algnd \cY}
\arrow[hook]{d}
&[-1.1cm]
:= \{ \cW \in \Cls_\cX : \cW \algnd \cY \}
\\
\{ \cW \in \Cls_\cX : \cW \algndfrom \cY \} =:
&
\Cls_\cX^{\algndfrom \cY}
\arrow[hook]{r}
&
\Cls_\cX
\end{tikzcd} \]
for the evident pullback diagram among full subposets of $\Cls_\cX$.
\end{definition}

\begin{example}
\label{ex.of.one.directional.alignment}
The diagram
\[ \begin{tikzcd}[column sep=2cm, row sep=0cm]
\cY
\arrow[hook, transform canvas={yshift=0.9ex}]{r}{i_L}
\arrow[leftarrow, transform canvas={yshift=-0.9ex}]{r}[yshift=-0.2ex]{\bot}[swap]{y}
&
\cX
\arrow[hookleftarrow, transform canvas={yshift=0.9ex}]{r}{i_L}
\arrow[transform canvas={yshift=-0.9ex}]{r}[yshift=-0.2ex]{\bot}[swap]{y}
&
\cZ
\\
\rotatebox{90}{$=:$}
&
\rotatebox{90}{$=:$}
&
\rotatebox{90}{$=:$}
\\
\Spectra
\arrow[hook, transform canvas={yshift=0.9ex}]{r}{E \mapsto (0 \ra E)}
\arrow[leftarrow, transform canvas={yshift=-0.9ex}]{r}[yshift=-0.2ex]{\bot}[swap]{\ev_1}
&
\Fun([1],\Spectra)
\arrow[hookleftarrow, transform canvas={yshift=0.9ex}]{r}{(E \ra 0) \mapsfrom E}
\arrow[transform canvas={yshift=-0.9ex}]{r}[yshift=-0.2ex]{\bot}[swap]{\fib}
&
\Spectra
\end{tikzcd} \]
depicts the $i_L \adj y$ adjunctions of two closed subcategories $\cY,\cZ \in \Cls_\cX$.  Note that $\cY \cap \cZ = 0$.  The composite
\[
\cZ
\xlonghookra{i_L}
\cX
\xlongra{y}
\cY
\]
is zero, and so $\cZ$ is aligned with $\cY$.  On the other hand, the composite
\[
\cY
\xlonghookra{i_L}
\cX
\xlongra{y}
\cZ
\]
is given by desuspension, and so $\cY$ is not aligned with $\cZ$.
\end{example}

\begin{observation}
\label{obs.mutually.aligned.if.containment}
If either $\cY \subseteq \cZ$ or $\cY \supseteq \cZ$, then $\cY$ and $\cZ$ are mutually aligned.
\end{observation}

\begin{observation}
The pullback diagram
\begin{equation}
\label{no.notation.commutative.square.of.ladjts.among.Y.cap.Z.and.Y.and.Z.and.X}
\begin{tikzcd}
\cY \cap \cZ
\arrow[hook]{r}
\arrow[hook]{d}
&
\cY
\arrow[hook]{d}{i_L}
\\
\cZ
\arrow[hook]{r}[swap]{i_L}
&
\cX
\end{tikzcd}
\end{equation}
lies in $\PrLSt \subset \Cat$.  We use this fact without further comment.
\end{observation}

\begin{local}
\label{local.notn.for.incln.of.intersection.of.closed.subcats}
In this subsection, we use the notation
\begin{equation}
\label{commutative.square.of.ladjts.among.Y.cap.Z.and.Y.and.Z.and.X}
\begin{tikzcd}
\cY \cap \cZ
\arrow[hook]{r}{i_\cY}
\arrow[hook]{d}[swap]{i_\cZ}
&
\cY
\arrow[hook]{d}{i_L}
\\
\cZ
\arrow[hook]{r}[swap]{i_L}
&
\cX
\end{tikzcd}
\end{equation}
for the commutative square \Cref{no.notation.commutative.square.of.ladjts.among.Y.cap.Z.and.Y.and.Z.and.X} of left adjoints, and we use the notation
\begin{equation}
\label{commutative.square.of.radjts.among.Y.cap.Z.and.Y.and.Z.and.X}
\begin{tikzcd}
\cY \cap \cZ
&
\cY
\arrow{l}[swap]{i_\cY^R}
\\
\cZ
\arrow{u}{i_\cZ^R}
&
\cX
\arrow{u}[swap]{y}
\arrow{l}{y}
\end{tikzcd}
\end{equation}
for its corresponding commutative square of right adjoints.
\end{local}

\begin{lemma}
\label{lem.equivalent.characterizations.of.alignment}
The following are equivalent.
\begin{enumerate}

\item\label{part.lem.alignment.itself}

There exists a factorization
\[ \begin{tikzcd}
\cY \cap \cZ
\arrow[hook]{r}{i_\cY}
&
\cY
\\
\cZ
\arrow[dashed]{u}{y'}
\arrow[hook]{r}[swap]{i_L}
&
\cX
\arrow{u}[swap]{y}
\end{tikzcd}~, \]
i.e.\! $\cZ$ is aligned with $\cY$.

\item\label{part.lem.equivalent.characterizations.of.alignment.iY.counit.an.equivalence}

The morphism
\begin{equation}
\label{composite.with.iY.counit.that.is.an.equivalence.iff.Z.aligned.with.Y}
i_\cY i_\cY^R  y  i_L
\xlongra{\varepsilon}
y i_L
\end{equation}
in $\Fun(\cZ,\cY)$ is an equivalence.

\item\label{part.lem.equivalent.characterizations.of.alignment.iZ.counit.an.equivalence}

The morphism
\begin{equation}
\label{composite.with.iZ.counit.that.is.an.equivalence.iff.Z.aligned.with.Y}
y i_L i_\cZ i_\cZ^R
\xlongra{\varepsilon}
y i_L
\end{equation}
in $\Fun(\cZ,\cY)$ is an equivalence.

\item\label{part.lem.equivalent.characterizations.of.alignment.factorization.is.radjt}

The lax-commutative square
\begin{equation}
\label{lax.comm.square.in.part.lem.equivalent.characterizations.of.alignment.factorization.is.radjt}
\begin{tikzcd}
\cY \cap \cZ
\arrow[hook]{r}{i_\cY}[swap, xshift=-0.1cm, yshift=-0.4cm]{\rotatebox{-45}{$\Rightarrow$}}
&
\cY
\\
\cZ
\arrow[hook]{r}[swap]{i_L}
\arrow{u}{i_\cZ^R}
&
\cX
\arrow{u}[swap]{y}
\end{tikzcd}
\end{equation}
determined by either commutative square \Cref{commutative.square.of.ladjts.among.Y.cap.Z.and.Y.and.Z.and.X} or \Cref{commutative.square.of.radjts.among.Y.cap.Z.and.Y.and.Z.and.X} commutes.
\end{enumerate}
Moreover, if these equivalent conditions are satisfied, then the factorization $y'$ admits canonical identifications
\[ 
i_\cY^R y i_L
\simeq
y'
\simeq
i_\cZ^R
~.
\]
\end{lemma}

\begin{proof}
We begin by proving the diagram of implications
\[ \begin{tikzcd}
\Cref{part.lem.alignment.itself}
\arrow[Rightarrow]{r}
&
\Cref{part.lem.equivalent.characterizations.of.alignment.iY.counit.an.equivalence}
\arrow[Leftrightarrow]{d}
\\
\Cref{part.lem.equivalent.characterizations.of.alignment.iZ.counit.an.equivalence}
\arrow[Leftrightarrow]{r}
&
\Cref{part.lem.equivalent.characterizations.of.alignment.factorization.is.radjt}
\arrow[Rightarrow]{ul}
\end{tikzcd}~. \]
\begin{itemize}

\item Given a factorization $y'$, we obtain an identification
\[ \begin{tikzcd}[row sep=0cm]
i_\cY i_\cY^R y i_L
\arrow{r}{\varepsilon}
&
y i_L
\\
\rotatebox{90}{$\simeq$}
&
\rotatebox{90}{$\simeq$}
\\
i_\cY i_\cY^R i_\cY y'
\arrow{r}{\sim}[swap]{\varepsilon}
&
i_\cY y'
\end{tikzcd} \]
among morphisms in $\Fun(\cZ,\cY)$.  This proves that $\Cref{part.lem.alignment.itself} \Rightarrow \Cref{part.lem.equivalent.characterizations.of.alignment.iY.counit.an.equivalence}$.

\item Trivially, $\Cref{part.lem.equivalent.characterizations.of.alignment.factorization.is.radjt} \Rightarrow \Cref{part.lem.alignment.itself}$.

\item Considering the lax-commutative square \Cref{lax.comm.square.in.part.lem.equivalent.characterizations.of.alignment.factorization.is.radjt} as being determined by the commutative square \Cref{commutative.square.of.radjts.among.Y.cap.Z.and.Y.and.Z.and.X}, its natural transformation is the composite $i_\cY i_\cZ^R \xra[\sim]{\eta} i_\cY i_\cZ^R y i_L \simeq i_\cY i_\cY^R y i_L \xra{\varepsilon} y i_L$.  This proves that $\Cref{part.lem.equivalent.characterizations.of.alignment.iY.counit.an.equivalence} \Leftrightarrow \Cref{part.lem.equivalent.characterizations.of.alignment.factorization.is.radjt}$.

\item Considering the lax-commutative square \Cref{lax.comm.square.in.part.lem.equivalent.characterizations.of.alignment.factorization.is.radjt} as being determined by the commutative square \Cref{commutative.square.of.ladjts.among.Y.cap.Z.and.Y.and.Z.and.X}, its natural transformation is the composite $i_\cY i_\cZ^R \xra[\sim]{\eta} y i_L i_\cY i_\cZ^R \simeq y i_L i_\cZ i_\cZ^R \xra{\varepsilon} y i_L$.  This proves that $\Cref{part.lem.equivalent.characterizations.of.alignment.iZ.counit.an.equivalence} \Leftrightarrow \Cref{part.lem.equivalent.characterizations.of.alignment.factorization.is.radjt}$.

\end{itemize}
We now conclude by observing that if \Cref{part.lem.equivalent.characterizations.of.alignment.iY.counit.an.equivalence} holds then setting $y' := i_\cY^R y i_L$ defines a factorization.
\end{proof}

\subsection{Fundamental operations on aligned subcategories}
\label{subsection.fund.opns.on.aligned.subcats}

In this subsection we undertake a deeper analysis of alignment, particularly regarding its interactions with colimits and intersections in $\Cls_\cX$ as well as its with quotients of $\cX$ by closed subcategories.

\begin{local}
In this subsection, given two closed subcategories $\cY,\cZ \in \Cls_\cX$ we continue to use the notation $i_\cY$, $i_\cZ$, $i_\cY^R$, and $i_\cZ^R$ of \Cref{local.notn.for.incln.of.intersection.of.closed.subcats}.
\end{local}

\begin{lemma}
\label{lem.colimits.preserve.alignment}
For any closed subcategory $\cY \in \Cls_\cX$, all four functors in the commutative square
\begin{equation}
\label{inclusions.of.posets.of.closed.subcats.with.alignment.conditions}
\begin{tikzcd}
\Cls_\cX^{\mutalgnd \cY}
\arrow[hook]{r}
\arrow[hook]{d}
&
\Cls_\cX^{\algnd \cY}
\arrow[hook]{d}
\\
\Cls_\cX^{\algndfrom \cY}
\arrow[hook]{r}
&
\Cls_\cX
\end{tikzcd}
\end{equation}
preserve colimits.
\end{lemma}

\begin{proof}
Since the commutative square \Cref{inclusions.of.posets.of.closed.subcats.with.alignment.conditions} is a pullback among full subposets of $\Cls_\cX$, it suffices to check that its right vertical functor and its lower horizontal functor both preserve colimits.  We address each of these in turn.

Suppose first that we are given any $\{ \cZ_s \in \Cls_\cX^{\algnd \cY} \}_{s \in S}$, and let us write $\cZ = \brax{\cZ_s}_{s \in S} \in \Cls_\cX$.  For each $s \in S$, by assumption we have a factorization
\begin{equation}
\label{diagram.for.showing.that.colimit.of.aligned.is.aligned}
\begin{tikzcd}
\cY \cap \cZ_s
\arrow[hook]{r}
&
\cY \cap \cZ
\arrow[hook]{r}
&
\cY
\\
\cZ_s
\arrow[hook]{r}[swap]{i_L}
\arrow[dashed]{u}
&
\cZ
\arrow[hook]{r}[swap]{i_L}
&
\cX
\arrow{u}[swap]{y}
\end{tikzcd}~.
\end{equation}
Because all solid functors in the diagram \Cref{diagram.for.showing.that.colimit.of.aligned.is.aligned} preserve colimits, we find that $\cZ$ is aligned with $\cY$, i.e.\! that $\cZ \in \Cls_\cX^{\algnd \cY}$.

Suppose now that we are given any $\{ \cZ_s \in \Cls_\cX^{\algndfrom \cY} \}_{s \in S}$, and let us write $\cZ = \brax{\cZ_s}_{s \in S} \in \Cls_\cX$.  For an arbitrary element $s \in S$, consider the diagram
\begin{equation}
\label{diagram.with.two.natural.transformations.the.lower.of.which.being.an.equivalence.means.that.Y.aligned.with.Z}
\begin{tikzcd}
\cY \cap \cZ_s
\arrow[hook]{r}{i_{\cZ_s}}[swap, xshift=-0.1cm, yshift=-0.4cm]{\rotatebox{-45}{$\Rightarrow$}}
&
\cZ_s
\\
\cY \cap \cZ
\arrow{u}{i^R}
\arrow[hook]{r}{i_\cZ}[swap, xshift=-0.1cm, yshift=-0.4cm]{\rotatebox{-45}{$\Rightarrow$}}
&
\cZ
\arrow{u}[swap]{y}
\\
\cY
\arrow{u}{i_\cY^R}
\arrow[hook]{r}[swap]{i_L}
&
\cX
\arrow{u}[swap]{y}
\end{tikzcd}~,
\end{equation}
in which the functor $i^R$ is the evident right adjoint.  By \Cref{lem.equivalent.characterizations.of.alignment}, to show that $\cY$ is aligned with $\cZ$ it suffices to show that the lower natural transformation in diagram \Cref{diagram.with.two.natural.transformations.the.lower.of.which.being.an.equivalence.means.that.Y.aligned.with.Z} is an equivalence.  Also by \Cref{lem.equivalent.characterizations.of.alignment}, because $\cY$ is aligned with $\cZ_s$, the composite natural transformation
\begin{equation}
\label{composite.natural.trans.in.diagram.with.two.natural.transformations.the.lower.of.which.being.an.equivalence.means.that.Y.aligned.with.Z}
i_{\cZ_s} i^R i_\cY^R
\longra
yyi_L
\end{equation}
in diagram \Cref{diagram.with.two.natural.transformations.the.lower.of.which.being.an.equivalence.means.that.Y.aligned.with.Z} is an equivalence.  This implies that the upper natural transformation in diagram \Cref{diagram.with.two.natural.transformations.the.lower.of.which.being.an.equivalence.means.that.Y.aligned.with.Z} is also an equivalence, as it is given by the composite
\[
i_{\cZ_s} i^R \xra[\sim]{\eta} i_{\cZ_s}
i^R i_\cY^R i_\cY
\xra[\sim]{\Cref{composite.natural.trans.in.diagram.with.two.natural.transformations.the.lower.of.which.being.an.equivalence.means.that.Y.aligned.with.Z}}
y y i_L i_\cY
\simeq
y y i_L i_\cZ \simeq y i_\cZ
~.
\]
So, the lower natural transformation in diagram \Cref{diagram.with.two.natural.transformations.the.lower.of.which.being.an.equivalence.means.that.Y.aligned.with.Z} is indeed an equivalence, because the functors $\{ \cZ \xra{y} \cZ_s \}_{s \in S}$ are jointly conservative.
\end{proof}

\begin{lemma}
\label{lemma.all.about.aligned.subcats}
Let $\cY,\cZ \in \Cls_\cX$ be closed subcategories, and suppose that $\cZ$ is aligned with $\cY$.
\begin{enumerate}

\item\label{part.alignment.lemma.intersection.is.closed}

The functor $i_\cZ$ is the inclusion of $\cY \cap \cZ$ as a closed subcategory of $\cZ$.

\item\label{part.alignment.lemma.induced.map.on.quotients}

Consider the resulting commutative diagram
\begin{equation}
\label{induced.morphism.on.presentable.quotients.when.Y.is.aligned.with.Z}
\begin{tikzcd}
\cY \cap \cZ
\arrow[hook]{r}{i_\cY}
\arrow[hook]{d}[swap]{i_\cZ}
&
\cY
\arrow[hook]{d}{i_L}
\\
\cZ
\arrow[hook]{r}{i_L}
\arrow{d}[swap]{p_L}
&
\cX
\arrow{d}{p_L}
\\
\cZ / (\cY \cap \cZ)
\arrow[dashed]{r}[swap]{i}
&
\cX / \cY
\end{tikzcd}
\end{equation}
in $\PrLSt$, in which $i$ is the canonical morphism between presentable quotients.

\begin{enumeratesub}

\item\label{subpart.alignment.lemma.image.is.closed}

The functor $i$ is the fully faithful inclusion of $\cZ / (\cY \cap \cZ)$ as a closed subcategory of $\cX / \cY$.

\item\label{subpart.alignment.lemma.nu.commutativity}

The lax-commutative square
\begin{equation}
\label{nu.commutativity.for.aligned.subcats}
\begin{tikzcd}
\cZ
\arrow[hook]{r}{i_L}[swap, yshift=-0.4cm]{\rotatebox{-45}{$\Rightarrow$}}
&
\cX
\\
\cZ / (\cZ \cap \cY)
\arrow{r}[swap]{i}
\arrow[hook]{u}{\nu}
&
\cX / \cY
\arrow[hook]{u}[swap]{\nu}
\end{tikzcd}
\end{equation}
determined by the lower commutative square in diagram \Cref{induced.morphism.on.presentable.quotients.when.Y.is.aligned.with.Z} commutes.

\item\label{subpart.alignment.lemma.yo.and.pL.commutativity}

Suppose further that $\cY$ is aligned with $\cZ$.  Then, the lax-commutative square
\begin{equation}
\label{yo.and.pL.commutativity.for.aligned.subcats}
\begin{tikzcd}
\cZ
\arrow{d}[swap]{p_L}
&
\cX
\arrow{l}[swap]{y}[yshift=-0.4cm]{\rotatebox{-45}{$\Rightarrow$}}
\arrow{d}{p_L}
\\
\cZ/(\cY \cap \cZ)
&
\cX / \cY
\arrow{l}{i^R}
\end{tikzcd}
\end{equation}
determined by the lower commutative square in diagram \Cref{induced.morphism.on.presentable.quotients.when.Y.is.aligned.with.Z} commutes.

\end{enumeratesub}

\end{enumerate}

\end{lemma}

\begin{proof}
We begin by proving part \Cref{part.alignment.lemma.intersection.is.closed}.  Because $i_\cZ$ is fully faithful, it remains to show that its right adjoint $i_\cZ^R$ preserves colimits.  For this, because $i_\cY$ is fully faithful and colimit-preserving, it suffices to show that the composite $i_\cY i_\cZ^R$ preserves colimits, which follows from the equivalence $i_\cY i_\cZ^R \simeq y i_L$ guaranteed by \Cref{lem.equivalent.characterizations.of.alignment}.

We now prove part \Cref{part.alignment.lemma.induced.map.on.quotients}\Cref{subpart.alignment.lemma.nu.commutativity}.  By definition, the natural transformation in the lax-commutative square \Cref{nu.commutativity.for.aligned.subcats} is the composite $i_L \nu \xra{\eta} \nu p_L i_L \nu \simeq \nu i p_L \nu \xra[\sim]{\varepsilon} \nu i$.  To show that it is an equivalence is therefore equivalent to showing that the composite functor
\[ \begin{tikzcd}
\cZ
\arrow[hook]{r}{i_L}
&
\cX
\\
\cZ / (\cY \cap \cZ)
\arrow[hook]{u}{\nu}
\end{tikzcd} \]
lands in the image of the functor
\[ \begin{tikzcd}
\cX
\\
\cX / \cY
\arrow[hook]{u}[swap]{\nu}
~.
\end{tikzcd} \]
This is equivalent to showing that the composite functor
\[ \begin{tikzcd}
&
\cY
\\
\cZ
\arrow[hook]{r}{i_L}
&
\cX
\arrow{u}[swap]{y}
\\
\cZ / (\cY \cap \cZ)
\arrow[hook]{u}{\nu}
\end{tikzcd} \]
is zero.  This follows from the commutativity of the diagram
\[ \begin{tikzcd}
\cY \cap \cZ
\arrow[hook]{r}{i_\cY}
&
\cY
\\
\cZ
\arrow{u}{i_\cZ^R}
\arrow[hook]{r}{i_L}
&
\cX
\arrow{u}[swap]{y}
\\
\cZ / (\cY \cap \cZ)
\arrow[hook]{u}{\nu}
\end{tikzcd} \]
guaranteed by \Cref{lem.equivalent.characterizations.of.alignment}, because its left vertical composite is zero.  So indeed, the lax-commutative square \Cref{nu.commutativity.for.aligned.subcats} is commutative.

We now prove part \Cref{part.alignment.lemma.induced.map.on.quotients}\Cref{subpart.alignment.lemma.image.is.closed}.  By part \Cref{part.alignment.lemma.induced.map.on.quotients}\Cref{subpart.alignment.lemma.nu.commutativity}, the functor $i$ is fully faithful.  Note too that by definition $i$ is colimit-preserving.  Passing to right adjoints in the lower commutative square in diagram \Cref{induced.morphism.on.presentable.quotients.when.Y.is.aligned.with.Z}, we obtain a commutative square
\[ \begin{tikzcd}
\cZ
&
\cX
\arrow{l}[swap]{y}
\\
\cZ / ( \cY \cap \cZ)
\arrow[hook]{u}{\nu}
&
\cX / \cY
\arrow{l}{i^R}
\arrow[hook]{u}[swap]{\nu}
\end{tikzcd}~, \]
which implies that $i^R$ is colimit-preserving.  So indeed, $i$ is the inclusion of a closed subcategory.

We now conclude by proving part \Cref{part.alignment.lemma.induced.map.on.quotients}\Cref{subpart.alignment.lemma.yo.and.pL.commutativity}.  By \Cref{lem.equivalent.characterizations.of.alignment} (with the roles of $\cY$ and $\cZ$ reversed), we have a commutative diagram
\begin{equation}
\label{induced.map.between.presentable.quotients.from.yoneda.restrictions}
\begin{tikzcd}
\cY \cap \cZ
\arrow[hook]{d}[swap]{i_\cZ}
&
\cY
\arrow{l}[swap]{i_\cY^R}
\arrow[hook]{d}{i_L}
\\
\cZ
\arrow{d}[swap]{p_L}
&
\cX
\arrow{l}[swap]{y}
\arrow{d}{p_L}
\\
\cZ / ( \cY \cap \cZ)
&
\cX / \cY
\arrow[dashed]{l}{j}
\end{tikzcd}
\end{equation}
in $\PrLSt$, in which $j$ is the canonical morphism between presentable quotients.  Hence, $j$ fits into a commutative diagram
\[ \begin{tikzcd}
\cZ
\arrow{d}[swap]{p_L}
&
\cX
\arrow{l}[swap]{y}
\\
\cZ / (\cY \cap \cZ)
&
\cX / \cY
\arrow[hook]{u}[swap]{\nu}
\arrow{l}{j}
\end{tikzcd}~. \]
On the other hand, note the commutative square
\[ \begin{tikzcd}
\cZ
&
\cX
\arrow{l}[swap]{y}
\\
\cZ / (\cY \cap \cZ)
\arrow[hook]{u}{\nu}
&
\cX / \cY
\arrow[hook]{u}[swap]{\nu}
\arrow{l}{j}
\end{tikzcd} \]
obtained by passing to right adjoints in the lower commutative square of diagram \Cref{induced.morphism.on.presentable.quotients.when.Y.is.aligned.with.Z}.  Using this, we obtain the identification $j \simeq p_L y \nu \simeq p_L \nu i^R \simeq i_R$.  Thereafter, we see that indeed the lax-commutative square \Cref{yo.and.pL.commutativity.for.aligned.subcats} is precisely the lower commutative square in diagram \Cref{induced.map.between.presentable.quotients.from.yoneda.restrictions}.
\end{proof}

\begin{observation}
\label{obs.if.aligned.preimage.of.closed.is.closed}
Let $\cY,\cZ \in \Cls_\cX$ be closed subcategories, and suppose that $\cZ$ is aligned with $\cY$.  By \Cref{lemma.all.about.aligned.subcats}\Cref{part.alignment.lemma.intersection.is.closed}, we have $(\cY \cap \cZ) \in \Cls_\cZ$.  It follows that $(\cY \cap \cZ) \in \Cls_\cX$, and thereafter that $(\cY \cap \cZ) \in \Cls_\cY$. In other words, all four functors in the pullback diagram
\begin{equation}
\label{intersection.closed.in.both.Y.and.Z}
\begin{tikzcd}
\cY \cap \cZ
\arrow[hook]{r}{i_\cY}
\arrow[hook]{d}[swap]{i_\cZ}
&
\cY
\arrow[hook]{d}{i_L}
\\
\cZ
\arrow[hook]{r}[swap]{i_L}
&
\cX
\end{tikzcd}
\end{equation}
are inclusions of closed subcategories.
\end{observation}

\begin{lemma}
\label{lem.if.aligned.then.iRs.form.a.pullback.square}
Let $\cY,\cZ \in \Cls_\cX$ be closed subcategories, and suppose that $\cZ$ is aligned with $\cY$. Then, the commutative square
\begin{equation}
\label{comm.square.of.iRs.that.is.a.pullback.assuming.alignment}
\begin{tikzcd}
\cY \cap \cZ
\arrow[hook]{r}{i_R}
\arrow[hook]{d}[swap]{i_R}
&
\cY
\arrow[hook]{d}{i_R}
\\
\cZ
\arrow[hook]{r}[swap]{i_R}
&
\cX
\end{tikzcd}
\end{equation}
in $\Cat$ obtained by taking right adjoints twice in the commutative square \Cref{intersection.closed.in.both.Y.and.Z} in $\Cat$ (which is possible by \Cref{obs.if.aligned.preimage.of.closed.is.closed}) is a pullback square.
\end{lemma}

\begin{proof}
By \Cref{lem.equivalent.characterizations.of.alignment}, the square
\[ \begin{tikzcd}
\cY \cap \cZ
\arrow[hook]{r}{i_L}
&
\cY
\\
\cZ
\arrow{u}{y}
\arrow[hook]{r}[swap]{i_L}
&
\cX
\arrow{u}[swap]{y}
\end{tikzcd} \]
commutes, which implies that the square
\begin{equation}
\label{iR.and.y.commute.given.alignment}
\begin{tikzcd}
\cY \cap \cZ
\arrow[hook]{d}[swap]{i_R}
&
\cY
\arrow{l}[swap]{y}
\arrow[hook]{d}{i_R}
\\
\cZ
&
\cX
\arrow{l}{y}
\end{tikzcd}
\end{equation}
commutes by passing to right adjoints. Now, consider the solid commutative diagram
\begin{equation}
\label{functor.from.iL.intersection.to.iR.intersection}
\begin{tikzcd}
\cY \cap \cZ
\arrow{rd}
\arrow[hook, bend left=10]{rrd}[sloped]{i_R}
\arrow[hook, bend right=10]{rdd}[sloped, swap]{i_R}
\\
&
\cY \cap_R \cZ
\arrow[hook]{r}[swap]{j_\cY}
\arrow[hook]{d}{j_\cZ}
\arrow[dashed, bend left=10]{lu}
&
\cY
\arrow[hook]{d}{i_R}
\\
&
\cZ
\arrow[hook]{r}[swap]{i_R}
&
\cX
\end{tikzcd}
\end{equation}
in which $\cY \cap_R \cZ$ denotes the pullback in $\Cat$. Because both functors to $\cZ$ in diagram \Cref{functor.from.iL.intersection.to.iR.intersection} are fully faithful, it suffices to show that there exists the dashed factorization of $j_\cZ$. This follows from the sequence of equivalences
\[
j_\cZ
\simeq
y i_R j_\cZ
\simeq
y i_R j_\cY
\simeq
i_R y j_\cY
~,
\]
in which the last equivalence follows from the commutativity of the square \Cref{iR.and.y.commute.given.alignment}.
\end{proof}

\begin{remark}
By \Cref{obs.if.aligned.preimage.of.closed.is.closed}, two closed subcategories $\cY,\cZ \in \Cls_\cX$ are mutually aligned if and only if the diagram
\[ \begin{tikzcd}
\cY \cap \cZ
\arrow[hook]{r}{i_L}
\arrow[hook]{d}[swap]{i_L}
&
\cY
\arrow[hook]{d}{i_L}
\\
\cZ
\arrow[hook]{r}[swap]{i_L}
&
\cX
\end{tikzcd} \]
defines a stratification of $\cX$ over $[1] \times [1]$.
\end{remark}

\begin{remark}
\label{rmk.image.of.composite.need.not.be.closed.if.not.aligned}
Given closed subcategories $\cY,\cZ \in \Cls_\cX$, the most important consequence of $\cZ$ being aligned with $\cY$ is that the image of the composite
\[
\cZ
\xlonghookra{i_L}
\cX
\xlongra{p_L}
\cX/\cY
\]
is a closed subcategory, as guaranteed by \Cref{lemma.all.about.aligned.subcats}\Cref{part.alignment.lemma.induced.map.on.quotients}\Cref{subpart.alignment.lemma.image.is.closed}. This need not hold if $\cZ$ is not aligned with $\cY$.  We may see this as follows.

Let us take $\cX := \Fun(\cI,\Spectra)$, where $\cI$ denotes the category generated by the quiver
\[ \begin{tikzcd}
u
\arrow[bend left]{r}{a}
&
v
\arrow[bend left]{l}{b}
\end{tikzcd}~, \]
i.e.\! the pushout
\[
\cI
:=
\colim \left( \begin{tikzcd}
\pt \sqcup \pt
\arrow{r}{(0,1)}
\arrow{d}[swap]{(1,0)}
&
{[1]}
\\
{[1]}
\end{tikzcd} \right)
~.
\]
Consider the full subcategories
\[
\cY := \{ E_\bullet \in \cX : E_u \simeq 0 \}
\subseteq \cX \supseteq
\{ E_\bullet \in \cX : E_v \simeq 0 \} =: \cZ
~.
\]
They are clearly closed under colimits.  Moreover, via the identifications
\[
\cY
\xra[\sim]{\ev_v}
\Spectra
\xla[\sim]{\ev_u}
\cZ~,
\]
their inclusions' right adjoints are given by the formulas
\[ \begin{tikzcd}[column sep=1.5cm, row sep=0cm]
\cY
\arrow[hook, transform canvas={yshift=0.9ex}]{r}{i_L}
\arrow[dashed, leftarrow, transform canvas={yshift=-0.9ex}]{r}[yshift=-0.2ex]{\bot}[swap]{y}
&
\cX
\arrow[hookleftarrow, transform canvas={yshift=0.9ex}]{r}{i_L}
\arrow[dashed, transform canvas={yshift=-0.9ex}]{r}[yshift=-0.2ex]{\bot}[swap]{y}
&
\cZ
\\
\rotatebox{90}{$\in$}
&
\rotatebox{90}{$\in$}
&
\rotatebox{90}{$\in$}
\\
\fib(E_b)
&
E_\bullet
\arrow[maps to]{r}
\arrow[maps to]{l}
&
\fib(E_a)
\end{tikzcd}~, \]
which preserve colimits so that $\cY$ and $\cZ$ are indeed closed subcategories of $\cX$ (which justifies the notations $i_L$ and $y$).  Note that $\cY \cap \cZ \simeq 0$.  On the other hand, the composite functors
\[
\cY
\xlonghookra{i_L}
\cX
\xlongra{y}
\cZ
\qquad
\text{and}
\qquad
\cY
\xlongla{y}
\cX
\xlonghookla{i_L}
\cZ
\]
may both be identified with desuspension, and in particular are equivalences.  So, $\cY$ is not aligned with $\cZ$ and $\cZ$ is not aligned with $\cY$.

Now, observe the identification
\[ \begin{tikzcd}[row sep=0cm]
\cX
\arrow[hookleftarrow]{r}{\nu}
&
\cX / \cY
\\
\rotatebox{90}{$=:$}
&
\rotatebox{90}{$\simeq$}
\\
\Fun(\cI,\Spectra)
\arrow[hookleftarrow]{r}
&
\{ E_\bullet \in \cX : E_b \textup{ is an equivalence} \}
\end{tikzcd}~, \]
and thereafter the identification of its left adjoint $\cX \xra{p_L} \cX / \cY$ as the assignment
\begin{align*}
\left( \begin{tikzcd}[ampersand replacement=\&]
E_u
\arrow[bend left]{r}{E_a}
\&
E_v
\arrow[bend left]{l}{E_b}
\end{tikzcd} \right)
& \longmapsto
\cofib \left(
\left( \begin{tikzcd}[ampersand replacement=\&]
0
\arrow[bend left]{r}
\&
\fib(E_b)
\arrow[bend left]{l}
\end{tikzcd} \right)
\longra
\left( \begin{tikzcd}[ampersand replacement=\&]
E_u
\arrow[bend left]{r}{E_a}
\&
E_v
\arrow[bend left]{l}{E_b}
\end{tikzcd} \right)
\right)
\\
& \simeq
\left( \begin{tikzcd}[ampersand replacement=\&]
E_u
\arrow[bend left]{r}{E_b E_a}
\&
E_u
\arrow[bend left]{l}{\id}[swap]{\sim}
\end{tikzcd} \right)
~.
\end{align*}
Hence, the composite
\[
\Spectra
\xla[\sim]{\ev_u}
\cZ
\xlonghookra{i_L}
\cX
\xra{p_L}
\cX/\cY
\]
is given by the assignment
\[
E
\longmapsto
\left( \begin{tikzcd}
E
\arrow[bend left]{r}{0}
&
E
\arrow[bend left]{l}{\id}[swap]{\sim}
\end{tikzcd} \right)
~,
\]
and so its image is not even closed under colimits -- nor does it define a fully faithful functor from $\cZ / (\cY \cap \cZ) \simeq \cZ / 0 \simeq \cZ$ to $\cX/\cY$.
\end{remark}

\begin{remark}
In \Cref{lemma.all.about.aligned.subcats}\Cref{part.alignment.lemma.induced.map.on.quotients}\Cref{subpart.alignment.lemma.yo.and.pL.commutativity}, the lax-commutative square \Cref{yo.and.pL.commutativity.for.aligned.subcats} need not commute if $\cY$ is not aligned with $\cZ$.  Indeed, in the situation of \Cref{ex.of.one.directional.alignment}, it may be identified with the canonical lax-commutative square
\[ \begin{tikzcd}
\Spectra
\arrow{d}[swap]{\id}[sloped, anchor=south]{\sim}
&
\Fun([1],\Spectra)
\arrow{l}[swap]{\fib}[yshift=-0.4cm, xshift=0.3cm]{\rotatebox{-45}{$\Rightarrow$}}
\arrow{d}{\ev_0}
\\
\Spectra
&
\Spectra
\arrow{l}{\id}[swap]{\sim}
\end{tikzcd}~. \]
\end{remark}

\begin{prop}
\label{prop.image.and.preimage.of.closed.subcats.are.closed}
For any closed subcategory $\cY \in \Cls_\cX$, taking the image or preimage (in $\Cat$) of a closed subcategory along either functor in the composite
\[
\cY
\xlonghookra{i_L}
\cX
\xlongra{p_L}
\cX/\cY
\]
yields a closed subcategory, and these constructions define adjunctions
\[ \begin{tikzcd}[column sep=1.5cm, row sep=1.5cm]
&
(\Cls_\cX)_{/\cY}
\arrow[hook]{d}
\\
\Cls_\cY
\arrow[dashed]{ru}[sloped]{\sim}
\arrow[hook, transform canvas={yshift=0.9ex}]{r}{i_L}
\arrow[leftarrow, transform canvas={yshift=-0.9ex}]{r}[yshift=-0.2ex]{\bot}[swap]{i_L^{-1}}
&
\Cls_\cX^{\algnd \cY}
\arrow[transform canvas={yshift=0.9ex}]{r}{p_L}
\arrow[hookleftarrow, transform canvas={yshift=-0.9ex}]{r}[yshift=-0.2ex]{\bot}[swap]{p_L^{-1}}
&
\Cls_{\cX/\cY}
\arrow[dashed]{ld}[sloped, swap]{\sim}
\\
&
(\Cls_\cX)_{\cY/}
\arrow[hook]{u}
\end{tikzcd}
\]
with fully faithful images as indicated.
\end{prop}

\begin{proof}
It is immediate that $i_L$ and $p_L^{-1}$ respectively carry closed subcategories of $\cY$ and $\cX/\cY$ to closed subcategories of $\cX$, which are aligned with $\cY$ by \Cref{obs.mutually.aligned.if.containment}.  Moreover, for any $\cZ \in \Cls_\cX^{\algnd \cY}$, we have $i_L^{-1}(\cZ) = (\cY \cap \cZ) \in \Cls_\cY$ by \Cref{obs.if.aligned.preimage.of.closed.is.closed} and $p_L(\cZ) = \cZ / (\cY \cap \cZ) \in \Cls_{\cX/\cY}$ by \Cref{lemma.all.about.aligned.subcats}\Cref{part.alignment.lemma.induced.map.on.quotients}\Cref{subpart.alignment.lemma.image.is.closed}.  The asserted co/reflective adjunctions among posets, as well as the identifications of the resulting fully faithful images, are now immediate.
\end{proof}

\begin{lemma}
\label{lemma.preimages.preserve.colimits}
Fix any closed subcategory $\cY \in \Cls_\cX$.
\begin{enumerate}

\item\label{lemma.part.iL.preimage.preserves.colimits}

The functor $\Cls_\cX^{\algnd \cY} \xra{i_L^{-1}} \Cls_\cY$ preserves colimits.

\item\label{lemma.part.pL.preimage.preserves.colimits}

The functor $\Cls_{\cX/\cY} \xra{p_L^{-1}} \Cls_\cX^{\algnd \cY}$ preserves nonempty colimits.

\end{enumerate}
\end{lemma}

\begin{proof}
We first prove part \Cref{lemma.part.iL.preimage.preserves.colimits}.  Let $\{ \cZ_s \in \Cls_\cX^{\algnd \cY} \}_{s \in S}$ be a set of closed subcategories of $\cX$ that are aligned with $\cY$.  We have an evident inclusion $\brax{ i_L^{-1}(\cZ_s)}_{s \in S} \subseteq i_L^{-1}(\brax{\cZ_s}_{s \in S})$.  On the other hand, because $\cX \xra{y} \cY$ preserves colimits, we also have an inclusion $i_L^{-1}(\brax{\cZ_s}_{s \in S}) \subseteq \brax{ i_L^{-1}(\cZ_s)}_{s \in S}$.

We now prove part \Cref{lemma.part.pL.preimage.preserves.colimits}.  Let now $\{ \cZ_s \in \Cls_{\cX/\cY} \}_{s \in S}$ be a nonempty set of closed subcategories of $\cX/\cY$.  We have an evident inclusion $\brax{ p_L^{-1}(\cZ_s)}_{s \in S} \subseteq p_L^{-1}(\brax{\cZ_s}_{s \in S})$.  On the other hand, for any $X \in p_L^{-1}(\brax{\cZ_s}_{s \in S})$, consider the co/fiber sequence $i_L y X \ra X \ra \nu p_L X$.  Because $i_L y X \in p_L^{-1}(0) \subseteq \brax{ p_L^{-1}(\cZ_s)}_{s \in S}$ (using that $S$ is nonempty), to show that $X \in \brax{ p_L^{-1}(\cZ_s)}_{s \in S}$ it suffices to show that $\nu p_L X \in \brax{ p_L^{-1}(\cZ_s)}_{s \in S}$, which follows from the fact that $\cX/\cY \xhookra{\nu} \cX$ preserves colimits.
\end{proof}


\begin{observation}
\label{obs.if.algnd.with.Y.and.W.then.iLinverse.algnd.with.W}
Fix any closed subcategory $\cY \in \Cls_\cX$ and any $\cW \in \Cls_\cY \subseteq \Cls_\cX$.  Then, by the equivalence $\Cref{part.lem.alignment.itself} \Leftrightarrow \Cref{part.lem.equivalent.characterizations.of.alignment.iY.counit.an.equivalence}$ of \Cref{lem.equivalent.characterizations.of.alignment} there exists a factorization
\[ \begin{tikzcd}
\Cls_\cY
&
\Cls_\cX^{\algnd \cY}
\arrow{l}[swap]{i_L^{-1}}
\\
\Cls_\cY^{\algnd \cW}
\arrow[hook]{u}
&
\Cls_\cX^{\algnd \cY} \cap \Cls_\cX^{\algnd \cW}
\arrow[hook]{u}
\arrow[dashed]{l}
\end{tikzcd}~: \]
that is, if $\cZ \in \Cls_\cX$ is aligned with both $\cY$ and $\cW$ then $i_L^{-1}(\cZ) := \cY \cap \cZ$ is aligned with $\cW$.
\end{observation}

\subsection{Gluing aligned subcategories}
\label{subsection.glue.aligned}

In this brief subsection, we establish gluing formulas for closed subcategories of $\cX$ in the presence of alignment. More precisely, one may view \Cref{lem.excision} (which merely requires alignment) as an excision principle and \Cref{lem.mayer.vietoris} (which requires mutual alignment) as a Mayer--Vietoris principle.\footnote{Recall that closed subcategories of a presentable stable $\infty$-category correspond to open subsets of a topological space, as indicated in \Cref{subsection.cbl}.}

\begin{local}
In this subsection, we fix closed subcategories $\cY,\cZ \in \Cls_\cX$.
\end{local}

\begin{remark}
In this subsection, we implicitly use \Cref{obs.if.aligned.preimage.of.closed.is.closed} (that if $\cZ$ is aligned with $\cY$ then $\cY \cap \cZ$ is a closed subcategory of both $\cY$ and $\cZ$).
\end{remark}

\begin{local}
Given co/reflective localizations
\[ \begin{tikzcd}[column sep=1.5cm]
\cC
\arrow[hook, transform canvas={yshift=0.9ex}]{r}{F}
\arrow[leftarrow, transform canvas={yshift=-0.9ex}]{r}[yshift=-0.2ex]{\bot}[swap]{G}
&
\cX
\end{tikzcd}
\qquad
\text{and}
\qquad
\begin{tikzcd}[column sep=1.5cm]
\cX
\arrow[transform canvas={yshift=0.9ex}]{r}{F'}
\arrow[hookleftarrow, transform canvas={yshift=-0.9ex}]{r}[yshift=-0.2ex]{\bot}[swap]{G'}
&
\cC'
\end{tikzcd} \]
of $\cX$, we write
\[
C_\cC
:=
FG
\xra{\varepsilon_\cC}
\id_\cX
\qquad
\text{and}
\qquad
\id_\cX
\xra{\eta_\cC}
G'F'
=:
L_{\cC'}
\]
for the corresponding co/monads on $\cX$ and their co/unit maps.\footnote{So, in the notation of \Cref{defn.Cth.stratum.and.geometric.localizn} we simply write $L_\sC := L_{\cX_\sC}$.} In particular, given a closed subcategory $\cY \in \Cls_\cX$ we obtain the endofunctors
\[
C_\cY := i_L y
~,~
\qquad
L_\cY := i_R y
~,~
\qquad
L_{\cX/\cY} := \nu p_L
~,~
\qquad
\text{and}
\qquad
C_{\cX/\cY} := \nu p_R
\]
of $\cX$.
\end{local}

\begin{lemma}
\label{lem.excision}
Suppose that $\cZ$ is aligned with $\cY$.
\begin{enumerate}

\item\label{part.excision.coloc} There is a canonical identification
\[
C_{\brax{\cY,\cZ}}
\simeq
\cofib ( \Sigma^{-1} C_\cZ L_{\cX/\cY}
\longra
\Sigma^{-1} L_{\cX/\cY}
\longra
C_\cY
)
~.
\]

\item\label{part.excision.loc} There is a canonical identification
\[
L_{\brax{\cY,\cZ}}
\simeq
\fib ( L_\cZ
\longra
\Sigma C_{\cX/\cZ}
\longra
\Sigma L_\cY C_{\cX/\cZ}
)
~.
\]

\end{enumerate}
\end{lemma}

\begin{proof}
We begin with part \Cref{part.excision.coloc}. For this, consider the morphism
\begin{equation}
\label{morphism.of.cofiber.sequences.for.C.brax.Y.Z}
\begin{tikzcd}
\Sigma^{-1} C_\cZ L_{\cX/\cY}
\arrow{d}[swap]{\Sigma^{-1} \varepsilon_\cZ L_{\cX/\cY}}
\arrow{r}
&
C_\cY
\arrow[equals]{d}
\arrow{r}
&
\cofib
\arrow{d}
\\
\Sigma^{-1} L_{\cX/\cY}
\arrow{r}
&
C_\cY
\arrow{r}[swap]{\varepsilon_\cY}
&
\id_\cX
\end{tikzcd}
\end{equation}
of cofiber sequences in $\Fun^\ex(\cX,\cX)$, where we simply write $\cofib$ for the indicated cofiber. It suffices to show that the right vertical morphism in diagram \Cref{morphism.of.cofiber.sequences.for.C.brax.Y.Z} becomes an equivalence after applying $C_\cY$ and $C_\cZ$. It is clear that it becomes an equivalence after applying $C_\cZ$. To see that it becomes an equivalence after applying $C_\cY$, it suffices to observe the containment
\[
\ker(C_\cY)
\subseteq
\ker(C_\cY C_\cZ)
\]
resulting from the fact that $\cZ$ is aligned with $\cY$.

We now turn to part \Cref{part.excision.loc}. For this, consider the morphism
\begin{equation}
\label{morphism.of.cofiber.sequences.for.L.brax.Y.Z}
\begin{tikzcd}
\id_\cX
\arrow{r}{\eta_\cZ}
\arrow{d}
&
L_\cZ
\arrow{r}
\arrow[equals]{d}
&
\Sigma C_{\cX/\cZ}
\arrow{d}{\Sigma \eta_\cY C_{\cX/\cZ}}
\\
\fib
\arrow{r}
&
L_\cZ
\arrow{r}
&
\Sigma L_\cY C_{\cX/\cZ}
\end{tikzcd}
\end{equation}
of cofiber sequences in $\Fun^\ex(\cX,\cX)$, where we simply write $\fib$ for the indicated fiber. It suffices to show that the left vertical morphism in diagram \Cref{morphism.of.cofiber.sequences.for.L.brax.Y.Z} becomes an equivalence after applying $L_\cY$ and $L_\cZ$. It is clear that it becomes an equivalence after applying $L_\cY$. To see that it becomes an equivalence after applying $L_\cZ$, it suffices to observe the containment
\[
\ker(L_\cZ) \subseteq \ker(L_\cZ L_\cY)
\]
resulting from the fact that $\cZ$ is aligned with $\cY$.
\end{proof}

\begin{lemma}
\label{lem.mayer.vietoris}
Suppose that $\cY$ and $\cZ$ are mutually aligned.
\begin{enumerate}

\item\label{part.mayer.vietoris.coloc} The commutative square
\begin{equation}
\label{mayer.vietoris.square.for.coloc}
\begin{tikzcd}
C_{\cY \cap \cZ}
\arrow{r}
\arrow{d}
&
C_\cY
\arrow{d}
\\
C_\cZ
\arrow{r}
&
C_{\brax{\cY,\cZ}}
\end{tikzcd}
\end{equation}
in $\Fun^\ex(\cX,\cX)$ is a pushout.

\item\label{part.mayer.vietoris.loc} The commutative square
\begin{equation}
\label{mayer.vietoris.square.for.loc}
\begin{tikzcd}
L_{\cY \cap \cZ}
&
L_\cY
\arrow{l}
\\
L_\cZ
\arrow{u}
&
L_{\brax{\cY,\cZ}}
\arrow{l}
\arrow{u}
\end{tikzcd}
\end{equation}
in $\Fun^\ex(\cX,\cX)$ is a pullback.

\end{enumerate}
\end{lemma}

\begin{proof}
We begin with part \Cref{part.mayer.vietoris.coloc}. It suffices to show that the square \Cref{mayer.vietoris.square.for.coloc} becomes a pushout after applying $C_\cY$ and $C_\cZ$.
\begin{itemize}

\item Applying $C_\cY$ to the square \Cref{mayer.vietoris.square.for.coloc}, we see that both vertical morphisms become equivalences, the right by inspection and the left because $\cZ$ is aligned with $\cY$.

\item Applying $C_\cZ$ to the square \Cref{mayer.vietoris.square.for.coloc}, we see that both horizontal morphisms become equivalences, the lower by inspection and the upper because $\cY$ is aligned with $\cZ$.

\end{itemize}
So the square \Cref{mayer.vietoris.square.for.coloc} is indeed a pushout.

We now turn to part \Cref{part.mayer.vietoris.loc}. It suffices to show that the square \Cref{mayer.vietoris.square.for.loc} becomes a pullback after applying $L_\cY$ and $L_\cZ$.
\begin{itemize}

\item Applying $L_\cY$ to the square \Cref{mayer.vietoris.square.for.loc}, we see that both vertical morphisms become equivalences, the right by inspection and the left because $\cY$ is aligned with $\cZ$.

\item Applying $L_\cZ$ to the square \Cref{mayer.vietoris.square.for.loc}, we see that both horizontal morphisms become equivalences, the lower by inspection and the upper because $\cZ$ is aligned with $\cY$.

\end{itemize}
So the square \Cref{mayer.vietoris.square.for.loc} is indeed a pullback.
\end{proof}

\subsection{Fundamental operations on stratifications}
\label{subsection.structure.theory}

In this subsection, we record our fundamental operations on stratifications.  For ease of navigation, it is organized into subsubsections.

\begin{local}
In this subsection, we fix a poset $\pos$, a stratification $\cZ_\bullet$ of $\cX$ over $\pos$, down-closed subsets $\sD,\sE \in \Down_\pos$, and a closed subcategory $\cY \in \Cls_\cX$.
\end{local}

\begin{definition}
We respectively say that the stratification $\cZ_\bullet$ is \bit{aligned} or \bit{mutually aligned} with $\cY$ if each of its values $\cZ_p$ is so.
\end{definition}

\begin{definition}
We name the key outputs of this subsection as follows.
\begin{enumerate}

\item \Cref{prop.restricted.stratn} provides a \bit{restricted stratification} of $\cY$ over $\pos$ (under the assumption that $\cZ_\bullet$ is mutually aligned with $\cY$).

\item Given a functor $\w{\cX} \ra \cX$ that is the quotient by a closed subcategory, \Cref{prop.pullback.stratn} provides a \bit{pullback stratification} of $\w{\cX}$ over $\pos$ (under the assumption that $\pos$ is nonempty).

\item \Cref{prop.quotient.stratn} provides a \bit{quotient stratification} of $\cX / \cY$ over $\pos$ (under the assumption that $\cZ_\bullet$ is aligned with $\cY$).

\item Given a functor $\pos \ra \posQ$ between posets, \Cref{prop.pushfwd.stratn} provides a \bit{pushforward stratification} of $\cX$ over $\posQ$.

\item Given a stratification of each stratum $\cX_p$ over a poset $\sR_p$, \Cref{prop.refined.stratn} provides a \bit{refined stratification} of $\cX$ over the \textit{wreath product} poset $\pos \wr \sR_\bullet$.

\end{enumerate} 
\end{definition}


\subsubsection{Preliminary results}

\begin{observation}
\label{obs.restricted.stratn.over.D}
The evident factorization
\[ \begin{tikzcd}
\pos
\arrow{r}{\cZ_\bullet}
&
\Cls_\cX
\\
\sD
\arrow[hook]{u}
\arrow[dashed]{r}[swap]{\cZ_\bullet}
&
\Cls_{\cZ_\sD}
\arrow[hook]{u}
\end{tikzcd} \]
is a stratification of $\cZ_\sD$ over $\sD$, whose $p\th$ stratum is $\cX_p$ for every $p \in \sD \subseteq \pos$.
\end{observation}

\begin{lemma}
\label{closed.subcats.are.mutually.aligned}
The closed subcategories $\cZ_\sD,\cZ_\sE \in \Cls_\cX$ are mutually aligned and $(\cZ_\sD \cap \cZ_\sE) = \cZ_{\sD \cap \sE}$.
\end{lemma}

\begin{proof}
We first show that the lax-commutative square
\begin{equation}
\label{lax.comm.square.of.generalizn.of.yo.commutativity}
\begin{tikzcd}
\cZ_{\sD \cap \sE}
\arrow[hook]{r}{i_L}[swap, xshift=-0.1cm, yshift=-0.4cm]{\rotatebox{-45}{$\Rightarrow$}}
&
\cZ_\sD
\\
\cZ_\sE
\arrow[hook]{r}[swap]{i_L}
\arrow{u}{y}
&
\cX
\arrow{u}[swap]{y}
\end{tikzcd}
\end{equation}
determined by the commutative square
\begin{equation}
\label{comm.square.of.iLs.giving.lax.comm.square.of.generalizn.of.yo.commutativity}
\begin{tikzcd}
\cZ_{\sD \cap \sE}
\arrow[hook]{r}{i_L}
\arrow[hook]{d}[swap]{i_L}
&
\cZ_\sD
\arrow[hook]{d}{i_L}
\\
\cZ_\sE
\arrow[hook]{r}[swap]{i_L}
&
\cX
\end{tikzcd}
\end{equation}
commutes.  By an identical argument to that proving the equivalence $\Cref{part.lem.alignment.itself} \Leftrightarrow \Cref{part.lem.equivalent.characterizations.of.alignment.factorization.is.radjt}$ of \Cref{lem.equivalent.characterizations.of.alignment}, it suffices to show that there exists a factorization
\begin{equation}
\label{factorization.for.alignment.from.down.closed}
\begin{tikzcd}
\cZ_{\sD \cap \sE}
\arrow[hook]{r}{i_L}
&
\cZ_\sD
\\
\cZ_\sE
\arrow[hook]{r}[swap]{i_L}
\arrow[dashed]{u}
&
\cX
\arrow{u}[swap]{y}
\end{tikzcd}~.
\end{equation}
In the special case that $\sD = (^\leq p)$ and $\sE = (^\leq q)$, this is precisely the stratification condition.  In order to prove the general case, we first prove the intermediate case that $\sE \in \Down_\pos$ is arbitrary but $\sD = (^\leq p)$ for some $p \in \pos$.  Then, for each $q \in \sE$, we have a factorization
\begin{equation}
\label{factorization.from.condition.star.in.intermediate.case.that.D.is.leq.p}
\begin{tikzcd}
\cZ_{(^\leq p) \cap (^\leq q)}
\arrow[hook]{r}{i_L}
&
\cZ_{(^\leq p) \cap \sE}
\arrow[hook]{r}{i_L}
&
\cZ_p
\\
\cZ_q
\arrow[dashed]{u}
\arrow[hook]{r}[swap]{i_L}
&
\cZ_\sE
\arrow[hook]{r}[swap]{i_L}
&
\cX
\arrow{u}[swap]{y}
\end{tikzcd}
\end{equation}
by the stratification condition.  So, the intermediate case follows from the facts that $\cZ_\sE := \brax{\cZ_q}_{q \in \sE}$ and that all solid morphisms in diagram \Cref{factorization.from.condition.star.in.intermediate.case.that.D.is.leq.p} preserve colimits.  Passing to the general case, for each $p \in \sD$ let us extend the lax-commutative square \Cref{lax.comm.square.of.generalizn.of.yo.commutativity} to a diagram
\begin{equation}
\label{lax.comm.rectangle.with.yo.comm.at.the.bottom}
\begin{tikzcd}
\cZ_{(^\leq p) \cap \sE}
\arrow[hook]{r}{i_L}
&
\cZ_p
\\
\cZ_{\sD \cap \sE}
\arrow[hook]{r}{i_L}[swap, xshift=-0.1cm, yshift=-0.4cm]{\rotatebox{-45}{$\Rightarrow$}}
\arrow{u}{y}
&
\cZ_\sD
\arrow{u}[swap]{y}
\\
\cZ_\sE
\arrow[hook]{r}[swap]{i_L}
\arrow{u}{y}
&
\cX
\arrow{u}[swap]{y}
\end{tikzcd}~,
\end{equation}
in which the upper (commutative) square is obtained by applying the intermediate case to the restricted stratification of $\cZ_\sD$ over $\sD$ of \Cref{obs.restricted.stratn.over.D} (replacing $\sD,\sE \in \Down_\pos$ respectively with $(^\leq p),(\sD \cap \sE) \in \Down_\sD$).  Note too that the intermediate case is precisely the assertion that the composite lax-commutative rectangle of diagram \Cref{lax.comm.rectangle.with.yo.comm.at.the.bottom} is in fact commutative.  So, the lax-commutative square \Cref{lax.comm.square.of.generalizn.of.yo.commutativity} must be commutative because the functors $\{ \cZ_\sD \xra{y} \cZ_p \}_{p \in \sD}$ are jointly conservative.

Now, the commutativity of the square \Cref{comm.square.of.iLs.giving.lax.comm.square.of.generalizn.of.yo.commutativity} implies that $\cZ_{\sD \cap \sE} \subseteq (\cZ_\sD \cap \cZ_\sE)$.  On the other hand, the existence of the factorization \Cref{factorization.for.alignment.from.down.closed} implies that $(\cZ_\sD \cap \cZ_\sE) \subseteq \cZ_{\sD \cap \sE}$, as any object of $(\cZ_\sD \cap \cZ_\sE)$ must lie in the image of the composite $\cZ_\sE \xhookra{i_L} \cX \xra{y} \cZ_\sD$.  So indeed, $(\cZ_\sD \cap \cZ_\sE) = \cZ_{\sD \cap \sE}$.  Hence, the factorization \Cref{factorization.for.alignment.from.down.closed} witnesses $\cZ_\sE$ as being aligned with $\cZ_\sD$.  That $\cZ_\sD$ is aligned with $\cZ_\sE$ follows by reversing the roles of $\sD$ and $\sE$.
\end{proof}

\begin{remark}
Evidently, a prestratification $\pos \xra{\cZ_\bullet'} \Cls_{\cX'}$ satisfies the stratification condition if for all $p,q \in \pos$ we have that
$
\cZ'_q
\algnd
\cZ'_p
$
and
$
(\cZ'_p \cap \cZ'_q) = \cZ'_{(^\leq p) \cap (^\leq q)}
$.
\Cref{closed.subcats.are.mutually.aligned} provides a converse.
\end{remark}

\subsubsection{Restricted stratifications}

\begin{prop}
\label{prop.restricted.stratn}
Suppose that the stratification $\pos \xra{\cZ_\bullet} \Cls_\cX$ is mutually aligned with $\cY \in \Cls_\cX$.
\begin{enumerate}

\item\label{part.restricted.stratn}

The composite functor
\begin{equation}
\label{restricted.stratification}
\begin{tikzcd}[row sep=0cm, column sep=1.5cm]
\pos
\arrow{r}{\cZ_\bullet}
&
\Cls_\cX^{\algnd \cY}
\arrow{r}{i_L^{-1}}
&
\Cls_{\cY}
\\
\rotatebox{90}{$\in$}
&
&
\rotatebox{90}{$\in$}
\\
p
\arrow[maps to]{rr}
&
&
i_L^{-1}(\cZ_p)
\end{tikzcd}
\end{equation}
is a stratification of $\cY$ over $\pos$.

\item\label{part.strata.of.restricted.stratn}

For any $p \in \pos$, the $p\th$ stratum of the stratification \Cref{restricted.stratification} is $i_L^{-1}(\cX_p)$.

\end{enumerate}
\end{prop}

\begin{proof}
We begin with part \Cref{part.restricted.stratn}.  By \Cref{lemma.preimages.preserve.colimits}\Cref{lemma.part.iL.preimage.preserves.colimits}, the composite functor \Cref{restricted.stratification} is a prestratification.  So, it remains to verify the stratification condition.  Choose any $p,q \in \pos$, and consider the diagram
\[ \begin{tikzcd}
&
\cZ_{(^\leq p) \cap (^\leq q)}
\arrow[hook]{rr}{i_L}
\arrow[leftarrow, dashed]{dd}
&
&
\cZ_p
\\
i_L^{-1}(\cZ_{(^\leq p) \cap (^\leq q)})
\arrow[hook]{ru}[sloped]{i_L}
\arrow[hook, crossing over]{rr}[pos=0.7]{i_L}
&
&
i_L^{-1}(\cZ_p)
\arrow[hook]{ru}[sloped, swap]{i_L}
\\
&
\cZ_q
\arrow[hook]{rr}[pos=0.3]{i_L}
&
&
\cX
\arrow{uu}[swap]{y}
\\
i_L^{-1}(\cZ_q)
\arrow[hook]{rr}[swap]{i_L}
\arrow[hook]{ru}[sloped]{i_L}
\arrow[dashed]{uu}
&
&
i_L^{-1}(\cX)
\arrow[crossing over]{uu}[pos=0.3]{y}
\arrow[hook]{ru}[sloped, swap]{i_L}
&
&[-2.6cm]
= \cY
\end{tikzcd} \]
in which the upper and lower squares commute by definition of $i_L^{-1}$ and the right square commutes because $\cY$ is aligned with $\cZ_p$. The back factorization exists because $\cZ_\bullet$ is a stratification, and hence the front factorization exists because the upper square is a pullback.  So, the stratification condition follows from the identification
\[
i_L^{-1}(\cZ_{(^\leq p) \cap (^\leq q)})
:=
i_L^{-1}( \brax{ \cZ_r }_{r \in (^\leq p) \cap (^\leq q)})
\simeq
\brax{ i_L^{-1}(\cZ_r)}_{r \in (^\leq p) \cap (^\leq q)}
\]
resulting from \Cref{lemma.preimages.preserve.colimits}\Cref{lemma.part.iL.preimage.preserves.colimits}.

We now turn to part \Cref{part.strata.of.restricted.stratn}.  Note that $\cY$ is aligned with $\cZ_{^< p}$ by \Cref{lem.colimits.preserve.alignment}.  By \Cref{obs.if.algnd.with.Y.and.W.then.iLinverse.algnd.with.W}, it follows that $i_L^{-1}(\cZ_p)$ is also aligned with $\cZ_{^< p}$.  Using this and \Cref{lemma.preimages.preserve.colimits}\Cref{lemma.part.iL.preimage.preserves.colimits}, we identify the $p\th$ stratum of the stratification \Cref{restricted.stratification} as
\begin{align*}
\frac{i_L^{-1}(\cZ_p)}{i_L^{-1}(\cZ_{^< p})}
&\simeq
\ker( i_L^{-1}(\cZ_p) \xlongra{y} i_L^{-1}(\cZ_{^< p}) )
\simeq
\ker ( i_L^{-1}(\cZ_p) \xlongra{y} i_L^{-1}(\cZ_{^< p}) \xlonghookra{i_L} \cZ_{^< p} )
\\
& \simeq
\ker ( i_L^{-1}( \cZ_p) \xlonghookra{i_L} \cZ_p \xlongra{y} \cZ_{^< p} )
\simeq
i_L^{-1}(\cX_p)
~,
\end{align*}
as desired.
\end{proof}

\begin{remark}
\label{rmk.restricted.stratn.recovers.that.for.down.closed}
Taking $\cY = \cZ_\sD$ in \Cref{prop.restricted.stratn}, we obtain a stratification of $\cZ_\sD$ over $\pos$, whose restriction to $\sD$ is the stratification of $\cZ_\sD$ over $\sD$ of \Cref{obs.restricted.stratn.over.D}.\footnote{In general, if the stratification $\pos \xra{\cZ_\bullet} \Cls_\cX$ has the property that $\cZ_p = \cZ_{(^\leq p) \cap \sD}$ for every $p \in \pos$, then its restriction $\sD \hookra \pos \xra{\cZ_\bullet} \Cls_\cX$ is evidently also a stratification.}
\end{remark}

\subsubsection{Pullback stratifications}

\begin{prop}
\label{prop.pullback.stratn}
Let $\w{\cX}$ be a presentable stable $\infty$-category.  Suppose that $\w{\cX} \xra{\pi} \cX$ is the quotient by a closed subcategory (i.e.\! the functor $p_L$ in a recollement), and suppose further that $\pos$ is nonempty.
\begin{enumerate}

\item\label{part.pullback.stratn}

The composite functor
\begin{equation}
\label{pullback.stratn}
\begin{tikzcd}[column sep=1.5cm, row sep=0cm]
\pos
\arrow{r}{\cZ_\bullet}
&
\Cls_\cX
\arrow{r}{\pi^{-1}}
&
\Cls_{\w{\cX}}
\\
\rotatebox{90}{$\in$}
&
&
\rotatebox{90}{$\in$}
\\
p
\arrow[maps to]{rr}
&
&
\pi^{-1}(\cZ_p)
\end{tikzcd}
\end{equation}
is a stratification of $\w{\cX}$ over $\pos$.

\item\label{part.strata.of.pullback.stratn}

For any $p \in \pos$, the $p\th$ stratum of the stratification \Cref{pullback.stratn} is $\cX_p$ if $(^< p) \not= \es$ and is $\pi^{-1}(\cX_p)$ if $(^< p) = \es$.

\end{enumerate}
\end{prop}

\begin{proof}
We begin with part \Cref{part.pullback.stratn}.  Because $\pos$ is nonempty, the functor \Cref{pullback.stratn} is a prestratification by \Cref{lemma.preimages.preserve.colimits}\Cref{lemma.part.pL.preimage.preserves.colimits}.  So, it remains to verify the stratification condition.  Choose any $p,q \in \pos$, and consider the diagram
\[ \begin{tikzcd}
&
\cZ_{(^\leq p) \cap (^\leq q)}
\arrow[hook]{rr}{i_L}
\arrow[leftarrow, dashed]{dd}
&
&
\cZ_p
\\
\pi^{-1}(\cZ_{(^\leq p) \cap (^\leq q)})
\arrow{ru}
\arrow[hook, crossing over]{rr}[pos=0.7]{i_L}
&
&
\pi^{-1}(\cZ_p)
\arrow{ru}
\\
&
\cZ_q
\arrow[hook]{rr}[pos=0.3]{i_L}
&
&
\cX
\arrow{uu}[swap]{y}
\\
\pi^{-1}(\cZ_q)
\arrow[hook]{rr}[swap]{i_L}
\arrow{ru}
\arrow[dashed]{uu}
&
&
\pi^{-1}(\cX)
\arrow{ru}[sloped, swap]{\pi}
\arrow[crossing over]{uu}[pos=0.3]{y}
&
&[-2.6cm]
= \w{\cX}
\end{tikzcd} \]
in which the upper and lower squares commute by definition of $\pi^{-1}$ and the right square commutes by \Cref{lemma.all.about.aligned.subcats}\Cref{part.alignment.lemma.induced.map.on.quotients}\Cref{subpart.alignment.lemma.yo.and.pL.commutativity} and \Cref{obs.mutually.aligned.if.containment}.  The back factorization exists because $\cZ_\bullet$ is a stratification, and hence the front factorization exists because the upper square is a pullback.  So, the stratification condition follows from the identification
\[
\pi^{-1}(\cZ_{(^\leq p) \cap (^\leq q)})
:=
\pi^{-1}(\brax{\cZ_r}_{r \in (^\leq p) \cap (^\leq q)})
\simeq
\brax{\pi^{-1}(\cZ_r)}_{r \in (^\leq p) \cap (^\leq q)}
\]
resulting from \Cref{lemma.preimages.preserve.colimits}\Cref{lemma.part.pL.preimage.preserves.colimits}.

We now turn to part \Cref{part.strata.of.pullback.stratn}.  In the case that $(^< p) \not= \es$, using \Cref{lemma.preimages.preserve.colimits}\Cref{lemma.part.pL.preimage.preserves.colimits} and \Cref{prop.image.and.preimage.of.closed.subcats.are.closed} we identify the $p\th$ stratum of the stratification \Cref{pullback.stratn} as
\[
\frac{\pi^{-1}(\cZ_p)}{ \brax{\pi^{-1}(\cZ_{p'})}_{p' < p}}
\simeq
\frac{\pi^{-1}(\cZ_p)}{\pi^{-1}(\brax{\cZ_{p'}}_{p' < p})}
=:
\frac{\pi^{-1}(\cZ_p)}{\pi^{-1}(\cZ_{^< p})}
\simeq
\frac{\pi^{-1}(\cZ_p)/\pi^{-1}(0)}{\pi^{-1}(\cZ_{^< p})/\pi^{-1}(0)}
\simeq
\frac{\cZ_p}{\cZ_{^< p}}
=:
\cX_p
~,
\]
as desired.  In the case that $(^< p) = \es$, we identify the $p\th$ stratum of the stratification \Cref{pullback.stratn} as
\[
\frac{\pi^{-1}(\cZ_p)}{ \brax{\pi^{-1}(\cZ_{p'})}_{p' < p}}
=
\frac{\pi^{-1}(\cZ_p)}{0}
=
\pi^{-1}(\cZ_p)
\simeq
\pi^{-1}( \cX_p)~,
\]
as desired.
\end{proof}

\subsubsection{Quotient stratifications}

\begin{prop}
\label{prop.quotient.stratn}
Suppose that the stratification $\pos \xra{\cZ_\bullet} \Cls_\cX$ is aligned with $\cY \in \Cls_\cX$.
\begin{enumerate}

\item\label{part.quotient.stratn}

The composite functor
\begin{equation}
\label{stratn.of.quotient.by.Y}
\begin{tikzcd}[row sep=0cm, column sep=1.5cm]
\pos
\arrow{r}{\cZ_\bullet}
&
\Cls_\cX^{\algnd \cY}
\arrow{r}{p_L}
&
\Cls_{\cX / \cY}
\\
\rotatebox{90}{$\in$}
&
&
\rotatebox{90}{$\in$}
\\
p
\arrow[maps to]{rr}
&
&
p_L(\cZ_p)
\end{tikzcd}
\end{equation}
is a stratification of $\cX / \cY$ over $\pos$.

\item\label{part.strata.of.quotient.stratn}
Suppose further that $\cY$ is aligned with the stratification $\cZ_\bullet$.  For any $p \in \pos$, the subcategory $i_L^{-1}(\cX_p) \subseteq \cX_p$ is closed and the $p\th$ stratum of the stratification \Cref{stratn.of.quotient.by.Y} is $\cX_p/i_L^{-1}(\cX_p)$.

\end{enumerate}
\end{prop}

\begin{proof}
Over the course of the proof, for clarity we write $\cY \xra{\iota} \cX \xra{\pi} \cX/\cY$ for the canonical functors.

We begin with part \Cref{part.quotient.stratn}.  The functor \Cref{stratn.of.quotient.by.Y} is a prestratification by \Cref{prop.image.and.preimage.of.closed.subcats.are.closed} and the fact that $\pi(\cX) = \cX / \cY$. It remains to check the stratification condition.  For any $p,q\in \pos$, we have the solid commutative diagram
\[
\begin{tikzcd}
&
\pi(\cZ_{(^\leq p) \cap (^\leq q)})
\arrow[hook]{rr}{i_L}
\arrow[leftarrow, dashed]{dd}
&
&
\pi(\cZ_p)
\\
\cZ_{(^\leq p) \cap (^\leq q)}
\arrow[crossing over, hook]{rr}[pos=0.7]{i_L}
\arrow{ru}[sloped]{p_L}
&
&
\cZ_p
\arrow{ru}[sloped, swap]{p_L}
\\
&
\pi(\cZ_q)
\arrow[hook]{rr}[pos=0.3]{i_L}
&
&
\pi(\cX)
\arrow[crossing over]{uu}[swap]{y}
&
&[-1.9cm] = \cX/\cY
\\
\cZ_q
\arrow{ru}[sloped]{p_L}
\arrow[dashed]{uu}
\arrow[hook]{rr}[swap]{i_L}
&
&
\cX
\arrow{ru}[sloped, swap]{p_L}
\arrow[crossing over]{uu}[pos=0.3]{y}
\end{tikzcd} ~,
\]
in which
the bottom and top squares commute by the functoriality of presentable quotients and
the right square commutes by \Cref{lemma.all.about.aligned.subcats}\Cref{part.alignment.lemma.induced.map.on.quotients}\Cref{subpart.alignment.lemma.yo.and.pL.commutativity}.  The front factorization exists because $\cZ_\bullet$ is a stratification, and hence the back factorization exists because the functor $\cZ_q \xra{p_L} \pi(\cZ_q)$ is surjective.  So, the stratification condition follows from the identification
\[
\pi(\cZ_{(^\leq p) \cap (^\leq q)})
:=
\pi(\brax{\cZ_r}_{r \in (^\leq p) \cap (^\leq q)})
\simeq
\brax{ \pi^{-1}(\cZ_r)}_{r \in (^\leq p) \cap (^\leq q)}
\]
resulting from \Cref{prop.image.and.preimage.of.closed.subcats.are.closed}.

We now proceed to part \Cref{part.strata.of.quotient.stratn}.  First of all, $\cY$ is aligned with $\cZ_{^< p}$ by \Cref{lem.colimits.preserve.alignment}, and thereafter $\cY \cap \cZ_p$ is aligned with $\cZ_{^< p}$ by \Cref{obs.if.algnd.with.Y.and.W.then.iLinverse.algnd.with.W}.  Hence, the fact that $\iota^{-1}(\cX_p) \in \Cls_{\cX_p}$ follows from \Cref{lemma.all.about.aligned.subcats}\Cref{part.alignment.lemma.induced.map.on.quotients}\Cref{subpart.alignment.lemma.image.is.closed} along with the observation that
\[
\frac{\cY \cap \cZ_p}{\cY \cap \cZ_{^< p}}
\simeq
\cY \cap \cX_p
=:
\iota^{-1}(\cX_p)
~.
\]
Using \Cref{prop.image.and.preimage.of.closed.subcats.are.closed}, we now identify the $p\th$ stratum of the stratification \Cref{stratn.of.quotient.by.Y} as
\[
\frac{\pi(\cZ_p)}{\brax{\pi(\cZ_{p'})}_{p' < p}}
\simeq
\frac{\pi(\cZ_p)}{\pi(\cZ_{^< p})}
\simeq
\frac{\cZ_p / (\cY \cap \cZ_p)}{\cZ_{^< p} / (\cY \cap \cZ_{^< p})}
\simeq
\frac{\cZ_p / \cZ_{^< p}}{ (\cY \cap \cZ_p) / (\cY \cap \cZ_{^< p})}
\simeq
\frac{\cX_p}{\cY \cap \cX_p}
=:
\frac{\cX_p}{\iota^{-1}(\cX_p)}
~,
\]
as desired.
\end{proof}

\begin{observation}
\label{obs.quotient.stratn.from.down.closed.over.smaller.poset}
Taking $\cY = \cZ_\sD$ in \Cref{prop.quotient.stratn}, we obtain a stratification of $\cX / \cZ_\sD =: \cX_{\pos \backslash \sD}$ over $\pos$, whose $p\th$ stratum is $0$ whenever $p \in \sD$ and is $\cX_p$ whenever $p \notin \sD$ (because in this case $((^\leq p) \cap \sD) \subseteq (^< p)$ (and using \Cref{closed.subcats.are.mutually.aligned})). Evidently, the restriction to $(\pos \backslash \sD) \subseteq \pos$ is also a stratification of $\cX_{\pos \backslash \sD}$.\footnote{In general, if the stratification $\pos \xra{\cZ_\bullet} \Cls_\cX$ has the property that $\cZ_p = 0$ for all $p \in \sD$, then its restriction $(\pos \backslash \sD) \hookra \pos \xra{\cZ_\bullet} \Cls_\cX$ is also a stratification.} And in fact, the entire gluing diagram of $\cX/\cZ_\sD$ with respect to this latter stratification is the restriction of that of $\cX$, in the sense that we have a pullback diagram
\[ \begin{tikzcd}
\GD(\cX/\cZ_\sD)
\arrow[hook]{r}
\arrow{d}
&
\GD(\cX)
\arrow{d}
\\
\pos \backslash \sD
\arrow[hook]{r}
&
\pos
\end{tikzcd}~. \]
This follows from the existence of a factorization
\[ \begin{tikzcd}[row sep=1.5cm, column sep=1.5cm]
\GD(\cX/\cZ_\sD)
\arrow[hook]{d}
\arrow[dashed]{rr}
&
&
\GD(\cX)
\arrow[hook]{d}
\\
(\cX/\cZ_\sD) \times (\pos \backslash \sD)
\arrow[hook]{r}[swap]{\nu \times \id_{\pos \backslash \sD}}
&
\cX \times (\pos \backslash \sD)
\arrow[hook]{r}
&
\cX \times \pos
\end{tikzcd}~, \]
which itself results from \Cref{lemma.all.about.aligned.subcats}\Cref{part.alignment.lemma.induced.map.on.quotients}\Cref{subpart.alignment.lemma.yo.and.pL.commutativity} (which applies by \Cref{closed.subcats.are.mutually.aligned}) by passing to right adjoints in the commutative square \Cref{yo.and.pL.commutativity.for.aligned.subcats}.
\end{observation}

\subsubsection{Pushforward stratifications}

\begin{prop}
\label{prop.pushfwd.stratn}
Suppose that $\pos \ra \posQ$ is any functor between posets.
\begin{enumerate}
\item\label{item.pushfwd.stratn}
The functor
\begin{equation}
\label{pushfwd.stratn}
\begin{tikzcd}[row sep=0cm]
\posQ
\arrow{r}
&
\Cls_\cX
\\
\rotatebox{90}{$\in$}
&
\rotatebox{90}{$\in$}
\\
q
\arrow[maps to]{r}
&
\cZ_q
&[-1.2cm]
:= \cZ_{\pos_{^\leq q}}
\end{tikzcd}
\end{equation}
defines a stratification of $\cX$ over $\posQ$.
\item\label{item.identify.strata.of.pushfwd.stratn}
For any $q \in \posQ$, the $q\th$ stratum of the stratification \Cref{pushfwd.stratn} is $\cX_{\pos_q}$.
\end{enumerate}
\end{prop}

\begin{proof}
We begin with part \Cref{item.pushfwd.stratn}.  Since $\cX = \brax{\cZ_p}_{p \in \pos}$, then also $\cX = \brax{\cZ_q }_{q \in \posQ}$.  So, it remains to check the stratification condition.  For any $q,r \in \posQ$, we must show that there is a factorization
\[ \begin{tikzcd}
\cZ_{\pos_{(^\leq q) \cap (^\leq r)}}
\arrow[hook]{r}{i_L}
&
\cZ_{\pos_{^\leq q}}
\\
\cZ_{\pos_{\leq r}}
\arrow[hook]{r}[swap]{i_L}
\arrow[dashed]{u}
&
\cX
\arrow{u}[swap]{y}
\end{tikzcd} \]
This follows from \Cref{closed.subcats.are.mutually.aligned} by taking $\sD = \pos_{^\leq q}$ and $\sE = \pos_{^\leq r}$ and noting that $\pos_{^\leq q} \cap \pos_{^\leq r} = \pos_{(^\leq q) \cap (^\leq r)}$.

We now proceed to part \Cref{item.identify.strata.of.pushfwd.stratn}.  We write $\cZ := \cZ_{\pos_{^\leq q}}$ for simplicity, and we apply \Cref{obs.restricted.stratn.over.D} (taking $\sD = \pos_{^\leq q}$) to pass to the restricted stratification
\[
\pos_{^\leq q}
\xra{\cZ_\bullet}
\Cls_\cZ
\]
of $\cZ$ over $\pos_{^\leq q}$ with the same strata.  Writing
\[
\Span
:=
\left\{ \begin{tikzcd}
s
\arrow{r}
\arrow{d}
&
t
\\
u
\end{tikzcd} \right\}
\]
for the walking span, we define a functor $\pos_{^\leq q} \xra{\pi} \Span$ between posets according to the prescriptions
\[
\pi^{-1}(s) = ( ^< \pos_q )
~,
\qquad
\pi^{-1}(t) =  ( ^\leq \pos_q \backslash ^< \pos_q )
~,
\qquad
\text{and}
\qquad
\pi^{-1}(u) =  ( \pos_{^< q} \backslash ^< \pos_q )
~.
\]
By part \Cref{item.pushfwd.stratn}, we obtain a stratification of $\cZ$ over $\Span$.  Thereafter, applying \Cref{obs.quotient.stratn.from.down.closed.over.smaller.poset} (and \Cref{prop.quotient.stratn}) with $\sD = \{ s \ra u \} \in \Down_\Span$, we obtain a quotient stratification
\[ \begin{tikzcd}[row sep=0cm]
\{ t \}
\arrow{r}
&
\Cls_{\cZ / \cZ_u}
\\
\rotatebox{90}{$\in$}
&
\rotatebox{90}{$\in$}
\\
t
\arrow[maps to]{r}
&
\cZ_t / \cZ_s
\end{tikzcd} \]
over the one-element poset (since $( ^\leq t) \cap \{ s \ra u \} = \{ s \}$).  In particular, we find that
\[
\cX_q
:=
\cZ_q / \cZ_{^< q}
:=
\cZ_{\pos_{^\leq q}} / \cZ_{\pos_{^< q}}
=:
\cZ / \cZ_u
\simeq
\cZ_t / \cZ_s
=:
\cZ_{^\leq \pos_q} / \cZ_{^< \pos_q}
=:
\cX_{\pos_q}
~,
\]
as desired.
\end{proof}

\subsubsection{Refined stratifications}

\begin{definition}
Given a functor
\[
\iota_0 \pos
\xra{\sR_\bullet}
\Poset
~,
\]
we define the \bit{wreath product} of $\pos$ with $\sR_\bullet$ to be the poset $\pos \wr \sR_\bullet$ whose objects are pairs $(p,r)$ where $p \in \pos$ and $r \in \sR_p$ equipped with the lexicographic ordering: $(p,r) \leq (p',r')$ in $\pos \wr \sR_\bullet$ if and only if either $p < p'$ in $\pos$ or else $p=p'$ in $\pos$ and $r \leq r'$ in $\sR_p$.  This comes equipped with a canonical functor
\[
\begin{tikzcd}[row sep=0cm]
\pos \wr \sR_\bullet
\arrow{r}
&
\pos
\\
\rotatebox{90}{$\in$}
&
\rotatebox{90}{$\in$}
\\
(p,r)
\arrow[maps to]{r}
&
p
\end{tikzcd}~.
\]
\end{definition}


\begin{prop}
\label{prop.refined.stratn}
Choose any functor $\iota_0 \pos \xra{\sR_\bullet} \Poset$ and, for each $p \in \pos$, a stratification
\begin{equation}
\label{stratn.of.stratum}
\begin{tikzcd}[row sep=0cm]
\sR_p
\arrow{r}{(\cY_p)_\bullet}
&
\Cls_{\cX_p}
\\
\rotatebox{90}{$\in$}
&
\rotatebox{90}{$\in$}
\\
r
\arrow[maps to]{r}
&
(\cY_p)_r
\end{tikzcd}
~.
\end{equation}
\begin{enumerate}
\item\label{item.refined.stratn}
The functor
\begin{equation}
\label{refined.stratn}
\begin{tikzcd}[row sep=0cm]
\pos \wr \sR_\bullet
\arrow{r}{\w{\cZ}_\bullet}
&
\Cls_\cX
\\
\rotatebox{90}{$\in$}
&
\rotatebox{90}{$\in$}
\\
(p,r)
\arrow[maps to]{r}
&
\w{\cZ}_{(p,r)}
&[-1.1cm]
:= p_L^{-1}((\cY_p)_r)
\end{tikzcd}
\end{equation}
defines a stratification of $\cX$ over $\pos \wr \sR_\bullet$.
\item\label{item.identify.strata.of.refined.stratn}
For any $(p,r) \in \pos \wr \sR_\bullet$, the $(p,r)\th$ stratum of the stratification \Cref{refined.stratn} is $(\cX_p)_r$.
\end{enumerate}
\end{prop}

\begin{proof}
We begin with part \Cref{item.refined.stratn}.

We first verify that the functor \Cref{refined.stratn} defines a prestratification.  For this, consider any $p \in \pos$.  If $\sR_p = \es$, then it must be the case that $\cX_p = 0$ and so $\cZ_p = \cZ_{^< p}$.  Otherwise, we have $\cZ_p = \brax{ \w{\cZ}_{(p,r)} }_{r \in \sR_p}$ by \Cref{lemma.preimages.preserve.colimits}\Cref{lemma.part.pL.preimage.preserves.colimits}.  Hence, we find that
\[
\cX
=
\brax{ \cZ_p }_{p \in \pos}
=
\brax{ \cZ_p }_{\{ p \in \pos : \sR_p \not= \es \}}
=
\brax{ \brax{ \w{\cZ}_{(p,r)} }_{r \in \sR_p} }_{\{ p \in \pos : \sR_p \not= \es \}}
=
\brax{ \w{\cZ}_{(p,r)} }_{(p,r) \in \pos \wr \sR_\bullet}
~.
\]
We note here that the same argument shows that for any $\sD \in \Down_\pos$ we have an identification
\begin{equation}
\label{equivalence.on.down.closed.of.wreath}
\w{\cZ}_{(\pos \wr \sR_\bullet)_\sD}
=
\cZ_\sD
\end{equation}
in $\Cls_\cX$.

We now verify the stratification condition.  By \Cref{obs.condn.star.vacuous.if.P.totally.ordered}, it suffices to verify it for incomparable pairs of elements of $\pos \wr \sR_\bullet$.  There are two types of such pairs: pairs $(p,r)$ and $(q,s)$ where $p$ and $q$ are incomparable in $\pos$, and pairs $(p,r)$ and $(p,s)$ where $r$ and $s$ are incomparable in $\sR_p$.  We address these two cases in turn.
\begin{itemize}
\item Choose elements $(p,r),(q,s) \in \pos \wr \sR_\bullet$ such that $p$ and $q$ are incomparable in $\pos$.  Note the equality
\[
(^\leq (p,r) ) \cap ( ^\leq (q,s) )
=
(\pos \wr \sR_\bullet)_{(^\leq p) \cap (^\leq q)}
\]
in $\Down_{\pos \wr \sR_\bullet}$.  Hence, we obtain a diagram
\[ \begin{tikzcd}
\w{\cZ}_{(^\leq (p,r) ) \cap ( ^\leq (q,s) )} =
&[-1cm]
\cZ_{(^\leq p) \cap (^\leq q)}
\arrow[hook]{r}{i_L}
&
\w{\cZ}_{(p,r)}
\arrow[hook]{r}{i_L}
&
\cZ_p
\\
&
\w{\cZ}_{(q,s)}
\arrow[hook]{r}[swap]{i_L}
&
\cZ_{q}
\arrow[hook]{r}[swap]{i_L}
\arrow[dashed]{ul}
&
\cX
\arrow{u}[swap]{y}
\end{tikzcd}~, \]
in which the identification is \Cref{equivalence.on.down.closed.of.wreath} with $\sD = (^\leq p) \cap (^\leq q)$ and the factorization is guaranteed by the stratification condition for the stratification of $\cX$ over $\pos$.

\item Given elements $(p,r) , (p,s) \in \pos \wr \sR_\bullet$ such that $r$ and $s$ are incomparable in $\sR_p$, the factorization
\[ \begin{tikzcd}
\w{\cZ}_{(^\leq (p,r) ) \cap ( ^\leq (p,s) )}
\arrow[hook]{rr}{i_L}
&
&
\w{\cZ}_{(p,r)}
\\
&
&
\cZ_p
\arrow{u}[swap]{y}
\\
\w{\cZ}_{(p,s)}
\arrow[dashed]{uu}
\arrow[hook]{r}[swap]{i_L}
&
\cZ_p
\arrow{ru}[sloped]{\id}
\arrow[hook]{r}[swap]{i_L}
&
\cX
\arrow{u}[swap]{y}
\end{tikzcd} \]
follows from \Cref{prop.pullback.stratn}.
\end{itemize}

We now proceed to part \Cref{item.identify.strata.of.refined.stratn}.  In light of the equalities
\[
^\leq (p,r)
=
\{ (p,r') \in \pos \wr \sR_\bullet : r' \leq r \}
\cup
\{ (p',r') \in \pos \wr \sR_\bullet : p' < p \}
=:
(p, (^\leq r))
\cup
(\pos \wr \sR_\bullet)_{^< p}
\]
and
\[
^< (p,r)
=
\{ (p,r') \in \pos \wr \sR_\bullet : r' <r \}
\cup
\{ (p',r') \in \pos \wr \sR_\bullet : p' < p \}
=:
(p,(^< r))
\cup
(\pos \wr \sR_\bullet)_{^< p}
\]
in $\Down_{\pos \wr \sR_\bullet}$, we find that
\[
\cX_{(p,r)}
:=
\frac{\w{\cZ}_{(p,r)}}{\w{\cZ}_{^< (p,r)}}
\simeq
\frac{\w{\cZ}_{(p,r)} / \w{\cZ}_{(\pos \wr \sR_\bullet)_{^< p}}}{\w{\cZ}_{^<(p,r)} / \w{\cZ}_{(\pos \wr \sR_\bullet)_{^< p}}}
=
\frac{\w{\cZ}_{(p,r)} / \cZ_{^< p} }{\w{\cZ}_{^<(p,r)} / \cZ_{^< p}}
\simeq
\frac{ (\cY_p)_r }{ (\cY_p)_{^< r} }
=:
(\cX_p)_r
~,
\]
as desired, using the identification \Cref{equivalence.on.down.closed.of.wreath} with $\sD = (^< p)$.
\end{proof}

\section{The $\cO$-monoidal reconstruction theorem}
\label{section.O.mon.reconstrn.thm}

In this section, we upgrade our macrocosm reconstruction theorem (\Cref{intro.thm.cosms}\Cref{intro.main.thm.macrocosm}) to one that accounts for operadic structures (\Cref{intro.thm.O.mon.reconstrn}).  We also establish the adelic stratification (\Cref{intro.thm.balmer}), which is a symmetric monoidal stratification of a presentably symmetric monoidal stable $\infty$-category (satisfying mild finiteness hypotheses) over the specialization poset of its Balmer spectrum.

This section is organized as follows.
\begin{itemize}

\item[\Cref{subsection.O.monoidal.infty.cats}:] We fix an $\infty$-operad $\cO$ (satisfying mild conditions) and recall the notions of $\cO$-monoidal $\infty$-categories and laxly $\cO$-monoidal functors.

\item[\Cref{subsection.ideals}:] We study the appropriate notion of an ideal subcategory of a presentably $\cO$-monoidal stable $\infty$-category.

\item[\Cref{subsection.O.monoidal.stratns}:] We define $\cO$-monoidal stratifications of a presentably $\cO$-monoidal stable $\infty$-category. We unpack the chromatic stratification of $\Spectra$ in \Cref{ex.chromatic.stratn.of.spectra}, which organizes the fundamental objects of chromatic homotopy theory.

\item[\Cref{subsection.O.algebra.objects.in.LModrlaxllaxB}:] We define the $\infty$-category that contains the $\cO$-monoidal gluing diagram of an $\cO$-monoidal stratification.

\item[\Cref{subsection.O.monoidal.reconstruction.thm}:] We prove \Cref{intro.thm.O.mon.reconstrn} as \Cref{thm.s.m.reconstrn}.

\item[\Cref{subsection.tt.geometry}:] We recall the basic notions of tensor-triangular geometry and then prove \Cref{intro.thm.balmer} as \Cref{thm.s.m.stratn.over.balmer.spectrum}. We discuss the adelic stratification of $\Spectra$ in \Cref{ex.adelic.stratn.of.spectra}. We explain how symmetric monoidal stratifications contribute to the theory of tensor-triangular geometry in \Cref{rmk.stratns.helps.tt.geometry}.

\end{itemize}

\subsection{Preliminaries on $\cO$-monoidal $\infty$-categories}
\label{subsection.O.monoidal.infty.cats}

In this subsection, we fix an $\infty$-operad $\cO$ satisfying mild conditions and recall the notions of $\cO$-monoidal $\infty$-categories and laxly $\cO$-monoidal functors.

\begin{remark}
We are primarily interested in symmetric monoidal $\infty$-categories.  Indeed, the reader will not lose much by simply reading every instance of the $\infty$-operad ``$\cO$'' as ``$\Comm$'' (a.k.a.\! ``$\EE_\infty$'', a.k.a.\! $\Fin_*$), every instance of ``$\cO$-monoidal'' as ``symmetric monoidal'', and so on.  However, we work in this greater generality because it requires almost no extra effort and yet encompasses other situations of potential interest, notably ($\EE_1$-)monoidal, braided (i.e.\! $\EE_2$-)monoidal, and more generally $\EE_n$-monoidal $\infty$-categories for any $1 \leq n \leq \infty$ (e.g.\! recall \Cref{rmk.stratn.of.module.cat}).
\end{remark}

\needspace{2\baselineskip}
\begin{notation}
\label{notation.for.operads.etc}
\begin{enumerate}
\item[]

\item\label{item.notation.for.operad}

We fix an $\infty$-operad
\[
\cO
~,
\]
which we assume
\begin{enumeratesub}
\item\label{item.require.unital}
 to be unital,
\item\label{item.require.reduced}
to be reduced (i.e.\! to have a contractible $\infty$-category of colors), and
\item\label{item.require.multiplications}
to have a nonempty space of binary operations.
\end{enumeratesub}
We write
\[
(\cO^\otimes \da \Fin_*) \in \Cat_{/\Fin_*}
\]
for its defining object.

\item

Justified by the fact that the functor $\cO^\otimes \ra \Fin_*$ restricts as an equivalence on underlying $\infty$-groupoids (by the assumption that $\cO$ is reduced), we notationally identify objects of $\cO^\otimes$ with their images in $\Fin_*$; for any $n \geq 0$ we write $\ul{n} := \{1,\ldots,n \} \in \Fin$ and $\ul{n}_+ := \{1,\ldots,n\}_+ \in \Fin_*$.

\item

For any $n \geq 0$, we write
\[ \begin{tikzcd}
\cO(n)
\arrow[hook]{r}
\arrow{d}
&
\hom_{\cO^\otimes} ( \ul{n}_+ , \ul{1}_+ )
\arrow{d}
\\
\pt
\arrow[hook]{r}
&
\hom_{\Fin_*} ( \ul{n}_+ , \ul{1}_+ )
\end{tikzcd} \]
for the fiber over the unique active morphism, the space of $n$-ary operations in $\cO$.

\item

We write
\[
\cO^\otimes_\cls \subseteq \cO^\otimes
\]
for the subcategory of \textit{closed} (a.k.a.\! \textit{inert}) morphisms.

\end{enumerate}
\end{notation}

\begin{remark}
A few comments regarding assumptions on the $\infty$-operad $\cO$ are in order.
\begin{enumerate}

\item
All three assumptions of \Cref{notation.for.operads.etc}\Cref{item.notation.for.operad} are motivated by examples and by a desire for simplicity of exposition; we expect that our results go through (mutatis mutandis) in greater generality.

\item
It follows from assumption \Cref{item.require.reduced} of \Cref{notation.for.operads.etc}\Cref{item.notation.for.operad} that $\cO$ is the underlying $\infty$-operad of an ordinary (i.e.\! single-colored) operad in topological spaces or simplicial sets.

\item
Assumption \Cref{item.require.multiplications} of \Cref{notation.for.operads.etc}\Cref{item.notation.for.operad} is primarily useful in that it allows us to simplify our notation, e.g.\! in \Cref{obs.Idl.in.Cls.stable.under.colims}, \Cref{rmk.check.O.mon.stratn.on.yonedaed.units}, and \Cref{obs.yo.comm.for.s.m.stratns}.  However, it also serves to guarantee that the unique morphism $\EE_0 \ra \cO$ from the initial reduced unital $\infty$-operad is not an equivalence; this is convenient, as a number of our results do not hold as stated in this degenerate case.

\item
The additional assumption that $\cO$ is \textit{quadratic} (i.e.\! that for all $n \geq 2$ every $n$-ary operation is ((possibly only noncanonically) equivalent to) an iterated composite of binary operations) would allow us to very slightly simplify certain conditions in \Cref{subsection.ideals} (from quantifying over all $n \geq 2$ to quantifying merely over $n=2$).

\end{enumerate}
\end{remark}

\needspace{2\baselineskip}
\begin{definition}
\label{defn.O.mon.and.laxly}
\begin{enumerate}
\item[]

\item
An \bit{$\cO$-monoidal $\infty$-category} is a reduced Segal functor
\[
\cO^\otimes
\xra{\cC^\otimes}
\Cat
~.
\]
We also write
\[
( \cC^\otimes \da \cO^\otimes)
\in
\coCart_{\cO^\otimes}
\]
for the cocartesian fibration that such a functor classifies, and we write
\[
\cC
:=
\cC^\otimes(\ul{1}_+)
\in
\Cat
\]
for its underlying $\infty$-category.  These assemble into the full subcategory
\[
\Alg_\cO(\Cat)
\subseteq
\Fun( \cO^\otimes , \Cat)
~,
\]
whose morphisms we refer to as \bit{$\cO$-monoidal functors}.

\item\label{item.rlaxly.O.mon}

We define the $\infty$-category whose objects are $\cO$-monoidal $\infty$-categories and whose morphisms are \bit{right-laxly $\cO$-monoidal functors} to be the indicated image in the diagram
\[ \begin{tikzcd}[column sep=1.5cm, row sep=1.5cm]
\Alg_\cO(\Cat)
\arrow[dashed, hook, two heads]{r}
\arrow[hook]{d}[swap]{\ff}
&
\Alg_\cO^\rlax(\Cat)
\arrow[hook]{d}{\ff}
\\
\coCart_{\cO^\otimes}
\arrow[hook, two heads]{r}
&
\Cat^{\cls}_{\cocart/\cO^\otimes}
\arrow[hook, two heads]{r}
\arrow{d}
&
\Cat_{\cocart/\cO^\otimes}
\arrow{d}
\\
&
\coCart_{\cO^\otimes_\cls}
\arrow[hook, two heads]{r}
&
\Cat_{\cocart/\cO^\otimes_\cls}
\end{tikzcd} \]
whose lower right square is a pullback.

\item\label{item.llaxly.O.mon}

We define the $\infty$-category whose objects are $\cO$-monoidal $\infty$-categories and whose morphisms are \bit{left-laxly $\cO$-monoidal functors} to be the indicated image in the diagram
\[ \begin{tikzcd}[column sep=1.5cm, row sep=1.5cm]
\Alg_\cO(\Cat)
\arrow[dashed, hook, two heads]{r}
\arrow[hook]{d}[swap]{\ff}
&
\Alg_\cO^\llax(\Cat)
\arrow[hook]{d}{\ff}
\\
\Cart_{(\cO^\otimes)^\op}
\arrow[hook, two heads]{r}
&
\Cat^{\cls}_{\cart/(\cO^\otimes)^\op}
\arrow[hook, two heads]{r}
\arrow{d}
&
\Cat_{\cart/(\cO^\otimes)^\op}
\arrow{d}
\\
&
\Cart_{(\cO^\otimes_\cls)^\op}
\arrow[hook, two heads]{r}
&
\Cat_{\cart/(\cO^\otimes_\cls)^\op}
\end{tikzcd} \]
whose lower right square is a pullback.

\end{enumerate}
\end{definition}

\begin{notation}
For each $n \geq 0$, we write
\[ \begin{tikzcd}[row sep=0cm]
\cO(n)
\times
\cC^{\times n}
\arrow{r}
&
\cC
\\
\rotatebox{90}{$\in$}
&
\rotatebox{90}{$\in$}
\\
\left( \mu , ( X_i)_{i \in \ul{n}} \right)
\arrow[maps to]{r}
&
\underset{\mu}{\dotimes} (X_i)_{i \in \ul{n}}
\end{tikzcd} \]
for the value of an $n$-ary operation $\mu \in \cO(n)$ on an $n$-tuple $(X_i)_{i \in \ul{n}} \in \cC^{\times n}$ of objects of $\cC$.
\end{notation}

\begin{remark}
\label{rmk.unpack.laxly.O.monoidal.functors}
For each $n \geq 0$, each $\mu \in \cO(n)$, and each $(X_i)_{i \in \ul{n}} \in \cC^{\times n}$, a right-laxly $\cO$-monoidal functor $\cC \xra{F} \cD$ determines a natural comparison morphism
\[
\bigotimes^\cD_\mu (F(X_i) )_{i \in \ul{n}}
\longra
F \left( \bigotimes^\cC_\mu (X_i)_{i \in \ul{n}} \right)
\]
in $\cD$.\footnote{In the case that $n=0$, by assumption the space $\cO(0)$ is contractible, and the comparison morphism determined by its unique point is a morphism
\[
\uno_\cD
\longra
F(\uno_\cC)
~.
\]
}
In fact, directly from the definitions, a right-laxly $\cO$-monoidal functor $\cC \xra{F} \cD$ determines a functor $\Alg_\cO(\cC) \xra{F} \Alg_\cO(\cD)$ on $\infty$-categories of $\cO$-algebras.  Dually, a left-laxly $\cO$-monoidal functor determines comparison morphisms in the opposite direction, and determines a functor on $\infty$-categories of $\cO$-coalgebras.
\end{remark}

\begin{observation}
It follows from \Cref{lemma.ptwise.radjt.has.ptwise.ladjt} that given an adjunction $F \adj G$ between the underlying $\infty$-categories of $\cO$-monoidal $\infty$-categories, the following two types of data are equivalent:
\begin{itemize}
\item the additional structure on the left adjoint $F$ of a left-laxly $\cO$-monoidal functor;
\item the additional structure on the right adjoint $G$ of a right-laxly $\cO$-monoidal functor.\footnote{Indeed, this fact motivates our choice of handedness in parts \Cref{item.rlaxly.O.mon} and \Cref{item.llaxly.O.mon} of \Cref{defn.O.mon.and.laxly}: we take concordance with the handedness of the adjoint as more fundamental than concordance with the handedness of the fibrations.}
\end{itemize}
It follows in particular e.g.\! that the right adjoint of an $\cO$-monoidal functor is canonically right-laxly $\cO$-monoidal.  We will use these facts without further comment.
\end{observation}

\subsection{Ideals in presentably $\cO$-monoidal $\infty$-categories}
\label{subsection.ideals}

In this subsection, we study the appropriate notion of an ideal subcategory of a presentably $\cO$-monoidal stable $\infty$-category. We also show as \Cref{prop.closed.ideal.subcats.are.central.idempotent} that these are equivalent data to certain idempotent objects.

\begin{local}
\label{notn.presentably.O.mon.R}
For the remainder of this section, we fix a presentably $\cO$-monoidal stable $\infty$-category $\cR$: that is, $\cR$ is a presentable stable $\infty$-category equipped with the structure of an $\cO$-monoidal $\infty$-category such that for all $n \geq 2$ and all $\mu \in \cO(n)$ the functor $\cR^{\times n} \xra{\otimes_\mu} \cR$ commutes with colimits separately in each variable.
\end{local}

\begin{notation}
\label{notation.omit.mu.from.tensor.when.with.unit}
We write $\uno_\cR \in \cR$ for the $\cO$-monoidal unit object of $\cR$.
\end{notation}

\begin{remark}
The object $\uno_\cR \in \cR$ is the unit with respect to all possible monoidal products in $\cR$: for any $n \geq 1$, for any $\mu \in \cO(n)$, and for any $X \in \cR$, there is a canonical equivalence
\[
\bigotimes_\mu (X , \uno_\cR , \ldots, \uno_\cR )
\xlongra{\sim}
X
\]
(where there are $(n-1)$ copies of $\uno_\cR$), and similarly where $X$ is put in a different slot from the first.
\end{remark}

\begin{notation}
\label{notation.omit.mu.from.tensor.when.unambiguous}
We simply write $\otimes := \otimes_\mu$ in any situation where this notation is canonically unambiguous, such as throughout \Cref{obs.omnibus.closed.ideal}. (This unambiguity will then be an implicit assertion.)
\end{notation}


\begin{definition}
\label{defn.ideal.subcat}
A full presentable stable subcategory $\cI \subseteq \cR$ is called an \bit{ideal} if it is contagious under the $\cO$-monoidal structure, i.e.\! for all $n \geq 2$ and all $\mu \in \cO(n)$ there exists a factorization
\[ \begin{tikzcd}
\cI \times \cR^{\times (n-1)}
\arrow{rr}{\otimes_\mu}
\arrow[dashed]{rd}
&[-0.5cm]
&
\cR
\\
&
\cI
\arrow[hook]{ru}[sloped, swap]{\ff}
\end{tikzcd}~. \]
\end{definition}

\begin{notation}
\label{notn.for.ideal.generated.by.stuff}
Given a set $\{ K_s \in \cR \}_{s \in S}$ of objects, we write $\brax{K_s}^\otimes_{s \in S}$ for the ideal that they generate. Likewise, given a subcategory $\cD \subseteq \cR$, we write $\brax{ \cD }^\otimes \subseteq \cR$ for the ideal that it generates.
\end{notation}

\begin{observation}
\label{obs.ideal.and.closed.subcat}
Suppose that $\cI \subseteq \cR$ is an ideal that is also a closed subcategory.  Then, $\cI$ inherits an $\cO$-monoidal structure with unit object $\uno_\cI := y(\uno_\cR) \in \cI$, such that in the adjunction
\begin{equation}
\label{iL.and.y.adjn.from.I.to.R}
\begin{tikzcd}[column sep=1.5cm]
\cI
\arrow[hook, transform canvas={yshift=0.9ex}]{r}{i_L}
\arrow[leftarrow, transform canvas={yshift=-0.9ex}]{r}[yshift=-0.2ex]{\bot}[swap]{y}
&
\cR
\end{tikzcd}
\end{equation}
the left adjoint $i_L$ is left-laxly $\cO$-monoidal and nonunitally $\cO$-monoidal, i.e.\! it preserves tensor products up to natural equivalence but the unit only up to a morphism 
\begin{equation}
\label{counit.morphism.for.unit.object.from.ideal.closed.subcat}
i_L(\uno_\cI)
:=
i_L(y(\uno_\cR))
\xra{\vareps_{\uno_\cR}}
\uno_\cR
~.
\end{equation}
It follows that the right adjoint $y$ is right-laxly $\cO$-monoidal.
\end{observation}

\begin{definition}
\label{defn.closed.ideal.subcat}
An ideal $\cI \subseteq \cR$ which is also a closed subcategory is called a \bit{closed ideal} if the right adjoint $y$ in the adjunction \Cref{iL.and.y.adjn.from.I.to.R} is $\cO$-monoidal.  We write
\[ \Idl_\cR \subseteq \Cls_\cR \]
for the full subposet consisting of the closed ideals.
\end{definition}

\begin{observation}
\label{obs.Idl.in.Cls.stable.under.colims}
Because the $\cO$-monoidal structure on $\cR$ commutes with colimits separately in each variable, the full subposet $\Idl_\cR \subseteq \Cls_\cR$ is stable under colimits.
\end{observation}

\begin{observation}
\label{obs.omnibus.closed.ideal}
Suppose that $\cI \subseteq \cR$ is a closed ideal, and consider the recollement
\[ \begin{tikzcd}[column sep=1.5cm]
\cI
\arrow[hook, bend left=45]{r}[description]{i_L}
\arrow[leftarrow]{r}[transform canvas={yshift=0.1cm}]{\bot}[swap,transform canvas={yshift=-0.1cm}]{\bot}[description]{\yo}
\arrow[bend right=45, hook]{r}[description]{i_R}
&
\cR
\arrow[bend left=45]{r}[description]{p_L}
\arrow[hookleftarrow]{r}[transform canvas={yshift=0.1cm}]{\bot}[swap,transform canvas={yshift=-0.1cm}]{\bot}[description]{\nu}
\arrow[bend right=45]{r}[description]{p_R}
&
\cR/\cI
\end{tikzcd}~. \]
It is straightforward to verify the following facts, which we will use without further comment.
\begin{enumerate}

\item\label{obs.item.closed.ideal.colocalization}
The object $i_L(\uno_\cI) \in \cR$ is an idempotent $\cO$-coalgebra object with counit morphism
\Cref{counit.morphism.for.unit.object.from.ideal.closed.subcat}.
Moreover, tensoring with this counit morphism implements the colocalization $i_L \adj y$: for any $X \in \cR$, the diagram
\[ \begin{tikzcd}[row sep=0cm, column sep=1.5cm]
i_L(\uno_\cI) \otimes X
\arrow{r}{\vareps_{\uno_\cR} \otimes \id_X}
&
\uno_\cR \otimes X
\\
\rotatebox{90}{$\simeq$}
&
\rotatebox{90}{$\simeq$}
\\
i_L(\uno_\cI \otimes y(X))
&
X
\\
\rotatebox{90}{$\simeq$}
\\
i_L y(X)
\arrow{ruu}[sloped, swap]{\vareps_X}
\end{tikzcd} \]
canonically commutes.\footnote{That is, for every $\mu \in \cO(2)$, the functor $i_L(\uno_\cI) \otimes_\mu (-)$ is canonically equivalent to the composite $i_L y$ (recall \Cref{notation.omit.mu.from.tensor.when.unambiguous}).} 

\item\label{obs.item.quotient.by.closed.ideal.localization}
There is a canonical $\cO$-monoidal structure on $\cR/\cI$, such that
\begin{enumeratesub}
\item the unit object is $\uno_{\cR/\cI} := p_L(\uno_\cR) \in \cR/\cI$,
\item the functor $p_L$ is $\cO$-monoidal, and
\item the functor $\nu$ is right-laxly $\cO$-monoidal and nonunitally $\cO$-monoidal, i.e.\! it preserves tensor products up to natural equivalence but the unit only up to a morphism
\begin{equation}
\label{unit.morphism.for.unit.object.in.quotient.by.ideal.closed.subcat}
\uno_\cR
\xra{\eta_{\uno_\cR}}
\nu(p_L(\uno_\cR))
=:
\nu(\uno_{\cR/\cI})
~.
\end{equation}
\end{enumeratesub}
Hence, the object $\nu(\uno_{\cR/\cI}) \in \cR$ is an idempotent $\cO$-algebra object with unit morphism \Cref{unit.morphism.for.unit.object.in.quotient.by.ideal.closed.subcat}.
Moreover, tensoring with this unit morphism implements the localization $p_L \adj \nu$: for any $X \in \cR$, the diagram
\[ \begin{tikzcd}[row sep=0cm, column sep=1.5cm]
\uno_\cR \otimes X
\arrow{r}{\eta_{\uno_\cR} \otimes \id_X}
&
\nu(\uno_{\cR/\cI}) \otimes X
\\
\rotatebox{90}{$\simeq$}
&
\rotatebox{90}{$\simeq$}
\\
X
\arrow{rdd}[sloped, swap]{\eta_X}
&
\nu(\uno_{\cR/\cI} \otimes p_L(X))
\\
&
\rotatebox{90}{$\simeq$}
\\
&
\nu(p_L(X))
\end{tikzcd} \]
canonically commutes.\footnote{In particular, $\cR/\cI \xhookra{\nu} \cR$ is also the inclusion of an ideal (which is not generally a closed ideal).}  

\end{enumerate}

\end{observation}

\needspace{2\baselineskip}
\begin{remark}
\label{rmk.idempotent.algebras.are.simple}
\begin{enumerate}
\item[]

\item\label{item.idempotent.coalgebras.are.simple}

An idempotent $\cO$-coalgebra object in $\cR$ is equivalently an object
\begin{equation}
\label{idempotent.coalgebra}
(C \xlongra{\varepsilon} \uno_\cR) \in \cR_{/\uno_\cR}
\end{equation}
such that for all $n \geq 2$ and all $\mu \in \cO(n)$ the morphism
\begin{equation}
\label{structure.morphism.for.idempotent.coalgebra}
\bigotimes_\mu ( C , \ldots, C)
\xra{\bigotimes_\mu ( \varepsilon, \id_C , \ldots, \id_C)}
\bigotimes_\mu ( \uno_\cR , C , \ldots, C )
\end{equation}
is an equivalence.

\item\label{item.idempotent.algebras.are.simple}

 An idempotent $\cO$-algebra object in $\cR$ is equivalently an object
\begin{equation}
\label{idempotent.algebra}
(\uno_\cR \xlongra{\eta} A) \in \cR_{\uno_\cR/}
\end{equation}
such that for all $n \geq 2$ and all $\mu \in \cO(n)$ the morphism
\begin{equation}
\label{structure.morphism.for.idempotent.algebra}
\bigotimes_\mu ( \uno_\cR , A , \ldots, A)
\xra{\bigotimes_\mu ( \eta , \id_A , \ldots, \id_A)}
\bigotimes_\mu ( A , \ldots, A )
\end{equation}
is an equivalence.
\end{enumerate}
\end{remark}

\needspace{2\baselineskip}
\begin{definition}
\label{defn.idempotents.and.centrality}
\begin{enumerate}
\item[]

\item An \bit{augmented idempotent} in $\cR$ is an object \Cref{idempotent.coalgebra} such that for all $n \geq 2$ and all $\mu \in \cO(n)$ the morphism \Cref{structure.morphism.for.idempotent.coalgebra} is an equivalence.\footnote{So, an augmented idempotent is equivalently an idempotent $\cO$-coalgebra by \Cref{rmk.idempotent.algebras.are.simple}\Cref{item.idempotent.coalgebras.are.simple}.} We say that it is \bit{central} if for all $n \geq 3$, all $\mu \in \cO(n)$, and all $X_1,\ldots,X_{n-2} \in \cR$, the morphism
\[
\bigotimes_\mu ( C , C , X_1, \ldots, X_{n-2})
\xra{\bigotimes_\mu ( \varepsilon, \id_C , \id_{X_1} , \ldots, \id_{X_{n-2}})}
\bigotimes_\mu ( \uno_\cR , C , X_1 \ldots , X_{n-2} )
\]
is an equivalence. We write
\[
\ZAug_\cR \subseteq \cR_{/\uno_\cR}
\]
for the full subcategory on the central augmented idempotents.

\item A \bit{coaugmented idempotent} in $\cR$ is an object \Cref{idempotent.algebra} such that for all $n \geq 2$ and all $\mu \in \cO(n)$ the morphism \Cref{structure.morphism.for.idempotent.algebra} is an equivalence.\footnote{So, a coaugmented idempotent is equivalently an idempotent $\cO$-algebra by \Cref{rmk.idempotent.algebras.are.simple}\Cref{item.idempotent.algebras.are.simple}.} We say that it is \bit{central} if for all $n \geq 3$, all $\mu \in \cO(n)$, and all $X_1,\ldots,X_{n-2} \in \cR$, the morphism
\[
\bigotimes_\mu ( \uno_\cR , A , X_1, \ldots, X_{n-2})
\xra{\bigotimes_\mu ( \eta , \id_A , \id_{X_1} , \ldots, \id_{X_{n-2}})}
\bigotimes_\mu ( A , A , X_1 \ldots , X_{n-2} )
\]
is an equivalence. We write
\[
\ZcoAug_\cR \subseteq \cR_{\uno_\cR/}
\]
for the full subcategory on the central coaugmented idempotents.

\end{enumerate}
\end{definition}

\begin{observation}
\label{obs.centrality.vacuous.or.easy}
In the case that $\cO$ is quadratic, it suffices to verify centrality for ternary operations. For instance, if $\cO = \EE_1$, an augmented idempotent $C \in \ZAug_\cR$ is central if and only if for every $X \in \cR$ the morphisms
\[
C \otimes X
\simeq
C \otimes X \otimes \uno_\cR
\xla{\id_C \otimes \id_X \otimes \varepsilon}
C \otimes X \otimes C
\xra{\varepsilon \otimes \id_X \otimes \id_C}
\uno_\cR \otimes X \otimes C
\simeq X \otimes C
\]
are equivalences, while a coaugmented idempotent $A \in \ZcoAug_\cR$ is central if and only if for every $X \in \cR$ the morphisms
\[
A \otimes X
\simeq
A \otimes X \otimes \uno_\cR
\xra{\id_A \otimes \id_X \otimes \eta}
A \otimes X \otimes A
\xla{\eta \otimes \id_X \otimes \id_A}
\uno_\cR \otimes X \otimes A
\simeq
X \otimes A
\]
are equivalences. If additionally $\cO(2)$ is connected (e.g.\! if $\cO = \EE_n$ for any $2 \leq n \leq \infty$), then the condition of centrality is vacuous: every co/augmented idempotent is automatically central.
\end{observation}

\needspace{2\baselineskip}
\begin{prop}
\label{prop.closed.ideal.subcats.are.central.idempotent}
\begin{enumerate}
\item[]

\item\label{idempotents.are.posets}
The full subcategories
\[
\ZAug_\cR
\subseteq
\cR_{/\uno_\cR}
\qquad
\text{and}
\qquad
\ZcoAug_\cR
\subseteq
\cR_{\uno_\cR/}
\]
are posets.

\item\label{item.identify.closed.ideals.as.idempotents}
There is a canonical commutative diagram
\begin{equation}
\label{comm.triangle.describing.closed.ideals}
\begin{tikzcd}[row sep=2cm]
&
\Idl_\cR
\arrow[transform canvas={xshift=0.5ex, yshift=0.5ex}]{rd}[sloped]{\cI \longmapsto \nu ( \uno_{\cR/\cI})}
\arrow[leftarrow, transform canvas={xshift=-0.5ex, yshift=-0.5ex}]{rd}[yshift=-0.2ex, sloped]{\sim}[swap, sloped]{\brax{\fib(\eta)}^\otimes}
\\
\ZAug_\cR
\arrow[transform canvas={yshift=0.9ex}]{rr}{\cofib(\varepsilon)}
\arrow[leftarrow, transform canvas={yshift=-0.9ex}]{rr}[yshift=-0.2ex]{\sim}[swap]{\fib(\eta)}
\arrow[transform canvas={xshift=-0.5ex, yshift=0.5ex}]{ru}[sloped]{C \longmapsto \brax{C}^\otimes}
\arrow[leftarrow, transform canvas={xshift=0.5ex, yshift=-0.5ex}]{ru}[yshift=-0.2ex, sloped]{\sim}[swap, sloped]{i_L(\uno_\cI) \longmapsfrom \cI}
&
&
\ZcoAug_\cR
\end{tikzcd}
\end{equation}
of equivalences.

\item\label{item.identify.closed.ideal.generated.by.augmented.idempotent}
Given a central augmented idempotent $C \in \ZAug_\cR$, for any $\tau \in \cO(2)$ we have an identification
\[
\brax{C}^\otimes
=
\cI_{C,\tau}
:=
\left\{ X \in \cR : \textup{the morphism } C \otimes_\tau X \xra{\varepsilon \otimes_\tau \id_X} \uno_\cR \otimes_\tau X \simeq X \textup{ is an equivalence} \right\}
~,
\]
and we may identify the right adjoint to its inclusion as
\[
\begin{tikzcd}[column sep=1.5cm]
\cI_{C,\tau}
\arrow[hook, transform canvas={yshift=0.9ex}]{r}
\arrow[dashed, leftarrow, transform canvas={yshift=-0.9ex}]{r}[yshift=-0.2ex]{\bot}[swap]{C \otimes_\tau (-)}
&
\cR
\end{tikzcd}
\]
with counit $C \otimes_\tau (-) \xra{\varepsilon \otimes_\tau \id } \uno_\cR \otimes_\tau (-) \simeq \id_\cR$.

\item\label{item.identify.closed.ideal.from.coaugmented.idempotent}
Given a central coaugmented idempotent $A \in \ZcoAug_\cR$, for any $\tau \in \cO(2)$ we have an identification
\[
\cR / \brax{ \fib(\eta) }^\otimes
=
\cR / \cI_{\fib(\eta) , \tau}
=
\left\{ X \in \cR : \textup{the morphism } X \simeq \uno_\cR \otimes_\tau X \xra{\eta \otimes_\tau \id_X} A \otimes_\tau X \textup{ is an equivalence} \right\}
~,
\]
and we may identify the left adjoint to its inclusion as
\[
\begin{tikzcd}[column sep=1.5cm]
\cR
\arrow[dashed, transform canvas={yshift=0.9ex}]{r}{A \otimes_\tau (-)}
\arrow[hookleftarrow, transform canvas={yshift=-0.9ex}]{r}[yshift=-0.2ex]{\bot}
&
\cR / \cI_{\fib(\eta) , \tau}
\end{tikzcd}
\]
with unit $\id_\cR \simeq \uno_\cR \otimes_\tau (-) \xra{\eta \otimes_\tau \id} A \otimes_\tau (-)$.

\end{enumerate}
\end{prop}

\begin{proof}
We fix arbitrary $C \in \ZAug_\cR$ and $\tau \in \cO(2)$, to which we will refer throughout the proof.

We begin by proving part \Cref{item.identify.closed.ideal.generated.by.augmented.idempotent}, and then use it implicitly through the remainder of the proof. We first verify that $\cI_{C,\tau} \subseteq \cR$ is an ideal. It is clearly a full presentable stable subcategory. Now, for any $n \geq 2$, any $\mu \in \cO(n)$, any $X \in \cI$, and any $Y_1,\ldots,Y_{n-1} \in \cR$, we may factor the morphism
\[
 C \otimes_\tau \bigotimes_\mu ( X , Y_1 , \ldots, Y_{n-1} )
\longra
\uno_\cR \otimes_\tau \bigotimes_\mu ( X , Y_1 , \ldots, Y_{n-1} )
\simeq
\bigotimes_\mu ( X , Y_1 , \ldots, Y_{n-1} )
\]
as the sequence of equivalences
\begin{align}
\label{with.C.involved.pop.out.C.on.X}
C \otimes_\tau \bigotimes_\mu ( X , Y_1 , \ldots, Y_{n-1} )
& \xlongla{\sim}
C \otimes_\tau \bigotimes_\mu ( C \otimes_\tau X , Y_1 , \ldots, Y_{n-1} )
\\
\label{collapse.outer.copy.of.C}
& \xlongra{\sim}
\uno_\cR \otimes_\tau \bigotimes_\mu ( C \otimes_\tau X , Y_1 , \ldots, Y_{n-1} )
\\
\nonumber
& \xlongra{\sim}
\bigotimes_\mu ( C \otimes_\tau X , Y_1 , \ldots, Y_{n-1} )
\\
\label{collapse.C.from.X}
& \xlongra{\sim}
\bigotimes_\mu ( X , Y_1 , \ldots, Y_{n-1} )
\end{align}
in which equivalences \Cref{with.C.involved.pop.out.C.on.X} \and \Cref{collapse.C.from.X} use that $X \in \cI_{C,\tau}$ and equivalence \Cref{collapse.outer.copy.of.C} uses the centrality of $C$. So indeed, the subcategory $\cI_{C,\tau} \subseteq \cR$ is an ideal. Now, we have $C \in \cI_{C,\tau}$ because $C$ is an augmented idempotent, so we obtain the containment $\brax{C}^\otimes \subseteq \cI_{C,\tau}$. On the other hand, clearly $(C \otimes_\tau X) \in \brax{C}^\otimes$ for any $X \in \cR$, which implies that $\brax{C}^\otimes \supseteq \cI_{C,\tau}$. This proves the asserted equality $\brax{C}^\otimes = \cI_{C,\tau}$. To verify that the right adjoint to its inclusion is as asserted, we observe that for any $X \in \cI_{C,\tau}$ and any $Y \in \cR$, we have $(C \otimes_\tau Y) \in \brax{C}^\otimes = \cI_{C,\tau}$ and moreover we have the commutative diagram
\[ \begin{tikzcd}[row sep=1.5cm]
\hom_{\cI_{C,\tau}}(X , C \otimes_\tau Y)
:=
&[-1.5cm]
\hom_\cR(X, C \otimes_\tau Y)
\arrow{r}
\arrow{d}[sloped, anchor=north]{\sim}
&
\hom_\cR(X,Y)
\arrow{d}[sloped, anchor=south]{\sim}
\arrow{ld}[sloped]{C \otimes_\tau (-)}
\\
&
\hom_\cR(C \otimes_\tau X , C \otimes_\tau Y)
\arrow{r}
&
\hom_\cR(C \otimes_\tau X , Y)
\end{tikzcd}~, \]
which implies that its upper morphism is an equivalence. This completes the proof of part \Cref{item.identify.closed.ideal.generated.by.augmented.idempotent}.

We now verify that the ideal $\cI_{C,\tau} \subseteq \cR$ is in fact a closed ideal. First of all, it is a closed subcategory because the right adjoint $\cR \xra{C \otimes_\tau (-)} \cI_{C,\tau}$ preserves colimits. So, it remains to verify that this right adjoint is $\cO$-monoidal. Clearly $\uno_{\cI_{C,\tau}} \simeq C$, and hence this right adjoint preserves unit objects. We now observe that for any $n \geq 2$, any $\mu \in \cO(n)$, and any $Y_1,\ldots,Y_n \in \cR$, we may factor the canonical morphism
\[
\bigotimes_\mu ( C \otimes_\tau Y_i )_{i \in \ul{n}}
\longra
C \otimes_\tau \left( \bigotimes_\mu (Y_i)_{i \in \ul{n}} \right)
\]
as the sequence of equivalences
\[
\bigotimes_\mu ( C \otimes_\tau Y_i )_{i \in \ul{n}}
\xlongla{\sim}
C \otimes_\tau \bigotimes_\mu ( C \otimes_\tau Y_i )_{i \in \ul{n}}
\xlongra{\sim}
C \otimes_\tau \bigotimes_\mu ( Y_i )_{i \in \ul{n}}
\]
using the centrality of $C$. So indeed, $\cI_{C,\tau} \subseteq \cR$ is a closed ideal.

We now verify that the association $C \mapsto \cI_{C,\tau}$ defines a functor
\[
\ZAug_\cR
\xra{\cI_{(-),\tau}}
\Idl_\cR
~:
\]
given a morphism
\begin{equation}
\label{morphism.in.ZAug}
\begin{tikzcd}
C
\arrow{rr}{\alpha}
\arrow{rd}[sloped, swap]{\varepsilon_C}
&
&
C'
\arrow{ld}[sloped, swap]{\varepsilon_{C'}}
\\
&
\uno_\cR
\end{tikzcd}
\end{equation}
in $\ZAug_\cR$, we must verify the inclusion $\cI_{C,\tau} \subseteq \cI_{C',\tau}$. Using the equality $\cI_{C,\tau} = \brax{C}^\otimes$, it suffices to verify that $C \in \cI_{C',\tau}$. For this, we apply the functor $C \otimes_\tau (-)$ to the commutative triangle \Cref{morphism.in.ZAug}, which yields a retraction diagram
\[ \begin{tikzcd}[row sep=1.5cm]
C \otimes_\tau C
\arrow{rr}{\id_C \otimes_\tau \alpha}
\arrow{rd}[sloped, swap]{\id_C \otimes_\tau \varepsilon_C}[sloped]{\sim}
&
&
C \otimes_\tau C'
\arrow{ld}[sloped, swap]{\id_C \otimes_\tau \varepsilon_{C'}}
\\
&
C \otimes_\tau \uno_\cR
\end{tikzcd}~, \]
which proves the claim since $(C \otimes_\tau C') \in \cI_{C',\tau}$ and $\cI_{C',\tau} \subseteq \cR$ is closed under retracts.

We now prove that the subcategory $\ZAug_\cR \subseteq \cR_{/\uno_\cR}$ is a poset, i.e.\! the first half of part \Cref{idempotents.are.posets}. Suppose there exists a morphism $C \ra C'$ in $\ZAug_\cR$. As we have just seen, this implies that $C \in \cI_{C',\tau} \subseteq \cR$. Hence, we find that
\[
\hom_{\ZAug_\cR}(C,C')
:=
\hom_{\cR_{/\uno_\cR}}(C,C')
\simeq
\hom_{(\cI_{C',\tau})_{/C'}}(C,C')
\simeq
\pt
~,
\]
as desired.

Now, given any closed ideal $\cI \in \Idl_\cR$, it is clear that $i_L(\uno_\cI) \in \cR_{/\uno_\cR}$ is a central augmented idempotent, and moreover that a morphism $\cI \subseteq \cI'$ in $\Idl_\cR$ determines a morphism $i_L(\uno_\cI) \ra i_L(\uno_{\cI'})$ in $\cR_{\uno_\cR}$: in other words, the association $\cI \mapsto i_L(\uno_\cI)$ defines a functor
\[
\ZAug_\cR
\xla{i_L(\uno_{(-)})}
\Idl_\cR
~.
\]
From here, we immediately obtain the mutually inverse equivalences on the left in diagram \Cref{comm.triangle.describing.closed.ideals}. It is straightforward to verify the horizontal mutually inverse equivalences in diagram \Cref{comm.triangle.describing.closed.ideals}. Part \Cref{item.identify.closed.ideals.as.idempotents} immediately follows, as do part \Cref{item.identify.closed.ideal.from.coaugmented.idempotent} and the second half of part \Cref{idempotents.are.posets}.
\end{proof}

\begin{cor}
\label{cor.ideal.gend.by.image.of.closed.ideal.is.closed}
Assume that $\cO$ is quadratic and that $\cO(2)$ is connected (e.g.\! $\cO = \EE_n$ for $2 \leq n \leq \infty$). Let $\cR \xra{F} \cR'$ be a morphism in $\Alg_\cO(\PrLSt)$, i.e.\! an $\cO$-monoidal left adjoint functor between presentably $\cO$-monoidal stable $\infty$-categories. Then for any closed ideal $\cI \in \Idl_\cR$, the ideal
\[
\cI' := \brax{F(\cI)}^\otimes \subseteq \cR'
\]
is a closed ideal of $\cR'$, and moreover
\[
i_L(\uno_{\cI'})
\simeq
F(i_L(\uno_\cI))
\in \ZAug_{\cR'}
\qquad
\text{and}
\qquad
\nu(\uno_{\cR'/\cI'})
\simeq
F(\nu(\uno_{\cR/\cI}))
\in \ZcoAug_{\cR'}
~.
\]
\end{cor}

\begin{proof}
Because $F$ is $\cO$-monoidal, it preserves co/augmented idempotents. Moreover, by \Cref{obs.centrality.vacuous.or.easy}, our assumptions on $\cO$ imply that the condition of centrality is vacuous, so that we obtain factorizations
\[
\begin{tikzcd}
\cR_{/\uno_\cR}
\arrow{r}{F}
&
\cR'_{/\uno_{\cR'}}
\\
\ZAug_\cR
\arrow[hook]{u}
\arrow[dashed]{r}[swap]{F}
&
\ZAug_{\cR'}
\arrow[hook]{u}
\end{tikzcd}
\qquad
\text{and}
\qquad
\begin{tikzcd}
\cR_{\uno_\cR/}
\arrow{r}{F}
&
\cR'_{\uno_{\cR'}/}
\\
\ZcoAug_\cR
\arrow[hook]{u}
\arrow[dashed]{r}[swap]{F}
&
\ZcoAug_{\cR'}
\arrow[hook]{u}
\end{tikzcd}
~.
\]
Now, using \Cref{prop.closed.ideal.subcats.are.central.idempotent}\Cref{item.identify.closed.ideals.as.idempotents}, we find that
\[
\cR'
\supseteq
\cI'
:=
\brax{F(\cI)}^\otimes
=
\brax{ F( \brax{i_L(\uno_\cI)}^\otimes ) }^\otimes
=
\brax{ F(i_L(\uno_\cI) ) }^\otimes
\]
is indeed a closed ideal with $i_L(\uno_{\cI'}) \simeq F(i_L(\uno_\cI))$. Using this, we compute that
\[
\nu(\uno_{\cR'/\cI'})
\simeq
\cofib ( i_L(\uno_{\cI'}) \xlongra{\varepsilon} \uno_{\cR'} )
\simeq
\cofib ( F ( i_L(\uno_\cI) \xlongra{\varepsilon} \uno_\cR ) )
\simeq
F ( \cofib ( i_L(\uno_\cI) \xlongra{\varepsilon} \uno_\cR ) )
\simeq
F ( \nu(\uno_{\cR/\cI} ) )
~,
\]
as desired.
\end{proof}

\begin{remark}
Let us assume for simplicity that $\cO = \Comm$, and let us simply write $\ulhom_\cR(-,-)$ for the internal hom bifunctor of $\cR$. Then, in light of \Cref{obs.omnibus.closed.ideal}\Cref{obs.item.closed.ideal.colocalization} we may identify the composite adjoints
\[ \begin{tikzcd}[column sep=1.5cm]
i_L(\uno_\cI) \otimes (-)
:
\cR
\arrow[transform canvas={yshift=0.9ex}]{r}{y}
\arrow[hookleftarrow, transform canvas={yshift=-0.9ex}]{r}[yshift=-0.2ex]{\bot}[swap]{i_R}
&
\cI
\arrow[hook, transform canvas={yshift=0.9ex}]{r}{i_L}
\arrow[leftarrow, transform canvas={yshift=-0.9ex}]{r}[yshift=-0.2ex]{\bot}[swap]{y}
&
\cR
:
\ulhom_\cR(i_L(\uno_\cI),-)
\end{tikzcd}~. \]
If $i_L(\uno_\cI) \in \cR$ is dualizable, then the composite right adjoint admits a further identification
\[
i_R y
\simeq
\ulhom_\cR(i_L(\uno_\cI),-)
\simeq
i_L(\uno_\cI)^\vee \otimes (-)
~, \]
in which case it itself admits a further right adjoint.  Because $y$ is a left adjoint and $i_R$ is fully faithful, this is the case if and only if $i_R$ itself admits a further right adjoint.
Likewise, in light of \Cref{obs.omnibus.closed.ideal}\Cref{obs.item.quotient.by.closed.ideal.localization} we may identify the composite adjoints
\[ \begin{tikzcd}[column sep=1.5cm]
\nu(\uno_{\cR/\cI}) \otimes (-)
:
\cR
\arrow[transform canvas={yshift=0.9ex}]{r}{p_L}
\arrow[hookleftarrow, transform canvas={yshift=-0.9ex}]{r}[yshift=-0.2ex]{\bot}[swap]{\nu}
&
\cR/\cI
\arrow[hook, transform canvas={yshift=0.9ex}]{r}{\nu}
\arrow[leftarrow, transform canvas={yshift=-0.9ex}]{r}[yshift=-0.2ex]{\bot}[swap]{p_R}
&
\cR
:
\ulhom_\cR(\nu(\uno_{\cR/\cI}),-)
\end{tikzcd}~. \]
Now, the dualizability of $\nu(\uno_{\cR/\cI})$ implies the further identification
\[
\nu p_R
\simeq
\ulhom_\cR(\nu(\uno_{\cR/\cI}),-)
\simeq
\nu(\uno_{\cR/\cI})^\vee \otimes (-)
~, \]
which implies that this composite right adjoint itself admits a further right adjoint.  Because $\nu$ is a fully faithful left adjoint, this is the case if and only if $p_R$ admits a further right adjoint.\footnote{It is not hard to see that $i_R$ admits a further right adjoint if and only if $p_R$ does.}  See e.g.\! \cite{BDS-GNdual} for more on these considerations.
\end{remark}

\subsection{$\cO$-monoidal stratifications}
\label{subsection.O.monoidal.stratns}

In this subsection, we define $\cO$-monoidal stratifications and study their basic properties. We also discuss the chromatic stratification of $\Spectra$ (\Cref{ex.chromatic.stratn.of.spectra}).

\begin{local}
For the remainder of this section, we fix a poset $\pos$.
\end{local}

\begin{definition}
\label{definition.O.monoidal.prestratn.and.stratn}
A prestratification of $\cR$ over $\pos$ is an \bit{$\cO$-monoidal prestratification} if it admits a factorization
\[ \begin{tikzcd}
\pos
\arrow{rr}
\arrow[dashed]{rd}
&
&
\Cls_\cR
\\
&
\Idl_\cR
\arrow[hook]{ru}[sloped, swap]{\ff}
\end{tikzcd}~. \]
An $\cO$-monoidal prestratification is an \bit{$\cO$-monoidal stratification} if its underlying prestratification is a stratification.
\end{definition}

\begin{observation}
\label{obs.down.closeds.go.to.closed.ideals.too}
Suppose that
\[
\pos
\xlongra{\cI_\bullet}
\Idl_\cR
\]
is an $\cO$-monoidal prestratification.  By \Cref{obs.Idl.in.Cls.stable.under.colims}, for any $\sD \in \Down_\pos$ we have
\[
\cI_\sD
:=
\brax{\cI_p}_{p \in \sD}
\in
\Idl_\cR
\subseteq
\Cls_\cR
~.
\]
\end{observation}

\begin{notation}
\label{notation.unit.object.for.I.D}
In the setting of \Cref{obs.down.closeds.go.to.closed.ideals.too}, we write
\[ \uno_{\cI_\sD} := y(\uno_\cR) \in \cI_\sD \]
for the $\cO$-monoidal unit object of $\cI_\sD$.
\end{notation}

\begin{remark}
\label{rmk.check.O.mon.stratn.on.yonedaed.units}
Suppose that
\[
\pos
\xlongra{\cI_\bullet}
\Idl_\cR
\]
is an $\cO$-monoidal stratification.  Then, for any $p \leq q$ in $\pos$ we have an equivalence
\[
i_L(\uno_{\cI_p})
\otimes
i_L(\uno_{\cI_q})
\xlongra{\sim}
i_L(\uno_{\cI_p})
~.
\]
More generally, for any $\sD \ra \sE$ in $\Down_\pos$ we have an equivalence
\[
i_L(\uno_{\cI_\sD})
\otimes
i_L(\uno_{\cI_\sE})
\xlongra{\sim}
i_L(\uno_{\cI_\sD})
~.
\]
Conversely, with the evident notation, there exists an (automatically unique) extension
\[ \begin{tikzcd}
\iota_0 \pos
\arrow{r}{\cI_\bullet}
\arrow[hook]{d}
&
\Idl_\cR
\\
\pos
\arrow[dashed]{ru}
\end{tikzcd} \]
if and only if for any $p \leq q$ in $\pos$ the canonical morphism
\[
i_L(\uno_{\cI_p})
\otimes
i_L(\uno_{\cI_q})
\longra
i_L(\uno_{\cI_p})
\]
is an equivalence.
\end{remark}

\begin{observation}
\label{obs.yo.comm.for.s.m.stratns}
An $\cO$-monoidal prestratification
\[
\pos
\xlongra{\cI_\bullet}
\Idl_\cR
\]
is a(n automatically $\cO$-monoidal) stratification if and only if for any $p,q \in \pos$ the canonical morphism
\[
i_L ( \uno_{\cI_p} )
\otimes
i_L(\uno_{\cI_q})
\otimes
i_L \left( \uno_{\cI_{(^\leq p) \cap (^\leq q)}} \right)
\longra
i_L(\uno_{\cI_p})
\otimes
i_L(\uno_{\cI_q})
\]
is an equivalence.
\end{observation}

\begin{observation}
Suppose that
\[
\pos
\xlongra{\cI_\bullet}
\Idl_\cR
\]
is an $\cO$-monoidal prestratification.  For each $\posQ \in \Down_\pos$, this restricts to an $\cO$-monoidal prestratification
\[
\posQ
\xlongra{\cI_\bullet} \Idl_{\cI_\posQ}
~.
\]
Hence, for every $p \in \pos$ the geometric localization functor
\[
\Phi_p
:
\cR
\xlongra{y}
\cI_p
\xra{p_L}
\cI_p / \cI_{^< p}
=:
\cR_p
\]
is $\cO$-monoidal.  It follows from the composite adjunction
\[ \begin{tikzcd}[column sep=1.5cm]
\Phi_p
:
\cR
\arrow[transform canvas={yshift=0.9ex}]{r}{y}
\arrow[hookleftarrow, transform canvas={yshift=-0.9ex}]{r}[yshift=-0.2ex]{\bot}[swap]{i_R}
&
\cI_p
\arrow[transform canvas={yshift=0.9ex}]{r}{p_L}
\arrow[hookleftarrow, transform canvas={yshift=-0.9ex}]{r}[yshift=-0.2ex]{\bot}[swap]{\nu}
&
\cR_p
:
\rho^p
\end{tikzcd} \]
that its right adjoint $\rho^p$ is right-laxly $\cO$-monoidal.  So for every $p \leq q$, the gluing functor
\[
\Gamma^p_q
:
\cR_p
\xlonghookra{\rho^p}
\cR
\xra{\Phi_q}
\cR_q
\]
is right-laxly $\cO$-monoidal.
\end{observation}

\begin{example}[the chromatic stratification of spectra]
\label{ex.chromatic.stratn.of.spectra}
Consider the presentably symmetric monoidal stable $\infty$-category $\cR = \Spectra$ of spectra.  We introduce the following notation.
\begin{itemize}
\item We write $n \in \NN$ for an arbitrary (finite, positive) natural number.
\item We respectively write $K_p(n)$ and $E_{p,n}$ for the $n\th$ Morava K- and E-theory spectra at the prime $p$.  By convention, we also set $K_p(0) = E_{p,0} = \QQ$.
\item For any $E \in \Spectra$, we respectively write $L_E$ and $A_E := \fib(\id_\Spectra \ra L_E)$ for $E$-localization and $E$-acyclification. We simply write $L_{(p)}$ for $p$-localization, $L_{p,n}$ for $E_{p,n}$-localization, and $L_{p,\infty}$ for $(\bigoplus_{0 \leq n < \infty} E_{p,n})$-localization, and similarly for the corresponding acyclifications.\footnote{We use this notation because it is standard, but note that it mildly conflicts with that of \Cref{defn.Cth.stratum.and.geometric.localizn}.}
\end{itemize}
Then, the \bit{chromatic stratification} is a symmetric monoidal stratification of $\Spectra$ over the poset $\pos_\Spectra$ described in \Cref{figure.primes.of.spectra}:\footnote{The adelic stratification of $\Spectra$ is also defined over the poset $\pos_\Spectra$, but it is slightly different (see \Cref{ex.adelic.stratn.of.spectra}).}
\begin{figure}
\[
\pos_{\Spectra}
=
\left(
\begin{tikzcd}
&
&
(0)
\\
((2),1)
\arrow{rru}
&
((3),1)
\arrow{ru}
&
((5),1)
\arrow{u}
&
\cdots
\arrow{lu}
\\
((2),2)
\arrow{u}
&
((3),2)
\arrow{u}
&
((5),2)
\arrow{u}
&
\cdots
\\
((2),3)
\arrow{u}
&
((3),3)
\arrow{u}
&
((5),3)
\arrow{u}
&
\cdots
\\
\vdots
\arrow{u}
&
\vdots
\arrow{u}
&
\vdots
\arrow{u}
&
\cdots
\\
((2),\infty)
\arrow{u}
&
((3),\infty)
\arrow{u}
&
((5),\infty)
\arrow{u}
&
\cdots
\end{tikzcd}
\right)
\]
\caption{The poset $\pos_\Spectra$ is the union of the totally ordered sets $\{ \{ ((p),\infty) \ra \cdots \ra ((p),2) \ra ((p),1) \ra (0) \} \}_{p \textup{ prime}}$ over their common maximal element $(0)$.
\label{figure.primes.of.spectra}}
\end{figure}
namely, it is the functor
\begin{equation}
\label{adelic.stratn.of.spectra}
\begin{tikzcd}[row sep=0cm]
\pos_\Spectra
\arrow{r}{\cI_\bullet}
&
\Idl_\Spectra
\\
\rotatebox{90}{$\in$}
&
\rotatebox{90}{$\in$}
\\
\mf{p}
\arrow[maps to]{r}
&
\cI_\mf{p}
\end{tikzcd}
\end{equation}
defined by the assignments
\[
\cI_{(0)} = \Spectra
~,
\qquad
\cI_{((p),n)}
=
A_{p,n-1} L_{(p)} \Spectra
~,
\qquad
\text{and}
\qquad
\cI_{((p),\infty)} = A_{p,\infty}L_{(p)} \Spectra
~.
\]
So, the minimal strata are equivalent to dissonant $p$-local spectra, while the remaining strata and their geometric localization adjunctions may be identified as
\[
\begin{tikzcd}[column sep=2.5cm]
\Spectra
\arrow[transform canvas={yshift=0.9ex}]{r}{\Phi_{(0)} = \QQ \otimes_\SS (-)}
\arrow[hookleftarrow, transform canvas={yshift=-0.9ex}]{r}[yshift=-0.2ex]{\bot}[swap]{\rho^{(0)}}
&
\Mod_\QQ
\simeq
\Spectra_{(0)}
\end{tikzcd}
\qquad
\text{and}
\qquad
\begin{tikzcd}[column sep=2.5cm]
\Spectra
\arrow[transform canvas={yshift=0.9ex}]{r}{\Phi_{((p),n)} = L_{K_p(n)}}
\arrow[hookleftarrow, transform canvas={yshift=-0.9ex}]{r}[yshift=-0.2ex]{\bot}[swap]{\rho^{((p),n)}}
&
L_{K_p(n)} \Spectra
\simeq
\Spectra_{((p),n)}
\end{tikzcd}
\]
(simply by verifying that their kernels are respectively $A_\QQ \Spectra$ and $A_{K_p(n)} \Spectra$).  The poset $\pos_\Spectra$ is not down-finite, and indeed the chromatic stratification \Cref{adelic.stratn.of.spectra} fails to converge for essentially the same reasons that the adelic stratification of $\Mod_\ZZ$ fails to converge as illustrated in \Cref{ex.intro.arithmetic}.

Of course, the failure of the poset $\pos_\Spectra$ to be down-finite is not simply due to the infinitude of the primes, but also to its failure to be artinian.  Let us therefore study the chromatic stratification of $L_{(p)} \Spectra$, a quotient stratification (in the sense of \Cref{prop.quotient.stratn}) of the chromatic stratification \Cref{adelic.stratn.of.spectra} of $\Spectra$: namely, writing
\[
\pos_\Spectra
\supset
\pos_{L_{(p)} \Spectra}
=
\{ ((p),\infty) \longra \cdots \longra ((p),3) \longra ((p),2) \longra ((p),1) \longra (0) \}
\]
and employing the identification $\pos_{L_{(p)}\Spectra} \cong (\NN^{\lcone\rcone})^\op = \{ \infty \ra \cdots \ra 2 \ra 1 \ra 0 \}$ for notational simplicity, this is the functor
\begin{equation}
\label{quotient.stratn.of.p.local.spectra}
\begin{tikzcd}[row sep=0cm]
\pos_{L_{(p)}\Spectra}
\arrow{r}{\cJ_\bullet}
&
\Idl_{L_{(p)}\Spectra}
\\
\rotatebox{90}{$\in$}
&
\rotatebox{90}{$\in$}
\\
\mf{p}
\arrow[maps to]{r}
&
\cJ_\mf{p}
\end{tikzcd}
\end{equation}
given by
\[
\cJ_{0} = L_{(p)} \Spectra
~,
\qquad
\cJ_{n} = \cI_{((p),n)} = A_{p,n-1} L_{(p)} \Spectra
~,
\qquad
\text{and}
\qquad
\cJ_\infty = \cI_{((p),\infty)} = A_{p,\infty} L_{(p)} \Spectra
~.
\]
In order to understand the behavior of the chromatic stratification \Cref{quotient.stratn.of.p.local.spectra} of $L_{(p)}\Spectra$, we pass further to its quotient stratification over
\[
[n]^\op = \{ n \ra \cdots \ra 0 \} \cong (\NN^{\lcone\rcone})^\op \backslash (^\leq (n+1))
~:
\]
this provides a (necessarily convergent) stratification of
\[
\cJ_0 / \cJ_{n+1}
:=
L_{(p)}\Spectra/A_{p,n}L_{(p)}\Spectra
\simeq
L_{p,n} \Spectra
\]
over $[n]^\op$, whose microcosm reconstruction theorem recovers the $n$-dimensional fracture cube of \cite{ACB-chromfrac} (recall \Cref{ex.gluing.stuff.over.brax.two}). Hence, the chromatic stratification \Cref{quotient.stratn.of.p.local.spectra} fails to converge as a result of the difference between harmonic localization and chromatic completion \cite{Tobi-chromcompl}.
\end{example}

\subsection{$\cO$-algebra objects in $\LMod^\rlax_{\llax.\cB}$}
\label{subsection.O.algebra.objects.in.LModrlaxllaxB}

In this subsection, we define the $\infty$-category that contains the $\cO$-monoidal gluing diagram of an $\cO$-monoidal stratification.

\begin{local}
In this subsection, we fix an $\infty$-category $\cB$.
\end{local}

\begin{observation}
\label{obs.lim.rlax.preserves.products}
In the composite
\[
\LMod_{\llax.\cB}
\longrsurjmono
\LMod^\rlax_{\llax.\cB}
\xra{\lim^\rlax_{\llax.\cB}}
\Cat
~,
\]
all three $\infty$-categories admit finite products and both functors preserve them.
\end{observation}

\begin{notation}
We write
\[
\Alg_\cO(\LMod^\rlax_{\llax.\cB})
\subseteq
\Fun( \cO^\otimes , \LMod^\rlax_{\llax.\cB})
\]
for the full subcategory on the reduced Segal objects.
\end{notation}

\begin{observation}
\label{obs.lift.limrlax.to.Oalgebras}
In light of \Cref{obs.lim.rlax.preserves.products}, we obtain a canonical lift
\[
\begin{tikzcd}[row sep=1.5cm, column sep=1.5cm]
\Alg_\cO(\LMod^\rlax_{\llax.\cB})
\arrow[dashed]{r}{\lim^\rlax_{\llax.\cB}}
\arrow{d}[swap]{\fgt}
&
\Alg_\cO(\Cat)
\arrow{d}{\fgt}
\\
\LMod^\rlax_{\llax.\cB}
\arrow{r}[swap]{\lim^\rlax_{\llax.\cB}}
&
\Cat
\end{tikzcd}
\]
as the restriction to reduced Segal objects of the value of the functor $\Fun(\cO^\otimes,-)$ on the finite-product-preserving functor $\LMod^\rlax_{\llax.\cB} \xra{\lim^\rlax_{\llax.\cB}} \Cat$.
\end{observation}

\begin{observation}
\label{obs.O.algs.in.LMod.rlax.llax.B.two.ways}
There is a canonical equivalence
\begin{equation}
\label{equivalence.between.descriptions.of.O.algs.in.LMod.rlax.llax.B}
\iota_0 \Alg_\cO(\LMod^\rlax_{\llax.\cB})
\simeq
\hom_{2\Cat} ( \llax(\cB) , \Alg_\cO^\rlax(\Cat) )
\end{equation}
of spaces. Indeed, by \Cref{thm.switching.yoga} we have an equivalence
\[
\hom_{\Cat} ( \cO^\otimes , \LMod^\rlax_{\llax.\cB} )
:=
\hom_{2\Cat} ( \cO^\otimes , 2\Cat_{1\cart/\llax(\cB)^{1\op}} )
\simeq
\hom_{2\Cat} ( \llax(\cB) , \Cat_{\cocart/\cO^\otimes} )
\]
of spaces, through which the equivalence \Cref{equivalence.between.descriptions.of.O.algs.in.LMod.rlax.llax.B} can be obtained as an equivalence of subspaces.\footnote{On the other hand, the $\infty$-categories $\Alg_\cO(\LMod^\rlax_{\llax.\cB})$ and $\hom_{2\Cat}(\llax(\cB) , \Alg_\cO^\rlax(\Cat))$ are not equivalent; this can already be seen in the case that $\cB = \pt$.} 
\end{observation}

\subsection{The $\cO$-monoidal reconstruction theorem}
\label{subsection.O.monoidal.reconstruction.thm}

In this subsection, we prove our $\cO$-monoidal reconstruction theorem.

\begin{theorem}
\label{thm.s.m.reconstrn}
Let $\cR$ be a presentably $\cO$-monoidal stable $\infty$-category, let $\pos$ be a poset, and let
\[
\pos
\xlongra{\cI_\bullet}
\Idl_\cR
\]
be an $\cO$-monoidal stratification.
\begin{enumerate}

\item\label{item.of.thm.s.m.reconstrn.lift.gluing.diagram}
There is a canonical lift
\[ \begin{tikzcd}[column sep=1.5cm]
\pos
\arrow[dashed]{r}[description]{\llax}{\GD^\otimes(\cR)}
\arrow{rd}[description, sloped]{\llax}[swap, sloped, yshift=-0.2ex]{\GD(\cR)}
&
\Alg_\cO^\rlax(\Cat)
\arrow{d}{\fgt}
\\
&
\Cat
\end{tikzcd} \]
of the gluing diagram of the underlying stratification of $\cR$ to an $\cO$-monoidal gluing diagram.

\item\label{item.of.thm.s.m.reconstrn.enhance.gluing.diagram.functor}

There is a canonical morphism
\begin{equation}
\label{s.m.comparison.morphism.in.thm.s.m.reconstrn}
\cR
\xlongra{\gd^\otimes}
\Glue^\otimes(\cR)
:=
\lim^\rlax \left(
\begin{tikzcd}[column sep=1.5cm]
\pos
\arrow{r}[description]{\llax}{\GD^\otimes(\cR)}
&
\Alg_\cO^\rlax(\Cat)
\end{tikzcd}
\right)
\end{equation}
in $\Alg_\cO(\Cat)$ whose image under the functor $\Alg_\cO(\Cat) \xra{\fgt} \Cat$ is the morphism
\begin{equation}
\label{plain.comparison.morphism.in.thm.s.m.reconstrn}
\cR
\xlongra{\gd}
\Glue(\cR)
:=
\lim^\rlax \left(
\begin{tikzcd}[column sep=1.5cm]
\pos
\arrow{r}[description]{\llax}{\GD(\cR)}
&
\Cat
\end{tikzcd}
\right)
\end{equation}
in $\Cat$, so that the adjunction
\[
\begin{tikzcd}[column sep=1.5cm]
\cR
\arrow[transform canvas={yshift=0.9ex}]{r}{\gd}
\arrow[leftarrow, transform canvas={yshift=-0.9ex}]{r}[yshift=-0.2ex]{\bot}[swap]{\lim_{\sd(\pos)}}
&
\Glue(\cR)
\end{tikzcd}
\]
between $\infty$-categories admits a canonical enhancement to an adjunction
\[
\begin{tikzcd}[column sep=1.5cm]
\cR
\arrow[transform canvas={yshift=0.9ex}]{r}{\gd^\otimes}
\arrow[leftarrow, transform canvas={yshift=-0.9ex}]{r}[yshift=-0.2ex]{\bot}[swap]{\lim_{\sd(\pos)}^\otimes}
&
\Glue^\otimes(\cR)
\end{tikzcd}
\]
between $\cO$-monoidal $\infty$-categories whose left adjoint is $\cO$-monoidal and whose right adjoint is right-laxly $\cO$-monoidal.

\end{enumerate}
In particular, if the morphism \Cref{plain.comparison.morphism.in.thm.s.m.reconstrn} is an equivalence then the morphism \Cref{s.m.comparison.morphism.in.thm.s.m.reconstrn} is also an equivalence.
\end{theorem}

\begin{proof}
We begin with part \Cref{item.of.thm.s.m.reconstrn.lift.gluing.diagram}. By \Cref{obs.O.algs.in.LMod.rlax.llax.B.two.ways}, it is equivalent to enhance the gluing diagram
\[
\GD(\cR)
\in
\LMod_{\llax.\pos}
\subseteq
\LMod^\rlax_{\llax.\pos}
\]
to an $\cO$-monoidal gluing diagram
\[
\GD^\otimes(\cR)
\in
\Alg_\cO(\LMod^\rlax_{\llax.\pos})
\subseteq
\Fun ( \cO^\otimes , \LMod^\rlax_{\llax.\pos} )
~. \]
We use \Cref{lemma.functors.into.LMod.rlax.llax.B} to construct this as a locally cocartesian fibration over $\cO^\otimes \times \pos$, which we define to be the full subcategory
\[ \begin{tikzcd}
\GD^\otimes(\cR)
\arrow[hook]{r}{\ff}
\arrow{d}
&
\cR^\otimes \times \pos
\arrow{ld}
\\
\cO^\otimes \times \pos
\end{tikzcd} \]
on the objects
\[
\{ ((X_1,\ldots,X_n),p) \in \cR^\otimes \times \pos : X_i \in \cR_p \subseteq \cR \textup{ for all } i \}
~.\footnote{Because we are working over both $\pos$ and $\cO^\otimes \times \pos$ in this proof, we avoid the potentially ambiguous notation $\ul{\cR^\otimes}$ for $\cR^\otimes \times \pos$.}
\]
We first observe that this is indeed a locally cocartesian fibration over $\cO^\otimes \times \pos$: over a morphism $(S_+,p) \xra{\w{\alpha}} (T_+ ,q)$ in $\cO^\otimes \times \pos$ lying over a morphism $S_+ \xra{{\alpha}} T_+$ in $\cO^\otimes$, for any $X \in (\cR_p)^{\times S}$ and $Y \in (\cR_q)^{\times T}$ we have the string of equivalences
\begin{align}
\label{check.loc.cocart.over.Finstar.x.P.step.one}
\hom^{\w{\alpha}}_{\cR^\otimes \times \pos} ( ( (\rho^p)^{\times S}(X) , p ) , ( (\rho^q)^{\times T}(Y) , q ) )
&\simeq
\hom^{{\alpha}}_{\cR^\otimes} ( (\rho^p)^{\times S}(X) , (\rho^q)^{\times T}(Y) )
\\
\label{check.loc.cocart.over.Finstar.x.P.step.two}
&\simeq
\hom_{\cR^{\times T}} ( {\alpha}_* (\rho^p)^{\times S}(X) , (\rho^q)^{\times T}(Y) )
\\
\label{check.loc.cocart.over.Finstar.x.P.step.three}
&\simeq
\hom_{(\cR_q)^{\times T}} ( (\Phi_q)^{\times T} {\alpha}_* ( \rho^p)^{\times S}(X) , Y )
\\
\label{check.loc.cocart.over.Finstar.x.P.step.four}
&\simeq
\hom_{(\cR_q)^{\times T}} ( {\alpha}_* ( \Phi_q)^{\times S} (\rho^p)^{\times S}(X) , Y )
\\
\nonumber
& =:
\hom_{(\cR_q)^{\times T}} ( {\alpha}_* (\Phi_q \rho^p)^{\times S}(X) , Y )
\\
\nonumber
& =:
\hom_{(\cR_q)^{\times T}} ( {\alpha}_* ( \Gamma^p_q)^{\times S}(X) , Y )
~,
\end{align}
in which
\begin{itemize}
\item equivalence \Cref{check.loc.cocart.over.Finstar.x.P.step.one} follows from the fact that $\pos$ is a poset,
\item equivalence \Cref{check.loc.cocart.over.Finstar.x.P.step.two} follows from the fact that $\cR^\otimes \ra \cO^\otimes$ is a locally cocartesian fibration,
\item equivalence \Cref{check.loc.cocart.over.Finstar.x.P.step.three} follows from the adjunction $\Phi_q \adj \rho^q$ (or really the adjunction $(\Phi_q)^{\times T} \adj (\rho^q)^{\times T}$), and
\item equivalence \Cref{check.loc.cocart.over.Finstar.x.P.step.four} follows from the fact that $\Phi_q$ is $\cO$-monoidal.
\end{itemize}
This identification of the cocartesian monodromy in the locally cocartesian fibration $\GD^\otimes(\cR) \da (\cO^\otimes \times \pos)$ immediately implies condition \Cref{condition.pb.to.brax.two} of \Cref{lemma.functors.into.LMod.rlax.llax.B}, and for its condition \Cref{condition.pb.to.C} we observe that for each $p \in \pos$ the pullback to $\cO^\otimes \times \{ p \}$ is the cocartesian fibration $(\cR_p)^\otimes \da \cO^\otimes$.  Thus, we have indeed constructed a functor
\[
\cO^\otimes
\xra{\GD^\otimes(\cR)}
\LMod^\rlax_{\llax.\pos}
~.
\]
This is moreover a reduced Segal functor, which evaluates on each object $\ul{n}_+ \in \cO^\otimes$ as the locally cocartesian fibration $\GD(\cR)^{\times_\pos n} \da \pos$ (the $n$-fold fiber product with itself of the locally cocartesian fibration $\GD(\cR) \da \pos$ (recall \Cref{obs.lim.rlax.preserves.products})).  In other words, it defines an object
\[
\GD^\otimes(\cR)
\in
\Alg_\cO(\LMod^\rlax_{\llax.\pos})
\]
lifting the object $\GD(\cR) \in \LMod^\rlax_{\llax.\pos}$, as desired.

We now proceed to part \Cref{item.of.thm.s.m.reconstrn.enhance.gluing.diagram.functor}.


 We begin by constructing a morphism
\begin{equation}
\label{morphism.onlimrlaxes.from.RotimesxP.to.MotimesR}
\lim^\rlax_{\llax.\pos}(\cR^\otimes \times \pos)
\longra
\lim^\rlax_{\llax.\pos}(\GD^\otimes(\cR))
~.
\end{equation}
By definition, we have a morphism
\[
\cR^\otimes \times \pos
\longla
\GD^\otimes(\cR)
\]
in $\LMod^\llax_{\llax.(\cO^\otimes \times \pos)}$, which on each fiber is a right adjoint: over the object $(S_+,p) \in \Fin_* \times \pos$ it is the product right adjoint
\[
\cR^{\times S}
\xla{(\rho^p)^{\times S}}
(\cR_p)^{\times S}
~.
\]
By \Cref{lemma.ptwise.radjt.has.ptwise.ladjt}, the fiberwise left adjoints
\[
\cR^{\times S}
\xra{(\Phi_p)^{\times S}} (\cR_p)^{\times S}
\]
therefore assemble into a morphism
\begin{equation}
\label{fiberwise.ladjts.in.LMod.rlax.llax.Finstar.x.P}
\cR^\otimes \times \pos
\longra
\GD^\otimes(\cR)
\end{equation}
in $\LMod^\rlax_{\llax.(\cO^\otimes \times \pos)}$.  We consider this morphism as a point in the lower right space in the diagram
\begin{equation}
\label{many.subspaces.of.loc.cocart.over.brax.one.x.Finstar.x.P}
\begin{tikzcd}[row sep=1.5cm]
\hom_\Cat ( [1] , \Alg_\cO(\LMod^\rlax_{\llax.\pos}) )
\arrow[hook]{d}
\\
\hom_\Cat ( [1] , \Fun ( \cO^\otimes , \LMod^\rlax_{\llax.\pos} ) )
\\[-1.5cm]
\rotatebox{90}{$\simeq$}
\\[-1.5cm]
\hom_\Cat ( [1] \times \cO^\otimes , \LMod^\rlax_{\llax.\pos} )
\arrow[hook]{d}
\\
\iota_0 \loc.\coCart_{[1] \times \cO^\otimes \times \pos}
\arrow[hookleftarrow]{r}
&
\hom_\Cat ( [1] , \LMod^\rlax_{\llax.(\cO^\otimes \times \pos)} )
\ni
\Cref{fiberwise.ladjts.in.LMod.rlax.llax.Finstar.x.P}
\end{tikzcd}
\end{equation}
of spaces, in which the upper vertical inclusion is definitional and the other two inclusions follow from \Cref{lemma.functors.into.LMod.rlax.llax.B}.  As such, we aim to show that the point \Cref{fiberwise.ladjts.in.LMod.rlax.llax.Finstar.x.P} lies in the upper left space of diagram \Cref{many.subspaces.of.loc.cocart.over.brax.one.x.Finstar.x.P}.  So, let us consider its image
\begin{equation}
\label{object.in.loc.coCart.over.brax.one.x.Finstar.x.P}
\left(
\begin{tikzcd}
\cE
\arrow{d}
\\
{[1] \times \cO^\otimes \times \pos}
\end{tikzcd}
\right)
\in
\iota_0 \loc.\coCart_{[1] \times \cO^\otimes \times \pos}
~.
\end{equation}
To first show that the point \Cref{object.in.loc.coCart.over.brax.one.x.Finstar.x.P} factors through the lower vertical inclusion in diagram \Cref{many.subspaces.of.loc.cocart.over.brax.one.x.Finstar.x.P}, we verify conditions \Cref{condition.pb.to.C} and \Cref{condition.pb.to.brax.two} of \Cref{lemma.functors.into.LMod.rlax.llax.B} in turn.
\begin{enumerate}
\item
For each $p \in \pos$, the pullback
\[ \begin{tikzcd}[row sep=1.5cm, column sep=2.5cm]
\cE_{| [1] \times \cO^\otimes \times \{p \}}
\arrow{r}
\arrow{d}
&
\cE
\arrow{d}
\\
{[1] \times \cO^\otimes}
\arrow{r}[swap]{(\id_{[1] \times \cO^\otimes} , \const_p )}
&
{[1] \times \cO^\otimes \times \pos}
\end{tikzcd} \]
is indeed a cocartesian fibration: it is classified by the morphism $\cR \xra{\Phi_p} \cR_p$ in $\Alg_\cO(\Cat)$.
\item
Any pair of a morphism $(i,S_+) \xra{(\leq,\alpha)} (j,T_+)$ in $[1] \times \cO^\otimes$ and a morphism $p \leq q$ in $\pos$ determines a functor $[2] \ra [1] \times \cO^\otimes \times \pos$ classifying the commutative triangle
\[ \begin{tikzcd}
(i,S_+,p)
\arrow{d}
\arrow{rd}
\\
(i,S_+,q)
\arrow{r}
&
(j,T_+,q)
\end{tikzcd}~, \]
and we must show that the resulting pullback
\begin{equation}
\label{pb.to.brax.two.in.pf.of.s.m.reconstrn}
\begin{tikzcd}
\cE_{|[2]}
\arrow{r}
\arrow{d}
&
\cE
\arrow{d}
\\
{[2]}
\arrow{r}
&
{[1] \times \cO^\otimes \times \pos}
\end{tikzcd}
\end{equation}
defines a cocartesian fibration over $[2]$.  By what we have already seen, this holds when $i=j$ (because both of the locally cocartesian fibrations $(\cR^\otimes \times \pos) \da (\cO^\otimes \times \pos)$ and $\GD^\otimes(\cR) \da (\cO^\otimes \times \pos)$ satisfy condition \Cref{condition.pb.to.brax.two}).
In the remaining case where $i=0$ and $j=1$, the pullback \Cref{pb.to.brax.two.in.pf.of.s.m.reconstrn} is the cocartesian fibration over $[2]$ classifying the commutative triangle
\[ \begin{tikzcd}
\cR^{\times S}
\arrow{d}[swap]{\id_{\cR^{\times S}}}
\arrow{rd}
\\
\cR^{\times S}
\arrow{r}
&
(\cR_p)^{\times T}
\end{tikzcd} \]
in $\Cat$ in which both rightward functors coincide with the diagonal composite in the commutative square
\[ \begin{tikzcd}
\cR^{\times S}
\arrow{r}{\alpha_*}
\arrow{d}[swap]{(\Phi_p)^{\times S}}
&
\cR^{\times T}
\arrow{d}{(\Phi_p)^{\times T}}
\\
(\cR_p)^{\times S}
\arrow{r}[swap]{\alpha_*}
&
(\cR_p)^{\times T}
\end{tikzcd} \]
in $\Cat$ (which commutes because $\cR \xra{\Phi_p} \cR_p$ is $\cO$-monoidal).
\end{enumerate}
Hence, the point \Cref{object.in.loc.coCart.over.brax.one.x.Finstar.x.P} does indeed factor through the lower vertical inclusion in diagram \Cref{many.subspaces.of.loc.cocart.over.brax.one.x.Finstar.x.P}.  Thereafter, considered as a point in $\hom_\Cat ( [1] , \Fun(\cO^\otimes , \LMod^\rlax_{\llax.\pos} ) )$, i.e.\! as a morphism in $\Fun(\cO^\otimes , \LMod^\rlax_{\llax.\pos} )$, its source and target evidently both lie in the full subcategory $\Alg_\cO(\LMod^\rlax_{\llax.\pos}) \subseteq \Fun(\cO^\otimes , \LMod^\rlax_{\llax.\pos} )$: its source is the composite
\begin{equation}
\label{source.of.parametrized.ladjt.over.Finstar.x.P}
\cO^\otimes
\xra{\cR^\otimes}
\Cat
\xra{- \times \pos}
\coCart_\pos
\longhookra
\loc.\coCart_\pos
=:
\LMod_{\llax.\pos}
\longrsurjmono
\LMod^\rlax_{\llax.\pos}~,
\end{equation}
while its target is the object $\GD^\otimes(\cR)$.  Therefore, the point \Cref{object.in.loc.coCart.over.brax.one.x.Finstar.x.P} lies in the uppermost space of diagram \Cref{many.subspaces.of.loc.cocart.over.brax.one.x.Finstar.x.P}, as we aimed to show.  Hence, we may take its postcomposition
\[
[1]
\longra
\Alg_\cO(\LMod^\rlax_{\llax.\pos})
\xra{\lim^\rlax_{\llax.\pos}}
\Alg_\cO(\Cat)
~,
\]
which provides the desired morphism \Cref{morphism.onlimrlaxes.from.RotimesxP.to.MotimesR}.

Now, as observed above, the object
\[
(\cR^\otimes \times \pos) \in
\Alg_\cO(\LMod^\rlax_{\llax.\pos})
\subseteq
\Fun ( \cO^\otimes , \LMod^\rlax_{\llax.\pos} )
\]
factors as the composite \Cref{source.of.parametrized.ladjt.over.Finstar.x.P}, so that we may identify the source $\lim^\rlax_{\llax.\pos}(\cR^\otimes \times \pos)$ of the morphism \Cref{morphism.onlimrlaxes.from.RotimesxP.to.MotimesR} in $\Alg_\cO(\Cat) \subseteq \Fun(\cO^\otimes,\Cat)$ in simple terms: it is the composite
\[
\cO^\otimes
\xra{\cR^\otimes}
\Cat
\xra{- \times \pos^\op}
\Cart_{\pos^\op}
\xlongra{\Gamma}
\Cat
~,
\]
which classifies the $\infty$-category $\Fun(\pos^\op,\cR)$ equipped with its pointwise $\cO$-monoidal structure.  This receives a canonical morphism
\[
\cR
\longra
\Fun(\pos^\op,\cR)
\]
in $\Alg_\cO(\Cat)$.\footnote{This canonical morphism factors through the strict limit $\lim_{\llax.\pos}(\cR^\otimes \times \pos)$, which can be similarly identified with $\Fun(|\pos| , \cR) \simeq \Fun ( |\pos^\op| , \cR)$ equipped with its pointwise $\cO$-monoidal structure.}
So, we obtain a composite comparison morphism
\[
\cR
\longra
\Fun ( \pos^\op , \cR)
\simeq
\lim^\rlax_{\llax.\pos} ( \cR^\otimes \times \pos)
\xra{\Cref{morphism.onlimrlaxes.from.RotimesxP.to.MotimesR}}
\lim^\rlax_{\llax.\pos} ( \GD^\otimes(\cR) )
\]
in $\Alg_\cO(\Cat)$.  Moreover, by construction, upon applying the forgetful functor
\[
\fgt
:
\Alg_\cO(\Cat)
\xlonghookra{\ff}
\Fun(\cO^\otimes,\Cat)
\xra{\ev_{\ul{1}_+}}
\Cat
\]
we recover the morphism
\[
\cR
\xlongra{\gd}
\lim^\rlax_{\llax.\pos} ( \GD(\cR) )
~,
\]
as desired.
\end{proof}

\begin{remark}
It is not possible to prove \Cref{thm.s.m.reconstrn} directly from \Cref{macrocosm.thm}, because a presentably $\cO$-monoidal $\infty$-category is not defined by a diagram in $\PrLSt$ (as the tensor product functors are required to be multi-cocontinuous rather than cocontinuous).
\end{remark}

\subsection{Symmetric monoidal stratifications and tensor-triangular geometry}
\label{subsection.tt.geometry}

In this subsection, we construct the adelic stratification of a presentably symmetric monoidal stable $\infty$-category (satisfying mild finiteness hypotheses) as \Cref{thm.s.m.stratn.over.balmer.spectrum}.  This is based in the theory of tensor-triangular geometry, which we begin by reviewing; we refer the reader to the survey \cite{Stevenson-tour} for more background on this topic, which highlights the interaction between the small and presentable settings.  We unpack the adelic stratification of $\Spectra$ in \Cref{ex.adelic.stratn.of.spectra}, and we explain how symmetric monoidal stratifications contribute to the theory of tensor-triangular geometry in \Cref{rmk.stratns.helps.tt.geometry}.

\begin{observation}
The homotopy category of a stable $\infty$-category is canonically triangulated, and (presentably) symmetric monoidal structures descend to (resp.\! exact and coproduct-preserving) symmetric monoidal structures.  Through this, one can largely apply results concerning triangulated categories to stable $\infty$-categories without any modification; for instance, the condition of an object being zero can be checked in the homotopy category, and the projection to the homotopy category preserves co/products (indeed, this is true for any $\infty$-category).
We use this fact without further comment.
\end{observation}

\begin{local}
\label{local.R.is.s.m.cpctly.gend}
For the remainder of this section, we specialize \Cref{notn.presentably.O.mon.R} to further assume that $\cO = \Comm$, i.e.\! that $\cR$ is a presentably symmetric monoidal stable $\infty$-category.  We assume moreover that $\cR$ is compactly generated, and that its full subcategory $\cR^\omega$ of compact objects inherits a symmetric monoidal structure (i.e.\! the unit object is compact and the tensor product of compact objects is again compact).
\end{local}

\begin{definition}
\label{defn.rigidly.cpctly.gend}
We say that $\cR$ is \bit{rigidly-compactly generated} if (in addition to the hypotheses of \Cref{local.R.is.s.m.cpctly.gend}) its full subcategory of dualizable (a.k.a.\! rigid) objects is precisely $\cR^\omega \subseteq \cR$.
\end{definition}

\begin{definition}
A full proper stable subcategory $\mf{p} \subsetneq \cR^\omega$ is called a \bit{thick prime ideal} if
\begin{itemize}
\item it is idempotent-complete,
\item it is contagious under the symmetric monoidal structure, and
\item for all $X,Y \in \cR^\omega$, if $X \otimes Y \in \mf{p}$ then $X \in \mf{p}$ or $Y \in \mf{p}$.
\end{itemize}
We write $\pos_\cR$ for the poset of thick prime ideal subcategories of $\cR^\omega$ ordered by inclusion.
\end{definition}

\begin{definition}
\label{defn.balmer.spectrum.and.support}
The \bit{Balmer spectrum} of $\cR^\omega$ is the topological space $\Spec(\cR^\omega) \in \Top$ defined as follows.  First of all, the underlying set of $\Spec(\cR^\omega)$ is that of thick prime ideals in $\cR^\omega$.  Then, for any object $X \in \cR^\omega$, we define its \bit{support} to be the subset
\[
\supp(X)
:=
\{
\mf{p} \in \Spec(\cR^\omega)
:
X \not\in \mf{p}
\}
~.
\]
Finally, the topology on $\Spec(\cR^\omega)$ is obtained by declaring that the subsets $\{ \supp(X) \subseteq \Spec(\cR^\omega) \}_{X \in \cR^\omega}$ are closed.\footnote{In fact, these subsets form a basis, so that every closed subset is of the form
\[
\bigcap_{s \in S}
\supp(X_s)
=
\{ \mf{p} \in \Spec(\cR^\omega)
:
\{ X_s \}_{s \in S} \cap \mf{p} = \es
\}
\]
for some set $\{ X_s \in \cR^\omega \}_{s \in S}$ of objects of $\cR^\omega$.}
\end{definition}

\begin{remark}
The specialization poset of the topological space $\Spec(\cR^\omega) \in \Top$ is precisely $\pos_\cR$: the membership $\mf{p} \in \ol{\{\mf{q}\}}$ is equivalent to the containment $\mf{p} \subseteq \mf{q}$.  So, we may consider the support of an object $X \in \cR^\omega$ either as a closed subset of $\Spec(\cR^\omega)$ or as a down-closed subset of $\pos_\cR$.
\end{remark}

\begin{remark}
Let $X$ be a qcqs scheme.  Thomason proved that there is a canonical isomorphism
\begin{equation}
\label{iso.of.topological.spaces.or.lrs}
\Spec(\Perf(X)) \cong X
\end{equation}
of topological spaces \cite[Theorem 3.15]{Thomason-classification},\footnote{See also \cite{Neeman-chrom} for an affine version of this result, which originates in \cite{Hopkins-global}.} with the correspondence being given by the support of perfect complexes.  Thereafter, Balmer upgraded the topological space $\Spec(\cR^\omega)$ to a ringed topological space \cite[Definition 6.1]{Balmer-specofprime}, in such a way that the isomorphism
\Cref{iso.of.topological.spaces.or.lrs} 
naturally upgrades to one of ringed topological spaces (and therefore one of schemes) \cite[Theorem 6.3]{Balmer-specofprime}.
\end{remark}

\begin{notation}
For each $\mf{p} \in \pos_\cR$, we define the full subcategories
\[
\cI^\omega_\mf{p}
:=
\{ X \in \cR^\omega : \supp(X) \subseteq ( ^\leq \mf{p} ) \}
\subseteq
\cR^\omega
\qquad
\text{and}
\qquad
\cI_\mf{p}
:=
\Ind(\cI^\omega_\mf{p})
=
\brax{\cI^\omega_\mf{p}}
\subseteq
\cR
~.
\]
\end{notation}

\begin{observation}
For each $\mf{p} \in \pos_\cR$, the subcategory $\cI_\mf{p} \subseteq \cR$ is obviously closed (recall \Cref{ex.cpct.objs.gen.clsd.subcat}), and in fact it is a closed ideal by \cite[Theorem 3.3.3]{HPS} 
(which is an abstraction of \cite[Corollary 8]{Miller-finite}).  As the assignment $\mf{p} \mapsto \cI_\mf{p}$ is order-preserving, we therefore obtain a functor
\begin{equation}
\label{hopefully.adelic.stratn}
\begin{tikzcd}[row sep=0cm]
\pos_\cR
\arrow{r}{\cI_\bullet}
&
\Idl_\cR
\\
\rotatebox{90}{$\in$}
&
\rotatebox{90}{$\in$}
\\
\mf{p}
\arrow[maps to]{r}
&
\cI_\mf{p}
\end{tikzcd}
~.
\end{equation}
\end{observation}

\begin{definition}
Whenever the functor \Cref{hopefully.adelic.stratn} is a symmetric monoidal stratification, we refer to it as the \bit{adelic stratification} of $\cR$ over $\pos_\cR$.
\end{definition}

\begin{theorem}
\label{thm.s.m.stratn.over.balmer.spectrum}
Suppose that $\cR$ is a rigidly-compactly generated presentably symmetric monoidal stable $\infty$-category, and suppose that $\cR = \brax{ \cI_\mf{p}}_{\mf{p} \in \pos_\cR}$.  Then, the functor \Cref{hopefully.adelic.stratn} defines a symmetric monoidal stratification of $\cR$ over $\pos_\cR$.
\end{theorem}

\begin{proof}
By assumption, the functor \Cref{hopefully.adelic.stratn} is a symmetric monoidal prestratification.  Note that if $X,Y \in \cR^\omega$ then
\[
\supp(X \otimes Y) \subseteq \supp(X) \cap \supp(Y)
\]
by definition of a thick prime ideal.  Since the symmetric monoidal structure commutes with colimits separately in each variable, the stratification condition follows from \Cref{obs.yo.comm.for.s.m.stratns}.
\end{proof}

\begin{remark}
\label{rmk.counterex.to.generation.condition.for.adelic}
We indicate an example in which the condition that $\cR = \brax{ \cI_\mf{p}}_{\mf{p} \in \pos_\cR}$ appearing in \Cref{thm.s.m.stratn.over.balmer.spectrum} fails to hold.\footnote{We thank Scott Balchin for pointing out this example to us.} Let $S$ be a countably infinite set, let $S^+$ denote its one-point compactification, let $R := \hom_\Top( S^+ ,\FF_2)$ denote the commutative ring of continuous $\FF_2$-valued functions on $S^+$, and let $\cR := \QC(\Spec(R))$. Then there are canonical homeomorphisms $S^+ \cong \Spec(R) \cong \Spec(\cR^\omega)$, and in particular the specialization poset $\pos_\cR$ is discrete (as $R$ has Krull dimension 0). However, the functor
\[
\cR
\xra{(y)_{\mf{p} \in \pos_\cR}}
\prod_{\mf{p} \in \pos_\cR} \cI_{\mf{p}}
\]
is not an equivalence.
\end{remark}



\begin{example}[the adelic stratification of spectra]
\label{ex.adelic.stratn.of.spectra}
The adelic stratification of $\Spectra$ is quite similar to its chromatic stratification (\Cref{ex.chromatic.stratn.of.spectra}): namely, it is the functor
\[
\begin{tikzcd}[row sep=0cm]
\pos_\Spectra
\arrow{r}{\cI_\bullet^f}
&
\Idl_\Spectra
\\
\rotatebox{90}{$\in$}
&
\rotatebox{90}{$\in$}
\\
\mf{p}
\arrow[maps to]{r}
&
\cI_\mf{p}^f
\end{tikzcd}
\]
defined by the assignments
\[
\cI_{(0)}^f = \Spectra
~,
\qquad
\cI_{((p),n)}^f
=
A_{p,n-1}^f L_{(p)}^f \Spectra
~,
\qquad
\text{and}
\qquad
\cI_{((p),\infty)}^f = 0
~,
\]
where
\begin{itemize}

\item we identify the poset of primes in the Balmer spectrum as $\pos_\Spectra$ (depicted in \Cref{figure.primes.of.spectra}) by \cite[Corollary 9.5]{Balmer-spectroix} (see also \cite{HS-nilptwo}),

\item we use the superscript $f$ to denote the finite localization/acyclification functors \cite[Definition 3]{Miller-finite}, and

\item the identifications of the minimal strata as zero follows from the fact that finite spectra are harmonic \cite[Corollary 4.5]{Rav-loc}.

\end{itemize}
Note that the telescope conjecture asserts that the morphisms $L_{p,n-1}^f \ra L_{p,n-1}$ (or equivalently the morphisms $A_{p,n-1}^f \ra A_{p,n-1}$) are equivalences.
\end{example}

\begin{remark}
\label{rmk.stratns.helps.tt.geometry}
We view the theory of symmetric monoidal stratifications as an important complement to the study of tensor-triangular geometry, for the following two reasons.
\begin{enumerate}

\item

While the Balmer spectrum has a universal property \cite[Theorem 3.2]{Balmer-specofprime} it can be quite difficult to compute.  By contrast, our general theory of (symmetric monoidal) stratifications is substantially more flexible.
\begin{enumeratesub}
\item
For instance, in addition to its rather subtle stratification indicated in \Cref{ex.chromatic.stratn.of.spectra}, the $\infty$-category $\Spectra$ of spectra admits an ``arithmetic'' stratification over $\pos_{\Mod_\ZZ}$, which behaves just as that of $\Mod_\ZZ$ itself as described in \Cref{ex.intro.arithmetic}.
\item
Likewise, as we prove in \Cref{thm.geom.stratn.of.SpgG}, for a compact Lie group $G$, the $\infty$-category $\Spectra^{\gen G}$ of genuine $G$-spectra admits a relatively straightforward stratification over the poset $\pos_G$ of closed subgroups of $G$; compare this with the computations of its Balmer spectrum \cite{BS-spec-SpgG,Tobiplusplus-Specgenfabgrp,Tobiplus-SpecofcpctLie}.
\end{enumeratesub}
This flexibility allows for the systematic study of tensor-triangulated categories that is compatible with, but not bound to, their Balmer spectra; and it is of course further augmented by the fundamental operations for (symmetric monoidal (recall \Cref{rmk.fund.opns.for.O.mon.stratns})) stratifications developed in \Cref{subsection.structure.theory}.

\item

Our theory of symmetric monoidal stratifications appears to provide a compelling framework for studying the ``presheaf of triangulated categories'' that serves as motivation throughout the literature on tensor-triangular geometry (originating with \cite{Balmer-pshvs}), enhancing as it does the presheaf of commutative rings introduced in \cite[Definition 6.1]{Balmer-specofprime}.  In this vein, we view our symmetric monoidal reconstruction theorem (\Cref{thm.s.m.reconstrn}) as encoding a form of descent for this (pre?)sheaf.  In particular, we expect that our theory straightforwardly recovers the reconstruction results of e.g.\! \cite{Balmer-suppfilt,BalmerFavi-gluing}.\footnote{Of course, this notion of descent is necessarily $\infty$-categorical, and cannot be carried through at the level of homotopy categories.  In particular, we expect that such recovery would repair the failure of uniqueness of gluings that arises in \cite{BalmerFavi-gluing}, which appears to come of working with homotopy categories instead of $\infty$-categories.}

\end{enumerate}
\end{remark}

\section{The geometric stratification of genuine $G$-spectra}
\label{section.genuine}

In this section, we prove our symmetric monoidal stratification of genuine $G$-spectra (\Cref{intro.thm.gen.G.spt}).  This gives a reconstruction theorem for genuine $G$-spectra when $G$ is a finite group, which we unpack in a number of examples.

\begin{local}
In this section, we write $G$ for an arbitrary compact Lie group, and we write $H$ for an arbitrary closed subgroup of $G$.
\end{local}

This section is organized as follows.
\begin{itemize}

\item[\Cref{subsection.stratn.of.SpgG}:] We set our conventions regarding genuine $G$-spectra and prove \Cref{intro.thm.gen.G.spt} as \Cref{thm.geom.stratn.of.SpgG}.

\item[\Cref{subsection.tate.constrn}:] We study the gluing functors of the geometric stratification of genuine $G$-spectra, which are versions of the Tate construction.

\item[\Cref{subsection.examples.of.SpgG}:] We unpack our reconstruction theorem for genuine $G$-spectra in the cases where $G \in \{ \Cyclic_p , \Cyclic_{p^2}, \Cyclic_{pq}, \Symm_3 \}$ (for $p$ and $q$ distinct primes).  We also discuss the geometric stratification of genuine $\TT$-spectra and the resulting reconstruction theorem for proper-genuine $\TT$-spectra.

\item[\Cref{subsection.categorical.fixedpoints}:] We specialize our nanocosm reconstruction theorem to give a formula for the categorical $H$-fixedpoints of genuine $G$-spectra (when $G$ is finite).

\end{itemize}


\subsection{The geometric stratification of genuine $G$-spectra}
\label{subsection.stratn.of.SpgG}

In this subsection, we establish the symmetric monoidal stratification of genuine $G$-spectra as \Cref{thm.geom.stratn.of.SpgG}.  We begin by laying out our notation and recalling the facts that we need; for further background on genuine $G$-spectra, we refer the reader to \cite{LMS,May-Alaska,ManMay-eq}.

\needspace{2\baselineskip}
\begin{notation}
\label{notn.omnibus.genuine.stuff}
\begin{enumerate}
\item[]

\item We write
\[ \Spaces^{\gen G} \]
for the $\infty$-category of \textit{genuine $G$-spaces}.

\item We write
\[ \Orb_G \subseteq \Spaces^{\gen G} \]
for the \textit{orbit $\infty$-category} of $G$, the full subcategory on those objects of the form $G/H$.

\item\label{item.of.notn.PG}
We write $\pos_G$ for the poset of conjugacy classes of closed subgroups of $G$ ordered by subconjugacy (the posetification (i.e.\! homwise $(-1)$-truncation) of $\Orb_G$).

\item We write
\[ \Spectra^{\gen G} \]
for the $\infty$-category of \textit{genuine $G$-spectra}, i.e.\! the stable $\infty$-category of spectral presheaves on $\Orb_G$ with the representation spheres inverted under the symmetric monoidal structure.

\item We write
\[ \begin{tikzcd}[column sep=2cm, row sep=0cm]
\Spaces^{\gen G}_*
\arrow[transform canvas={yshift=0.9ex}]{r}{\Sigma^\infty_G}
\arrow[leftarrow, transform canvas={yshift=-0.9ex}]{r}[yshift=-0.2ex]{\bot}[swap]{\Omega^\infty_G}
&
\Spectra^{\gen G}
\end{tikzcd} \]
for the adjoint functors of (\textit{genuine $G$-})\textit{suspension spectrum} and (\textit{pointed genuine $G$-})\textit{infinite loopspace}.

\item We write
\[
\Spectra^{\htpy G}
:=
\Fun(\BG,\Spectra)
\]
for the $\infty$-category of \textit{homotopy $G$-spectra}.

\item We write
\[ \begin{tikzcd}[column sep=2cm, row sep=0cm]
\Spectra^{\gen G}
\arrow[transform canvas={yshift=0.9ex}]{r}{U_G}
\arrow[hookleftarrow, transform canvas={yshift=-0.9ex}]{r}[yshift=-0.2ex]{\bot}[swap]{\beta_G}
&
\Spectra^{\htpy G}
\end{tikzcd} \]
for the adjunction -- a reflective localization -- whose left adjoint is the forgetful functor and whose right adjoint is the \textit{Borel-complete genuine $G$-spectrum} functor.\footnote{That is, $\beta_G$ is the inclusion of the full subcategory of \textit{Borel-complete} genuine $G$-spectra, i.e.\! those objects $E \in \Spectra^{\gen G}$ such that the canonical map $E^H \ra E^{\htpy H}$ (from genuine $H$-fixedpoints to homotopy $H$-fixedpoints) is an equivalence for all closed subgroups $H \leq G$.}
We may also omit the subscripts, simply writing $U \adj \beta$ instead of $U_G \adj \beta_G$.

\end{enumerate}
\end{notation}

\begin{warning}
\Cref{notn.omnibus.genuine.stuff}\Cref{item.of.notn.PG} introduces a mild clash: given closed subgroups $H$ and $K$ of $G$, we may write $H \leq K$ when $H$ is subconjugate to $K$ but not necessarily actually contained in it. On the other hand, in such situations we generally assume (without real loss of generality) that $H$ is in fact contained in $K$. To emphasize that we truly mean containment, we use the notation $\subseteq$.
\end{warning}

\begin{remark}
We will often refer to the set $\{ G/H \in \Orb_G \}_{H \in \pos_G}$ (and variants thereof).  This may appear to be ill-defined, as the objects of $\pos_G$ are only conjugacy classes of subgroups of $G$.  However, a conjugation relation $H' = gHg^{-1}$ determines an equivalence $G/H' \simeq G/H$.  Thus, this notation is effectively unambiguous.
\end{remark}

\begin{notation}
We respectively write
\[
\Normzer(H) := \Normzer_G(H)
\qquad
\text{and}
\qquad
\Weyl(H)
:=
\Weyl_G(H)
:=
\Normzer(H)/H
\]
for the normalizer and Weyl group of the closed subgroup $H \leq G$.
\end{notation}

\begin{observation}
We record the following facts, which we use without further comment.

\begin{enumerate}
\item
The set $\{ G/H \in \Orb_G \subseteq \Spaces^{\gen G} \}_{H \in \pos_G}$ of orbits compactly generates $\Spaces^{\gen G}$: by Elmendorf's theorem, the restricted Yoneda functor is an equivalence
\[ 
\Spaces^{\gen G}
\xlongra{\sim}
\Fun(\Orb_G^\op , \Spaces)
~. \]
Under this identification, the genuine $H$-fixedpoints functor $(-)^H$ corresponds to evaluation at the object $(G/H)^\circ \in \Orb_G^\op$.

\item
The set $\{ \Sigma^\infty_G (G/H)_+ \in \Spectra^{\gen G} \}_{H \in \pos_G}$ of suspension spectra of orbits compactly generates $\Spectra^{\gen G}$.


\item
The $\infty$-categories $\Spaces^{\gen G}$ and $\Spaces^{\gen G}_*$ are both presentably symmetric monoidal, with their respective cartesian product (denoted $\times$) and smash product (denoted $\wedge$) defined pointwise: that is, these symmetric monoidal structures commute with taking genuine fixedpoints.

\item
The $\infty$-category $\Spectra^{\gen G}$ is presentably symmetric monoidal via the smash product (denoted $\otimes$).

\item
The genuine $G$-suspension spectrum functor
\[ \Spaces^{\gen G}_* \xra{\Sigma^\infty_G} \Spectra^{\gen G} \]
is symmetric monoidal.


\item
The Weyl group $\Weyl(H)$ is (the underlying $\infty$-group of) the compact Lie group of $G$-equivariant automorphisms of $G/H$.


\item
Given a normal closed subgroup $H \in \pos_G$, the categorical $H$-fixedpoints functor fits into a commutative square
\[ \begin{tikzcd}[row sep=1.5cm, column sep=1.5cm]
\Spaces_*^{\gen (G/H)}
\arrow[leftarrow]{r}{(-)^H}
\arrow[leftarrow]{d}[swap]{\Omega^\infty_{G/H}}
&
\Spaces_*^{\gen G}
\arrow[leftarrow]{d}{\Omega^\infty_G}
\\
\Spectra^{\gen (G/H)}
\arrow[leftarrow]{r}[swap]{(-)^H}
&
\Spectra^{\gen G}
\end{tikzcd} \]
that is obtained by passing to right adjoints in the commutative square
\[ \begin{tikzcd}[row sep=1.5cm, column sep=1.5cm]
\Spaces_*^{\gen (G/H)}
\arrow{r}{\Res^{G/H}_G}
\arrow{d}[swap]{\Sigma^\infty_{G/H}}
&
\Spaces_*^{\gen G}
\arrow{d}{\Sigma^\infty_G}
\\
\Spectra^{\gen (G/H)}
\arrow{r}[swap]{\Res^{G/H}_G}
&
\Spectra^{\gen G}
\end{tikzcd} \]
in $\CAlg(\PrL)$ (which itself is deduced from the universal property of genuine $G/H$-spectra). For an arbitrary closed subgroup $H \in \pos_G$, the categorical $H$-fixedpoints functor is the composite
\[
(-)^H
:
\Spectra^{\gen G}
\xra{\Res^G_{\Normzer(H)}}
\Spectra^{\gen \Normzer(H)}
\xra{(-)^H}
\Spectra^{\gen \Weyl(H)}
~.
\]

\item Categorical fixedpoints compose: if $K \leq H  \leq \Normzer_G(K) \leq G$ then the triangle
\[ \begin{tikzcd}[row sep=1.5cm, column sep=1.5cm]
\Spectra^{\gen G}
\arrow{r}{(-)^K}
\arrow{rd}[sloped, swap]{(-)^H}
&
\Spectra^{\gen \Weyl_G(K)}
\arrow{d}{(-)^{H/K}}
\\
&
\Spectra^{\gen \Weyl_G(H)}
\end{tikzcd} \]
commutes.\footnote{Note the canonical isomorphism $\Weyl_{\Weyl_G(K)}(H/K) \cong \Weyl_G(H)$.}

\item At the level of underlying homotopy $\Weyl(H)$-spectra, categorical $H$-fixedpoints are corepresented by $\Sigma^\infty_G(G/H)_+$: the diagram
\[ \begin{tikzcd}[row sep=1cm, column sep=3.5cm]
\Spectra^{\gen G}
\arrow{r}{(-)^H}
\arrow{rd}[sloped, swap]{\ulhom_{\Spectra^{\gen G}}(\Sigma^\infty_G(G/H)_+,-)}
&
\Spectra^{\gen \Weyl(H)}
\arrow{d}{U}
\\
&
\Spectra^{\htpy \Weyl(H)}
\end{tikzcd} \]
canonically commutes.

\end{enumerate}
\end{observation}

\begin{notation}
\label{notn.no.extra.notn.for.categorical.fixedpts}
We often simply write
\[
(-)^H
:
\Spectra^{\gen G}
\xra{(-)^H}
\Spectra^{\gen \Weyl(H)}
\xlongra{U}
\Spectra^{\htpy \Weyl(H)}
\]
for the composite.\footnote{This is in contrast with our conventions for geometric fixedpoints appearing in \Cref{defn.geom.H.fps}.} Our meaning will always be clear from context.
\end{notation}

\begin{notation}
We denote by $\tensoring$ the action on $\Spectra^{\gen G}$ of $\Spaces^{\gen G}_*$. So by definition, for any $X \in \Spaces^{\gen G}_*$ and $E \in \Spectra^{\gen G}$ we have
\[
X
\tensoring
E
\simeq
\Sigma^\infty_G X \otimes E
\in
\Spectra^{\gen G}
~.
\]
\end{notation}

\begin{definition}
\label{defn.geometric.prestratn.of.gen.G.spt}
The \bit{geometric prestratification} of $\Spectra^{\gen G}$ over $\pos_G$ is the functor
\[
\hspace{3.5cm}
\begin{tikzcd}[column sep=1.5cm, row sep=0cm]
\pos_G
\arrow{r}{\Spectra^{\gen G}_{^\leq \bullet}}
&
\Cls_{\Spectra^{\gen G}}
\\
\rotatebox{90}{$\in$}
&
\rotatebox{90}{$\in$}
\\
H
\arrow[maps to]{r}
&
\Spectra^{\gen G}_{^\leq H}
&[-1.8cm]
:=
\brax{\Sigma^\infty_G (G/K)_+}_{K \leq H}
\end{tikzcd} \]
sending an element $H \in \pos_G$ to the closed subcategory generated by the set
\[ \{ \Sigma^\infty_G ( G/K)_+ \in (\Spectra^{\gen G})^\omega \}_{K \leq H} \]
of compact objects (recall \Cref{ex.cpct.objs.gen.clsd.subcat}).
\end{definition}

\begin{definition}
A \bit{family} is an element of the poset $\Down_{\pos_G}$, i.e.\! a set of closed subgroups of $G$ that is closed under subconjugacy.  To align with standard notation, we denote an arbitrary family by $\ms{F} \in \Down_{\pos_G}$, and given an element $\sD \in \Down_{\pos_G}$ we also write $\ms{F}_\sD := \sD$.
\end{definition}

\begin{local}
In this subsection, in the course of proving that the geometric prestratification is in fact a symmetric monoidal stratification, we may write
\[ \cI_H := \Spectra^{\gen G}_{^\leq H} ~, \]
for brevity. Similarly, for any family $\ms{F} \in \Down_{\pos_G}$, we may write
\[
\cI_\ms{F}
:=
\brax{ \cI_K }_{K \in \ms{F}}
\simeq
\brax{ \Sigma^\infty_G (G/K)_+ }_{K \in \ms{F}}
~.
\]
\end{local}

\begin{notation}
For any family $\ms{F} \in \Down_{\pos_G}$, we write $\EFunptd \in \Spaces^{\gen G}$ for the genuine $G$-space characterized by the fact that
\[ ( \EFunptd )^H \simeq \left\{ \begin{array}{ll}
\pt~, & H \in \ms{F} \\
\es~, & H \not\in \ms{F}
\end{array} \right.
~.\footnote{Said differently, $(\EFunptd \da \ms{O}_G) \in \RFib(\ms{O}_G) \simeq \Fun(\Orb_G^\op,\Spaces) \simeq \Spaces^{\gen G}$ fits into a pullback square
\[ \begin{tikzcd}[ampersand replacement=\&]
\EFunptd
\arrow[hook]{r}{\ff}
\arrow{d}
\&
\Orb_G
\arrow{d}
\\
\ms{F}
\arrow[hook]{r}[swap]{\ff}
\&
\pos_G
\end{tikzcd}~. \]
}
\]
\end{notation}

\begin{definition}
\label{defn.isotropy.separation.sequence}
For any family $\ms{F} \in \Down_{\pos_G}$, the corresponding \bit{isotropy separation sequence} is the cofiber sequence
\begin{equation}
\label{isotropy.separation.sequence}
\EF
\longra
S^0
\longra
\wEF
\end{equation}
in $\Spaces^{\gen G}_*$, where the first morphism is obtained by applying the functor $\Spaces^{\gen G} \xra{(-)_+} \Spaces_*^{\gen G}$ to the unique morphism $\EFunptd \ra \pt$ in $\Spaces^{\gen G}$.
\end{definition}

\begin{observation}
Applying the genuine $H$-fixedpoints functor $(-)^H$ to the isotropy separation sequence \Cref{isotropy.separation.sequence}, we obtain the cofiber sequence
\[ 
\left(
\EF
\longra
S^0
\longra
\wEF
\right)^H
\simeq
\left\{ \begin{array}{ll}
S^0 \xlongra{\sim} S^0 \longra \pt
~,
&
H \in \ms{F}
\\
\pt \longra S^0 \xlongra{\sim} S^0
~,
&
H \not\in \ms{F}
\end{array} \right. \]
in $\Spaces_*$.  Extending \Cref{defn.idempotents.and.centrality} and \Cref{obs.centrality.vacuous.or.easy} to the unstable setting in the evident way, we find that the objects
\[
( \EF \longra S^0 ) \in (\Spaces^{\gen G}_*)_{/S^0}
\qquad
\text{and}
\qquad
( \Sigma^\infty_G \EF \longra \Sigma^\infty_G S^0 \simeq \SS ) \in ( \Spectra^{\gen G} )_{/ \SS}
\]
are central augmented idempotents and that the objects
\[
( S^0 \longra \wEF ) \in (\Spaces^{\gen G}_*)_{S^0/}
\qquad
\text{and}
\qquad
( \SS \simeq \Sigma^\infty_G S^0 \longra \Sigma^\infty_G \wEF ) \in ( \Spectra^{\gen G} )_{\SS /}
\]
are central coaugmented idempotents. We use these facts without further comment.
\end{observation}

\begin{observation}
\label{obs.EFQ.colocalizes.SpgG.into.ZQ}
For any family $\ms{F} \in \Down_{\pos_G}$, 
the counit of the adjunction
\[ \begin{tikzcd}[column sep=2cm, row sep=0cm]
\cI_\ms{F}
\arrow[hook, transform canvas={yshift=0.9ex}]{r}{i_L}
\arrow[leftarrow, transform canvas={yshift=-0.9ex}]{r}[yshift=-0.2ex]{\bot}[swap]{y}
&
\Spectra^{\gen G}
\end{tikzcd} \]
at an object $X \in \Spectra^{\gen G}$ is the morphism
\begin{equation}
\label{counit.of.ZQ.into.SpgG}
\EF \tensoring X
\longra
S^0 \tensoring X
\simeq
X~.
\end{equation}
In particular, the full subcategory $i_L(\cI_\ms{F}) \subseteq \Spectra^{\gen G}$ consists of those objects $X \in \Spectra^{\gen G}$ such that the counit morphism \Cref{counit.of.ZQ.into.SpgG} is an equivalence.
\end{observation}

\begin{observation}
\label{obs.can.prestratn.of.SpgG.is.sm}
It follows from \Cref{obs.EFQ.colocalizes.SpgG.into.ZQ} that for any family $\ms{F} \in \Down_{\pos_G}$, the closed subcategory $\cI_\ms{F} \subseteq \Spectra^{\gen G}$ is a closed ideal subcategory (as anticipated by the notation), with symmetric monoidal unit object $\Sigma^\infty_G \EF \simeq i_L(y(\uno_{\Spectra^{\gen G}}))$.  In particular, there exists a factorization
\[ \begin{tikzcd}
\pos
\arrow{rr}{\cI_\bullet}
\arrow[dashed]{rd}
&
&
\Cls_{\Spectra^{\gen G}}
\\
&
\Idl_{\Spectra^{\gen G}}
\arrow[hook]{ru}[sloped, swap]{\ff}
\end{tikzcd}~: \]
the geometric prestratification of $\Spectra^{\gen G}$ is a symmetric monoidal prestratification.
\end{observation}

\begin{observation}
\label{obs.if.X.in.stratum.RH.then.X.equivt.to.X.smash.EdeltaH}
It follows from \Cref{obs.EFQ.colocalizes.SpgG.into.ZQ} that the unit of the adjunction
\[ \begin{tikzcd}[column sep=2cm, row sep=0cm]
\cI_H
\arrow[transform canvas={yshift=0.9ex}]{r}{p_L}
\arrow[hookleftarrow, transform canvas={yshift=-0.9ex}]{r}[yshift=-0.2ex]{\bot}[swap]{\nu}
&
\Spectra^{\gen G}_H
\end{tikzcd} \]
at an object $X \in i_L( \cI_H) \subseteq \Spectra^{\gen G}$ is the morphism
\[
X
\simeq
S^0 \tensoring X
\longra
\wEFltH \tensoring X
~.
\]
Hence, the full subcategory $i_L(\nu(\Spectra^{\gen G}_H)) \subseteq \Spectra^{\gen G}$ consists of those objects $X \in \Spectra^{\gen G}$ such that in the canonical commutative square
\begin{equation}
\label{square.connecting.X.to.smash.with.EdeltaH}
\begin{tikzcd}[row sep=1.5cm]
\EFH \tensoring X
\arrow{r}
\arrow{d}
&
X
\arrow{d}
\\
( \EFH \wedge \wEFltH ) \tensoring X
\arrow{r}
&
\wEFltH \tensoring X
\end{tikzcd}
\end{equation}
the upper and right morphisms are equivalences. In turn, this is the case if and only if the square \Cref{square.connecting.X.to.smash.with.EdeltaH} consists entirely of equivalences.
\end{observation}

\begin{notation}
For brevity, we write
\[ \sE \delta_H := ( \EFH \wedge \wEFltH ) \in \Spaces^{\gen G}_*~. \]
This notation is motivated by the Dirac delta function: this pointed genuine $G$-space is characterized by the fact that
\[
( \sE \delta_H )^K
\simeq
\left\{
\begin{array}{ll}
S^0~,
&
K = H
\\
\pt~,
&
K \not= H
\end{array}
\right.
~.
\]
\end{notation}

\begin{observation}
The object $\sE \delta_H \in \Spaces^{\gen G}_*$ is idempotent with respect to the smash product. 
\end{observation}

\begin{notation}
We define the family
\[ (^{\not\geq} H ) := \{ K \in \pos_G : K \not\geq H \} \in \Down_{\pos_G} ~. \]
\end{notation}

\begin{definition}
\label{defn.geom.H.fps}
The \bit{geometric $H$-fixedpoints} functor
\[
\Spectra^{\gen G}
\xra{\Phi^H_\gen}
\Spectra^{\gen \Weyl(H)}
\]
is defined by the formula
\[
\Phi_\gen^H(X)
:=
( \wEFgeomHfps \tensoring X)^H
~.
\]
We will be primarily interested in the composite
\[
\Phi^H
:
\Spectra^{\gen G}
\xra{\Phi^H_\gen}
\Spectra^{\gen \Weyl(H)}
\xlongra{U}
\Spectra^{\htpy \Weyl(H)}
~, \]
which we refer to by the same name.
\end{definition}

\begin{remark}
For any normal closed subgroup $H \leq G$, there is a canonical commutative diagram
\begin{equation}
\label{align.localized.G.spt.with.G.mod.H.spt}
\begin{tikzcd}[row sep=1.5cm, column sep=1.5cm]
\Spectra^{\gen G}
\arrow{r}{p_L}
\arrow{rd}[sloped, swap]{\Phi^H_\gen}
&
\Spectra^{\gen G} / \cI_{\ms{F}_{^{\not\geq} H}}
\arrow[leftrightarrow]{d}[sloped, anchor=south]{\sim}
\\
&
\Spectra^{\gen (G/H)}
\end{tikzcd}~.
\end{equation}
Recall from \Cref{prop.quotient.stratn} \and \Cref{obs.quotient.stratn.from.down.closed.over.smaller.poset} that the geometric stratification of $\Spectra^{\gen G}$ over $\pos_G$ determines a quotient stratification of $\Spectra^{\gen G} / \cI_{\ms{F}_{^{\not\geq} H}}$ over $\pos_G \backslash (^{\not\geq} H)$. Under the equivalence in diagram \Cref{align.localized.G.spt.with.G.mod.H.spt}, this corresponds to the geometric stratification of $\Spectra^{\gen (G/H)}$ over $\pos_{G/H}$ (recall the third isomorphism theorem).
\end{remark}

\begin{observation}
\label{obs.geom.fixedpoints.by.any.other.name}
One may also define the functor $\Phi^H$ (but not the functor $\Phi^H_\gen$) using the family $(^< H) \in \Down_{\pos_G}$, in fact using any family $\ms{F} \in \Down_{\pos_G}$ that does not contain $H$ and such that moreover $H \in \pos_G \backslash \ms{F}$ is a minimal element. Namely, for any such family we have a canonical equivalence
\[
\Phi^H (- )
\simeq
( \wEF \tensoring (-))^H
\]
in $\Fun ( \Spectra^{\gen G} , \Spectra^{\htpy \Weyl(H)})$.

To explain this, observe that $(^< H) \in \Down_{\pos_G}$ is the initial such family, so that for any such family $\ms{F} \in \Down_{\pos_G}$ we have a canonical morphism
\begin{equation}
\label{morphism.between.Etildes.for.geomfps}
\w{\sE} \ms{F}_{^< H}
\longra
\wEF
\end{equation}
in $\Spaces^{\gen G}_*$ determined by the inclusion $(^< H) \subseteq \ms{F}$. Then, we claim that for any $X \in \Spectra^{\gen G}$, the composite functor
\begin{equation}
\label{composite.taking.any.morphism.between.suitable.Etildes.for.geom.Hfps.to.equivce}
\Spaces^{\gen G}_*
\xra{\Sigma^\infty_G}
\Spectra^{\gen G}
\xra{(-) \otimes X}
\Spectra^{\gen G}
\xra{(-)^H}
\Spectra^{\gen \Weyl(H)}
\xlongra{U}
\Spectra^{\htpy \Weyl(H)}
\end{equation}
carries the morphism \Cref{morphism.between.Etildes.for.geomfps} to an equivalence (although its truncation ending at $\Spectra^{\gen \Weyl(H)}$ does not generally do so). Indeed, this follows from the fact that in the commutative diagram
\[ \begin{tikzcd}
\Spectra^{\gen G}
\arrow{r}{(-)^H}
\arrow{d}[swap]{\Res^G_H}
&
\Spectra^{\gen \Weyl(H)}
\arrow{r}{U}
&
\Spectra^{\htpy \Weyl(H)}
\arrow{d}{\fgt}
\\
\Spectra^{\gen H}
\arrow{rr}[swap]{(-)^H}
&
&
\Spectra
\end{tikzcd}~, \]
the left vertical functor is symmetric monoidal and carries the morphism $\Sigma^\infty_G \Cref{morphism.between.Etildes.for.geomfps}$ to an equivalence while the right vertical functor is conservative. Hence, the composite \Cref{composite.taking.any.morphism.between.suitable.Etildes.for.geom.Hfps.to.equivce} carries the span
\[
\wEFgeomHfps
\longla
\w{\sE} \ms{F}_{^< H}
\longra
\wEF
\]
in $\Spaces^{\gen G}_*$ to a span of equivalences.
\end{observation}

\begin{observation}
Geometric fixedpoints functors compose: if $K \leq H  \leq \Normzer_G(K) \leq G$ then the triangle
\[ \begin{tikzcd}[row sep=1.5cm, column sep=1.5cm]
\Spectra^{\gen G}
\arrow{r}{\Phi^K_\gen}
\arrow{rd}[sloped, swap]{\Phi^H_\gen}
&
\Spectra^{\gen \Weyl_G(K)}
\arrow{d}{\Phi^{H/K}_\gen}
\\
&
\Spectra^{\gen \Weyl_G(H)}
\end{tikzcd} \]
commutes. We use this fact without further comment.
\end{observation}

\begin{observation}
\label{obs.geom.fps.is.sm}
The geometric $H$-fixedpoints functor
\[
\Spectra^{\gen G}
\xra{\Phi^H}
\Spectra^{\htpy \Weyl(H)}
\]
is symmetric monoidal.
\end{observation}

\begin{observation}
\label{obs.geom.fps.commutes.w.suspension}
The geometric $H$-fixedpoints functor fits into a canonical commutative diagram
\[ \begin{tikzcd}[row sep=1.5cm, column sep=1.5cm]
\Spaces^{\gen G}_*
\arrow{r}{\Sigma^\infty_G}
\arrow{d}[swap]{(-)^H}
&
\Spectra^{\gen G}
\arrow{d}{\Phi^H}
\\
\Spaces^{\htpy \Weyl(H)}_*
\arrow{r}[swap]{\Sigma^\infty}
&
\Spectra^{\htpy \Weyl(H)}
\end{tikzcd}~. \]
\end{observation}

\begin{observation}
\label{obs.geomHfps.smashes.to.localizer}
There is a unique nonzero morphism
\[
\sE \delta_H
\longra
\wEFgeomHfps
\]
in $\Spaces^{\gen G}_*$, and it becomes an equivalence
\[
\sE \delta_H
\simeq
\sE \delta_H
\wedge
\sE \delta_H
\xlongra{\sim}
\wEFgeomHfps
\wedge
\sE \delta_H
\]
upon smashing it with its source.
\end{observation}

\begin{theorem}
\label{thm.geom.stratn.of.SpgG}
The geometric prestratification of $\Spectra^{\gen G}$ over $\pos_G$ is a symmetric monoidal stratification.  Moreover,
\begin{enumerate}
\item\label{item.identify.Hth.stratum.of.canonical.stratn} its $H\th$ stratum is the $\infty$-category
\[ \Spectra^{\gen G}_H \simeq \Spectra^{\htpy \Weyl(H)} \]
of homotopy $\Weyl(H)$-spectra, and
\item\label{item.identify.Hth.geom.localizn.functor.of.canonical.stratn} its $H\th$ geometric localization functor is the geometric $H$-fixedpoints functor
\[
\Spectra^{\gen G}
\xra{\Phi^H}
\Spectra^{\htpy \Weyl(H)}
\simeq
\Spectra^{\gen G}_H
~.
\]
\end{enumerate}
\end{theorem}

\begin{proof}
Applying \Cref{obs.can.prestratn.of.SpgG.is.sm}, we see that it suffices to show that the geometric prestratification is a stratification. For this, we first verify the two asserted identifications for the geometric prestratification of $\Spectra^{\gen G}$ over $\pos_G$, and then we use these identifications to verify that it is indeed a stratification.

Towards verifying the two identifications, for any $X \in \Spectra^{\gen G}$, referring to the functors in the diagram
\[ \begin{tikzcd}[column sep=2cm, row sep=0cm]
\Spectra^{\gen G}
\arrow[hookleftarrow, transform canvas={yshift=0.9ex}]{r}{i_L}
\arrow[transform canvas={yshift=-0.9ex}]{r}[yshift=-0.2ex]{\bot}[swap]{y}
&
\cI_H
\arrow[transform canvas={yshift=0.9ex}]{r}{p_L}
\arrow[hookleftarrow, transform canvas={yshift=-0.9ex}]{r}[yshift=-0.2ex]{\bot}[swap]{\nu}
&
\Spectra^{\gen G}_H
\end{tikzcd} \]
we compute in $\Spectra^{\htpy \Weyl(H)}$ that
\begin{align}
\label{identify.H.fps.as.H.geomfps.step.one}
( i_L \nu p_L y (X) )^H
&\simeq
( \sE \delta_H \tensoring X )^H
\\
\label{identify.H.fps.as.H.geomfps.step.two}
&\simeq
( (\wEFgeomHfps \wedge \sE \delta_H) \tensoring X )^H
\\
\nonumber
&\simeq
( \wEFgeomHfps \tensoring ( \sE \delta_H \tensoring X ))^H
\\
\nonumber
&:=
\Phi^H ( \sE \delta_H \tensoring X )
\\
\label{identify.H.fps.as.H.geomfps.step.four}
& \simeq
\Phi^H ( \Sigma^\infty_G ( \sE \delta_H) ) \otimes \Phi^H(X)
\\
\label{identify.H.fps.as.H.geomfps.step.five}
& \simeq
\Sigma^\infty ( ( \sE \delta_H)^H ) \otimes \Phi^H (X)
\\
\label{identify.H.fps.as.H.geomfps.step.six}
& \simeq
\Phi^H(X)
~,
\end{align}
where
\begin{itemize}
\item equivalence \Cref{identify.H.fps.as.H.geomfps.step.one} follows from \Cref{obs.if.X.in.stratum.RH.then.X.equivt.to.X.smash.EdeltaH},
\item equivalence \Cref{identify.H.fps.as.H.geomfps.step.two} follows from \Cref{obs.geomHfps.smashes.to.localizer},
\item equivalence \Cref{identify.H.fps.as.H.geomfps.step.four} follows from \Cref{obs.geom.fps.is.sm},
\item equivalence \Cref{identify.H.fps.as.H.geomfps.step.five} follows from \Cref{obs.geom.fps.commutes.w.suspension}, and
\item equivalence \Cref{identify.H.fps.as.H.geomfps.step.six} follows from the equivalence $(\sE\delta_H)^H \simeq S^0$ in $\Spaces^{\htpy \Weyl(H)}_*$.
\end{itemize}
Now, to verify part \Cref{item.identify.Hth.stratum.of.canonical.stratn}, we begin by observing via the recollement
\[ \begin{tikzcd}[column sep=1.5cm]
\cI_{^< H}
\arrow[hook, bend left=45]{r}[description]{i_L}
\arrow[leftarrow]{r}[transform canvas={yshift=0.1cm}]{\bot}[swap,transform canvas={yshift=-0.1cm}]{\bot}[description]{\yo}
\arrow[bend right=45, hook]{r}[description]{i_R}
&
\cI_H
\arrow[bend left=40]{r}[description]{p_L}
\arrow[hookleftarrow]{r}[transform canvas={yshift=0.1cm}]{\bot}[swap,transform canvas={yshift=-0.1cm}]{\bot}[description]{\nu}
\arrow[bend right=40]{r}[description]{p_R}
&
\Spectra^{\gen G}_H
\end{tikzcd} \]
that the object $p_L(\Sigma^\infty_G(G/H)_+) \in \Spectra^{\gen G}_H$ is a compact generator, so that it suffices to verify that the composite morphism
\begin{equation}
\label{composite.morphism.from.susp.spectrum.of.Weyl.group.to.endomorphisms.in.XH}
\begin{aligned}
\Sigma^\infty_+ \Weyl(H)
& \simeq
\Sigma^\infty \End_{\Spaces^{\gen G}_*}( ( G/H)_+ )
\\
& \xra{\Sigma^\infty_G}
\ulEnd_{\Spectra^{\gen G}}(\Sigma^\infty_G(G/H)_+)
\\
& \underset{\sim}{\xlongra{y}}
\ulEnd_{\cI_H}(\Sigma^\infty_G(G/H)_+)
\\
& \xra{p_L}
\ulEnd_{\Spectra^{\gen G}_H}(p_L(\Sigma^\infty_G(G/H)_+))
\end{aligned}
\end{equation}

of ring spectra is an equivalence.  For this, by adjunction we compute that
\begin{align}
\nonumber
\ulEnd_{\Spectra^{\gen G}_H}(p_L(\Sigma^\infty_G(G/H)_+))
&:=
\ulhom_{\Spectra^{\gen G}_H}( p_L(\Sigma^\infty_G(G/H)_+) , p_L(\Sigma^\infty_G(G/H)_+) )
\\
\nonumber
&\simeq
\ulhom_{\cI_H}(\Sigma^\infty_G(G/H)_+ , \nu p_L(\Sigma^\infty_G(G/H)_+) )
\\
\nonumber
&\simeq
\ulhom_{\Spectra^{\gen G}} ( \Sigma^\infty_G(G/H)_+ , i_L \nu p_L ( \Sigma^\infty_G(G/H)_+ ) )
\\
\nonumber
&\simeq
( i_L \nu p_L ( \Sigma^\infty_G(G/H)_+ ) )^H
\\
\label{get.to.Weyl.step.one}
&\simeq
\Phi^H ( \Sigma^\infty_G(G/H)_+ )
\\
\label{get.to.Weyl.step.two}
&\simeq
\Sigma^\infty ( ( (G/H)_+ )^H )
\\
\nonumber
&\simeq
\Sigma^\infty ( \hom_{\Spaces^{\gen G}}(G/H,G/H)_+ )
\\
\nonumber
&\simeq
\Sigma^\infty ( \Weyl(H)_+ )
~,
\end{align}
where equivalence \Cref{get.to.Weyl.step.one} follows from the equivalences \Cref{identify.H.fps.as.H.geomfps.step.one}-\Cref{identify.H.fps.as.H.geomfps.step.six} and equivalence \Cref{get.to.Weyl.step.two} follows from \Cref{obs.geom.fps.commutes.w.suspension}.  This string of equivalences of spectra evidently underlies the composite morphism \Cref{composite.morphism.from.susp.spectrum.of.Weyl.group.to.endomorphisms.in.XH} of ring spectra, which proves part \Cref{item.identify.Hth.stratum.of.canonical.stratn}.  To verify part \Cref{item.identify.Hth.geom.localizn.functor.of.canonical.stratn}, we compute for any $X \in \Spectra^{\gen G}$ that
\begin{align}
\nonumber
\ulhom_{\Spectra^{\gen G}_H} ( p_L ( \Sigma^\infty_G(G/H)_+ ) , p_L y (X) )
&\simeq
\ulhom_{\cI_H} ( \Sigma^\infty_G ( G/H)_+ , \nu p_L y (X) )
\\
\nonumber
&\simeq
\ulhom_{\Spectra^{\gen G}} ( \Sigma^\infty_G ( G/H)_+ , i_L \nu p_L y (X) )
\\
\nonumber
&\simeq
( i_L \nu p_L y (X) )^H
\\
\label{get.to.H.geom.fps.of.X.step.two}
&\simeq
\Phi^H(X)~,
\end{align}
where equivalence \Cref{get.to.H.geom.fps.of.X.step.two} follows from equivalences \Cref{identify.H.fps.as.H.geomfps.step.one}-\Cref{identify.H.fps.as.H.geomfps.step.six}.

We now verify that the geometric prestratification of $\Spectra^{\gen G}$ over $\pos_G$ is indeed a stratification. Using Observations \ref{obs.yo.comm.for.s.m.stratns} \and \ref{obs.EFQ.colocalizes.SpgG.into.ZQ}, it suffices to observe that for any $\sD,\sD' \in \Down_{\pos_G}$ the morphism
\[
(\sE \ms{F}_\sD)_+
\wedge
(\sE \ms{F}_{\sD'})_+
\wedge
(\sE \ms{F}_{\sD \cap \sD'})_+
\longra
(\sE \ms{F}_\sD)_+
\wedge
(\sE \ms{F}_{\sD'})_+
\]
in $\Spaces^{\gen G}_*$ is an equivalence.
\end{proof}



\subsection{The proper Tate construction}
\label{subsection.tate.constrn}

In this brief subsection we discuss the gluing functors of the geometric stratification of genuine $G$-spectra, which are versions of the Tate construction.



\begin{observation}
\label{obs.gluing.functors.for.SpgG}
By definition, the $H\th$ geometric localization functor of the geometric stratification of genuine $G$-spectra is the left adjoint in the composite adjunction
\[ \begin{tikzcd}[column sep=2cm, row sep=0cm]
\Phi^H
:
\Spectra^{\gen G}
\arrow[transform canvas={yshift=0.9ex}]{r}{\Res^G_{\Normzer (H)}}
\arrow[leftarrow, transform canvas={yshift=-0.9ex}]{r}[yshift=-0.2ex]{\bot}[swap]{\coInd^G_{\Normzer (H)}}
&
\Spectra^{\gen \Normzer (H)}
\arrow[transform canvas={yshift=0.9ex}]{r}{\Phi^H_\gen}
\arrow[hookleftarrow, transform canvas={yshift=-0.9ex}]{r}[yshift=-0.2ex]{\bot}[swap]{\rho^H_\gen}
&
\Spectra^{\gen \Weyl (H)}
\arrow[transform canvas={yshift=0.9ex}]{r}{U_{\Weyl(H)}}
\arrow[hookleftarrow, transform canvas={yshift=-0.9ex}]{r}[yshift=-0.2ex]{\bot}[swap]{\beta_{\Weyl(H)}}
&
\Spectra^{\htpy \Weyl (H)}
:
\rho^H
\end{tikzcd}
~.
 \]
It follows that for any $H \leq K$ in $\pos_G$, the gluing functor $\Gamma^H_K$ is the composite
\begin{equation}
\label{gluing.functor.for.SpgG.in.general}
\begin{tikzcd}[column sep=2cm, row sep=0.5cm]
&
\Spectra^{\gen \Normzer (K)}
\arrow{r}{\Phi^K_\gen}
&
\Spectra^{\gen \Weyl (K)}
\arrow{r}{U_{\Weyl(K)}}
&
\Spectra^{\htpy \Weyl (K)}
\\
\Spectra^{\gen G}
\arrow{ru}[sloped]{\Res^G_{\Normzer (K)}}
\\
&
\Spectra^{\gen \Normzer (H)}
\arrow{lu}[sloped, swap]{\coInd^G_{\Normzer (H)}}
\arrow[hookleftarrow]{r}[swap]{\rho^H_\gen}
&
\Spectra^{\gen \Weyl (H)}
\arrow[hookleftarrow]{r}[swap]{\beta_{\Weyl(H)}}
&
\Spectra^{\htpy \Weyl (H)}
\arrow[dashed]{uu}[swap]{\Gamma^H_K}
\end{tikzcd}~.
\end{equation}
When $H$ and $K$ are both normal subgroups of $G$ (which is automatic when $G$ is abelian), then the composite \Cref{gluing.functor.for.SpgG.in.general} reduces to the composite
\[
\begin{tikzcd}[row sep=1.5cm, column sep=1.5cm]
\Spectra^{\gen ( G/K) }
\arrow{r}{U_{G/K}}
&
\Spectra^{\htpy ( G/K)}
\\
\Spectra^{\gen (G/H)}
\arrow[hookleftarrow]{r}[swap]{\beta_{G/H}}
\arrow{u}{\Phi^{K/H}_\gen}
&
\Spectra^{\htpy ( G/H)}
\arrow[dashed]{u}[swap]{\Gamma^H_K}
\end{tikzcd}~.
\]
\end{observation}

\begin{observation}
\label{H.to.K.gluing.is.zero.when.K.not.leq.NH}
The subcomposite
\[
\Spectra^{\gen \Normzer (H)}
\xra{\coInd_{\Normzer(H)}^G}
\Spectra^{\gen G}
\xra{\Res^G_{\Normzer(K)}}
\Spectra^{\gen \Normzer(K)}
\xra{\Phi^K_\gen}
\Spectra^{\gen \Weyl (K)}
\]
of the composite \Cref{gluing.functor.for.SpgG.in.general} is zero whenever $\Normzer (H) \not\geq K$.
\end{observation}

\begin{definition}
\label{defn.proper.H.tate.in.terms.of.genuine}
We define the \bit{proper $H$-Tate construction} to be the composite functor
\[ \begin{tikzcd}
(-)^{\tate H}
:
\Spectra^{\htpy G}
\arrow[hook]{r}{\beta}
&
\Spectra^{\gen G}
\arrow{r}{\Phi^H_\gen}
\arrow[bend right=15]{rr}[swap]{\Phi^H}
&
\Spectra^{\gen \Weyl(H)}
\arrow{r}{U}
&
\Spectra^{\htpy \Weyl(H)}
\end{tikzcd}~. \]
\end{definition}


\begin{remark}
\label{rmk.genzd.tate.reduces.to.tate.Cp}
We make \Cref{defn.proper.H.tate.in.terms.of.genuine} here in the interest of self-containment, but in fact the proper $H$-Tate construction
\[
\Spectra^{\htpy G}
\xra{(-)^{\tate H}}
\Spectra^{\htpy \Weyl(H)}
\]
admits a description making no reference to genuine equivariant homotopy theory, at least assuming that $G$ is finite. Namely, we prove as \cite[Proposition \ref{mackey:prop.proper.Tate.from.genuine.G.objs}]{AMR-mackey} that it is given by quotienting by norms from all proper subgroups of $H$: it is the lower composite in the left Kan extension diagram
\[ \begin{tikzcd}[row sep=1.5cm, column sep=1.5cm]
\Spectra^{\htpy G}
\arrow{r}{(-)^{\htpy H}}[swap, xshift=-0.8cm, yshift=-0.8cm]{\rotatebox{55}{$\Leftarrow$}}
\arrow{d}[swap]{p}
&
\Spectra^{\htpy \Weyl(H)}
\\
\Spectra^{\htpy G} /^\St \cI
\arrow[dashed]{ru}
\end{tikzcd}~, \]
where $p$ denotes the projection to the stable quotient by the thick ideal subcategory $\cI \subseteq \Spectra^{\htpy G}$ generated by the objects $\{ \Sigma^\infty(G/K)_+ \in \Spectra^{\htpy G} \}_{K < H}$.
In particular, when $G=H=\Cyclic_p$ for a prime $p$, this recovers the ordinary Tate construction
\[
(-)^{\tate \Cyclic_p}
\simeq
(-)^{\st \Cyclic_p}
:=
\cofib \left(
(-)_{\htpy \Cyclic_p}
\xra{\Nm_{\Cyclic_p}}
(-)^{\htpy \Cyclic_p}
\right)
~.
\]
\end{remark}

\begin{remark}
\label{rmk.nonabelian.monodromy}
Assuming that $G$ is finite, as \cite[Theorem \ref{mackey:intro.thm.gluing.functors}]{AMR-mackey} we identify the gluing functor
\[
\Spectra^{\htpy \Weyl(H)}
\xra{\Gamma^H_K}
\Spectra^{\htpy \Weyl(K)}
\]
for any $H \leq K$ in $\pos_G$: writing
\[
\tilde{C}(H,K)
:=
\{ g \in G : H \subseteq gKg^{-1} \subseteq \Normzer(H) \}
~,
\]
it is given by the formula
\[
E
\longmapsto
\bigoplus_{[g] \in \Normzer(H) \backslash \tilde{C}(H,K) / \Normzer(K)}
\Ind_{(\Normzer(H) \cap ( \Normzer(gKg^{-1} ) ) )/(gKg^{-1})}^{\Weyl(K)} E^{\tate (gKg^{-1})/H}
~.
\]
(In particular, by \Cref{rmk.genzd.tate.reduces.to.tate.Cp} this description also makes no reference to genuine equivariant homotopy theory.)
\end{remark}





\begin{observation}
\label{obs.NS.proper.tate.is.p.primary}
We record here the following arithmetic facts surrounding the proper Tate construction, which we use in \Cref{subsection.examples.of.SpgG}.
\begin{enumerate}

\item\label{vanishes.if.not.prime.power.order}
By \cite[Lemma II.6.7]{NS}, if $G$ is a finite group whose order is not a prime power, then the proper Tate construction $(-)^{\tate G}$ vanishes.

\item\label{Cp.tate.vanishes.if.p.acts.invertibly}
By \cite[Lemma I.2.8]{NS}, if $E \in \Spectra^{\htpy \Cyclic_p}$ and $p$ acts invertibly on $\pi_n E$ for all $n \in \ZZ$, then $E^{\tate \Cyclic_p} \simeq 0$.

\item\label{Cp.tate.is.p.complete}
By the Segal conjecture \cite{Lin-Segal,AGM-Segal} combined with \cite[Theorem I.3.1]{NS}, for any $E \in \Spectra^{\htpy \Cyclic_p}$ the spectrum $E^{\tate \Cyclic_p} \in \Spectra$ is $p$-complete.
\end{enumerate}
\end{observation}

\begin{warning}
In \cite{AMR-trace}, for brevity we omit the word ``proper'' from the terminology ``proper Tate construction''.
\end{warning}


\subsection{Examples of reconstruction of genuine $G$-spectra}
\label{subsection.examples.of.SpgG}

In this subsection, we give a number of examples of reconstruction (via \Cref{macrocosm.thm}) that follow from the geometric stratification of genuine $G$-spectra (\Cref{thm.geom.stratn.of.SpgG}).  It is straightforward but notationally cumbersome to describe the symmetric monoidal structures (which result from \Cref{thm.s.m.reconstrn}), and so we omit them from the present discussion.

\begin{local}
In this subsection, in the interest of uniformity, even in the case that $G$ is the trivial group we may include the forgetful functor
\[
\Spectra^{\gen G}
\xlongra{U}
\Spectra^{\htpy G}
\]
in our notation.
\end{local}

\begin{remark}
In this subsection, we continue to distinguish between the two geometric $H$-fixedpoints functors appearing in the commutative diagram
\[ \begin{tikzcd}
\Spectra^{\gen G}
\arrow{r}{\Phi^H_\gen}
\arrow{rd}[sloped, swap]{\Phi^H}
&
\Spectra^{\gen \Weyl(H)}
\arrow{d}{U}
\\
&
\Spectra^{\htpy \Weyl(H)}
\end{tikzcd}~, \]
as introduced in \Cref{defn.geom.H.fps}.  The $H\th$ geometric localization functor for the geometric stratification of genuine $G$-spectra is the functor $\Phi^H$, but we also use its identification as the composite $U \Phi^H_\gen$ in order to describe the structure maps in the right-lax limit (as first indicated in \Cref{rmk.lax.limits.have.structure.maps}), which are given by the unit maps of various adjunctions of the form $U \adj \beta$.
\end{remark}

\begin{notation}
\label{notation.SphWbullet}
We write
\[
\begin{tikzcd}[column sep=1.5cm]
\pos_G
\arrow{r}[description, yshift=-0.05cm]{\llax}{\Spectra^{\htpy \Weyl_G(\bullet)}}
&
\PrSt
\end{tikzcd}
\]
for the gluing diagram of the geometric stratification of genuine $G$-spectra.
\end{notation}

\begin{example}[genuine $\Cyclic_p$-spectra]
\label{example.genuine.Cp.spectra}
Let $\Cyclic_p$ denote the cyclic group of order $p$, where $p$ is a prime.  Its poset of conjugacy classes of closed subgroups is
\[
\pos_{\Cyclic_p}
=
\left\{
e
\longra
\Cyclic_p
\right\}
~.
\]
Theorems \ref{thm.geom.stratn.of.SpgG} \and \ref{macrocosm.thm} provide an equivalence
\begin{equation}
\label{reconstrn.of.Sp.g.Cp}
\hspace{-1cm}
\begin{aligned}
\Spectra^{\gen \Cyclic_p}
& \underset{\sim}{\xlongra{\gd}}
\lim^\rlax_{\llax.\pos_{\Cyclic_p}} \left( \Spectra^{\htpy \Weyl_{\Cyclic_p}(\bullet)} \right)
:=
\lim^\rlax \left( \Spectra^{\htpy \Cyclic_p} \xra{(-)^{\tate \Cyclic_p}} \Spectra \right)
\\
& := 
\left\{ 
\left(
E_0 \in \Spectra^{\htpy \Cyclic_p}
~,~
E_1 \in \Spectra
~,
\begin{tikzcd}
E_1
\arrow{d}
\\
(E_0)^{\tate \Cyclic_p}
\end{tikzcd} 
\right)
\right\}
\end{aligned}
~.
\end{equation}
Via the equivalence \Cref{reconstrn.of.Sp.g.Cp}, a genuine $\Cyclic_p$-spectrum $E \in \Spectra^{\gen \Cyclic_p}$ is specified by the data of
\begin{itemize}
\item its underlying homotopy $\Cyclic_p$-spectrum
\[
E_0
:=
UE
\in
\Spectra^{\htpy \Cyclic_p}
~,
\]
\item its geometric $\Cyclic_p$-fixedpoints spectrum
\[
E_1
:=
\Phi^{\Cyclic_p} E
:=
U \Phi^{\Cyclic_p}_\gen E
\in
\Spectra
~,
\]
and
\item the gluing data of a comparison map
\[
U \Phi^{\Cyclic_p}_\gen
\left(
E
\longra
\beta U E
\right)
=:
\left(
E_1
\longra
(E_0)^{\tate \Cyclic_p}
\right)
\]
from $E_1$ to the $\Cyclic_p$-Tate construction of $E_0$ (recall \Cref{rmk.genzd.tate.reduces.to.tate.Cp}).
\end{itemize}
In other words, we have a recollement
\begin{equation}
\label{recollement.of.Sp.g.Cp}
\begin{tikzcd}[column sep=1.5cm]
\Spectra^{\htpy \Cyclic_p}
\arrow[hook, bend left=45]{r}[description]{i_L}
\arrow[leftarrow]{r}[transform canvas={yshift=0.2cm}]{\bot}[swap,transform canvas={yshift=-0.2cm}]{\bot}[description]{U}
\arrow[bend right=45, hook]{r}[description]{\beta}
&
\Spectra^{\gen \Cyclic_p}
\arrow[bend left=45]{r}[description, pos=0.55]{\Phi^{\Cyclic_p}}
\arrow[hookleftarrow]{r}[transform canvas={yshift=0.2cm}]{\bot}[swap,transform canvas={yshift=-0.2cm}]{\bot}[description]{\rho^{\Cyclic_p}}
\arrow[bend right=45]{r}[description, pos=0.55]{p_R}
&
\Spectra
\end{tikzcd}~.
\end{equation}
\end{example}

\begin{remark}
It is not hard to see the Wirthm\"{u}ller isomorphism $\Ind_e^{\Cyclic_p} \simeq \coInd_e^{\Cyclic_p}$ within the context of \Cref{example.genuine.Cp.spectra}.  Indeed, writing $\pt \simeq \sB e \xra{\iota} \BC_p$ for the canonical basepoint, the adjoint functors $\Ind_e^{\Cyclic_p} \adj \Res_e^{\Cyclic_p} \adj \coInd_e^{\Cyclic_p}$ are obtained as the horizontal composites in the diagram
\[
\begin{tikzcd}[column sep=1.5cm]
\Spectra
\arrow[bend left=45]{r}[description, pos=0.45]{\iota_!}
\arrow[leftarrow]{r}[transform canvas={yshift=0.2cm}]{\bot}[swap,transform canvas={yshift=-0.2cm}]{\bot}[description, pos=0.55]{\iota^*}
\arrow[bend right=45]{r}[description, pos=0.45]{\iota_*}
&
\Spectra^{\htpy \Cyclic_p}
\arrow[hook, bend left=45]{r}[description]{i_L}
\arrow[leftarrow]{r}[transform canvas={yshift=0.2cm}]{\bot}[swap,transform canvas={yshift=-0.2cm}]{\bot}[description]{U}
\arrow[bend right=45, hook]{r}[description]{\beta}
&
\Spectra^{\gen \Cyclic_p}
\end{tikzcd}~.
\]
Note that for any $E \in \Spectra$ we have
\[
\iota_!(E)
\simeq
\coprod_{\Cyclic_p/e} E
\simeq
\bigoplus_{\Cyclic_p/e} E
\simeq
\prod_{\Cyclic_p/e} E
\simeq
\iota_*(E)
~:
\]
both adjoints to the forgetful functor $\iota^*$ are given by inducing up from $e$ to $\Cyclic_p$.  On the other hand, in the recollement \Cref{recollement.of.Sp.g.Cp}, we see that for any $E \in \Spectra^{\htpy \Cyclic_p}$ we have
\[
i_L(E)
=
( E \longmapsto E^{\tate \Cyclic_p} \longla 0)
\qquad
\text{and}
\qquad
\beta(E)
=
( E \longmapsto E^{\tate \Cyclic_p} \xlongla{\sim} E^{\tate \Cyclic_p} )
\]
(via the identification of \Cref{lem.reconstrn.for.recollement}).  Hence, the equivalence
\[
\Ind_e^{\Cyclic_p}
:=
i_L \iota_!
\simeq
\beta \iota_*
=:
\coInd_e^{\Cyclic_p}
\]
follows from the fact that the $\Cyclic_p$-Tate construction vanishes on any homotopy $\Cyclic_p$-spectra that are induced from the proper subgroup $e < \Cyclic_p$.
\end{remark}

\begin{example}[genuine $\Cyclic_{p^2}$-spectra]
\label{example.genuine.Cpsquared.spectra}
Let $\Cyclic_{p^2}$ denote the cyclic group of order $p^2$, where $p$ is a prime.  Its poset of conjugacy classes of closed subgroups is
\[
\pos_{\Cyclic_{p^2}}
=
\left\{
\begin{tikzcd}[column sep=0.5cm]
&
\Cyclic_p
\arrow{rd}
\\
e
\arrow{ru}
\arrow{rr}
&
&
\Cyclic_{p^2}
\end{tikzcd}
\right\}
~. \]
Theorems \ref{thm.geom.stratn.of.SpgG} \and \ref{macrocosm.thm} provide an equivalence
\begin{equation}
\label{gen.Cpsquared.spectra.as.rlax.lim}
\Spectra^{\gen \Cyclic_{p^2}}
\underset{\sim}{\xlongra{\gd}}
\lim^\rlax_{\llax.\pos_{\Cyclic_{p^2}}} \left( \Spectra^{\htpy \Weyl_{\Cyclic_{p^2}}(\bullet)} \right)
:=
\lim^\rlax
\left(
\begin{tikzcd}[row sep=1.5cm]
&
\Spectra^{\htpy \Cyclic_p}
\arrow{rd}[sloped]{(-)^{\tate \Cyclic_p}}
\\
\Spectra^{\htpy \Cyclic_{p^2}}
\arrow{ru}[sloped]{(-)^{\tate \Cyclic_p}}
\arrow{rr}[transform canvas={yshift=0.6cm}]{\rotatebox{90}{$\Rightarrow$}}[swap]{(-)^{\tate \Cyclic_{p^2}}}
&
&
\Spectra
\end{tikzcd}
\right)
~.
\end{equation}
Via the equivalence \Cref{gen.Cpsquared.spectra.as.rlax.lim}, a genuine $\Cyclic_{p^2}$-spectrum $E \in \Spectra^{\gen \Cyclic_{p^2}}$ is specified by the following data, which is precisely that of an object of this right-lax limit.\footnote{Right-lax limits of left-lax left $[2]$-modules are described in \Cref{example.limits.of.lax.actions.when.laxness.disagrees}\Cref{example.limits.of.lax.actions.when.laxness.disagrees.rlax.lim.of.llax.action}. See also \Cref{ex.gluing.stuff.over.brax.two} for a discussion of macrocosm and microcosm reconstruction for an arbitrary stratification over $[2]$.}
\begin{itemize}

\item First of all, it determines the objects
\begin{itemize}

\item $E_0 := U E \in \Spectra^{\htpy \Cyclic_{p^2}}$,

\item $E_1 := \Phi^{\Cyclic_p} E := U \Phi^{\Cyclic_p}_\gen E \in \Spectra^{\htpy \Cyclic_p}$, and

\item $E_2 := \Phi^{\Cyclic_{p^2}} E := U \Phi^{\Cyclic_{p^2}}_\gen E \in \Spectra$,

\end{itemize}
the homotopy-equivariant spectra underlying the genuine-equivariant spectra which are its geometric fixedpoints with respect to the various subgroups of $\Cyclic_{p^2}$.

\item Thereafter, the unit maps of various adjunctions of the form $U \adj \beta$ yield
\begin{itemize}
\item a map
\begin{equation}
\label{str.map.one.to.two}
U \Phi^{\Cyclic_p}_\gen
\left(
E
\longra
\beta U E
\right)
=:
\left(
E_1
\longra
\left( U \Phi^{\Cyclic_p}_\gen \beta \right)
E_0
\right)
=:
\left(
E_1
\longra
( E_0 )^{\tate \Cyclic_p}
\right)
\end{equation}
in $\Spectra^{\htpy \Cyclic_p}$,
\item a map
\begin{equation}
\label{str.map.zero.to.two}
U \Phi^{\Cyclic_{p^2}}_\gen
\left(
E
\longra
\beta U E
\right)
=:
\left(
E_2
\longra
\left( U \Phi^{\Cyclic_{p^2}}_\gen \beta \right) E_0
\right)
=:
\left(
E_2
\longra
(E_0)^{\tate \Cyclic_{p^2}}
\right)
\end{equation}
in $\Spectra$, and
\item a map
\begin{equation}
\label{str.map.zero.to.one}
U \Phi^{\Cyclic_p}_\gen
\left(
\Phi^{\Cyclic_p}_\gen E
\longra
\beta U \Phi^{\Cyclic_p}_\gen E
\right)
\simeq
\left(
U \Phi^{\Cyclic_{p^2}}_\gen E
\longra
\left( U \Phi^{\Cyclic_p}_\gen \beta \right) \left( U \Phi^{\Cyclic_p}_\gen E \right)
\right)
=:
\left(
E_2
\longra
( E_1 )^{\tate \Cyclic_p}
\right)
\end{equation}
in $\Spectra$.
\end{itemize}

\item Finally, these maps fit into a commutative square
\begin{equation}
\label{comm.square.for.gen.Cpsquared.spt}
\begin{tikzcd}[row sep=1.5cm]
E_2
\arrow{r}{\Cref{str.map.zero.to.one}}
\arrow{d}[swap]{\Cref{str.map.zero.to.two}}
&
(E_1)^{\tate \Cyclic_p}
\arrow{d}{\Cref{str.map.one.to.two}^{\tate \Cyclic_p}}
\\
(E_0)^{\tate \Cyclic_{p^2}}
\arrow{r}
&
\left( (E_0)^{\tate \Cyclic_p} \right)^{\tate \Cyclic_p}
\end{tikzcd}
\end{equation}
in $\Spectra$, as a consequence of the commutativity of the diagram
\[ \begin{tikzcd}[row sep=1.5cm]
\left( U \Phi^{\Cyclic_p}_\gen \right) \left( \Phi^{\Cyclic_p}_\gen \right)
\arrow{r}
\arrow{d}
&
\left( U \Phi^{\Cyclic_p}_\gen \right) \beta U \left( \Phi^{\Cyclic_p}_\gen \right)
\arrow{d}
\\
\left( U \Phi^{\Cyclic_p}_\gen \right) \left( \Phi^{\Cyclic_p}_\gen \right) U \beta
\arrow{r}
&
\left( U \Phi^{\Cyclic_p}_\gen \right) \beta U \left( \Phi^{\Cyclic_p}_\gen \right) \beta U
\end{tikzcd} \]
in $\Fun(\Spectra^{\gen \Cyclic_{p^2}},\Spectra)$ and the canonical equivalence $\Phi^{\Cyclic_{p^2}}_\gen \simeq \Phi^{\Cyclic_p}_\gen \Phi^{\Cyclic_p}_\gen$.

\end{itemize}
Indeed, the lower morphism in the commutative square \Cref{comm.square.for.gen.Cpsquared.spt} is precisely the component at $E_0 \in \Spectra^{\htpy \Cyclic_{p^2}}$ of the natural transformation in the lax-commutative triangle appearing in equivalence \Cref{gen.Cpsquared.spectra.as.rlax.lim}.
\end{example}

\begin{remark}
\label{rmk.strict.Cpsquared.spectra}
Note that we have an equivalence
\[
(-)^{\tate \Cyclic_{p^2}}
\simeq
\left( (-)^{\htpy \Cyclic_p} \right)^{\tate \Cyclic_p}
\]
in $\Fun(\Spectra^{\htpy \Cyclic_{p^2}} , \Spectra)$ (see e.g.\! \cite[Lemma \ref{mackey:lem.projection.formula.categorical.then.geometric.is.geometric}]{AMR-mackey}). Using this, we can apply results of Nikolaus--Scholze to identify certain genuine $\Cyclic_{p^2}$-spectra $E \in \Spectra^{\gen \Cyclic_{p^2}}$ as strict (\Cref{defn.strict.objects}\Cref{item.defn.of.strict.object}). Strictness amounts to the assertion that the underlying homotopy $\Cyclic_{p^2}$-spectrum $E_0 := \Phi^e E \in \Spectra^{\htpy \Cyclic_{p^2}}$ satisfies the condition that the morphism
\[
(E_0)^{\tate \Cyclic_{p^2}}
\longra
\left( (E_0)^{\tate \Cyclic_p} \right)^{\tate \Cyclic_p}
\]
is an equivalence (so that the morphism \Cref{str.map.zero.to.two} is uniquely determined by the morphisms \Cref{str.map.one.to.two} and \Cref{str.map.zero.to.one}). Namely, this condition is guaranteed to hold assuming that the underlying spectrum $E_0 \in \Spectra$
\begin{itemize}

\item is bounded below by \cite[Lemma I.2.1]{NS}, or alternatively

\item admits a $\ZZ$-module structure by \cite[Footnote 9]{NS} (see also \cite[Lemma I.2.7]{NS}).

\end{itemize}
In fact, similar arguments can be applied to give a simplified description of the strict objects of $\Spectra^{\gen \Cyclic_{p^n}}$, as described in \cite[Remark II.4.8]{NS} (see 
\cite[\S 4]{Shah-recstrat}).\footnote{Note that bounded below objects of $\Spectra^{\gen \Cyclic_{p^n}}$ need not be strict (assuming $n \geq 3$).} Indeed, \cite[Theorem \ref{mackey:intro.thm.gen.Cpn.Z.mods}]{AMR-mackey} applies them to give a simplified description of $\ZZ$-linear genuine $\Cyclic_{p^n}$-spectra (a.k.a.\! derived Mackey functors).
\end{remark}

\begin{example}[genuine $\Cyclic_{pq}$-spectra]
\label{example.gen.Cpq.spt}
Let $\Cyclic_{pq} \cong \Cyclic_p \times \Cyclic_q$ denote the cyclic group of order $pq$, where $p$ and $q$ are distinct primes.  Its poset of conjugacy classes of closed subgroups is
\[
\pos_{\Cyclic_{pq}}
=
\left\{
\begin{tikzcd}
e
\arrow{r}
\arrow{d}
&
\Cyclic_p
\arrow{d}
\\
\Cyclic_q
\arrow{r}
&
\Cyclic_{pq}
\end{tikzcd}
\right\}
~. \]
Theorems \ref{thm.geom.stratn.of.SpgG} \and \ref{macrocosm.thm} provide an equivalence
\begin{equation}
\label{lax.comm.diagram.for.Cpq}
\Spectra^{\gen \Cyclic_{pq}}
\underset{\sim}{\xlongra{\gd}}
\lim^\rlax_{\llax.\pos_{\Cyclic_{pq}}} \left( \Spectra^{\htpy \Weyl_{\Cyclic_{pq}}(\bullet)} \right)
:=
\lim^\rlax
\left(
\begin{tikzcd}[row sep=2cm, column sep=1.5cm]
\Spectra^{\htpy \Cyclic_{pq}}
\arrow{r}{(-)^{\tate \Cyclic_p}}
\arrow{d}[swap]{(-)^{\tate \Cyclic_q}}
\arrow{rd}[sloped]{(-)^{\tate \Cyclic_{pq}}}[sloped, transform canvas={xshift=0.5cm, yshift=0.5cm}]{\Uparrow}[sloped, swap, transform canvas={xshift=-0.5cm, yshift=-0.5cm}]{\Downarrow}
&
\Spectra^{\htpy \Cyclic_q}
\arrow{d}{(-)^{\tate \Cyclic_q}}
\\
\Spectra^{\htpy \Cyclic_p}
\arrow{r}[swap]{(-)^{\tate \Cyclic_p}}
&
\Spectra
\end{tikzcd}
\right)
~.
\end{equation}
By \Cref{obs.NS.proper.tate.is.p.primary}, all three functors $\Spectra^{\htpy \Cyclic_{pq}} \ra \Spectra$ appearing in the lax-commutative diagram in equivalence \Cref{lax.comm.diagram.for.Cpq} are zero (the two composite functors by parts \Cref{Cp.tate.vanishes.if.p.acts.invertibly} \and \Cref{Cp.tate.is.p.complete}, the direct functor by part \Cref{vanishes.if.not.prime.power.order}).  It follows that via the equivalence \Cref{lax.comm.diagram.for.Cpq}, a genuine $\Cyclic_{pq}$-spectrum $E \in \Spectra^{\gen \Cyclic_{pq}}$ is completely specified by the data of
\begin{itemize}
\item the objects
\[ \begin{tikzcd}[row sep=1.5cm]
E_{00}
:=
U E
\in \Spectra^{\htpy \Cyclic_{pq}}
&
E_{01}
:=
\Phi^{\Cyclic_p} E
\in \Spectra^{\htpy \Cyclic_q}
\\
E_{10}
:=
\Phi^{\Cyclic_q} E
\in \Spectra^{\htpy \Cyclic_p}
&
E_{11}
:=
\Phi^{\Cyclic_{pq}} E
\in \Spectra
\end{tikzcd} \]
and
\item the structure maps
\[ \begin{tikzcd}[row sep=0.5cm, column sep=0.5cm]
&
(E_{00})^{\tate \Cyclic_p}
&
E_{01}
\arrow{l}
\\
(E_{00})^{\tate \Cyclic_q}
&
&
E_{01}^{\tate \Cyclic_q}
\\
E_{10}
\arrow{u}
&
E_{10}^{\tate \Cyclic_p}
&
E_{11}
\arrow{u}
\arrow{l}
\end{tikzcd}~.\footnote{These data are organized so as to reflect their positions within the diagram appearing in the equivalence \Cref{lax.comm.diagram.for.Cpq}.}
\]
\end{itemize}
\end{example}

\begin{remark}
\Cref{example.gen.Cpq.spt} makes manifest the equivalence
\[
\Spectra^{\gen \Cyclic_p} \otimes \Spectra^{\gen \Cyclic_q}
\xlongra{\sim}
\Spectra^{\gen \Cyclic_{pq}}
\]
(where the tensor product is taken in $\PrLSt$).  In other words, a genuine $\Cyclic_{pq}$-spectrum is equivalent data to a genuine $\Cyclic_p$-object in genuine $\Cyclic_q$-spectra (and vice versa).\footnote{We refer the reader to the brief discussion of \cite[\S\S\ref{mackey:section.defn.gen.G.objects}-\ref{mackey:section.stratn.of.tensor.products}]{AMR-mackey} for further support regarding these assertions.}
\end{remark}

\begin{example}[genuine $\Symm_3$-spectra]
\label{ex.gen.Sthree.spectra}
Let $\Symm_3$ denote the symmetric group on three letters.  Its poset of conjugacy classes of closed subgroups is
\[
\pos_{\Symm_3}
=
\left\{
\begin{tikzcd}
e
\arrow{r}
\arrow{d}
&
\Cyclic_2
\arrow{d}
\\
\Cyclic_3
\arrow{r}
&
\Symm_3
\end{tikzcd}
\right\}
~, \]
where $\Cyclic_3 = \sA_3$ denotes the alternating group (the normal subgroup of sign-preserving symmetries) and $\Cyclic_2$ denotes the equivalence class of the three (non-normal) order-two subgroups generated by the three transpositions.  Theorems \ref{thm.geom.stratn.of.SpgG} \and \ref{macrocosm.thm} provide an equivalence
\begin{equation}
\label{reconstrn.of.Sp.g.Sthree}
\Spectra^{\gen \Symm_3}
\underset{\sim}{\xlongra{\gd}}
\lim^\rlax_{\llax.\pos_{\Symm_3}} \left( \Spectra^{\htpy \Weyl_{\Symm_3}(\bullet)} \right)
:=
\lim^\rlax
\left(
\begin{tikzcd}[row sep=2cm, column sep=1.5cm]
\Spectra^{\htpy \Symm_3}
\arrow{r}{(-)^{\tate \Cyclic_2}}
\arrow{d}[swap]{(-)^{\tate \Cyclic_3}}
\arrow{rd}[sloped]{(-)^{\tate \Symm_3}}[sloped, transform canvas={xshift=0.5cm, yshift=0.5cm}]{\Uparrow}[sloped, swap, transform canvas={xshift=-0.5cm, yshift=-0.5cm}]{\Downarrow}
&
\Spectra
\arrow{d}
\\
\Spectra^{\htpy \Cyclic_2}
\arrow{r}[swap]{(-)^{\tate \Cyclic_2}}
&
\Spectra
\end{tikzcd}
\right)
~.
\end{equation}
By \Cref{obs.NS.proper.tate.is.p.primary}, the functors $(-)^{\tate \Symm_3}$ and $((-)^{\tate \Cyclic_3})^{\tate \Cyclic_2}$ are zero (the former by part \Cref{vanishes.if.not.prime.power.order}, the latter by parts \Cref{Cp.tate.vanishes.if.p.acts.invertibly} \and \Cref{Cp.tate.is.p.complete}).  Moreover, by \Cref{H.to.K.gluing.is.zero.when.K.not.leq.NH}, the gluing functor corresponding to the relation $\Cyclic_2 \ra \Symm_3$ is also zero.  Therefore, via the equivalence \Cref{reconstrn.of.Sp.g.Sthree}, a genuine $\Symm_3$-spectrum $E \in \Spectra^{\gen \Symm_3}$ is completely specified by the data of
\begin{itemize}
\item the objects
\[ \begin{tikzcd}[row sep=1.5cm]
E_{00}
:=
U E
\in \Spectra^{\htpy \Symm_3}
&
E_{01}
:=
\Phi^{\Cyclic_2} E
\in \Spectra
\\
E_{10}
:=
\Phi^{\Cyclic_3} E
\in \Spectra^{\htpy \Cyclic_2}
&
E_{11}
:=
\Phi^{\Symm_3} E
\in \Spectra
\end{tikzcd} \]
and

\item the structure maps
\[ \begin{tikzcd}[row sep=0.5cm, column sep=0.5cm]
&
(E_{00})^{\tate \Cyclic_2}
&
E_{01}
\arrow{l}
\\
(E_{00})^{\tate \Cyclic_3}
\\
E_{10}
\arrow{u}
&
(E_{10})^{\tate \Cyclic_2}
&
E_{11}
\arrow{l}
\end{tikzcd}~.
\]

\end{itemize}
\end{example}

\begin{example}[genuine and proper-genuine $\TT$-spectra]
\label{stratn.of.Sp.g.T.and.Sp.g.proper.T}
Let $\TT$ denote the circle group.  Its poset of conjugacy classes of closed subgroups admits an identification
\[
\pos_\TT \cong (\Ndiv)^\rcone
\]
as the right cone on the poset of natural numbers ordered by divisibility (under which the subgroup $\Cyclic_n \leq \TT$ corresponds to the element $n \in \Ndiv \subseteq (\Ndiv)^\rcone$), which we use implicitly for notational convenience.  The gluing diagram
\[
\begin{tikzcd}[column sep=1.5cm]
(\Ndiv)^\rcone
\arrow{r}[description, yshift=-0.05cm]{\llax}{\Spectra^{\htpy \Weyl_\TT(\bullet)}}
&
\PrSt
\end{tikzcd}
\]
of the geometric stratification of genuine $\TT$-spectra may be depicted as
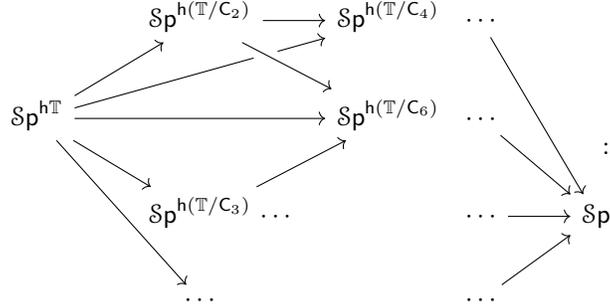
\[
\begin{tikzcd}
&
\Spectra^{\htpy (\TT/\Cyclic_2)}
\arrow{r}
&
\Spectra^{\htpy (\TT/\Cyclic_4)}
&[-0.8cm]
\cdots
\arrow{rdd}
\\
\Spectra^{\htpy \TT}
\arrow{ru}
\arrow{rd}
\arrow{rr}
\arrow{rdd}
\arrow{rru}
&
&
\Spectra^{\htpy (\TT/\Cyclic_6)}
\arrow[leftarrow, crossing over]{lu}
&
\cdots
\arrow{rd}
\\
&
\Spectra^{\htpy (\TT/\Cyclic_3)}
\arrow{ru}
&
&
\cdots
\arrow{r}
&
\Spectra
&[-5.8cm]
\cdots
\\
&
\cdots
&
&
\cdots
\arrow{ru}
\end{tikzcd}
~:
\]
\begin{itemize}
\item its values are described by the assignments
\[
r
\longmapsto
\begin{tikzcd}[column sep=1.5cm]
\Spectra^{\htpy(\TT/\Cyclic_r)}
\arrow[leftrightarrow]{r}{(\TT/\Cyclic_r) \cong \TT}[swap]{\sim}
&
\Spectra^{\htpy \TT}
\end{tikzcd}
\qquad
\text{and}
\qquad
\infty
\longmapsto
\Spectra^{\htpy (\TT/\TT)} \simeq \Spectra
~,
\]
\item it assigns to morphisms $r \ra rs$ and $r \ra \infty$ the horizontal functors in the diagrams
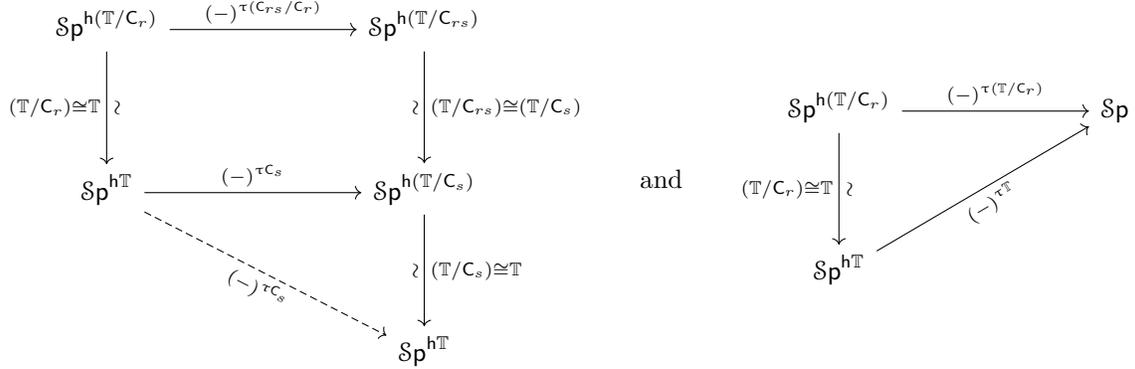
\[ \begin{tikzcd}[row sep=1.5cm, column sep=2.5cm]
\Spectra^{\htpy (\TT/\Cyclic_r)}
\arrow{r}{(-)^{\tate (\Cyclic_{rs}/\Cyclic_r)}}
\arrow{d}[sloped, anchor=south]{\sim}[swap]{(\TT/\Cyclic_r) \cong \TT}
&
\Spectra^{\htpy (\TT/\Cyclic_{rs})}
\arrow{d}[sloped, anchor=north]{\sim}{(\TT/\Cyclic_{rs}) \cong (\TT/\Cyclic_s)}
\\
\Spectra^{\htpy \TT}
\arrow{r}{(-)^{\tate \Cyclic_s}}
\arrow[dashed]{rd}[swap, sloped]{(-)^{\tate \Cyclic_s}}
&
\Spectra^{\htpy (\TT/\Cyclic_s)}
\arrow{d}[sloped, anchor=north]{\sim}{(\TT/\Cyclic_s) \cong \TT}
\\
&
\Spectra^{\htpy \TT}
\end{tikzcd}
\qquad
\text{and}
\qquad
\begin{tikzcd}[row sep=1.5cm, column sep=2.5cm]
\Spectra^{\htpy ( \TT/\Cyclic_r)}
\arrow{r}{(-)^{\tate ( \TT/\Cyclic_r)}}
\arrow{d}[sloped, anchor=south]{\sim}[swap]{(\TT/\Cyclic_r) \cong \TT}
&
\Spectra
\\
\Spectra^{\htpy \TT}
\arrow{ru}[sloped, swap]{(-)^{\tate \TT}}
\end{tikzcd}
\]
(in which the notation for the dashed functor is mildly abusive), and
\item we have suppressed the natural transformations
\[
(-)^{\tate \Cyclic_{rs}}
\longra
\left( (-)^{\tate \Cyclic_r} \right)^{\tate \Cyclic_s}
\]
(not to mention higher coherences) for typographical ease.
\end{itemize}
Theorems \ref{thm.geom.stratn.of.SpgG} \and \ref{macrocosm.thm} provide an adjunction
\begin{equation}
\label{macrocosm.adjn.for.Sp.g.T}
\begin{tikzcd}[column sep=2cm]
\Spectra^{\gen \TT}
\arrow[transform canvas={yshift=0.9ex}]{r}{\gd}
\arrow[leftarrow, transform canvas={yshift=-0.9ex}]{r}[yshift=-0.2ex]{\bot}[swap]{\lim_{\sd((\Ndiv)^\rcone)}}
&
\lim^\rlax_{\llax.(\Ndiv)^\rcone} \left( \Spectra^{\htpy \Weyl_\TT(\bullet)} \right)
\end{tikzcd}
~.
\end{equation}
However, \Cref{macrocosm.thm} does not guarantee that the adjunction \Cref{macrocosm.adjn.for.Sp.g.T} is an equivalence, because the poset $(\Ndiv)^\rcone$ is not down-finite (recall \Cref{remark.stratn.of.SpgG.prob.not.convergent}).  On the other hand, there is evidently a restricted stratification
\begin{equation}
\label{factorizn.of.stratn.of.Sp.g.T.to.stratn.of.Sp.g.proper.T}
\begin{tikzcd}
(\Ndiv)^\rcone
\arrow{r}{\Spectra^{\gen \TT}_{^\leq \bullet}}
&
\Cls_{\Spectra^{\gen \TT}}
\\
\Ndiv
\arrow[hook]{u}
\arrow[dashed]{r}
&
\Cls_{\Spectra^{\gen^\proper \TT}}
\arrow[hook]{u}
\end{tikzcd}
\end{equation}
of the presentable stable $\infty$-category $\Spectra^{\gen^\proper \TT}$ of \textit{proper}-genuine $\TT$-spectra (recall \Cref{obs.restricted.stratn.over.D}).\footnote{The right vertical functor of diagram \Cref{factorizn.of.stratn.of.Sp.g.T.to.stratn.of.Sp.g.proper.T} arises from the fact that we may identify proper-genuine $\TT$-spectra as the closed subcategory $\Spectra^{\gen^\proper \TT} \in \Cls_{\Spectra^{\gen \TT}}$ consisting of those objects $E \in \Spectra^{\gen \TT}$ such that the canonical morphism $E^\TT \ra E^{\htpy \TT}$ is an equivalence.}  As the poset $\Ndiv$ is down-finite, \Cref{macrocosm.thm} provides an equivalence
\[
\Spectra^{\gen^\proper \TT}
\underset{\sim}{\xlongra{\gd}}
\lim^\rlax_{\llax.\Ndiv} \left( \Spectra^{\htpy \Weyl_{\TT}(\bullet)} \right)
~.
\]
\end{example}

\subsection{Categorical fixedpoints via stratifications}
\label{subsection.categorical.fixedpoints}

In this subsection, we describe categorical fixedpoints of genuine $G$-spectra as well as restriction and transfer morphisms among them in terms of the geometric stratification.

\begin{local}
In this subsection, we assume that the group $G$ is finite.
\end{local}

\begin{observation}
The poset $\pos_G$ is finite, and hence the geometric stratification of $\Spectra^{\gen G}$ over it converges by \Cref{macrocosm.thm}. We use this fact without further comment.
\end{observation}

\begin{observation}
\label{obs.categorical.fixedpoints}
Given a genuine $G$-spectrum $E \in \Spectra^{\gen G}$, using the nanocosm reconstruction of \Cref{intro.thm.cosms}\Cref{intro.main.thm.nanocosm} (recall \Cref{rmk.nanocosm}), we may identify its categorical $H$-fixedpoints via the equivalences
\begin{align}
\nonumber
E^H
& \simeq
\hom_{\Spectra^{\gen G}} ( \Sigma^\infty_G(G/H)_+ , E )
\\
\nonumber
& \simeq
\lim_{([n] \xra{\varphi} \pos_G ) \in \sd(\pos_G)}
\hom_{\Spectra^{\htpy \Weyl(\varphi(n))}}
(
\Phi^{\varphi(n)} ( \Sigma^\infty_G(G/H)_+)
,
\Gamma_\varphi \Phi^{\varphi(0)} E
)
\\
\label{reduce.nanocosm.for.categorical.fixedpoints.by.commuting.fixedpoints.past.suspension}
& \simeq
\lim_{([n] \xra{\varphi} \pos_G ) \in \sd(\pos_G)}
\hom_{\Spectra^{\htpy \Weyl(\varphi(n))}}
(
\Sigma^\infty ((G/H)^{\varphi(n)})_+
,
\Gamma_\varphi \Phi^{\varphi(0)} E
)
\\
\label{reduce.to.leq.H.for.categorical.fixedpoints}
& \simeq
\lim_{([n] \xra{\varphi} (^\leq H) ) \in \sd(^\leq H)}
\hom_{\Spectra^{\htpy \Weyl(\varphi(n))}}
(
\Sigma^\infty ((G/H)^{\varphi(n)})_+
,
\Gamma_\varphi \Phi^{\varphi(0)} E
)
\end{align}
in $\Spectra^{\htpy \Weyl(H)}$, in which
\begin{itemize}

\item equivalence \Cref{reduce.nanocosm.for.categorical.fixedpoints.by.commuting.fixedpoints.past.suspension} follows from \Cref{obs.geom.fps.commutes.w.suspension} and

\item equivalence \Cref{reduce.to.leq.H.for.categorical.fixedpoints} follows from the facts

\begin{itemize}

\item that the functor
\[
\sd(^\leq H)
\longra
\sd(\pos_G)
\]
is a fully faithful right fibration and

\item that for any $K \not\leq H$ in $\pos_G$ we have
\[
\Sigma^\infty ((G/H)^{K})_+
\simeq
\Sigma^\infty (\es)_+
\simeq
0
~.
\]

\end{itemize}

\end{itemize}
\end{observation}

\begin{example}[categorical $e$-fixedpoints]
Suppose that $H = e \leq G$ is the trivial subgroup. Then, for any genuine $G$-spectrum $E \in \Spectra^{\gen G}$, the composite equivalence of \Cref{obs.categorical.fixedpoints} reduces to an equivalence
\[
E^e
\simeq
\hom_{\Spectra^{\htpy G}} ( \Sigma^\infty ((G/e)^e)_+ , U E )
\simeq
\hom_{\Spectra^{\htpy G}} ( \Sigma^\infty (G/e)_+ , U E )
\simeq
U E
\]
in $\Spectra^{\htpy G}$.
\end{example}

\begin{example}[categorical $\Cyclic_p$-fixedpoints]
Suppose that $H = G = \Cyclic_p$ (and recall \Cref{example.genuine.Cp.spectra}). For any genuine $\Cyclic_p$-spectrum $E \in \Spectra^{\gen \Cyclic_p}$, the composite equivalence of \Cref{obs.categorical.fixedpoints} reduces to an equivalence
\[
\hspace{-2cm}
E^{\Cyclic_p}
\simeq
\lim \left( \begin{tikzcd}
&
\hom_{\Spectra} ( \Sigma^\infty ( ( \Cyclic_p / \Cyclic_p)^{\Cyclic_p})_+ , \Phi^{\Cyclic_p} E )
\arrow{d}
\\
\hom_{\Spectra^{\htpy \Cyclic_p}} ( \Sigma^\infty ( (\Cyclic_p / \Cyclic_p)^e )_+ , U E )
\arrow{r}
&
\hom_{\Spectra} ( \Sigma^\infty ( ( \Cyclic_p / \Cyclic_p)^{\Cyclic_p} )_+ , (U E)^{\tate \Cyclic_p} )
\end{tikzcd} \right)
\simeq
\lim \left( \begin{tikzcd}
&
\Phi^{\Cyclic_p} E
\arrow{d}
\\
(UE)^{\htpy \Cyclic_p}
\arrow{r}
&
(UE)^{\tate \Cyclic_p}
\end{tikzcd} \right)
\]
in $\Spectra$.
\end{example}

\begin{example}[categorical $\Cyclic_{p^2}$-fixedpoints]
Suppose that $H = G = \Cyclic_{p^2}$ (and recall \Cref{example.genuine.Cpsquared.spectra}). For any genuine $\Cyclic_{p^2}$-spectrum $E \in \Spectra^{\gen \Cyclic_{p^2}}$, the composite equivalence of \Cref{obs.categorical.fixedpoints} yields a limit diagram
\[ \begin{tikzcd}
&
\Phi^{\Cyclic_{p^2}} E
\arrow{rr}
\arrow{dd}
&
&
(\Phi^{\Cyclic_p} E)^{\tate \Cyclic_p}
\arrow{dd}
\\
E^{\Cyclic_{p^2}}
\arrow[crossing over]{rr}
\arrow{ru}
\arrow{dd}
&
&
(\Phi^{\Cyclic_p} E)^{\htpy \Cyclic_p}
\arrow{ru}
\\
&
(UE)^{\tate \Cyclic_{p^2}}
\arrow{rr}
&
&
((UE)^{\tate \Cyclic_p})^{\tate \Cyclic_p}
\\
(UE)^{\htpy \Cyclic_{p^2}}
\arrow{rr}
\arrow{ru}
&
&
((UE)^{\tate \Cyclic_p})^{\htpy \Cyclic_p}
\arrow{ru}
\arrow[leftarrow, crossing over]{uu}
\end{tikzcd} \]
in $\Spectra$.
\end{example}

\begin{local}
For the remainder of this subsection, we fix a subgroup $K \subseteq H$ of the chosen subgroup $H \subseteq G$.
\end{local}

\begin{definition}
\label{defn.relative.Weyl.group}
The \bit{relative Weyl group} of the nested pair $K \subseteq H$ of subgroups of $G$ is
\[
\Weyl(K \subseteq H)
:=
\Weyl_G(K \subseteq H)
:=
\frac{\Normzer_G(K) \cap \Normzer_G(H)}{K}
~,
\]
the quotient by $K$ of the intersection of the normalizers of $K$ and $H$ in $G$.\footnote{More invariantly, one can also describe $\Weyl(K \subseteq H)$ as the group of automorphisms of the object $(G/K \ra G/H) \in \Ar(\Orb_G) \subseteq \Ar(\Spaces^{\gen G})$.} By definition, this comes equipped with homomorphisms
\[ \begin{tikzcd}[row sep=0cm]
\Weyl(K)
&
\Weyl(K \subseteq H)
\arrow{l}
\arrow{rr}
&&
\Weyl(H)
\\
\rotatebox{90}{$=:$}
&
\rotatebox{90}{$=:$}
&
&
\rotatebox{90}{$=:$}
\\
{\displaystyle \frac{\Normzer(K)}{K} }
&
{\displaystyle \frac{\Normzer(K) \cap \Normzer(H)}{K} }
\arrow{l}
\arrow{r}
&
{\displaystyle \frac{\Normzer(K) \cap \Normzer(H)}{\Normzer(K) \cap H} }
\arrow{r}
&
{\displaystyle \frac{\Normzer(H)}{H} }
\end{tikzcd}~. \]
\end{definition}

\begin{observation}
\label{obs.restriction.between.catl.fixedpts}
Restriction defines a natural transformation
\[ \begin{tikzcd}
\Spectra^{\gen G}
\arrow{r}{(-)^H}[swap, yshift=-0.4cm]{\rotatebox{45}{$\Leftarrow$}}
\arrow{d}[swap]{(-)^K}
&
\Spectra^{\htpy \Weyl_G(H)}
\arrow{d}
\\
\Spectra^{\htpy \Weyl_G(K)}
\arrow{r}
&
\Spectra^{\htpy \Weyl_G(K \subseteq H)}
\end{tikzcd}~, \]
which is corepresented by the morphism
\[
\Sigma^\infty_G(G/K \longra G/H)_+
\]
in $\Spectra^{\gen G}$. In terms of nanocosm reconstruction, for any genuine $G$-spectrum $E \in \Spectra^{\gen G}$ it may be expressed as the composite
\begin{align}
\label{identify.categorical.H.fixedpoints.in.describing.restriction}
E^H
& \simeq
\lim_{([n] \xra{\varphi} (^\leq H) ) \in \sd(^\leq H)}
\hom_{\Spectra^{\htpy \Weyl(\varphi(n))}}
(
\Sigma^\infty ((G/H)^{\varphi(n)})_+
,
\Gamma_\varphi \Phi^{\varphi(0)} E
)
\\
\label{restrict.to.small.diagram.in.describing.restriction}
& \longra
\lim_{([n] \xra{\varphi} (^\leq K) ) \in \sd(^\leq K)}
\hom_{\Spectra^{\htpy \Weyl(\varphi(n))}}
(
\Sigma^\infty ((G/H)^{\varphi(n)})_+
,
\Gamma_\varphi \Phi^{\varphi(0)} E
)
\\
\label{change.from.G.mod.H.to.G.mod.K.in.describing.restriction}
& \longra
\lim_{([n] \xra{\varphi} (^\leq K) ) \in \sd(^\leq K)}
\hom_{\Spectra^{\htpy \Weyl(\varphi(n))}}
(
\Sigma^\infty ((G/K)^{\varphi(n)})_+
,
\Gamma_\varphi \Phi^{\varphi(0)} E
)
\\
\label{identify.categorical.K.fixedpoints.in.describing.restriction}
& \simeq
E^K
~,
\end{align}
where
\begin{itemize}

\item the equivalences \Cref{identify.categorical.H.fixedpoints.in.describing.restriction} \and \Cref{identify.categorical.K.fixedpoints.in.describing.restriction} follow from \Cref{obs.categorical.fixedpoints},

\item the morphism \Cref{restrict.to.small.diagram.in.describing.restriction} is that on limits induced by the functor
\[
\sd(^\leq K)
\longra
\sd(^\leq H)
~,
\]
and

\item the morphism \Cref{change.from.G.mod.H.to.G.mod.K.in.describing.restriction} is that on limits determined by a morphism in $\Fun(\sd(^\leq K),\Spectra)$ whose component at an object $([n] \xra{\varphi} (^\leq K)) \in \sd(^\leq K)$ is precomposition with the morphism
\[
\Sigma^\infty ( ( G/K \longra G/H)^{\varphi(n)} )_+
\]
in $\Spectra^{\htpy \Weyl(\varphi(n))}$.
\end{itemize}
\end{observation}

\begin{observation}
\label{obs.transfer.between.catl.fixedpts}
Transfer defines a natural transformation
\[ \begin{tikzcd}
\Spectra^{\gen G}
\arrow{r}{(-)^H}[swap, yshift=-0.4cm]{\rotatebox{45}{$\Rightarrow$}}
\arrow{d}[swap]{(-)^K}
&
\Spectra^{\htpy \Weyl_G(H)}
\arrow{d}
\\
\Spectra^{\htpy \Weyl_G(K)}
\arrow{r}
&
\Spectra^{\htpy \Weyl_G(K \subseteq H)}
\end{tikzcd}~, \]
which is corepresented by a morphism
\begin{equation}
\label{morphism.corepresenting.transfer}
\Sigma^\infty_G(G/H)_+
\longra
\Sigma^\infty_G(G/K)_+
\end{equation}
in $\Spectra^{\gen G}$.\footnote{The morphism \Cref{morphism.corepresenting.transfer} may be obtained by applying the functor $\Spectra^{\gen H} \xra{\Ind_H^G} \Spectra^{\gen G}$ to the morphism
\[
\Sigma^\infty_H(H/H)_+
\longra
\Sigma^\infty_H(H/K)_+
\simeq
\coInd_K^H(\Sigma^\infty_K(K/K)_+)
\]
corresponding to the identity morphism
\[
\Res^H_K(\Sigma^\infty_H(H/H)_+)
\longra
\Sigma^\infty_K(K/K)_+
\]
in $\Spectra^{\gen K}$.} In terms of nanocosm reconstruction, for any genuine $G$-spectrum $E \in \Spectra^{\gen G}$ it may be expressed as the composite
\begin{align}
\label{identify.categorical.K.fixedpoints.in.describing.transfer}
E^K
& \simeq
\lim_{([n] \xra{\varphi} (^\leq K) ) \in \sd(^\leq K)}
\hom_{\Spectra^{\htpy \Weyl(\varphi(n))}}
(
\Sigma^\infty ((G/K)^{\varphi(n)})_+
,
\Gamma_\varphi \Phi^{\varphi(0)} E
)
\\
\label{expand.categorical.K.fixedpoints.to.a.limit.over.sd.leq.H}
& \simeq
\lim_{([n] \xra{\varphi} (^\leq H) ) \in \sd(^\leq H)}
\hom_{\Spectra^{\htpy \Weyl(\varphi(n))}}
(
\Sigma^\infty ((G/K)^{\varphi(n)})_+
,
\Gamma_\varphi \Phi^{\varphi(0)} E
)
\\
\label{morphism.in.describing.transfer}
& \longra
\lim_{([n] \xra{\varphi} (^\leq H) ) \in \sd(^\leq H)}
\hom_{\Spectra^{\htpy \Weyl(\varphi(n))}}
(
\Sigma^\infty ((G/H)^{\varphi(n)})_+
,
\Gamma_\varphi \Phi^{\varphi(0)} E
)
\\
\label{identify.categorical.H.fixedpoints.in.describing.transfer}
& \simeq
E^H
~,
\end{align}
where
\begin{itemize}
\item the equivalences \Cref{identify.categorical.K.fixedpoints.in.describing.transfer}, \Cref{expand.categorical.K.fixedpoints.to.a.limit.over.sd.leq.H}, \and \Cref{identify.categorical.H.fixedpoints.in.describing.transfer} follow from \Cref{obs.categorical.fixedpoints}, and

\item the morphism \Cref{morphism.in.describing.transfer} is that on limits determined by a morphism in $\Fun(\sd(^\leq H),\Spectra)$ whose component at an object $([n] \xra{\varphi} (^\leq H)) \in \sd(^\leq H)$ is precomposition with the morphism
\[
\Sigma^\infty((G/H)^{\varphi(n)})_+
\simeq
\Phi^{\varphi(n)} ( \Sigma^\infty_G(G/H)_+)
\xra{\Phi^{\varphi(n)}\Cref{morphism.corepresenting.transfer}}
\Phi^{\varphi(n)} ( \Sigma^\infty_G(G/K)_+)
\simeq
\Sigma^\infty((G/K)^{\varphi(n)})_+
\
\]
in $\Spectra^{\htpy \Weyl(\varphi(n))}$ (using \Cref{obs.geom.fps.commutes.w.suspension}).
\end{itemize}
\end{observation}

\section{The metacosm reconstruction theorem}
\label{section.reconstrn}

In this section, we prove the metacosm reconstruction theorem (\Cref{intro.thm.cosms}\Cref{intro.main.thm.metacosm}), which easily implies the macrocosm reconstruction theorem (\Cref{intro.thm.cosms}\Cref{intro.main.thm.macrocosm}) as proved in \Cref{section.strat}. It is organized as follows.
\begin{itemize}

\item[\Cref{subsection.stratns.of.rlax.lims}:] We establish a canonical stratification of certain right-lax limits.

\item[\Cref{subsection.metacosm}:] We prove \Cref{intro.thm.cosms}\Cref{intro.main.thm.metacosm} as \Cref{metacosm.thm}.  Recall that this is an adjunction, which is an equivalence when the poset is down-finite.  Its left adjoint takes a stratified presentable stable $\infty$-category to its gluing diagram; its right adjoint is essentially constructed in \Cref{subsection.stratns.of.rlax.lims}.

\item[\Cref{subsection.strict.stratns}:] We explain the theory of strict stratifications.

\end{itemize}

\begin{local}
In this section, we fix a poset $\pos$.
\end{local}

\subsection{Stratifications of right-lax limits}
\label{subsection.stratns.of.rlax.lims}

In this subsection we prove the omnibus \Cref{prop.metacosm.input.first.get.stratn}, which establishes a canonical stratification of certain right-lax limits as well as a number of its essential properties.

\begin{definition}
\label{defn.LMod.rlax.L}
A \bit{presentable stable left-lax left $\pos$-module} is a left-lax left $\pos$-module whose fibers are presentable stable $\infty$-categories and whose monodromy functors are exact and accessible.  These assemble into a subcategory
\[
\LMod^{\rlax,L}_{\llax.\pos}(\PrSt)
\subseteq
\LMod^\rlax_{\llax.\pos}
\]
whose morphisms are those morphisms in $\LMod^\rlax_{\llax.\pos}$ that are fiberwise left adjoints.
\end{definition}

\begin{local}
In this subsection, we fix a presentable stable left-lax left $\pos$-module
\[
(\cE \da \pos)
\in
\LMod^{\rlax,L}_{\llax.\pos}(\PrSt)
~.
\]
For any morphism $p \ra q$ in $\pos$ we write
\[
\cE_p
\xra{\Gamma^p_q}
\cE_q
\]
for its corresponding cocartesian monodromy functor.
\end{local}

\begin{local}
In this subsection, we write
\[
\cX := \lim^\rlax_{\llax.\pos}(\cE) \in \Cat
~,
\]
and for any subposet $\posQ \subseteq \pos$ we write
\[
\cX
:=
\lim^\rlax_{\llax.\pos}(\cE)
\xra{\Phi_\posQ}
\lim^\rlax_{\llax.\posQ}(\cE)
\]
for the restriction functor.
\end{local}

\begin{observation}
\label{obs.limrlax.is.accessible}
It follows from \Cref{lem.strictification} that $\cX$ is accessible.
\end{observation}

\begin{observation}
It follows from \Cref{lem.strictification} that $\cX$ is stable. We use this fact without further comment.
\end{observation}

\needspace{2\baselineskip}
\begin{prop}
\label{prop.metacosm.input.first.get.stratn}
\begin{enumerate}
\item[]

\item\label{metacosm.input.presentable}
The $\infty$-category $\cX := \lim^\rlax_{\llax.\pos}(\cE)$ is cocomplete, and hence presentable by \Cref{obs.limrlax.is.accessible}.

\item\label{metacosm.input.conservative}
The functor
\[
\cX
\xra{(\Phi_p)_{p \in \pos}}
\prod_{p \in \pos} \cE_p
\]
is conservative.

\item\label{metacosm.input.cocts.restrn.to.any.subposet}
For any subposet $\posQ \subseteq \pos$, the restriction functor
\[
\cX
:=
\lim^\rlax_{\llax.\pos}(\cE)
\xra{\Phi_\posQ}
\lim^\rlax_{\llax.\posQ}(\cE)
\]
preserves colimits, and hence admits a right adjoint by part \Cref{metacosm.input.presentable}.

\item\label{metacosm.input.describe.adjts}
Choose any $\sD \in \Down_\pos$.

\begin{enumeratesub}

\item\label{metacosm.input.describe.iL}
The restriction functor
\[
\cX
:=
\lim^\rlax_{\llax.\pos}(\cE)
\xlongra{y}
\lim^\rlax_{\llax.\sD}(\cE)
=:
\cZ_\sD
\]
admits not only a right adjoint $i_R$ as guaranteed by part \Cref{metacosm.input.cocts.restrn.to.any.subposet} but also a fully faithful left adjoint $i_L$, whose image consists of those objects $X \in \cX$ such that $\Phi_q(X) \simeq 0$ for all $q \in \pos \backslash \sD$.  In particular, for any $p \in \pos$, we may consider
\begin{equation}
\label{pth.closed.subcat.in.stratn.of.rlax.lim}
\cZ_p
:=
\lim^\rlax_{\llax.(^\leq p)}(\cE)
\end{equation}
as a closed subcategory of $\cX$ via $i_L$.

\item\label{metacosm.input.describe.nu}
The right adjoint $\nu$ to the restriction functor
\[
\cX
:=
\lim^\rlax_{\llax.\pos}(\cE)
\xra{p_L}
\lim^\rlax_{\llax.(\pos\backslash\sD)}(\cE)
\]
guaranteed by part \Cref{metacosm.input.cocts.restrn.to.any.subposet} is fully faithful, and its image consists of those objects $X \in \cX$ such that $\Phi_q(X) \simeq 0$ for all $q \in \sD$.

\end{enumeratesub}

\item\label{metacosm.input.get.stratn}
The closed subcategories \Cref{pth.closed.subcat.in.stratn.of.rlax.lim} assemble into a stratification
\begin{equation}
\label{the.stratn.of.rlax.lim}
\begin{tikzcd}[row sep=0cm]
\pos
\arrow{r}{\cZ_\bullet}
&
\Cls_\cX
&[-1.2cm]
:= \Cls_{\lim^\rlax_{\llax.\pos}(\cE)}
\\
\rotatebox{90}{$\in$}
&
\rotatebox{90}{$\in$}
\\
p
\arrow[maps to]{r}
&
\cZ_p
&
:= \lim^\rlax_{\llax.(^\leq p)}(\cE)
\end{tikzcd}
~.
\end{equation}
Moreover, our existing notation is consistent with this stratification in the following ways.
\begin{enumeratesub}

\item\label{notational.consistency.for.ZD}

For any $\sD \in \Down_\pos$, we have
\[
\cZ_\sD
:=
\lim^\rlax_{\llax.\sD}(\cE)
\simeq
\brax{ \lim^\rlax_{\llax.(^\leq p)}(\cE) }_{p \in \sD}
=:
\brax{\cZ_p}_{p \in \sD}
~.
\]

\item\label{notational.consistency.for.PhiC}

For any $\sC \in \Conv_\pos$, the $\sC\th$ stratum of the stratification \Cref{the.stratn.of.rlax.lim} is
\[
\cX_\sC
:=
\cZ_{^\leq \sC} / \cZ_{^< \sC}
\simeq
\lim^\rlax_{\llax.\sC}(\cE)
~,
\]
and its $\sC\th$ geometric localization functor
\[
\cX
:=
\lim^\rlax_{\llax.\pos}(\cE)
\xra{\Phi_\sC}
\lim^\rlax_{\llax.\sC}(\cE)
\simeq
\cX_\sC
\]
is the restriction functor.

\item\label{gluing.in.X.is.mdrmy.in.E}

For any $p<q$ in $\pos$, the lax-commutative square
\begin{equation}
\label{lax.comm.square.that.actually.commutes.to.show.mdrmy.is.as.expected.in.lim.rlax.llax.P.E}
\begin{tikzcd}
\cE_p
\arrow[hook]{r}{\rho^p}[swap, xshift=0.2cm, yshift=-0.45cm]{\rotatebox{45}{$\Leftarrow$}}
\arrow[equals]{d}
&
\lim^\rlax_{\llax.\pos}(\cE)
\arrow{d}{\Phi_{\{p<q\}}}
&[-1.4cm]
=: \cX
\\
\cE_p
\arrow[hook]{r}[swap]{\rho^p}
&
\lim^\rlax_{\llax.\{p<q\}}(\cE)
\end{tikzcd}
\end{equation}
determined by the commutative square
\begin{equation}
\label{comm.square.of.restrns.to.show.mdrmy.is.as.expected.in.lim.rlax.P.E}
\begin{tikzcd}
\cE_p
\arrow[equals]{d}
&
\lim^\rlax_{\llax.\pos}(\cE)
\arrow{d}{\Phi_{\{p<q\}}}
\arrow{l}[swap]{\Phi_p}
&[-1.4cm]
=: \cX
\\
\cE_p
&
\lim^\rlax_{\llax.\{p<q\}}(\cE)
\arrow{l}{\Phi_p}
\end{tikzcd}
\end{equation}
commutes.  In particular, for every morphism $p \ra q$ in $\pos$, there is a canonical identification
\[ \begin{tikzcd}[row sep=0cm]
\cX_p
\arrow{r}{\Gamma^p_q}
&
\cX_q
\\
\rotatebox{90}{$\simeq$}
&
\rotatebox{90}{$\simeq$}
\\
\cE_p
\arrow{r}[swap]{\Gamma^p_q}
&
\cE_q
\end{tikzcd} \]
between the corresponding gluing functor for $\cX$ (with respect to the stratification \Cref{the.stratn.of.rlax.lim}) and the corresponding monodromy functor of $\cE$.

\end{enumeratesub}

\end{enumerate}
\end{prop}

\begin{warning}
In the statement and proof of \Cref{prop.metacosm.input.first.get.stratn}, we use notation corresponding to recollements (such as $i_L$, $y$, etc.) even before those recollements have been established.
\end{warning}

\begin{definition}
\label{defn.stable.recollement}
A \bit{stable recollement} is a diagram \Cref{recollement.in.intro} among stable $\infty$-categories such that there are equalities \Cref{equalities.in.defn.of.recollement}. (In particular, we use the same notation for the functors involved in a stable recollement as we do for those involved in a recollement.)
\end{definition}

\begin{remark}
A recollement in the sense of \Cref{defn.recollement.in.intro} is simply a stable recollement among presentable stable $\infty$-categories.
\end{remark}

\begin{observation}
\Cref{lem.reconstrn.for.recollement} applies not just to recollements but to stable recollements: neither the statement nor the proof relies on presentability in any way.  We will use this fact without further comment.
\end{observation}

\begin{lemma}
\label{prop.metacosm.input.first.get.stratn.for.brax.one}
\Cref{prop.metacosm.input.first.get.stratn} holds when $\pos = [1]$.
\end{lemma}

\begin{proof}
It is immediate that we have a stable recollement
\begin{equation}
\label{stable.recollement.in.proof.of.stratification.from.module.over.brax.one}
\begin{tikzcd}[column sep=1.5cm]
\cE_0
\arrow[hook, bend left=45]{r}[description]{i_L}
\arrow[leftarrow]{r}[transform canvas={yshift=0.1cm}]{\bot}[swap,transform canvas={yshift=-0.1cm}]{\bot}[description]{\yo}
\arrow[bend right=45, hook]{r}[description]{i_R}
&
\cX
\arrow[bend left=45]{r}[description]{p_L}
\arrow[hookleftarrow]{r}[transform canvas={yshift=0.1cm}]{\bot}[swap,transform canvas={yshift=-0.1cm}]{\bot}[description]{\nu}
\arrow[bend right=45]{r}[description]{p_R}
&
\cE_1
\end{tikzcd}
\end{equation}
in which, writing
\begin{equation}
\label{typical.object.in.limrlax.over.brax.one.in.proof.of.main.metacosm.ingredient}
( E_0 \longmapsto \Gamma^0_1(E_0) \xlongla{\gamma} E_1 )
\end{equation}
for an arbitrary object of $\cX$ (where $E_i \in \cE_i$ for $i \in [1]$),
\begin{itemize}
\item the three functors with source $\cX$ are defined by the formulas
\[
y \Cref{typical.object.in.limrlax.over.brax.one.in.proof.of.main.metacosm.ingredient} := E_0
~,
\qquad
p_L \Cref{typical.object.in.limrlax.over.brax.one.in.proof.of.main.metacosm.ingredient} := E_1
~,
\qquad
\text{and}
\qquad
p_R \Cref{typical.object.in.limrlax.over.brax.one.in.proof.of.main.metacosm.ingredient} := \fib(\gamma)
~,
\]
\item the two functors with source $\cE_0$ are defined by the formulas
\[
i_L(E) := ( E \longmapsto \Gamma^0_1(E) \longla 0 )
\qquad
\text{and}
\qquad
i_R(E) := ( E \longmapsto \Gamma^0_1(E) \xlongla{\sim} \Gamma^0_1(E) )
~,
\]
and
\item the one functor with source $\cE_1$ is defined by the formula
\[
\nu(E) := ( 0 \longmapsto 0 \longla E )
~.
\]
\end{itemize}
In particular, we have an evident identification $\Gamma^0_1 \simeq p_L i_R$.  Moreover, applying \Cref{lem.reconstrn.for.recollement} to the stable recollement \Cref{stable.recollement.in.proof.of.stratification.from.module.over.brax.one}, it is straightforward to verify that any functor $\cI \xra{F} \cX$ has a colimit
\[
\left(
\colim_\cI(yF)
\longmapsto
\Gamma^0_1(\colim_\cI(y F))
\simeq
p_L i_R ( \colim_\cI(y F))
\longla
\colim_\cI(p_L i_R y F)
\xla{\eta_{y \adj i_R}}
\colim_\cI ( p_L F)
\right)
~,
\]
so that $\cX$ is cocomplete.  The remaining claims are now evident.
\end{proof}

\begin{lemma}
\label{prop.metacosm.input.first.get.stratn.for.brax.n}
\Cref{prop.metacosm.input.first.get.stratn} holds when $\pos = [n] \in \bDelta$ (for any $n \geq 0$).
\end{lemma}

\begin{proof}
The claim is immediate if $n=0$, and if $n=1$ this is the content of \Cref{prop.metacosm.input.first.get.stratn.for.brax.one}.  So suppose that $n \geq 2$.  Let us write $\cY := \lim^\rlax_{\llax.\{1 < \cdots < n\}}(\cE)$.

Consider the functor $[n] \xra{\pi} [1]$ characterized by the fact that $\pi^{-1}(0)=\{0\}$.  In light of \Cref{lem.strictification}, using the composability of right Kan extensions with respect to the composite $\sd([n]) \xra{\sd(\pi)} \sd([1]) \ra \pt$, we obtain a pullback square
\[
\begin{tikzcd}
\cX
\arrow{r}
\arrow{d}
&
\Fun([1],\cY)
\arrow{d}{t}
\\
\cE_0
\arrow{r}
&
\cY
\end{tikzcd}
\]
in which the left vertical functor and the composite $\cX \ra \Fun([1],\cY) \xra{s} \cY$ are the canonical restriction functors. 
This immediately yields a stable recollement
\begin{equation}
\label{recollement.Ezero.X.Y.in.proof.of.stratn.for.brax.n}
\begin{tikzcd}[column sep=1.5cm]
\cE_0
\arrow[hook, bend left=45]{r}[description]{i_L}
\arrow[leftarrow]{r}[transform canvas={yshift=0.1cm}]{\bot}[swap,transform canvas={yshift=-0.1cm}]{\bot}[description]{\yo}
\arrow[bend right=45, hook]{r}[description]{i_R}
&
\cX
\arrow[bend left=45]{r}[description]{p_L}
\arrow[hookleftarrow]{r}[transform canvas={yshift=0.1cm}]{\bot}[swap,transform canvas={yshift=-0.1cm}]{\bot}[description]{\nu}
\arrow[bend right=45]{r}[description]{p_R}
&
\cY
\end{tikzcd}~,
\end{equation}
in which the functors $y$ and $p_L$ are the canonical restriction functors.  Note moreover that $\cE_0$ is presentable by assumption, $\cY$ is presentable by induction, and the composite functor $p_L i_R$ is accessible in light of \Cref{obs.limrlax.is.accessible}.  So, it follows from \Cref{prop.metacosm.input.first.get.stratn.for.brax.one} that $\cX$ is presentable: that is, we have proved part \Cref{metacosm.input.presentable}.

Using the recollement \Cref{recollement.Ezero.X.Y.in.proof.of.stratn.for.brax.n} and \Cref{lem.reconstrn.for.recollement}, we see by induction that the functor
\[
\cX
\xra{(\Phi_i)_{i \in [n]}}
\prod_{i \in [n]} \cE_i
\]
is conservative and preserves colimits; in particular, we have proved part \Cref{metacosm.input.conservative}.  Since any subposet $\posQ \subseteq [n]$ whose inclusion is not an isomorphism is of the form $\posQ \cong \coprod_{j=1}^k [i_j]$ where $i_j < n$ for all $j$, we then also see by induction (with respect to parts \Cref{metacosm.input.conservative} \and \Cref{metacosm.input.cocts.restrn.to.any.subposet}) that the restriction functor
\[
\cX
\xra{\Phi_\posQ}
\lim^\rlax_{\llax.\posQ}(\cE)
\]
preserves colimits.  So, we have proved part \Cref{metacosm.input.cocts.restrn.to.any.subposet}.

We now turn to part \Cref{metacosm.input.describe.adjts}.  If $\sD = \es$ then part \Cref{metacosm.input.describe.adjts} is trivial, while if $\sD = \{0\}$ then part \Cref{metacosm.input.describe.adjts} follows from the recollement \Cref{recollement.Ezero.X.Y.in.proof.of.stratn.for.brax.n} (and part \Cref{metacosm.input.conservative} applied to $\cY$).  So, we may assume that $\sD = (^\leq i) = [i]$ where $1 \leq i \leq n$.  Noting the factorization
\[
\cX
\longra
\cY
\longra
\lim^\rlax_{\llax.([n]\backslash [i])}(\cE)
\]
of the restriction functor, we find that part \Cref{metacosm.input.describe.adjts}\Cref{metacosm.input.describe.nu} follows from induction and the recollement \Cref{recollement.Ezero.X.Y.in.proof.of.stratn.for.brax.n}.  So it remains to prove part \Cref{metacosm.input.describe.adjts}\Cref{metacosm.input.describe.iL}.  For this, we introduce the notation
\[
\cW_i := \lim^\rlax_{\llax.\{1 < \cdots < i\}}(\cE)
\]
and make the following observations.
\begin{itemize}

\item By induction, we have $\cW_i \in \Cls_\cY$.

\item Replacing $[n]$ with $[i$], the recollement \Cref{recollement.Ezero.X.Y.in.proof.of.stratn.for.brax.n} becomes an analogous recollement
\begin{equation}
\label{recollement.Ezero.Zi.Wi.in.proof.of.stratn.for.brax.n}
\begin{tikzcd}[column sep=1.5cm]
\cE_0
\arrow[hook, bend left=45]{r}[description]{i_L}
\arrow[leftarrow]{r}[transform canvas={yshift=0.1cm}]{\bot}[swap,transform canvas={yshift=-0.1cm}]{\bot}[description]{\yo}
\arrow[bend right=45, hook]{r}[description]{i_R}
&
\cZ_i
\arrow[bend left=45]{r}[description]{p_L}
\arrow[hookleftarrow]{r}[transform canvas={yshift=0.1cm}]{\bot}[swap,transform canvas={yshift=-0.1cm}]{\bot}[description]{\nu}
\arrow[bend right=45]{r}[description]{p_R}
&
\cW_i
\end{tikzcd}~.
\end{equation}

\item
The diagram
\begin{equation}
\label{comm.diagram.giving.yoneda.functor.from.X.to.Zi}
\begin{tikzcd}
\cE_0
\arrow[equals]{d}
\arrow[leftarrow]{r}{y}
&
\cX
\arrow{r}{p_L}
\arrow{d}[swap]{y}
&
\cY
\arrow{d}{y}
\\
\cE_0
\arrow[leftarrow]{r}[swap]{y}
&
\cZ_i
\arrow{r}[swap]{p_L}
&
\cW_i
\end{tikzcd}
\end{equation}
among restriction functors commutes.

\item
The fully faithful inclusion $i_R$ of recollement \Cref{recollement.Ezero.X.Y.in.proof.of.stratn.for.brax.n} (resp.\! \Cref{recollement.Ezero.Zi.Wi.in.proof.of.stratn.for.brax.n}) has image consisting of those objects $X \in \cX$ (resp.\! $X \in \cZ_i$) such that for all $j \in \{1 < \cdots < n\}$ (resp.\! $j \in \{ 1 < \cdots < i \}$) the structure morphism $\Phi_j(X) \ra \Gamma^0_j(\Phi_0(X))$ is an equivalence.  It follows that the lax-commutative square
\[ \begin{tikzcd}
\cE_0
\arrow[hook]{r}{i_R}[swap, yshift=-0.4cm]{\rotatebox{45}{$\Leftarrow$}}
\arrow[equals]{d}
&
\cX
\arrow{d}{y}
\\
\cE_0
\arrow[hook]{r}[swap]{i_R}
&
\cZ_i
\end{tikzcd} \]
determined by the left commutative square in diagram \Cref{comm.diagram.giving.yoneda.functor.from.X.to.Zi} commutes.

\end{itemize}
Using these observations and applying \Cref{lem.reconstrn.for.recollement} to the recollements \Cref{recollement.Ezero.X.Y.in.proof.of.stratn.for.brax.n} and \Cref{recollement.Ezero.Zi.Wi.in.proof.of.stratn.for.brax.n}, we find that the restriction functor $\cX \xra{y} \cZ_i$ is described by the formula
\[ \begin{tikzcd}[row sep=0cm]
\cZ_i
\simeq
&[-1.5cm]
\lim^\rlax \left( \cE_0 \xra{p_L i_R} \cW_i \right)
\arrow[leftarrow]{r}{y}
&
\lim^\rlax \left( \cE_0 \xra{p_L i_R} \cY \right)
&[-1.4cm]
\simeq \cX
\\
&
\rotatebox{90}{$\in$}
&
\rotatebox{90}{$\in$}
\\
&
(E \longmapsto p_L i_R (E) \longla y(Y))
&
(E \longmapsto p_L i_R (E) \longla Y)
\arrow[maps to]{l}
\end{tikzcd}~, \]
so that it admits a left adjoint described by the formula
\[ \begin{tikzcd}[row sep=0cm]
\cZ_i
\simeq
&[-1.3cm]
\lim^\rlax \left( \cE_0 \xra{p_L i_R} \cW_i \right)
\arrow{r}{i_L}
&
\lim^\rlax \left( \cE_0 \xra{p_L i_R} \cY \right)
&[-1.7cm]
\simeq \cX
\\
&
\rotatebox{90}{$\in$}
&
\rotatebox{90}{$\in$}
\\
&
( E \longmapsto p_L i_R(E) \longla W)
\arrow[maps to]{r}
&
( E \longmapsto p_L i_R(E) \longla i_L(W))
\end{tikzcd}~, \]
which by induction is fully faithful and has image as desired.

We now conclude with part \Cref{metacosm.input.get.stratn}.  Observe that the closed subcategories
\[
\left\{ \cZ_i := \lim^\rlax_{\llax.(^\leq i)}(\cE) \in \Cls_\cX \right\}_{i \in [n]}
\]
evidently assemble into a functor $[n] \xra{\Cref{the.stratn.of.rlax.lim}} \Cls_\cX$, which is clearly a prestratification and hence is a stratification by \Cref{obs.condn.star.vacuous.if.P.totally.ordered}.  Moreover, assertion \Cref{metacosm.input.get.stratn}\Cref{notational.consistency.for.ZD} is trivial, and assertion \Cref{metacosm.input.get.stratn}\Cref{notational.consistency.for.PhiC} follows from part \Cref{metacosm.input.describe.adjts}\Cref{metacosm.input.describe.nu} (applied to $\sD$ instead of $\pos$).  To prove part \Cref{metacosm.input.get.stratn}\Cref{gluing.in.X.is.mdrmy.in.E}, in light of the commutative diagram
\[ \begin{tikzcd}[row sep=1.5cm]
\cE_p
\arrow[hook]{rr}{\rho^p}
\arrow[hook]{rd}[sloped, swap]{\rho^p}
&
&
\cX
\arrow{rr}{\Phi_q}
\arrow{rd}[sloped]{\Phi_{[n]_{p//q}}}
&
&
\cE_q
\\
&
\lim^\rlax_{\llax.[n]_{p//q}}(\cE)
\arrow{rr}[swap]{\id}
\arrow[hook]{ru}[sloped]{\rho^{[n]_{p//q}}}
&
&
\lim^\rlax_{\llax.[n]_{p//q}}(\cE)
\arrow{ru}[sloped, swap]{\Phi_q}
\end{tikzcd}~, \]
we see that it suffices to assume that $p=0$ and $q=n$.  Moreover, applying part \Cref{metacosm.input.conservative} of \Cref{prop.metacosm.input.first.get.stratn.for.brax.one}, we see that it suffices to prove that the natural transformation of diagram \Cref{lax.comm.square.that.actually.commutes.to.show.mdrmy.is.as.expected.in.lim.rlax.llax.P.E} becomes an equivalence upon postcomposition with the functor
\[
\lim^\rlax_{\llax.\{0<n\}}(\cE)
\xra{\Phi_n}
\cE_n
~:
\]
that is, that the natural transformation in the diagram
\[ \begin{tikzcd}
&
\cZ_{n-1}
\arrow[hook]{d}{\rho^{[n-1]}}
\\
\cE_0
\arrow[hook]{ru}[sloped]{\rho^0}
\arrow[hook]{r}{\rho^0}[swap, yshift=-0.4cm]{\rotatebox{45}{$\Leftarrow$}}
\arrow[hook]{d}[swap]{\rho^0}
&
\cX
\arrow{d}{\Phi_n}
\\
\lim^\rlax_{\llax.\{0<n\}}(\cE)
\arrow{r}[swap]{\Phi_n}
&
\cE_n
\end{tikzcd} \]
is an equivalence.  By \Cref{lem.strictification}, every object of $\cZ_{n-1} := \lim^\rlax_{\llax.[n-1]}(\cE)$ is the limit of a diagram indexed by the finite poset $\sd([n-1])$; by our inductive hypothesis, for the image $\rho^0(X) \in \cZ_{n-1}$ of any object $X \in \cE_0$, this diagram is equivalent to its right Kan extension from the full subposet $\sd_0([n-1]) \subseteq \sd([n-1])$ on those objects $([i] \hookra [n-1]) \in \sd([n-1])$ whose image contains $0 \in [n-1]$.  Note that this finite limit is preserved by the composite $\cZ_{n-1} \xhookra{\rho^{[n-1]}} \cX \xra{\Phi_n} \cE_n$ of exact functors.  Because $(\{ 0 \} \hookra [n-1]) \in \sd_0([n-1])$ is an initial object, it follows that the composite functor $\cE_0 \xhookra{\rho^0} \cX \xra{\Phi_n} \cE_n$ is canonically equivalent to the monodromy functor $\cE_0 \ra \cE_n$, which proves the claim.
\end{proof}

\begin{proof}[Proof of \Cref{prop.metacosm.input.first.get.stratn}]
Observe the equivalence
\begin{equation}
\label{proof.of.metacosm.ingredient.rewrite.X.as.limit.over.Delta.over.P}
\cX
:=
\lim^\rlax_{\llax.\pos}(\cE)
\simeq
\lim_{([n] \da \pos)^\circ \in (\bDelta_{/\pos})^\op} \left( \lim^\rlax_{\llax.[n]} ( \cE ) \right)
~.
\end{equation}
It follows from \Cref{prop.metacosm.input.first.get.stratn.for.brax.n} that the functor $(\bDelta_{/\pos})^\op \xra{\limrlaxfam(\cE)} \Cat$ factors through the subcategory $\PrLSt \subset \Cat$: each $\infty$-category $\lim^\rlax_{\llax.[n]}(\cE)$ is presentable by its part \Cref{metacosm.input.presentable}, and for each morphism $[m] \ra [n]$ in $\bDelta_{/\pos}$ the corresponding restriction functor $\lim^\rlax_{\llax.[m]}(\cE) \la \lim^\rlax_{\llax.[n]}(\cE)$ preserves colimits by its parts \Cref{metacosm.input.conservative} \and \Cref{metacosm.input.cocts.restrn.to.any.subposet}.  Hence, the identification \Cref{proof.of.metacosm.ingredient.rewrite.X.as.limit.over.Delta.over.P} shows that $\cX$ is presentable; that is, we have proved part \Cref{metacosm.input.presentable}.  Using part \Cref{metacosm.input.conservative} of \Cref{prop.metacosm.input.first.get.stratn.for.brax.n}, equivalence \Cref{proof.of.metacosm.ingredient.rewrite.X.as.limit.over.Delta.over.P} also proves part \Cref{metacosm.input.conservative}. Thereafter, the evident functoriality of equivalence \Cref{proof.of.metacosm.ingredient.rewrite.X.as.limit.over.Delta.over.P} in the variable $\pos$ proves part \Cref{metacosm.input.cocts.restrn.to.any.subposet}.

We now prove part \Cref{metacosm.input.describe.adjts}\Cref{metacosm.input.describe.iL}. Given our fixed element $\sD \in \Down_\pos$, observe the adjunction
\[
\begin{tikzcd}[column sep=1.5cm]
\bDelta_{/\sD}
\arrow[hook, transform canvas={yshift=0.9ex}]{r}
\arrow[leftarrow, transform canvas={yshift=-0.9ex}]{r}[yshift=-0.2ex]{\bot}[swap]{(-)\cap \sD}
&
\bDelta_{/\pos}
\end{tikzcd}
\]
in which the right adjoint is given by intersection with $\sD \subseteq \pos$; thereafter, observe its opposite adjunction
\begin{equation}
\label{adjn.betw.Delta.over.D.op.and.Delta.over.P.op}
\begin{tikzcd}[column sep=1.5cm]
(\bDelta_{/\pos})^\op
\arrow[transform canvas={yshift=0.9ex}]{r}{((-)\cap \sD)^\op}
\arrow[hookleftarrow, transform canvas={yshift=-0.9ex}]{r}[yshift=-0.2ex]{\bot}
&
(\bDelta_{/\sD})^\op
\end{tikzcd}
~.
\end{equation}
Using the unit of the adjunction \Cref{adjn.betw.Delta.over.D.op.and.Delta.over.P.op}, we obtain a morphism
\begin{equation}
\label{morphism.in.Fun.bDelta.over.P.op.Cat.using.unit.of.adjn.betw.Delta.over.D.op.and.Delta.over.P.op}
\begin{tikzcd}[column sep=1.5cm]
(\bDelta_{/\pos})^\op
\arrow{rr}{\id}[swap, yshift=-0.3cm]{\Downarrow}
\arrow{rd}[sloped, swap]{((-)\cap \sD)^\op}
&
&
(\bDelta_{/\pos})^\op
\arrow{r}{\limrlaxfam(\cE)}
&
\Cat
\\
&
(\bDelta_{/\sD})^\op
\arrow[hook]{ru}
\end{tikzcd}
\end{equation}
in $\Fun((\bDelta_{/\pos})^\op,\Cat)$, which upon taking limits over $(\bDelta_{/\pos})^\op$ yields a morphism
\begin{equation}
\label{diagrammatic.restriction.morphism.over.Delta.over.P.from.limrlaxllaxE.over.bullet.to.same.over.bullet.intersect.D}
\lim_{([n] \da \pos)^\circ \in (\bDelta_{/\pos})^\op} \left( \lim^\rlax_{\llax.[n]}(\cE) \right)
\longra
\lim_{([n] \da \pos)^\circ \in (\bDelta_{/\pos})^\op} \left( \lim^\rlax_{\llax.([n] \cap \sD)}(\cE) \right)
~.
\end{equation}
On the one hand, the source of the morphism \Cref{diagrammatic.restriction.morphism.over.Delta.over.P.from.limrlaxllaxE.over.bullet.to.same.over.bullet.intersect.D} is identified as $\cX$ via equivalence \Cref{proof.of.metacosm.ingredient.rewrite.X.as.limit.over.Delta.over.P}.  On the other hand, because the functor $(\bDelta_{/\pos})^\op \xra{((-) \cap \sD)^\op} (\bDelta_{/\sD})^\op$ is initial (being a left adjoint), we may identify the target of the morphism \Cref{diagrammatic.restriction.morphism.over.Delta.over.P.from.limrlaxllaxE.over.bullet.to.same.over.bullet.intersect.D} as
\[
\lim_{([n] \da \pos)^\circ \in (\bDelta_{/\pos})^\op} \left( \lim^\rlax_{\llax.([n] \cap \sD)}(\cE) \right)
\simeq
\lim_{([n] \da \sD)^\circ \in (\bDelta_{/\sD})^\op} \left( \lim^\rlax_{\llax.[n]}(\cE) \right)
\simeq
\lim^\rlax_{\llax.\sD}(\cE)
~.
\]
Hence, the morphism \Cref{diagrammatic.restriction.morphism.over.Delta.over.P.from.limrlaxllaxE.over.bullet.to.same.over.bullet.intersect.D} is the restriction morphism
\[
\cX
:=
\lim^\rlax_{\llax.\pos}(\cE)
\xlongra{y}
\lim^\rlax_{\llax.\sD}(\cE)
~.
\]
We now make the following observations regarding the morphism \Cref{morphism.in.Fun.bDelta.over.P.op.Cat.using.unit.of.adjn.betw.Delta.over.D.op.and.Delta.over.P.op} in $\Fun((\bDelta_{/\pos})^\op,\Cat)$.
\begin{itemize}
\item
For each object $([n] \da \pos)^\circ \in (\bDelta_{/\pos})^\op$, the component of the morphism \Cref{morphism.in.Fun.bDelta.over.P.op.Cat.using.unit.of.adjn.betw.Delta.over.D.op.and.Delta.over.P.op} is the restriction functor
\begin{equation}
\label{component.at.n.over.P.is.lim.rlax.restricting.down.to.int.with.D}
\lim^\rlax_{\llax.[n]}(\cE)
\xlongra{y}
\lim^\rlax_{\llax.([n] \cap \sD)}(\cE)
~.
\end{equation}
By part \Cref{metacosm.input.describe.adjts}\Cref{metacosm.input.describe.iL} of \Cref{prop.metacosm.input.first.get.stratn.for.brax.n}, the functor \Cref{component.at.n.over.P.is.lim.rlax.restricting.down.to.int.with.D} admits a fully faithful left adjoint $i_L$, whose image consists of those objects $X \in \lim^\rlax_{\llax.[n]}(\cE)$ such that $\Phi_q(X) \simeq 0$ for all $q \in ([n] \cap (\pos \backslash \sD))$.
\item
For each morphism $([m] \da \pos)^\circ \ra ([n] \da \pos)^\circ$ in $(\bDelta_{/\pos})^\op$, i.e.\! for each commutative triangle
\[ \begin{tikzcd}
{[m]}
\arrow{rd}
&
&
{[n]}
\arrow{ll}
\arrow{ld}
\\
&
\pos
\end{tikzcd}~, \]
the component of the morphism \Cref{morphism.in.Fun.bDelta.over.P.op.Cat.using.unit.of.adjn.betw.Delta.over.D.op.and.Delta.over.P.op} is the commutative square
\begin{equation}
\label{comm.square.of.restrn.functors.from.m.to.n.cap.D}
\begin{tikzcd}
\lim^\rlax_{\llax.[m]}(\cE)
\arrow{r}{y}
\arrow{d}
&
\lim^\rlax_{\llax.([m] \cap \sD)}(\cE)
\arrow{d}
\\
\lim^\rlax_{\llax.[n]}(\cE)
\arrow{r}[swap]{y}
&
\lim^\rlax_{\llax.([n] \cap \sD)}(\cE)
\end{tikzcd}
\end{equation}
of restriction functors.  Moreover, the lax-commutative square
\[ \begin{tikzcd}
\lim^\rlax_{\llax.[m]}(\cE)
\arrow[hookleftarrow]{r}{i_L}[swap, yshift=-0.45cm]{\rotatebox{-30}{$\Leftarrow$}}
\arrow{d}
&
\lim^\rlax_{\llax.([m] \cap \sD)}(\cE)
\arrow{d}
\\
\lim^\rlax_{\llax.[n]}(\cE)
\arrow[hookleftarrow]{r}[swap]{i_L}
&
\lim^\rlax_{\llax.([n] \cap \sD)}(\cE)
\end{tikzcd} \]
determined by the commutative square \Cref{comm.square.of.restrn.functors.from.m.to.n.cap.D} is in fact commutative as a result of our characterization of both functors $i_L$.
\end{itemize}
Hence, we find that the morphism \Cref{morphism.in.Fun.bDelta.over.P.op.Cat.using.unit.of.adjn.betw.Delta.over.D.op.and.Delta.over.P.op} in $\Fun((\bDelta_{/\pos})^\op,\Cat)$ admits a left adjoint
\begin{equation}
\label{ladjt.of.morphism.in.Fun.bDelta.over.P.op.Cat.using.unit.of.adjn.betw.Delta.over.D.op.and.Delta.over.P.op}
\lim^\rlax_{\llax.(\bullet \cap \sD)}(\cE)
\longra
\limrlaxfam(\cE)
\end{equation}
whose components are fully faithful.  Therefore, upon taking limits over $(\bDelta_{/\pos})^\op$, we obtain a fully faithful left adjoint
\[
\begin{tikzcd}[column sep=1.5cm]
\lim^\rlax_{\llax.\sD}(\cE)
\arrow[hook, dashed, transform canvas={yshift=0.9ex}]{r}{i_L}
\arrow[leftarrow, transform canvas={yshift=-0.9ex}]{r}[yshift=-0.2ex]{\bot}[swap]{y}
&
\lim^\rlax_{\llax.\pos}(\cE)
\end{tikzcd}
~.
\]
In order to characterize its image, we note that by construction, for any $([n] \da \pos)^\circ \in (\bDelta_{/\pos})^\op$ we have a commutative square
\begin{equation}
\label{comm.square.iL.and.y.on.rlax.limits.including.restriction.to.D}
\begin{tikzcd}
\lim^\rlax_{\llax.\sD}(\cE)
\arrow[hook]{r}{i_L}
\arrow{d}[swap]{y}
&
\lim^\rlax_{\llax.\pos}(\cE)
\arrow{d}{y}
\\
\lim^\rlax_{\llax.([n]\cap \sD)}(\cE)
\arrow[hook]{r}[swap]{i_L}
&
\lim^\rlax_{\llax.[n]}(\cE)
\end{tikzcd}~.
\end{equation}
Taking $n=0$, the commutative square \Cref{comm.square.iL.and.y.on.rlax.limits.including.restriction.to.D} immediately implies that for any $X \in \lim^\rlax_{\llax.\sD}(\cE)$ and any $q \in \pos \backslash \sD$ we have $\Phi_q(i_L(X)) \simeq 0$.  On the other hand, given an object $X \in \lim^\rlax_{\llax.\pos}(\cE)$ such that $\Phi_q(X) \simeq 0$ whenever $q \in \pos \backslash \sD$, again using the commutative square \Cref{comm.square.iL.and.y.on.rlax.limits.including.restriction.to.D} with $n=0$, by part \Cref{metacosm.input.conservative} we see that the counit morphism $i_L y X \ra X$ is an equivalence.  So, we have proved part \Cref{metacosm.input.describe.adjts}\Cref{metacosm.input.describe.iL}.

Part \Cref{metacosm.input.describe.adjts}\Cref{metacosm.input.describe.nu} follows from an essentially identical argument to part \Cref{metacosm.input.describe.adjts}\Cref{metacosm.input.describe.iL}.

We now conclude with part \Cref{metacosm.input.get.stratn}.  We first observe that the closed subcategories
\[
\left\{ \cZ_p := \lim^\rlax_{\llax.(^\leq p)}(\cE) \in \Cls_\cX \right\}_{p \in \pos}
\]
evidently assemble into a functor $\pos \xra{\Cref{the.stratn.of.rlax.lim}} \Cls_\cX$, which is a prestratification by part \Cref{metacosm.input.conservative} and satisfies the stratification condition as a result of part \Cref{metacosm.input.describe.adjts}\Cref{metacosm.input.describe.iL} (applied to both $\pos$ and $(^\leq p)$).  Moreover, assertion \Cref{metacosm.input.get.stratn}\Cref{notational.consistency.for.ZD} follows from part \Cref{metacosm.input.conservative} (applied to $\sD$ instead of $\pos$), and assertion \Cref{metacosm.input.get.stratn}\Cref{notational.consistency.for.PhiC} follows from part \Cref{metacosm.input.describe.adjts}\Cref{metacosm.input.describe.nu} (applied to $\sD$ instead of $\pos$).  To prove part \Cref{metacosm.input.get.stratn}\Cref{gluing.in.X.is.mdrmy.in.E}, writing $\sC \in \Conv_\pos$ for the convex hull of the subset $\{p,q\} \subseteq \pos$ (i.e.\! the full subposet on those elements $r \in \pos$ such that there exist morphisms $p \ra r \ra q$), in light of the commutative diagram
\[ \begin{tikzcd}
\cE_p
\arrow[hook]{rr}{\rho^p}
\arrow[hook]{rd}[sloped, swap]{\rho^p}
&
&
\cX
\arrow{rr}{\Phi_q}
\arrow{rd}[sloped]{\Phi_\sC}
&
&
\cE_q
\\
&
\lim^\rlax_{\llax.\sC}(\cE)
\arrow{rr}[swap]{\id}
\arrow[hook]{ru}[sloped]{\rho^\sC}
&
&
\lim^\rlax_{\llax.\sC}(\cE)
\arrow{ru}[sloped, swap]{\Phi_q}
\end{tikzcd}~, \]
we see that it suffices to assume that $p \in \pos$ is initial and $q \in \pos$ is terminal.  Now, consider the full subposet $\sd_{p,q}(\pos) \subseteq \sd(\pos)$ consisting of those objects $([i] \hookra \pos) \in \sd(\pos)$ that contain both elements $p$ and $q$ in their image, and consider the morphisms
\begin{equation}
\label{ladjt.in.pshvs.of.Cats.on.sd.p.q.P}
\const_{\cE_p}
\longla
\limrlaxfam(\cE)
\end{equation}
in $\Fun(\sd_{p,q}(\pos)^\op,\Cat)$ whose components are given by restriction.  Because the inclusion $\sd_{p,q}(\pos) \subseteq \sd(\pos)$ is final (so that its opposite is initial) and moreover $\sd_{p,q}(\pos)$ has contractible $\infty$-groupoid completion (as it has an initial object), applying the functor $\lim_{\sd_{p,q}(\pos)^\op}$ to the morphism \Cref{ladjt.in.pshvs.of.Cats.on.sd.p.q.P} yields the morphism
\[
\cE_p
\xla{\Phi_p}
\lim^\rlax_{\llax.\pos}(\cE)
\]
in $\Cat$.  On the other hand, by \Cref{metacosm.input.get.stratn}\Cref{gluing.in.X.is.mdrmy.in.E} of \Cref{prop.metacosm.input.first.get.stratn.for.brax.n}, the morphism \Cref{ladjt.in.pshvs.of.Cats.on.sd.p.q.P} admits a right adjoint
\begin{equation}
\label{radjt.in.pshvs.of.Cats.on.sd.p.q.P}
\const_{\cE_p}
\longra
\limrlaxfam(\cE)
\end{equation}
in $\Fun(\sd_{p,q}(\pos)^\op,\Cat)$.  The component at the object $([1] \xra{\{p<q\}} \pos)^\circ \in \sd_{p,q}(\pos)^\op$ of the limiting cone of the morphism \Cref{ladjt.in.pshvs.of.Cats.on.sd.p.q.P} is the commutative square \Cref{comm.square.of.restrns.to.show.mdrmy.is.as.expected.in.lim.rlax.P.E}, and so the component at that same object of the limiting cone of the morphism \Cref{radjt.in.pshvs.of.Cats.on.sd.p.q.P} is the desired commutative square \Cref{lax.comm.square.that.actually.commutes.to.show.mdrmy.is.as.expected.in.lim.rlax.llax.P.E}.
\end{proof}

\subsection{The metacosm reconstruction theorem}
\label{subsection.metacosm}

In this subsection, we prove the metacosm reconstruction theorem as \Cref{metacosm.thm}.

\begin{definition}
\label{defn.Strat.P}
Let $\cX$ and $\cX'$ be $\pos$-stratified presentable stable $\infty$-categories. We define a morphism between them to be a left adjoint functor $\cX \ra \cX'$ satisfying the condition that for every $p \in \pos$ there exist (necessarily unique) factorizations
\[ \begin{tikzcd}
\cX
\arrow{r}
&
\cX'
\\
\cZ_p
\arrow[hook]{u}{i_L}
\arrow[dashed]{r}
&
\cZ_p'
\arrow[hook]{u}[swap]{i_L}
\end{tikzcd}
\qquad
\text{and}
\qquad
\begin{tikzcd}
\cX
\arrow{r}
\arrow{d}[swap]{y}
&
\cX'
\arrow{d}{y}
\\
\cZ_p
\arrow[dashed]{r}
&
\cZ_p'
\end{tikzcd}
~.
\]
In this way, we obtain an $\infty$-category
\[
\Strat_\pos
\]
that we refer to as that of \bit{$\pos$-stratified presentable stable $\infty$-categories}.
\end{definition}

\begin{observation}
\label{obs.forget.from.Strat.to.PrL.conservative}
The forgetful functor $\Strat_\pos \ra \PrLSt$ is conservative.
\end{observation}

\begin{notation}
\label{notn.stratn.of.one.rlaxlim.at.a.time}
For any $(\cE \da \pos) \in \LMod^{\rlax,L}_{\llax.\pos}(\PrSt)$, we write
\[
\limrlaxfam(\cE)
\in \Strat_\pos
\]
for the $\pos$-stratified presentable stable $\infty$-category $\lim^\rlax_{\llax.\pos}(\cE)$ of \Cref{prop.metacosm.input.first.get.stratn}.
\end{notation}

\begin{observation}
Given a morphism
\begin{equation}
\label{morphism.in.LMod.rlax.L.llax.pos.PrSt.that.gives.morphism.in.Strat.P}
\cE
\longra
\cE'
\end{equation}
in $\LMod^{\rlax,L}_{\llax.\pos}(\PrSt)$, the induced functor
\begin{equation}
\label{induced.morphism.on.rlax.lims.for.showing.that.lim.rlax.llax.bullet.defines.a.functor}
\lim^\rlax_{\llax.\pos}(\cE)
\longra
\lim^\rlax_{\llax.\pos}(\cE')
\end{equation}
lies in $\Strat_\pos$: in other words, we may upgrade \Cref{notn.stratn.of.one.rlaxlim.at.a.time} to a functor
\[
\LMod^{\rlax,L}_{\llax.\pos}(\PrSt)
\xra{\limrlaxfam}
\Strat_\pos
~.
\]
Indeed, the functor \Cref{induced.morphism.on.rlax.lims.for.showing.that.lim.rlax.llax.bullet.defines.a.functor} preserves colimits by parts \Cref{metacosm.input.conservative} and \Cref{metacosm.input.cocts.restrn.to.any.subposet} of \Cref{prop.metacosm.input.first.get.stratn}, it obviously commutes with the restriction functors $y$, and it commutes with their left adjoints $i_L$ by \Cref{prop.metacosm.input.first.get.stratn}\Cref{metacosm.input.describe.adjts}\Cref{metacosm.input.describe.iL}. We use this fact without further comment.
\end{observation}

\begin{observation}
For any $\pos$-stratified presentable stable $\infty$-category $\cX \in \Strat_\pos$, its gluing diagram
\[
\GD(\cX)
\in
\LMod_{\llax.\pos}
\]
is in fact a presentable stable left-lax left $\pos$-module. We use this fact without further comment.
\end{observation}

\begin{theorem}
\label{metacosm.thm}
There is a canonical adjunction
\begin{equation}
\label{adjn.of.metacosm.thm}
\begin{tikzcd}[column sep=1.5cm]
\Strat_\pos
\arrow[transform canvas={yshift=0.9ex}]{r}{\GD}
\arrow[hookleftarrow, transform canvas={yshift=-0.9ex}]{r}[yshift=-0.2ex]{\bot}[swap]{\limrlaxfam}
&
\LMod^{\rlax,L}_{\llax.\pos}(\PrSt)
\end{tikzcd}
\end{equation}
whose right adjoint is fully faithful, which is an equivalence whenever $\pos$ is down-finite.
\end{theorem}


\begin{proof}
Fix arbitrary objects $\cX \in \Strat_\pos$ and $(\cE \da \pos) \in \LMod^{\rlax,L}_{\llax.\pos}(\PrSt)$.  The adjunction \Cref{adjn.of.metacosm.thm} may be extracted from the commutative diagram
\begin{equation}
\label{diagram.of.hom.spaces.proving.metacosm.adjn}
\begin{tikzcd}
\hom_{\Strat_\pos} ( \cX , \limrlaxfam(\cE) )
\arrow[hook]{r}
\arrow[leftrightarrow]{dd}[sloped, anchor=north]{\sim}
&
\hom_{\PrLSt} ( \cX, \lim^\rlax_{\llax.\pos}(\cE) )
\arrow[hook]{r}
\arrow[leftrightarrow]{d}[sloped, anchor=south]{\sim}
&
\hom_\Cat ( \cX, \lim^\rlax_{\llax.\pos}(\cE) )
\arrow[leftrightarrow]{d}[sloped, anchor=south]{\sim}
\\
&
\hom_{\LMod^{\rlax,L}_{\llax.\pos}(\PrSt)} (\ul{\cX},\cE)
\arrow[hook]{r}
\arrow[leftrightarrow]{d}[sloped, anchor=south]{\sim}
&
\hom_{\LMod^\rlax_{\llax.\pos}} (\ul{\cX},\cE)
\\
\hom_{\LMod^{\llax,R}_{\llax.\pos}(\PrSt)} ( \cE , \GD(\cX) )
\arrow[hook]{r}
\arrow[leftrightarrow]{d}[sloped, anchor=north]{\sim}
&
\hom_{\LMod^{\llax,R}_{\llax.\pos}(\PrSt)} ( \cE , \ul{\cX} )
\\
\hom_{\LMod^{\rlax,L}_{\llax.\pos}(\PrSt)} ( \GD(\cX) , \cE )
\end{tikzcd}
\end{equation}

in $\Spaces$ that we explain presently.
\begin{itemize}

\item The equivalence in the right column of diagram \Cref{diagram.of.hom.spaces.proving.metacosm.adjn} follows from the adjunction $\const \adj \lim^\rlax_{\llax.\pos}$.

\item The notation $\LMod^{\llax,R}_{\llax.\pos}(\PrSt)$ has the evident meaning, analogous to the notation $\LMod^{\rlax,L}_{\llax.\pos}(\PrSt)$ introduced in \Cref{defn.LMod.rlax.L}.\footnote{It is also explained in \Cref{notn.systematic.for.lax.modules.valued.in.St.idem.and.PrSt.etc}.}

\item By parts \Cref{metacosm.input.conservative} and \Cref{metacosm.input.cocts.restrn.to.any.subposet} of \Cref{prop.metacosm.input.first.get.stratn}, a functor $\cX \ra \lim^\rlax_{\llax.\pos}(\cE)$ preserves colimits if and only if for every $p \in \pos$ the composite functor $\cX \ra \lim^\rlax_{\llax.\pos}(\cE) \ra \cE_p$ preserves colimits.  Hence, in diagram \Cref{diagram.of.hom.spaces.proving.metacosm.adjn} the equivalence in the right column factors as the upper equivalence in the middle column.

\item The lower equivalences in the left and middle columns of diagram \Cref{diagram.of.hom.spaces.proving.metacosm.adjn} follow directly from \Cref{lemma.ptwise.radjt.has.ptwise.ladjt}: over each object $p \in \pos$, these equivalences are obtained by passage between adjoints.

\item In diagram \Cref{diagram.of.hom.spaces.proving.metacosm.adjn}, we deduce the factorization of the composite equivalence in the middle column as the upper equivalence in the left column as follows.

\begin{itemize}

\item Given a morphism $\cX \ra \limrlaxfam(\cE)$ in $\Strat_\pos$, it is immediate that for each $p \in \pos$ there exists a factorization
\[ \begin{tikzcd}
\cX
\arrow{r}
\arrow{d}[swap]{\Phi_p}
&
\lim^\rlax_{\llax.\pos}(\cE)
\arrow{r}
&
\cE_p
\\
\cX_p
\arrow[dashed, bend right=10]{rru}
\end{tikzcd} \]
that is necessarily a left adjoint.  This proves the downwards factorization.

\item Suppose we are given a morphism $\ul{\cX} \la \cE$ in $\LMod^{\llax,R}_{\llax.\pos}(\PrSt)$ that admits a factorization
\begin{equation}
\label{factorization.through.gluing.diagram}
\begin{tikzcd}
\ul{\cX}
&
\cE
\arrow{l}
\arrow[dashed]{ld}
\\
\GD(\cX)
\arrow[hook]{u}{\ff}
\end{tikzcd}~.
\end{equation}
Fix any $p \in \pos$, and note that the existence of the factorization \Cref{factorization.through.gluing.diagram} implies (and in fact is equivalent to) the existence for every $q \in \pos$ of a factorization
\begin{equation}
\label{factorization.through.strata.of.composite.with.projection.to.fiber}
\begin{tikzcd}
\cX
\arrow{r}
\arrow{d}[swap]{\Phi_q}
&
\lim^\rlax_{\llax.\pos}(\cE)
\arrow{d}{\Phi_q}
\\
\cX_q
\arrow[dashed]{r}
&
\cE_q
\end{tikzcd}~.
\end{equation}
We make the following observations.
\begin{itemize}

\item

In the diagram
\[ \begin{tikzcd}
{\displaystyle\prod_{q \not\leq p} \cX_q}
\arrow{r}
&
{\displaystyle\prod_{q \not\leq p} \cE_q}
\\
\cX
\arrow{u}{ ( \Phi_q)_{q \not\leq p} }
\arrow{r}
&
\lim^\rlax_{\llax.\pos}(\cE)
\arrow{u}[swap]{ ( \Phi_q)_{q \not\leq p} }
\\
\cZ_p
\arrow[hook]{u}{i_L}
\arrow[dashed]{r}
&
\lim^\rlax_{\llax.(^\leq p)}(\cE)
\arrow[hook]{u}[swap]{i_L}
\end{tikzcd}~, \]
the upper square commutes as a result of the factorizations \Cref{factorization.through.strata.of.composite.with.projection.to.fiber} and the left vertical composite is zero as a result of the stratification condition.  Because the right vertical composite is a fiber sequence by \Cref{prop.metacosm.input.first.get.stratn}\Cref{metacosm.input.describe.adjts}\Cref{metacosm.input.describe.iL}, we obtain the indicated factorization.

\item

The existence of a factorization
\[ \begin{tikzcd}
\cX
\arrow{r}
\arrow{d}[swap]{y}
&
\lim^\rlax_{\llax.\pos}(\cE)
\arrow{d}{y}
\\
\cZ_p
\arrow[dashed]{r}
&
\lim^\rlax_{\llax.(^\leq p)}(\cE)
\end{tikzcd} \]
is equivalent to the assertion that if an object $X \in \cX$ is in the kernel of the functor $\cX \xra{y} \cZ_p$ then it is sent to zero under the composite $\cX \ra \lim^\rlax_{\llax.\pos}(\cE) \xra{y} \lim^\rlax_{\llax.(^\leq p)}(\cE)$.  This latter assertion follows from the diagram
\[ \begin{tikzcd}
\cX
\arrow{r}
\arrow{d}[swap]{y}
&
\lim^\rlax_{\llax.\pos}(\cE)
\arrow{d}{y}
\\
\cZ_p
\arrow{d}[swap]{(\Phi_q)_{q \leq p}}
&
\lim^\rlax_{\llax.(^\leq p)}(\cE)
\arrow{d}{(\Phi_q)_{q \leq p}}
\\
{\displaystyle\prod_{q \leq p} \cX_q}
\arrow{r}
&
{\displaystyle\prod_{q \leq p} \cE_q}
\end{tikzcd}~, \]
which commutes on account of the factorizations \Cref{factorization.through.strata.of.composite.with.projection.to.fiber} and in which the lower right vertical functor is conservative by \Cref{prop.metacosm.input.first.get.stratn}\Cref{metacosm.input.conservative} (applied to the poset $(^\leq p)$).

\end{itemize}
It follows that our chosen morphism $\ul{\cX} \la \cE$ in $\LMod^{\llax,R}_{\llax.\pos}(\PrSt)$ corresponds to a morphism not just in $\PrLSt$ but in $\Strat_\pos$: i.e., it proves the upwards factorization.

\end{itemize}

\end{itemize}

We now prove that the counit of the adjunction \Cref{adjn.of.metacosm.thm} is an equivalence.  Unwinding the equivalences of diagram \Cref{diagram.of.hom.spaces.proving.metacosm.adjn}, we see that the counit is given by the following sequence of operations.
\begin{itemize}
\item Begin with the counit morphism
\begin{equation}
\label{first.morphism.in.studying.counit.in.metacosm.adjn}
\ul{\lim^\rlax_{\llax.\pos}(\cE)}
\longra
\cE
\end{equation}
in $\LMod^\rlax_{\llax.\pos}$ of the adjunction $\const \adj \lim^\rlax_{\llax.\pos}$, which lies in $\LMod^{\rlax,L}_{\llax.\pos}(\PrSt)$: over each $p \in \pos$ it restricts as the left adjoint
\[
\lim^\rlax_{\llax.\pos}(\cE)
\xra{\Phi_p}
\cE_p
~. \]
\item Use \Cref{lemma.ptwise.radjt.has.ptwise.ladjt} to pass to the corresponding morphism
\begin{equation}
\label{second.morphism.in.studying.counit.in.metacosm.adjn}
\ul{\lim^\rlax_{\llax.\pos}(\cE)}
\longla
\cE
\end{equation}
in $\LMod^{\llax,R}_{\llax.\pos}(\PrSt)$ to the morphism \Cref{first.morphism.in.studying.counit.in.metacosm.adjn} in $\LMod^{\rlax,L}_{\llax.\pos}(\PrSt)$, which restricts over each $p \in \pos$ as the right adjoint
\[
\lim^\rlax_{\llax.\pos}(\cE)
\xlonghookla{\rho^p}
\cE_p
~.
\]
\item Observe the factorization of the morphism \Cref{second.morphism.in.studying.counit.in.metacosm.adjn} in $\LMod^{\llax,R}_{\llax.\pos}(\PrSt)$ as
\begin{equation}
\label{factorizn.from.E.to.GD.of.its.own.glued.cat}
\begin{tikzcd}
\ul{\lim^\rlax_{\llax.\pos}(\cE)}
&
\cE
\arrow{l}
\arrow[dashed]{ld}
\\
\GD(\limrlaxfam(\cE) )
\arrow[hook]{u}
\end{tikzcd}~.
\end{equation}
\item Use \Cref{lemma.ptwise.radjt.has.ptwise.ladjt} to pass to the corresponding morphism
\begin{equation}
\label{the.counit.of.the.metacosm.adjn}
\GD(\limrlaxfam(\cE) )
\longra
\cE
\end{equation}
in $\LMod^{\rlax,L}_{\llax.\pos}(\PrSt)$ to the factorization of diagram \Cref{factorizn.from.E.to.GD.of.its.own.glued.cat} in $\LMod^{\llax,R}_{\llax.\pos}(\PrSt)$.
\end{itemize}
Evidently, the factorization of diagram \Cref{factorizn.from.E.to.GD.of.its.own.glued.cat} restricts as an equivalence over each $p \in \pos$.  In fact, it is an equivalence by \Cref{prop.metacosm.input.first.get.stratn}\Cref{metacosm.input.get.stratn}\Cref{gluing.in.X.is.mdrmy.in.E}.  Hence the counit \Cref{the.counit.of.the.metacosm.adjn} is also an equivalence.

We now study the unit of the adjunction \Cref{adjn.of.metacosm.thm}.  In order to verify that its component at the object $\cX \in \Strat_\pos$ to be an equivalence, by \Cref{obs.forget.from.Strat.to.PrL.conservative} it suffices to show that the underlying morphism
\begin{equation}
\label{underlying.morphism.in.PrL.of.unit.of.metacosm.adjn}
\cX
\longra
\lim^\rlax_{\llax.\pos}(\GD(\cX))
\end{equation}
in $\PrLSt$ is an equivalence.  Unwinding the equivalences of diagram \Cref{diagram.of.hom.spaces.proving.metacosm.adjn}, we see that the morphism \Cref{underlying.morphism.in.PrL.of.unit.of.metacosm.adjn} is the composite
\[
\cX
\longra
\lim^\rlax_{\llax.\pos}(\cX)
\longra
\lim^\rlax_{\llax.\pos}(\GD(\cX))
\]
in which the first functor is the unit of the adjunction $\const \adj \lim^\rlax_{\llax.\pos}$ and the second morphism is obtained by applying \Cref{lemma.ptwise.radjt.has.ptwise.ladjt} to the defining morphism $\ul{\cX} \hookla \GD(\cX)$ in $\LMod^{\llax,R}_{\llax.\pos}(\PrSt)$ (which restricts over each $p \in \pos$ as the right adjoint $\cX \xhookla{\rho^p} \cX_p$) and then applying the functor $\lim^\rlax_{\llax.\pos}$.

We now prove that the morphism \Cref{underlying.morphism.in.PrL.of.unit.of.metacosm.adjn} is an equivalence under the assumption that $\pos$ is finite.  We proceed by induction on the number of elements of $\pos$, the base case where $\pos=\es$ being trivial.  So, choose any maximal element $\infty \in \pos$, and write $\pos' := \pos \backslash \{ \infty\} \in \Down_\pos$ for its complement.  This defines a functor $\pos \xra{\pi} [1]$ with $\pi^{-1}(0) = \pos'$ and $\pi^{-1}(1) = \{\infty\}$.  Taking pushforwards of stratifications along $\pi$ via \Cref{prop.pushfwd.stratn} allows us to consider the morphism \Cref{underlying.morphism.in.PrL.of.unit.of.metacosm.adjn} as lying in $\Strat_{[1]}$.  To show that the morphism \Cref{underlying.morphism.in.PrL.of.unit.of.metacosm.adjn} is an equivalence, by \Cref{obs.clsd.subcat.gives.recollement} and \Cref{lem.reconstrn.for.recollement} it suffices to show that the lax-commutative square
\begin{equation}
\label{a.priori.lax.comm.square.for.proving.reconstrn.for.finite.posets}
\begin{tikzcd}[row sep=1.5cm]
\cX_{\pos'}
\arrow[hook]{r}{i_R}[swap, xshift=1.7cm, yshift=-0.9cm]{\rotatebox{30}{$\Leftarrow$}}
\arrow{d}[sloped, anchor=north]{\sim}
&
\cX
\arrow{r}{p_L}
&
\cX_\infty
\arrow{d}[sloped, anchor=south]{\sim}
\\
\lim^\rlax_{\llax.\pos'}(\GD(\cX))
\arrow[hook]{r}[swap]{i_R}
&
\lim^\rlax_{\llax.\pos}(\GD(\cX))
\arrow{r}[swap]{p_L}
&
\lim^\rlax_{\llax.\{\infty\}}(\GD(\cX))
\end{tikzcd}
\end{equation}
(whose left vertical morphism is an equivalence by induction) commutes.  By \Cref{lem.strictification}, every object of $\cX_{\pos'}$ is a limit indexed over the finite poset $\sd(\pos')$ of objects lying in the images of the fully faithful inclusions $\cX_p \xhookra{\rho^p} \cX_{\pos'}$ for elements $p \in \pos'$; because all functors in the diagram \Cref{a.priori.lax.comm.square.for.proving.reconstrn.for.finite.posets} are exact, it suffices to show that its natural transformation is an equivalence when restricted along each such fully faithful inclusion.  After this restriction, the source is precisely the gluing functor
\[
\Gamma^p_\infty
:
\cX_p
\xlonghookra{\rho^p}
\cX
\xra{\Phi_\infty}
\cX_\infty
~;
\]
by \Cref{prop.metacosm.input.first.get.stratn}\Cref{metacosm.input.get.stratn}\Cref{gluing.in.X.is.mdrmy.in.E} the target is (canonically equivalent to) the gluing functor $\Gamma^p_\infty$ as well, and unwinding the construction of the morphism \Cref{underlying.morphism.in.PrL.of.unit.of.metacosm.adjn} we see that the natural transformation in diagram \Cref{a.priori.lax.comm.square.for.proving.reconstrn.for.finite.posets} is indeed an equivalence.  So when $\pos$ is finite the morphism \Cref{underlying.morphism.in.PrL.of.unit.of.metacosm.adjn} is indeed an equivalence.

We now prove that the morphism \Cref{underlying.morphism.in.PrL.of.unit.of.metacosm.adjn} is an equivalence under the assumption that $\pos$ is down-finite.  Let us write $\Down^\fin_\pos \subseteq \Down_\pos$ for the full subposet on the finite down-closed subsets of $\pos$.  Consider the composite
\begin{equation}
\label{composite.in.pshvs.of.cats.over.Down.fin.P}
\const_\cX
\longra
\cZ_\bullet
\longra
\limrlaxfam(\GD(\cX))
\end{equation}
in $\Fun ( ( \Down^\fin_\pos)^\op , \PrLSt)$, in which
\begin{itemize}
\item the functor $\cZ_\bullet$ takes a morphism $\sD_0^\circ \ra \sD_1^\circ$ in $(\Down^\fin_\pos)^\op$ corresponding to a morphism $\sD_0 \la \sD_1$ in $\Down^\fin_\pos$ to the functor $\cZ_{\sD_0} \xra{y} \cZ_{\sD_1}$,
\item the functor $\limrlaxfam(\GD(\cX))$ takes a morphism $\sD_0^\circ \ra \sD_1^\circ$ in $(\Down^\fin_\pos)^\op$ corresponding to a morphism $\sD_0 \la \sD_1$ in $\Down^\fin_\pos$ to the restriction functor
\[
\lim^\rlax_{\llax.\sD_0}(\GD(\cX))
\longra
\lim^\rlax_{\llax.\sD_1}(\GD(\cX))
\]
(recall that this lies in $\PrLSt$ by parts \Cref{metacosm.input.presentable} and \Cref{metacosm.input.cocts.restrn.to.any.subposet} of \Cref{prop.metacosm.input.first.get.stratn}),
\item the component at $\sD^\circ \in (\Down^\fin_\pos)^\op$ of the first morphism is the functor $\cX \xra{y} \cZ_\sD$, and
\item the component at $\sD^\circ \in (\Down^\fin_\pos)^\op$ of the second morphism is the functor
\[
\cZ_\sD
\longra
\lim^\rlax_{\llax.\sD}(\GD(\cX))
\]
obtained as the instance of the functor \Cref{underlying.morphism.in.PrL.of.unit.of.metacosm.adjn} in the case of the restricted stratification of \Cref{obs.restricted.stratn.over.D}.
\end{itemize}
Applying the functor $\lim_{(\Down^\fin_\pos)^\op}$ to the composite \Cref{composite.in.pshvs.of.cats.over.Down.fin.P}, we obtain the upper composite in the commutative diagram
\begin{equation}
\label{factorizn.of.und.fctr.of.metacosm.unit.through.limit.of.closeds}
\begin{tikzcd}
\cX
\arrow{r}
\arrow{rrd}[sloped, swap]{\Cref{underlying.morphism.in.PrL.of.unit.of.metacosm.adjn}}
&
\lim_{\sD^\circ \in (\Down^\fin_\pos)^\op} \cZ_\sD
\arrow{r}
&
\lim_{\sD^\circ \in (\Down^\fin_\pos)^\op} \left( \lim^\rlax_{\llax.\sD}(\GD(\cX)) \right)
\\
&
&
\lim^\rlax_{\llax.\pos}(\GD(\cX))
\arrow{u}
\end{tikzcd}
\end{equation}
in $\PrLSt$.  Because $\pos$ is down-finite, the canonical morphism
\[
\colim
\left(
\Down^\fin_\pos
\xra{\fgt}
\Cat
\right)
\xlongra{\sim}
\pos
\]
in $\Cat$ is an equivalence.  This implies that in diagram \Cref{factorizn.of.und.fctr.of.metacosm.unit.through.limit.of.closeds}, the upper left horizontal morphism is an equivalence (by definition of a prestratification) and also the right vertical morphism is an equivalence (note that $\Down^\fin_\pos$ is filtered).  Meanwhile, because each $\sD \in \Down^\fin_\pos$ is finite, the second morphism in the composite \Cref{composite.in.pshvs.of.cats.over.Down.fin.P} is an equivalence, which implies that the upper right horizontal morphism in diagram \Cref{factorizn.of.und.fctr.of.metacosm.unit.through.limit.of.closeds} is an equivalence.  So the morphism \Cref{underlying.morphism.in.PrL.of.unit.of.metacosm.adjn} is an equivalence.
\end{proof}

\subsection{Strict stratifications}
\label{subsection.strict.stratns}

In this brief subsection, we lay out the general theory of strict stratifications.

\needspace{2\baselineskip}
\begin{definition}
\label{defn.strict.stratns}
\begin{enumerate}
\item[]

\item We say that
$
\mathscr{F}
\in
\LMod^{\rlax,L}_{\llax.\pos}(\PrSt)
$
is \bit{strict} if it lies in the full subcategory
\[
\LMod^{\rlax,L}_{\pos}(\PrSt)
\subseteq
\LMod^{\rlax,L}_{\llax.\pos}(\PrSt)
~.
\]

\item We say that $\cX \in \Strat_\pos$ is \bit{strict} if it is convergent (\Cref{defn.convergent.stratn}) and moreover its gluing diagram $\GD(\cX) \in \LMod^{\rlax,L}_{\llax.\pos}(\PrSt)$ is strict.

\end{enumerate}
\end{definition}

\begin{observation}
\label{obs.stratn.strict.iff.all.objects.strict}
Note that $\cX \in \Strat_\pos$ is strict if and only if it is convergent and its gluing functors strictly compose, i.e.\! for every composable sequence $p \ra q \ra r$ in $\pos$ the morphism
\[
\Gamma^p_r
\xra{\eta_q}
\Gamma^q_r \Gamma^p_q
\]
in $\Fun(\cX_p,\cX_r)$ is an equivalence.  It follows that $\cX$ is strict if and only if every object $X \in \cX$ is strict (\Cref{defn.strict.objects}).
\end{observation}

\begin{remark}
Choose any $\mathscr{F} \in \LMod^{\rlax,L}_{\llax.\pos}(\PrSt)$.  Considering it as an object $\mathscr{F} \in \LMod_{\llax.\pos}$, through \Cref{lem.strictification} we obtain an object $\Strict(\mathscr{F}) \in \LMod_{\sd(\pos)}$ and an equivalence
\[
\lim^\rlax_{\llax.\pos}(\mathscr{F})
\simeq
\lim_{\sd(\pos)}(\Strict(\mathscr{F}))
~.
\]
However, in contrast with \Cref{obs.strict.microcosm.gluing.diagram.iff.factors.from.sd.to.TwAr}, the strictness of $\mathscr{F}$ is \textit{not} equivalent to the existence of a factorization
\begin{equation}
\label{factorizn.of.strictification.of.possibly.strict.diagram.in.Pr.St}
\begin{tikzcd}
\sd(\pos)
\arrow{r}{\Strict(\mathscr{F})}
\arrow{d}[swap]{(\min \ra \max)}
&
\Cat
\\
\TwAr(\pos)
\arrow[dashed]{ru}
\end{tikzcd}~.
\end{equation}
This distinction is already visible when $\pos = [2]$, in which case the factorization \Cref{factorizn.of.strictification.of.possibly.strict.diagram.in.Pr.St} exists if and only if the $\infty$-category $\mathscr{F}_2$ is an $\infty$-groupoid.
\end{remark}

\begin{observation}
\label{obs.TwAr.is.localization.of.sd.in.loc.coCart}
The commutative triangle
\begin{equation}
\label{min.to.max.from.sd.P.to.TwAr.P}
\begin{tikzcd}
\sd(\pos)
\arrow{rr}{(\min \ra \max)}
\arrow{rd}[sloped, swap]{\max}
&
&
\TwAr(\pos)
\arrow{ld}[sloped, swap]{t}
\\
&
\pos
\end{tikzcd}
\end{equation}
defines a morphism in $\loc.\coCart_\pos$, and moreover $\TwAr(\pos) \in \coCart_\pos \subseteq \loc.\coCart_\pos$. Moreover, by \Cref{lem.sd.P.localizes.onto.TwAr.P} (recall \Cref{defn.iso.min.and.or.max}), the functor
\[
\sd(\pos) \xra{(\min \ra \max)} \TwAr(\pos)
\]
is precisely the localization at the comparison morphisms in the locally cocartesian fibration $\sd(\pos) \xra{\max} \pos$ as well as their locally cocartesian pushforwards. It follows that the morphism \Cref{min.to.max.from.sd.P.to.TwAr.P} is the initial morphism from $\sd(\pos) \in \loc.\coCart_\pos$ to an object of the full subcategory $\coCart_\pos \subseteq \loc.\coCart_\pos$.
\end{observation}

\begin{observation}
\label{obs.strict.stratn.gives.macrocosm.regluing.over.sdP}
By \Cref{obs.TwAr.is.localization.of.sd.in.loc.coCart}, for any $\cE \in \coCart_\pos$ we have an equivalence
\[
\lim^\rlax_\pos(\cE)
:=
\Fun^\cocart_{/\pos} ( \sd(\pos) , \cE)
\xla[\sim]{(\min \ra \max)^*}
\Fun^\cocart_{/\pos} ( \TwAr(\pos) , \cE )
=:
\Gamma_{\pos^\op} \left( \cE^\cocartdual \right)
~.
\]
In particular, if $\cX \in \Strat_\pos$ is strict, then taking $\cE = \GD(\cX)$ gives a canonical equivalence
\[
\cX
\xlongra{\sim}
\Gamma_{\pos^\op} \left( \GD(\cX)^\cocartdual \right)
~.
\]
\end{observation}

\section{Variations on the metacosm reconstruction theorem}
\label{section.variations}

In this section, we provide three variations on metacosm reconstruction (\Cref{intro.thm.cosms}\Cref{intro.main.thm.metacosm}, proved as \Cref{metacosm.thm}). It is organized as follows.
\begin{itemize}

\item[\Cref{subsection.closed.split.thick.subcats}:] We recall some preliminary notions regarding various sorts of subcategories of an idempotent-complete stable $\infty$-category that is not necessarily presentable.

\item[\Cref{subsection.stable.stratns}:] We extend our theory of stratifications to the case of idempotent-complete stable $\infty$-categories that are not necessarily presentable; for disambiguation, we refer to these as \textit{stable stratifications}. We establish metacosm reconstruction for stable stratifications over finite posets as \Cref{thm.stable.metacosm}.

\item[\Cref{subsection.strict.metacosm}:] Our definitions of morphisms in the $\infty$-categories of (resp.\! stable) stratifications require commutativity for $i_L$ and $y$. We show as \Cref{thm.strict.metacosm} that additionally requiring commutativity for $i_R$ corresponds to strict (as opposed to possibly right-lax) morphisms between left-lax left modules over our poset. We refer to such morphisms between stratifications as \textit{strict}.

\item[\Cref{subsection.reflection}:] We establish the theory of \textit{reflection} (as discussed in \Cref{subsection.verdier}) for (resp.\! stable) stratifications over a finite poset: this is a dual form of reconstruction, which is functorial for strict morphisms between stratifications. We begin by establishing reflection for stable stratifications (which are the more natural context for reflection) as \Cref{thm.reflection.for.stable.stratns}. Using this, we establish reflection for stratifications (i.e.\! \Cref{intro.thm.reflection}) as \Cref{cor.reflection.for.presentable.stratns}.\footnote{More precisely, we prove part \Cref{intro.reflection.thm.metacosm} of \Cref{intro.thm.reflection}; parts \Cref{intro.reflection.thm.macrocosm}-\Cref{intro.reflection.thm.nanocosm} follow trivially therefrom. (We have stated \Cref{intro.thm.reflection} in four parts primarily to highlight the parallel with \Cref{intro.thm.cosms}.)} We also give formulas expressing the gluing functors and reflected gluing functors in terms of each other as \Cref{prop.Gamma.check.as.tfib.of.Gammas.and.Gamma.as.tcofib.of.Gamma.checks}.

\end{itemize}

\begin{local}
In this section, we fix a poset $\pos$ and an idempotent-complete stable $\infty$-category $\cC$.
\end{local}

\begin{remark}
\label{rmk.stable.always.idempotent.cplt.for.ease.of.language}
It is straightforward to treat the more general case of stable $\infty$-categories that are not necessarily idempotent-complete. We restrict to idempotent-complete stable $\infty$-categories merely to ease our language (e.g.\! so that we can recover $\cC \simeq \Ind(\cC)^\omega \subseteq \Ind(\cC)$ as the compact objects of its ind-completion).
\end{remark}

\begin{notation}
\label{notn.systematic.for.lax.modules.valued.in.St.idem.and.PrSt.etc}
We extend the notation $\LMod^{\rlax,L}_{\llax.\pos}(\PrSt)$ of \Cref{defn.LMod.rlax.L} to a systematic notational scheme for the various $\infty$-categories of lax left $\pos$-modules that appear in this section.
\begin{itemize}

\item The subscript on $\LMod$ indicates the handedness of the lax left $\pos$-modules that we consider.

\item The parenthetical indicates the restrictions placed both on the fibers and monodromy functors of objects as well as on the fiberwise behavior of morphisms. (Those that arise are $\St^\idem$, $\PrLomegaSt$, $\PrLSt$, and $\PrSt$.)

\item A superscript $\llax$ or $\rlax$ on $\LMod$ indicates the handedness of the laxness that we allow for the morphisms. (The absence of either of these indicates that we require strictly $\pos$-equivariant morphisms.)

\item A superscript $L$ on $\LMod$ indicates that morphisms are additionally required to be fiberwise left adjoints. (This will only arise in the case that the parenthetical is $\PrSt$.)

\end{itemize}
\end{notation}

\subsection{Closed, split, and thick subcategories}
\label{subsection.closed.split.thick.subcats}

In this subsection, we recall some preliminary notions regarding various sorts of subcategories of an idempotent-complete stable $\infty$-category that is not necessarily presentable.

\begin{definition}
\label{defn.thick.split.closed}
A full stable subcategory $\cZ \subseteq \cC$ is called
\begin{enumerate}

\item

\bit{thick} if it is idempotent-complete,

\item

\bit{split} if it is thick and its inclusion extends to a diagram
\[ \begin{tikzcd}[column sep=1.5cm]
\cZ
\arrow[hook, transform canvas={yshift=0.9ex}]{r}
\arrow[leftarrow, dashed, transform canvas={yshift=-0.9ex}]{r}[yshift=-0.2ex]{\bot}
&
\cC
\end{tikzcd}~, \]
and

\item\label{defn.small.closed}
\bit{closed} if it is thick and its inclusion extends to a diagram
\[ \begin{tikzcd}[column sep=1.5cm]
\cZ
\arrow[hook, bend left]{r}
\arrow[dashed,leftarrow]{r}[transform canvas={yshift=0.05cm}]{\bot}[swap,transform canvas={yshift=-0.05cm}]{\bot}
\arrow[dashed,bend right, hook]{r}
&
\cC
\end{tikzcd}
~. \]

\end{enumerate}
These various sorts of full stable subcategories of $\cC$ assemble into posets ordered by inclusion, which we respectively denote by
\[
\thicksub_\cC
~,
\qquad
\splitsub_\cC
~,
\qquad
\text{and}
\qquad
\clssub_\cC
~.
\]
\end{definition}

\begin{observation}
\label{obs.cls.and.Cls}
If $\cC$ is presentable, then there is a canonical equivalence
\[
\Cls_\cC
\simeq
\clssub_\cC
~.\footnote{It follows that the terminology of \Cref{defn.thick.split.closed}\Cref{defn.small.closed} is unambiguous.}
\]
\end{observation}


\begin{observation}
We record a number of basic facts surrounding \Cref{defn.thick.split.closed}, which we thereafter use without further comment.\footnote{Many of these facts have already been discussed in \Cref{subsection.add.loc.invts}.}
\begin{enumerate}

\item There are fully faithful inclusions
\[
\clssub_\cC
\longhookra
\splitsub_\cC
\longhookra
\thicksub_\cC
~.
\]

\item The poset $\thicksub_\cC$ has all colimits.

\item Ind-completion defines a fully faithful colimit-preserving functor
\[
\thicksub_\cC
\xlonghookra{\Ind}
\Cls_{\Ind(\cC)}
~,
\]
whose image consists of those closed subcategories $\cZ \in \Cls_{\Ind(\cC)}$ that are compactly generated.\footnote{The inclusion $\cZ \xlonghookra{i_L} \Ind(\cC)$ automatically preserves compact objects, as its right adjoint $\Ind(\cC) \xra{y} \cZ$ preserves colimits.}

\item The image of the composite functor 
\[
\splitsub_\cC
\longhookra
\thicksub_\cC
\xlonghookra{\Ind}
\Cls_{\Ind(\cC)}
\]
consists of those closed subcategories $\cZ \in \Cls_{\Ind(\cC)}$ such that the functor
\[
\cZ
\xlonghookra{i_R}
\Ind(\cC)
\]
preserves colimits, or equivalently such that the composite functor
\[
\Ind(\cC)
\xlongra{y}
\cZ
\xlonghookra{i_R}
\Ind(\cC)
\]
preserves colimits.\footnote{This implies that $\cZ$ is compactly generated, with compact objects the image of the composite $\cC \simeq \Ind(\cC)^\omega \hookra \Ind(\cC) \xra{y} \cZ$.}

\item The image of the composite functor
\[
\clssub_\cC
\longhookra
\splitsub_\cC
\longhookra
\thicksub_\cC
\xlonghookra{\Ind}
\Cls_{\Ind(\cC)}
\]
consists of those closed subcategories $\cZ \in \Cls_{\Ind(\cC)}$ such that the functor
\[
\cZ
\xlonghookra{i_R}
\Ind(\cC)
\]
preserves colimits and compact objects, or equivalently such that the composite functor
\[
\Ind(\cC)
\xlongra{y}
\cZ
\xlonghookra{i_R}
\Ind(\cC)
\]
preserves colimits and compact objects.

\end{enumerate}
\end{observation}

\begin{notation}
\label{notation.for.thick.closure}
Given a set $\{ \cZ_s \in \thicksub_\cC \}_{s \in S}$ of thick subcategories of $\cC$, we write
\[
\brax{ \cZ_s }^\thick_{s \in S} \in \thicksub_\cC
\]
for the thick subcategory that they generate, i.e.\! the colimit of the functor $S \xra{\cZ_\bullet} \thicksub_\cC$.
\end{notation}

\begin{notation}
Given a thick subcategory $\cZ \in \thicksub_\cC$, we may write
\[
\cC /^{\St^\idem} \cZ
\in
\St^\idem
\]
for the idempotent-complete stable quotient of $\cC$ by $\cZ$, i.e.\! the cofiber of the inclusion in $\St^\idem$. However, we usually simply write
\[
\cC / \cZ
:=
\cC /^{\St^\idem} \cZ
\]
for this, which is unambiguous since we restrict our attention to the idempotent-complete context (recall \Cref{rmk.stable.always.idempotent.cplt.for.ease.of.language}) and due to \Cref{obs.prbl.stable.quotient.is.idem.cplt.stable.quotient}.
\end{notation}

\begin{remark}
Concretely, the idempotent-complete stable quotient of $\cC$ by a thick subcategory $\cZ \in \thicksub_\cC$ may be realized as the full subcategory
\[
\cC /^{\St^\idem} \cZ
\simeq
(\Ind(\cC)/\Ind(\cZ))^\omega \subseteq \Ind(\cC)/\Ind(\cZ)
\]
of compact objects of the corresponding presentable quotient.\footnote{By contrast, the stable quotient $\cC /^\St \cZ$ (i.e.\! the cofiber of the inclusion in $\St$) may be realized as the image of the composite
\[
\cC
\simeq
\Ind(\cC)^\omega
\longhookra
\Ind(\cC)
\xra{p_L}
\Ind(\cC)/\Ind(\cZ)
~;
\]
its idempotent-completion recovers $\cC /^{\St^\idem} \cZ$.} On the other hand, the idempotent-complete stable quotient of $\cC$ by a split subcategory $\cZ \in \splitsub_\cC$ may be realized more simply as $\ker(\cC \ra \cZ)$.\footnote{In particular, in this case the canonical morphism $\cC /^\St \cZ \ra \cC /^{\St^\idem} \cZ$ is an equivalence.}
\end{remark}

\begin{observation}
\label{obs.prbl.stable.quotient.is.idem.cplt.stable.quotient}
If $\cC$ is presentable and $\cZ \subseteq \cC$ is a full presentable stable subcategory, then the idempotent-complete stable quotient and the presentable quotient of $\cC$ by $\cZ$ coincide: the canonical morphism
\[
\cC /^{\St^\idem} \cZ
\longra
\cC / \cZ
\]
is an equivalence. Indeed, the presentable quotient satisfies the universal property of the stable quotient: given any stable $\infty$-category $\cD$ and any exact functor $\cC \xra{F} \cD$ such that $F i_L \simeq 0$, the morphism
\[
F
\longra
\nu p_L F
\]
is an equivalence (because for each $X \in \cC$ the cofiber sequence $i_L y X \ra X \ra \nu p_L X$ is carried by $F$ to a cofiber sequence). We use this fact without further comment.
\end{observation}

\begin{observation}
The functor
\[
\St^\idem
\xra{\Ind}
\PrLSt
\]
preserves colimits. In particular, given a thick subcategory $\cZ \in \thicksub_\cC$ we have an equivalence
\[
\Ind( \cC / \cZ)
\simeq
\Ind(\cC) / \Ind(\cZ)
~.
\]
We use this fact without further comment.
\end{observation}

\begin{observation}
The inclusion of a closed subcategory $\cZ \in \clssub_\cC$ extends to a stable recollement
\begin{equation}
\label{stable.recollement.from.closed.subcat}
\begin{tikzcd}[column sep=1.5cm]
\cZ
\arrow[hook, bend left=45]{r}[description]{i_L}
\arrow[leftarrow]{r}[transform canvas={yshift=0.1cm}]{\bot}[swap,transform canvas={yshift=-0.1cm}]{\bot}[description]{\yo}
\arrow[bend right=45, hook]{r}[description]{i_R}
&
\cC
\arrow[bend left=45]{r}[description]{p_L}
\arrow[hookleftarrow]{r}[transform canvas={yshift=0.1cm}]{\bot}[swap,transform canvas={yshift=-0.1cm}]{\bot}[description]{\nu}
\arrow[bend right=45]{r}[description]{p_R}
&
\cC
&[-1.8cm]
/ \cZ
\end{tikzcd}
~.
\end{equation}
We use this fact without further comment.
\end{observation}

\subsection{Stratifications of stable $\infty$-categories}
\label{subsection.stable.stratns}

In this subsection, we extend our theory of stratifications to the case of idempotent-complete stable $\infty$-categories that are not necessarily presentable; we refer to these as \textit{stable stratifications}. We establish metacosm reconstruction for stable stratifications over finite posets as \Cref{thm.stable.metacosm}. We state this result as quickly as possible; much of the rest of the subsection is devoted to its proof. Although we define stable stratifications in terms of stratifications of ind-completions, we also characterize them in a way that does not make reference to ind-completions as \Cref{prop.characterize.stable.stratns}.

\begin{definition}
\label{defn.stable.stratn}
A \bit{stable stratification} of $\cC$ over $\pos$ is a functor
\[ 
\begin{tikzcd}[row sep=0cm]
\pos
\arrow{r}{\cZ_\bullet}
&
\clssub_\cC
\\
\rotatebox{90}{$\in$}
&
\rotatebox{90}{$\in$}
\\
p
\arrow[maps to]{r}
&
\cZ_p
\end{tikzcd}
\]
such that the composite functor
\[
\begin{tikzcd}[row sep=0cm]
\pos
\arrow{r}{\cZ_\bullet}
&
\clssub_\cC
\arrow{r}{\Ind}
&
\Cls_{\Ind(\cC)}
\\
\rotatebox{90}{$\in$}
&
&
\rotatebox{90}{$\in$}
\\
p
\arrow[maps to]{rr}
&
&
\Ind(\cZ_p)
\end{tikzcd}
\]
is a stratification. In this situation, we may also say that $\cC$ is \bit{stably $\pos$-stratified}.
\end{definition}

\begin{definition}
We define the $\infty$-category
\[
\strat_\pos
\]
of \bit{stably $\pos$-stratified idempotent-complete stable $\infty$-categories} analogously to the $\infty$-category $\Strat_\pos$ of \Cref{defn.Strat.P}: its objects are stably $\pos$-stratified idempotent-complete stable $\infty$-categories, and its morphisms are those exact functors that commute with both the $i_L$ inclusions and the $y$ projections.
\end{definition}

\begin{observation}
\label{obs.image.of.strat.in.Strat}
Ind-completion defines a faithful functor
\[
\strat_\pos
\xlonghookra{\Ind}
\Strat_\pos
~.
\]
Explicitly, an object $\cX \in \Strat_\pos$ is in its image precisely when its underlying presentable stable $\infty$-category $\cX \in \PrLSt$ is compactly generated and moreover there exists a factorization
\[ \begin{tikzcd}
\pos
\arrow{r}
\arrow[dashed]{rd}
&
\Cls_\cX
\\
&
\clssub_{\cX^\omega}
\arrow[hook]{u}[swap]{\Ind}
\end{tikzcd}
\]
of its defining functor, and a morphism $\cX \ra \cX'$ in $\Strat_\pos$ between objects in its image lies in its image precisely when its underlying morphism $\cX \ra \cX'$ in $\PrLSt$ preserves compact objects (i.e.\! lies in the subcategory $\PrLomegaSt \subseteq \PrLSt$).
\end{observation}

\begin{theorem}
\label{thm.stable.metacosm}
Assume that $\pos$ is finite. Then, the metacosm adjunction \Cref{adjn.of.metacosm.thm} restricts to an equivalence
\begin{equation}
\label{small.metacosm.equivalence}
\begin{tikzcd}[column sep=1.5cm]
\strat_\pos
\arrow[transform canvas={yshift=0.9ex}]{r}{\GD}
\arrow[leftarrow, transform canvas={yshift=-0.9ex}]{r}[yshift=-0.0ex]{\sim}[swap]{\limrlaxfam}
&
\LMod^\rlax_{\llax.\pos}(\St^\idem)
\end{tikzcd}
~.
\end{equation}
\end{theorem}

\begin{definition}
\label{defn.ind.aligned}
We say that two thick subcategories $\cY,\cZ \in \thicksub_\cC$ are (resp.\! \bit{mutually}) \bit{aligned} if the two closed subcategories $\Ind(\cY),\Ind(\cZ) \in \Cls_{\Ind(\cC)}$ are (resp.\! mutually) aligned.
\end{definition}

\begin{remark}
In the case that $\cC$ is presentable and $\cY, \cZ \in \Cls_\cC \simeq \clssub_\cC \subseteq \thicksub_\cC$, it is not hard to see that \Cref{defn.ind.aligned} coincides with \Cref{defn.aligned}.
\end{remark}

\begin{definition}
We respectively say that a closed subcategory of $\Ind(\cC)$ is \bit{compact-thick}, \bit{compact-split}, or \bit{compact-closed} if it is the ind-completion of a thick, split, or closed subcategory of $\cC$.
\end{definition}

\begin{lemma}
\label{lem.ind.aligned.implies.thick.union.closed}
Let $\cY,\cZ \in \clssub_\cC$ be closed subcategories, and suppose that $\cZ$ is aligned with $\cY$. Then, the thick subcategory $\brax{\cY,\cZ}^\thick \subseteq \cC$ generated by $\cY$ and $\cZ$ is a closed subcategory.
\end{lemma}

\begin{proof}
Note the identification
\[
\Ind( \brax{\cY,\cZ}^\thick )
=
\brax{\Ind(\cY),\Ind(\cZ)}
\in
\Cls_{\Ind(\cC)}
~.
\]
Now, by \Cref{lem.excision}\Cref{part.excision.coloc} (and the fact that $\Ind(\cZ)$ is aligned with $\Ind(\cY)$), we have that the composite
\[
\Ind(\cC)
\xlongra{y}
\brax{\Ind(\cY),\Ind(\cZ)}
\xlonghookra{i_R}
\Ind(\cC)
\]
preserves colimits and compact objects, which proves the claim.
\end{proof}

\begin{cor}
\label{cor.union.over.finite.down.closed.gives.small.closed}
Fix a stable stratification $\cZ_\bullet$ of $\cC$ over $\pos$. For every finite down-closed subset $\sD \in \Down_\pos^\fin$, the thick subcategory
\[
\brax{\cZ_p}^\thick_{p \in \sD} \subseteq \cC
\]
generated by the corresponding closed subcategories is itself a closed subcategory.
\end{cor}

\begin{proof}
This follows by applying Lemmas \ref{lem.ind.aligned.implies.thick.union.closed} \and \ref{closed.subcats.are.mutually.aligned} inductively.
\end{proof}


\begin{proof}[Proof of \Cref{thm.stable.metacosm}]
Under the assumption that $\pos$ is finite (and hence down-finite), the metacosm adjunction \Cref{adjn.of.metacosm.thm} is an equivalence by \Cref{metacosm.thm}. It therefore suffices to prove that there exist factorizations
\begin{equation}
\label{factorizn.of.ladjt.for.small.reconstruction}
\begin{tikzcd}
\Strat_\pos
\arrow{r}{\GD}[swap]{\sim}
&
\LMod^{\rlax,L}_{\llax.\pos}(\PrSt)
\\
\strat_\pos
\arrow[hook]{u}{\Ind}
\arrow[dashed]{r}
&
\LMod^\rlax_{\llax.\pos}(\St^\idem)
\arrow[hook]{u}[swap]{\LMod^\rlax_{\llax.\pos}(\Ind)}
\end{tikzcd}
\end{equation}
and
\begin{equation}
\label{factorizn.of.radjt.for.small.reconstruction}
\begin{tikzcd}[column sep=1.5cm]
\Strat_\pos
&
\LMod^{\rlax,L}_{\llax.\pos}(\PrSt)
\arrow{l}{\sim}[swap]{\limrlaxfam}
\\
\strat_\pos
\arrow[hook]{u}{\Ind}
&
\LMod^\rlax_{\llax.\pos}(\St^\idem)
\arrow[hook]{u}[swap]{\LMod^\rlax_{\llax.\pos}(\Ind)}
\arrow[dashed]{l}
\end{tikzcd}
~.
\end{equation}

We first prove that factorization \Cref{factorizn.of.ladjt.for.small.reconstruction} exists.\footnote{In fact, for this factorization to exist it suffices that $\pos$ merely be down-finite.} Fix a stable stratification
\[
\pos \xra{\cZ_\bullet} \clssub_\cC
~,
\]
and consider the composite stratification
\[
\pos \xra{\cZ_\bullet} \clssub_\cC \xra{\Ind} \Cls_{\Ind(\cC)}
~.
\]
Because $\pos$ is finite, every $\sD \in \Down_\pos$ is finite. Hence, by \Cref{cor.union.over.finite.down.closed.gives.small.closed}, for every $\sD \in \Down_\pos$ the closed subcategory
\[
\Ind(\cZ_\bullet)_\sD
:=
\brax{\Ind(\cZ_p)}_{p \in \sD}
=
\Ind( \brax{ \cZ_p}^\thick_{p \in \sD} )
\in
\Cls_{\Ind(\cC)}
\]
is compact-closed. It follows that for every $p \in \pos$, all functors in the recollement
\[ \begin{tikzcd}
\Ind(\cZ_\bullet)_{^< p}
\arrow[hook, bend left=45]{r}[description]{i_L}
\arrow[leftarrow]{r}[transform canvas={yshift=0.1cm}]{\bot}[swap,transform canvas={yshift=-0.1cm}]{\bot}[description]{\yo}
\arrow[bend right=45, hook]{r}[description]{i_R}
&
\Ind(\cZ_\bullet)_p
\arrow[bend left=45]{r}[description]{p_L}
\arrow[hookleftarrow]{r}[transform canvas={yshift=0.1cm}]{\bot}[swap,transform canvas={yshift=-0.1cm}]{\bot}[description]{\nu}
\arrow[bend right=45]{r}[description]{p_R}
&
\Ind(\cC)_p
\end{tikzcd} \]
preserve colimits and compact objects, and hence all functors in the composite adjunction
\[
\begin{tikzcd}[column sep=1.5cm]
\Phi_p
:
\Ind(\cC)
\arrow[transform canvas={yshift=0.9ex}]{r}{y}
\arrow[hookleftarrow, transform canvas={yshift=-0.9ex}]{r}[yshift=-0.2ex]{\bot}[swap]{i_R}
&
\Ind(\cZ_\bullet)_p
\arrow[transform canvas={yshift=0.9ex}]{r}{p_L}
\arrow[hookleftarrow, transform canvas={yshift=-0.9ex}]{r}[yshift=-0.2ex]{\bot}[swap]{\nu}
&
\Ind(\cC)_p
:
\rho^p
\end{tikzcd}
\]
preserve colimits and compact objects. This implies that factorization \Cref{factorizn.of.ladjt.for.small.reconstruction} exists on objects, and thereafter it is straightforward to see that it exists on morphisms as well.

We now prove that factorization \Cref{factorizn.of.radjt.for.small.reconstruction} exists. To avoid unnecessary notation involving ind-completions, we simply begin with an object
\[
(\cE \da \pos) \in \LMod^\rlax_{\llax.\pos}(\PrLomegaSt)
~,
\]
and prove that its image under the composite
\[ \begin{tikzcd}[column sep=1.5cm]
\Strat_\pos
&
\LMod^\rlax_{\llax.\pos}(\PrSt)
\arrow{l}{\sim}[swap]{\limrlaxfam}
\\
&
\LMod^{\rlax,L}_{\llax.\pos}(\PrLomegaSt)
\arrow[hook]{u}
\end{tikzcd} \]
lies in the image of the inclusion
\[
\strat_\pos
\xlonghookra{\Ind}
\Strat_\pos
~,
\]
as described in \Cref{obs.image.of.strat.in.Strat}. To further simplify our notation, we write
\[
\cX
:=
\lim^\rlax_{\llax.\pos}(\cE)
~,
\]
and for any $p \in \pos$ and any $\sC \in \Conv_\pos$ we write
\[
\cZ_p
:=
\lim^\rlax_{\llax.(^\leq p)}(\cE)
\qquad
\text{and}
\qquad
\cX_\sC
:=
\lim^\rlax_{\llax.\sC}(\cE)
~,
\]
as justified by \Cref{prop.metacosm.input.first.get.stratn}\Cref{metacosm.input.get.stratn}; in particular, we have $\cX_p = \cE_p$.

As a preliminary observation, we note that for every $p,q \in \pos$ the composite functor
\begin{equation}
\label{p.q.gluing.functor.for.stratn.of.rlaxlimfam.from.PrLomegaSt}
\cX_p
\xlonghookra{\rho^p}
\cX
\xra{\Phi_q}
\cX_q
\end{equation}
is the monodromy functor $\cE_p \xra{\cE_{p \leq q}} \cE_q$ if $p \leq q$ (by \Cref{prop.metacosm.input.first.get.stratn}\Cref{metacosm.input.get.stratn}\Cref{gluing.in.X.is.mdrmy.in.E}) and is zero if $p \not\leq q$; in particular, in either case it preserves colimits and compact objects.

Now, fix any $p \in \pos$. The functor
\[
\cX \xra{(\Phi_q)_{q \in \pos}} \prod_{q \in \pos} \cX_q
\]
is conservative (because $\pos$ is finite) and preserves colimits. Combining this with the fact that for all $q \in \pos$ the composite functor \Cref{p.q.gluing.functor.for.stratn.of.rlaxlimfam.from.PrLomegaSt} preserves colimits, it follows that in the adjunction
\[ \begin{tikzcd}[column sep=1.5cm]
\cX
\arrow[transform canvas={yshift=0.9ex}]{r}{\Phi_p}
\arrow[hookleftarrow, transform canvas={yshift=-0.9ex}]{r}[yshift=-0.2ex]{\bot}[swap]{\rho^p}
&
\cX_p
\end{tikzcd} \]
the right adjoint $\rho^p$ preserves colimits, which implies that the left adjoint $\Phi_p$ preserves compact objects.

We now claim that the converse also holds: if an object $X \in \cX$ satisfies the condition that $\Phi_p(X) \in \cX_p$ is compact for all $p \in \pos$, then it is compact. To see this, for any filtered diagram $\cI \xra{Y_\bullet} \cX$ we compute that
\begin{align}
\label{show.Phi.p.jointly.detect.compactness.use.nanocosm.first.time}
\ulhom_\cX(X , \colim_\cI(Y_\bullet))
& \simeq
\lim_{([n] \xra{\varphi} \pos) \in \sd(\pos)}
\left(
\ulhom_{\cX_{\varphi(n)}} ( \Phi_{\varphi(n)}(X) , \Gamma_\varphi \Phi_{\varphi(0)}(\colim_\cI(Y_\bullet)))
\right)
\\
\label{show.Phi.p.jointly.detect.compactness.use.that.gluing.functors.preserve.colimits}
& \simeq
\lim_{([n] \xra{\varphi} \pos) \in \sd(\pos)}
\left(
\ulhom_{\cX_{\varphi(n)}} ( \Phi_{\varphi(n)}(X) , \colim_\cI (\Gamma_\varphi \Phi_{\varphi(0)}(Y_\bullet)))
\right)
\\
\label{show.Phi.p.jointly.detect.compactness.use.that.Phi.p.X.is.compact}
& \simeq
\lim_{([n] \xra{\varphi} \pos) \in \sd(\pos)}
\left(
\colim_\cI
\left(
\ulhom_{\cX_{\varphi(n)}} ( \Phi_{\varphi(n)}(X) , \Gamma_\varphi \Phi_{\varphi(0)}(Y_\bullet))
\right)
\right)
\\
\label{show.Phi.p.jointly.detect.compactness.use.that.I.is.filtered.and.sd.P.is.finite}
& \simeq
\colim_\cI
\left(
\lim_{([n] \xra{\varphi} \pos) \in \sd(\pos)}
\left(
\ulhom_{\cX_{\varphi(n)}} ( \Phi_{\varphi(n)}(X) , \Gamma_\varphi \Phi_{\varphi(0)}(Y_\bullet))
\right)
\right)
\\
\label{show.Phi.p.jointly.detect.compactness.use.nanocosm.second.time}
& \simeq
\colim_\cI
\left(
\ulhom_\cX ( X , Y_\bullet)
\right)
~,
\end{align}
where
\begin{itemize}

\item equivalences \Cref{show.Phi.p.jointly.detect.compactness.use.nanocosm.first.time} \and \Cref{show.Phi.p.jointly.detect.compactness.use.nanocosm.second.time} use nanocosm reconstruction (recall \Cref{rmk.nanocosm}),

\item equivalence \Cref{show.Phi.p.jointly.detect.compactness.use.that.gluing.functors.preserve.colimits} follows from the fact that the composite
\[
\Gamma_\varphi \Phi_{\varphi(0)}
:=
\Gamma^{\varphi(n-1)}_{\varphi(n)} \cdots \Gamma^{\varphi(0)}_{\varphi(1)} \Phi_{\varphi(0)}
:=
\Phi_{\varphi(n)} \rho^{\varphi(n-1)} \cdots \Phi_{\varphi(1)} \rho^{\varphi(0)} \Phi_{\varphi(0)}
\]
preserves colimits,

\item equivalence \Cref{show.Phi.p.jointly.detect.compactness.use.that.Phi.p.X.is.compact} uses the assumption that $\Phi_p(X) \in \cX_p$ is compact for all $p \in \pos$, and

\item equivalence \Cref{show.Phi.p.jointly.detect.compactness.use.that.I.is.filtered.and.sd.P.is.finite} uses the facts that $\cI$ is filtered and $\sd(\pos)$ is finite (because $\pos$ is finite).

\end{itemize}
So in fact, an object $X \in \cX$ is compact if and only if the object $\Phi_p(X) \in \cX_p$ is compact for all $p \in \pos$. It now follows that for every $p \in \pos$ the functor $\rho^p$ preserves compact objects, because for every $q \in \pos$ the functor \Cref{p.q.gluing.functor.for.stratn.of.rlaxlimfam.from.PrLomegaSt} preserves compact objects.

We now verify that $\cX$ is compactly generated. We proceed by induction on the cardinality of $\pos$, the base case where $\pos = \es$ being trivial. So, assume that $\pos \not= \es$, choose any minimal element $-\infty \in \pos$, and write $\pos' := \pos \backslash \{ - \infty \} \in \Conv_\pos$ for its complement. This defines a functor $\pos \xra{\pi} [1]$ with $\pi^{-1}(0) = \{ -\infty \}$ and $\pi^{-1}(1) = \pos'$. Taking pushforwards of stratifications along $\pi$ via \Cref{prop.pushfwd.stratn} yields a recollement
\[ \begin{tikzcd}[column sep=1.5cm]
\cX_{-\infty}
\arrow[hook, bend left=45]{r}[description]{i_L}
\arrow[leftarrow]{r}[transform canvas={yshift=0.1cm}]{\bot}[swap,transform canvas={yshift=-0.1cm}]{\bot}[description]{\yo}
\arrow[bend right=45, hook]{r}[description]{i_R}
&
\cX
\arrow[bend left=45]{r}[description]{p_L}
\arrow[hookleftarrow]{r}[transform canvas={yshift=0.1cm}]{\bot}[swap,transform canvas={yshift=-0.1cm}]{\bot}[description]{\nu}
\arrow[bend right=45]{r}[description]{p_R}
&
\cX_{\pos'}
\end{tikzcd} \]
in which $\cX_{-\infty} = \cE_{-\infty}$ is compactly generated by assumption and $\cX_{\pos'} = \lim^\rlax_{\llax.\pos'}(\cE)$ is compactly generated by induction. Moreover, by \Cref{prop.metacosm.input.first.get.stratn}\Cref{metacosm.input.describe.adjts} the functors $i_L$ and $\nu$ are both given by extension by 0, and so preserve compact objects by our above characterization of compact objects in $\cX$ (applied also to $\cX_{\pos'}$). Since $i_L$ and $\nu$ also both preserve colimits, we find that the smallest cocomplete full subcategory of $\cX$ containing its compact objects is in fact all of $\cX$, i.e.\! that $\cX$ is compactly generated.

We now show that for every $p \in \pos$ the closed subcategory $\cZ_p \in \Cls_\cX$ is compact-closed. For this, we use the restricted stratification of $\cZ_p$ over $(^\leq p)$ of \Cref{obs.restricted.stratn.over.D}, which converges since the poset $(^\leq p)$ is finite. For each $q \in (^\leq p)$, we use the notation
\[ \begin{tikzcd}[column sep=1.5cm]
\cZ_p
\arrow[transform canvas={yshift=0.9ex}]{r}{\w{\Phi}_q}
\arrow[hookleftarrow, transform canvas={yshift=-0.9ex}]{r}[yshift=-0.2ex]{\bot}[swap]{\w{\rho}^q}
&
\cX_q
\end{tikzcd} \]
for the corresponding geometric localization adjunction, and we write
\[
\w{L}_q := \w{\rho}^q \w{\Phi}_q
\]
for the corresponding idempotent endofunctor of $\cZ_p$. Consider the commutative diagram
\begin{equation}
\label{diagram.for.showing.iR.preserves.colimits.and.compacts.for.small.reconstrn}
\begin{tikzcd}[column sep=2cm, row sep=1.5cm]
&
\lim^\rlax_{\llax.(^\leq p)} ( \GD ( \cZ_p))
\arrow[hook]{d}
\\
\cZ_p
\arrow{ru}[sloped]{\gd}[sloped, swap]{\sim}
\arrow{r}{\gd'}
\arrow{rd}[sloped]{\sim}[sloped, swap]{\id_{\cZ_p}}
&
\Fun(\sd(^\leq p) , \cZ_p)
\arrow[hook]{r}{\Fun ( \sd(^\leq p) , i_R )}
\arrow{d}{\lim_{\sd(^\leq p)}}
&
\Fun ( \sd(^\leq p) , \cX)
\arrow{d}{\lim_{\sd(^\leq p)}}
\\
&
\cZ_p
\arrow[hook]{r}[swap]{i_R}
&
\cX
\end{tikzcd}
\end{equation}
in $\Cat$ in which
\begin{itemize}

\item the functor $\gd'$ is described by the formula
\[ \begin{tikzcd}[row sep=0cm]
\cZ_p
\arrow{r}{\gd'}
&
\Fun(\sd(^\leq p),\cZ_p)
\\
\rotatebox{90}{$\in$}
&
\rotatebox{90}{$\in$}
\\
X
\arrow[maps to]{r}
&
\left(
\left( [n] \xlongra{\varphi} (^\leq p) \right)
\longmapsto
\w{L}_\varphi(X)
\right)
\end{tikzcd}
\]
where we write
\[
\w{L}_\varphi
:=
\w{L}_{\varphi(n)} \cdots \w{L}_{\varphi(0)}
\]
for brevity (recall \Cref{obs.formula.for.mocrocosm.gluing.functor.without.using.Lphi.notation} (and \Cref{rmk.shorter.version.of.formula.for.mocrocosm.gluing.functor.using.Lphi.notation})),

\item the commutativity of the left two triangles follow from macrocosm reconstruction (\Cref{macrocosm.thm}) for the stratification of $\cZ_p$ over $(^\leq p)$, and

\item the square commutes because $\sd(^\leq p)$ is finite (since $(^\leq p)$ is finite) and $i_R$ is exact.

\end{itemize}
Observe that the functor
\[
\Fun(\sd(^\leq p),\cX) \xra{\lim_{\sd(^\leq p)}} \cX
\]
carries pointwise colimits to colimits and carries pointwise compact objects to compact objects (both using the facts that $\cX$ is stable and that $\sd(^\leq p)$ is finite). So, using the commutativity of diagram \Cref{diagram.for.showing.iR.preserves.colimits.and.compacts.for.small.reconstrn}, to show that the functor
\[
\cZ_p
\xlonghookra{i_R}
\cX
\]
preserves colimits and compact objects, it suffices to show that the composite
\[
\cZ_p
\xlongra{\gd'}
\Fun ( \sd(^\leq p) , \cZ_p)
\xhookra{\Fun ( \sd(^\leq p),i_R)}
\Fun ( \sd(^\leq p) , \cX)
\]
carries colimits to pointwise colimits and carries compact objects to pointwise compact objects, or equivalently that for every $([n] \xra{\varphi} (^\leq p)) \in \sd(^\leq p)$ the composite
\[
\cZ_p
\xra{\w{L}_\varphi}
\cZ_p
\xlonghookra{i_R}
\cX
\]
preserves colimits and compact objects. For this, it suffices to show that the composite
\[
\cZ_p
\xra{\w{L}_\varphi}
\cZ_p
\xlonghookra{i_R}
\cX
\xra{(\Phi_q)_{q \in \pos}}
\prod_{q \in \pos} \cX_q
\]
carries colimits to pointwise colimits and carries compact objects to pointwise compact objects -- the former because the functor $(\Phi_q)_{q \in \pos}$ is conservative (because $\pos$ is finite) and preserves colimits, the latter by our above characterization of compact objects in $\cX$. Equivalently, it suffices to show that for every $q \in \pos$ the composite
\begin{equation}
\label{functor.from.Zp.to.Xq.for.small.reconstrn}
\cZ_p
\xra{\w{L}_\varphi}
\cZ_p
\xlonghookra{i_R}
\cX
\xra{\Phi_q}
\cX_q
\end{equation}
preserves colimits and compact objects. To see this, first observe the factorization of the composite \Cref{functor.from.Zp.to.Xq.for.small.reconstrn} according to the commutative diagram
\[ \begin{tikzcd}
\cZ_p
\arrow{r}{\w{L}_\varphi}
\arrow{d}[swap]{\w{L}_{\varphi_{|[n-1]}}}
&
\cZ_p
\arrow[hook]{r}{i_R}
&
\cX
\arrow{r}{\Phi_q}
&
\cX_q
\\
\cZ_p
\arrow{r}[swap]{\w{\Phi}_{\varphi(n)}}
&
\cX_{\varphi(n)}
\arrow[hook]{u}{\w{\rho}^{\varphi(n)}}
\arrow[hook]{ru}[sloped, swap]{\rho^{\varphi(n)}}
\end{tikzcd}~, \]
in which the square commutes by definition and the commutativity of the triangle follows from the commutativity of the triangle
\[ \begin{tikzcd}
\cZ_p
\arrow[hook]{r}{i_R}
&
\cX
\\
\cZ_{\varphi(n)}
\arrow[hook]{u}{i_R}
\arrow[hook]{ru}[sloped, swap]{i_R}
\end{tikzcd}
~.
\]
Now, because the composite functor \Cref{p.q.gluing.functor.for.stratn.of.rlaxlimfam.from.PrLomegaSt} (replacing $p$ with $\varphi(n)$) preserves colimits and compact objects, it follows that the functor
\[
\cZ_p
\xlonghookra{i_R}
\cX
\]
preserves colimits and compact objects, i.e.\! the closed subcategory $\cZ_p \in \Cls_\cX$ is indeed compact-closed.

We have shown that the factorization \Cref{factorizn.of.radjt.for.small.reconstruction} exists on objects, and thereafter it is straightforward to see that it exists on morphisms as well.
\end{proof}

\begin{remark}
One may interpret our proof of \Cref{thm.stable.metacosm} as establishing the commutativity of the functor
\[
\St^\idem
\xra{\Ind}
\PrSt
\]
with certain right-lax limits.
\end{remark}

\begin{prop}
\label{prop.characterize.stable.stratns}
Choose any functor
\[ 
\begin{tikzcd}[row sep=0cm]
\pos
\arrow{r}{\cZ_\bullet}
&
\clssub_\cC
\\
\rotatebox{90}{$\in$}
&
\rotatebox{90}{$\in$}
\\
p
\arrow[maps to]{r}
&
\cZ_p
\end{tikzcd}
~,
\]
and consider the composite functor
\[
\begin{tikzcd}[row sep=0cm]
\pos
\arrow{r}{\cZ_\bullet}
&
\clssub_\cC
\arrow{r}{\Ind}
&
\Cls_{\Ind(\cC)}
\\
\rotatebox{90}{$\in$}
&
&
\rotatebox{90}{$\in$}
\\
p
\arrow[maps to]{rr}
&
&
\Ind(\cZ_p)
\end{tikzcd}
~.
\]
\begin{enumerate}

\item\label{item.stable.prestratns} The composite $\Ind(\cZ_\bullet)$ is a prestratification if and only if $\brax{\cZ_p}^\thick_{p \in \pos} = \cC$.

\item\label{stable.stratn.condition}

The composite $\Ind(\cZ_\bullet)$ satisfies the stratification condition if and only if for every $p,q \in \pos$,

\begin{enumeratesub}

\item\label{stable.stratn.condition.thick.is.closed}

the thick subcategory
\[
\cZ_{(^\leq p) \cap (^\leq q)}
:=
\brax{\cZ_r}^\thick_{r \in (^\leq p ) \cap (^\leq q)}
\subseteq
\cC
\]
is in fact a closed subcategory, and

\item\label{stable.stratn.condition.usual.factorizn}

there exists a factorization
\[ \begin{tikzcd}
\cZ_{(^\leq p ) \cap (^\leq q)}
\arrow[hook]{r}{i_L}
&
\cZ_p
\\
\cZ_q
\arrow[hook]{r}[swap]{i_L}
\arrow[dashed]{u}
&
\cC
\arrow{u}[swap]{y}
\end{tikzcd}
~.
\]

\end{enumeratesub}

\end{enumerate}
\end{prop}

\begin{lemma}
\label{lem.ind.aligned.implies.intersection.closed}
Let $\cY,\cZ \in \clssub_\cC$ be closed subcategories, and suppose that $\cZ$ is aligned with $\cY$. Then, the intersection $(\cY \cap \cZ) \subseteq \cC$ is a closed subcategory, and moreover we have an identification
\[
\Ind(\cY \cap \cZ) = (\Ind(\cY) \cap \Ind(\cZ))
\in
\Cls_{\Ind(\cC)}
~.
\]
\end{lemma}

\begin{proof}
Consider the pullback square
\begin{equation}
\label{iL.inclusions.into.Ind.C.for.showing.ind.aligned.implies.intersection.closed}
\begin{tikzcd}[column sep=1.5cm]
\Ind(\cY) \cap \Ind(\cZ)
\arrow[hook]{r}{i_{\Ind(\cY)}}
\arrow[hook]{d}[swap]{i_{\Ind(\cZ)}}
&
\Ind(\cY)
\arrow[hook]{d}{i_L = \Ind(i_L)}
\\
\Ind(\cZ)
\arrow[hook]{r}[swap]{i_L = \Ind(i_L) }
&
\Ind(\cC)
\end{tikzcd}
\end{equation}
in $\PrLSt$. By \Cref{obs.if.aligned.preimage.of.closed.is.closed} (and the fact that $\Ind(\cZ)$ is aligned with $\Ind(\cY)$), we have that $(\Ind(\cY) \cap \Ind(\cZ)) \in \Cls_{\Ind(\cC)}$. Thereafter, by \Cref{lem.if.aligned.then.iRs.form.a.pullback.square}, the commutative square
\begin{equation}
\label{iR.inclusions.into.Ind.C.for.showing.ind.aligned.implies.intersection.closed}
\begin{tikzcd}[column sep=1.5cm]
\Ind(\cY) \cap \Ind(\cZ)
\arrow[hook]{r}{i_R}
\arrow[hook]{d}[swap]{i_R}
&
\Ind(\cY)
\arrow[hook]{d}{i_R = \Ind(i_R)}
\\
\Ind(\cZ)
\arrow[hook]{r}[swap]{i_R = \Ind(i_R) }
&
\Ind(\cC)
\end{tikzcd}
\end{equation}
in $\Cat$ obtained by taking right adjoints twice in the commutative square \Cref{iL.inclusions.into.Ind.C.for.showing.ind.aligned.implies.intersection.closed} is a pullback square. In particular, the identifications $i_R = \Ind(i_R)$ in the pullback square \Cref{iR.inclusions.into.Ind.C.for.showing.ind.aligned.implies.intersection.closed} imply that it lies in $\PrLSt$. It follows that $(\Ind(\cY) \cap \Ind(\cZ)) \in \Cls_{\Ind(\cC)}$ is compact-closed. Now, using the fact that $(\Ind(\cY) \cap \Ind(\cZ)) \xhookra{i_L} \Ind(\cC)$ preserves compact objects (because its right adjoint preserves colimits), we obtain the composite identification
\[
\clssub_\cC
\ni
(\Ind(\cY) \cap \Ind(\cZ))^\omega
=
( (\Ind(\cY) \cap \Ind(\cZ)) \cap \Ind(\cC)^\omega)
=
( \Ind(\cY) \cap \Ind(\cZ) \cap \cC)
=
(\cY \cap \cZ)
~,
\]
which proves both assertions.
\end{proof}

\begin{proof}[Proof of \Cref{prop.characterize.stable.stratns}]
Part \Cref{item.stable.prestratns} is clear. So, we proceed to part \Cref{stable.stratn.condition}. First of all, it is clear that for any $p,q \in \pos$, if the functor $\cZ_\bullet$ satisfies conditions \Cref{stable.stratn.condition.thick.is.closed} \and \Cref{stable.stratn.condition.usual.factorizn} then the composite $\Ind(\cZ_\bullet)$ satisfies the stratification condition (using the fact that the functor $\thicksub_\cC \xhookra{\Ind} \Cls_{\Ind(\cC)}$ commutes with colimits). Conversely, suppose that the composite $\Ind(\cZ_\bullet)$ satisfies the stratification condition. Then, by \Cref{closed.subcats.are.mutually.aligned} we have that the closed subcategories $\Ind(\cZ_p) , \Ind(\cZ_q) \in \Cls_{\Ind(\cC)}$ are mutually aligned and moreover
\begin{equation}
\label{intersection.of.Inds.is.Ind.of.union.of.leq.p.and.leq.q}
\Cls_{\Ind(\cC)}
\ni
(\Ind(\cZ_p) \cap \Ind(\cZ_q))
=
\brax{ \Ind(\cZ_r) }_{r \in (^\leq p) \cap (^\leq q)}
=
\Ind( \brax{ \cZ_r }_{r \in (^\leq p) \cap (^\leq q)}^\thick )
=:
\Ind( \cZ_{(^\leq p) \cap (^\leq q)} )
~.
\end{equation}
In particular, the closed subcategories $\cZ_p , \cZ_q \in \clssub_\cC$ are mutually aligned, which by \Cref{lem.ind.aligned.implies.intersection.closed} implies that
\begin{equation}
\label{intersection.of.Inds.is.Ind.of.intersection}
\Ind( \cZ_p \cap \cZ_q ) =
( \Ind(\cZ_p) \cap \Ind(\cZ_q) )
\in
\Cls_{\Ind(\cC)}
~.
\end{equation}
As the functor $\thicksub_\cC \xhookra{\Ind} \Cls_{\Ind(\cC)}$ is fully faithful, the identifications \Cref{intersection.of.Inds.is.Ind.of.union.of.leq.p.and.leq.q} \and \Cref{intersection.of.Inds.is.Ind.of.intersection} along with \Cref{lem.ind.aligned.implies.intersection.closed} now imply that we have an identification
\[
\clssub_\cC
\ni
(\cZ_p \cap \cZ_q)
=
\Ind( \cZ_{(^\leq p) \cap (^\leq q)} )^\omega
~,
\]
i.e.\! the functor $\cZ_\bullet$ satisfies condition \Cref{stable.stratn.condition.thick.is.closed}. Now, by \Cref{lem.equivalent.characterizations.of.alignment} the square
\[ \begin{tikzcd}
\Ind(\cZ_p) \cap \Ind(\cZ_q)
\arrow[hook]{r}{i_L}
&
\Ind(\cZ_p)
\\
\Ind(\cZ_q)
\arrow{u}{y}
\arrow[hook]{r}[swap]{i_L}
&
\Ind(\cC)
\arrow{u}[swap]{y}
\end{tikzcd} \]
commutes, and by \Cref{lem.ind.aligned.implies.intersection.closed} all four of its functors preserve compact objects (because their right adjoints preserve colimits). Hence, again using identification \Cref{intersection.of.Inds.is.Ind.of.union.of.leq.p.and.leq.q} we see that the functor $\cZ_\bullet$ satisfies condition \Cref{stable.stratn.condition.usual.factorizn}.
\end{proof}

\subsection{Strict morphisms among stratifications}
\label{subsection.strict.metacosm}

In this brief subsection, we introduce strict morphisms among (resp.\! stable) stratifications and show as \Cref{thm.strict.metacosm} that they correspond through metacosm reconstruction to strict (as opposed to possibly right-lax) morphisms between left-lax left $\pos$-modules.

\begin{definition}
We say that a morphism $\cX \ra \cX'$ in $\Strat_\pos$ or in $\strat_\pos$ is \bit{strict} if for every $p \in \pos$ there exists a (necessarily unique) factorization
\[ \begin{tikzcd}
\cX
\arrow{r}
&
\cX'
\\
\cZ_p
\arrow[hook]{u}{i_R}
\arrow[dashed]{r}
&
\cZ_p'
\arrow[hook]{u}[swap]{i_R}
\end{tikzcd}
\]
(with the evident notation). We denote by
\[
\Strat^\strict_\pos
\subseteq
\Strat_\pos
\qquad
\text{and}
\qquad
\strat^\strict_\pos
\subseteq
\strat_\pos
\]
the respective subcategories on the strict morphisms.
\end{definition}


\needspace{2\baselineskip}
\begin{theorem}
\label{thm.strict.metacosm}
\begin{enumerate}
\item[]

\item\label{strict.presentable.metacosm}

Assume that $\pos$ is down-finite. Then, the metacosm equivalence \Cref{adjn.of.metacosm.thm} of \Cref{metacosm.thm} restricts to an equivalence
\[ \begin{tikzcd}[column sep=1.5cm]
\Strat_\pos^\strict
\arrow[transform canvas={yshift=0.9ex}]{r}{\GD}
\arrow[leftarrow, transform canvas={yshift=-0.9ex}]{r}[yshift=-0.0ex]{\sim}[swap]{\limrlaxfam}
&
\LMod^L_{\llax.\pos}(\PrSt)
\end{tikzcd}
~.
\]

\item\label{strict.small.metacosm}

Assume that $\pos$ is finite. Then, the metacosm equivalence \Cref{small.metacosm.equivalence} of \Cref{thm.stable.metacosm} restricts to an equivalence
\[
\begin{tikzcd}[column sep=1.5cm]
\strat_\pos^\strict
\arrow[transform canvas={yshift=0.9ex}]{r}{\GD}
\arrow[leftarrow, transform canvas={yshift=-0.9ex}]{r}[yshift=-0.0ex]{\sim}[swap]{\limrlaxfam}
&
\LMod_{\llax.\pos}(\St^\idem)
\end{tikzcd}
~.
\]

\end{enumerate}
\end{theorem}

\begin{proof}
We begin with part \Cref{strict.presentable.metacosm}. On the one hand, it is clear that there exists a factorization
\[ \begin{tikzcd}
\Strat_\pos
\arrow{r}{\GD}[swap]{\sim}
&
\LMod^{\rlax,L}_{\llax.\pos}(\PrSt)
\\
\Strat^\strict_\pos
\arrow[dashed]{r}
\arrow[hook, two heads]{u}
&
\LMod^L_{\llax.\pos}(\PrSt)
\arrow[hook, two heads]{u}
\end{tikzcd} ~. \]
So, it remains to show that there exists a factorization
\[ \begin{tikzcd}[column sep=1.5cm]
\Strat_\pos
&
\LMod^{\rlax,L}_{\llax.\pos}(\PrSt)
\arrow{l}[swap]{\limrlaxfam}{\sim}
\\
\Strat^\strict_\pos
\arrow[hook, two heads]{u}
&
\LMod^L_{\llax.\pos}(\PrSt)
\arrow[hook, two heads]{u}
\arrow[dashed]{l}
\end{tikzcd} ~. \]
Given a morphism $\cE \ra \cE'$ in $\LMod^L_{\llax.\pos}(\PrSt)$, for each $p \in \pos$ we obtain a commutative square
\begin{equation}
\label{comm.square.of.yo.restrns.for.proving.strict.maps.between.modules.gives.strict.map.in.Strat}
\begin{tikzcd}
\lim^\rlax_{\llax.\pos}(\cE)
\arrow{r}
\arrow{d}[swap]{y}
&
\lim^\rlax_{\llax.\pos}(\cE')
\arrow{d}{y}
\\
\lim^\rlax_{\llax.(^\leq p)}(\cE)
\arrow{r}
&
\lim^\rlax_{\llax.(^\leq p)}(\cE')
\end{tikzcd}
~,
\end{equation}
and it suffices to show that the corresponding lax-commutative square
\begin{equation}
\label{lax.comm.square.with.iRs.for.proving.strict.maps.between.modules.gives.strict.map.in.Strat}
\begin{tikzcd}
\lim^\rlax_{\llax.\pos}(\cE)
\arrow{r}[yshift=-0.8cm]{\rotatebox{-45}{$\Leftarrow$}}
&
\lim^\rlax_{\llax.\pos}(\cE')
\\
\lim^\rlax_{\llax.(^\leq p)}(\cE)
\arrow[hook]{u}{i_R}
\arrow{r}
&
\lim^\rlax_{\llax.(^\leq p)}(\cE')
\arrow[hook]{u}[swap]{i_R}
\end{tikzcd}
\end{equation}
commutes. For this, we use the restricted stratifications of $\lim^\rlax_{\llax.(^\leq p)}(\cE)$ and $\lim^\rlax_{\llax.(^\leq p)}(\cE')$ over $(^\leq p)$ of \Cref{obs.restricted.stratn.over.D}, which converge since the poset $(^\leq p)$ is finite (because $\pos$ is down-finite). To simplify our notation, we write
\[
\cX
:=
\lim^\rlax_{\llax.\pos}(\cE)
\qquad
\text{and}
\qquad
\cX'
:=
\lim^\rlax_{\llax.\pos}(\cE')
~,
\]
and for any $p \in \pos$ we write
\[
\cZ_p
:=
\lim^\rlax_{\llax.\pos}(\cE)
~,
\qquad
\cZ_p'
:=
\lim^\rlax_{\llax.\pos}(\cE')
~,
\qquad
\cX_p
:=
\cE_p
~,
\qquad
\text{and}
\qquad
\cX_p'
:=
\cE_p'
~,
\]
as justified by \Cref{prop.metacosm.input.first.get.stratn}\Cref{metacosm.input.get.stratn}. Moreover, for each $q \in (^\leq p)$, we use the notation
\[
\begin{tikzcd}[column sep=1.5cm]
\lim^\rlax_{\llax.(^\leq p)}(\cE)
\arrow[transform canvas={yshift=0.9ex}]{r}{\w{\Phi}_q}
\arrow[hookleftarrow, transform canvas={yshift=-0.9ex}]{r}[yshift=-0.2ex]{\bot}[swap]{\w{\rho}^q}
&
\cE_q
\end{tikzcd}
\qquad
\text{and}
\qquad
\begin{tikzcd}[column sep=1.5cm]
\lim^\rlax_{\llax.(^\leq p)}(\cE')
\arrow[transform canvas={yshift=0.9ex}]{r}{\w{\Phi}_q}
\arrow[hookleftarrow, transform canvas={yshift=-0.9ex}]{r}[yshift=-0.2ex]{\bot}[swap]{\w{\rho}^q}
&
\cE_q'
\end{tikzcd}
\]
for the corresponding geometric localization adjunctions. Now, for each $q \in (^\leq p)$, we may extend the commutative square \Cref{comm.square.of.yo.restrns.for.proving.strict.maps.between.modules.gives.strict.map.in.Strat} to a commutative diagram
\[
\begin{tikzcd}
\cX
\arrow{r}
\arrow{d}{y}
\arrow[bend right]{dd}[swap]{\Phi_q}
&
\cX'
\arrow{d}[swap]{y}
\arrow[bend left]{dd}{\Phi_q}
\\
\cZ_p
\arrow{r}
\arrow{d}{\w{\Phi}_q}
&
\cZ_p'
\arrow{d}[swap]{\w{\Phi}_q}
\\
\cX_q
\arrow{r}
&
\cX'_q
\end{tikzcd}
~,
\]
which determines the lower two squares in the lax-commutative diagram
\begin{equation}
\label{lax.comm.diagram.with.iRs.and.rhos.for.proving.strict.maps.between.modules.gives.strict.map.in.Strat}
\begin{tikzcd}
{\displaystyle \prod_{p \in \pos} \cX_p}
\arrow{r}
&
{\displaystyle \prod_{p \in \pos} \cX_p'}
\\
\cX
\arrow{r}[yshift=-0.8cm]{\rotatebox{-45}{$\Leftarrow$}}
\arrow{u}{(\Phi_p)_{p \in \pos}}
&
\cX'
\arrow{u}[swap]{(\Phi_p)_{p \in \pos}}
\\
\cZ_p
\arrow[hook]{u}[swap]{i_R}
\arrow{r}[yshift=-0.8cm]{\rotatebox{-45}{$\Leftarrow$}}
&
\cZ_p'
\arrow[hook]{u}{i_R}
\\
\cX_q
\arrow{r}
\arrow[hook]{u}[swap]{\w{\rho}^q}
\arrow[bend left, hook]{uu}{\rho^q}
&
\cX'_q
\arrow[hook]{u}{\w{\rho}^q}
\arrow[bend right, hook]{uu}[swap]{\rho^q}
\end{tikzcd}
\end{equation}
whose middle square is \Cref{lax.comm.square.with.iRs.for.proving.strict.maps.between.modules.gives.strict.map.in.Strat}. Because the upper two vertical functors in diagram \Cref{lax.comm.diagram.with.iRs.and.rhos.for.proving.strict.maps.between.modules.gives.strict.map.in.Strat} are conservative, its composite natural transformation is an equivalence by \Cref{prop.metacosm.input.first.get.stratn}\Cref{metacosm.input.get.stratn}\Cref{gluing.in.X.is.mdrmy.in.E} (and the fact that $\Phi_p \rho^q \simeq 0$ whenever $q \not< p$). Meanwhile, the same argument (applied to the poset $(^\leq p)$) shows that the lower natural transformation in diagram \Cref{lax.comm.diagram.with.iRs.and.rhos.for.proving.strict.maps.between.modules.gives.strict.map.in.Strat} is also an equivalence. So, the upper natural transformation in diagram \Cref{lax.comm.diagram.with.iRs.and.rhos.for.proving.strict.maps.between.modules.gives.strict.map.in.Strat} is an equivalence on every object in the image of $\w{\rho}^q$. Now, microcosm reconstruction for $\cZ_p$ (\Cref{intro.thm.cosms}\Cref{intro.main.thm.microcosm}) implies that each of its objects is a limit over $\sd(^\leq p)$ of objects in the image of $\w{\rho}^q$ for various $q \in (^\leq p)$ (using the fact that $\sd(^\leq p)$ is finite (because $(^\leq p)$ is finite because $\pos$ is down-finite)). So, because $\sd(^\leq p)$ is finite, the upper natural transformation in diagram \Cref{lax.comm.diagram.with.iRs.and.rhos.for.proving.strict.maps.between.modules.gives.strict.map.in.Strat} is indeed an equivalence; in other words, the lax-commutative square \Cref{lax.comm.square.with.iRs.for.proving.strict.maps.between.modules.gives.strict.map.in.Strat} commutes.

Now, part \Cref{strict.small.metacosm} follows from part \Cref{strict.presentable.metacosm} and the fact that the commutative squares
\[
\begin{tikzcd}
\strat^\strict_\pos
\arrow[hook]{r}
\arrow[hook]{d}
&
\strat_\pos
\arrow[hook]{d}{\Ind}
\\
\Strat^\strict_\pos
\arrow[hook]{r}
&
\Strat_\pos
\end{tikzcd}
\qquad
\text{and}
\qquad
\begin{tikzcd}
\LMod_{\llax.\pos}(\St^\idem)
\arrow[hook]{r}
\arrow[hook]{d}
&
\LMod^\rlax_{\llax.\pos}(\St^\idem)
\arrow[hook]{d}{\LMod^\rlax_{\llax.\pos}(\Ind)}
\\
\LMod^L_{\llax.\pos}(\PrSt)
\arrow[hook]{r}
&
\LMod^{\rlax,L}_{\llax.\pos}(\PrSt)
\end{tikzcd}
\]
are both pullbacks.
\end{proof}

\subsection{Reflection}
\label{subsection.reflection}

In this subsection, we establish the theory of \textit{reflection} for (resp.\! stable) stratifications. We establish reflection for stable stratifications over a finite poset as \Cref{thm.reflection.for.stable.stratns}; using this, we establish reflection for stratifications over a down-finite poset as \Cref{cor.reflection.for.presentable.stratns}.

The key input to the proof of \Cref{thm.reflection.for.stable.stratns} is the fact that a stable stratification of $\cC$ over a finite poset determines a \textit{reflected} stable stratification of $\cC^\op$ over the same poset, which we prove as \Cref{prop.stratn.of.opposite}. Another important input to the proof of \Cref{cor.reflection.for.presentable.stratns} (in addition to \Cref{thm.reflection.for.stable.stratns}) is the fact that a stratification of a presentable stable $\infty$-category may also be considered as a stable stratification thereof, which we prove as \Cref{prop.stratn.gives.stable.stratn}.\footnote{This may be contrasted with the fact that stable stratifications are \textit{definitionally} related to stratifications (although recall \Cref{prop.characterize.stable.stratns}).}

We also give a direct formula for the gluing functors in terms of the reflected gluing functors and reversely, as total co/fibers; this is recorded as \Cref{prop.Gamma.check.as.tfib.of.Gammas.and.Gamma.as.tcofib.of.Gamma.checks}.

\begin{local}
In this subsection we assume that our poset $\pos$ is finite, except in Corollaries \ref{cor.reflection.for.presentable.stratns} \and \ref{cor.reflection.for.finite.intervals} and \Cref{rmk.finite.intervals.but.not.down.finite} (which apply under strictly weaker hypotheses on $\pos$). Moreover, we fix a stable stratification $\cZ_\bullet$ of $\cC$ over $\pos$.
\end{local}

\needspace{2\baselineskip}
\begin{definition}
\begin{enumerate}
\item[]

\item

For any $p \in \pos$, we write
\[ \begin{tikzcd}[column sep=1.5cm]
\cZ_{^< p}
\arrow[hook, bend left=45]{r}[description]{i_L}
\arrow[leftarrow]{r}[transform canvas={yshift=0.1cm}]{\bot}[swap,transform canvas={yshift=-0.1cm}]{\bot}[description]{\yo}
\arrow[bend right=45, hook]{r}[description]{i_R}
&
\cZ_p
\arrow[bend left=45]{r}[description]{p_L}
\arrow[hookleftarrow]{r}[transform canvas={yshift=0.1cm}]{\bot}[swap,transform canvas={yshift=-0.1cm}]{\bot}[description]{\nu}
\arrow[bend right=45]{r}[description]{p_R}
&
\cZ_p
&[-1.8cm]
/ \cZ_{^< p}
=:
\cC_p
\end{tikzcd}
\]
for the idempotent-complete stable quotient participating in the indicated recollement guaranteed by \Cref{cor.union.over.finite.down.closed.gives.small.closed}, which we refer to as the \bit{$p\th$ stratum} of the stable stratification.

\item

For any $p \in \pos$, we write
\[
\begin{tikzcd}[column sep=1.5cm]
\Phi_p
:
\cC
\arrow[transform canvas={yshift=0.9ex}]{r}{y}
\arrow[hookleftarrow, transform canvas={yshift=-0.9ex}]{r}[yshift=-0.2ex]{\bot}[swap]{i_R}
&
\cZ_p
\arrow[transform canvas={yshift=0.9ex}]{r}{p_L}
\arrow[hookleftarrow, transform canvas={yshift=-0.9ex}]{r}[yshift=-0.2ex]{\bot}[swap]{\nu}
&
\cC_p
:
\rho^p
\end{tikzcd}
\qquad
\text{and}
\qquad
\begin{tikzcd}[column sep=1.5cm]
\lambda^p
:
\cC_p
\arrow[hook, transform canvas={yshift=0.9ex}]{r}{\nu}
\arrow[leftarrow, transform canvas={yshift=-0.9ex}]{r}[yshift=-0.2ex]{\bot}[swap]{p_R}
&
\cZ_p
\arrow[hook, transform canvas={yshift=0.9ex}]{r}{i_L}
\arrow[leftarrow, transform canvas={yshift=-0.9ex}]{r}[yshift=-0.2ex]{\bot}[swap]{y}
&
\cC
:
\Psi_p
\end{tikzcd}
\]
for the indicated composite adjoint functors. We respectively refer to the functors $\Phi_p$ and $\Psi_p$ as the corresponding \bit{geometric localization functor} and \bit{reflected geometric localization functor}.

\item

For any morphism $p \ra q$ in $\pos$, we write
\[
\Gamma^p_q
:
\cC_p
\xlonghookra{\rho^p}
\cC
\xra{\Phi_q}
\cC_q
\qquad
\text{and}
\qquad
\widecheck{\Gamma}^p_q
:
\cC_p
\xlonghookra{\lambda^p}
\cC
\xra{\Psi_q}
\cC_q
\]
for the indicated composite functors, which we respectively refer to as the corresponding \bit{gluing functor} and \bit{reflected gluing functor}.

\end{enumerate}
\end{definition}

\begin{definition}
\label{defn.tcofib.and.tfib}
Fix a stable $\infty$-category $\cD \in \St$ and an $\infty$-category $\cJ \in \Cat$.
\begin{enumerate}

\item\label{item.defn.tcofib}

Suppose that $\cJ$ admits a terminal object $t \in \cJ$, and write $\cJ_0 := \cJ \backslash \{ t \}$. Then, for any functor $\cJ \xra{F} \cD$, we define its \bit{total cofiber} to be
\[
\tcofib(F)
:=
\cofib
\left(
\colim_{\cJ_0}(F)
\longra
F(t)
\right)
\in
\cD
~.
\]

\item\label{item.defn.tfib}

Suppose that $\cJ$ admits an initial object $i \in \cJ$, and write $\cJ_0 := \cJ \backslash \{ i \}$. Then, for any functor $\cJ \xra{F} \cD$, we define its \bit{total fiber} to be
\[
\tfib(F)
:=
\fib
\left(
F(i)
\longra
\lim_{\cJ_0}(F)
\right)
\in
\cD
~.
\]

\end{enumerate}
\end{definition}

\begin{remark}
Here are two alternative descriptions of the total cofiber functor (using the notation of \Cref{defn.tcofib.and.tfib}\Cref{item.defn.tcofib}).
\begin{enumerate}

\item It is the left adjoint
\[ \begin{tikzcd}[column sep=1.5cm]
\Fun(\cJ,\cD)
\arrow[dashed, transform canvas={yshift=0.9ex}]{r}{\tcofib}
\arrow[leftarrow, transform canvas={yshift=-0.9ex}]{r}[yshift=-0.2ex]{\bot}[swap]{\delta_t}
&
\cD
\end{tikzcd} \]
to the ``Dirac delta at $t$'' functor (i.e.\! extension by zero over $\cJ_0$).

\item It is the composite
\[
\Fun(\cJ,\cD)
\longra
\Fun(\cJ_+,\cD)
\xra{\colim_{\cJ_+}}
\cD
~,
\]
where we write $\cJ_+ := \cJ \coprod_{\cJ_0} \cJ_0^\rcone$ and the first functor is extension by zero over the cone point of $\cJ_0^\rcone$.

\end{enumerate}
Of course, the total fiber functor admits dual descriptions.
\end{remark}

\begin{prop}
\label{prop.Gamma.check.as.tfib.of.Gammas.and.Gamma.as.tcofib.of.Gamma.checks}

Fix a nonidentity morphism $p < q$ in $\pos$.

\begin{enumerate}

\item\label{part.Gamma.check.as.tfib.of.Gammas}

There is a canonical equivalence
\[
\widecheck{\Gamma}^p_q
\simeq
\tfib
\left(
\sd(\pos)^{|p}_{|q}
\xra{\Sigma^{-1} \Gamma_\bullet}
\Fun^\ex(\cC_p,\cC_q)
\right)
\]
in $\Fun^\ex(\cC_p,\cC_q)$, where the functor $\Gamma_\bullet$ is given by \Cref{obs.assemble.Gamma.phi.functorially.in.phi}.

\item\label{part.Gamma.as.tcofib.of.Gamma.checks}

There is a canonical equivalence
\[
\Gamma^p_q
\simeq
\tcofib
\left(
\left( \sd(\pos)^{|p}_{|q} \right)^\op
\xra{\Sigma \widecheck{\Gamma}_\bullet}
\Fun^\ex(\cC_p,\cC_q)
\right)
\]
in $\Fun^\ex(\cC_p,\cC_q)$, where the functor $\widecheck{\Gamma}_\bullet$ is given by \Cref{obs.assemble.Gamma.phi.functorially.in.phi} and \Cref{prop.stratn.of.opposite}.\footnote{\Cref{prop.stratn.of.opposite} does not rely on the present result in any way.}

\end{enumerate}
\end{prop}

\begin{lemma}
\label{lemma.reflected.gluing.diagram.in.brax.one.case}
Given a stable recollement \Cref{stable.recollement.from.closed.subcat}, there is a canonical equivalence
\[
p_R i_L
\simeq
\Sigma^{-1} p_L i_R
~;
\]
that is, \Cref{prop.Gamma.check.as.tfib.of.Gammas.and.Gamma.as.tcofib.of.Gamma.checks} holds when $\pos = [1]$.
\end{lemma}

\begin{proof}
Consider the commutative diagram
\[ \begin{tikzcd}
\nu p_R i_L y
\arrow{r}
\arrow{d}
&
i_L y
\arrow{r}
\arrow{d}
&
0
\arrow{d}
\\
\nu p_R
\arrow{r}
\arrow{d}
&
\id_\cC
\arrow{r}
\arrow{d}
&
\nu p_L
\arrow{d}
\\
0
\arrow{r}
&
i_R y
\arrow{r}
&
\nu p_L i_R y
\end{tikzcd} \]
in $\Fun^\ex(\cC,\cC)$ (in which all morphisms are co/units or zero). Note that the lower-right square is a pullback by (the proof of) \Cref{lem.reconstrn.for.recollement}, while the upper-left square is a pushout by an identical argument (or by appealing to \Cref{prop.stratn.of.opposite}). On the other hand, the lower-left and upper-right squares are clearly both co/fiber sequences. So, the outer square is a pullback, which proves the claim (by precomposing with either $i_L$ or $i_R$ and postcomposing with either $p_L$ or $p_R$).
\end{proof}

\begin{proof}[Proof of \Cref{prop.Gamma.check.as.tfib.of.Gammas.and.Gamma.as.tcofib.of.Gamma.checks}]
Part \Cref{part.Gamma.as.tcofib.of.Gamma.checks} follows from part \Cref{part.Gamma.check.as.tfib.of.Gammas} by applying \Cref{prop.stratn.of.opposite}, so it suffices to verify the latter.

We reduce to the case that $p \in \pos$ is initial and $q \in \pos$ is terminal, using \Cref{obs.image.of.strat.in.Strat} in order to apply our previous results regarding stratifications of presentable stable $\infty$-categories, as follows. First of all, by passing to the restricted stratification of $\cZ_q$ over $(^\leq q) \in \Down_\pos$ (\Cref{obs.restricted.stratn.over.D}), we may clearly assume that $q \in \pos$ is terminal. From here, we claim that passing to the quotient stratification of $\cC / \cZ_{^{\not \geq} p}$ over $(^\geq p) = \pos \backslash (^{\not\geq} p)$ (\Cref{obs.quotient.stratn.from.down.closed.over.smaller.poset}) allows us to assume that $p \in \pos$ is initial. On the one hand, the fact that the gluing diagram of $\cC$ over $\pos$ restricts to that of $\cC/\cZ_{^{\not\geq} p}$ over $\pos \backslash (^{\not\geq} p)$ (as explained in \Cref{obs.quotient.stratn.from.down.closed.over.smaller.poset}) implies that the diagram
\[ \begin{tikzcd}
\sd(\pos)^{|p}_{|q}
\arrow{r}{\Gamma_\bullet}
&
\Fun^\ex(\cC_p,\cC_q)
\\
\sd(\pos \backslash (^{\not\geq} p))^{|p}_{|q}
\arrow{u}[sloped, anchor=south]{\sim}
\arrow{r}[swap]{\Gamma_\bullet}
&
\Fun^\ex( ( \cC / \cZ_{^{\not \geq} p} )_p , ( \cC / \cZ_{^{\not \geq} p} )_q )
\arrow[leftrightarrow]{u}[sloped, anchor=north]{\sim}
\end{tikzcd} \]
commutes. On the other hand, the diagram
\[ \begin{tikzcd}
\cZ_p
\arrow[hook]{r}{i_L}
&
\cC
\arrow{r}{p_R}
&
\cC_q
\\
\cC_p
\arrow[hook]{r}[swap]{i_L}
\arrow[hook]{u}{\nu}
&
\cC/\cZ_{^{\not\geq} p}
\arrow{r}[swap]{p_R}
\arrow[hook]{u}{\nu}
&
(\cC/\cZ_{^{\not\geq} p})_q
\arrow{u}[sloped, anchor=north]{\sim}
\end{tikzcd} \]
commutes: the left square commutes by \Cref{lemma.all.about.aligned.subcats}\Cref{part.alignment.lemma.induced.map.on.quotients}\Cref{subpart.alignment.lemma.nu.commutativity} (which applies by \Cref{closed.subcats.are.mutually.aligned}), while the right square commutes by passing to left adjoints (recall \Cref{prop.quotient.stratn}). It follows that we obtain a commutative diagram
\[ \begin{tikzcd}[row sep=0cm]
\cC_p
\arrow{r}{\widecheck{\Gamma}^p_q}
&
\cC_q
\\
\rotatebox{90}{$\simeq$}
&
\rotatebox{90}{$\simeq$}
\\
(\cC/\cZ_{^{\not\geq} p})_p
\arrow{r}[swap]{\widecheck{\Gamma}^p_q}
&
(\cC/\cZ_{^{\not\geq} p})_q
\end{tikzcd}
~.
\]
So we may indeed assume that $p \in \pos$ is initial.

To simplify our notation, we write $\pos' := \pos \backslash \{ q \} = (^< q) \in \Down_\pos$. Observe that we obtain a stable recollement
\begin{equation}
\label{stable.recollement.in.proof.of.Gamma.and.Gamma.check.in.terms.of.each.other}
\begin{tikzcd}[column sep=1.5cm]
\cC_{\pos'}
\arrow[hook, bend left=45]{r}[description]{i_L}
\arrow[leftarrow]{r}[transform canvas={yshift=0.1cm}]{\bot}[swap,transform canvas={yshift=-0.1cm}]{\bot}[description]{\yo}
\arrow[bend right=45, hook]{r}[description]{i_R}
&
\cC
\arrow[bend left=45]{r}[description]{p_L}
\arrow[hookleftarrow]{r}[transform canvas={yshift=0.1cm}]{\bot}[swap,transform canvas={yshift=-0.1cm}]{\bot}[description]{\nu}
\arrow[bend right=45]{r}[description]{p_R}
&
\cC_q
\end{tikzcd}
~.
\end{equation}
To finish the proof, we work within the context of the commutative diagram
\begin{equation}
\label{big.comm.diagram.for.identifying.reflected.gluing.functor.in.terms.of.tfib}
\begin{tikzcd}[row sep=1.5cm]
\cC_p
\arrow[hook]{d}[swap]{i_L}
\arrow[bend left]{rrd}[sloped]{\Sigma \widecheck{\Gamma}^p_q}
\\
\cC_{\pos'}
\arrow[hook]{r}{i_R}
\arrow{d}[swap]{L_\bullet}
\arrow[bend left=25]{rr}{\Sigma p_R i_L}
&
\cC
\arrow{r}{p_L}
&
\cC_q
\\
\Fun(\sd(\pos') , \cC_{\pos'})
\arrow[hook]{r}[swap]{i_R}
&
\Fun(\sd(\pos') , \cC )
\arrow{r}[swap]{p_L}
&
\Fun(\sd(\pos') , \cC_q)
\arrow{r}[swap]{\pi_*}
\arrow{u}{\lim_{\sd(\pos')}}
&
\Fun ( \sd([1]) , \cC_q )
\arrow{lu}[sloped]{\lim_{\sd([1])}}
\end{tikzcd}
\end{equation}
in $\St$, in which
\begin{itemize}

\item the upper left (curved) triangle commutes by definition of $\widecheck{\Gamma}^p_q$,

\item the middle (flattened) triangle commutes by applying \Cref{lemma.reflected.gluing.diagram.in.brax.one.case} to the stable recollement \Cref{stable.recollement.in.proof.of.Gamma.and.Gamma.check.in.terms.of.each.other},

\item the lower left vertical functor $L_\bullet$ is the functor \Cref{L.bullet.as.a.functor.to.coaugmented.endofunctors.of.X} of \Cref{obs.functoriality.of.L.phi.in.the.variable.phi},

\item the lower left rectangle commutes due to the equivalence $\lim_{\sd(\pos')} \circ L_\bullet \simeq \id_{\cC_{\pos'}}$ that follows from \Cref{macrocosm.thm} (using that $\sd(\pos')$ is finite),

\item the functor $\sd(\pos') \xra{\pi} \sd([1])$ to the walking cospan is given by the prescriptions
\[
\pi^{-1}(0)
=
\{([0] \xra{p} \pos') \}
~,
\qquad
\pi^{-1}(01)
=
\sd(\pos')^{|p} \backslash ( [ 0 ] \xra{p} \pos')
~,
\qquad
\text{and}
\qquad
\pi^{-1}(1)
=
\sd(\pos') \backslash \sd(\pos')^{|p}
~,
\]
and

\item the lower right triangle commutes because right Kan extensions compose.

\end{itemize}
We claim that the composite exact functor $\cC_p \ra \Fun( \sd([1]) , \cC_q)$ in diagram \Cref{big.comm.diagram.for.identifying.reflected.gluing.functor.in.terms.of.tfib} selects the evident cospan
\begin{equation}
\label{cospan.whose.limit.is.total.fiber.in.formula.for.Gamma.and.Gamma.check.in.terms.of.each.oter}
\begin{tikzcd}
&
0
\arrow{d}
\\
\Gamma^p_q
\arrow{r}
&
\lim_{\sd(\pos)^{|p}_{|q} \backslash ( [1] \xra{p < q} \pos)} ( \Gamma_\bullet )
\end{tikzcd}
\end{equation}
in $\Fun^\ex(\cC_p,\cC_q)$. To see this, observe first that the functor $\sd(\pos') \xra{\pi} \sd([1])$ is a cartesian fibration: over the morphism $0 \ra 01$ this follows from the fact that the object $([0] \xra{p} \pos') \in \sd(\pos')^{|p}$ is initial, while its cartesian monodromy over the morphism $1 \ra 01$ is given by removing the element $p \in \pos'$ from each object. Therefore, the right Kan extension $\pi_*$ is computed by fiberwise limit. From here, it suffices to make the following two observations.
\begin{itemize}

\item The composite functor $\cC_p \xhookra{i_L} \cC_{\pos'} \xra{L_\bullet} \Fun ( \sd(\pos') , \cC_{\pos'})$ takes values in the full subcategory of functors that restrict to zero on $\sd(\pos') \backslash \sd(\pos')^{|p}$. This implies that the composite functor $\cC_p \ra \Fun(\sd([1]) , \cC_q) \xra{\ev_1} \cC_q$ is indeed zero.

\item For each object $([n] \xra{\varphi} \pos') \in \sd(\pos')^{|p}$, the composite functor
\[
\cC_p
\xlonghookra{i_L}
\cC_{\pos'}
\xra{L_\bullet}
\Fun ( \sd(\pos') , \cC_{\pos'})
\xra{\ev_\varphi}
\cC_{\pos'}
\]
is canonically equivalent to the composite
\[
\cC_p
=
\cC_{\min(\varphi)}
\xra{\Gamma_\varphi}
\cC_{\max(\varphi)}
\xhookra{\rho^{\max(\varphi)}}
\cC_{\pos'}
~.
\]
This explains the identification of the composite functor $\cC_p \ra \Fun(\sd([1]),\cC_q)$ at the objects $0,01 \in \sd([1])$.

\end{itemize}
So, the claim follows from the fact that the limit of the cospan \Cref{cospan.whose.limit.is.total.fiber.in.formula.for.Gamma.and.Gamma.check.in.terms.of.each.oter} is by definition the total fiber of the functor $\sd(\pos)^{|p}_{|q} \xra{\Gamma_\bullet} \Fun^\ex(\cC_p,\cC_q)$.
\end{proof}

\begin{remark}
\label{rmk.introduce.gluing.and.reflected.gluing.diagrams}
We now simultaneously introduce the gluing diagram and reflected gluing diagram of our stable stratification $\cZ_\bullet$ of $\cC$ over $\pos$. The former was already implicitly defined in \Cref{thm.stable.metacosm}, but we nevertheless spell it out here for clarity and in order to highlight the comparison.\footnote{Inspecting \Cref{defn.gluing.diagram} and \Cref{thm.stable.metacosm}, it is clear that \Cref{defn.stable.gluing.diagram.and.stable.reflected.gluing.diagram}\Cref{defn.part.stable.gluing.diagram} coincides with the gluing diagram as implicitly defined in \Cref{thm.stable.metacosm}.}
\end{remark}

\needspace{2\baselineskip}
\begin{notation}
\begin{enumerate}
\item[]

\item

We define the full subcategory
\[
\GD(\cC)
:=
\{ (X,p) \in \cC \times \pos
:
X \in \rho^p(\cC_p)
\}
\subseteq
\cC \times \pos
~,
\]
which we consider as an object of $\Cat_{/\pos}$.

\item

We define the full subcategory
\[
\widecheck{\GD}(\cC)
:=
\{ (X,p^\circ) \in \cC \times \pos^\op
:
X \in \lambda^p(\cC_p)
\}
\subseteq
\cC \times \pos^\op
~,
\]
which we consider as an object of $\Cat_{/\pos^\op}$.
\end{enumerate}
\end{notation}

\needspace{2\baselineskip}
\begin{observation}
\label{obs.stable.gluing.diagram.and.stable.reflected.gluing.diagram}
\begin{enumerate}
\item[]

\item\label{item.obs.stable.gluing.diagram}

The functor
\[
\GD(\cC)
\longra
\pos
\]
is a locally cocartesian fibration, whose monodromy functor over each morphism $p \ra q$ in $\pos$ is the functor
\[
\cC_p
\xra{\Gamma^p_q}
\cC_q
~.
\]
Moreover, its fibers are idempotent-complete and stable and its monodromy functors are exact. We therefore consider it as defining an object
\[
\GD(\cC)
\in
\LMod_{\llax.\pos}(\St^\idem)
\subseteq
\LMod_{\llax.\pos}
:=
\loc.\coCart_\pos
~.
\]

\item\label{item.obs.stable.reflected.gluing.diagram}

The functor
\[
\widecheck{\GD}(\cC)
\longra
\pos^\op
\]
is a locally cartesian fibration, whose monodromy functor over each morphism $p^\circ \la q^\circ$ in $\pos^\op$ is the functor
\[
\cC_p
\xra{\widecheck{\Gamma}^p_q}
\cC_q
~.
\]
Moreover, its fibers are idempotent-complete and stable and its monodromy functors are exact. We therefore consider it as defining an object
\[
\widecheck{\GD}(\cC)
\in
\LMod_{\rlax.\pos}(\St^\idem)
\subseteq
\LMod_{\rlax.\pos}
:=
\RMod_{\rlax.\pos^\op}
:=
\loc.\Cart_{\pos^\op}
~.
\]

\end{enumerate}
\end{observation}

\needspace{2\baselineskip}
\begin{definition}
\label{defn.stable.gluing.diagram.and.stable.reflected.gluing.diagram}
\begin{enumerate}
\item[]

\item\label{defn.part.stable.gluing.diagram}

We refer to the object
\[
\GD(\cC)
\in
\LMod_{\llax.\pos}(\St^\idem)
\]
as the \bit{gluing diagram} of the stratification.

\item\label{defn.part.stable.reflected.gluing.diagram}

We refer to the object
\[
\widecheck{\GD}(\cC)
\in
\LMod_{\rlax.\pos}(\St^\idem)
\]
as the \bit{reflected gluing diagram} of the stratification.

\end{enumerate}
\end{definition}

\begin{theorem}
\label{thm.reflection.for.stable.stratns}
There is a canonical commutative diagram
\[ \begin{tikzcd}[column sep=1.5cm, row sep=1.5cm]
&
{\displaystyle \prod_{p \in \pos} \St^\idem}
\\
\LMod_{\rlax.\pos}(\St^\idem)
\arrow[yshift=0.9ex]{r}{\limllaxfam}
\arrow[leftarrow, yshift=-0.9ex]{r}{\sim}[swap]{\widecheck{\GD}}
\arrow{ru}[sloped]{(\ev_p)_{p \in \pos}}
\arrow{rd}[sloped, swap]{\lim^\llax_{\rlax.\pos}}
&
\strat_\pos^\strict
\arrow[yshift=0.9ex]{r}[sloped]{\GD}
\arrow[leftarrow, yshift=-0.9ex]{r}{\sim}[sloped, swap]{\limrlaxfam}
\arrow{d}{\fgt}
\arrow{u}[swap]{((-)_p)_{p \in \pos}}
&
\LMod_{\llax.\pos}(\St^\idem)
\arrow{lu}[sloped]{(\ev_p)_{p \in \pos}}
\arrow{ld}[sloped, swap]{\lim^\rlax_{\llax.\pos}}
\\
&
\St^\idem
\end{tikzcd}
~.
\]
\end{theorem}

\begin{remark}
In what follows, we use the notation $(-)^\refl$ to denote opposite $\infty$-categories that are considered in some nonstandard (``reflected'') way. (We explain both usages of this notation as they arise; see \Cref{notn.reflected.closed.subcat} \and \Cref{defn.reflected.stable.stratn}.) We continue to use the notation $(-)^\op$ to denote the opposite $\infty$-category considered in its own right.
\end{remark}

\begin{definition}
\label{defn.reflected.closed.subcat}
The \bit{reflected closed subcategory} (or simply \bit{reflection}) of a closed subcategory $\cZ \in \clssub_\cC$ of $\cC$ is the closed subcategory
\[
\cZ^\op
\xlonghookra{i_R^\op}
\cC^\op
\]
of $\cC^\op$.\footnote{The right adjoint of $i_R^\op$ is $y^\op$, and the right adjoint of $y^\op$ is $i_L^\op$; see \Cref{notn.reflected.closed.subcat}.}
\end{definition}

\begin{observation}
\label{obs.reflection.is.an.equivalence.and.squares.to.identity}
Passage to reflected closed subcategories determines an equivalence
\[ \begin{tikzcd}[row sep=0cm]
\clssub_\cC
\arrow{r}{(-)^\refl}[swap]{\sim}
&
\clssub_{\cC^\op}
\\
\rotatebox{90}{$\in$}
&
\rotatebox{90}{$\in$}
\\
\left( \cZ \xlonghookra{i_L} \cC \right)
\arrow[maps to]{r}
&
\left( \cZ^\op \xlonghookra{i_R^\op} \cC^\op \right)
\end{tikzcd}~, \]
which when applied twice yields the identity functor
\[
\id_{\clssub_\cC}
:
\clssub_\cC
\xra[\sim]{(-)^\refl}
\clssub_{\cC^\op}
\xra[\sim]{(-)^\refl}
\clssub_{(\cC^\op)^\op}
\simeq
\clssub_\cC
~.
\]
We use these facts without further comment.
\end{observation}

\begin{notation}
\label{notn.reflected.closed.subcat}
As indicated in \Cref{obs.reflection.is.an.equivalence.and.squares.to.identity}, given a closed subcategory $\cZ \in \clssub_\cC$ we write $\cZ^\refl \in \clssub_{\cC^\op}$ for its reflection. Moreover, we write
\[
\begin{tikzcd}[column sep=1.5cm]
\cZ^\refl
\arrow[hook, bend left=50]{r}[description]{i_L^\refl}
\arrow[leftarrow]{r}[transform canvas={yshift=0.15cm}]{\bot}[swap,transform canvas={yshift=-0.15cm}]{\bot}[description]{\yo^\refl}
\arrow[bend right=50, hook]{r}[description]{i_R^\refl}
&
\cC^\op
\arrow[bend left=50]{r}[description]{p_L^\refl}
\arrow[hookleftarrow]{r}[transform canvas={yshift=0.15cm}]{\bot}[swap,transform canvas={yshift=-0.15cm}]{\bot}[description]{\nu^\refl}
\arrow[bend right=50]{r}[description]{p_R^\refl}
&
\cC^\op
&[-1.8cm]
/ \cZ^\refl
\end{tikzcd}
\qquad
:=
\qquad
\begin{tikzcd}[column sep=1.5cm]
\cZ^\op
\arrow[hook, bend left=50]{r}[description]{i_R^\op}
\arrow[leftarrow]{r}[transform canvas={yshift=0.15cm}]{\bot}[swap,transform canvas={yshift=-0.15cm}]{\bot}[description]{\yo^\op}
\arrow[bend right=50, hook]{r}[description]{i_L^\op}
&
\cC^\op
\arrow[bend left=50]{r}[description]{p_R^\op}
\arrow[hookleftarrow]{r}[transform canvas={yshift=0.15cm}]{\bot}[swap,transform canvas={yshift=-0.15cm}]{\bot}[description]{\nu^\op}
\arrow[bend right=50]{r}[description]{p_L^\op}
&
(\cC
&[-1.8cm]
/ \cZ)^\op
\end{tikzcd}
\]
for the functors in the stable recollement that is opposite to the stable recollement \Cref{stable.recollement.from.closed.subcat}.\footnote{Here, we implicitly use the evident fact that the involution $\St^\idem \xra[\sim]{(-)^\op} \St^\idem$ preserves stable recollements.}
\end{notation}

\begin{prop}
\label{prop.stratn.of.opposite}
The composite
\[
\cZ_\bullet^\refl
:
\pos
\xra{\cZ_\bullet}
\clssub_\cC
\xra[\sim]{(-)^\refl}
\clssub_{\cC^\op}
\]
is a stable stratification of $\cC^\op$ over $\pos$.
\end{prop}

\begin{observation}
Passage to opposites defines an equivalence
\[ \begin{tikzcd}[row sep=0cm]
\thicksub_\cC
\arrow[leftrightarrow]{r}{(-)^\op}[swap]{\sim}
&
\thicksub_{\cC^\op}
\\
\rotatebox{90}{$\in$}
&
\rotatebox{90}{$\in$}
\\
(\cZ \subseteq \cC)
\arrow[maps to]{r}
&
(\cZ^\op \subseteq \cC^\op)
\end{tikzcd}
~.
\]
In particular, it preserves colimits, so that given a set $\{ \cY_s \in \thicksub_\cC \}_{s \in S}$ of thick subcategories of $\cC$, we have an identification
\[
\brax{ \cY_s^\op }^\thick_{s \in S}
=
\left( \brax{ \cY_s }^\thick_{s \in S} \right)^\op
\in
\thicksub_{\cC^\op}
~.
\]
We use this fact without further comment.
\end{observation}

\begin{proof}[Proof of \Cref{prop.stratn.of.opposite}]
We apply the criteria of \Cref{prop.characterize.stable.stratns}.

We begin with condition \Cref{item.stable.prestratns} of \Cref{prop.characterize.stable.stratns}. Observe first that
\[
\brax{i_R(\cZ_p)}^\thick_{p \in \pos}
\supseteq
\brax{ \rho^p(\cC_p) }^\thick_{p \in \pos}
=
\cC
~,
\]
where the equality is guaranteed by \Cref{thm.stable.metacosm} (and the fact that $\pos$ is finite). Since $\cC \in \thicksub_\cC$ is terminal, this implies the equality
\begin{equation}
\label{iR.inclusions.thickly.generate.small.stable.C}
\brax{ i_R ( \cZ_p ) }^\thick_{p \in \pos}
=
\cC
\in
\thicksub_\cC
~,
\end{equation}
which implies the equality
\begin{align*}
\brax{\cZ_p^\refl}^\thick_{p \in \pos}
& :=
\brax{ i_L^\refl(\cZ_p^\refl) }^\thick_{p \in \pos}
:=
\brax{ i_R^\op(\cZ_p^\op) }^\thick_{p \in \pos}
=
\brax{ i_R(\cZ_p)^\op }^\thick_{p \in \pos}
\\
& =
\left( \brax{ i_R(\cZ_p) }^\thick_{p \in \pos} \right)^\op
=
\cC^\op
\in
\thicksub_{\cC^\op}
~.
\end{align*}

Before turning to condition \Cref{stable.stratn.condition} of \Cref{prop.characterize.stable.stratns}, we make some preliminary deductions.

Fix any $p,q \in \pos$. First of all, applying \Cref{prop.characterize.stable.stratns}\Cref{stable.stratn.condition}\Cref{stable.stratn.condition.thick.is.closed} to the stable stratification $\pos \xra{\cZ_\bullet} \clssub_\cC$, we find that the thick subcategory
\[
\cZ_{(^\leq p) \cap (^\leq q)}
:=
\brax{ \cZ_r }^\thick_{r \in (^\leq p) \cap (^\leq q)}
\in
\thicksub_\cC
\]
is a closed subcategory. Thereafter, by \Cref{obs.restricted.stratn.over.D} it is clear that the evident factorization
\begin{equation}
\label{factorizn.for.stable.stratn.of.Z.leq.p.and.leq.q.for.stable.stratn.of.opposite}
\begin{tikzcd}
\pos
\arrow{r}{\cZ_\bullet}
&
\clssub_\cC
\\
(^\leq p) \cap (^\leq q)
\arrow[dashed]{r}[swap]{\cZ_\bullet}
\arrow[hook]{u}
&
\clssub_{\cZ_{(^\leq p) \cap (^\leq q)}}
\arrow[hook]{u}
\end{tikzcd}
\end{equation}
defines a stable stratification of $\cZ_{(^\leq p) \cap (^\leq q)}$ over $(^\leq p) \cap (^\leq q)$. For any $r \in (^\leq p) \cap (^\leq q)$, let us denote by
\[
\cZ_r
\xlonghookra{\w{i_R}}
\cZ_{(^\leq p) \cap (^\leq r)}
\]
the corresponding $i_R$ inclusion, so that we have a commutative triangle
\begin{equation}
\label{commutative.triangle.of.iR.inclusions.for.showing.opposite.is.stratified}
\begin{tikzcd}
\cZ_r
\arrow[hook]{r}{\w{i_R}}
\arrow[hook]{rd}[sloped, swap]{i_R}
&
\cZ_{(^\leq p) \cap (^\leq q)}
\arrow[hook]{d}{i_R}
\\
&
\cC
\end{tikzcd}
~.
\end{equation}
Then, the equality \Cref{iR.inclusions.thickly.generate.small.stable.C} applied to the factorization \Cref{factorizn.for.stable.stratn.of.Z.leq.p.and.leq.q.for.stable.stratn.of.opposite} 
becomes an equality
\[
\brax{ \w{i_R}(\cZ_r) }^\thick_{r \in (^\leq p) \cap (^\leq q)}
=
\cZ_{(^\leq p) \cap (^\leq q)}
\in
\thicksub_{\cZ_{(^\leq p) \cap (^\leq q)}}
~,
\]
which by the commutativity of the triangle \Cref{commutative.triangle.of.iR.inclusions.for.showing.opposite.is.stratified} yields an equality
\[
\brax{ i_R ( \cZ_r )}^\thick_{r \in (^\leq p) \cap (^\leq q)}
=
i_R(\cZ_{(^\leq p) \cap (^\leq q)})
\in
\thicksub_\cC
~.
\]
This implies the composite equality
\begin{align*}
\brax{ \cZ_r^\refl }^\thick_{r \in (^\leq p) \cap (^\leq q)}
& :=
\brax{ i_L^\refl(\cZ_r^\refl) }^\thick_{r \in (^\leq p) \cap (^\leq q)}
:=
\brax{ i_R^\op ( \cZ_r^\op ) }^\thick_{r \in (^\leq p) \cap (^\leq q)}
\\
&
=
\brax{ i_R ( \cZ_r)^\op }^\thick_{r \in (^\leq p) \cap (^\leq q)}
=
\left(
\brax{ i_R ( \cZ_r ) }^\thick_{r \in (^\leq p) \cap (^\leq q)}
\right)^\op
\\
& =
i_R(\cZ_{(^\leq p) \cap (^\leq q)})^\op
\in
\thicksub_{\cC^\op}
~,
\end{align*}
which we record for readability as the equality
\begin{equation}
\label{identify.subcat.thickly.gend.by.dual.iLs.as.iR.of.intersection.op}
\brax{ \cZ_r^\refl }^\thick_{r \in (^\leq p) \cap (^\leq q)}
=
i_R(\cZ_{(^\leq p) \cap (^\leq q)})^\op
\in
\thicksub_{\cC^\op}
~.
\end{equation}

We now turn to condition \Cref{stable.stratn.condition} of \Cref{prop.characterize.stable.stratns}. Applying \Cref{prop.characterize.stable.stratns}\Cref{stable.stratn.condition}\Cref{stable.stratn.condition.thick.is.closed} to the stable stratification $\pos \xra{\cZ_\bullet} \clssub_\cC$, we immediately find that equality \Cref{identify.subcat.thickly.gend.by.dual.iLs.as.iR.of.intersection.op} implies part \Cref{stable.stratn.condition.thick.is.closed} of condition \Cref{stable.stratn.condition} of \Cref{prop.characterize.stable.stratns}. Then, applying \Cref{prop.characterize.stable.stratns}\Cref{stable.stratn.condition}\Cref{stable.stratn.condition.usual.factorizn} to the stable stratification $\pos \xra{\cZ_\bullet} \clssub_\cC$ with the roles of $p$ and $q$ reversed and invoking \Cref{lem.equivalent.characterizations.of.alignment}, we obtain a commutative square
\[ \begin{tikzcd}
\cZ_{(^\leq p) \cap (^\leq q)}
\arrow[hook]{d}[swap]{i_L}
&
\cZ_p
\arrow[hook]{d}{i_L}
\arrow{l}[swap]{y}
\\
\cZ_q
&
\cC
\arrow{l}{y}
\end{tikzcd}~, \]
which upon passing to right adjoints yields a commutative square
\[
\begin{tikzcd}
\cZ_{(^\leq p) \cap (^\leq q)}
\arrow[hook]{r}{i_R}
&
\cZ_p
\\
\cZ_q
\arrow{u}{y}
\arrow[hook]{r}[swap]{i_R}
&
\cC
\arrow{u}[swap]{y}
\end{tikzcd}~,
\]
which upon passing to opposites and applying equality \Cref{identify.subcat.thickly.gend.by.dual.iLs.as.iR.of.intersection.op} yields a commutative square
\[
\begin{tikzcd}
\brax{ \cZ_r^\refl }_{r \in (^\leq p) \cap (^\leq q)}
\arrow[hook]{r}{i_L^\refl}
&
\cZ_p^\refl
\\
\cZ_q^\refl
\arrow{u}{y^\refl}
\arrow[hook]{r}[swap]{i_L^\refl}
&
\cC^\op
\arrow{u}[swap]{y^\refl}
\end{tikzcd}
~,
\]
which verifies part \Cref{stable.stratn.condition.usual.factorizn} of condition \Cref{stable.stratn.condition} of \Cref{prop.characterize.stable.stratns}.
\end{proof}

\begin{definition}
\label{defn.reflected.stable.stratn}
We refer to the stable stratification $\cZ_\bullet^\refl$ of $\cC^\op$ over $\pos$ of \Cref{prop.stratn.of.opposite} as the \bit{reflected stable stratification} (or simply the \bit{reflection}) of the stable stratification $\cZ_\bullet$ of $\cC$ over $\pos$, and we denote it by
\[
\cC^\refl
:=
\left( \pos \xra{\cZ_\bullet^\refl} \clssub_{\cC^\op} \right)
\in
\strat_\pos^\strict
~.
\]
\end{definition}

\begin{observation}
\label{obs.involution.of.strat.strict}
Passage to reflected stable stratifications determines an involution
\[ \begin{tikzcd}[row sep=0cm]
\strat^\strict_\pos
\arrow[leftrightarrow]{r}{(-)^\refl}[swap]{\sim}
&
\strat^\strict_\pos
\\
\rotatebox{90}{$\in$}
&
\rotatebox{90}{$\in$}
\\
\cC
\arrow[maps to]{r}
&
\cC^\refl
\end{tikzcd}
~.
\]
\end{observation}

\needspace{2\baselineskip}
\begin{observation}
\label{obs.lax.modules.and.lax.limits}
\begin{enumerate}
\item[]

\item\label{nonstable.equivce.betw.llax.and.rlax.LModP}

There is a canonical equivalence
\begin{equation}
\label{equivce.op.between.llax.P.and.rlax.P}
\begin{tikzcd}[row sep=0cm]
\LMod_{\llax.\pos}
:=
&[-1cm]
\loc.\coCart_\pos
\arrow[leftrightarrow]{r}{(-)^\op}[swap]{\sim}
&
\loc.\Cart_{\pos^\op}
&[-1.1cm]
&[-3.2cm]
=:
\RMod_{\rlax.\pos^\op}
=:
\LMod_{\rlax.\pos}
\\
&
\rotatebox{90}{$\in$}
&
\rotatebox{90}{$\in$}
\\
&
(\cE \da \pos)
\arrow[maps to]{r}
&
(\cE \da \pos)^\op
&
:=
(\cE^\op \da \pos^\op)
\end{tikzcd}~.
\end{equation}

\item\label{item.restricted.equivalence.op.between.stable.idem.llax.P.and.rlax.P}

The equivalence \Cref{equivce.op.between.llax.P.and.rlax.P} of part \Cref{nonstable.equivce.betw.llax.and.rlax.LModP} restricts to an equivalence
\[
\begin{tikzcd}
\LMod_{\llax.\pos}(\St^\idem)
\arrow[dashed, leftrightarrow]{r}{(-)^\op}[swap]{\sim}
\arrow[hook]{d}
&
\LMod_{\rlax.\pos}(\St^\idem)
\arrow[hook]{d}
\\
\LMod_{\llax.\pos}
\arrow[leftrightarrow]{r}{\sim}[swap]{(-)^\op}
&
\LMod_{\rlax.\pos}
\end{tikzcd}
~.
\]

\item\label{item.fiberwise.opposite.aligns.lax.limits}

In view of the identification
\[
\left( \sd(\pos) \xra{\max} \pos \right)^\op
\simeq
\left( \sd(\pos^\op)^\op \xra{\min} \pos^\op \right)
~,
\]
the equivalence \Cref{equivce.op.between.llax.P.and.rlax.P} of part \Cref{nonstable.equivce.betw.llax.and.rlax.LModP} participates in a commutative square
\[
\begin{tikzcd}
\LMod_{\llax.\pos}
\arrow[leftrightarrow]{r}{(-)^\op}[swap]{\sim}
\arrow{d}[swap]{\lim^\rlax_{\llax.\pos}}
&
\LMod_{\rlax.\pos}
\arrow{d}{\lim^\llax_{\rlax.\pos}}
\\
\Cat
\arrow[leftrightarrow]{r}[swap]{(-)^\op}{\sim}
&
\Cat
\end{tikzcd}
~.\footnote{This may be seen as resulting from the fact that the equivalence \Cref{equivce.op.between.llax.P.and.rlax.P} between $\infty$-categories enhances to an equivalence $\LMod_{\llax.\pos} \simeq ( \LMod_{\rlax.\pos})^{2\op}$ between $(\infty,2)$-categories.}
\]

\end{enumerate}
\end{observation}

\begin{proof}[Proof of \Cref{thm.reflection.for.stable.stratns}]
First of all, it follows immediately from the definitions that the upper two triangles commute. Next, the inverse equivalences on the right are precisely the content of \Cref{thm.strict.metacosm}\Cref{strict.small.metacosm}, which also implies the commutativity of the lower right triangle. Thereafter, unwinding its definition, we see that the construction $\widecheck{\GD}$ is precisely the composite functor
\[
\widecheck{\GD}
:
\strat^\strict_\pos
\xra[\sim]{(-)^\refl}
\strat^\strict_\pos
\xra[\sim]{\GD}
\LMod_{\llax.\pos}(\St^\idem)
\xra[\sim]{(-)^\op}
\LMod_{\rlax.\pos}(\St^\idem)
\]
(as asserted by \Cref{thm.reflection.for.stable.stratns}), in which the three functors are respectively equivalences by \Cref{obs.involution.of.strat.strict}, \Cref{thm.strict.metacosm}\Cref{strict.small.metacosm}, and \Cref{obs.lax.modules.and.lax.limits}\Cref{item.restricted.equivalence.op.between.stable.idem.llax.P.and.rlax.P}. This implies that the functor $\widecheck{\GD}$ is indeed an equivalence. Combining these three results with \Cref{obs.lax.modules.and.lax.limits}\Cref{item.fiberwise.opposite.aligns.lax.limits} justifies the notation $\limllaxfam$ for its inverse (referring to the evident analog of \Cref{notn.stratn.of.one.rlaxlim.at.a.time}), and in particular implies the commutativity of the lower left triangle.
\end{proof}

\begin{prop}
\label{prop.stratn.gives.stable.stratn}
Fix a presentable stable $\infty$-category $\cX$. For any stratification
\[
\pos \longra \Cls_\cX
\]
of $\cX$ over $\pos$, its postcomposition
\[
\pos
\longra
\Cls_\cX
\xlongra{\sim}
\clssub_\cX
\]
with the equivalence of \Cref{obs.cls.and.Cls} is a stable stratification.
\end{prop}

\begin{proof}
Choose any closed subcategories $\cY,\cZ \in \Cls_\cX$ such that $\cZ$ is aligned with $\cY$. By \Cref{lem.excision}\Cref{part.excision.coloc}, the colocalization $i_L y$ into the closed subcategory $\brax{\cY,\cZ} \in \Cls_\cX$ is the cofiber of a morphism from an object of $\cZ \subseteq \cX$ to an object of $\cY \subseteq \cX$. This implies that the inclusion
\[
\brax{\cY,\cZ}^\thick
\subseteq
\brax{\cY,\cZ}
\]
in $\thicksub_\cX$ is an equality. Hence, the claim follows from \Cref{closed.subcats.are.mutually.aligned} \and \Cref{prop.characterize.stable.stratns} (and the fact that $\pos$ is finite).
\end{proof}

\begin{observation}
\label{obs.presentable.gluing.diagram.is.stable.gluing.diagram}
Considering a stratification of a presentable stable $\infty$-category as a stable stratification via \Cref{prop.stratn.gives.stable.stratn} does not change its gluing diagram: Definitions \ref{defn.gluing.diagram} \and \ref{defn.stable.gluing.diagram.and.stable.reflected.gluing.diagram}\ref{defn.part.stable.gluing.diagram} are compatible.
\end{observation}

\begin{notation}
We use a hat in order to emphasize that we are referring to a huge $\infty$-category whose objects are possibly large.
\end{notation}

\begin{observation}
\label{obs.Strat.into.what.strat}
By \Cref{prop.stratn.gives.stable.stratn}, we have inclusions
\[ \begin{tikzcd}
\Strat_\pos
\arrow[hook, dashed]{r}
&
\what{\strat}_\pos
\\
\Strat_\pos^\strict
\arrow[hook, dashed]{r}
\arrow[hook]{u}
&
\what{\strat}_\pos^\strict
\arrow[hook]{u}
\end{tikzcd}
~.
\]
\end{observation}

\begin{cor}
\label{cor.reflection.for.presentable.stratns}
Assume that the poset $\pos$ is down-finite. Then, there is a canonical commutative diagram
\[
\begin{tikzcd}[column sep=1.5cm, row sep=1.5cm]
&
{\displaystyle \prod_{p \in \pos} \PrSt}
\\
\LMod_{\rlax.\pos}^L(\PrSt)
\arrow[yshift=0.9ex]{r}{\limllaxfam}
\arrow[leftarrow, yshift=-0.9ex]{r}{\sim}[swap]{\widecheck{\GD}}
\arrow{ru}[sloped]{(\ev_p)_{p \in \pos}}
\arrow{rd}[sloped, swap]{\lim^\llax_{\rlax.\pos}}
&
\Strat_\pos^\strict
\arrow[yshift=0.9ex]{r}[sloped]{\GD}
\arrow[leftarrow, yshift=-0.9ex]{r}{\sim}[sloped, swap]{\limrlaxfam}
\arrow{d}{\fgt}
\arrow{u}[swap]{((-)_p)_{p \in \pos}}
&
\LMod_{\llax.\pos}^L(\PrSt)
\arrow{lu}[sloped]{(\ev_p)_{p \in \pos}}
\arrow{ld}[sloped, swap]{\lim^\rlax_{\llax.\pos}}
\\
&
\PrSt
\end{tikzcd}
~.\footnote{By $\widecheck{\GD}$ here we refer to the evident analog of \Cref{defn.stable.gluing.diagram.and.stable.reflected.gluing.diagram}\Cref{defn.part.stable.reflected.gluing.diagram}.}
\]
\end{cor}

\begin{proof}
We first address the case that $\pos$ is finite. By interpreting \Cref{thm.reflection.for.stable.stratns} in a larger universe and appealing to \Cref{obs.presentable.gluing.diagram.is.stable.gluing.diagram}, it suffices to verify the image factorizations
\begin{equation}
\label{factor.large.GD.and.GD.check.through.stable.ones}
\begin{tikzcd}
\LMod_{\rlax.\pos}(\what{\St^\idem})
&
\what{\strat}_\pos^\strict
\arrow{l}[swap]{\widecheck{\GD}}{\sim}
\arrow{r}{\GD}[swap]{\sim}
&
\LMod_{\llax.\pos}(\what{\St^\idem})
\\
\LMod^L_{\rlax.\pos}(\PrSt)
\arrow[hook]{u}
\arrow[dashed, leftarrow]{r}[swap]{\sim}
&
\Strat_\pos^\strict
\arrow[dashed]{r}[swap]{\sim}
\arrow[hook]{u}
&
\LMod^L_{\llax.\pos}(\PrSt)
\arrow[hook]{u}
\end{tikzcd}
\end{equation}
of the indicated composites, where the middle vertical inclusion is that of \Cref{obs.Strat.into.what.strat}. The lower right equivalence in diagram \Cref{factor.large.GD.and.GD.check.through.stable.ones} follows from \Cref{thm.strict.metacosm}\Cref{strict.presentable.metacosm}. To conclude, we observe the outer commutative rectangle in diagram \Cref{factor.large.GD.and.GD.check.through.stable.ones}: the upper composite equivalence is the identity on fibers, and the conditions of accessibility of monodromy functors coincide by \Cref{prop.Gamma.check.as.tfib.of.Gammas.and.Gamma.as.tcofib.of.Gamma.checks}.

We now turn to the case that $\pos$ is merely down-finite. Writing $\Down_\pos^\fin \subseteq \Down_\pos$ for the full subposet on the finite down-closed subposets of $\pos$, observe that $\pos \simeq \colim_{\sD \in \Down_\pos^\fin}(\sD)$. Now, for an arbitrary finite poset $\posQ$, we have just argued that we have a commutative diagram
\[
\begin{tikzcd}[column sep=1.5cm, row sep=1.5cm]
&
{\displaystyle \prod_{q \in \posQ} \PrSt}
\\
\LMod_{\rlax.\posQ}^L(\PrSt)
\arrow[yshift=0.9ex]{r}{\limllaxfam}
\arrow[leftarrow, yshift=-0.9ex]{r}{\sim}[swap]{\widecheck{\GD}}
\arrow{ru}[sloped]{(\ev_q)_{q \in \posQ}}
&
\Strat_\posQ^\strict
\arrow[yshift=0.9ex]{r}[sloped]{\GD}
\arrow[leftarrow, yshift=-0.9ex]{r}{\sim}[sloped, swap]{\limrlaxfam}
\arrow{u}[swap]{((-)_q)_{q \in \posQ}}
&
\LMod_{\llax.\posQ}^L(\PrSt)
\arrow{lu}[sloped]{(\ev_q)_{q \in \posQ}}
\end{tikzcd}
~.
\]
Moreover, this diagram is clearly contravariantly functorial as we vary $\posQ$ over the category of finite posets and inclusions of down-closed subposets. From here, it is not hard to see that we obtain the desired diagram for $\pos$ by passing to cofiltered limits over $(\Down_\pos^\fin)^\op$.
\end{proof}

\begin{cor}
\label{cor.reflection.for.finite.intervals}
Let $\pos$ be an arbitrary poset whose intervals are finite. Then, there is a canonical commutative diagram
\[
\begin{tikzcd}[column sep=1.5cm, row sep=1.5cm]
&
{\displaystyle \prod_{p \in \pos} \PrSt}
\\
\LMod_{\rlax.\pos}^L(\PrSt)
\arrow{ru}[sloped]{(\ev_p)_{p \in \pos}}
\arrow[leftrightarrow]{rr}[swap]{\widecheck{(-)}}{\sim}
&
&
\LMod_{\llax.\pos}^L(\PrSt)
\arrow{lu}[sloped]{(\ev_p)_{p \in \pos}}
\end{tikzcd}
~.
\]
Moreover, on monodromy functors, the equivalence $\widecheck{(-)}$ acts as described in \Cref{prop.Gamma.check.as.tfib.of.Gammas.and.Gamma.as.tcofib.of.Gamma.checks}.
\end{cor}

\begin{proof}
This diagram is clearly contravariantly functorial as we vary $\pos$ over the category of finite posets and inclusions of finite convex subposets. So, the claim follows from \Cref{cor.reflection.for.presentable.stratns} by passing to cofiltered limits over the poset $(\Conv_\pos^\fin)^\op$, the opposite of the poset of finite convex subposets of $\pos$.
\end{proof}

\begin{remark}
\label{rmk.finite.intervals.but.not.down.finite}
In the situation of \Cref{cor.reflection.for.finite.intervals}, if the poset $\pos$ is not down-finite then the equivalence does not necessarily commute with the lax limit functors: rather, we have a commutative diagram
\[ \begin{tikzcd}[row sep=1.5cm]
\LMod_{\rlax.\pos}^L(\PrSt)
\arrow{rd}[sloped, swap]{\lim^\llax_{\rlax.\pos}}
\arrow[leftrightarrow]{rr}{\widecheck{(-)}}[swap]{\sim}[yshift=-1cm]{\Rightarrow}
&
&
\LMod_{\llax.\pos}^L(\PrSt)
\arrow{ld}[sloped, swap]{\lim^\rlax_{\llax.\pos}}
\\
&
\PrSt
\end{tikzcd}
~,
\]
in which the components of the natural transformation are left adjoints. For example, let us take $\pos = \ZZ$ and fix a presentable stable $\infty$-category $\cV \in \PrSt$. Then, taking the constant diagram
\[
\ul{\cV}
\in
\LMod^L_{\rlax.\ZZ}(\PrSt)
\]
we obtain an adjunction
\[ \begin{tikzcd}[column sep=1.5cm]
\lim^\llax_{\rlax.\ZZ}(\ul{\cV})
\simeq
\Fun(\ZZ,\cV)
\arrow[transform canvas={yshift=0.9ex}]{r}
\arrow[hookleftarrow, transform canvas={yshift=-0.9ex}]{r}[yshift=-0.2ex]{\bot}
&
\Ch(\cV)
\simeq
\lim^\rlax_{\llax.\ZZ}(\widecheck{\ul{\cV}})
\end{tikzcd}
~,
\]
in which the right adjoint is fully faithful with image the subcategory of \textit{complete} filtered objects, i.e.\! those whose limit is zero (compare with \Cref{ex.reflection.and.dold.kan}).
\end{remark}

\appendix

\section{Actions and limits, strict and lax}
\label{section.lax.actions.and.limits}

In this section, we provide definitions of strict, left-lax, and right-lax modules over $\infty$-categories: in effect, functors of the corresponding sort into the $(\infty,2)$-category $\Cat$.\footnote{The terminology ``module'' is inspired by ordinary group actions: for instance, a left (resp.\! right) $G$-module in an $\infty$-category $\cC$ is the data of a functor $\BG \ra \cC$ (resp.\! $\BG^\op \ra \cC$).} We also provide definitions of strict, left-lax, and right-lax functors among them (and in particular, limits thereof); perhaps surprisingly, these various notions are actually well-defined in all nine cases.  Moreover, we record a number of fundamental results regarding these notions.

\begin{local}
Throughout this section, we fix a base $\infty$-category $\cB$.
\end{local}

This section is organized as follows.
\begin{itemize}

\item[\Cref{subsection.lax.actions}:] We introduce all of the notions of $\cB$-modules and most of the notions of equivariant functors.

\item[\Cref{subsection.lax.limits}:] We introduce the more straightforward sorts of limits.

\item[\Cref{subsection.lax.subsection.with.mixed.handedness}:] We introduce the remaining sorts of equivariant functors (and in particular the remaining sorts of limits), using the theory of $(\infty,2)$-categories developed in \Cref{section.inftytwocats}: namely, those in which the handedness of the laxness of the $\cB$-modules disagrees with that of the equivariant functors among them.

\item[\Cref{subsection.sd.of.posets}:] We study the subdivision $\sd(\cB) \in \Cat$.

\item[\Cref{subsection.lax.limits.via.subdivisions.in.hard.case}:] We give an alternative and more explicit description of the right-lax limit of a left-lax left $\cB$-module using $\sd(\cB)$.

\item[\Cref{subsection.strictification}:] In the case that the only retracts in $\cB$ are equivalences, we provide a useful alternative description of the right-lax limit of a left-lax left $\cB$-module as the strict limit of a strict left $\sd(\cB)$-module.

\end{itemize}

\begin{remark}
In \S\S\ref{subsection.lax.actions} \and \ref{subsection.lax.limits} we give a comprehensive account of the theory, explaining all possible handednesses and how they relate.  However, thereafter we specialize in order to streamline our discussion.
\end{remark}

\begin{remark}
We omit essentially all mention of lax \textit{colimits}, as we will have no explicit need for them.  On the other hand, they will certainly be present: for example, the left-lax colimit of a functor $\cB \ra \Cat$ is nothing other than the total $\infty$-category $\cE$ of the cocartesian fibration $\cE \da \cB$ that it classifies. (See, e.g.\! \cite{GHN} for a discussion of lax colimits of strict functors along these lines.)
\end{remark}

\begin{remark}
The lax $\cB$-modules and lax equivariant functors that we study are all strictly unital (in the sense that the corresponding functors to $\Cat$ strictly respect identity morphisms).\footnote{Of course, more general definitions exist (see \Cref{subsection.basics.of.inftytwocats}).}  This stands in contrast with the laxly $\cO$-monoidal functors between $\cO$-monoidal $\infty$-categories that arise in \Cref{section.O.mon.reconstrn.thm}: as described in \Cref{rmk.unpack.laxly.O.monoidal.functors}, we do not require those to be strictly unital (in the sense that we do not require them to strictly respect the unit objects of $\cO$-monoidal structures).
\end{remark}


\subsection{Strict and lax actions}
\label{subsection.lax.actions}

In this subsection, we introduce all of the notions of $\cB$-modules and most of the notions of equivariant functors.  We begin with an omnibus definition, which the remainder of the subsection is dedicated to discussing.

\begin{definition}
\label{define.almost.all.modules}
In \Cref{figure.define.almost.all.modules}, various $\infty$-categories of \bit{$\cB$-modules} depicted on the left side are defined as indicated on the right side.  The objects in the $\infty$-categories in the upper left diagram are (various sorts of) \bit{left} $\cB$-modules, while the objects in the $\infty$-categories in the lower left diagram are (various sorts of) \bit{right} $\cB$-modules.  In both diagrams on the left side, we refer
\begin{itemize}
\item to the objects
\begin{itemize}
\item in the middle rows as (\bit{strict}) $\cB$-modules,
\item in the top rows as \bit{left-lax} $\cB$-modules, and
\item in the bottom rows as \bit{right-lax} $\cB$-modules,
\end{itemize}
and
\item to the morphisms
\begin{itemize}
\item in the middle columns as (\bit{strictly}) \bit{equivariant},
\item in the left columns as \bit{left-lax equivariant}, and
\item in the right columns as \bit{right-lax equivariant}.
\end{itemize}
\end{itemize}
So in our notation, laxness of the actions is indicated by a subscript (placed before ``.$\cB$''), while laxness of the morphisms is indicated by a superscript.
\begin{sidewaysfigure}
\vspace{450pt}
\[ \begin{tikzcd}[row sep=1.5cm]
\LMod^\llax_{\llax.\cB}
\arrow[\surjmonoleft]{r}
&
\LMod_{\llax.\cB}
\\
\LMod^\llax_\cB
\arrow[\surjmonoleft]{r}
\arrow[hook]{u}{\ff}
&
\LMod_\cB
\arrow[hook, two heads]{r}
\arrow[hook]{u}[swap]{\ff}
\arrow[hook]{d}[swap]{\ff}
&
\LMod^\rlax_\cB
\arrow[hook]{d}{\ff}
\\
&
\LMod_{\rlax.\cB}
\arrow[hook, two heads]{r}
&
\LMod^\rlax_{\rlax.\cB}
\end{tikzcd}
\qquad
:=
\qquad
\begin{tikzcd}[row sep=1.5cm]
\Cat_{\loc.\cocart/\cB}
\arrow[\surjmonoleft]{r}
&
\loc.\coCart_\cB
\\
\Cat_{\cocart/\cB}
\arrow[\surjmonoleft]{r}
\arrow[hook]{u}{\ff}
&
\coCart_\cB
\arrow[hook]{u}[swap]{\ff}
\\[-1.7cm]
&
&[-1.3cm]
\rotatebox{-30}{$\simeq$}
&[-1.2cm]
\\[-1.7cm]
&
&
&
\Cart_{\cB^\op}
\arrow[hook, two heads]{r}
\arrow[hook]{d}[swap]{\ff}
&
\Cat_{\cart/\cB^\op}
\arrow[hook]{d}{\ff}
\\
&
&
&
\loc.\Cart_{\cB^\op}
\arrow[hook, two heads]{r}
&
\Cat_{\loc.\cart/\cB^\op}
\end{tikzcd} \]

\vspace{50pt}

\[ \begin{tikzcd}[row sep=1.5cm]
\RMod^\llax_{\llax.\cB}
\arrow[\surjmonoleft]{r}
&
\RMod_{\llax.\cB}
\\
\RMod^\llax_\cB
\arrow[\surjmonoleft]{r}
\arrow[hook]{u}{\ff}
&
\RMod_\cB
\arrow[hook, two heads]{r}
\arrow[hook]{u}[swap]{\ff}
\arrow[hook]{d}[swap]{\ff}
&
\RMod^\rlax_\cB
\arrow[hook]{d}{\ff}
\\
&
\RMod_{\rlax.\cB}
\arrow[hook, two heads]{r}
&
\RMod^\rlax_{\rlax.\cB}
\end{tikzcd}
\qquad
:=
\qquad
\begin{tikzcd}[row sep=1.5cm]
\Cat_{\loc.\cocart/\cB^\op}
\arrow[\surjmonoleft]{r}
&
\loc.\coCart_{\cB^\op}
\\
\Cat_{\cocart/\cB^\op}
\arrow[\surjmonoleft]{r}
\arrow[hook]{u}{\ff}
&
\coCart_{\cB^\op}
\arrow[hook]{u}[swap]{\ff}
\\[-1.5cm]
&
&[-1.3cm]
\rotatebox{-30}{$\simeq$}
&[-1.2cm]
\\[-1.5cm]
&
&
&
\Cart_\cB
\arrow[hook, two heads]{r}
\arrow[hook]{d}[swap]{\ff}
&
\Cat_{\cart/\cB}
\arrow[hook]{d}{\ff}
\\
&
&
&
\loc.\Cart_\cB
\arrow[hook, two heads]{r}
&
\Cat_{\loc.\cart/\cB}
\end{tikzcd} \]

\vspace{50pt}

\caption{The commutative diagrams of monomorphisms among $\infty$-categories on the left are defined to be those on the right.}
\label{figure.define.almost.all.modules}
\end{sidewaysfigure}
\end{definition}

\begin{remark}
We give definitions in \Cref{subsection.lax.subsection.with.mixed.handedness} that extend the diagrams of \Cref{figure.define.almost.all.modules} to full $3 \times 3$ grids.
\end{remark}

\begin{example}
\label{lax.equivariance.of.strict.modules.over.walking.arrow}
Let us unwind the definitions of the $\infty$-categories
\[
\LMod_\cB
~,
\qquad
\LMod_\cB^\llax
~,
\qquad
\RMod_\cB
~,
\qquad
\text{and}
\qquad
\RMod_\cB^\rlax
\]
in the simplest nontrivial case, namely when $\cB = [1]$.
\begin{enumerate}
\item\label{left.lax.equivariance.of.strict.modules}
Let $\cE \da [1]$ and $\cF \da [1]$ be cocartesian fibrations, the unstraightenings of functors
\[
[1]
\xra{\brax{\cE_{0} \xra{E} \cE_{1}}}
\Cat
\]
and
\[
[1]
\xra{\brax{\cF_{0} \xra{F} \cF_{1}}}
\Cat
~,
\]
respectively.  Then, let us consider a left-lax equivariant functor
\[ \begin{tikzcd}
\cE
\arrow{rr}{\alpha}
\arrow{rd}
&
&
\cF
\arrow{ld}
\\
&
{[1]}
\end{tikzcd}~. \]
Given a cocartesian morphism $e \ra E(e)$ in $\cE$ with $e \in \cE_{0}$ and $E(e) \in \cE_{1}$, the functor $\alpha$ takes it to some not-necessarily-cocartesian morphism $\alpha(e) \ra \alpha(E(e))$ in $\cF$ with $\alpha(e) \in \cF_{0}$ and $\alpha(E(e)) \in \cF_{1}$.  This admits a unique factorization
\[ \begin{tikzcd}
\alpha(e)
\arrow[dashed]{r}
\arrow{rd}
&
F(\alpha(e))
\arrow{d}
\\
&
\alpha(E(e))
\end{tikzcd} \]
as a cocartesian morphism followed by a fiber morphism.  This operation is functorial in $e \in \cE_{0}$, which implies that our left-lax equivariant functor amounts to the data of a lax-commutative square
\[ \begin{tikzcd}
\cE_{0}
\arrow{r}{E}[swap, transform canvas={yshift=-0.4cm}]{\rotatebox{45}{$\Rightarrow$}}
\arrow{d}[swap]{\alpha_{0}}
&
\cE_{1}
\arrow{d}{\alpha_{1}}
\\
\cF_{0}
\arrow{r}[swap]{F}
&
\cF_{1}
\end{tikzcd}~. \]
To say that the left-lax equivariant functor is actually strictly equivariant is equivalently to say that this square actually commutes, i.e.\! that the natural transformation is a natural equivalence.

\item\label{right.lax.equivariance.of.strict.modules}
Dually, let $\cE \da [1]$ and $\cF \da [1]$ be cartesian fibrations, the unstraightenings of functors
\[
[1]^\op
\xra{\brax{\cE_{0^\circ} \xla{E} \cE_{1^\circ}}}
\Cat
\]
and
\[
[1]^\op
\xra{\brax{\cF_{0^\circ} \xla{F} \cF_{1^\circ}}}
\Cat
~,
\]
respectively.  Then, a right-lax equivariant functor
\[ \begin{tikzcd}
\cE
\arrow{rr}{\alpha}
\arrow{rd}
&
&
\cF
\arrow{ld}
\\
&
{[1]}
\end{tikzcd} \]
likewise amounts to the data of a lax-commutative square
\[ \begin{tikzcd}
\cE_{0^\circ}
\arrow[leftarrow]{r}{E}[swap, transform canvas={yshift=-0.4cm}]{\rotatebox{-45}{$\Rightarrow$}}
\arrow{d}[swap]{\alpha_{0^\circ}}
&
\cE_{1^\circ}
\arrow{d}{\alpha_{1^\circ}}
\\
\cF_{0^\circ}
\arrow[leftarrow]{r}[swap]{F}
&
\cF_{1^\circ}
\end{tikzcd}~. \]
To say that the right-lax equivariant functor is actually strictly equivariant is equivalently to say that this square actually commutes, i.e.\! that the natural transformation is a natural equivalence.
\end{enumerate}
\end{example}

\begin{example}
\label{ex.lax.mors.betw.strict.G.modules}
Let us unwind the definitions of the $\infty$-categories
\[
\LMod^\llax_\cB
~,
\qquad
\LMod^\rlax_\cB
~,
\qquad
\RMod^\rlax_\cB
~,
\qquad
\text{and}
\qquad
\RMod^\llax_\cB
\]
in the simple but illustrative case that $\cB = \BG$ for a group or monoid $G$.  Choose any two objects
\[
\cE
,
\cF
\in
\Cat_{({\sf co})\cart/\BG^{(\op)}}
~,
\]
with the two choices of whether or not to include the parenthesized bits made independently.  These are classified by left or right $G$-actions on the fibers $\cE_0$ and $\cF_0$ over the basepoint of $\BG^{(\op)}$ -- right if the choices coincide, left if they do not -- and morphisms between them are left-lax equivariant in the case of ``$\cocart$'' and right-lax equivariant in the case of ``$\cart$''.  In all four cases, a morphism
\[ \begin{tikzcd}
\cE
\arrow{rr}{\alpha}
\arrow{rd}
&
&
\cF
\arrow{ld}
\\
&
\BG^{(\op)}
\end{tikzcd} \]
is the data of a functor
\[ \cE_0 \xra{\alpha_0} \cF_0 \]
on underlying $\infty$-categories equipped with certain natural transformations indexed over all $g \in G$, as recorded in \Cref{table.laxness.G.action}.
\begin{figure}[h]
\begin{tabular}{ | c | c | }
\hline
$\LMod^\llax_\BG$
&
$g \cdot \alpha_0(-) \longra \alpha_0(g \cdot -)$
\\
\hline
$\RMod^\rlax_\BG$
&
$\alpha_0(- \cdot g) \longra \alpha_0(-) \cdot g$
\\
\hline
$\LMod^\rlax_\BG$
&
$\alpha_0(g \cdot -) \longra g \cdot \alpha_0( - )$
\\
\hline
$\RMod^\llax_\BG$
&
$\alpha_0(- \cdot g) \longra \alpha_0(-) \cdot g$
\\
\hline
\end{tabular}
\caption{Given two $\infty$-categories equipped with (strict) left or right $G$-actions, defining a left- or right-lax equivariant functor between them amounts to defining a functor on underlying $\infty$-categories along with compatible lax structure maps indexed by $g \in G$, as indicated.
\label{table.laxness.G.action}}
\end{figure}
Moreover, these must be equipped with compatibility data with respect to the multiplication in $G$: for example, in the case of $\LMod^\rlax_\BG$, for all $g,h \in G$ the diagram
\[ \begin{tikzcd}
\alpha_0(ghe)
\arrow{rr}
\arrow{rd}
&
&
gh\alpha_0(e)
\\
&
g\alpha_0(he)
\arrow{ru}
\end{tikzcd} \]
must commute, naturally in $e \in \cE_0$.
\end{example}

\begin{example}
\label{example.lax.actions}
Let us unwind the definitions of the $\infty$-categories
\[
\LMod^\llax_{\llax.\cB}
\qquad
\text{and}
\qquad
\RMod^\rlax_{\rlax.\cB}
\]
in the simplest nontrivial case, namely when $\cB = [2]$.
\begin{enumerate}
\item
\label{describe.and.map.locally.cocart.fibns}
\begin{enumerate}[label=(\alph*)]
\item\label{describe.locally.cocart.fibn} Let $\cE \da [2]$ be a locally cocartesian fibration; let us write $\cE_{i}$ for its fibers (for $i \in [2]$) and $E_{ij}$ for its cocartesian monodromy functors (for $0 \leq i < j \leq 2$).  An object $e \in \cE_{0}$ determines a pair of composable locally cocartesian morphisms $e \ra E_{01}(e) \ra E_{12}(E_{01}(e))$ with $E_{01}(e) \in \cE_{1}$ and $E_{12}(E_{01}(e)) \in \cE_{2}$.  Their composite is a not-necessarily-locally-cocartesian morphism, which admits a unique factorization
\[ \begin{tikzcd}
e
\arrow[dashed]{r}
\arrow{rd}
&
E_{02}(e)
\arrow{d}
\\
&
E_{12}(E_{01}(e))
\end{tikzcd} \]
as a locally cocartesian morphism followed by a fiber morphism.  This operation is functorial in $e \in \cE_{0}$, which implies that our left-lax left $[2]$-module amounts to the data of a lax-commutative triangle
\[ \begin{tikzcd}
&
\cE_{1}
\arrow{rd}[sloped]{E_{12}}
\\
\cE_{0}
\arrow{ru}[sloped]{E_{01}}
\arrow{rr}[transform canvas={yshift=0.3cm}]{\rotatebox{90}{$\Rightarrow$}}[swap]{E_{02}}
&
&
\cE_{2}
\end{tikzcd}~. \]
This should be thought as the unstraightening of a \textit{left-lax} functor
\[ \begin{tikzcd}[column sep=1.5cm]
{[2]}
\arrow{r}[description]{\llax}
&
\Cat
\end{tikzcd} \]
of $(\infty,2)$-categories.
\item Let $\cE \da [2]$ and $\cF \da [2]$ be locally cocartesian fibrations, and let us continue to use notation as in part \Cref{describe.locally.cocart.fibn} for both $\cE$ and $\cF$.  Then, a left-lax equivariant functor
\[ \begin{tikzcd}
\cE
\arrow{rr}{\alpha}
\arrow{rd}
&
&
\cF
\arrow{ld}
\\
&
{[2]}
\end{tikzcd} \]
amounts to the data of left-lax equivariant functors over the three nonidentity morphisms in $[2]$ (as described in \Cref{lax.equivariance.of.strict.modules.over.walking.arrow}\Cref{left.lax.equivariance.of.strict.modules}), along with an equivalence between the composite 2-morphisms
\[ \begin{tikzcd}[column sep=1.5cm]
&
\cE_{1}
\arrow{rd}[sloped]{E_{12}}[swap, transform canvas={xshift=0.25cm, yshift=-1.1cm}]{\rotatebox{30}{$\Rightarrow$}}
\arrow{dd}{\alpha_{1}}
\\
\cE_{0}
\arrow{ru}[sloped]{E_{01}}[swap, transform canvas={xshift=-0.2cm, yshift=-1cm}]{\rotatebox{55}{$\Rightarrow$}}
\arrow{dd}[swap]{\alpha_{0}}
&
&
\cE_{2}
\arrow{dd}{\alpha_{2}}
\\
&
\cF_{1}
\arrow{rd}[sloped]{F_{12}}
\\
\cF_{0}
\arrow{ru}[sloped]{F_{01}}
\arrow{rr}[transform canvas={yshift=0.3cm}]{\rotatebox{90}{$\Rightarrow$}}[swap]{F_{02}}
&
&
\cF_{2}
\end{tikzcd} \]
and
\[ \begin{tikzcd}[column sep=1.5cm]
&
\cE_{1}
\arrow{rd}[sloped]{E_{12}}
\\
\cE_{0}
\arrow{ru}[sloped]{E_{01}}
\arrow{rr}[transform canvas={yshift=0.3cm}]{\rotatebox{90}{$\Rightarrow$}}[swap]{E_{02}}[swap, transform canvas={yshift=-1.1cm}]{\rotatebox{30}{$\Rightarrow$}}
\arrow{d}[swap]{\alpha_{0}}
&
&
\cE_{2}
\arrow{d}{\alpha_{2}}
\\[1.25cm]
\cF_{0}
\arrow{rr}[swap]{F_{02}}
&
&
\cF_{2}
\end{tikzcd} \]
(i.e.\! a 3-morphism filling in the triangular prism).
\end{enumerate}
\item
\label{describe.and.map.locally.cart.fibns}
\begin{enumerate}[label=(\alph*)]
\item\label{describe.locally.cart.fibn}
Dually, let $\cE \da [2]$ be a locally cartesian fibration; let us write $\cE_{i^\circ}$ for its fibers (for $i \in [2]$) and $E_{j^\circ i^\circ}$ for its cartesian monodromy functors (for $0 \leq i < j \leq 2$).  Then, this right-lax right $[2]$-module amounts to the data of a lax-commutative triangle
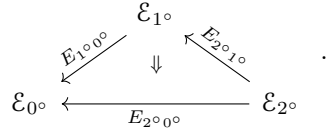
\[ \begin{tikzcd}
&
\cE_{1^\circ}
\arrow[leftarrow]{rd}[sloped]{E_{2^\circ1^\circ}}
\\
\cE_{0^\circ}
\arrow[leftarrow]{ru}[sloped]{E_{1^\circ0^\circ}}
\arrow[leftarrow]{rr}[transform canvas={yshift=0.3cm}]{\rotatebox{-90}{$\Rightarrow$}}[swap]{E_{2^\circ0^\circ}}
&
&
\cE_{2^\circ}
\end{tikzcd}~. \]
This should be thought as the unstraightening of a \textit{right-lax} functor
\[ \begin{tikzcd}[column sep=1.5cm]
{[2]^\op}
\arrow{r}[description]{\rlax}
&
\Cat
\end{tikzcd} \]
of $(\infty,2)$-categories.
\item Let $\cE \da [2]$ and $\cF \da [2]$ be locally cartesian fibrations, and let us continue to use notation as in part \Cref{describe.locally.cart.fibn} for both $\cE$ and $\cF$.  Then, a right-lax equivariant functor
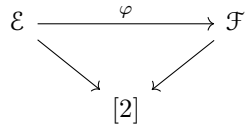
\[ \begin{tikzcd}
\cE
\arrow{rr}{\alpha}
\arrow{rd}
&
&
\cF
\arrow{ld}
\\
&
{[2]}
\end{tikzcd} \]
amounts to the data of right-lax equivariant functors over the three nonidentity morphisms in $[2]$ (as described in \Cref{lax.equivariance.of.strict.modules.over.walking.arrow}\Cref{right.lax.equivariance.of.strict.modules}), along with an equivalence between the composite 2-morphisms
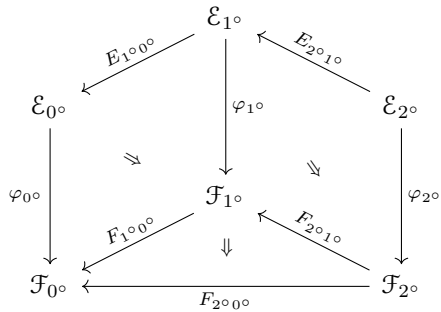
\[ \begin{tikzcd}[column sep=1.5cm]
&
\cE_{1^\circ}
\arrow[leftarrow]{rd}[sloped]{E_{2^\circ1^\circ}}[swap, transform canvas={xshift=0.2cm, yshift=-1.2cm}]{\rotatebox{-55}{$\Rightarrow$}}
\arrow{dd}{\alpha_{1^\circ}}
\\
\cE_{0^\circ}
\arrow[leftarrow]{ru}[sloped]{E_{1^\circ0^\circ}}[swap, transform canvas={xshift=-0.3cm, yshift=-1.1cm}]{\rotatebox{-30}{$\Rightarrow$}}
\arrow{dd}[swap]{\alpha_{0^\circ}}
&
&
\cE_{2^\circ}
\arrow{dd}{\alpha_{2^\circ}}
\\
&
\cF_{1^\circ}
\arrow[leftarrow]{rd}[sloped]{F_{2^\circ1^\circ}}
\\
\cF_{0^\circ}
\arrow[leftarrow]{ru}[sloped]{F_{1^\circ0^\circ}}
\arrow[leftarrow]{rr}[transform canvas={yshift=0.3cm}]{\rotatebox{-90}{$\Rightarrow$}}[swap]{F_{2^\circ0^\circ}}
&
&
\cF_{2^\circ}
\end{tikzcd} \]
and
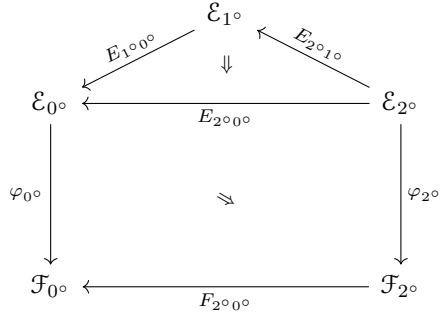
\[ \begin{tikzcd}[column sep=1.5cm]
&
\cE_{1^\circ}
\arrow[leftarrow]{rd}[sloped]{E_{2^\circ1^\circ}}
\\
\cE_{0^\circ}
\arrow[leftarrow]{ru}[sloped]{E_{1^\circ0^\circ}}
\arrow[leftarrow]{rr}[transform canvas={yshift=0.3cm}]{\rotatebox{-90}{$\Rightarrow$}}[swap]{E_{2^\circ0^\circ}}[swap, transform canvas={yshift=-1.1cm}]{\rotatebox{-30}{$\Rightarrow$}}
\arrow{d}[swap]{\alpha_{0^\circ}}
&
&
\cE_{2^\circ}
\arrow{d}{\alpha_{2^\circ}}
\\[1.25cm]
\cF_{0^\circ}
\arrow[leftarrow]{rr}[swap]{F_{2^\circ0^\circ}}
&
&
\cF_{2^\circ}
\end{tikzcd} \]
(i.e.\! a 3-morphism filling in the triangular prism).
\end{enumerate}
\end{enumerate}
\end{example}

\begin{example}
Let us unwind the definitions of the $\infty$-categories
\[
\LMod_{\llax.\cB}
~,
\qquad
\RMod_{\rlax.\cB}
~,
\qquad
\LMod_{\rlax.\cB}
~,
\qquad
\text{and}
\qquad
\RMod_{\llax.\cB}
\]
in the simple but illustrative case that $\cB = \BG$ for a group or monoid $G$.  Choose an object
\[
\cE
\in \Cat_{\loc.({\sf co})\cart/\BG^{(\op)}}
~,
\]
with the two choices of whether or not to include the parenthesized bits made independently.  Write $\cE_0$ for the fiber over the basepoint of $\BG^{(\op)}$, the underlying $\infty$-category.  Then, this is the data of an endofunctor $(g \cdot -)$ or $(- \cdot g)$ of $\cE_0$ for each $g \in G$, along with compatible natural transformations, as recorded in \Cref{lax.group.action}.
\begin{figure}[h]
\begin{tabular}{|c|c|}
\hline
$\LMod_{\llax.\BG}$
&
$(gh \cdot -) \longra g \cdot (h \cdot -)$
\\
\hline
$\RMod_{\rlax.\BG}$
&
$(- \cdot g) \cdot h \longra (- \cdot gh)$
\\
\hline
$\LMod_{\rlax.\BG}$
&
$g \cdot (h \cdot -) \longra (gh \cdot -)$
\\
\hline
$\RMod_{\llax.\BG}$
&
$(- \cdot gh) \longra (- \cdot g) \cdot h$
\\
\hline
\end{tabular}
\caption{Equipping an $\infty$-category with a left- or right-lax left or right $G$-action amounts to defining endofunctors indexed by $g \in G$, equipped with lax structure maps corresponding to multiplication in $G$, as indicated.
\label{lax.group.action}}
\end{figure}
Of course, these must also be compatible with iterated multiplication in $G$.
\end{example}

\begin{observation}
\label{obs.loc.coCart.B.as.a.limit}
Consider a colimit
\begin{equation}
\label{colim.in.Cat.for.considering.limit.presentation.of.categories.of.fibrations}
\cB \simeq \colim_{i \in \cI}(\cB_i)
\end{equation}
in $\Cat$. By un/straightening, it is clear that the canonical functor
\[
\coCart_\cB
\longra
\lim_{i^\circ \in \cI^\op} ( \coCart_{\cB_i} )
\]
is an equivalence (of $(\infty,2)$-categories). On the other hand, the canonical functor
\begin{equation}
\label{fctr.to.limit.of.categories.of.loc.cocart.fibns}
\loc.\coCart_\cB
\longra
\lim_{i^\circ \in \cI^\op} ( \loc.\coCart_{\cB_i})
\end{equation}
is not generally an equivalence. However, the functor \Cref{fctr.to.limit.of.categories.of.loc.cocart.fibns} is an equivalence under the condition that the colimit \Cref{colim.in.Cat.for.considering.limit.presentation.of.categories.of.fibrations}, considered in complete Segal spaces, is in fact a colimit in simplicial spaces. This follows from \Cref{t11} using \Cref{obs.laxification.of.one.cats}.\footnote{It is also easy to see directly without appealing to un/straightening using the fact that simplicial spaces is an $\infty$-topos.}
\end{observation}

\subsection{Strict and lax limits}
\label{subsection.lax.limits}

In this subsection, we introduce the more straightforward sorts of limits.  We begin with an omnibus definition, which the remainder of the subsection is dedicated to discussing.

\begin{notation}
\label{notn.Fun.super.cocart.or.cart.over.B}
Given two objects
\[
(\cE \da \cB)
,
(\cF \da \cB)
\in
\Cat_{/\cB}~,
\]
we write
\[
\Fun^{({\sf co})\cart}_{/\cB}(\cE,\cF)
\subseteq
\Fun_{/\cB} ( \cE, \cF)
\]
for the full subcategory on those functors which take all locally (co)cartesian morphisms over $\cB$ in $\cE$ to locally (co)cartesian morphisms over $\cB$ in $\cF$.
As a special case, we write
\[
\Gamma^{({\sf co})\cart}(\cF)
:=
\Gamma_\cB^{({\sf co})\cart}(\cF)
:=
\Fun^{({\sf co})\cart}_{/\cB}(\cB,\cF)
\]
(using the subscript in the case that there is any potential ambiguity).
\end{notation}

\begin{definition}
\label{define.almost.all.limits}
In \Cref{figure.define.almost.all.limits}, we define various \bit{limit} functors on various $\infty$-categories of $\cB$-modules.
\begin{figure}
\[
\begin{tikzcd}[row sep=1.25cm, column sep=1.5cm]
\LMod^\llax_{\llax.\cB}
\arrow[bend left]{rdd}[description]{\lim^\llax_{\llax.\cB}}
\\
\LMod_{\llax.\cB}
\arrow{rd}[description]{\lim_{\llax.\cB}}[sloped, transform canvas={xshift=-0.2cm, yshift=0.8cm}]{\rotatebox{90}{$\Rightarrow$}}
\arrow[hook, two heads]{u}
\\
\LMod_\cB
\arrow[hook]{u}{\ff}
\arrow[dashed]{r}[description]{\lim_\cB}
&
\Cat
\\
\LMod_{\rlax.\cB}
\arrow[hookleftarrow]{u}{\ff}
\arrow{ru}[description]{\lim_{\rlax.\cB}}[swap, sloped, transform canvas={xshift=-0.2cm, yshift=-0.8cm}]{\rotatebox{-90}{$\Rightarrow$}}
\\
\LMod^\rlax_{\rlax.\cB}
\arrow[\surjmonoleft]{u}
\arrow[bend right]{ruu}[description]{\lim^\rlax_{\rlax.\cB}}
\end{tikzcd}
\qquad
:=
\qquad
\begin{tikzcd}[row sep=1.25cm, column sep=1.5cm]
\Cat_{\loc.\cocart/\cB}
\arrow[bend left]{rddd}[description]{\Gamma}
\\
\loc.\coCart_\cB
\arrow{rdd}[description]{\Gamma^\cocart}[sloped, transform canvas={xshift=-0.2cm, yshift=1cm}]{\rotatebox{90}{$\Rightarrow$}}
\arrow[hook, two heads]{u}
\\
\coCart_\cB
\arrow[hook]{u}{\ff}
\\[-1.25cm]
\rotatebox{90}{$\simeq$}
&
\Cat
\\[-1.25cm]
\Cart_{\cB^\op}
\\
\loc.\Cart_{\cB^\op}
\arrow[hookleftarrow]{u}{\ff}
\arrow{ruu}[description]{\Gamma^\cart}[swap, sloped, transform canvas={xshift=-0.2cm, yshift=-1cm}]{\rotatebox{-90}{$\Rightarrow$}}
\\
\Cat_{\loc.\cart/\cB^\op}
\arrow[\surjmonoleft]{u}
\arrow[bend right]{ruuu}[description]{\Gamma}
\end{tikzcd}
\]

\vspace{20pt}

\[
\begin{tikzcd}[row sep=1.25cm, column sep=1.5cm]
\RMod^\llax_{\llax.\cB}
\arrow[bend left]{rdd}[description]{\lim^\llax_{\llax.\cB^\op}}
\\
\RMod_{\llax.\cB}
\arrow{rd}[description]{\lim_{\llax.\cB^\op}}[sloped, transform canvas={xshift=-0.2cm, yshift=0.8cm}]{\rotatebox{90}{$\Rightarrow$}}
\arrow[hook, two heads]{u}
\\
\RMod_\cB
\arrow[hook]{u}{\ff}
\arrow[dashed]{r}[description]{\lim_{\cB^\op}}
&
\Cat
\\
\RMod_{\rlax.\cB}
\arrow[hookleftarrow]{u}{\ff}
\arrow{ru}[description]{\lim_{\rlax.\cB^\op}}[swap, sloped, transform canvas={xshift=-0.2cm, yshift=-0.8cm}]{\rotatebox{-90}{$\Rightarrow$}}
\\
\RMod^\rlax_{\rlax.\cB}
\arrow[\surjmonoleft]{u}
\arrow[bend right]{ruu}[description]{\lim^\rlax_{\rlax.\cB^\op}}
\end{tikzcd}
\qquad
:=
\qquad
\begin{tikzcd}[row sep=1.25cm, column sep=1.5cm]
\Cat_{\loc.\cocart/\cB^\op}
\arrow[bend left]{rddd}[description]{\Gamma}
\\
\loc.\coCart_{\cB^\op}
\arrow{rdd}[description]{\Gamma^\cocart}[sloped, transform canvas={xshift=-0.2cm, yshift=1cm}]{\rotatebox{90}{$\Rightarrow$}}
\arrow[hook, two heads]{u}
\\
\coCart_{\cB^\op}
\arrow[hook]{u}{\ff}
\\[-1.25cm]
\rotatebox{90}{$\simeq$}
&
\Cat
\\[-1.25cm]
\Cart_\cB
\\
\loc.\Cart_\cB
\arrow[hookleftarrow]{u}{\ff}
\arrow{ruu}[description]{\Gamma^\cart}[swap, sloped, transform canvas={xshift=-0.2cm, yshift=-1cm}]{\rotatebox{-90}{$\Rightarrow$}}
\\
\Cat_{\loc.\cart/\cB}
\arrow[\surjmonoleft]{u}
\arrow[bend right]{ruuu}[description]{\Gamma}
\end{tikzcd}
~.
\]

\vspace{20pt}

\caption{The rightwards functors to $\Cat$ on the left are defined to be those on the right (except that each dashed functor may be defined as either adjacent composite: the inner triangles all commute).}
\label{figure.define.almost.all.limits}
\end{figure}
Our notation is largely concordant with that of \Cref{define.almost.all.modules}; we indicate the handedness of the original module in the subscript by writing $\cB$ for left modules and $\cB^\op$ for right modules.\footnote{This coincides with the corresponding notation for $G$-actions: the limit of a left (resp.\! right) $G$-action is a limit over $\BG$ (resp.\! over $\BG^\op$).}  We refer to a limit functor according to its superscript (which is more relevant anyways), e.g.\! we refer to $\lim^\llax_{\llax.\cB^\op}$ as the \bit{left-lax limit} functor. We also write e.g.\!
\[
\lim^\llax_\cB
:
\LMod_\cB
\longhookra
\LMod^\llax_{\llax.\cB}
\xra{\lim^\llax_{\llax.\cB}}
\Cat
\]
for the composite functor, which carries each strict left $\cB$-module to its left-lax limit.
\end{definition}

\begin{example}
Let us unwind the definitions of the functors in the diagrams
\[
\begin{tikzcd}[column sep=1.5cm]
\LMod_\cB
\arrow[bend left]{r}{\lim_\cB}[swap, transform canvas={yshift=-0.3cm}]{\rotatebox{-90}{$\Rightarrow$}}
\arrow[bend right]{r}[swap]{\lim^\llax_\cB}
&
\Cat
\end{tikzcd}
\qquad
\text{and}
\qquad
\begin{tikzcd}[column sep=1.5cm]
\RMod_\cB
\arrow[bend left]{r}{\lim_{\cB^\op}}[swap, transform canvas={yshift=-0.3cm}]{\rotatebox{-90}{$\Rightarrow$}}
\arrow[bend right]{r}[swap]{\lim^\rlax_{\cB^\op}}
&
\Cat
\end{tikzcd}
\]
in the simple but illustrative case that $\cB = \BG$ for a group or monoid $G$.
\begin{enumerate}
\item Suppose that
\[
(\cE \da \BG)
\in
\coCart_\BG
=:
\LMod_\BG
\]
is classified by a left $G$-action on $\cE_0$.
\begin{enumerate}[label=(\alph*)]
\item An object of the strict limit is given by an object $e \in \cE_0$ equipped with equivalences
\[ g \cdot e \xlongra{\sim} e \]
for all $g \in G$ that are compatible with the multiplication in $G$.
\item An object of the left-lax limit is given by an object $e \in \cE_0$ equipped with morphisms
\[ g \cdot e \longra e \]
for all $g \in G$ that are compatible with the multiplication in $G$.
\end{enumerate}
\item Suppose that
\[
(\cE \da \BG)
\in
\Cart_\BG
=:
\RMod_\BG
\]
is classified by a right $G$-action on $\cE_0$.
\begin{enumerate}[label=(\alph*)]
\item An object of the strict limit is given by an object $e \in \cE_0$ equipped with equivalences
\[ e \xlongra{\sim} e \cdot g \]
for all $g \in G$ that are compatible with the multiplication in $G$.
\item An object of the left-lax limit is given by an object $e \in \cE_0$ equipped with morphisms
\[ e \longra e \cdot g \]
for all $g \in G$ that are compatible with the multiplication in $G$.
\end{enumerate}
\end{enumerate}
\end{example}

\begin{example}
\label{example.lax.limits.with.agreeing.lax.action}
Let us unwind the definitions of the functors
\[
\LMod^\llax_{\llax.\cB}
\xra{\lim^\llax_{\llax.\cB}}
\Cat
\qquad
\text{and}
\qquad
\RMod^\rlax_{\rlax.\cB}
\xra{\lim^\rlax_{\rlax.\cB}}
\Cat
\]
in the simplest nontrivial case, namely when $\cB = [2]$.
\begin{enumerate}

\item\label{item.example.lax.limits.with.agreeing.lax.action.loc.cocart}

Let $\cE \da [2]$ be a locally cocartesian fibration, and let us employ the notation of \Cref{example.lax.actions}\Cref{describe.and.map.locally.cocart.fibns}\Cref{describe.locally.cocart.fibn}.  Then, an object of the left-lax limit of this left-lax left $[2]$-module is given by the data of
\begin{itemize}
\item objects $e_i \in \cE_{i}$ (for $0 \leq i \leq 2$),
\item morphisms
\[
E_{ij}(e_i)
\xra{\vareps_{ij}}
e_j
\]
(for $0 \leq i < j \leq 2$), and
\item a commutative square
\[ \begin{tikzcd}[column sep=1.5cm, row sep=1.5cm]
E_{02}(e_0)
\arrow{r}{\vareps_{02}}
\arrow{d}
&
e_2
\\
E_{12}(E_{01}(e_0))
\arrow{r}[swap]{E_{12}(\vareps_{01})}
&
E_{12}(e_1)
\arrow{u}[swap]{\vareps_{12}}
\end{tikzcd} \]
in $\cE_{2}$, where the morphism on the left is the canonical one (recall \Cref{example.lax.actions}\ref{describe.and.map.locally.cocart.fibns}\ref{describe.locally.cocart.fibn}).
\end{itemize}
Note that the structure map $\vareps_{02}$ is \textit{canonically} determined by the structure maps $\vareps_{01}$ and $\vareps_{12}$.

\item\label{item.example.lax.limits.with.agreeing.lax.action.loc.cart}

Let $\cE \da [2]$ be a locally cartesian fibration, and let us employ the notation of \Cref{example.lax.actions}\Cref{describe.and.map.locally.cart.fibns}\Cref{describe.locally.cart.fibn}.  Then, an object of the right-lax limit of this right-lax right $[2]$-module is given by the data of
\begin{itemize}
\item objects $e_{i^\circ} \in \cE_{i^\circ}$ (for $0 \leq i \leq 2$),
\item morphisms 
\[
e_{i^\circ}
\xra{\vareps_{j^\circ i^\circ}}
E_{j^\circ i^\circ}(e_{j^\circ})
\]
(for $0 \leq i < j \leq 2$), and
\item a commutative square
\[ \begin{tikzcd}[column sep=2cm, row sep=2cm]
e_{0^\circ}
\arrow{r}{\vareps_{2^\circ0^\circ}}
\arrow{d}[swap]{\vareps_{1^\circ0^\circ}}
&
E_{2^\circ0^\circ}(e_{2^\circ})
\\
E_{1^\circ0^\circ}(e_{1^\circ})
\arrow{r}[swap]{E_{1^\circ0^\circ}(\vareps_{2^\circ1^\circ})}
&
E_{1^\circ0^\circ}(E_{2^\circ1^\circ}(e_{2^\circ}))
\arrow{u}
\end{tikzcd} \]
in $\cE_{0^\circ}$, where the morphism on the right is the canonical one (recall \Cref{example.lax.actions}\ref{describe.and.map.locally.cart.fibns}\ref{describe.locally.cart.fibn}).
\end{itemize}
Note that the structure map $\vareps_{2^\circ0^\circ}$ is likewise \textit{canonically} determined by the structure maps $\vareps_{2^\circ1^\circ}$ and $\vareps_{1^\circ0^\circ}$.

\end{enumerate}
\end{example}

\subsection{Lax actions and lax limits with mixed handedness}
\label{subsection.lax.subsection.with.mixed.handedness}

In this subsection, we introduce lax morphisms among lax modules (and in particular lax limits) of mixed handedness, using the theory of $(\infty,2)$-categories developed in \Cref{section.inftytwocats}.\footnote{Here we freely refer to some basic $(\infty,2)$-categorical notions (such as left- and right-laxification and 1-co/cartesian fibrations) described there.} We also record some key results here: an $(\infty,1)$-categorical description of such morphisms (\Cref{lemma.functors.into.LMod.rlax.llax.B}), as well as two results that relate such morphisms with those in the $\infty$-categories appearing in \Cref{subsection.lax.actions} via passage to adjoints (Lemmas \ref{lemma.ptwise.radjt.has.ptwise.ladjt} \and \ref{lem.get.r.lax.left.adjt}). To streamline our discussion, we restrict our attention to \textit{left} $\cB$-modules; the case of right $\cB$-modules is obtained by replacing $\cB$ with $\cB^\op$.

\begin{observation}
By \Cref{t11}, we can describe the $\infty$-categories of lax modules as well as their lax limit functors described in \S\S\ref{subsection.lax.actions} \and \ref{subsection.lax.limits} in terms of $(\infty,2)$-categories according to the identifications
\[
\begin{tikzcd}[column sep=0.5cm]
\LMod_{\llax.\cB}
\arrow[hook, two heads]{rr}
\arrow{rd}[sloped, swap]{\lim_{\llax.\cB}}
&
&
\LMod^\llax_{\llax.\cB}
\arrow{ld}[sloped, swap]{\lim^\llax_{\llax.\cB}}
\\
&
\Cat
\end{tikzcd}
\qquad
\simeq
\qquad
\begin{tikzcd}[column sep=0.5cm]
1\coCart_{\llax(\cB)}
\arrow[hook, two heads]{rr}
\arrow{rd}[sloped, swap]{\lim_{\llax(\cB)}}
&
&
2\Cat_{1\cocart/\llax(\cB)}
\arrow{ld}[sloped, swap]{\lim^\llax_{\llax(\cB)}}
\\
&
\Cat
\end{tikzcd}
\]
and 
\[
\begin{tikzcd}[column sep=0.5cm]
\LMod_{\rlax.\cB}
\arrow[hook, two heads]{rr}
\arrow{rd}[sloped, swap]{\lim_{\rlax.\cB}}
&
&
\LMod^\rlax_{\rlax.\cB}
\arrow{ld}[sloped, swap]{\lim^\rlax_{\rlax.\cB}}
\\
&
\Cat
\end{tikzcd}
\qquad
\simeq
\qquad
\begin{tikzcd}[column sep=0.5cm]
1\Cart_{\rlax(\cB)^{1\op}}
\arrow[hook, two heads]{rr}
\arrow{rd}[sloped, swap]{\lim_{\rlax(\cB)}}
&
&
2\Cat_{1\cart/\rlax(\cB)^{1\op}}
\arrow{ld}[sloped, swap]{\lim^\rlax_{\rlax(\cB)}}
\\
&
\Cat
\end{tikzcd}
~.
\]
\end{observation}

\begin{definition}
We define \bit{right-lax equivariant} morphisms between left-lax left $\cB$-modules and right-lax limits thereof as well as \bit{left-lax equivariant} morphisms between right-lax left $\cB$-modules and left-lax limits thereof according to the diagrams
\[
\begin{tikzcd}[row sep=1.2cm]
\LMod^\rlax_{\llax.\cB}
\arrow{rd}[sloped, swap]{\lim^\rlax_{\llax.\cB}}
&[-1cm]
:=
&[-1cm]
2\Cat_{1\cart/\llax(\cB)^{1\op}}
\arrow{ld}[sloped, swap]{\lim^\rlax_{\llax(\cB)}}
\\
&
\Cat
\end{tikzcd}
\qquad
\text{and}
\qquad
\begin{tikzcd}[row sep=1.2cm]
\LMod^\llax_{\rlax.\cB}
\arrow{rd}[sloped, swap]{\lim^\llax_{\rlax.\cB}}
&[-1cm]
:=
&[-1cm]
2\Cat_{1\cocart/\rlax(\cB)}
\arrow{ld}[sloped, swap]{\lim^\llax_{\rlax(\cB)}}
\\
&
\Cat
\end{tikzcd}
~.
\]
\end{definition}

\begin{observation}
\label{obs.that.defn.of.rlax.lim.extends}
As the notation suggests, we have canonical equivalences
\[
\iota_0 \LMod^\rlax_{\llax.\cB}
\simeq
\iota_0 \LMod^\llax_{\llax.\cB}
\]
on spaces of objects. Indeed, by \Cref{thm.un.straightening.for.two.cats} and ($1\&2\op$ applied to) \Cref{t11}, we have a commutative square
\[ \begin{tikzcd}
\LMod_{\llax.\cB}
\arrow[hook, two heads]{r}
&
\LMod^\rlax_{\llax.\cB}
\\
\LMod_\cB
\arrow[hook]{u}{\ff}
\arrow[hook, two heads]{r}
&
\LMod^\rlax_\cB
\arrow[hook]{u}[swap]{\ff}
\end{tikzcd} \]
extending the evident span in the diagram of \Cref{figure.define.almost.all.modules}, and moreover the two diagrams there extend to $3 \times 3$ commutative squares in the evident way. We use these facts without further comment.
\end{observation}

\begin{remark}
We give a purely $(\infty,1)$-categorical description of the right-lax limit of a left-lax left $\cB$-module in \Cref{subsection.lax.limits.via.subdivisions.in.hard.case} (see \Cref{prop.rlax.llaxB.via.sd}). In the main body of the work, we take this as an alternate definition. In the case that the only retracts in $\cB$ are equivalences, we also give another description in \Cref{subsection.strictification} as a strict limit over its subdivision category.
\end{remark}

\begin{lemma}\label{lemma.functors.into.LMod.rlax.llax.B}
For any $\infty$-category $\cC$, the datum of a functor
\[ \cC \longra \LMod^\rlax_{\llax.\cB} \]
is equivalent to the datum of a locally cocartesian fibration $\cE \da (\cC \times \cB)$ satisfying the following two conditions.
\begin{enumerate}
\item\label{condition.pb.to.C}
For every object $b \in \cB$, the base change $\cE_b \ra \cC$ is a (strict) cocartesian fibration.
\item
\label{condition.pb.to.brax.two}
For any pair of morphisms $c \ra c'$ in $\cC$ and $b \ra b'$ in $\cB$, the pullback along the functor $[2] \ra \cC \times \cB$ selecting the commutative triangle
\[ \begin{tikzcd}
(c,b)
\arrow{rd}
\arrow{d}
\\
(c,b')
\arrow{r}
&
(c',b')
\end{tikzcd} \]
is a (strict) cocartesian fibration.

\end{enumerate}
\end{lemma}

\begin{proof}
Using \Cref{thm.switching.yoga}, we have the identification
\[
\hom_{\iota_1 2\Cat}(\cC , \LMod^\rlax_{\llax.\cB})
:=
\hom_{\iota_1 2\Cat}(\cC , 2\Cat_{1\cart/\llax(\cB)^{1\op}} )
\simeq
\hom_{\iota_1 2\Cat} ( \llax(\cB) , 2\Cat_{1\cocart/\cC} )
~.
\]
The result now follows from \Cref{thm.unstraightening.yoga}.
\end{proof}

\begin{lemma}
\label{lemma.ptwise.radjt.has.ptwise.ladjt}
The datum of a morphism $\cE_0 \la \cE_1$ in $\LMod^\llax_{\llax.\cB}$ whose restriction to each $b \in \cB$ is a right adjoint is equivalent to the datum of a morphism $\cE_0 \ra \cE_1$ in $\LMod^\rlax_{\llax.\cB}$ whose restriction to each $b \in \cB$ is a left adjoint, with the equivalence given fiberwise by passing to adjoints.
\end{lemma}

\begin{proof}
This follows from \Cref{lemma.twocategorical.passing.to.adjoints.and.swapping.laxness} by taking $\cC = \llax(\cB)$.
\end{proof}

\begin{lemma}
\label{lem.get.r.lax.left.adjt}
A morphism $\cE_0 \la \cE_1$ in $\LMod_{\llax.\cB}$ whose restriction to each $b \in \cB$ is a right adjoint becomes a right adjoint in $\LMod^\rlax_{\llax.\cB}$, i.e.\! there exists a (necessarily unique) extension
\[ \begin{tikzcd}
{[1]^\op}
\arrow{r}
\arrow[hook]{d}
&
\LMod_{\llax.\cB}
\arrow[hook, two heads]{d}
\\
\Adj
\arrow[dashed]{r}
&
\LMod^\rlax_{\llax.\cB}
\end{tikzcd}
~.
\]
\end{lemma}

\begin{proof}
This follows from \Cref{lemma.twocategorical.extn.to.an.adjn.in.twoCatonecartoverC} by taking $\cC = \llax(\cB)$.
\end{proof}

\subsection{Subdivisions}
\label{subsection.sd.of.posets}

In this subsection, we study subdivisions of $\infty$-categories.

\begin{local}
In this subsection, we fix a poset $\pos \in \Poset$.
\end{local}

\begin{definition}
\label{defn.subdivision.of.a.poset}
The \bit{subdivision} of $\pos$ is the full subcategory
\[
\sd(\pos)
\subseteq
\bDelta_{/\pos}
:=
\bDelta
\times_\Cat
\Cat_{/\pos}
\]
on the conservative (or equivalently injective) functors $[n] \ra \pos$.
\end{definition}

\begin{definition}
\label{defn.iso.min.and.or.max}
A morphism $[m] \xra{\alpha} [n]$ in $\bDelta$ is called \bit{isomin} if $\alpha(0)=0$, \bit{isomax} if $\alpha(m) = n$.
We use the same terminology for morphisms in $\sd(\pos)$ according to their images under the forgetful functor $\sd(\pos) \ra \bDelta$.
\end{definition}

\begin{remark}
The $\infty$-category $\sd(\pos)$ is in fact a poset, namely the full subposet of the power set $\power(\pos)$ (ordered by inclusion) on those subsets of $\pos$ which are nonempty and totally ordered.
\end{remark}

\begin{example}
For each $[n] \in \bDelta \subseteq \Poset$, we have an identification
\[
\sd([n])
\simeq
\power_{\not=\es}([n])
\]
of its subdivision with its power set with its initial element removed, which is a punctured $(n+1)$-cube.
\end{example}

\begin{observation}
\label{obs.adjn.betw.Delta.over.pos.and.sd.pos}
The defining fully faithful inclusion $\sd(\pos) \subseteq \bDelta_{/\pos}$ admits a left adjoint
\[ \begin{tikzcd}[column sep=1.5cm]
\bDelta_{/\pos}
\arrow[dashed, transform canvas={yshift=0.9ex}]{r}{\im}
\arrow[hookleftarrow, transform canvas={yshift=-0.9ex}]{r}[yshift=-0.2ex]{\bot}
&
\sd(\pos)
\end{tikzcd}
~, \]
which takes each functor $[n] \xra{\varphi} \pos$ to the second factor in its epi-mono factorization
\[ \begin{tikzcd}
{[n]}
\arrow{rr}{\varphi}
\arrow[two heads, dashed]{rd}
&
&
\pos
\\
&
\im(\varphi)
\arrow[dashed, hook]{ru}[sloped, swap]{\ff}
\end{tikzcd}~. \]
\end{observation}

\begin{observation}
Subdivisions of posets assemble into a functor
\[
\Poset
\xlongra{\sd}
\Cat
\]
(whose unstraightening is a full subcategory of that of the functor $\Poset \xra{\bDelta_{/(-)}} \Cat$).
\end{observation}

\begin{lemma}
\label{lemma.new.and.old.subdivisions.agree}
The commutative triangle
\[ \begin{tikzcd}
\bDelta
\arrow[hook]{r}{\ff}
\arrow[hook]{d}[swap]{\ff}
&
\Poset
\arrow{r}{\sd}
&
\Cat
\\
\Poset
\arrow[bend right=10]{rru}[sloped, swap]{\sd}
\end{tikzcd} \]
is a left Kan extension diagram.
\end{lemma}

\begin{proof}
We must show that the canonical functor
\[
\colim
\left(
\bDelta_{/\pos}
\xra{\fgt}
\bDelta
\longhookra
\Poset
\xlongra{\sd}
\Cat
\right)
\longra
\sd(\pos)
\]
is an equivalence.  By \Cref{obs.adjn.betw.Delta.over.pos.and.sd.pos}, the functor $\sd(\pos) \hookra \bDelta_{/\pos}$ is final.  So, it is equivalent to show that the functor
\begin{equation}
\label{canonical.functor.from.colim.over.sdP.to.sd.P}
\colim
\left(
\sd(\pos)
\longhookra 
\bDelta_{/\pos}
\xra{\fgt}
\bDelta
\longhookra
\Poset
\xlongra{\sd}
\Cat
\right)
\longra
\sd(\pos)
\end{equation}
is an equivalence.  Consider the composite
\begin{equation}
\label{composite.from.sdP.to.Cat}
\sd(\pos)
\longhookra 
\bDelta_{/\pos}
\xra{\fgt}
\bDelta
\longhookra
\Poset
\xlongra{\sd}
\Cat
\end{equation}
(whose colimit is the source of the functor \Cref{canonical.functor.from.colim.over.sdP.to.sd.P}).  It is not hard to see that the unstraightening of the composite \Cref{composite.from.sdP.to.Cat} is the cocartesian fibration
\begin{equation}
\label{cocart.fibn.from.ArsdP.to.sdP}
\Ar(\sd(\pos))
\xlongra{t}
\sd(\pos)
~,
\end{equation}
and that the composite
\begin{equation}
\label{composite.from.ArsdP.to.sdP}
\Ar(\sd(\pos))
\longra
\colim \Cref{composite.from.sdP.to.Cat}
\xra{\Cref{canonical.functor.from.colim.over.sdP.to.sd.P}}
\sd(\pos)
\end{equation}
is precisely the functor $\Ar(\sd(\pos)) \xra{s} \sd(\pos)$, where the first functor in the composite \Cref{composite.from.ArsdP.to.sdP} is the localization of $\Ar(\sd(\pos))$ with respect to the cocartesian morphisms in the cocartesian fibration \Cref{cocart.fibn.from.ArsdP.to.sdP}.  These cocartesian morphisms are precisely the morphisms that are sent to equivalences by the functor $\Ar(\sd(\pos)) \xra{s} \sd(\pos)$.  But this functor is itself a localization (because it admits a fully faithful left adjoint), which shows that the functor \Cref{canonical.functor.from.colim.over.sdP.to.sd.P} is an equivalence.
\end{proof}

\begin{definition}
\label{define.sd}
Justified by \Cref{lemma.new.and.old.subdivisions.agree}, we define the \bit{subdivision} endofunctor on $\Cat$ as the left Kan extension
\[ \begin{tikzcd}
\bDelta
\arrow{r}{\sd}
\arrow[hook]{d}[swap]{\ff}
&
\Cat
\\
\Cat
\arrow[dashed]{ru}[swap, sloped]{\sd}
\end{tikzcd}~. \]
\end{definition}

\begin{lemma}
\label{t1}
There is a canonical functor
\[
\bDelta_{/\cB}
\longrightarrow
\sd(\cB)
\]
witnessing a localization on those morphisms $( [n] \da \cB ) \to ( [m] \da \cB)$ in $\bDelta_{/\cB}$ for which the underlying morphism $[n]\to [m]$ in $\bDelta$ is surjective.

\end{lemma}

\begin{proof}
Let $\what{\sd}(\cB) \to \bDelta_{/\cB}$ be the unstraightening of the composite functor
\begin{equation}
\label{e1}
\bDelta_{/\cB}
\xra{\fgt}
\bDelta
\xlongra{\sd}
\Cat
~.
\end{equation}
Note the natural transformation by fully faithful functors $\sd([\bullet]) \to \bDelta_{/[\bullet]}$, a morphism in $\Fun(\bDelta,\Cat)$.
Note also that the unstraightening of $\bDelta_{/\cB} \xra{\sf fgt} \bDelta \xra{\bDelta_{/[\bullet]}} \Cat$ is the cocartesian fibration $\Ar(\bDelta_{/\cB}) \xra{t} \bDelta_{/\cB}$.
Therefore, we identify
\[
\what{\sd}(\cB)
\subseteq
\Ar(\bDelta_{/\cB})
\]
as the full subcategory consisting of those objects
\[
([k]\to [n] \to \cB)
\]
in which $[k]\to [n]$ is injective.  
Observe that the fully faithful left adjoint
\[ \begin{tikzcd}[column sep=1.5cm, row sep=0cm]
\bDelta_{/\cB}
\arrow[dashed, hook, transform canvas={yshift=0.9ex}]{r}{\brax{\id}}
\arrow[leftarrow, transform canvas={yshift=-0.9ex}]{r}[yshift=-0.2ex]{\bot}[swap]{s}
&
\Ar(\bDelta_{/\cB})
\end{tikzcd} \]
 factors through $\what{\sd}(\cB)$.
Therefore, the composite functor
\begin{equation}
\label{e4}
\what{\sd}(\cB)
\longhookra
\Ar(\bDelta_{/\cB})
\xra{s}
\bDelta_{/\cB}
\end{equation}
is a right adjoint localization.

Next, by definition we have an equivalence $\colim \Cref{e1} \xra{\sim} \sd(\cB)$.
Therefore, there is a canonical functor
\begin{equation}
\label{e3}
\what{\sd}(\cB)
\longrightarrow
\sd(\cB)
\end{equation}
witnessing a localization on those morphisms in $\what{\sd}(\cB)$ that are cocartesian over $\bDelta_{/\cB}$.
This collection of morphisms is precisely those that are defined by a diagram
\[ \begin{tikzcd}
{[k]}
\arrow{r}
\arrow{d}
&
{[n]}
\arrow{r}
\arrow{d}
&
\cB
\arrow{d}[sloped, anchor=south]{=}
\\
{[k']}
\arrow{r}
&
{[n']}
\arrow{r}
&
\cB
\end{tikzcd} \]
in which the morphism $[k] \ra [k']$ is surjective.
Clearly, this localization \Cref{e3} factors through the localization \Cref{e4}.
Because localizations satisfy a two-out-of-three property, we have a resulting localization functor
\[
\bDelta_{/\cB}
\longrightarrow
\sd(\cB)
~.
\]
By direct inspection, this localization is as asserted.
\end{proof}

\begin{prop}
\label{t19}
Let $\cC$ be an $\infty$-category.
Let $\cB^{\op}\xra{E} \Cat$ be a functor from the opposite of an $\infty$-category.
Consider its cartesian unstraightening ${\sf Un}^R(E) \da \cB$.
There is a canonical identification
\[
{\sf Un}^{L} 
\left(
\Fun(E(\bullet),\cC)
\right)
\simeq
\Fun^{\sf rel}_{/\cB} \left(
{\sf Un}^{R}(E)
, 
\ul{\cC}
\right)
\]
in $\Cat_{/\cB}$, in which on the left is the cocartesian unstraightening of the composite functor
\[
\Fun(E(\bullet),\cC)
:
\cB
\xra{E^{\op}}
\Cat^{\op}
\xra{\Fun(-,\cC)}
\Cat
~.
\]
In particular, the functor $\Fun^{\sf rel}_{/\cB} \left(
{\sf Un}^{R}(E)
, 
\ul{\cC}
\right) \da \cB$ is a cocartesian fibration.
\end{prop}

\begin{proof}
Fix an object $(\cK \da \cB) \in \Cat_{/\cB}$. We have the sequence of identifications
\begin{align}
\label{first.equivce.for.t19}
\hom_{\Cat_{/\cB}} \left(
\cK , 
{\sf Un}^L \left( \Fun ( E(\bullet) , \cC) \right)
\right)
& \simeq
\hom_{\Fun(\cB,\Cat)} \left( \cK_{/\bullet} , \Fun(E(\bullet) , \cC ) \right)
\\
\label{second.equivce.for.t19}
& \simeq
\lim_{(b_s \ra b_t)^\circ \in \TwAr(\cB)^\op} \hom_\Cat
\left(
\cK_{/b_t} , \Fun(E(b_s) , \cC )
\right)
\\
\nonumber
& \simeq
\lim_{(b_s \ra b_t)^\circ \in \TwAr(\cB)^\op} \hom_\Cat \left(
E(b_s) \times \cK_{/b_t}
,
\cC
\right)
\\
\nonumber
& \simeq
\hom_\Cat \left(
\colim_{(b_s \ra b_t) \in \TwAr(\cB)} E(b_s) \times \cK_{/b_t}
,
\cC
\right)
\\
\label{fifth.equivce.for.t19}
& \simeq
\hom_\Cat \left( {\sf Un}^R(E)_{|\cK} , \cC \right)
\\
\nonumber
& =:
\hom_{\Cat_{/\cB}} \left( \cK , \Fun^\rel_{/\cB} ( {\sf Un}^R(E) , \ul{\cC} ) \right)
\end{align}
in $\cS$, in which equivalences \Cref{first.equivce.for.t19}, \Cref{second.equivce.for.t19}, and \Cref{fifth.equivce.for.t19} respectively follow from Theorem 4.5, Proposition 5.1, and Corollary 7.6 of \cite{GHN}.
\end{proof}

\begin{notation}
We write $\Ar^{\sf inrt}(\bDelta^\op) \subset \Ar(\bDelta^\op)$ for the full subcategory whose objects are the inert morphisms (i.e.\! the opposites of morphisms $\bDelta$ that are injective and convex).
Via evaluation at source, we regard $\Ar^{\sf inrt}(\bDelta^\op)$ as a category over $\bDelta^{\op}$.
\end{notation}

\begin{observation}
\label{t200}
The fiber of $\Ar^{\sf inrt}(\bDelta^\op) \xra{s} \bDelta^{\op}$ over $[n]^\circ \in \bDelta^{\op}$ is the full subcategory $(\bDelta_{/^{\sf inrt}[n]})^{\op} \subset (\bDelta_{/[n]})^{\op}$ whose objects are the inert morphisms.
It is easy to see (e.g.\! using the inert-active factorization system on $\bDelta^{\op}$) that the functor $\Ar^{\sf inrt}(\bDelta^\op) \xra{s} \bDelta^{\op}$ is a cartesian fibration.
\end{observation}

\begin{notation}
\label{notn.coSpan}
Given an $\infty$-category $\cD$, we write
\[
\cSpan(\cD)
\subseteq
\Fun^\rel_{/\bDelta^\op}
\left(
\Ar^{\sf inrt}(\bDelta^\op)
,
\ul{\cD}
\right)
\]
for the full subcategory on those objects $( [n] , ((\bDelta_{/^{\sf inrt} [n]})^{\op} \xra{F} \cD ) )$ such that for every $0 \leq i < j < k < l \leq n$, the functor $F$ carries (the opposite of) the diagram
\[ \begin{tikzcd}
{[j,k]}
\arrow[hook]{r}
\arrow[hook]{d}
&
{[j,l]}
\arrow[hook]{d}
\\
{[i,k]}
\arrow[hook]{r}
&
{[i,l]}
\end{tikzcd} \]
in $\bDelta_{/^{\sf inrt} [n]}$ to a pullback diagram in $\cD$.
\end{notation}

\begin{remark}
In \Cref{notn.coSpan}, if $\cD$ admits pullbacks, then by \Cref{t19} the functor $\cSpan(\cD) \ra \bDelta^\op$ is a cocartesian fibration whose maximal sub-left fibration is the unstraightening of the complete Segal space $\bDelta^\op \xra{\Span(\cD)} \Spaces$ corresponding to the $\infty$-category of spans in $\cD$ (as studied e.g.\! in \cite{rune-spans}).\footnote{This follows from the fact that the functor $\Ar^{\sf inrt}(\bDelta^\op) \xra{s} \bDelta^\op$ is a cartesian fibration (via the inert-active factorization system of $\bDelta^\op$) that corresponds to the functor $\bDelta \xra{\Sigma^\bullet} \Cat$ of \cite[Definition 5.1]{rune-spans}.}
\end{remark}

\begin{local}
\label{notn.sdprime}
We write $\sd'(\cB)$ for the following $\infty$-category of spans in $\bDelta_{/\cB}$:
\begin{itemize}

\item its objects are those objects $([n] \da \cB) \in \bDelta_{/\cB}$ such that for each $0 < i \leq n$ the composite functor $[i-1,i] \hookra [n] \ra \cB$ is conservative, and

\item its morphisms are those spans
\[
([n] \da \cB)
\longla
([k] \da \cB)
\longra
([m] \da \cB)
\]
for which the functor $[n] \la [k]$ is surjective and the functor $[k] \ra [m]$ is injective.
\end{itemize}
More specifically, we may define $\sd'(\cB)$ via the unstraightening of its corresponding complete Segal space, which is the evident subcategory of $\cSpan(\bDelta_{/\cB})$.\footnote{It is straightforward to verify that this $\infty$-category is indeed well-defined.}
\end{local}

\begin{lemma}
\label{lemma.equivce.between.sdprime.and.sd}
There is a canonical equivalence $\sd'(\cB) \simeq \sd(\cB)$.
\end{lemma}

\begin{proof}
\Cref{t200} gives that $\Ar^{\sf inrt}(\bDelta^\op) \xra{s} \bDelta^\op$ is a cartesian fibration.
Hence, the projection
\[
\Fun^\rel_{/\bDelta^\op} \left( \Ar^{\sf inrt} ( \bDelta^\op ) , \ul{\bDelta_{/\cB}} \right)
\longra
\bDelta^\op
\]
is a cocartesian fibration. Note that we have a subcategory inclusion
\begin{equation}
\label{e10}
\left( \bDelta_{/\sd'(\cB)} \right)^\op
\longhookra
\Fun^\rel_{/\bDelta^\op} \left( \Ar^{\sf inrt} ( \bDelta^\op ) , \ul{\bDelta_{/\cB}} \right)
~,
\end{equation}
by definition of $\sd'(\cB)$.

Now, consider the cocartesian fibration $\cU \ra \bDelta$ that straightens to the standard inclusion $\bDelta \hookra \Cat$. This cocartesian fibration has the feature that for any $\infty$-category $\cC$, we have a canonical inclusion
\[
\left( \bDelta_{/\cC} \right)^\op
\longhookra
\Fun^\rel_{/\bDelta^\op} ( \cU^\op , \ul{\cC} )
\]
which is that of the maximal sub-left fibration of the indicated $\infty$-category of relative functors.

Now, observe the morphism
\begin{equation}
\label{e8}
\begin{tikzcd}[row sep=0cm]
\cU
\arrow{r}
&
\Ar^{\sf inrt}(\bDelta)
&[-1.4cm]
\simeq
\Ar^{\sf inrt}(\bDelta^\op)^\op
\\
\rotatebox{90}{$\in$}
&
\rotatebox{90}{$\in$}
\\
(i \in [n])
\arrow[maps to]{r}
&
([n]_{\leq i} \hookra [n])
\end{tikzcd}
\end{equation}
in $\Cat_{/\bDelta}$. Hence, we obtain the composite morphism
\begin{equation}
\label{composite.functor.from.Delta.over.sdprimeB}
\left( \bDelta_{/\sd'(\cB)} \right)^\op
\xhookra{\Cref{e10}}
\Fun^\rel_{/\bDelta^\op} \left( \Ar^{\sf inrt} ( \bDelta^\op ) , \ul{\bDelta_{/\cB}} \right)
\xra{\Cref{e8}^*}
\Fun^\rel_{/\bDelta^\op} \left( \cU^\op , \ul{\bDelta_{/\cB}} \right)
\end{equation}
in $\Cat_{/\bDelta^\op}$. By inspection, the composite functor \Cref{composite.functor.from.Delta.over.sdprimeB} factors through the subcategory
\[
\Fun^\rel_{/\bDelta^\op} \left( \cU^\op , \ul{\bDelta_{/\cB}} \right)^{\sf surj}
\subseteq
\Fun^\rel_{/\bDelta^\op} \left( \cU^\op , \ul{\bDelta_{/\cB}} \right)
\]
with the same objects but only those morphisms that project via the forgetful functor $\bDelta_{/\cB} \xra{\fgt} \bDelta$ to surjections. Using the surjective-injective factorization system on $\bDelta$ as well as the localization-conservative factorization system on $\Cat$, we find that the resulting functor
\[
\left( \bDelta_{/\sd'(\cB)} \right)^\op
\longra
\Fun^\rel_{/\bDelta^\op} \left( \cU^\op , \ul{\bDelta_{/\cB}} \right)^{\sf surj}
\]
over $\bDelta^\op$ is fully faithful on fibers, with the image on fibers over $[n]^\circ \in \bDelta^\op$ consisting of those objects $([k_0] \ra \cdots \ra [k_n] \ra \cB)$ such that
\begin{itemize}
\item for each $0 < i \leq n$ the functor $[k_{i-1}] \ra [k_i]$ is injective, and
\item for each $0 < j \leq k_n$ the composite functor $[j-1,j] \hookra [k_n] \ra \cB$ is conservative.
\end{itemize}
Moreover, using the same two factorization systems, this fully faithful inclusion on fibers over $[n]^\circ \in \bDelta^\op$ admits a left adjoint. Therefore, the composite functor
\begin{equation}
\label{e12}
\left( \bDelta_{/\sd'(\cB)} \right)^\op
\longra
\Fun^\rel_{/\bDelta^\op} \left( \cU , \ul{\bDelta_{/\cB}} \right)^{\sf surj}
\longra
\left(
\Fun^\rel_{/\bDelta^\op} \left( \cU , \ul{\bDelta_{/\cB}} \right)^{\sf surj}
\right)^{\what{\ \ \ }}_{{\sf l.fib}}
\end{equation}
to the fiberwise $\infty$-groupoid completion is an equivalence. In particular, the target of the functor, considered as a left fibration over $\bDelta^\op$, unstraightens to a complete Segal space. By \cite[Theorem 3.8]{rnerves}, this complete Segal space corresponds to the localization of $\bDelta_{/\cB}$ at those morphisms that project to surjections under the forgetful functor $\bDelta_{/\cB} \xra{\fgt} \bDelta$. By \Cref{t1}, we obtain the desired equivalence $\sd'(\cB) \xra{\sim} \sd(\cB)$.
\end{proof}

\begin{cor}
\label{t3}
Suppose that every retraction in $\cB$ is trivial (i.e.\! is a pair of inverse equivalences). Then, we have an identification $\sd(\cB) \subseteq \bDelta_{/\cB}$ as the full subcategory on those objects $([n] \da \cB) \in \bDelta_{/\cB}$ defined by conservative functors.
\end{cor}

\begin{proof}
This follows directly from \Cref{lemma.equivce.between.sdprime.and.sd} (and \Cref{notn.sdprime}).
\end{proof}

\begin{notation}
For any $[n] \in \bDelta$, there are evident functors
\begin{equation}
\label{max.and.min.from.sd.braxn.and.sd.braxnop}
\sd([n])
\xra{\max}
[n]
\qquad
\text{and}
\qquad
\sd([n])^\op
\xra{\min}
[n]
~,
\end{equation}
which respectively take a nonempty subset of $[n]$ to its maximal or minimal element. By functoriality of left Kan extension, these induce augmentations in $\Fun(\Cat,\Cat)$ that we likewise denote by
\[
\sd
\xra{\max}
\id
\qquad
\text{and}
\qquad
\sd^\op
\xra{\min}
\id
~.
\]
\end{notation}

\needspace{2\baselineskip}
\begin{lemma}
\label{t4}
\begin{enumerate}

\item[]

\item\label{t4.part.one}
The functor $\sd(\cB) \xra{\max} \cB$ is a locally cocartesian fibration.
Furthermore, for each locally cocartesian fibration $(\cE \da \cB) \in \loc.\coCart_\cB$, the canonical functor
\begin{equation}
\label{e20}
\ulhom_{\loc.\coCart_{\cB}} \left( \sd(\cB) , \cE \right)
\longrightarrow
\lim_{( [n] \da \cB) \in \bDelta_{/\cB}} 
\ulhom_{\loc.\coCart_{[n]}} \left( \sd([n]) , \cE_{|[n]} \right)
\end{equation}
is an equivalence, which is functorial in the object $(\cE \da \cB) \in \loc.\coCart_\cB$.

\item\label{t4.part.two}
The functor $\sd(\cB) \xra{\min} \cB^{\op}$ is a locally cartesian fibration.
Furthermore, for each locally Cartesian fibration $(\cE \da \cB^{\op}) \in \loc.\Cart_{\cB^\op}$, the canonical functor
\[
\ulhom_{\loc.\Cart_{\cB^{\op}}} \left( \sd(\cB) , \cE \right)
\longrightarrow
\lim_{( [n] \da \cB) \in \bDelta_{/\cB}} 
\ulhom_{\loc.\Cart_{[n]^{\op}}} \left( \sd([n]) , \cE_{|[n]^{\op}} \right)
\]
is an equivalence, which is functorial in the object $(\cE \da \cB^{\op}) \in \loc.\Cart_{\cB^{\op}}$.

\end{enumerate}
\end{lemma}

\begin{example}
\label{ex.max.from.sdbraxtwo.to.braxtwo.is.a.loc.cocart.fibn}
In the case that $[n] \in \bDelta$, the functors \Cref{max.and.min.from.sd.braxn.and.sd.braxnop} are respectively a locally cocartesian fibration and a locally cartesian fibration (as asserted by \Cref{t4}): in both cases the monodromy functors are given by union, as illustrated in \Cref{sd.of.brackets.2}.
\begin{figure}[h]
\begin{equation}
\label{the.locally.cocart.fibn.max.from.sd.two.to.two}
\begin{tikzcd}
&
2
\arrow{rr}
\arrow{dd}
&
&
12
\arrow{dd}
\\
&
&
1
\arrow{ru}
&
&
&
&
&
\sd([2])
\arrow{dddd}{\max}
\\
&
02
\arrow{rr}
&
&
012
\\
0
\arrow{ru}
\arrow{rr}
&
&
01
\arrow[crossing over, leftarrow]{uu}
\arrow{ru}
\\
&
&
&
2
\\
0
\arrow{rr}
\arrow{rrru}
&
&
1
\arrow{ru}
&
&
&
&
&
{[2]}
\end{tikzcd}
\end{equation}
\begin{equation}
\label{unstraightening.of.max.from.sd.two.to.two}
\begin{tikzcd}[row sep=1.5cm]
&
{[1]}
\arrow{rd}[sloped]{(\id,\const_1)}
\\
{[0]}
\arrow{rr}[swap]{(1,0)}[yshift=0.75cm]{\Uparrow}
\arrow{ru}[sloped]{1}
&
&
{[1] \times [1]}
\end{tikzcd}
\end{equation}
\caption{The functor $\sd([2]) \xra{\max} [2]$ is a locally cocartesian fibration, as illustrated in diagram \Cref{the.locally.cocart.fibn.max.from.sd.two.to.two}; its unstraightening is illustrated in diagram \Cref{unstraightening.of.max.from.sd.two.to.two}.  Note that the functor $\sd([1]) \xra{\max} [1]$ can also be seen in diagram \Cref{the.locally.cocart.fibn.max.from.sd.two.to.two} in three different ways, corresponding to the three nonidentity morphisms in $[2]$.}
\label{sd.of.brackets.2}
\end{figure}
\end{example}

\begin{proof}[Proof of \Cref{t4}]
We prove part \Cref{t4.part.one}; the proof of part \Cref{t4.part.two} is essentially identical.

To show that $\sd(\cB) \xra{\max} \cB$ is a locally cocartesian fibration, by \Cref{lemma.equivce.between.sdprime.and.sd} it is equivalent to show that the functor
\[
\sd'(\cB)
\xra{([n] \xra{\varphi} \cB) \longmapsto \varphi(n)}
\cB
\]
is a locally cocartesian fibration. For this, given an object $([n] \xra{\varphi} \cB) \in \sd'(\cB)$ we must show that there exists a locally cocartesian lift of any morphism $\varphi(n) \xra{f} b$ in $\cB$. It suffices to assume that $f$ is not an equivalence. In this case, it is clear from the definition of $\sd'(\cB)$ that the morphism $\varphi \ra \psi$ therein defined by the span $\varphi \xla{\sim} \varphi \xra{\iota} \psi$ is a locally cocartesian lift of $f$, where the morphism $\iota$ in $\bDelta_{/\cB}$ is defined by the commutative triangle
\begin{equation}
\label{triangle.in.Delta.over.B.inclusion.from.brax.n.to.brax.n.plus.one}
\begin{tikzcd}
{[n]}
\arrow[hook]{rr}{i \longmapsto i}
\arrow{rd}[sloped, swap]{\varphi}
&
&
{[n+1]}
\arrow{ld}[sloped, swap]{\psi}
\\
&
\cB
\end{tikzcd}
\end{equation}
in $\Cat$ in which $\psi(n \ra (n+1)) = f$.
In other words, a morphism in $\sd(\cB)$ is locally cocartesian over $\cB$ if it is the image of a morphism \Cref{triangle.in.Delta.over.B.inclusion.from.brax.n.to.brax.n.plus.one} in $\bDelta_{/\cB}$ under the localization $\bDelta_{/\cB} \ra \sd(\cB)$ of \Cref{t1}.

We now establish the equivalence \Cref{e20}. 
By the definition of $\sd(\cB)$ as a colimit, we have that the functor
\begin{equation}
\label{e21}
\ulhom_{\Cat_{/\cB}} \left( \sd(\cB) , \cE \right)
\longrightarrow
\lim_{( [n] \da \cB) \in \bDelta_{/\cB}} 
\ulhom_{\Cat_{/[n]}} \left( \sd([n]) , \cE_{|[n]} \right)
\end{equation}
is an equivalence. 
Using the above description of the locally cocartesian morphisms of the functor $\sd(\cB) \xra{\max} \cB$, we see that for each object $([n]\da \cB) \in \bDelta_{/\cB}$ the corresponding functor $\sd([n]) \to \sd(\cB)$ carries locally cocartesian morphisms over $[n]$ to locally cocartesian morphisms over $\cB$, and moreover that each locally cocartesian morphism in $\sd(\cB)$ is the image of such a morphism.
It follows that the equivalence \Cref{e21} restricts as the equivalence \Cref{e20}.
\end{proof}

\subsection{Lax limits with mixed handedness via subdivisions}
\label{subsection.lax.limits.via.subdivisions.in.hard.case}

In this subsection, we give an alternative description of left-lax limits of right-lax left $\cB$-modules in terms of the subdivision of $\cB$ (\Cref{prop.rlax.llaxB.via.sd}). In the main body of the work, we use this description as an alternate definition.

\begin{prop}
\label{prop.rlax.llaxB.via.sd}
Right-lax limits of left-lax left $\cB$-modules are corepresented by the object $(\sd(\cB) \da \cB) \in \loc.\coCart_\cB =: \LMod_{\llax.\cB}$. In other words, there is a canonical commutative diagram
\[
\begin{tikzcd}[column sep=1.5cm]
\LMod_{\llax.\cB}
\arrow[hook, two heads]{r}
&
\LMod_{\llax.\cB}^\rlax
\arrow{rdd}[sloped]{\lim^\rlax_{\llax.\cB}}
\\[-0.7cm]
\rotatebox{90}{$=:$}
\\[-0.7cm]
\loc.\coCart_\cB
\arrow{rr}[sloped, swap]{\Fun^\cocart_{/\cB}(\sd(\cB),-)}
&
&
\Cat
\end{tikzcd}
~.
\]
\end{prop}

\begin{proof}
This is \Cref{thm.defns.of.rlax.limit.over.onecats.agree}.
\end{proof}



\begin{example}
\label{example.limits.of.lax.actions.when.laxness.disagrees}
Using \Cref{prop.rlax.llaxB.via.sd}, let us unwind the definitions of the functors
\[
\LMod_{\llax.\cB}
\xra{\lim^\rlax_{\llax.\cB}}
\Cat
\qquad
\text{and}
\qquad
\RMod_{\rlax.\cB}
\xra{\lim^\llax_{\rlax.\cB^\op}}
\Cat
\]
in the simplest nontrivial case, namely when $\cB = [2]$.\footnote{Recall the description of $\sd([2])$ given in \Cref{sd.of.brackets.2}.}
\begin{enumerate}

\item\label{example.limits.of.lax.actions.when.laxness.disagrees.rlax.lim.of.llax.action}
Let $\cE \da [2]$ be a locally cocartesian fibration, and let us employ the notation of \Cref{example.lax.actions}\Cref{describe.and.map.locally.cocart.fibns}\Cref{describe.locally.cocart.fibn}.  Then, an object of the right-lax limit of this left-lax left $[2]$-module is given by the data of
\begin{itemize}
\item objects $e_i \in \cE_{i}$ (for $0 \leq i \leq 2$),
\item morphisms
\[
e_j
\xra{\vareps_{ij}}
E_{ij}(e_i)
\]
(for $0 \leq i < j \leq 2$), and
\item a commutative square
\[ \begin{tikzcd}[column sep=1.5cm, row sep=1.5cm]
e_2
\arrow{r}{\varepsilon_{01}}
\arrow{d}[swap]{\varepsilon_{02}}
&
E_{12}(e_1)
\arrow{d}{E_{12}(\varepsilon_{01})}
\\
E_{02}(e_0)
\arrow{r}
&
E_{12}(E_{01}(e_0))
\end{tikzcd} \]
in $\cE_{2}$, where the lower morphism is the canonical one (recall \Cref{example.lax.actions}\ref{describe.and.map.locally.cocart.fibns}\ref{describe.locally.cocart.fibn}).
\end{itemize}
Note that the structure map $\vareps_{02}$ is \textit{not} generally determined by the structure maps $\vareps_{01}$ and $\vareps_{12}$ (in contrast with \Cref{example.lax.limits.with.agreeing.lax.action}\Cref{item.example.lax.limits.with.agreeing.lax.action.loc.cocart}).

\item Let $\cE \da [2]$ be a locally cartesian fibration, and let us employ the notation of \Cref{example.lax.actions}\Cref{describe.and.map.locally.cart.fibns}\Cref{describe.locally.cart.fibn}.  Then, an object of the left-lax limit of this right-lax right $[2]$-module is given by the data of
\begin{itemize}
\item objects $e_{i^\circ} \in \cE_{i^\circ}$ (for $0 \leq i \leq 2$),
\item morphisms 
\[
E_{j^\circ i^\circ}(e_{j^\circ})
\xra{\vareps_{j^\circ i^\circ}}
e_{i^\circ}
\]
(for $0 \leq i < j \leq 2$), and
\item a commutative square
\[ \begin{tikzcd}[column sep=2cm, row sep=2cm]
E_{1^\circ0^\circ}(E_{2^\circ1^\circ}(e_{2^\circ}))
\arrow{r}
\arrow{d}[swap]{E_{1^\circ0^\circ}(\vareps_{2^\circ1^\circ})}
&
E_{2^\circ0^\circ}(e_{2^\circ})
\arrow{d}{\vareps_{2^\circ0^\circ}}
\\
E_{1^\circ0^\circ}(e_{1^\circ})
\arrow{r}[swap]{\vareps_{1^\circ0^\circ}}
&
e_{0^\circ}
\end{tikzcd} \]
in $\cE_{0^\circ}$, , where the upper morphism is the canonical one (recall \Cref{example.lax.actions}\ref{describe.and.map.locally.cart.fibns}\ref{describe.locally.cart.fibn}).
\end{itemize}
Note that the structure map $\vareps_{2^\circ0^\circ}$ is likewise \textit{not} generally determined by the structure maps $\vareps_{2^\circ1^\circ}$ and $\vareps_{1^\circ0^\circ}$ (again in contrast with \Cref{example.lax.limits.with.agreeing.lax.action}\Cref{item.example.lax.limits.with.agreeing.lax.action.loc.cart}).

\end{enumerate}
\end{example}

\begin{remark}
\label{rmk.why.compatibility.cubes}
It is because we are taking e.g.\! the \textit{right}-lax limit of a \textit{left}-lax module that we end up with the perhaps unfamiliar compatibility conditions of the commutative squares in \Cref{example.limits.of.lax.actions.when.laxness.disagrees}.  Comparing with \Cref{example.lax.limits.with.agreeing.lax.action}, we see that the analogous compatibility condition for e.g.\! the left-lax limit of a left-lax module as a section of a locally cocartesian fibration is simply that the section preserves composition of morphisms -- which is of course built into the very definition of a functor.
\end{remark}

\begin{example}
\label{obs.rlax.lim.of.trivial.action}
Consider the projection from the product
\begin{equation}
\label{projection.from.product.B.x.G.to.B}
\ul{\cG}
:=
\cG \times \cB
\xlongra{\pr}
\cB
\end{equation}
as an object of $\LMod_\cB$. In the diagram
\[ \begin{tikzcd}
\sd(\cB)
\arrow{r}{\min}
\arrow{d}[swap]{\max}
&
{\cB^\op}
\\
\cB
\end{tikzcd}~, \]
the horizontal functor is the localization at the locally cocartesian morphisms with respect to the vertical functor.\footnote{This follows from the case in which $\cB = [n]$, which follows from \Cref{ex.max.from.sdbraxtwo.to.braxtwo.is.a.loc.cocart.fibn}.} Hence, we find that
\[
\lim^\rlax_\cB ( \ul{\cG} )
:=
\Fun^\cocart_{/\cB} \left( \sd(\cB) , \ul{\cG} \right)
:=
\Fun^\cocart_{/\cB} \left( \sd(\cB) , \cG \times \cB \right)
\simeq
\Fun \left( \cB^\op , \cG \right)
~.
\]
Dually, considering $\Cref{projection.from.product.B.x.G.to.B} \in \RMod_\cB$, we have that
\[
\lim^\llax_{\cB^\op} ( \ul{\cG})
\simeq
\Fun(\cB,\cG)
~. \]
\end{example}

\begin{observation}
\label{obs.map.from.strict.limit.to.lax.limit}
Given a $\cB$-module of any sort, there are canonical fully faithful inclusions from its strict limit to its various lax limits. In terms of \Cref{prop.rlax.llaxB.via.sd}, the canonical morphism
\[
\lim_{\llax.\cB}(-)
\longhookra
\lim^\rlax_{\llax.\cB}(-)
\]
in $\Fun(\LMod_{\llax.\cB} , \Cat)$ is corepresented by the epimorphism (in fact localization) $\cB \xla{\max} \sd(\cB)$ in $\LMod_{\llax.\cB} := \loc.\coCart_\cB$.
\end{observation}

\subsection{An alternative description of right-lax limits of left-lax modules}
\label{subsection.strictification}

In this subsection, we provide a useful alternative description of right-lax limits of left-lax modules.

\begin{definition}
For any object $\varphi \in \sd(\cB)$, its \bit{isomax undercategory} is the fiber
\[ \begin{tikzcd}
\sd(\cB)_{\varphi/\isomax}
\arrow{r}
\arrow{dd}
&
\sd(\cB)_{\varphi/}
\arrow{d}{t}
\\
&
\sd(\cB)
\arrow{d}{\max}
\\
\{ \max(\varphi) \}
\arrow{r}
&
\cB
\end{tikzcd}~, \]
i.e.\! the $\infty$-category of isomax morphisms in $\sd(\cB)$ with source $\varphi$.
\end{definition}

\begin{observation}
\label{t18}
Suppose that the only retracts in $\cB$ are equivalences. Observe the factorization system on $\bDelta$, whose left factor consists of isomax morphisms and whose right factor consists of isomin morphisms that are moreover consecutive inclusions: it takes a morphism $[m] \xra{\alpha} [n]$ in $\bDelta$ to the factorization
\[ \begin{tikzcd}[column sep=0.5cm]
{[m]}
\arrow{rr}{\alpha}
\arrow[dashed]{rd}[swap, sloped]{\alpha_L}
&
&
{[n]}
\\
&
{[n]_{/\alpha(m)}}
\arrow[dashed]{ru}[swap, sloped]{\alpha_R}
\end{tikzcd}~. \]
This lifts to a factorization system on $\bDelta_{/\cB}$, which restricts to a factorization system on $\sd(\cB) \subseteq \bDelta_{/\cB}$ (recall \Cref{t3}).
\end{observation}

\begin{notation}
\label{notn.composite.loc.cocart.mdrmy.fctr}
Given a locally cocartesian fibration $\cE \da \cB$ and an object $([m] \xra{\varphi} \cB) \in \sd(\cB)$ (using \Cref{lemma.equivce.between.sdprime.and.sd}), we write
\[
\cE_\varphi
:
\cE_{\varphi(0)}
\xra{\cE_{\varphi(\{0<1\})}}
\cE_{\varphi(1)}
\xra{\cE_{\varphi(\{1<2\})}}
\cdots
\xra{\cE_{\varphi(\{(m-1)<m\})}}
\cE_{\varphi(m)}
\]
for the composite of locally cocartesian monodromy functors.
\end{notation}

\begin{observation}
\label{obs.unstraightening.preserves.colimits}
The forgetful functor
\[
\coCart_{\cB}
\longra
\Cat_{/\cB}
\]
preserves colimits.  Indeed, by \cite[Theorem 7.4]{GHN}, the composite functor
\[
\Fun(\cB, \Cat)
\simeq
\coCart_\cB
\longhookra
\Cat_{/\cB}
\xra{\fgt}
\Cat
\]
is given by a formula which evidently preserves colimits.  Thereafter, the claim follows from the fact the functor
$\Cat_{/\cB} \xra{\fgt} \Cat$ is conservative and preserves colimits.
\end{observation}

\begin{observation}
\label{t16}
Let $( \cE \xra{\pi} \cB ) \in \loc.\coCart_\cB$ be a locally cocartesian fibration.
The postcomposition functor
\begin{equation}
\label{e23}
\coCart_{\cE}
\xra{\pi \circ (-)}
\loc.\coCart_\cB
\end{equation}
admits a right adjoint, which we describe presently.

First of all, using \Cref{obs.unstraightening.preserves.colimits}, we see that the composite forgetful functor $\coCart_\cE \xra{\fgt} \Cat_{/\cE} \xra{\pi \circ (-)} \Cat_{/\cB}$ preserves colimits. Moreover, any cocone in $\loc.\coCart_\cB \subseteq \Cat_{/\cB}$ that is a colimit diagram in $\Cat_{/\cB}$ is also a colimit diagram in $\loc.\coCart_\cB$. Therefore, the functor \Cref{e23} preserves colimits.

Now, notice that the composite functor
\[
y_\pi
\colon
\cE^{\op}
\xlongra{\Yo}
\Fun(\cE , \Spaces)
\longhookra
\Fun(\cE , \Cat)
\simeq
\coCart_\cE
\xra{\pi\circ(-)}
\loc.\coCart_{\cB}
\]
evaluates as
\[
\cE \ni
e
\longmapsto
( \cE_{e/} \da \cB )
\in \loc.\coCart_{\cB}
~.\footnote{See e.g.\! \cite[Lemma 4.7]{GHN}.}
\]
It follows that the right adjoint to \Cref{e23} is given through un/straightening by the functor
\[ \begin{tikzcd}[row sep=0cm]
\loc.\coCart_\cB
\arrow{r}
&
\coCart_\cE
\simeq
&[-1.7cm]
\Fun(\cE,\Cat)
\\
\rotatebox{90}{$\in$}
&
&
\rotatebox{90}{$\in$}
\\
(\cE \da \cB)
\arrow[maps to]{rr}
&
&
\ulhom_{\loc.\coCart_\cB}(y_\pi , \cE)
\end{tikzcd}
~.
\]
\end{observation}

\begin{notation}
\label{d5}
Let $( \cE \xra{\pi} \cB ) \in \loc.\coCart_{\cB}$ be a locally cocartesian fibration.
We write
\[ \begin{tikzcd}[column sep=1.5cm]
\coCart_\cE
\arrow[transform canvas={yshift=0.9ex}]{r}{\pi \circ (-)}
\arrow[leftarrow, dashed, transform canvas={yshift=-0.9ex}]{r}[yshift=-0.2ex]{\bot}[swap]{\Strict_\pi}
&
\loc.\coCart_\cB
\end{tikzcd} \]
for the right adjoint given by \Cref{t16}. In the case that $(\cE \xra{\pi} \cB) = (\sd(\cB) \xra{\max} \cB)$ (recall \Cref{t4}\Cref{t4.part.one}), we simply write $\Strict := \Strict_\max$.
\end{notation}

\begin{observation}
\label{t17}
Let $( \cE \xra{\pi} \cB ) \in \loc.\coCart_{\cB}$ be a locally cocartesian fibration.
By \Cref{t16}, there are canonical equivalences
\[
\ulhom_{\loc.\coCart_{\cB}}( \cE , - ) 
\simeq
\Gamma^{\cocart}_{\cE}
\left(
\Strict_\pi(-) 
\right)
\simeq
\lim
\left(
\cE
\xra{\Strict_\pi (-)}
\Cat
\right)
\]
in $\Fun(\loc.\coCart_\cB , \Cat)$.
\end{observation}

\begin{lemma}
\label{lem.strictification}
Suppose that the only retracts in $\cB$ are equivalences.
Then, the diagram
\begin{equation}
\label{e24}
\begin{tikzcd}
\LMod_{\llax.\cB}
\arrow{rr}{\lim^\rlax_{\llax.\cB}}
\arrow{rd}[sloped, swap]{\Strict}
&
&
\Cat
\\
&
\LMod_{\sd(\cB)}
\arrow{ru}[sloped, swap]{\lim_{\sd(\cB)}}
\end{tikzcd}
\end{equation}
canonically commutes. Moreover, given a left-lax left $\cB$-module $(\cE \da \cB) \in \LMod_{\llax.\cB}$, the left $\sd(\cB)$-module $(\Strict(\cE) \da \sd(\cB)) \in \LMod_{\sd(\cB)}$ has the following properties.
\begin{itemize}
\item Its fiber over an object $([m] \xra{\varphi} \cB) \in \sd(\cB)$ is the $\infty$-category
\[
\Strict(\cE)_\varphi
:=
\Fun
\left(
\sd(\cB)_{\varphi/\isomax}
,
\cE_{\max(\varphi)}
\right)
~.
\]
\item Over a morphism 
\begin{equation}
\label{morphism.in.sd.P}
\begin{tikzcd}
{[m]}
\arrow[hook]{rr}{\alpha}
\arrow[hook]{rd}[sloped, swap]{\varphi}
&
&
{[n]}
\\
&
\cB
\arrow[hookleftarrow]{ru}[sloped, swap]{\psi}
\end{tikzcd}
\end{equation}
in $\sd(\cB)$, its cocartesian monodromy functor
\begin{equation}
\label{cocart.mdrmy.functor.in.strictification}
\Fun \left( \sd(\cB)_{\varphi/\isomax} , \cE_{\max(\varphi)} \right)
=:
\Strict(\cE)_\varphi
\xra{\Strict(\cE)_\alpha}
\Strict(\cE)_\psi
:=
\Fun \left( \sd(\cB)_{\psi/\isomax} , \cE_{\max(\psi)} \right)
\end{equation}
evaluates on a functor
\begin{equation}
\label{arbitrary.functor.from.isomax.undercat.to.corresponding.fiber}
\sd(\cB)_{\varphi/\isomax}
\xlongra{F}
\cE_{\max(\varphi)}
\end{equation}
as a functor that evaluates as
\[ \begin{tikzcd}[column sep=1.5cm, row sep=0cm]
\sd(\cB)_{\psi/\isomax}
\arrow{r}{\Strict(\cE)_\alpha(F)}
&
\cE_{\max(\psi)}
\\
\rotatebox{90}{$\in$}
&
\rotatebox{90}{$\in$}
\\
(\psi \xlongra{\beta} \omega)
\arrow[maps to]{r}
&
\cE_{\omega'}(F((\beta\alpha)_L))
\end{tikzcd}~, \]
where writing $([k] \xra{\omega} \cB) \in \sd(\cB)$ we denote by
\[
\omega'
:=
\left(
[k]_{\max((\beta\alpha)_R)/}
\longhookra
[k]
\xlonghookra{\omega}
\cB
\right)
\in
\sd(\cB)
\]
the restriction of $\omega$ to the undercategory of $\max((\beta\alpha)_R) \in [k]$.

\end{itemize}
Furthermore, in the case that $\cB$ is an ordinary category, 
the above properties characterize the functor $\Strict$.
\end{lemma}

\begin{example}
Suppose that $\cB = [1]$, so that
\[
\sd(\cB)
=
\sd([1])
= \left(
\begin{tikzcd}
&
1
\arrow{d}
\\
0
\arrow{r}
&
01
\end{tikzcd} \right)
~. 
\]
Given a (locally) cocartesian fibration $\cE \da [1]$ classified by a diagram $\cE_0 \xra{F} \cE_1$, the functor $\sd([1]) \xra{\Strict(\cE)} \Cat$ selects the diagram
\[ \begin{tikzcd}
&
\Fun([1],\cE_1)
\arrow{d}{t}
\\
\cE_0
\arrow{r}[swap]{F}
&
\cE_1
\end{tikzcd}~, \]
whose limit is indeed
\[
\lim^\rlax_{\llax.[1]}(\cE)
\simeq
\lim^\rlax_{[1]} ( \cE )
:=
\Gamma \left( \left( \begin{tikzcd}
\cE
\arrow{d}
\\
{[1]}
\end{tikzcd} \right)^\cocartdual
\right)
~.
\]
\end{example}

\begin{example}
Suppose that $\cB = [2]$, so that $\sd(\cB) = \sd([2])$ is as depicted in diagram \Cref{the.locally.cocart.fibn.max.from.sd.two.to.two} of \Cref{sd.of.brackets.2}.  Given a locally cocartesian fibration $\cE \da [2]$ selecting a lax-commutative triangle
\[ \begin{tikzcd}
&
\cE_{1}
\arrow{rd}[sloped]{G}
\\
\cE_{0}
\arrow{ru}[sloped]{F}
\arrow{rr}[transform canvas={xshift=-0.1cm, yshift=0.3cm}]{\eta \Uparrow}[swap]{H}
&
&
\cE_{2}
\end{tikzcd}~, \]
the functor $\sd([2]) \xra{\Strict(\cE)} \Cat$ selects the diagram
\[ \begin{tikzcd}[row sep=1.5cm]
&
\Fun([1] \times [1],\cE_2)
\arrow{rr}{(\id,\const_1)^*}
\arrow{dd}[swap]{(\const_1,\id)^*}
&
&
\Fun([1],\cE_2)
\arrow{dd}{t}
\\
&
&
\Fun([1],\cE_1)
\arrow{ru}[sloped]{G}
\\
&
\Fun([1],\cE_2)
\arrow{rr}[swap, pos=0.3]{t}
&
&
\cE_2
\\
\cE_0
\arrow{rr}[swap]{F}
\arrow{ru}[sloped]{\eta}
&
&
\cE_1
\arrow{ru}[sloped, swap]{G}
\arrow[crossing over, leftarrow]{uu}[pos=0.7]{t}
\end{tikzcd}~, \]
whose limit is indeed $\lim^\rlax_{\llax.[2]}(\cE)$ (as described in \Cref{example.limits.of.lax.actions.when.laxness.disagrees}\Cref{example.limits.of.lax.actions.when.laxness.disagrees.rlax.lim.of.llax.action}).
\end{example}

\begin{remark}
Let $\pos$ be a poset. By construction, for any inclusion $\sD \hookra \pos$ of a down-closed subset we have a commutative square
\[ \begin{tikzcd}
\LMod_{\llax.\pos}
\arrow{r}{\Strict}
\arrow{d}
&
\LMod_{\sd(\pos)}
\arrow{d}
\\
\LMod_{\llax.\sD}
\arrow{r}[swap]{\Strict}
&
\LMod_{\sd(\sD)}
\end{tikzcd} \]
(in which the horizontal functors are those of \Cref{lem.strictification} and the vertical functors are restriction).  
\end{remark}


\begin{notation}
\label{notation.source.and.or.target.restricted.sd}
For any $b,c \in \cB$, we write
\[
\begin{tikzcd}
\sd(\cB)^{|b}
\arrow{r}
\arrow{d}
&
\sd(\cB)
\arrow{d}{\min}
\\
\pt
\arrow{r}[swap]{b^\circ}
&
\cB^\op
\end{tikzcd}
~,
\qquad
\begin{tikzcd}
\sd(\cB)_{|c}
\arrow{r}
\arrow{d}
&
\sd(\cB)
\arrow{d}{\max}
\\
\pt
\arrow{r}[swap]{c}
&
\cB
\end{tikzcd}
~,
\qquad
\text{and}
\qquad
\begin{tikzcd}
\sd(\cB)^{|b}_{|c}
\arrow{r}
\arrow{d}
&
\sd(\cB)
\arrow{d}{(\min,\max)}
\\
\pt
\arrow{r}[swap]{(b^\circ,c)}
&
\cB^\op \times \cB
\end{tikzcd}
 \]
for the indicated pullbacks.
\end{notation}

\begin{remark}
\Cref{notation.source.and.or.target.restricted.sd} is chosen so as to be suggestive e.g.\! of the pullback
\[ \begin{tikzcd}[ampersand replacement=\&]
\TwAr(\cB)^{|b}_{|c}
\arrow{r}
\arrow{d}
\&
\TwAr(\cB)
\arrow{d}{(s,t)}
\\
\pt
\arrow{r}[swap]{(b^\circ,c)}
\&
\cB^\op \times \cB
\end{tikzcd}
~.
\]
Indeed, the locally cocartesian fibration $\sd(\cB) \xra{(\min,\max)} \cB^\op \times \cB$ may be thought of as the unstraightening of the composite
\[ \begin{tikzcd}[column sep=1.5cm]
\cB^\op \times \cB
\arrow{r}[description]{\llax}
&
\llax(\cB^\op) \times \llax(\cB)
\simeq
\llax(\cB)^{1\op} \times \llax(\cB)
\arrow{r}{\ulhom_{\llax(\cB)}}
&
\Cat
\end{tikzcd}
~.
\]
(For a precise relationship between $\TwAr$ and $\sd$, see \Cref{lem.sd.P.localizes.onto.TwAr.P}.)
\end{remark}

\begin{observation}
For any $b \in \cB$, the composite functor
\[
\max
:
\sd(\cB)^{|b}
\longra
\sd(\cB)
\xra{\max}
\cB
\]
is a locally cocartesian fibration, whose locally cocartesian morphisms are precisely those that map to locally cocartesian morphisms in the locally cocartesian fibration $\sd(\cB) \xra{\max} \cB$.
\end{observation}

\begin{lemma}
\label{lemma.free.loc.cocart.fibn.on.constant.functor.to.a.poset}
Fix any $\cC \in \Cat$ and $b \in \cB$.
\begin{enumerate}

\item\label{the.free.loc.cocart.fibn.is.indeed.a.loc.cocart.fibn}

The composite functor
\[
\cC \times \sd(\cB)^{|b}
\xlongra{\pr}
\sd(\cB)^{|b}
\xra{\max}
\cB
\]
is a locally cocartesian fibration, whose locally cocartesian morphisms are those that project to equivalences in $\cC$ and to locally cocartesian morphisms in $\sd(\cB)^{|b}$ with respect to the locally cocartesian fibration $\sd(\cB)^{|b} \xra{\max} \cB$.

\item\label{the.free.loc.cocart.fibn.is.indeed.free}

The morphism
\[ \begin{tikzcd}[column sep=1.5cm]
\cC
\arrow{rr}
{\left( \id_\cC,\const_{([0] \xra{b} \cB)} \right)}
\arrow{rd}[sloped, swap]{\const_b}
&
&
\cC \times \sd(\cB)^{|b}
\arrow{ld}[sloped, swap]{\max \circ \pr}
\\
&
\cB
\end{tikzcd} \]
in $\Cat_{/\cB}$ witnesses its target as the free locally cocartesian fibration on its source: for any object $(\cE \da \cB) \in \loc.\coCart_\cB$, restriction defines an equivalence
\[
\Fun( \cC , \cE_{b} )
\simeq
\ulhom_{\Cat_{/\cB}}(\cC,\cE)
\xlongla{\sim}
\ulhom_{\loc.\coCart_\cB} ( \cC \times \sd(\cB)^{|b} , \cE )
~.
\]

\end{enumerate}
\end{lemma}

\begin{observation}
\label{obs.loc.cocart.fibns.to.rcone}
Let us write
\[
\cB
\xlonghookra{i}
\cB^\rcone
\xlongra{p}
\pt^\rcone
=
[1]
\]
for the evident functors. Then, we have a canonical pullback square
\begin{equation}
\label{pullback.for.loc.coCart.over.rcone}
\begin{tikzcd}
\loc.\coCart_{\cB^\rcone}
\arrow{r}{i^*}
\arrow{d}[swap]{p \circ (-)}
&
\loc.\coCart_\cB
\arrow{d}{\fgt}
\\
\coCart_{[1]}
\arrow{r}[swap]{0^*}
&
\Cat
\end{tikzcd}
\end{equation}
among $(\infty,2)$-categories. (In particular, a locally cocartesian fibration $\cE \ra \cB^\rcone$ is equivalent data to a locally cocartesian fibration $\cE_{|\cB} \ra \cB$ along with a functor $\cE_{|\cB} \ra \cE_\infty$.) To see this, observe first the right adjoint
\[ \begin{tikzcd}[column sep=1.5cm, row sep=0cm]
\Cat_{/[1]}
\arrow[transform canvas={yshift=0.9ex}]{r}{0^*}
\arrow[leftarrow, dashed, transform canvas={yshift=-0.9ex}]{r}[yshift=-0.2ex]{\bot}[swap]{(-)^\rcone}
&
\Cat
\end{tikzcd}~. \]
This implies that we have a pullback square
\[ \begin{tikzcd}
\Cat_{/\cB^\rcone}
\arrow{r}{i^*}
\arrow{d}[swap]{p \circ (-)}
&
\Cat_{/\cB}
\arrow{d}{\fgt}
\\
\Cat_{/[1]}
\arrow{r}[swap]{0^*}
&
\Cat
\end{tikzcd}
\]
among $(\infty,2)$-categories, which it is easy to check restricts to give the pullback square \Cref{pullback.for.loc.coCart.over.rcone}.
\end{observation}

\begin{proof}[Proof of \Cref{lemma.free.loc.cocart.fibn.on.constant.functor.to.a.poset}]
Part \Cref{the.free.loc.cocart.fibn.is.indeed.a.loc.cocart.fibn} is clear.

We first prove part \Cref{the.free.loc.cocart.fibn.is.indeed.free} in the case that $\cB = [n]$ and $b = 0$, by induction on $n$. The case that $n=0$ is a tautology, so we may assume that $n \geq 1$. For the inductive step, note the evident equivalence
\[
\sd([n-1])^{|0} \times [1]
\xlongra{\sim}
\sd([n])^{|0}
\]
in $\Cat_{/[1]}$ that selects the natural transformation carrying an object $([m] \hookra [n-1]) \in \sd([n-1])^{|0}$ to the morphism
\[ \begin{tikzcd}
{[m]}
\arrow[hook]{rr}
\arrow[hook]{rd}
&
&
{[m]^\rcone}
\\
&
{[n-1]^\rcone}
\arrow[hookleftarrow]{ru}
\end{tikzcd} \]
in $\sd([n])^{|0}$. 
Therefore, restriction along the inclusion $\{0\} \hookrightarrow [1]$ determines an equivalence
\[
\ulhom_{\coCart_{[1]}}( \cC \times \sd([n])^{|0} , \cE )
\xlongra{\sim}
\ulhom_{\Cat}( \cC \times \sd([n-1])^{|0} , \cE_{|[n-1]} )
~.
\]
It follows from \Cref{obs.loc.cocart.fibns.to.rcone} that restriction along the inclusion $[n-1] \hookrightarrow [n-1]^{\lcone} = [n]$ determines an equivalence
\[
\ulhom_{\loc.\coCart_{[n]}} ( \cC \times \sd([n])^{|0} , \cE )
\xlongra{\sim}
\ulhom_{\loc.\coCart_{[n-1]}} ( \cC \times \sd([n-1])^{|0} , \cE_{|[n-1]} )
~.
\]
The assertion now follows by induction on $n$.

We now prove part \Cref{the.free.loc.cocart.fibn.is.indeed.free} in the general case.
We may clearly assume that $b \in \cB$ is initial.
Using this assumption, the fully faithful functor 
\begin{equation}
\label{e15}
\sd(\cB)^{|b}
\xlonghookra{\ff}
\sd(\cB)
\end{equation}
is a right adjoint, with left adjoint given at the level of $\bDelta_{/\cB}$ by taking a functor $[n] \ra \cB$ to the functor $[n]^\lcone \cong [n+1] \ra \cB$ carrying the cone point to $b \in \cB$ (using \Cref{t1}). In particular, the fully faithful functor \Cref{e15} is final.
Therefore, we obtain a canonical equivalence
\[
\colim_{([n]\da \cB) \in \sd(\cB)^{|b}} \sd([n])
\xlongra{\sim}
\colim_{([n]\da \cB) \in \sd(\cB)} \sd([n])
=:
\sd(\cB)
\]
in $\Cat_{/\cB}$, in which the colimits can be computed in simplicial spaces.
Then, using that pullbacks in simplicial spaces commute with colimits, we obtain a canonical equivalence
\[
\colim_{([n]\da \cB) \in \sd(\cB)^{|b}}
\sd([n])^{|0}
\longra
\sd(\cB)^{|b}
\]
in $\Cat_{/\cB}$.
Using that $\Cat \xra{\cC \times - } \Cat$ preserves colimits, we then obtain a composite equivalence
\[
\colim_{([n]\da \cB) \in \sd(\cB)^{|b}}
\cC \times \sd([n])^{|0} 
\xlongra{\sim}
\cC \times \colim_{([n]\da \cB) \in \sd(\cB)^{|b}}
\sd([n])^{|0} 
\xlongra{\sim}
\cC \times \sd(\cB)^{|b}
\]
in $\Cat_{/\cB}$.
Therefore, we have an equivalence
\[
\ulhom_{\Cat_{/\cB}}(\cC \times \sd(\cB)^{|b}, \cE) 
\simeq
\lim_{(([n]\da \cB) \in \sd(\cB)^{|b})^{\op}} \ulhom_{\Cat_{/[n]}}(\cC \times \sd([n])^{|0}, \cE_{|[n]} )
\]
in $\Cat$.
Observe that this equivalence restricts to give an equivalence
\[
\ulhom_{\loc.\coCart_\cB}(\cC \times \sd(\cB)^{|b}, \cE) 
\simeq
\lim_{(([n]\da \cB) \in \sd(\cB)^{|b})^{\op}} 
\ulhom_{\loc.\coCart_{[n]}}(\cC \times \sd([n])^{|0}, \cE_{|[n]} )
\]
in $\Cat$.
The assertion now follows from the case that $\cB = [n]$ and $b=0$.
\end{proof}

\begin{proof}[Proof of \Cref{lem.strictification}]
The canonical commutativity of the triangle \Cref{e24} is an instance of \Cref{t17}.

Next, for any object $\varphi \in \sd(\cB)$, by \Cref{t18} we have the identification
\begin{equation}
\label{e25}
\begin{tikzcd}[column sep=1.5cm]
\sd(\cB)_{\varphi/}
\arrow{rr}[swap]{\sim}{( \varphi \xlongra{\alpha} \psi ) \longmapsto (\alpha_L , \psi')}
\arrow{rd}[sloped, swap]{t}
&
&
\sd(\cB)_{\varphi/\isomax} \times \sd(\cB)^{| \max(\varphi)}
\arrow{ld}
\\
&
\sd(\cB)
\end{tikzcd}
\end{equation}
in $\Cat_{/\sd(\cB)}$, where 

\begin{itemize}

\item as in the statement of the result, writing $([k] \xra{\psi} \cB) \in \sd(\cB)$ we denote by
\[
\psi'
:=
\left(
[k]_{\max(\alpha_R)/}
\longhookra
[k]
\xlonghookra{\psi}
\cB
\right)
\in
\sd(\cB)
\]
the restriction of $\psi$ to the undercategory of $\max(\alpha_R) \in [k]$, and

\item the lower right morphism is given by concatenation.

\end{itemize}
We now identify the fiber of $(\Strict(\cE) \da \sd(\cB)) \in \Cat_{/\sd(\cB)}$ over $\varphi \in \sd(\cB)$ through the composite equivalence
\begin{equation}
\label{identify.fiber.of.strictification.as.functor.cat.from.isomax.undercat}
\begin{aligned}
\Strict(\cE)_\varphi
&\simeq
\ulhom_{\loc.\coCart_\cB} \left( \sd(\cB)_{\varphi/} , \cE \right)
\\
&\simeq
\ulhom_{\loc.\coCart_\cB} \left( \sd(\cB)_{\varphi/\isomax} \times \sd(\cB)^{| \max(\varphi)} , \cE \right)
\\
&\simeq
\Fun \left( \sd(\cB)_{\varphi/\isomax} , \cE_{\max(\varphi)} \right)
~,
\end{aligned}
\end{equation}
in which the individual equivalences respectively follow from \Cref{t16}, diagram \Cref{e25}, and \Cref{lemma.free.loc.cocart.fibn.on.constant.functor.to.a.poset}. This establishes that the fibers are indeed as claimed.

Next, by \Cref{t16} the locally cocartesian monodromy functor over the morphism \Cref{morphism.in.sd.P} in $\sd(\cB)$ is the functor
\[
\ulhom_{\loc.\coCart_\cB} \left( \sd(\cB)_{\varphi/} , \cE \right)
\xra{(\alpha^*)^*}
\ulhom_{\loc.\coCart_\cB} \left( \sd(\cB)_{\psi/} , \cE \right)
~.
\]
Given a functor \Cref{arbitrary.functor.from.isomax.undercat.to.corresponding.fiber} and writing  $\w{F} \in \ulhom_{\loc.\coCart_\cB} \left( \sd(\cB)_{\varphi/} , \cE \right)$ for its corresponding object under the equivalence \Cref{identify.fiber.of.strictification.as.functor.cat.from.isomax.undercat}, 
we see that its image under the cocartesian monodromy functor \Cref{cocart.mdrmy.functor.in.strictification} is given by the factorization
\[ \begin{tikzcd}
&
\sd(\cB)_{\varphi/\isomax}
\arrow{r}{F}
\arrow[hook]{d}
&
\cE_{\max(\varphi)}
\arrow[hook]{d}
\\
\sd(\cB)_{\psi/}
\arrow{r}{\alpha^*}
&
\sd(\cB)_{\varphi/}
\arrow{r}[swap]{\w{F}}
&
\cE
\\
\sd(\cB)_{\psi/\isomax}
\arrow[dashed]{rr}
\arrow[hook]{u}
&
&
\cE_{\max(\psi)}
\arrow[hook]{u}
\end{tikzcd}~. \]
This establishes that the monodromy functors are indeed as claimed.

We now turn to the final statement of the claim.
Note that the functor $\Strict$ is adjunct to the composite functor
\begin{equation}
\label{e60}
\begin{tikzcd}[row sep=0cm]
\sd(\cB)
\arrow{r}
&
\LMod_{\llax.\cB}^{\op}
\arrow[hook]{r}{\ff}
&
\Fun \left( \LMod_{\llax.\cB} , \Cat \right)
\\
\rotatebox{90}{$\in$}
&
\rotatebox{90}{$\in$}
\\
\varphi
\arrow[maps to]{r}
&
\left( \sd(\cB)_{/\varphi} \longra \sd(\cB) \xra{\max} \cB \right)
\end{tikzcd}
~.
\end{equation}
In the case that $\cB$ is an ordinary category (with no nontrivial retracts), the $\infty$-category $\bDelta_{/\cB}$ is an ordinary category.
Using \Cref{t3}, we see that the subcategory $\sd(\cB) \subset \bDelta_{/\cB}$ is an ordinary category.
It follows that for any $\varphi , \psi \in \sd(\cB)$, the space of morphisms
\[
\hom_{\LMod_{\llax.\cB}}
\left(
\sd(\cB)_{/\psi} 
,
\sd(\cB)_{/\varphi}
\right)
\]
is discrete.  
It follows that the functor \Cref{e60} is uniquely determined by its values on objects and morphisms.
\end{proof}

\section{Some $(\infty,2)$-category theory}
\label{section.inftytwocats}

In this section, we establish some aspects of $(\infty,2)$-category theory. (In the main body of the work, we only refer to applications thereof that are recorded in \Cref{section.lax.actions.and.limits}.)

A large part of \S\S\ref{subsection.basics.of.inftytwocats}-\ref{subsection.Adj} is adapted from \cite[Appendix A]{GR}. As explained in \cite[Chapter 10, \S 0.4]{GR}, some of the results there rely on results whose proofs do not appear in the literature at the time of writing. Here, we give a logically complete account of the material that we use; in particular, this section does not logically depend on \cite{GR} in any way. Nevertheless, we provide references where appropriate.

This section is organized as follows.
\begin{itemize}

\item[\Cref{subsection.basics.of.inftytwocats}:] We introduce some basic notions in $(\infty,2)$-category theory.

\item[\Cref{subsection.fibns.for.inftytwocats}:] We define various notions of fibrations among $(\infty,2)$-categories.

\item[\Cref{subsection.un.straightening}:] We give (a lax version of) un/straightening for $(\infty,2)$-categories.

\item[\Cref{subsection.cartesian.yoga}:] We study parametrized versions of un/straightening for $(\infty,2)$-categories.

\item[\Cref{subsection.Adj}:] We study adjunctions in $(\infty,2)$-categories, and prove parametrized versions of the mate correspondence.

\item[\Cref{subsection.lax.limits.over.twocats}:] We define lax limits in $\Cat$ over $(\infty,2)$-categories, and give an alternative $(\infty,1)$-categorical description in the case that the base is the left-laxification of an $(\infty,1)$-category.

\end{itemize}

\subsection{Basic notions in $(\infty,2)$-category theory}
\label{subsection.basics.of.inftytwocats}

In this subsection, we discuss (strict and) lax versions of functors and natural transformations among $(\infty,2)$-categories. Relatedly, we discuss various laxifications of an $(\infty,2)$-category, and we give an explicit identification in one important case (\Cref{prop.llax.of.brax.n}). We also introduce the class of thin $(\infty,2)$-categories (\Cref{defn.thin.twocat}), for which homotopy-coherence data is vacuous (see \Cref{obs.thinness.is.great}) -- analogously to the class of posets among $(\infty,1)$-categories.

\begin{definition}[\cite{Barwick-thesis}]
An \bit{$(\infty,2)$-category} is a complete Segal $(\infty,1)$-category whose $0\th$ $\infty$-category is an $\infty$-groupoid.\footnote{Said differently, it is a two-fold complete Segal space.} These assemble into a full subcategory
\[
\iota_1 2\Cat
\subset
\Fun ( \bDelta^\op , \Cat )
~,
\]
whose morphisms we refer to as \bit{functors} (or occasionally \bit{strict functors} in order to contrast with the notions introduced in \Cref{defn.lax.fctrs}). We consider $\infty$-categories (which we may refer to as $(\infty,1)$-categories for emphasis) as forming a full subcategory of $\iota_1 2\Cat$ according to the pullback
\[ \begin{tikzcd}
\Cat
\arrow[hook]{r}{\ff}
\arrow[hook]{d}[swap]{\ff}
&
\iota_1 2\Cat
\arrow[hook]{d}{\ff}
\\
\Fun ( \bDelta^\op , \Spaces)
\arrow[hook]{r}[swap]{\ff}
&
\Fun ( \bDelta^\op , \Cat )
\end{tikzcd}
~.
\]
We refer to an $(\infty,2)$-category as a \bit{2-category} if its hom-$\infty$-categories are ordinary categories.\footnote{Note that these are most naturally modeled by the classical notion of a bicategory.}
\end{definition}

\begin{notation}
\label{d4}
Presenting $(\infty,1)$-categories as complete Segal spaces gives a fully faithful functor $\iota_1 \Cat \hookrightarrow \Fun(\bDelta^{\op},\Spaces)$.
We write
\[
\bDelta \times \bDelta
\xlongra{\theta}
\iota_1 2\Cat
\]
for the functor corepresenting the resulting fully faithful functor 
\[
\iota_1 2\Cat 
\longhookra
\Fun ( \bDelta^\op , \Cat )
\longhookra
\Fun \left( \bDelta^\op , \Fun(\bDelta^{\op} , \Spaces) \right)
\simeq
\Fun( \bDelta^\op \times  \bDelta^{\op} , \Spaces)
~.
\]
(An explicit description of $\theta$ is given in the discussion preceding \cite[Lemma 14.5]{BSP-unicity}.)
\end{notation}

\begin{notation}
Given an $(\infty,2)$-category $\cC \in \iota_1 2\Cat$ and objects $c,c' \in \cC$, we write $\hom_\cC(c,c') \in \Cat$ for the $\infty$-category of morphisms from $c$ to $c'$.
\end{notation}

\begin{notation}
The $\infty$-category $\iota_1 2\Cat$ is cartesian closed \cite{Rezk-theta,BSP-unicity}, and we denote its internal hom by $\Fun(-,-)$. By \cite{GH-enr,Haug-rect}, it follows that $\iota_1 2\Cat$ is the underlying $(\infty,1)$-category of an $(\infty,3)$-category. We write $2\Cat$ for its underlying $(\infty,2)$-category.\footnote{That is, we do not make any use of the full $(\infty,3)$-category of $(\infty,2)$-categories. In particular, we only ever consider $(\infty,2)$-categorical laxness in $2\Cat$, i.e.\! laxness in 1-morphisms but not 2-morphisms.}
\end{notation}

\begin{definition}
Given an $(\infty,2)$-category $\bDelta^\op \xra{\cC} \Cat$, its \bit{1-opposite} and \bit{2-opposite} are the $(\infty,2)$-categories
\[
\cC^{1\op}
:
\bDelta^\op
\xra[\sim]{\rev}
\bDelta^\op
\xlongra{\cC}
\Cat
\qquad
\text{and}
\qquad
\cC^{2\op}
:
\bDelta^\op
\xlongra{\cC}
\Cat
\xra[\sim]{(-)^\op}
\Cat
\]
given respectively by pre- and postcomposing with the indicated involutions.\footnote{Here, $\rev$ denotes the involution given by reversing linear orders.} These two operations define commuting involutions of $\iota_1 2\Cat$, and we write
\[
(-)^{1\&2\op}
:=
((-)^{1\op})^{2\op}
\simeq
((-)^{2\op})^{1\op}
\]
for their composite.
\end{definition}

\begin{notation}
\label{notn.oint.for.two.cats}
We denote by
\[
(-)^\oi
:
\iota_1 2\Cat
\xlonghookra{\ff}
\Fun ( \bDelta^\op , \Cat )
\xlongra{\sim}
\coCart_{\bDelta^\op}
\]
the composite functor carrying an $(\infty,2)$-category to its corresponding cocartesian fibration over $\bDelta^\op$.
\end{notation}

\begin{definition}
A morphism in $\bDelta$ is called \bit{convex} if its image is convex.\footnote{In \cite{GR}, such morphisms are referred to as ``idle''.} A morphism in $\bDelta$ is called \bit{inert} if it is convex and injective (or equivalently conservative). We use the same terms for corresponding morphisms in $\bDelta^\op$.
\end{definition}

\begin{definition}[{\cite[Chapter 10, \S 3.1.3]{GR}}]
\label{defn.lax.fctrs}
Given $(\infty,2)$-categories $\cC,\cD \in \iota_1 2\Cat$, a \bit{non-unital right-lax functor} from $\cC$ to $\cD$ is a morphism
\[ \begin{tikzcd}
\cC^\oi
\arrow{rr}
\arrow{rd}
&
&
\cD^\oi
\arrow{ld}
\\
&
\bDelta^\op
\end{tikzcd} \]
in $\Cat_{\cocart/\bDelta^\op}$ that preserves inert-cocartesian morphisms. It is called a (\bit{unital}) \bit{right-lax functor} if it preserves convex-cocartesian morphisms. These respectively define the morphisms in $\infty$-categories that we denote by
\[
\iota_1 2\Cat_{\nonu\rlax}
\qquad
\text{and}
\qquad
\iota_1 2\Cat_{\rlax}
~,
\]
so that we have monomorphisms
\[
\iota_1 2\Cat
\longhookra
\iota_1 2\Cat_\rlax
\longhookra
\iota_1 2\Cat_{\nonu \rlax}
~.
\]
We define a (\bit{non-unital} or \bit{unital}) \bit{left-lax functor} from $\cC$ to $\cD$ to be a (respectively non-unital or unital) right-lax functor from $\cC^{2\op}$ to $\cD^{2\op}$, and we use the evident corresponding notation. Given $(\infty,2)$-categories $\cC,\cD \in 2\Cat$, we write
\[ \begin{tikzcd}[column sep=0.6cm]
\cC
\arrow[squiggly]{r}
&
\cD
\end{tikzcd} \]
to denote a (possibly) lax functor (be it non-unital or unital, right- or left-lax).\footnote{Outside of \Cref{section.inftytwocats}, we explicitly label our arrows according to the handedness of the laxness.}
\end{definition}

\begin{remark}
One can similarly define left-lax functors in terms of cartesian fibrations over $\bDelta$. We systematically privilege right-lax functors in our treatment here (so that for instance we do not introduce a cartesian variant of \Cref{notn.oint.for.two.cats}).
\end{remark}

\begin{remark}
Informally, a lax functor is one that only laxly respects composition of 1-morphisms. More specifically, given a pair of composable 1-morphisms $c_0 \xra{\varphi} c_1 \xra{\psi} c_2$, a right-lax functor $F$ determines a 2-morphism $F(\psi) \circ F(\varphi) \ra F(\psi \circ \varphi)$ while a left-lax functor $G$ determines a 2-morphism $G(\psi \circ \varphi) \ra G(\psi) \circ G(\varphi)$. A lax functor is strict precisely when these 2-morphisms are all invertible.
\end{remark}

\begin{remark}
Our primary interest will be in \textit{unital} lax functors. While non-unital lax functors are of independent interest, for our purposes they serve as an auxiliary notion (see Definitions \ref{defn.nonu.rlaxification} \and \ref{defn.unital.rlaxification}).
\end{remark}

\begin{definition}[{\cite[Chapter 10, \S 3.2.7]{GR}}]
Fix $(\infty,2)$-categories $\cC,\cD \in 2\Cat$.
\begin{enumerate}

\item

A \bit{right-lax natural transformation} between right-lax functors from $\cC$ to $\cD$ is a right-lax functor
\begin{equation}
\label{the.underlying.rlax.fctr.of.a.rlax.or.llax.nat.trans}
\begin{tikzcd}[column sep=0.6cm]
\cC \times {[1]}
\arrow[squiggly]{r}
&
\cD
\end{tikzcd}
\end{equation}
that is strict on pairs of composable 1-morphisms of the form $(c_0,0) \xra{(\varphi,\id_0)} (c_1,0) \xra{(\id_{c_1},\iota)} (c_1,1)$.\footnote{That is, its precomposition with the corresponding functor $[2] \ra \cC \times [1]$ defines a strict functor $[2] \ra \cD$.}

\item

A \bit{left-lax natural transformation} between right-lax functors from $\cC$ to $\cD$ is a right-lax functor \Cref{the.underlying.rlax.fctr.of.a.rlax.or.llax.nat.trans} that is strict on pairs of composable 1-morphisms of the form $(c_0,0) \xra{(\id_{c_0},\iota)} (c_0,1) \xra{(\varphi,\id_1)} (c_1,1)$.

\end{enumerate}
Dually, a \bit{right-} (resp.\! \bit{left-})\bit{lax natural transformation} between left-lax functors from $\cC$ to $\cD$ is a right- (resp.\! left-)lax natural transformation between the corresponding right-lax functors from $\cC^{2\op}$ to $\cD^{2\op}$.\footnote{Evidently, these can also be expressed as left-lax functors $\cC \times {[1]} \laxra \cD$.} As in \Cref{defn.lax.fctrs}, given (right- or left-lax) functors $F$ and $G$, we simply write
\[
\begin{tikzcd}[column sep=0.6cm]
F
\arrow[squiggly]{r}
&
G
\end{tikzcd}
\]
to denote a lax natural transformation between them (regardless of its handedness).
\end{definition}

\begin{remark}
\label{rmk.rlax.nat.trans.has.lax.squares}
Given left- or right-lax functors $F$ and $G$ from $\cC$ to $\cD$, a right-lax natural transformation from $F$ to $G$ specifies the data, for each 1-morphism $c_0 \ra c_1$ in $\cC$, of a lax-commutative square
\[
\begin{tikzcd}
F(c_0)
\arrow{r}
\arrow{d}
\arrow{rd}
&
F(c_1)
\arrow{d}
\\
G(c_0)
\arrow{r}[description, xshift=0.4cm, yshift=0.8cm]{\rotatebox{45}{$\Rightarrow$}}
&
G(c_1)
\end{tikzcd}
\qquad
=:
\qquad
\begin{tikzcd}
F(c_0)
\arrow{r}
\arrow{d}
&
F(c_1)
\arrow{d}
\\
G(c_0)
\arrow{r}[yshift=0.4cm]{\rotatebox{45}{$\Rightarrow$}}
&
G(c_1)
\end{tikzcd}
\qquad
:=
\qquad
\begin{tikzcd}
F(c_0)
\arrow{r}
\arrow{d}
\arrow{rd}
&
F(c_1)
\arrow{d}
\\
G(c_0)
\arrow{r}[description, yshift=0.4cm, xshift=-0.4cm]{\rotatebox{45}{$\Rightarrow$}}
&
G(c_1)
\end{tikzcd}
\]
in $\cD$ -- the square on the left (resp.\! right) applying in the case that $F$ and $G$ are left-lax (resp.\! right-lax).
\end{remark}

\begin{remark}
The paper \cite{GHL-gray-lax-fctrs} studies lax functors and lax natural transformations in a more combinatorial model of $(\infty,2)$-categories. We expect that these notions agree.
\end{remark}

\begin{definition}
\label{defn.thin.twocat}
An $(\infty,2)$-category is called a \bit{thin 2-category} (or simply \bit{thin}) if its hom-$\infty$-categories lie in $\Poset \subset \Cat$ and its endomorphism $\infty$-categories are all equivalent to $\pt \in \Cat$.\footnote{This notion is strictly stronger than that of gauntness, which merely requires that all \textit{invertible} $k$-morphisms are identities. It is also strictly stronger than the requirement that every $k$-morphism has a contractible space of endomorphisms.}
\end{definition}

\begin{observation}
\label{obs.thinness.is.great}
We collect the following apparent facts about thin 2-categories, which we use without further reference.
\begin{enumerate}

\item

Thin 2-categories form a full subcategory of the $(\infty,2)$-category (in fact strict 2-category) of strict 2-categories. Under this identification, non-unital right-lax functors correspond to lax functors, while unital right-lax functors correspond to normal lax functors. Moreover, right-lax natural transformations correspond to lax natural transformations.

\item

Given a thin 2-category $\cD$ and an $(\infty,2)$-category $\cC$, a (possibly left- or right-lax) functor from $\cC$ to $\cD$ is uniquely determined by its values on 1-morphisms, i.e.\! by the morphism of sets
\[
\pi_0 \hom_{\iota_1 2\Cat} ( [1] , \cC )
\longra
\hom_{\iota_1 2\Cat} ( [1] , \cD)
~.
\]
Hence, given (possibly left- or right-lax) functors from $\cC$ to $\cD$, a (possibly left- or right-lax) natural transformation between them is uniquely determined by its values on objects, i.e.\! by the morphism of sets
\[
\pi_0 \iota_0 \cC
\longra
\hom_{\iota_1 2\Cat} ( [1] , \cD)
~.
\]

\item

Given a functor $\cC \ra \cD$ in $2\Cat$ such that $\cD$ is thin, for every object $d \in \cD$ the functor $\cC_d \ra \cC$ is fully faithful.

\item

Given a thin 2-category $\cD$, the forgetful functor 
\[
2\Cat_{/\cD}
\longra
2\Cat
\]
is 1-full (i.e.\! it is fully faithful on hom-$\infty$-categories).\footnote{Said differently, given $\cC_0,\cC_1 \in 2\Cat_{/\cD}$ it is merely a condition for a morphism $\cC_0 \ra \cC_1$ in $2\Cat$ to lie in $2\Cat_{/\cD}$.}

\end{enumerate}
\end{observation}

\begin{observation}
\label{t13}
Lax transformations can be composed in the following sense. 
Let $\cC , \cD \in 2\Cat$.
Consider the bisimplicial space 
\[
\bDelta^{\op}\times \bDelta^{\op}
\xra{\w{\Fun}_{\rlax}^{\rlax}(\cC,\cD)}
\Spaces
\]
that is the subfunctor of 
$
\hom_{\iota_1 2\Cat_{\rlax}}\left(
\cC \times \theta(-)
,
\cD
\right)
$
carrying each $([i]^\circ,[j]^\circ) \in \bDelta^\op \times \bDelta^\op$ (recall \Cref{d4}) to the subspace of those right-lax functors $\cC \times \theta([i],[j]) \laxra \cD$ that are strict on pairs of composable 1-morphisms of the form $(c_0,x) \ra (c_1,x) \ra (c_1,y)$, where $c_0 \ra c_1$ is a 1-morphism in $\cC$ and $x \ra y$ is a 1-morphism in $\theta([i]^\circ,[j]^\circ)$.

It is clear that if $\cC,\cD\in 2\Cat$ are thin, then the bisimplicial space $\w{\Fun}_{\rlax}^{\rlax}(\cC,\cD)$ presents a thin 2-category.\footnote{\label{footnote.Funrlaxrlax.always.a.two.cat}In fact, this bisimplicial space presents an $(\infty,2)$-category for arbitrary $\cC,\cD\in 2\Cat$; this follows from the Yoneda lemma for $(\infty,2)$-categories (see \cite{hinich-yoneda-enriched}) and Theorems \ref{t11} \and \ref{thm.unstraightening.yoga} below. (Since we do not need this fact, we do not give a detailed argument.)
}
\end{observation}

\begin{notation}
\label{d3}
Given $(\infty,2)$-categories $\cC,\cD \in 2\Cat$, we write
\[
\Fun_{\rlax}^{\rlax}(\cC,\cD)
\in
2\Cat
\]
for the $(\infty,2)$-category presented by the bisimplicial space $\w{\Fun}_{\rlax}^{\rlax}(\cC,\cD)$ of \Cref{t13}.\footnote{That is, $\Fun^\rlax_\rlax(\cC,\cD)$ denotes the $(\infty,2)$-category obtained by applying the left adjoint from bisimplicial spaces to $(\infty,2)$-categories, although this application is always vacuous by \Cref{footnote.Funrlaxrlax.always.a.two.cat} (and not just in case $\cC$ and $\cD$ are thin).}
\end{notation}

\begin{definition}[{\cite[Chapter 11, \S A.1]{GR}}]
\label{defn.nonu.rlaxification}
Given an object $(\cE \da \bDelta^\op) \in \coCart_{\bDelta^\op}$, we write
\begin{equation}
\label{diagram.defining.Freeact}
\begin{tikzcd}
\Freeact(\cE)
:=
&[-1cm]
\cE \underset{\bDelta^\op}{\times} \Ar^\act(\bDelta^\op)
\arrow{r}
\arrow{d}
&
\Ar^\act(\bDelta^\op)
\arrow{r}{t}
\arrow{d}{s}
&
\bDelta^\op
\\
&
\cE
\arrow{r}
&
\bDelta^\op
\end{tikzcd}
\end{equation}
for the indicated fiber product (where $\Ar^\act(\bDelta^\op) \subset \Ar(\bDelta^\op)$ denotes the full subcategory on the active morphisms). Noting that the functor $\Ar^\act(\bDelta^\op) \xra{t} \bDelta^\op$ is a cocartesian fibration, we find that the horizontal composite of diagram \Cref{diagram.defining.Freeact} defines a functor
\[
\coCart_{\bDelta^\op} \xra{\Freeact} \coCart_{\bDelta^\op}
~.
\]

In particular, given an $(\infty,2)$-category $\cC \in 2\Cat$, it is straightforward
to see that $\Freeact(\cC^\oi) \in \coCart_{\bDelta^\op}$ defines an $(\infty,2)$-category,\footnote{Namely, the straightening of $\Freeact(\cC^\oi)$ satisfies the Segal and completeness conditions, and its $\infty$-category of $0$-simplices is an $\infty$-groupoid.
The last claim follows from the fact that $\bDelta^{[0]/^{\act}} = \{[0]\}$, which gives an equivalence $\Freeact(\cC^\oi)_{[0]^\circ} \simeq \cC^\oi_{[0]^\circ}$.
The Segal condition follows from commutativity for each $[n] \in \bDelta$ of the square
\[
\begin{tikzcd}[ampersand replacement=\&]
\bDelta^{[n]/^{\act}}
\arrow{r}
\arrow{d}[swap]{\fgt}
\&
\Fun(\Span,\Cat)
\arrow{d}{\colim}
\\
\bDelta
\arrow[hook]{r}[swap]{\ff}
\&
\Cat
\end{tikzcd}
\]
in which the upper horizontal functor is given by pullback to the span $[n-1] \hookla \{n-1\} \hookra \{n-1 < n\}$ in $\bDelta_{/[n]}$.
The completeness condition is evident from that of $\cC^{\oi}$.
} 
which we denote by $\rlax^\nonu(\cC) \in 2\Cat$ (so that $\rlax^\nonu(\cC)^\oi \simeq \Freeact(\cC^\oi)$) and refer to as its \bit{non-unital right-laxification}. Altogether, this defines a functor
\[
\iota_1 2\Cat
\xra{\rlax^\nonu(-)}
\iota_1 2\Cat
~.
\]
\end{definition}

\begin{observation}
\label{obs.universal.property.of.Freeact}
The functor $\Freeact$ has the following universal property. First of all, the functors $s,t \in \hom_\Cat( \Ar^\act(\bDelta^\op) , \bDelta^\op )$ admit a common section $\bDelta^\op \xra{\id_{(-)}} \Ar^\act(\bDelta^\op)$, which induces for any $\cE \in \coCart_{\bDelta^\op}$ a natural morphism
\[
\cE
\longra
\Freeact(\cE)
\]
in $\Cat_{/\bDelta^\op}$. Then, by \cite[Proposition 2.18]{AMR-fact}, 
for any $\cF \in \coCart_{\bDelta^\op}$, restriction therealong defines a monomorphism
\[
\hom_{\coCart_{\bDelta^\op}} ( \Freeact(\cE) , \cF)
\longhookra
\hom_{\Cat_{/\bDelta^\op}} ( \cE , \cF )
\]
in $\Spaces$ whose image consists of those morphisms that preserve cocartesian lifts of inert morphisms in $\bDelta^\op$.
\end{observation}

\begin{observation}[{\cite[Chapter 11, Theorem A.1.5]{GR}}]
By \Cref{obs.universal.property.of.Freeact}, non-unital right-laxification defines a left adjoint
\[ \begin{tikzcd}[column sep=2cm]
\iota_1 2\Cat_{\nonu\rlax}
\arrow[dashed, transform canvas={yshift=0.9ex}]{r}{\rlax^\nonu(-)}
\arrow[hookleftarrow, transform canvas={yshift=-0.9ex}]{r}[yshift=-0.2ex]{\bot}
&
\iota_1 2\Cat
\end{tikzcd}
\]
to the inclusion; in particular, for any $(\infty,2)$-category $\cC \in 2\Cat$, we have a universal non-unital right-lax functor
\[
\begin{tikzcd}[column sep=0.6cm]
\cC
\arrow[squiggly]{r}
&
\rlax^\nonu(\cC)
\end{tikzcd}
~.
\]
We use this fact without further comment.
\end{observation}

\begin{observation}
\label{obs.get.quasi.unit.two.mor}
Noting the identification
\[
\rlax^\nonu(\pt)^\oi
\simeq
\Ar^\act(\bDelta^\op)
\xlongra{t}
\bDelta^\op
~,
\]
we see that $\rlax^\nonu(\pt)$ has a single object $\ast := ( [0]^\circ \xra{\sim} [0]^\circ ) \in \rlax^\nonu(\pt)$ as well as a distinguished 2-morphism
\[ \begin{tikzcd}[column sep=1.5cm]
\ast
\arrow[bend left=40]{r}{\id_\ast = ( [0]^\circ \ra [1]^\circ )}[yshift=-0.7cm]{\Downarrow}
\arrow[bend right=40]{r}[swap]{e_\ast := ( [1]^\circ \xra{\sim} [1]^\circ ) }
&
\ast
\end{tikzcd}
~.
\]
By the functoriality of $\rlax^\nonu(-)$, for any object $c \in \cC \in 2\Cat$ we obtain a canonical 2-morphism $\id_c \ra e_c$ in $\rlax^\nonu(\cC)$.
\end{observation}

\begin{definition}
\label{defn.unital.rlaxification}
We refer to the 2-morphism $\id_c \ra e_c$ in $\rlax^\nonu(\cC)$ of \Cref{obs.get.quasi.unit.two.mor} as the \bit{quasi-unit 2-morphism} corresponding to the object $c \in \cC$. Inverting these determines an $(\infty,2)$-category
\[
\rlax(\cC)
\in
2\Cat
~,
\]
which we refer to as the (\bit{unital}) \bit{right-laxification} of $\cC$. This construction defines an endofunctor
\[
\iota_1 2\Cat
\xra{\rlax(-)}
\iota_1 2\Cat
\]
equipped with a natural epimorphism from $\rlax^\nonu(-)$.
\end{definition}

\begin{observation}
\label{t10}
A non-unital right-lax functor $\cC \overset{F}{\laxra} \cD$ between $(\infty,2)$-categories
is unital if and only if it carries quasi-unit 2-morphisms to invertible 2-morphisms.
Indeed, let $\cC^{\oi} \xra{F^{\oi}} \cD^{\oi}$ be the corresponding morphism in $\Cat_{/\bDelta^\op}$. By definition, $F^\oi$ preserves inert-cocartesian morphisms.
The functor $F$ is unital if and only if $F^\oi$ additionally preserves surjective-cocartesian morphisms.
Because both $\cC^{\oi}$ and $\cD^{\oi}$ satisfy the Segal condition, $F^{\oi}$ preserves surjective-cocartesian morphisms if and only if it preserves cocartesian morphisms over the morphism $( [1] \to [0])^\circ$.
And by definition, $F$ carries quasi-unit 2-morphisms to invertible 2-morphisms if and only if $F^{\oi}$ preserves such cocartesian morphisms.
\end{observation}

\begin{observation}
By \Cref{t10}, the
composite
\[
\begin{tikzcd}[column sep=0.6cm]
\cC
\arrow[squiggly]{r}
&
\rlax^\nonu(\cC)
\arrow{r}
&
\rlax(\cC)
\end{tikzcd}
\]
is a (unital) right-lax functor. 
Moreover, it is the universal right-lax functor from $\cC$. 
We use this fact without further comment.
\end{observation}

\begin{observation}
\label{obs.rlax.commutes.with.oneop}
For any $\cC, \cD \in 2\Cat$, pullback along the involution $\bDelta^\op \xra[\sim]{\rev} \bDelta^\op$ defines equivalences
\[
\hom_{\iota_1 2\Cat_{\nonu\rlax}}(\cC , \cD)
\xlongra{\sim}
\hom_{\iota_1 2\Cat_{\nonu\rlax}}(\cC^{1\op},\cD^{1\op})
\]
and
\[
\hom_{\iota_1 2\Cat_{\rlax}}(\cC , \cD)
\xlongra{\sim}
\hom_{\iota_1 2\Cat_{\rlax}}(\cC^{1\op},\cD^{1\op})
\]
in $\Spaces$. Therefore, we have canonical equivalences
\[
\rlax^\nonu(\cC^{1\op})
\simeq
\rlax^\nonu(\cC)^{1\op}
\qquad
\text{and}
\qquad
\rlax(\cC^{1\op})
\simeq
\rlax(\cC)^{1\op}
~.
\]
\end{observation}

\begin{observation}
\label{obs.laxification.of.one.cats}
In the case that $\cC \in \Cat \subset 2\Cat$ is an $\infty$-category, we have a natural equivalence
\[
\colim_{([n] \da \cC) \in \bDelta_{/\cC}} \rlax([n])
\xlongra{\sim}
\rlax(\cC)
~;
\]
by the universal property of $\rlax(\cC) \in \iota_1 2\Cat$, this follows from the equivalence
\[
\colim_{([n] \da \cC) \in \bDelta_{/\cC}} [n]^\oi
\xlongra{\sim}
\cC^\oi
\]
in which the colimit can be computed either in $\coCart_{\bDelta^\op}$ or in $\Cat_{/\bDelta^\op}$ by \Cref{obs.unstraightening.preserves.colimits}.
\end{observation}

\begin{definition}
Given an $(\infty,2)$-category $\cC \in 2\Cat$, we define its \bit{non-unital left-laxification} and its (\bit{unital}) \bit{left-laxification} respectively as the $(\infty,2)$-categories
\[
\llax^\nonu(\cC)
:=
\rlax^\nonu(\cC^{2\op})^{2\op}
\qquad
\text{and}
\qquad
\llax(\cC)
:=
\rlax(\cC^{2\op})^{2\op}
~.
\]
\end{definition}

\begin{prop}
\label{prop.llax.of.brax.n}
The left-laxification $\llax([n]) \in 2\Cat$ is the thin 2-category that is characterized as follows: its objects are those of $[n]$, and for any $i,j \in [n]$ the poset $\hom_{\llax([n])}(i,j)$ is that of strictly increasing sequences
\[
i
<
k_1
<
\cdots
<
k_l
<
j
\]
in $[n]$ (for some $l \geq 0$) from $i$ to $j$ (ordered by inclusion), with composition given by concatenation.\footnote{This $(\infty,2)$-category can be presented as the simplicially-enriched category $\fC (\Delta^n)$ (where $\fC$ denotes the left adjoint of the homotopy-coherent nerve functor to simplicial sets), but thought of as enriched in $\infty$-categories (via the Joyal model structure) rather than in spaces (via the Kan--Quillen model structure).}
\end{prop}

\begin{proof}
We establish the corresponding description of $\rlax([n]) \simeq \llax([n]^{2\op})^{2\op} \simeq \llax([n])^{2\op}$.

We begin by noting the identification
\[
[n]^\oi
\simeq
\left(
(\bDelta_{/[n]})^\op
\xra{\fgt}
\bDelta^\op
\right)
\]
in $\Cat_{/\bDelta^\op}$. For $i,j \in [n]$ with $i \leq j$, let us write $[i,j] := [n]_{i//j} \in \bDelta$ for the corresponding closed interval. Using this notation, $\rlax^\nonu([n]) \in 2\Cat$ can be characterized as follows: its objects are those of $[n]$, and for $i \leq j$ we have
\[
\hom_{\rlax^\nonu([n])} ( i , j )
\simeq
\Ar^\act(\bDelta^\op)_{|[i,j]^\circ}
~,
\]
with composition given by concatenation (and for $i > j$ we have $\hom_{\rlax^\nonu([n])}(i,j) = \es$).\footnote{In other words, $\hom_{\rlax^\nonu([n])}(i,j)$ has objects the nondecreasing sequences in $[n]$ from $i$ to $j$.}

Now, let us define the further pullback
\begin{equation}
\label{pullback.defining.prime.rlax.brax.n.oint}
\begin{tikzcd}
'\rlax([n])^\oi
\arrow[hook]{r}
\arrow{d}
&
\rlax^\nonu([n])^\oi
\arrow{d}
\arrow{r}
&
\Ar^\act(\bDelta^\op)
\arrow{d}{s}
\arrow{r}{t}
&
\bDelta^\op
\\
(\bDelta^\inj_{/[n]})^\op
\arrow[hook]{r}
&
(\bDelta_{/[n]})^\op
\arrow{r}
&
\bDelta^\op
\\[-0.8cm]
&
\rotatebox{90}{$\simeq$}
\\[-0.8cm]
&
{[n]}^\oi
\end{tikzcd}
\end{equation}
in $\Cat$, where we write $\bDelta^\inj_{/[n]} := (\bDelta^\inj)_{/[n]}$ and we consider $'\rlax([n])^\oi \in \Cat_{/\bDelta^\op}$ via the upper horizontal composite. We claim that $'\rlax([n])^\oi \simeq \rlax([n])^\oi$, which will prove the desired result.

We first note that $'\rlax([n])^\oi \in \Cat_{/\bDelta^\op}$ lies in the image of the monomorphism $\iota_1 2\Cat \xhookra{(-)^\oi} \Cat_{/\bDelta^\op}$; we write $'\rlax([n]) \in 2\Cat$ for the corresponding $(\infty,2)$-category. Moreover, in diagram \Cref{pullback.defining.prime.rlax.brax.n.oint}, the upper left horizontal functor is fully faithful (because the lower left horizontal functor is) and admits a right adjoint
\[ \begin{tikzcd}[column sep=1.5cm]
'\rlax([n])^\oi
\arrow[hook, transform canvas={yshift=0.9ex}]{r}
\arrow[dashed, leftarrow, transform canvas={yshift=-0.9ex}]{r}[yshift=-0.2ex]{\bot}[swap]{q}
&
\rlax^\nonu([n])^\oi
\end{tikzcd} \]
in $\Cat_{/\bDelta^\op}$.\footnote{On objects, this right adjoint is given by taking images of morphisms to $[n] \in \bDelta$.} In particular, $q$ is a localization (considered in $\Cat_{/\bDelta^\op}$ or in $\Cat$). Moreover, it is clear that $q$ defines a morphism in $\coCart_{/\bDelta^\op}$ and therefore a functor $'\rlax([n]) \xla{\tilde{q}} \rlax^\nonu([n])$ in $2\Cat$. Hence, by the Segal condition, it follows that $\tilde{q}$ is a localization at certain 2-morphisms; and unwinding the definitions, we see that these are generated under (horizontal) composition by the quasi-unit 2-morphisms of the objects of $[n]$.\footnote{Given an object $i \in [n]$, its corresponding quasi-unit 2-morphism corresponds to the diagram
\[
\begin{tikzcd}[ampersand replacement=\&]
{[1]}
\arrow{rr}
\arrow{rd}
\&
\&
{[0]}
\arrow{ld}[sloped, swap]{i}
\\
\&
{[n]}
\end{tikzcd}
~,
\]
considered as a morphism in $\bDelta_{/[n]}$.}
\end{proof}

\subsection{Fibrations}
\label{subsection.fibns.for.inftytwocats}

In this subsection, we introduce several notions of fibrations among $(\infty,2)$-categories. These will feature in our study of un/straightening in \Cref{subsection.un.straightening}.

\begin{local}
Throughout this subsection, we fix a functor $\cE \xra{\pi} \cC$ between $(\infty,2)$-categories.
\end{local}

\begin{definition}[{\cite[Chapter 11, Definition 1.1.2]{GR}}]
\label{defn.cart.fibn}
We say that a 1-morphism $e_0 \ra e_1$ in $\cE$ is \bit{cartesian} (with respect to $\pi$), or \bit{$\pi$-cartesian}, if for all $e \in \cE$ the commutative square
\begin{equation}
\label{comm.square.for.defining.a.cartesian.onemorphism}
\begin{tikzcd}
\hom_\cE(e,e_0)
\arrow{r}
\arrow{d}
&
\hom_\cE(e,e_1)
\arrow{d}
\\
\hom_\cC(\pi(e),\pi(e_0))
\arrow{r}
&
\hom_\cC(\pi(e),\pi(e_1))
\end{tikzcd}
\end{equation}
in $\Cat$ is a pullback. We then say that $\cE \xra{\pi} \cC$ is a \bit{2-cartesian fibration} if the following conditions are satisfied.
\begin{enumerate}

\item\label{defn.cart.fibn.part.cartesian.lifts}

For every object $e \in \cE$ and 1-morphism $c \ra \pi(e)$ in $\cC$ there exists a cartesian 1-morphism in $\cE$ covering it with target $e$.

\item\label{defn.cart.fibn.part.require.cocart.fibn}

For all $e_0,e_1 \in \cE$, the morphism
\[
\hom_\cE(e_0,e_1)
\longra
\hom_\cC(\pi(e_0),\pi(e_1))
\]
in $\Cat$ is a cocartesian fibration.

\item\label{defn.cart.fibn.part.composition}

For all $e_0,e_1,e_2 \in \cE$, in the commutative square
\[ \begin{tikzcd}
\hom_\cE(e_0,e_1)
\times
\hom_\cE(e_1,e_2)
\arrow{r}
\arrow{d}
&
\hom_\cE(e_0,e_2)
\arrow{d}
\\
\hom_\cC(\pi (e_0),\pi(e_1))
\times
\hom_\cC(\pi(e_1),\pi(e_2))
\arrow{r}
&
\hom_\cC(\pi(e_0),\pi(e_2))
\end{tikzcd} \]
in $\Cat$, the upper horizontal functor preserves cocartesian morphisms with respect to the vertical functors (which are cocartesian fibrations by condition \Cref{defn.cart.fibn.part.require.cocart.fibn}).

\end{enumerate}
We say that $\cE \xra{\pi} \cC$ is a \bit{homwise cocartesian fibration} if condition \Cref{defn.cart.fibn.part.require.cocart.fibn} is satisfied, and a \bit{strict homwise cocartesian fibration} if additionally condition \Cref{defn.cart.fibn.part.composition} is satisfied. If $\cE \xra{\pi} \cC$ is a homwise cocartesian fibration, we refer to the 2-morphisms in $\cE$ that define cocartesian 1-morphisms in a hom-$\infty$-category $\hom_\cE(e_0,e_1)$ as \bit{cocartesian 2-morphisms} (with respect to $\pi$).

We write
\[
2\Cat_{2\cart/\cC}
\subseteq
2\Cat_{/\cC}
\]
for the 1-full subcategory on the 2-cartesian fibrations, whose 1-morphisms are those that preserve cocartesian 2-morphisms. Moreover, we write
\[
2\Cart_\cC
\subseteq
2\Cat_{2\cart/\cC}
\]
for the 1-full subcategory on the same objects, whose 1-morphisms are those that additionally preserve cartesian 1-morphisms.

A \bit{1-cartesian fibration} is a 2-cartesian fibration whose fibers are $(\infty,1)$-categories. We write
\[
1\Cart_\cC
\subseteq
2\Cart_\cC
\qquad
\text{and}
\qquad
2\Cart_{1\cart/\cC}
\subseteq
2\Cat_{2\cart/\cC}
\]
for the full subcategories on the 1-cartesian fibrations.

Dually, we say that $\cE \xra{\pi} \cC$ is a \bit{2-cocartesian fibration} if $\cE^{1\&2\op} \xra{\pi^{1\&2\op}} \cC^{1\&2\op}$ is a 2-cartesian fibration. We use the evident notation and terminology for the corresponding related notions.
\end{definition}

\begin{observation}
\label{obs.check.strict.homwise.cocart.fibn.on.oi}
The functor $\cE \xra{\pi} \cC$ in $2\Cat$ is a strict homwise cocartesian fibration if and only if the corresponding functor $\cE^\oi \xra{\pi^\oi} \cC^\oi$ in $\Cat$ is a cocartesian fibration.\footnote{For the forwards implication, $\pi$ being a homwise cocartesian fibration implies that $\pi^\oi$ is a locally cocartesian fibration, and thereafter its strictness guarantees composability of the locally cocartesian morphisms.}
\end{observation}

\begin{definition}[{\cite[Chapter 11, Definition 3.1.2]{GR}}]
\label{defn.loc.cart.fibn}
We say that a 1-morphism $e_0 \ra e_1$ in $\cE$ classified by a functor $[1] \xra{\varphi} \cE$ is \bit{locally cartesian} (with respect to $\pi$), or \bit{locally $\pi$-cartesian}, if it defines a cartesian 1-morphism with respect to the pullback $(\pi\varphi)^*\cE \ra [1]$. We then say that $\cE \xra{\pi} \cC$ is a \bit{locally 2-cartesian fibration} if for every object $e \in \cE$ and 1-morphism $c \ra \pi(e)$ in $\cC$ there exists a locally cartesian 1-morphism in $\cE$ covering it with target $e$ and moreover $\pi$ is a strict homwise cocartesian fibration.

We employ the evident variants of the remaining notation and terminology of \Cref{defn.cart.fibn}, e.g.\! the 1-full subcategories
\[
\loc.2\Cart_\cC
\subseteq
2\Cat_{\loc.2\cart/\cC}
\subseteq
2\Cat_{/\cC}
\]
and the notion of a \bit{locally 2-cocartesian fibration} are defined similarly.
\end{definition}

\begin{lemma}
\label{lemma.locallytwocart.is.twocart.if.compose}
Suppose that $\cE \xra{\pi} \cC$ is a locally 2-cartesian fibration. Then, it is a 2-cartesian fibration if and only if its locally cartesian 1-morphisms are closed under composition.
\end{lemma}

\begin{proof}
Clearly, in a 2-cartesian fibration the cartesian 1-morphisms are closed under composition. Conversely, suppose that the locally cartesian 1-morphisms in $\cE$ are closed under composition. Then, each locally cartesian 1-morphism $e_0 \xra{\varphi} e_1$ in $\cE$ is in fact a cartesian 1-morphism. Indeed, since $\cE \xra{\pi} \cC$ is a strict homwise cocartesian fibration, the commutative square \Cref{comm.square.for.defining.a.cartesian.onemorphism} is a pullback if and only if it induces an equivalence on fibers.
So, choose an object $e \in \cE$ and a morphism $\pi(e) \xra{f} \pi(e_0)$ in $\cC$, let $e_0' \xra{\tilde{f}} e_0$ be a locally cartesian $f$, and consider the commutative triangle
\begin{equation}
\label{span.showing.loctwocart.is.twocart}
\begin{tikzcd}[column sep=0.5cm]
&
\hom_{\cE_{\pi(e)}}(e, e_0')
\arrow{ld}[sloped]{\tilde{f} \circ -}
\arrow{rd}[sloped]{\varphi \circ \tilde{f} \circ -}
\\
\hom^f_\cE(e,e_0)
\arrow{rr}[swap]{\varphi \circ -}
&
&
\hom^{\pi(\varphi) \circ f}_\cE(e,e_1)
\end{tikzcd}
\end{equation}
in $\Cat$ (in which the two lower terms denote the evident fibers over $f$). By assumption, the composite 1-morphism $e_0' \xra{\tilde{f}} e_0 \xra{\varphi} e_1$ is locally cartesian. It now follows from the definition of a locally cartesian fibration that both downwards functors in diagram \Cref{span.showing.loctwocart.is.twocart} are equivalences, which implies that its lower horizontal functor is an equivalence as well. Hence, $\varphi$ is indeed a cartesian 1-morphism.
\end{proof}

\begin{observation}
\label{obs.pull.back.loc.two.cart.fibns.along.rlax.fctr}
It follows from \Cref{obs.check.strict.homwise.cocart.fibn.on.oi} that we can pull back locally 2-cartesian fibrations along right-lax functors: given a diagram
\[ \begin{tikzcd}
&
\cE
\arrow{d}{\loc.2\cart}
\\
\cD
\arrow[squiggly]{r}[swap]{F}
&
\cC
\end{tikzcd} \]
(in which $F$ is a right-lax functor), we obtain a locally 2-cartesian fibration $F^* \cE \ra \cD$ via the pullback square
\[ \begin{tikzcd}
(F^* \cE)^\oi
\arrow{r}
\arrow{d}
&
\cE^\oi
\arrow{d}
\\
\cD^\oi
\arrow{r}
&
\cC^\oi
\end{tikzcd} \]
in $\Cat_{/\bDelta^\op}$ (in fact in $\Cat_{\cocart/\bDelta^\op}$).\footnote{The Segal condition for $(F^* \cE)^\oi$ follows from the fact that it can be checked over the subcategory of inert morphisms in $\bDelta^\op$, and the completeness condition therefor follows from the fact that (unital) right-lax functors preserve equivalences. Thereafter, the conditions of \Cref{defn.loc.cart.fibn} for the functor $F^*\cE \ra \cD$ follow from the facts that the functor $(F^* \cE)^\oi \ra \cD^\oi$ is a cocartesian fibration and that for every functor $[1] \ra \cD$ the composite $[1]^\oi \ra \cD^\oi \ra \cC^\oi$ defines a strict (as opposed to right-lax) functor $[1] \ra \cC$.} Evidently, if $F$ is a strict functor then this construction coincides with ordinary pullback therealong.
\end{observation}

\begin{observation}
\label{obs.two.cart.fibns}
We collect the following apparent facts about (locally) 2-cartesian fibrations.
\begin{enumerate}

\item\label{obs.two.cart.fibns.if.fibers.are.one.cats}

Suppose that the functor $\cE \xra{\pi} \cC$ is a homwise cocartesian fibration and moreover all its fibers are $(\infty,1)$-categories. Then it is automatically a \textit{strict} homwise cocartesian fibration.

\item\label{obs.two.cart.fibns.one.cart.fibn.iff}

The functor $\cE \xra{\pi} \cC$ in $2\Cat$ is a (locally) 1-cartesian fibration if and only if the functor $\iota_1 \cE \xra{\iota_1(\pi)} \iota_1 \cC$ in $\Cat$ is a (resp.\! locally) cartesian fibration and moreover for all $e_0,e_1 \in \cE$ the functor $\hom_\cE(e_0,e_1) \ra \hom_\cC(\pi(e_0),\pi(e_1))$ is a left fibration.

\item\label{obs.two.cart.fibns.over.a.one.cat}

Suppose that $\cC \in 2\Cat$ is an $(\infty,1)$-category. Then, the functor $\cE \xra{\pi} \cC$ is automatically a strict homwise cocartesian fibration. Hence, the inclusions $2\Cat_{2\cart/\cC} \subseteq 2\Cat_{\loc.2\cart/\cC} \subseteq 2\Cat_{/\cC}$ are fully faithful. In particular, $\cE \xra{\pi} \cC$ is a locally 2-cartesian fibration if and only if its pullback along every functor $[1] \ra \cC$ defines a 2-cartesian fibration over $[1]$.

\item
\label{obs.strict.maps.among.locallyfibns.are.equivalences.if.fiberwise.so}

A morphism in $\loc.2\Cart_\cC$ is an equivalence if and only if it's an equivalence on fibers (because the latter condition implies that it is both surjective and fully faithful).

\end{enumerate}
\end{observation}

\subsection{Un/straightening}
\label{subsection.un.straightening}

In this subsection, we consider variants of the Grothendieck construction for $(\infty,2)$-categories. Its main result is \Cref{thm.iota.one.version.of.loctwoCart.equivalent.to.twoCart.over.rlaxification}, which establishes the Grothendieck construction for lax functors as an equivalence of $\infty$-categories. We later enhance it to an equivalence of $(\infty,2)$-categories (see \Cref{t11}).

\begin{definition}
We refer to the equivalences of \Cref{thm.un.straightening.for.two.cats} as \bit{un/straightening}, and to the equivalence \Cref{lax.un.straightening} of \Cref{thm.iota.one.version.of.loctwoCart.equivalent.to.twoCart.over.rlaxification} as \bit{lax un/straightening}.
\end{definition}

\begin{remark}
\Cref{thm.un.straightening.for.two.cats} appears as \cite[Chapter 11, Theorem-Construction 1.1.8(b)]{GR}, while \Cref{thm.iota.one.version.of.loctwoCart.equivalent.to.twoCart.over.rlaxification} is a slight variant of \cite[Chapter 11, Theorem-Construction 3.2.2]{GR} (with essentially the same proof).
\end{remark}

\begin{local}
In this subsection, we fix an $(\infty,2)$-category $\cC \in 2\Cat$.
\end{local}

\begin{theorem}
\label{thm.un.straightening.for.two.cats}
There are canonical equivalences
\[
\Fun ( \cC , 2\Cat)
\simeq
2\coCart_\cC
\qquad
\text{and}
\qquad
\Fun( \cC^{1\op} , 2\Cat)
\simeq
2\Cart_\cC
\]
in $2\Cat$, which are functorial in $\cC \in 2\Cat^{1\op}$.
\end{theorem}

\begin{proof}
This is a special case of \cite[Theorem 6.21]{Nuiten-straightening}.\footnote{Note that our definition of $2\Cat$ agrees with that of \cite{Nuiten-straightening} by \cite[Remark 4.21]{Nuiten-straightening}.}
\end{proof}

\begin{observation}
\label{t14}
Let $\cE \da [1]$ be a 2-cartesian fibration.
Then, the corresponding monodromy functor $\cE_1 \ra \cE_0$ is the composite
\[
\cE_1
\underset{\sim}{\xla{\ev_1}}
\Gamma^{\cart}_{[1]}(\cE)
\xra{\ev_0}
\cE_0
~.
\]
Indeed, we have an equivalence
\[
\Gamma^{\cart}_{[1]}(\cE)
\simeq
\lim
\left(
[1]^{\op}
\to
2\Cat
\right)
\]
in $2\Cat$, where the functor $[1]^{\op} \to 2\Cat$ corresponds to $\cE\da [1]$ through \Cref{thm.un.straightening.for.two.cats}.
The assertion then follows from the fact that
given a functor $[1]^{\op} \xra{F} \cC$ to an $(\infty,1)$-category, 
the morphism $F(1^\circ) \xra{ F( (0 \to 1)^\circ )} F(0^\circ)$ in $\cC$
is the composite
\[
F(1^\circ)
\simeq
\lim\left(
\{1\}^{\op}
\longhookra
[1]^{\op}
\xlongra{F}
\cC
\right)
\xlongla{\sim}
\lim\left(
[1]^{\op}
\xlongra{F}
\cC
\right)
\longra
\lim\left(
\{0\}^{\op}
\longhookra
[1]^{\op}
\xlongra{F}
\cC
\right)
\simeq
F(0^\circ)
~.
\]
\end{observation}

\begin{observation}
\label{t15}
Let $\cE \da [1]$ be a 2-cartesian fibration.
By the definition of a 2-cartesian fibration and \Cref{cor.existence.of.radjt.detectable.ptwise}, the functor $\Gamma_{[1]}(\cE) \xra{\ev_1} \cE_1$ is admits a right adjoint $\ev_1^R$, which carries each object to the cartesian section of which it is the target. Hence, the unit is an equivalence $\id_{\cE_1} \xra{\sim} \ev_1^R \circ \ev_1$, and the image of $\ev_1^R$ lies in the full subcategory $\Gamma^{\cart}_{[1]}(\cE) \subset \Gamma_{[1]}(\cE)$ of cartesian sections.
Hence, $\ev_1^R$ is the composite functor
\[
\cE_1
\xlongra{\sim}
\Gamma^{\cart}_{[1]}(\cE)
\longhookra
\Gamma_{[1]}(\cE)
~.
\]
Therefore, using \Cref{t14} we see that the monodromy functor can also be identified as the composite
\[
\cE_1
\xra{\ev_1^R}
\Gamma_{[1]}(\cE) 
\xra{\ev_0}
\cE_0
~.
\]
\end{observation}

\begin{theorem}
\label{thm.iota.one.version.of.loctwoCart.equivalent.to.twoCart.over.rlaxification}
Pullback (in the sense of \Cref{obs.pull.back.loc.two.cart.fibns.along.rlax.fctr}) along the universal right-lax functor
\[ \begin{tikzcd}[column sep=0.6cm]
\cC
\arrow[squiggly]{r}{\theta}
&
\rlax(\cC)
\end{tikzcd} \]
determines an equivalence
\[
\theta^\ast
\colon
\iota_1 2\Cart_{\rlax(\cC)}
\xlongra{\sim}
\iota_1 \loc.2\Cart_\cC
~.
\]
In particular, there is an equivalence
\begin{equation}
\label{lax.un.straightening}
\iota_1 \loc.2\Cart_\cC
\simeq
\iota_1 \Fun ( \rlax(\cC^{1\op}) , 2\Cat )
~.
\end{equation}
\end{theorem}

The remainder of this subsection is devoted to proving \Cref{thm.iota.one.version.of.loctwoCart.equivalent.to.twoCart.over.rlaxification}.
We first construct an inverse to the pullback functor $\theta^\ast$ as the leftmost factorization in a diagram
\[ \begin{tikzcd}
\iota_1 \loc.2\Cart_\cC
\arrow[dashed]{d}[swap]{\Phi}
\arrow[dashed]{rd}[sloped, swap]{\Phi^\nonu}
\arrow[dashed]{rrd}[sloped]{\Phi^\nonu}
\arrow[dashed, bend left=10]{rrrd}[sloped]{\Phi^\nonu(-)^\oi}
\\
\iota_1 2\Cart_{\rlax(\cC)}
\arrow[hook]{r}
&
\iota_1 2\Cart_{\rlax^\nonu(\cC)}
\arrow[hook]{r}
&
\iota_1 2\Cat_{/\rlax^\nonu(\cC)}
\arrow[hook]{r}[swap]{(-)^\oi}
&
\Cat_{/\rlax^\nonu(\cC)^\oi}
\end{tikzcd}
~.
\]

\begin{local}
Let $S_{m,n}$ denote the thin 2-category
\[ \begin{tikzcd}[row sep=0.4cm, column sep=0.5cm]
\bullet
\arrow{r}[description, yshift=-0.405cm]{\rotatebox{45}{$\Leftarrow$}}
\arrow{d}
&
\bullet
\arrow{r}[description, yshift=-0.405cm]{\rotatebox{45}{$\Leftarrow$}}
\arrow{d}
&
\cdots
\arrow{r}[description, yshift=-0.405cm]{\rotatebox{45}{$\Leftarrow$}}
&
\bullet
\arrow{d}
\\
\bullet
\arrow{r}[description, yshift=-0.405cm]{\rotatebox{45}{$\Leftarrow$}}
\arrow{d}
&
\bullet
\arrow{r}[description, yshift=-0.405cm]{\rotatebox{45}{$\Leftarrow$}}
\arrow{d}
&
\cdots
\arrow{r}[description, yshift=-0.405cm]{\rotatebox{45}{$\Leftarrow$}}
&
\bullet
\arrow{d}
\\[-0.2cm]
\myvdots
\arrow{d}
&
\myvdots
\arrow{d}
&
\ddots
&
\myvdots
\arrow{d}
\\
\bullet
\arrow{r}[description, yshift=0.405cm]{\rotatebox{45}{$\Leftarrow$}}
&
\bullet
\arrow{r}[description, yshift=0.405cm]{\rotatebox{45}{$\Leftarrow$}}
&
\cdots
\arrow{r}[description, yshift=0.405cm]{\rotatebox{45}{$\Leftarrow$}}
&
\bullet
\end{tikzcd} \]
with $m$ vertical 1-morphisms in each column and $n$ horizontal 1-morphisms in each row. These assemble into a bicosimplicial thin 2-category $\bDelta \times \bDelta \xra{S_{\bullet,\bullet}} \iota_1 2\Cat$, and we write
\[
\Sq := \hom_{\iota_1 2\Cat} ( S_{\bullet,\bullet} , - )
:
\iota_1 2\Cat
\longra
\Fun(\bDelta^\op \times \bDelta^\op , \Spaces)
\]
for the functor that it corepresents. We use these notations in the present subsection.
\end{local}

\begin{observation}
\label{t28}
By \cite[Corollary 4.4.2]{HORR-pasting} (see also \cite{Columbus-thesis}), the functor $\Sq$ lands in the full subcategory of double $\infty$-categories: simplicial objects in $\Cat \subseteq \Fun(\bDelta^\op , \Spaces)$ satisfying the Segal and completeness conditions.\footnote{So, a double $\infty$-category is an $(\infty,2)$-category if and only if its $0\th$ $\infty$-category is an $\infty$-groupoid.} 
We write
\[
\iota_1 2\Cat
\xra{\Sq(-)^\oi}
\coCart_{\bDelta^\op}
\]
for the cocartesian unstraightening of the functor carrying each object $[n]^\circ \in \bDelta^\op$ to the functor $\iota_1 2\Cat \xra{\Sq(-)_{\bullet,n}} \Cat \subset \Fun(\bDelta^\op , \Spaces)$. 
Note that 
\[
(-)^\oi 
~\subseteq~
\Sq(-)^\oi
\]
is a 1-full subcategory: it contains the same objects, and in the fiber over $[n]^\circ \in \bDelta^\op$ its 1-morphisms are those functors $S_{1,n} \ra (-)$ that carry the vertical 1-morphisms to equivalences.
\end{observation}

\begin{observation}
\label{t21}
Any morphism $I \xra{\varphi} J$ in $\bDelta^\act$ admits a canonical factorization
\begin{equation}
\label{factorizn.of.active.arrow.in.Delta}
\begin{tikzcd}
I
\arrow{rr}{\varphi}
\arrow[dashed]{rd}[sloped, swap]{i \longmapsto i}
&
&
J
\\
&
K
\arrow[dashed]{ru}[sloped]{j \longmapsto j}[sloped, swap]{i \longmapsto \varphi(i)}
\end{tikzcd}
\end{equation}
in $\bDelta^\act$, where $K := (I \sqcup J) / \{ \max(I) \sim \max(J) \}$ with ordering characterized by the requirement that for all $i \in I$ and $j \in J$ we have $i < j$ if and only if $\varphi(i) \leq j$. This determines a monomorphism $\Ar^\act(\bDelta^\op) \hookra \Fun([2],\bDelta^\op)$, which is adjoint to natural transformations
\begin{equation}
\label{e50}
t
\longla
\mu
\longla
s
\end{equation}
in $\Fun(\Ar^\act(\bDelta^\op) , \bDelta^\op)$ (whose components at $(I \xra{\varphi} J)^\circ \in \Ar^\act(\bDelta^\op)$ correspond to the diagram \Cref{factorizn.of.active.arrow.in.Delta} in $\bDelta^\act$).
\end{observation}

\begin{local}
\label{t20}
Any functor $\Ar^\act(\bDelta^\op) \xra{\chi} \bDelta^\op$ determines an endofunctor
\[
\coCart_{\bDelta^\op}
\xra{\Freeact_\chi}
\coCart_{\bDelta^\op}
\]
given by taking $\cE \da \bDelta^\op$ to the fiber product
\[
\begin{tikzcd}
\Freeact_\chi(\cE)
\arrow{r}
\arrow{d}
&
\Ar^\act(\bDelta^\op)
\arrow{r}{t}
\arrow{d}{\chi}
&
\bDelta^\op
\\
\cE
\arrow{r}
&
\bDelta^\op
\end{tikzcd}
\]
considered in $\coCart_{\bDelta^\op}$ by the horizontal composite (compare with \Cref{defn.nonu.rlaxification}); in particular, $\Freeact := \Freeact_s$ where $s$ is as in \Cref{e50}. This defines a functor $\Fun(\Ar^\act(\bDelta^\op) , \bDelta^\op) \ra \Fun ( \coCart_{\bDelta^\op},\coCart_{\bDelta^\op})$.\footnote{Indeed, the functoriality is as follows: given any $\cC,\cD \in \Cat$, pullback assembles as a functor
\[
\Fun(\cC,\cD)
\times
\coCart_\cD
\longrightarrow
\coCart_\cC
\]
(which corresponds through straightening to composition of functors).}
\end{local}

\begin{local}
\label{t22}
Given an object $(\cE \da \cC) \in \loc.2\Cart_\cC$, we define the 1-full subcategory
\[
\Phi^\nonu(\cE)^\oi
\subseteq
\Freeact_\mu(\Sq(\cE)^\oi)
\]
as follows, using the notation of diagram \Cref{factorizn.of.active.arrow.in.Delta} in \Cref{t21} throughout. 
First of all, an object of $\Phi^\nonu(\cE)^\oi$ is given by functors $I \ra K \xra{e_\bullet} \cE$ such that for every $k \in K \backslash \{ \max(K) \}$, the morphism $e_k \ra e_{k+1}$ is sent to an equivalence in $\cC$ if $k \in I$ and is locally cartesian over $\cC$ if $k \in J$. Then, a morphism in $\Phi^\nonu(\cE)^\oi$ from $I' \ra K' \xra{e'_\bullet} \cE$ to $I \ra K \xra{e_\bullet} \cE$ is given by (the opposite of) a morphism
\begin{equation}
\label{e26}
\begin{tikzcd}
I'
\arrow{r}
&
K'
\arrow{r}
&
J'
\\
I
\arrow{r}
\arrow{u}
&
K
\arrow{r}
\arrow{u}{\beta}
&
J
\arrow{u}
\end{tikzcd}
\end{equation}
in $\Ar^\act(\bDelta) \subset \Fun([2]^{\op},\bDelta)$
along with a diagram
\begin{equation}
\label{diagram.of.twomors.for.defining.Phi.E.oi}
\begin{tikzcd}[row sep=1.5cm]
e'_{\beta(\min(K))}
\arrow{r}[swap, yshift=-0.4cm]{\rotatebox{45}{$\xLeftarrow[\eta_{\min(K)}]{}$}}
\arrow{d}[swap]{\gamma_{\min(K)}}
&
e'_{\beta(\min(K) + 1)}
\arrow{r}[swap, xshift=-0.1cm, yshift=-0.4cm]{\rotatebox{45}{$\xLeftarrow[\eta_{\min(K)+1}]{}$}}
\arrow{d}[swap]{\gamma_{\min(K)+1}}
&
\cdots
\arrow{r}
\arrow{r}[swap, xshift=0.2cm, yshift=-0.4cm]{\rotatebox{45}{$\xLeftarrow[\eta_{\max(K)-1}]{}$}}
&
e'_{\beta(\max(K))}
\arrow{d}{\gamma_{\max(K)}}
\\
e_{\min(K)}
\arrow{r}
&
e_{\min(K)+1}
\arrow{r}
&
\cdots
\arrow{r}
&
e_{\max(K)}
\end{tikzcd}
\end{equation}
in $\cE$,\footnote{Beware that the upper row in diagram \Cref{diagram.of.twomors.for.defining.Phi.E.oi} may not define an object of $\Phi^\nonu(\cE)^\oi$, although for every $k \in (I \backslash \{ \max(I) \}) \subseteq (K \backslash \{ \max(K) \})$ the 1-morphism $e'_{\beta(k)} \ra e'_{\beta(k+1)}$ is sent to an equivalence in $\cC$.} such that
\begin{itemize}

\item for every $k \in K$, the 1-morphism $\gamma_k$ in $\cE$ is sent to an equivalence in $\cC$;


\item for every $i \in I \subseteq K$, the 1-morphism $\gamma_i$ in $\cE$ is an equivalence; and


\item for every $i \in (I \backslash \{ \max(I) \}) \subseteq (K \backslash \{\max(K)\})$, the 2-morphism $\eta_i$ in $\cE$ is sent to an equivalence in $\cC$;

\item for every $j \in (J \backslash \{ \max(J) \}) \subseteq (K \backslash \{ \max(K) \})$, the 2-morphism $\eta_j$ in $\cE$ is a cocartesian 2-morphism over $\cC$.



\end{itemize}

\end{local}

\begin{observation}
\label{t24}
By the definition of $\Phi^\nonu(\cE)^\oi$, we have a factorization
\begin{equation}
\label{factorizn.of.PhinonuEoi.to.rlaxnonuCoi}
\begin{tikzcd}[column sep=1.5cm]
\Phi^\nonu(\cE)^\oi
\arrow[hook]{rrr}
\arrow[dashed]{d}
&[-1.7cm]
&
&
\Freeact_\mu(\Sq(\cE)^\oi)
\arrow{d}
\\
\rlax^\nonu(\cC)^\oi
&
\simeq
\Freeact(\cC^\oi)
\arrow[hook]{r}
&
\Freeact(\Sq(\cC)^\oi)
\arrow[hook]{r}[swap]{s \longra \mu}
&
\Freeact_\mu(\Sq(\cC)^\oi)
\end{tikzcd}~,
\end{equation}
where the lower right horizontal functor (induced by the natural transformation $s \ra \mu$) is 1-full since $\Sq(\cC)$ is a double $\infty$-category (in particular it is complete). 
\end{observation}

\begin{lemma}
\label{t23}
The functor $\Phi^\nonu(\cE)^\oi \to  \rlax^\nonu(\cC)^\oi$ of \Cref{t24} is a cocartesian fibration. 
The cocartesian morphisms are those in which, for each $i \in I \backslash \{ \max(I) \}$, the corresponding 2-morphism $\eta_i$ (as in diagram \Cref{diagram.of.twomors.for.defining.Phi.E.oi}) is invertible (or, equivalently, for each $k\in K \setminus  \{\max(K) \}$, the 2-morphism $\eta_k$ is cocartesian).

\end{lemma}

\begin{proof}
Observe that the functor $\Freeact_\mu(\Sq(\cE)^\oi)_{|\rlax^\nonu(\cC)^\oi} \to \rlax^\nonu(\cC)^\oi$ from the indicated base change is a cocartesian fibration: a morphism in its source is cocartesian if and only if in the corresponding diagram \Cref{diagram.of.twomors.for.defining.Phi.E.oi} each $\gamma_i$ is invertible, for each $i \in I \setminus \{ \max(I)\}$, $\eta_i$ is invertible, and for each $j \in J  \setminus \{\max(J)\}$, the 2-morphism $\eta_j$ is cocartesian.
Consider the subcategory $\cU \subseteq \Freeact_\mu(\Sq(\cE)^\oi)_{|\rlax^\nonu(\cC)^\oi}$ consisting of those morphisms such that in its corresponding diagram \Cref{diagram.of.twomors.for.defining.Phi.E.oi}, for each $i\in I$, the 1-morphisms $\gamma_i$ are equivalences and,
for each $j \in (J \backslash \{ \max(J) \})$, the 2-morphism $\eta_j$ is a cocartesian 2-morphism over $\cC$.
Notice that $\cU$ contains all cocartesian morphisms over $\rlax^\nonu(\cC)^\oi$; in particular, $\cU \da \rlax^\nonu(\cC)^\oi$ is a cocartesian fibration and the inclusion $\cU \hookra \Freeact_\mu(\Sq(\cE)^\oi)_{|\rlax^\nonu(\cC)^\oi}$ preserves cocartesian morphisms.

Now, observe the monomorphism $\Phi^\nonu(\cE)^\oi \hookra \Freeact_\mu(\Sq(\cE)^\oi)_{|\rlax^\nonu(\cC)^\oi}$. 
Inspecting the definition of $\rlax^\nonu(\cC)^\oi$, and of $\cU$, observe that this monomorphism factors as a fully faithful functor $\Phi^\nonu(\cE)^\oi \hookra \cU$.
We claim that for each object $x \in \rlax^\nonu(\cC)^\oi$ the inclusion $(\Phi^\nonu(\cE)^\oi)_x \hookra \cU_x$ admits a left adjoint. 
For this, fix an object $\tilde{x} := (I \hookra K \xra{e_\bullet} \cE) \in \cU_x$. We must construct an initial object in the undercategory $((\Phi^\nonu(\cE)^\oi)_x)_{\tilde{x}/}$.
By definition of $\cU$, for each $k\in I$ the 1-morphism $e_k \to e_{k+1}$ is sent to an equivalence in $\cC$.
To construct the desired initial object, we enforce that each $e_k \to e_{k+1}$ is locally cartesian over $\cC$ for each $k\in J$.
We do this by working backwards in the finite linearly ordered set $J$: inductively, take the morphism $e'_k \to e'_{k+1} = e_{k+1}$ to be the locally cartesian lift of the image in $\cC$ of the 1-morphism $e_k \to e_{k+1}$.  
This resulting object $\tilde{y}\in (\Phi^\nonu(\cE)^\oi)_{x} $ receives a canonical morphism from $\tilde{x}$.
As so, this object is initial.
Moreover, in the canonical morphism $\tilde{x} \to \tilde{y}$, for each $k\in K \setminus \{\max(K)\}$, the corresponding 2-morphism $\eta_k$ (as in \Cref{diagram.of.twomors.for.defining.Phi.E.oi}) is invertible.
It follows that $\Phi^\nonu(\cE)^\oi  \da \rlax^\nonu(\cC)^\oi$ is a locally cocartesian fibration, 
with cocartesian morphisms those in which, for each $i \in I \backslash \{ \max(I) \}$, the corresponding 2-morphism $\eta_i$ (as in diagram \Cref{diagram.of.twomors.for.defining.Phi.E.oi}) is invertible.
It is clear that the locally cocartesian morphisms are closed under composition.
Therefore, $\Phi^\nonu(\cE)^\oi  \da \rlax^\nonu(\cC)^\oi$ is indeed a cocartesian fibration, as claimed.  
\end{proof}


\begin{observation}
\label{t25}
By \Cref{t23}, we obtain a functor
\[
\iota_1 \loc.2\Cart_\cC
\xra{\Phi^\nonu(-)^\oi}
\coCart_{\rlax^\nonu(\cC)^\oi}
~.
\]
For the composite cocartesian fibration,
\begin{equation}
\label{e51}
\Phi^\nonu(\cE)^\oi
\longra
\rlax^\nonu(\cC)^\oi
\longra
\bDelta^\op
~,
\end{equation}
using \Cref{t22}, the cocartesian morphisms are given as those in which
\begin{itemize}
\item the morphism $J \to J'$ (in diagram \Cref{e26}) is inert,
\item each morphism $\gamma$ (in diagram \Cref{diagram.of.twomors.for.defining.Phi.E.oi}) is carried to an equivalence in $\cC$, and
\item each 2-morphism $\eta$ (in diagram \Cref{diagram.of.twomors.for.defining.Phi.E.oi}) is invertible.
\end{itemize}
Moreover, the cocartesian fibration \Cref{e51} defines an $(\infty,2)$-category.\footnote{From the description of cocartesian morphisms in $\Phi^\nonu(\cE)^\oi \da \rlax^\nonu(\cC)^\oi$ of \Cref{t23}, 
the inclusion $\Phi^\nonu(\cE)^\oi \hookrightarrow \Freeact_\mu(\Sq(\cE^\oi))_{|\rlax^\nonu(\cC)^\oi}$ preserves cocartesian morphisms over inert morphisms in $\bDelta^{\op}$.
Therefore, the Segal condition for $\Phi^\nonu(\cE)^\oi \da \bDelta^{\op}$ follows from that for $\Sq(\cE)^\oi  \da \bDelta^{\op}$. 
Similarly, the completeness condition for $\Phi^\nonu(\cE)^\oi \da \bDelta^{\op}$ follows from
that for $\Sq(\cE)^\oi  \da \bDelta^{\op}$.
Furthermore, we clearly have $\Phi^\nonu(\cE)^\oi_0 \simeq \iota_0 (\cE)$.
}
Explicitly, its objects are those of $\cE$, its 1-morphisms are strings of 1-morphisms $e_0 \ra \cdots \ra e_m$ in $\cE$ with $m \geq 1$ such that the 1-morphism $e_0 \ra e_1$ in $\cE$ is sent to an equivalence in $\cC$ and the 1-morphisms $e_i \ra e_{i+1}$ in $\cE$ are locally cartesian over $\cC$ for all $1 \leq i < m$, and a typical 2-morphism is given by a diagram
\begin{equation}
\label{e100}
\begin{tikzcd}[ampersand replacement=\&]
e'_0
\arrow{r}[swap, yshift=-0.25cm]{\rotatebox{45}{$\xLeftarrow[\eta_0]{}$}}
\arrow{d}[swap]{\gamma_0}[sloped, anchor=south]{\sim}
\&
e'_1
\arrow{r}
\arrow{d}[swap]{\gamma_1}
\arrow{rd}[sloped]{\gamma_2}[swap, xshift=0.1cm, yshift=0.1cm]{\rotatebox{45}{$\xLeftarrow[\eta_1]{}$}}
\&
e'_2
\arrow{r}[swap, yshift=-0.3cm, xshift=0.6cm]{\rotatebox{30}{$\xLeftarrow[\eta_2]{}$}}
\&
e'_3
\arrow{r}
\&
e'_4
\arrow{r}[swap, xshift=0.8cm, yshift=-0.3cm]{\rotatebox{30}{$\xLeftarrow[\eta_3]{}$}}
\arrow{d}[swap]{\gamma_3}
\&
e'_5
\arrow{r}
\&
e'_6
\arrow{rr}[swap, yshift=-0.3cm, xshift=0.4cm]{\rotatebox{30}{$\xLeftarrow[\eta_5]{}$}}
\arrow{d}[swap]{\gamma_4}
\arrow{rd}[sloped]{\gamma_5}[swap, xshift=0.1cm, yshift=0.1cm]{\rotatebox{45}{$\xLeftarrow[\eta_4]{}$}}
\&
\&
e'_7
\arrow{d}{\gamma_6}[sloped, anchor=north]{\sim}
\\
e_0
\arrow{r}
\&
e_1
\arrow{r}
\&
e_2
\arrow{rr}
\&
\&
e_3
\arrow{rr}
\&
\&
e_4
\arrow{r}
\&
e_5
\arrow{r}
\&
e_6
\end{tikzcd}
\end{equation}
in $\cE$ in which the 2-morphism $\eta_0$ is sent to an equivalence in $\cC$ and all remaining 2-morphisms in the diagram are required to be cocartesian over $\cC$.

Altogether, we obtain a functor
\[
\iota_1 \loc.2\Cart_\cC
\xra{\Phi^\nonu}
2\Cat_{/\rlax^\nonu(\cC)}
~.
\]

\end{observation}

\begin{observation}
\label{t26}
The functor $\Phi^\nonu$ factors through $2\Cart_{\rlax^\nonu(\cC)} \subseteq 2\Cat_{/\rlax^\nonu(\cC)}$. 
Indeed, by \Cref{obs.check.strict.homwise.cocart.fibn.on.oi}, the functor $\Phi^\nonu(\cE) \ra \rlax^\nonu(\cC)$ is a strict homwise cocartesian fibration;
inspecting the definition of $\Phi^\nonu(\cE)^\oi$ (see also \Cref{t25}),
its cartesian 1-morphisms are those 1-morphisms in $\Phi^\nonu(\cE)$ in which all constituent 1-morphisms in $\cE$ are locally cartesian (i.e.\! the first 1-morphism is an equivalence), and by \Cref{t23}, its 2-cocartesian morphisms are those in which each constituent $\eta$ (as in \Cref{e100}) is cocartesian (i.e.\! the first 2-morphism $\eta_0$ is an equivalence).
\end{observation}

\begin{observation}
\label{t100}
There exists a factorization
\[ \begin{tikzcd}[row sep=0.3cm]
\Phi^\nonu(\cE)^\oi
\arrow{dd}
\arrow[dashed]{rrrr}
\arrow[hook]{rd}
&[-0.5cm]
&
&
&[-0.5cm]
\cE^\oi
\arrow{dd}
\\
&
\Freeact_\mu(\Sq(\cE)^\oi)
\arrow{r}{\mu \longra t}
\arrow{dd}
&
\Freeact_t(\Sq(\cE)^\oi)
\simeq
\Ar^\act(\bDelta^\op)
\underset{\bDelta^\op}{\times}
\Sq(\cE)^\oi
\arrow{r}{\pr}
\arrow[xshift=-1.4cm]{dd}
&
\Sq(\cE)^\oi
\arrow[hookleftarrow]{ru}
\arrow{dd}
\\
\rlax^\nonu(\cC)^\oi
\arrow[hook]{rd}
&
&
&
&
\cC^\oi
\\
&
\Freeact_\mu(\Sq(\cC)^\oi)
\arrow{r}[swap]{\mu \longra t}
&
\Freeact_t(\Sq(\cC)^\oi)
\simeq
\Ar^\act(\bDelta^\op)
\underset{\bDelta^\op}{\times}
\Sq(\cC)^\oi
\arrow{r}[swap]{\pr}
&
\Sq(\cC)^\oi
\arrow[hookleftarrow]{ru}
\end{tikzcd} \]
in which the leftmost vertical functor comes from \Cref{t24} and the two diagonal inclusions on the right arise from \Cref{t28}. In other words, we have a commutative square
\begin{equation}
\label{e98}
\begin{tikzcd}
\Phi^\nonu(\cE)^\oi
\arrow{r}
\arrow{d}
&
\cE^\oi
\arrow{d}
\\
\rlax^\nonu(\cC)^\oi
\arrow{r}
&
\cC^\oi
\end{tikzcd}
~.
\end{equation}
Moreover, the commutative square \Cref{e98} in $\Cat$ lifts to $\Cat_{/\bDelta^{\op}}$, and in fact lies in the subcategory $2\Cat \subset \coCart_{\bDelta^\op} \subset \Cat_{/\bDelta^{\op}}$.
\end{observation}

\begin{observation}
\label{t27}
The functor $\Phi^\nonu$ factors further through $2\Cart_{\rlax(\cC)} \subseteq 2\Cart_{\rlax^\nonu(\cC)}$, a subcategory via un/straightening.
In particular, we obtain a functor
\[
\iota_1 \loc.2\Cart_\cC
\xlongra{\Phi}
2\Cart_{\rlax(\cC)}
~.
\]
Indeed, because the construction of $\Phi^\nonu$ commutes with pullbacks in the variable $\cC \in \iota_1 2\Cat^\op$, it suffices to check the case that $\cC = \pt$. 
And in this case we have $\Phi^\nonu(\cE) \simeq \cE \times \rlax^\nonu(\pt)$, where the projection $\Phi^\nonu(\cE) \ra \cE$ is given by \Cref{e98}.

\end{observation}

\begin{proof}[Proof of \Cref{thm.iota.one.version.of.loctwoCart.equivalent.to.twoCart.over.rlaxification}]
The second statement follows from the first using \Cref{obs.rlax.commutes.with.oneop} and un/straightening.

We will show that $\Phi$ is inverse to the functor $\theta^\ast$.
We first verify the equivalence $\theta^* \Phi \simeq \id_{\iota_1 \loc.2\Cart_\cC}$. 
For this, given any $\cE \in \iota_1 \loc.2\Cart_\cC$, 
we have the commutative diagram \Cref{e98}.
Now, $\theta^{\oi}$ is the composite $\cC^{\oi} \xra{( \theta^{\nonu})^{\oi}} \rlax^\nonu(\cC)^\oi \to \rlax(\cC)^\oi$, and by definition of $\Phi(\cE)^{\oi} \da \rlax(\cC)^{\oi}$, its base change along $\rlax^{\nonu}(\cC)^{\oi} \to \rlax(\cC)$ is $\Phi^{\nonu}(\cE)^{\oi} \da \rlax^{\nonu}(\cC)^{\oi}$. 
Thus, we have a functor 
\[
(\theta^{\oi})^\ast \Phi(\cE)^{\oi} 
\simeq
((\theta^{\nonu})^{\oi})^\ast \Phi^{\nonu}(\cE)^{\oi}  
\longrightarrow
\cE^{\oi}
\]
over $\cC^{\oi}$.
By the description of cocartesian morphisms in $(\theta^{\oi})^\ast \Phi(\cE)^{\oi} $ over $\bDelta^{\op}$ as in \Cref{t25}, this functor preserves cocartesian morphisms over $\bDelta^{\op}$.
In other words, it defines a morphism
\begin{equation}
\label{e28}
\theta^\ast \Phi(\cE)
\longrightarrow 
\cE
\end{equation}
in $2\Cat$.
By construction, this functor is an equivalence between spaces of objects.
Because $\cE \to \cC$ is a locally 2-cartesian fibration, each 1-morphism in $\cE$ uniquely factors as a 1-morphism in a fiber over $\cC$ followed by a locally cartesian 1-morphism over $\cC$, and similarly for 2-morphisms in $\cE$.\footnote{The $(\infty,2)$-category $\theta^* \Phi(\cE)$ can be described as follows: its objects are those of $\cE$, a 1-morphism from $e_0$ to $e_1$ is a pair of 1-morphisms $e_0 \ra e \ra e_1$ in $\cE$ such that $e_0 \ra e$ is sent to an equivalence in $\cC$ and $e \ra e_1$ is locally cartesian over $\cC$, and a 2-morphism from $e_0 \ra e' \ra e_1$ to $e_0 \ra e \ra e_1$ is given by a diagram
\[ \begin{tikzcd}[ampersand replacement=\&]
e_0
\arrow{r}[swap, yshift=-0.25cm]{\rotatebox{45}{$\xLeftarrow[\eta_0]{}$}}
\arrow{d}[sloped, anchor=north]{=}
\&
e'
\arrow{r}[swap, yshift=-0.25cm]{\rotatebox{45}{$\xLeftarrow[\eta_1]{}$}}
\arrow{d}
\&
e_1
\arrow{d}[sloped, anchor=south]{=}
\\
e_0
\arrow{r}
\&
e
\arrow{r}
\&
e_1
\end{tikzcd} \]
in $\cE$ in which the 2-morphism $\eta_0$ is sent to an equivalence in $\cC$ and the 2-morphism $\eta_1$ is cocartesian over $\cC$. In these terms, the corresponding functor $\theta^* \Phi(\cE) \ra \cE$ is given by composition of 1- and 2-morphisms (and this is clearly both surjective and fully faithful).}
Therefore, the functor \Cref{e28} also induces an equivalence between spaces of 1- and 2-morphisms.
Hence, it is an equivalence.
With the evident functoriality of \Cref{e28}, this supplies the natural equivalence $\theta^\ast \Phi \simeq \id_{\iota_1 \loc.2\Cart_\cC}$.

We now verify the equivalence $\Phi \theta^* \simeq \id_{\iota_1 2\Cart_{\rlax(\cC)}}$. For this, fix an object $\cE \in \iota_1 2\Cart_{\rlax(\cC)}$, and let us write $\cE^\nonu \in \iota_1 2\Cart_{\rlax^\nonu(\cC)}$ for its pullback. Observe the diagram
\[ \begin{tikzcd}
&
&
(\cE^\nonu)^\oi
\arrow{d}
\\
\rlax^\nonu(\cC)^\oi
\arrow{r}
\arrow[bend right]{rr}[yshift=0.2cm]{\Uparrow}[swap]{\id_{\rlax^\nonu(\cC)^\oi}}
&
\cC^\oi
\arrow{r}
&
\rlax^\nonu(\cC)^\oi
\end{tikzcd} \]
in $\Cat$, in which the left horizontal functor lies in $\coCart_{\bDelta^\op}$ and corresponds to $\id_\cC$ and the 2-morphism is the unit of an adjunction in $\Cat_{\cocart/\bDelta^\op}$.\footnote{The counit of this adjunction is the equivalence between the composite $\cC^\oi \ra \rlax^\nonu(\cC)^\oi \ra \cC^\oi$ and $\id_{\cC^\oi}$. To see that this indeed gives an adjunction, it suffices to verify that it gives an adjunction fiberwise over $\bDelta^\op$. Thereafter, the Segal condition reduces the verification to the fibers over the objects $[0]^\circ,[1]^\circ \in \bDelta^\op$, in which cases the assertion is evident.} Since $(\cE^\nonu)^\oi \ra \rlax^\nonu(\cC)^\oi$ is a cocartesian fibration (recall \Cref{obs.check.strict.homwise.cocart.fibn.on.oi}), we obtain a commutative square
\[ \begin{tikzcd}
\cE^\nonu
\arrow{r}
\arrow{d}
&
\theta^* \cE
\arrow{d}
\\
\rlax^\nonu(\cC)
\arrow{r}
&
\cC
\end{tikzcd} \]
in $2\Cat$. Applying $\Sq(-)^\oi$, the composite $\cC^\oi \ra \rlax^\nonu(\cC)^\oi \ra \Sq(\rlax^\nonu(\cC))^\oi$ gives rise to a morphism
\begin{equation}
\label{equivce.between.SqEnonuoi.timesoverSqrlaxnonuCoi.Coi.and.SqiotastarEoi.timesoverSqCoi.Coi}
\Sq(\cE^\nonu)^\oi
\underset{\Sq(\rlax^\nonu(\cC))^\oi}{\times}
\cC^\oi
\longra
\Sq(\theta^* \cE)^\oi
\underset{\Sq(\cC)^\oi}{\times}
\cC^\oi
\end{equation}
in $\coCart_{\bDelta^\op}$,\footnote{This follows from the fact that for any locally 2-cartesian fibration $\cE' \xra{\pi'} \cC'$ the corresponding functor
\[
\Sq(\cE')^\oi \underset{\Sq(\cC')^\oi}{\times} \cC'^\oi
\longra
\cC'^\oi
\]
is a cocartesian fibration (as in \Cref{obs.check.strict.homwise.cocart.fibn.on.oi}).} which is an equivalence because it is so on fibers over each $[n]^\circ \in \bDelta^\op$. Next, we observe the factorization
\[
\begin{tikzcd}
\Phi^\nonu(\theta^* \cE)^\oi
\arrow[hook]{r}
\arrow[dashed]{dd}[swap]{(\alpha^\nonu)^{\oi}}
&
\Freeact_\mu \left( \Sq(\theta^* \cE)^\oi \underset{\Sq(\cC)^\oi}{\times} \cC^\oi \right)
\arrow[leftarrow]{r}[swap]{\sim}{\Cref{equivce.between.SqEnonuoi.timesoverSqrlaxnonuCoi.Coi.and.SqiotastarEoi.timesoverSqCoi.Coi}}
&
\Freeact_\mu \left( \Sq(\cE^\nonu)^\oi \underset{\Sq(\rlax^\nonu(\cC))^\oi}{\times} \cC^\oi \right)
\arrow{d}
\\
&
&
\Freeact_\mu ( \Sq(\cE^\nonu)^\oi )
\arrow{d}
\\
(\cE^\nonu)^\oi
\arrow[hook]{rr}
&
&
\Sq(\cE^\nonu)^\oi
\end{tikzcd} \]in $\Cat_{\cocart/\bDelta^\op}$, in which the lower right vertical functor is the composite $\pr \circ (\mu\to t)$ as in \Cref{t100}.
Explicitly, the factorization is given by the functor
\begin{equation}
\label{e32}
\begin{tikzcd}[row sep=0cm]
\Phi^\nonu(\theta^* \cE)^\oi 
\arrow{r}{(\alpha^{\nonu})^{\oi}}
&
\Sq(\cE^\nonu)^\oi
\\
\rotatebox{90}{$\in$}
&
\rotatebox{90}{$\in$}
\\
(I \xlongra{\sigma} K \xlongra{e_\bullet} \theta^\ast \cE)
\arrow[maps to]{r}
&
(I  \xra{ \pr^{\oi}( e_\bullet) \circ \sigma } \cE )
\end{tikzcd}
~,
\end{equation}
where $(\theta^\ast \cE)^{\oi} \xra{\pr^{\oi}} \cE^{\oi}$ is the canonical projection over $\bDelta^{\op}$.\footnote{Note that this factorization indeed exists because by definition of morphisms in $\Phi^\nonu(\theta^* \cE)^\oi$, the composite functor $\Phi^\nonu(\theta^* \cE)^\oi \to \Sq(\cE^\nonu)^\oi$ carries 1-morphisms to 1-morphisms given by functors $S_{1,n} \to \cE^\nonu$ that carry vertical 1-morphisms to equivalences (see \Cref{t28}).}

Using the description of cocartesian morphisms in $\Phi^{\nonu}(\theta^\ast \cE)^{\oi}$ over $\bDelta^{\op}$ from \Cref{t25}, the morphism $(\alpha^{\nonu})^{\oi}$ lies in the subcategory 
$\coCart_{\bDelta^\op} \subset \Cat_{\cocart/\bDelta^\op}$.
Hence, we obtain a functor 
\[
\Phi^\nonu(\theta^* \cE)
\xra{\alpha^\nonu}
\cE^\nonu
\]
in $\iota_1 2\Cat_{/\rlax^\nonu(\cC)}$.\footnote{Informally, the corresponding functor $\Phi^\nonu (\theta^* \cE) \ra \cE^\nonu$ is given by composing 1- and 2-morphisms in $\cE^\nonu$.}
Using the description of cartesian 1-morphisms and cocartesian 2-morphisms in $\Phi^{\nonu}(\theta^\ast \cE)$ over $\rlax^{\nonu}(\cC)$ as in \Cref{t26}, we see that the functor $\alpha^{\nonu}$ preserves cartesian 1-morphisms and cocartesian 2-morphisms.
Hence, it also defines a natural morphism 
\[
\Phi(\theta^* \cE)
\xlongra{\alpha}
\cE
\]
in the full subcategory $\iota_1 2\Cart_{\rlax(\cC)} \subseteq \iota_1 2\Cart_{\rlax^\nonu(\cC)}$.
This functor $\alpha$ is evidently functorial in $\cE \in \iota_1 2\Cart_{\rlax(\cC)}$; that is, $\alpha$ defines a natural transformation $\Phi \theta^* \xra{\alpha} \id_{\iota_1 2\Cart_{\rlax(\cC)}}$.

It remains to show that $\alpha$ is an equivalence.  
By \Cref{obs.two.cart.fibns}\Cref{obs.strict.maps.among.locallyfibns.are.equivalences.if.fiberwise.so}, $\alpha$ is an equivalence provided it restricts as an equivalence between fibers over $\rlax(\cC)$.  
Because the right-lax functor $\cC \overset{\theta}\laxra \rlax(\cC)$ induces an equivalence between spaces of objects, it suffices to show that $\theta^\ast \alpha$ is an equivalence.
This follows from the above verification that $\theta^* \Phi \simeq \id_{\iota_1 \loc.2\Cart_\cC}$.
\end{proof}

\subsection{Cartesian yoga}
\label{subsection.cartesian.yoga}

In this subsection, we establish two enhancements of un/straightening (Theorems \ref{thm.switching.yoga} \and \ref{thm.unstraightening.yoga}), as well as an enhancement of lax un/straightening to an equivalence of $(\infty,2)$-categories (\Cref{t11}). We begin by stating the main results, and then prove them in turn at the end of this subsection based on supporting lemmas.

\begin{theorem}\label{thm.switching.yoga}
Let $\cC,\cD \in \iota_1 2\Cat$ be $(\infty,2)$-categories. Then, there is a natural equivalence
\[
\hom_{\iota_1 2\Cat_\rlax} ( \cC , 2\Cat_{\loc.2\cocart/\cD} )
\simeq
\hom_{\iota_1 2\Cat_\llax} ( \cD , 2\Cat_{\loc.2\cart/\cC^{1\op}} )
\]
in $\Spaces$.\footnote{Note that this equivalence does not generally extend to one between hom-$\infty$-categories.}
This equivalence is contravariantly functorial in both variables,\footnote{In particular, it restricts to an equivalence $\hom_{\iota_1 2\Cat_\rlax} ( \cC , 2\Cat_{\loc.1\cocart/\cD} ) \simeq \hom_{\iota_1 2\Cat_\llax} ( \cD , 2\Cat_{\loc.1\cart/\cC^{1\op}} )$.}
and which is the identity endofunctor of $2\Cat$ in the case that $\cC = \cD = \pt$.
Moreover, it restricts to equivalences
\begin{align*}
\hom_{\iota_1 2\Cat} ( \cC , 2\Cat_{\loc.2\cocart/\cD} )
& \simeq
\hom_{\iota_1 2\Cat_\llax} ( \cD , 2\Cat_{2\cart/\cC^{1\op}} )
~,
\\
\hom_{\iota_1 2\Cat_\rlax} ( \cC , 2\Cat_{2\cocart/\cD} )
& \simeq
\hom_{\iota_1 2\Cat} ( \cD , 2\Cat_{\loc.2\cart/\cC^{1\op}} )
~,
\qquad
\text{and}
\\
\hom_{\iota_1 2\Cat_\rlax} ( \cC , \loc.2\coCart_\cD )
& \simeq
\hom_{\iota_1 2\Cat_\llax} ( \cD , \loc.2\Cart_{\cC^{1\op}} )
~.
\end{align*}
\end{theorem}

\begin{observation}
\label{obs.pullback.along.rlax.fctrs.as.a.twofunctor}
Given a right-lax functor $\cC \overset{F}{\laxra} \cD$ between $(\infty,2)$-categories, \Cref{thm.switching.yoga} gives a functor
\[
2\Cat_{\loc.2\cart/\cD}
\xlongra{F^*}
2\Cat_{\loc.2\cart/\cC}
~,
\]
which restricts to a functor
\[
\loc.2\Cart_\cD
\xlongra{F^*}
\loc.2\Cart_\cC
~:
\]
namely, for any $(\infty,2)$-category $\cX \in \iota_1 2\Cat$, we obtain a natural morphism
\begin{align*}
\hom_{\iota_1 2\Cat} (\cX , 2\Cat_{\loc.2\cart/\cD} )
& \simeq
\hom_{\iota_1 2\Cat_\rlax} ( \cD^{1\op} , 2\Cat_{2\cocart / \cX} )
\\
& \longra
\hom_{\iota_1 2\Cat_\rlax} ( \cC^{1\op} , 2\Cat_{2\cocart / \cX} )
\\
& \simeq
\hom_{\iota_1 2\Cat} (\cX , 2\Cat_{\loc.2\cart/\cC} )
\end{align*}
in $\Spaces$. 
Evidently, this is contravariantly functorial in $\cC \in \iota_1 2\Cat_{\rlax}$.
\end{observation}

\begin{theorem}
\label{t11}
For any $(\infty,2)$-category $\cC \in \iota_1 2\Cat$, pullback along the universal right-lax functor $\cC \laxra \rlax(\cC)$ (using \Cref{obs.pullback.along.rlax.fctrs.as.a.twofunctor}) defines an equivalence
\[
2\Cat_{2\cart/\rlax(\cC)}
\xlongra{\sim}
2\Cat_{\loc.2\cart/\cC}
\]
in $2\Cat$. This equivalence is contravariantly functorial in $\cC \in \iota_1 2\Cat$ via pullback. 
In particular, it restricts to an equivalence
\[
2\Cart_{\rlax(\cC)}
\xlongra{\sim}
\loc.2\Cart_\cC
~.\footnote{Inspecting the proof, it is obvious that on $\iota_0$ this equivalence coincides with that of \Cref{thm.iota.one.version.of.loctwoCart.equivalent.to.twoCart.over.rlaxification}. In fact this is also true on $\iota_1$, but we do not use that here.}
\]
\end{theorem}

\begin{proof}
For any $(\infty,2)$-category $\cX \in \iota_1 2\Cat$, we have the sequence of equivalences
\begin{align}
\label{string.of.equivalences.proving.pullback.along.univ.rlax.fctr.is.equivce.first}
\hom_{\iota_1 2\Cat} ( \cX , 2\Cat_{2\cart/\rlax(\cC)} )
& \simeq
\hom_{\iota_1 2\Cat} ( \rlax(\cC)^{1\op} , 2\Cat_{2\cocart/\cX} )
\\
\label{string.of.equivalences.proving.pullback.along.univ.rlax.fctr.is.equivce.second}
& \simeq
\hom_{\iota_1 2\Cat} ( \rlax(\cC^{1\op}) , 2\Cat_{2\cocart/\cX} )
\\
\nonumber
& \simeq
\hom_{\iota_1 2\Cat_\rlax} ( \cC^{1\op} , 2\Cat_{2\cocart/\cX} )
\\
\label{string.of.equivalences.proving.pullback.along.univ.rlax.fctr.is.equivce.third}
& \simeq
\hom_{\iota_1 2\Cat}\left(
\cX 
,
2\Cat_{\loc.2\cart/\cC}
\right)
~,
\end{align}
in which equivalences \Cref{string.of.equivalences.proving.pullback.along.univ.rlax.fctr.is.equivce.first} \and \Cref{string.of.equivalences.proving.pullback.along.univ.rlax.fctr.is.equivce.third} follow from \Cref{thm.switching.yoga} and equivalence \Cref{string.of.equivalences.proving.pullback.along.univ.rlax.fctr.is.equivce.second} follows from \Cref{obs.rlax.commutes.with.oneop}.
\end{proof}

\begin{theorem}
\label{thm.unstraightening.yoga}
Let $\cC , \cD \in \iota_1 2\Cat$ be $(\infty,2)$-categories. Then, unstraightening gives a monomorphism
\[
\hom_{\iota_1 2\Cat_\llax} (\cC , 2\Cat_{\loc.2\cocart/\cD})
\longhookra
\iota_0 \loc.2\coCart_{\cC \times \cD}
\]
in $\Spaces$, whose image consists of those locally 2-cocartesian fibrations $(\cE \da (\cC \times \cD)) \in \iota_0 \loc.2\coCart_{\cC \times \cD}$ that satisfy the following condition.\footnote{The further subspace $\hom_{\iota_1 2\Cat_\llax}(\cC,2\Cat_{2\cocart/\cD}) \subseteq \hom_{\iota_1 2\Cat_\llax}(\cC , 2\Cat_{\loc.2\cocart/\cD})$ corresponds to the additional condition that for every object $c \in \cC$ the functor $\cE_c \ra \cD$ is a 2-cocartesian fibration, by the naturality of un/straightening.}
\begin{itemize}

\item[$(\ast$)]

For any pair of 1-morphisms $c \ra c'$ in $\cC$ and $d \ra d'$ in $\cD$, the pullback of $\cE \da (\cC \times \cD)$ along the functor $[2] \ra \cC \times \cD$ selecting the commutative triangle
\[ \begin{tikzcd}
(c,d)
\arrow{r}
\arrow{rd}
&
(c',d)
\arrow{d}
\\
&
(c',d')
\end{tikzcd} \]
is a 2-cocartesian fibration.\footnote{Recall \Cref{rmk.rlax.nat.trans.has.lax.squares}.}

\end{itemize}
Moreover, the further subspace
\[
\hom_{\iota_1 2\Cat} (\cC , 2\Cat_{\loc.2\cocart/\cD})
\subseteq
\hom_{\iota_1 2\Cat_\llax} (\cC , 2\Cat_{\loc.2\cocart/\cD})
\]
corresponds to the additional condition that for every object $d \in \cD$ the functor $\cE_d \ra \cC$ is a 2-cocartesian fibration.
\end{theorem}

\begin{observation}
\label{t12}
Let $\cC , \cD \in 2\Cat$, and recall the $(\infty,2)$-category $\Fun_{\rlax}^{\rlax}(\cC,\cD) \in 2\Cat$ from \Cref{d3}.
Then, for any fixed 2-cartesian fibration $\cE \da \cD$, pullback defines a functor
\begin{equation}
\label{e16}
\begin{tikzcd}[row sep=0cm, column sep=1.5cm]
\Fun^{\rlax}_{\rlax}(\cC,\cD)^{1\op}
\arrow{r}{(-)^* \cE}
&
2\Cat_{2\loc.\cart/\cC}
\\
\rotatebox{90}{$\in$}
&
\rotatebox{90}{$\in$}
\\
F
\arrow[maps to]{r}
&
(F^* \cE \da \cC)
\end{tikzcd}
~.
\end{equation}
Indeed, for each $([i],[j]) \in \bDelta \times \bDelta$, pullback of $(\cE \da \cD)$ defines a map
\[
\hom_{\iota_1 2\Cat_{\rlax} }
\left(
\cC \times \theta([i],[j]) 
,
\cD
\right)
\longrightarrow
\iota_0
2\Cat_{
2 \loc.\cart / \cC \times \theta([i],[j])
}
\]
in $\Spaces$; applying ($1\& 2 \op$ of) \Cref{thm.unstraightening.yoga}, this restricts to a presentation of the desired map.

The functor \Cref{e16} evidently has the following functorialities.
\begin{enumerate}

\item
\label{t12.part.1}
For any (strict) functor $\cD' \xra{G} \cD$, there is a canonical commutative triangle
\[ \begin{tikzcd}
\Fun^\rlax_\rlax(\cC,\cD')^{1\op}
\arrow{rr}{(-)^*(G^* \cE)}
\arrow{rd}[sloped, swap]{G \circ -}
&
&
2\Cat_{2\loc.\cart/\cC}
\\
&
\Fun^\rlax_\rlax(\cC,\cD)^{1\op}
\arrow{ru}[sloped, swap]{(-)^* \cE}
\end{tikzcd} \]
in $2\Cat$.

\item
\label{t12.part.2}
For any right-lax functor $\cC' \overset{H}{\laxra} \cC$, there is a canonical commutative triangle
\[ \begin{tikzcd}
\Fun^\rlax_\rlax(\cC,\cD)^{1\op}
\arrow{rr}{(-)^* \cE}
\arrow{rd}[sloped, swap]{- \circ H}
&
&
2\Cat_{2\loc.\cart/\cC'}
\\
&
\Fun^\rlax_\rlax(\cC',\cD)^{1\op}
\arrow{ru}[sloped, swap]{(-)^* \cE}
\end{tikzcd} \]
in $2\Cat$.

\end{enumerate}
\end{observation}

The remainder of this subsection is dedicated to the proofs of Theorems \ref{thm.switching.yoga} \and \ref{thm.unstraightening.yoga}, which rely on some preliminary results.

\begin{definition}
\label{defn.smushing}
Given a locally 2-co/cartesian fibration $\cE \xra{\pi} \cC$, we say that a functor $\cE \ra \cD$ \bit{smushes} $\pi$ if it carries all $\pi$-co/cartesian 1- and 2-morphisms to equivalences.
\end{definition}

\begin{lemma}\label{lemma.for.switching.yoga}
Suppose we are given a span $\cC \xla{F} \cE \xra{G} \cD$ in $2\Cat$ such that $F$ is a locally 2-cartesian fibration that $G$ smushes. Then, the following conditions are equivalent.
\begin{enumerate}

\item\label{condn.G.is.loc.two.cocart.and.F.smushes.it}

The functor $G$ is a (locally) 2-cocartesian fibration that $F$ smushes.

\item\label{condns.fibers.are.two.cocart.and.cart.mdrmy.preserves.cart.twomors}

The following two conditions are satisfied.
\begin{enumerate}[label=(\roman*)]

\item\label{condn.fibers.are.two.cocart}

For every object $c \in \cC$, the functor $\cE_c \ra \cD$ is a (resp.\! locally) 2-cocartesian fibration.

\item\label{condn.cart.mdrmy.preserves.cart.twomors}

For every 1-morphism $c_1 \ra c_2$ in $\cC$, the corresponding cartesian monodromy functor $\cE_{c_1} \la \cE_{c_2}$ (which lies over $\cD$ by the assumption that $G$ smushes $F$) lies in $2\Cat_{\loc.2\cocart/\cD} \subseteq 2\Cat_{/\cD}$ (i.e.\! it preserves cartesian 2-morphisms over $\cD$).

\end{enumerate}

\end{enumerate}
\end{lemma}

\begin{proof}
We will prove the version involving the two instances of the word ``locally''; in particular, we will explicitly identify the locally cocartesian 1-morphisms over $\cD$. Using this, the version without the word ``locally'' follows from \Cref{lemma.locallytwocart.is.twocart.if.compose}.

We begin by establishing the result in the case that $\cD \in \Cat \subset 2\Cat$ is an $(\infty,1)$-category. Since both conditions are compatible with pullback in the variable $\cD$, it suffices to consider the case that $\cD = [1]$.

Now, suppose first that the functor $\cE \da (\cC \times [1])$ satisfies condition \Cref{condn.G.is.loc.two.cocart.and.F.smushes.it}. Given an object $e_0 \in \cE_{(c,0)}$, observe that the cocartesian 1-morphism $e_0 \ra e_1$ in $\cE \da [1]$ lifting $0 \ra 1$ canonically lifts to a cocartesian 1-morphism in $\cE_c \da [1]$. Thus, condition \Cref{condn.G.is.loc.two.cocart.and.F.smushes.it} implies condition \Cref{condns.fibers.are.two.cocart.and.cart.mdrmy.preserves.cart.twomors} (assuming that $\cD \in \Cat$).

In the other direction, suppose that the functor $\cE \da (\cC \times [1])$ satisfies condition \Cref{condns.fibers.are.two.cocart.and.cart.mdrmy.preserves.cart.twomors}. By \Cref{obs.two.cart.fibns}\Cref{obs.two.cart.fibns.over.a.one.cat} it suffices to show that a cocartesian 1-morphism $e \ra e'$ in $\cE_c \da [1]$ lifting $0 \ra 1$ is also a cocartesian 1-morphism in $\cE \da [1]$. 
To see this, recall that $\cC \xla{F} \cE$ is a locally 2-cartesian fibration, and observe that since $G$ smushes $F$, the functor $\cE \la \cE_1$ is a morphism in $\loc.2\Cart_\cC$. Therefore, for any $e_1 \in \cE_1$ we obtain a commutative triangle
\begin{equation}
\label{comm.triangle.defining.morphism.of.cocart.fibns.over.homCFeFeone}
\begin{tikzcd}
\hom_\cE(e,e_1)
\arrow[leftarrow]{rr}
\arrow{rd}
&
&
\hom_{\cE_1} ( e' , e_1)
\arrow{ld}
\\
&
\hom_\cC(F(e),F(e_1))
\end{tikzcd}
\end{equation}
(recall that $F(e) = c$) that defines a morphism in $\coCart_{\hom_\cC(F(e),F(e_1))}$. Hence, to show that the horizontal functor in diagram \Cref{comm.triangle.defining.morphism.of.cocart.fibns.over.homCFeFeone} is an equivalence, it suffices to verify that it is an equivalence on fibers over an arbitrary object $\varphi \in \hom_\cC(F(e),F(e_1))$. Letting $e_1' \ra e_1$ be a cartesian 1-morphism in $\cE$ lifting the 1-morphism $F(e) \xra{\varphi} F(e_1)$ in $\cC$, we see that on fibers the horizontal functor in diagram \Cref{comm.triangle.defining.morphism.of.cocart.fibns.over.homCFeFeone} is the composite equivalence
\[
\hom_\cE^\varphi ( e , e_1 )
\simeq
\hom_{\cE_{F(e)}}^\varphi ( e , e_1' )
\simeq
\hom_{\cE_{(F(e),1)}} ( e' , e_1' )
\simeq
\hom_{\cE_1} ( e' , e_1 )
~.
\]
So indeed, condition \Cref{condns.fibers.are.two.cocart.and.cart.mdrmy.preserves.cart.twomors} implies 
condition \Cref{condn.G.is.loc.two.cocart.and.F.smushes.it} (assuming that $\cD \in \Cat$).

We now consider the case that $\cD \in 2\Cat$ is an arbitrary $(\infty,2)$-category. Note that conditions \Cref{condn.G.is.loc.two.cocart.and.F.smushes.it} and \Cref{condns.fibers.are.two.cocart.and.cart.mdrmy.preserves.cart.twomors} refer to both 1- and 2-morphisms, and the above special case establishes the equivalence of the conditions on 1-morphisms. Thus,
the relevant functors to $\cD$ are locally 2-cocartesian fibrations if and only if they are strict homwise cartesian fibrations.

Now, fix any objects $e,e' \in \cE$ and let us respectively write $c,c' \in \cC$ and $d,d' \in \cD$ for their images under $F$ and $G$ respectively. Then, we obtain a span
\[
\hom_\cC(c,c')
\xla{F_{e,e'}}
\hom_\cE(e,e')
\xra{G_{e,e'}}
\hom_\cD(d,d')
\]
in $\Cat$ in which $F_{e,e'}$ is a cocartesian fibration that $G_{e,e'}$ smushes. By ($(-)^{1 \& 2\op}$ applied to) the above special case, the functor $G_{e,e'}$ is a cartesian fibration that $F_{e,e'}$ smushes if and only if for every object $\varphi \in \hom_\cC(c,c')$ the functor $\hom^\varphi_\cE(e,e') \ra \hom_\cD(d,d')$ is a cartesian fibration. This latter condition holds if and only if for every object $e'' \in \cE_c$ the functor $\hom_{\cE_c} (e,e'') \ra \hom_\cD(G(e),G(e''))$ is a cartesian fibration: indeed, given an object $\varphi \in \hom_\cC(c,c')$, letting $e'' \ra e'$ be a locally cartesian 1-morphism lift in $\cE$ we obtain a commutative triangle
\[ \begin{tikzcd}
\hom^\varphi_\cE(e,e')
\arrow{rd}
&
&
\hom_{\cE_c}(e,e'')
\arrow{ll}[swap]{\sim}
\arrow{ld}
\\
&
\hom_\cD(d,d')
\end{tikzcd} \]
in $\Cat$ using the 2-cartesianness of $\cE$ over $\cC$, so that in particular one functor is a cartesian fibration if and only if the other is. Unwinding the definitions, we see that $\cE \xra{G} \cD$ is a strict homwise locally cartesian fibration if and only if the functors $\cE_c \ra \cD$ are such for all objects $c \in \cC$ and moreover condition \Cref{condns.fibers.are.two.cocart.and.cart.mdrmy.preserves.cart.twomors}\Cref{condn.cart.mdrmy.preserves.cart.twomors} is satisfied, which proves the claim.
\end{proof}

\begin{proof}[Proof of \Cref{thm.switching.yoga}]
Fix a right-lax functor $\cC \laxra 2\Cat_{\loc.2\cocart/\cD}$. By lax (cartesian) unstraightening, this is equivalent data to a morphism
\[ \begin{tikzcd}
\cE
\arrow{rr}{(F,G)}
\arrow{rd}[sloped, swap]{F}
&
&
\cC^{1\op} \times \cD
\arrow{ld}[sloped, swap]{\pr}
\\
&
\cC^{1\op}
\end{tikzcd} \]
in $\loc.2\Cart_{\cC^{1\op}}$ such that the upper horizontal functor satisfies condition\Cref{condns.fibers.are.two.cocart.and.cart.mdrmy.preserves.cart.twomors} of \Cref{lemma.for.switching.yoga}. Therefore, by \Cref{lemma.for.switching.yoga}, $G$ is a locally 2-cocartesian fibration that $F$ smushes. By lax (cocartesian) straightening and ($(-)^{1\&2\op}$ applied to) \Cref{lemma.for.switching.yoga},\footnote{That is, we interchange the words ``cartesian'' and ``cocartesian'' in the statement of \Cref{lemma.for.switching.yoga}.} this is equivalent data to a left-lax functor $\cD \laxra 2\Cat_{\loc.2\cart/\cC^{1\op}}$. The functoriality in $\cC,\cD \in \iota_1 2\Cat^\op$ is clear. Moreover, the first two specializations are evident from the construction. The third specialization is implemented by imposing the following condition: for any 1-morphisms $c \ra c'$ in $\cC^{1\op}$ and $d \ra d'$ in $\cD$ as well as a lifted commutative square
\[
\begin{tikzcd}
e
\arrow{r}{\alpha}
\arrow{d}[swap]{\beta}
&
e'
\arrow{d}{\gamma}
\\
e''
\arrow{r}[swap]{\delta}
&
e'''
\end{tikzcd}
\longmapsto
\begin{tikzcd}
(c,d)
\arrow{r}
\arrow{d}
&
(c',d)
\arrow{d}
\\
(c,d')
\arrow{r}
&
(c',d')
\end{tikzcd} \]
in $\cE$, such that $\alpha$ is locally cartesian and $\gamma$ is locally cocartesian, then $\beta$ is locally cocartesian if and only if $\delta$ is locally cartesian.\footnote{Here, we mean e.g.\! that $\alpha$ is locally cartesian with respect to the composite $\cE \ra \cC^{1\op} \times \cD \ra \cC^{1\op}$, or equivalently with respect to the functor $\cE_d \ra \cC^{1\op}$.}
\end{proof}

\begin{lemma}\label{lemma.for.unstraightening.yoga}
Suppose we are given a span $\cC \xla{F} \cE \xra{G} \cD$ in $2\Cat$. Then, the following conditions are equivalent.
\begin{enumerate}

\item\label{condn.functor.to.product.is.loc.two.cocart}

The the functor $\cE \xra{(F,G)} \cC \times \cD$ is a locally 2-cocartesian fibration that satisfies condition $(\ast)$ of \Cref{thm.unstraightening.yoga}.

\item\label{condns.in.cocart.cocart}

The following three conditions are satisfied.
\begin{enumerate}[label=(\roman*)]

\item\label{condn.F.locally.twocart.and.G.smushes.F}

The functor $F$ is a locally 2-cocartesian fibration that $G$ smushes.

\item\label{condn.fibers.are.two.cocart.in.symmetric.case}

For every object $c \in \cC$, the functor $\cE_c \ra \cD$ is a locally 2-cocartesian fibration.

\item\label{condn.cocart.mdrmy.preserves.cart.twomors}

For every 1-morphism $c_1 \ra c_2$ in $\cC$, the corresponding cocartesian monodromy functor $\cE_{c_1} \ra \cE_{c_2}$ (which lies over $\cD$ by the assumption that $G$ smushes $F$) lies in $2\Cat_{\loc.2\cocart/\cD}$ (i.e.\! it preserves cartesian 2-morphisms over $\cD$).

\end{enumerate}

\end{enumerate}
\end{lemma}

\begin{proof}
Suppose that condition \Cref{condn.functor.to.product.is.loc.two.cocart} is satisfied.

First of all, clearly condition \Cref{condns.in.cocart.cocart}\Cref{condn.fibers.are.two.cocart.in.symmetric.case} is satisfied.

We now establish condition \Cref{condns.in.cocart.cocart}\Cref{condn.F.locally.twocart.and.G.smushes.F}. For this, fix any object $e \in \cE$, write $(c,d) := (F(e),G(e)) \in \cC \times \cD$ for its image, and fix a 1-morphism $c \xra{\varphi} c'$ in $\cC$. Let $e \ra e'$ be the locally cocartesian 1-morphism lift in $\cE$ of the 1-morphism $(c,d) \xra{(\varphi,\id_d)} (c',d)$ in $\cC \times \cD$. We claim that this is also a locally cocartesian 1-morphism lift of the 1-morphism $c \xra{\varphi} c'$ in $\cC$. To see this, fix an object $f \in \cE_{c'}$, and write $d' := G(f) \in \cD$ for its image. Then, we have a commutative triangle
\begin{equation}
\label{comm.triangle.in.Cat.defining.a.morphism.in.CarthomDdddoubleprime}
\begin{tikzcd}
\hom_{\cE_{c'}}(e',f)
\arrow{rr}
\arrow{rd}
&
&
\hom^\varphi_\cE(e,f)
\arrow{ld}
\\
&
\hom_\cD(d,d')
\end{tikzcd}
\end{equation}
in $\Cat$ defining a morphism in $\Cart_{\hom_\cD(d,d')}$ (using that $(F,G)$ is a locally 2-cocartesian fibration, so that its pullback along $[1] \times \cD \xra{\varphi \times \id_\cD} \cC \times \cD$ is as well). So to verify that the upper horizontal functor in diagram \Cref{comm.triangle.in.Cat.defining.a.morphism.in.CarthomDdddoubleprime} is an equivalence, it suffices to check that it is an equivalence on fibers over an arbitrary object $\psi \in \hom_\cD(d,d')$. For this, let $e' \ra e''$ be a locally cocartesian 1-morphism lift in $\cE$ of the 1-morphism $(c',d) \xra{(\id_{c'},\psi)} (c',d')$ in $\cC \times \cD$. Then, we have equivalences
\[
\hom^\psi_{\cE_{c'}}(e',f)
\simeq
\hom_{\cE_{(c',d')}}(e'',f)
\simeq
\hom^{(\varphi,\psi)}_\cE(e,f)
~,
\]
the latter by condition $(\ast)$ of \Cref{thm.unstraightening.yoga}. It follows that the functor $\cE \xra{F} \cC$ admits locally cocartesian 1-morphism lifts. To see that it is a homwise cartesian fibration, observe that for any pair of objects $e,e' \in \cE$, writing $(c,d),(c',d') \in \cC \times \cD$ for their images under $(F,G)$, in the composite
\[
\hom_\cE(e,e')
\longra
\hom_{\cC \times \cD}( (c,d) , (c',d') )
\simeq
\hom_\cC(c,c') \times \hom_\cD(d,d')
\longra
\hom_\cC(c,c')
\]
the first functor is a cartesian fibration by assumption and hence the composite is as well. From here, the fact that it is a strict homwise cartesian fibration follows from the fact that $(F,G)$ is. Moreover, it follows immediately from these considerations that $G$ smushes $F$.

We now establish condition \Cref{condns.in.cocart.cocart}\Cref{condn.cocart.mdrmy.preserves.cart.twomors}. Choose any 1-morphism $c \xra{\varphi} c'$ in $\cC$ and objects $e,f \in \cE_c$, and let $e \ra e'$ and $f \ra f'$ be locally cocartesian 1-morphism lifts in $\cE$ of $\varphi$. Then, the locally cocartesian monodromy functor $\cE_c \ra \cE_{e'}$ acts on hom-$\infty$-categories via the diagram
\begin{equation}
\label{zigzag.for.showing.preservation.of.cartesian.twomors}
\hom_{\cE_c}(e,f)
\longra
\hom_\cE^\varphi(e,f')
\xlongla{\sim}
\hom_{\cE_{c'}}(e',f')
\end{equation}
in $\Cat$, which lies over $\hom_\cD(G(e'),G(f'))$ because $G$ smushes $F$. Hence, condition \Cref{condns.in.cocart.cocart}\Cref{condn.cocart.mdrmy.preserves.cart.twomors} follows from the fact that the left functor in diagram \Cref{zigzag.for.showing.preservation.of.cartesian.twomors} preserves cartesian morphisms because $(F,G)$ is a strict homwise cartesian fibration.

Now, suppose that condition \Cref{condns.in.cocart.cocart} is satisfied.

Fix a 1-morphism $(c,d) \ra (c',d')$ in $\cC \times \cD$ and an object $e \in \cE_{(c,d)}$. Let $e \ra e'$ be a locally $F$-cocartesian 1-morphism lift of the 1-morphism $c \ra c'$ in $\cC$; this lies over the object $d \in \cD$ because $G$ smushes $F$ (by condition \Cref{condns.in.cocart.cocart}\Cref{condn.F.locally.twocart.and.G.smushes.F}). Then, let $e' \ra e''$ be a locally cocartesian 1-morphism lift of $d \ra d'$ with respect to the locally 2-cocartesian fibration $\cE_{c'} \ra \cD$ (using condition \Cref{condns.in.cocart.cocart}\Cref{condn.fibers.are.two.cocart.in.symmetric.case}). By conditions \Cref{condns.in.cocart.cocart}\Cref{condn.F.locally.twocart.and.G.smushes.F} and \Cref{condns.in.cocart.cocart}\Cref{condn.fibers.are.two.cocart.in.symmetric.case}, the composite $e \ra e' \ra e''$ is a locally $(F,G)$-cocartesian 1-morphism. 
Thus, locally $(F,G)$-cocartesian 1-morphism lifts exist, and moreover condition $(\ast)$ of \Cref{thm.unstraightening.yoga} is satisfied.

Now, choose any pair of objects $e,e' \in \cE$, and respectively write $(c,d),(c',d') \in \cC \times \cD$ for their images. Then, we have a span
\[
\hom_\cC(c,c')
\xla{F_{e,e'}}
\hom_\cE(e,e')
\xra{G_{e,e'}}
\hom_\cD(d,d')
\]
in $\Cat$ in which $F$ is a cartesian fibration that $G$ smushes. Now, choose any 1-morphism $c \xra{\varphi} c'$ in $\cC$, and write $e \ra e_1$ for its locally $F$-cocartesian 1-morphism lift in $\cE$. Then, we have an equivalence
\[
\hom^\varphi_\cE(e,e')
\simeq
\hom_{\cE_{c_1}}(e_1,e')
\]
in $\Cat$, which lies over $\hom_\cD(d,d')$ because $G$ smushes $F$. Thus, the functor $\hom^\varphi_\cE(e,e') \ra \hom_\cD(d,d')$ is a cartesian fibration. Moreover, for every morphism $\varphi_1 \ra \varphi_2$ in $\hom_\cC(c,c')$, the corresponding cartesian monodromy functor $\hom^{\varphi_1}_\cE(e,e') \la \hom^{\varphi_2}_\cE(e,e')$ defines a morphism in $\Cart_{\hom_\cD(d,d')}$: it lies over $\hom_\cD(d,d')$ since $G$ smushes $F$, and moreover preserves cartesian morphisms since $F$ is a strict homwise cartesian fibration. It follows that the functor
\[
\hom_\cE(e,e')
\longra
\hom_\cC(c,c')
\times
\hom_\cD(d,d')
\simeq
\hom_{\cC \times \cD} ( (c,d) , (c',d') )
\]
is a cartesian fibration, i.e.\! that $\cE \xra{(F,G)} \cC \times \cD$ is a homwise cartesian fibration. In particular, $\cE \da \cD$ is a homwise cartesian fibration.  
Explicitly, a 2-morphism $\alpha$ in $\cE$ is cartesian over $\cD$ if and only if it can be expressed as a composite
\begin{equation}
\label{e61}
\begin{tikzcd}
e
\arrow{r}{\varphi}
&
e_0
\arrow[bend left]{r}[pos=0.47]{}[swap, yshift=-0.5ex, xshift=0.4ex]{\Downarrow\alpha'}
\arrow[bend right]{r}[swap, pos=0.47]{}
&
e'
\end{tikzcd} 
\end{equation}
in which the 1-morphism $\varphi$ is locally cocartesian over $\cC$ and the 2-morphism $\alpha'$ is the image of a 2-morphism in $\cE_{F(e')}$ that is cartesian over $\cD$.

To see that $\cE \xra{(F,G)} \cC \times \cD$ is a strict homwise cartesian fibration, we observe that $\cE \xra{F} \cC$ is such, so that it suffices to show that $\cE \xra{G} \cD$ is such as well.
That is, we must show that the cartesian 2-morphisms in $\cE$ over $\cD$ are closed under precomposition and postcomposition by (arbitrary) 1-morphisms in $\cE$.
Closure under precomposition follows from the fact that each 1-morphism in $\cE$ factors as a 1-morphism in $\cE$ that is locally cocartesian over $\cC$ followed by a 1-morphism in $\cE$ that lies in a fiber over $\cC$, together with condition \Cref{condns.in.cocart.cocart}\Cref{condn.fibers.are.two.cocart.in.symmetric.case} (namely that for each $c\in \cC$, the functor $\cE_c \to \cD$ is a strict homwise cartesian fibration).
Using closure under precomposition, closure under postcomposition reduces to the case in which the 1-morphism $\varphi$ in diagram \Cref{e61} is an equivalence.
Let $\alpha$ be a 2-morphism in $\cE$ that is cartesian over $\cD$, and let $e' \xra{\psi} e''$ be a 1-morphism in $\cE$.
We must show that the 2-morphism $\psi \circ \alpha$ in $\cE$ is cartesian over $\cD$.
By condition \Cref{condns.in.cocart.cocart}\Cref{condn.fibers.are.two.cocart.in.symmetric.case}, the functor $\cE_{F(e')} \da \cD$ is a homwise cartesian fibration.
Therefore, the 2-morphism $\psi \circ \alpha$ is cartesian over $\cD$ in the case that $\psi$ is in the image of $\cE_{F(e')}$.
Since $\psi$ is a composite of a 1-morphism in $\cE$ that is locally cocartesian over $\cC$ followed by a 1-morphism in $\cE_{F(e'')}$, 
we can reduce to the case in which $\psi$ is locally cocartesian over $\cC$.
In this case, $\psi \circ \alpha$ is cartesian over $\cD$ by condition \Cref{condns.in.cocart.cocart}\Cref{condn.cocart.mdrmy.preserves.cart.twomors}.
\end{proof}

\begin{proof}[Proof of \Cref{thm.unstraightening.yoga}]
Fix a left-lax functor $\cC \laxra 2\Cat_{\loc.2\cocart/\cD}$. By lax unstraightening, this is equivalent data to a morphism
\[ \begin{tikzcd}
\cE
\arrow{rr}
\arrow{rd}
&
&
\cC \times \cD
\arrow{ld}
\\
&
\cC
\end{tikzcd} \]
in $\loc.2\coCart_\cC$ such that the upper horizontal functor satisfies condition \Cref{condns.in.cocart.cocart} of \Cref{lemma.for.unstraightening.yoga}. The first assertion now follows from \Cref{lemma.for.unstraightening.yoga}. For the second assertion, it suffices to observe that the functor $\cE \ra \cC$ is a 2-cocartesian fibration if and only if the functors $\cE_d \ra \cC$ are 2-cocartesian fibrations for all $d \in \cD$, which follows from \Cref{lemma.locallytwocart.is.twocart.if.compose} and the fact that $\cE \ra \cD$ smushes $\cE \ra \cC$.
\end{proof}

\subsection{Adjunctions}
\label{subsection.Adj}

In this subsection, we discuss adjunctions in $(\infty,2)$-categories (including in $2\Cat$). Our main results are Lemmas \ref{lemma.twocategorical.passing.to.adjoints.and.swapping.laxness} \and \ref{lemma.twocategorical.extn.to.an.adjn.in.twoCatonecartoverC}, which give parametrized versions of the mate correspondence; their proofs are adapted from \cite[Chapter 12, \S\S 3-4]{GR}.

\begin{local}
Throughout this subsection, we fix $(\infty,2)$-categories $\cC,\cD \in 2\Cat$.
\end{local}

\begin{definition}
\label{defn.ladjt}
A 1-morphism $c \xra{L} d$ in $\cC$ is a \bit{left adjoint} if there exist a 1-morphism $c \xla{R} d$ (a \bit{right adjoint}) and 2-morphisms $\id_c \xra{\eta} RL$ and $LR \xra{\varepsilon} \id_d$ (a \bit{unit} and \bit{counit}, respectively) such that the composite 2-morphisms
\begin{equation}
\label{two.morphisms.for.Zorro}
L
\simeq
L \id_c
\xra{\id_L \eta}
LRL
\xra{\varepsilon \id_L}
\id_d L
\simeq
L
\qquad
\text{and}
\qquad
R
\simeq
\id_c R
\xra{\eta \id_R}
RLR
\xra{\id_R \varepsilon}
R \id_d
\simeq
R
\end{equation}
are homotopic to identity 2-morphisms. We write $\Adj \in \iota_1 2\Cat$ for the object corepresenting (the space of) left adjoints.\footnote{Indeed, assigning to an $(\infty,2)$-category the space of left adjoints in it assembles as a functor $\iota_1 2\Cat \to \Spaces$ (a subfunctor of that corepresented by $[1]$);
this functor evidently preserves limits and filtered colimits.
Hence, the presentability of $\iota_1 2\Cat$ implies this functor is indeed corepresented by an $(\infty,2)$-category.
(For an explicit description thereof, see \cite{RV-adjns}.)} We generally consider $\Adj \in 2\Cat_{[1]/}$ via the epimorphism $[1] \ra \Adj$ corepresenting the universal left adjoint. Dually, we say that a 1-morphism $c \xla{R} d$ in $\cC$ is a \bit{right adjoint} if there exist $(L,\eta,\varepsilon)$ as above.
\end{definition}

\begin{observation}
It is immediate from \Cref{defn.ladjt} that we have a canonical equivalence $\Adj \simeq \Adj^{1\&2\op}$, and moreover that $\Adj^{1\op} \simeq \Adj^{2\op}$ corepresents the space of right adjoints. Furthermore, it follows e.g.\! from the description of $\Adj \in 2\Cat$ given in \cite{RV-adjns} that there is also a canonical equivalence $\Adj \simeq \Adj^{1\op}$; in particular, adjoints are unique when they exist. We use these facts without further comment.
\end{observation}

\begin{definition}
We write
\[
2\biCart_\cC := 2\coCart_\cC \cap 2\Cart_\cC
\qquad
\text{and}
\qquad
2\Cat_{2\bicart/\cC} := 2\Cat_{2\cocart/\cC} \cap 2\Cat_{2\cart/\cC}
\]
and refer to objects of these $(\infty,2)$-categories as \bit{2-bicartesian fibrations} over $\cC$.
\end{definition}

\begin{lemma}
\label{basic.adjns.lemma}
Pullback along the functor $[1] \ra \Adj$ defines an equivalence
\[
\hom_{\iota_1 2\Cat}(\Adj,2\Cat)
\xlongra{\sim}
\iota_0 2\biCart_{[1]}
~.
\]
\end{lemma}

\begin{proof}
By straightening, the pullback functor $\iota_0 2\coCart_{\Adj} \ra \iota_0 2\coCart_{[1]}$ is a monomorphism. We show that it is an equivalence onto the subspace $\iota_0 2\biCart_{[1]} \subset \iota_0 2\coCart_{[1]}$.

Suppose first that we are given a functor $[1] \ra 2\Cat$ that selects a left adjoint $\cC \xra{L} \cD$ in $2\Cat$. Consider the cocartesian unstraightening $\cE \da [1]$ of $L$. Using any adjunction data $(R,\eta,\varepsilon)$ extending $L$, we show that $\cE \da [1]$ is a 2-cartesian fibration. Given any object $d \in \cD$, the morphism $LRd \xra{\varepsilon} d$ in $\cD$ determines a morphism $Rd \ra d$ in $\cE$, which it is easy to see is cartesian. Hence, by \Cref{obs.two.cart.fibns}\Cref{obs.two.cart.fibns.over.a.one.cat}, $\cE \da [1]$ is a 2-cartesian fibration.

In the other direction, suppose that $\cE \da [1]$ is a 2-bicartesian fibration. We must show that its cocartesian unstraightening $[1] \xra{F} 2\Cat$ selects a left adjoint. Let us write $\cC \xra{L} \cD$ for the 1-morphism selected by its cocartesian straightening and $\cC \xla{R} \cD$ for the 1-morphism selected by its cartesian straightening. Now, an evident composite $[2] \times [2] \xra{G} [1] \xra{F} 2\Cat$ classifies a diagram
\[
\begin{tikzcd}
\cC
\arrow{r}{=}
\arrow{d}[sloped, anchor=north]{=}
&
\cC
\arrow{r}{=}
\arrow{d}[sloped, anchor=north]{=}
&
\cC
\arrow{d}{L}
\\
\cC
\arrow{r}{=}
\arrow{d}[sloped, anchor=north]{=}
&
\cC
\arrow{r}{L}
\arrow{d}[swap]{L}
&
\cD
\arrow{d}[sloped, anchor=south]{=}
\\
\cC
\arrow{r}[swap]{L}
&
\cD
\arrow{r}[swap]{=}
&
\cD
\end{tikzcd}
\]
in $2\Cat$.\footnote{Here and in what follows, we orient our diagrams according to the same conventions as matrices (rows (top to bottom) before columns (left to right)).} Write $\w{\cE} := G^* \cE$. Clearly, the canonical functor $\w{\cE} \da [2] \times [2]$ is a 2-bicartesian fibration. Applying cocartesian straightening in the second coordinate, we obtain the first functor in the composite
\[
[2]
\longra
2\Cat_{2\bicart/[2]}
\xlongra{\fgt}
2\Cat_{2\cart/[2]}
~,
\]
which by \Cref{thm.unstraightening.yoga} gives an object of $2\Cat_{\loc.2\cart/([2]^\op \times [2])}$ that is classified by a diagram
\begin{equation}
\label{two.by.two.grid.in.which.Zorro.identities}
\begin{tikzcd}
\cC
\arrow[leftarrow]{r}{=}
\arrow{d}[sloped, anchor=north]{=}
&
\cC
\arrow[leftarrow]{r}{=}[description, xshift=0.1cm, yshift=-0.5cm]{\rotatebox{-45}{$\xRightarrow{\eta}$}}
\arrow{d}[sloped, anchor=north]{=}
&
\cC
\arrow{d}{L}
\\
\cC
\arrow[leftarrow]{r}{=}[description, xshift=0.1cm, yshift=-0.5cm]{\rotatebox{-45}{$\xRightarrow{\eta}$}}
\arrow{d}[sloped, anchor=north]{=}
&
\cC
\arrow[leftarrow]{r}{R}[description, xshift=0.1cm, yshift=-0.5cm]{\rotatebox{-45}{$\xRightarrow{\varepsilon}$}}
\arrow{d}[swap]{L}
&
\cD
\arrow{d}[sloped, anchor=south]{=}
\\
\cC
\arrow[leftarrow]{r}[swap]{R}
&
\cD
\arrow[leftarrow]{r}[swap]{=}
&
\cD
\end{tikzcd}
\qquad
:=
\qquad
\begin{tikzcd}
\cC
\arrow[leftarrow]{r}{=}
\arrow{d}[sloped, anchor=north]{=}
&
\cC
\arrow[leftarrow]{r}{=}[description, xshift=0.3cm, yshift=-0.8cm]{\rotatebox{-45}{$\Rightarrow$}}
\arrow{d}[sloped, anchor=north]{=}
\arrow{ld}[sloped]{=}
&
\cC
\arrow{d}{L}
\arrow{ld}[sloped]{=}
\\
\cC
\arrow[leftarrow]{r}{=}[description, xshift=0.3cm, yshift=-0.8cm]{\rotatebox{-45}{$\Rightarrow$}}
\arrow{d}[sloped, anchor=north]{=}
&
\cC
\arrow[leftarrow]{r}{R}[description, xshift=0.3cm, yshift=-0.8cm]{\rotatebox{-45}{$\Rightarrow$}}
\arrow{d}[swap]{L}
\arrow{ld}[sloped]{=}
&
\cD
\arrow{d}[sloped, anchor=south]{=}
\arrow{ld}[sloped]{LR}
\\
\cC
\arrow[leftarrow]{r}[swap]{R}
&
\cD
\arrow[leftarrow]{r}[swap]{=}
&
\cD
\end{tikzcd}
\end{equation}
in $2\Cat$; in particular, this defines the 2-morphisms $\eta$ and $\varepsilon$ as (straightenings of) locally 2-cartesian fibrations over $[2]$ via appropriate functors $[2] \ra ([2]^\op \times [2])$. Moreover, the requisite identifications of the composite 2-morphisms \Cref{two.morphisms.for.Zorro} follow from the functoriality of un/straightening with respect to functors $([1]^\op \times [1]) \ra ([2]^\op \times [2])$, namely those selecting the right half and the bottom half of diagram \Cref{two.by.two.grid.in.which.Zorro.identities} respectively.
\end{proof}

\begin{cor}
\label{cor.existence.of.radjt.detectable.ptwise}
A functor $\cC \xra{L} \cD$ in $2\Cat$ is a left adjoint if and only if for every object $d \in \cD$ the functor $\hom_\cD(L(-),d) \in \Fun(\cC^{1\op},\Cat)$ is representable, i.e.\! there exist an object $c \in \cC$ and a 1-morphism $L(c) \ra d$ in $\cD$ such that the composite
\[
\hom_\cC(c',c)
\longra
\hom_\cD(L(c'),L(c))
\longra
\hom_\cD(L(c'),d)
\]
is an equivalence in $\Cat$ for all $c' \in \cC$.
\end{cor}

\begin{proof}
This is immediate from \Cref{basic.adjns.lemma}.
\end{proof}

\begin{lemma}
\label{lemma.twocategorical.passing.to.adjoints.and.swapping.laxness}
The datum of a morphism
\begin{equation}
\label{morphism.in.llax.over.arbitrary.twocat.for.first.adjn.lemma}
\cE_0
\longla
\cE_1
\end{equation}
in $2\Cat_{1\cocart/\cC}$ that on fibers over each object $c \in \cC$ is a right adjoint is equivalent to the datum of a morphism
\begin{equation}
\label{morphism.in.rlax.over.arbitrary.twocat.for.first.adjn.lemma}
(\cE_0)^\cocartdual
\longra
(\cE_1)^\cocartdual
\end{equation}
in $2\Cat_{1\cart/\cC^{1\op}}$ that on fibers over each object $c \in \cC^{1\op}$ is a left adjoint, with the equivalence given fiberwise by passing to adjoints.\footnote{Here, just as for $(\infty,1)$-categories we write $(-)^\cocartdual$ for the cocartesian duality equivalence $2\coCart_\cC \xra{\sim} 2\Cart_{\cC^{1\op}}$.}
\end{lemma}

\begin{proof}
Let us consider the morphism \Cref{morphism.in.llax.over.arbitrary.twocat.for.first.adjn.lemma} as a functor
\[
[1]^\op
\longra
2\Cat_{1\cocart/\cC}
~.
\]
By \Cref{thm.switching.yoga}, this is equivalent data to a functor
\[
\cC
\longra
2\Cat_{1\cart/[1]}
\simeq
1\Cat_{1\cart/[1]}
=
\Cat_{\cart/[1]}
\]
that factors through the subcategory $\Cat_{\bicart/[1]} := \Cat_{\cocart/[1]} \cap \Cat_{\cart/[1]} \subset \Cat_{\cart/[1]}$. Applying \Cref{thm.switching.yoga} to the resulting composite
\[
\cC
\longra
\Cat_{\bicart/[1]}
\xra{\fgt}
\Cat_{\cocart/[1]}
\simeq
2\Cat_{1\cocart/[1]}
~,
\]
we obtain a functor
\[
[1]
\longra
2\Cat_{1\cart/\cC^{1\op}}
~,
\]
which selects the desired morphism \Cref{morphism.in.rlax.over.arbitrary.twocat.for.first.adjn.lemma}. It is now clear that this construction indeed defines an equivalence of spaces.
\end{proof}

\begin{remark}
A more restrictive version of \Cref{lemma.twocategorical.passing.to.adjoints.and.swapping.laxness} is proved in \cite{HHLN-laxmonadjns-twovarfibns-calcmates}.
\end{remark}

\begin{lemma}
\label{lemma.twocategorical.extn.to.an.adjn.in.twoCatonecartoverC}
Given a solid diagram
\begin{equation}
\label{extn.for.second.adjn.lemma}
\begin{tikzcd}
{[1]^\op}
\arrow{r}{F}
\arrow[hook]{d}
&
2\Cat_{1\cart/\cC}
\\
\Adj^{1\op}
\arrow[dashed]{ru}
\end{tikzcd}
~,
\end{equation}
there exists an extension (i.e.\! $F$ selects a right adjoint in $2\Cat_{1\cart/\cC}$) if and only if the following conditions are satisfied.
\begin{enumerate}

\item\label{condition.fiberwise.radjt}

On fibers over each object $c \in \cC$, the functor $F$ selects a right adjoint.

\item\label{condition.morphism.in.twocatonecartoverC.is.strict}

The functor $F$ factors through the subcategory $1 \Cart_\cC \subseteq 2\Cat_{1\cart/\cC}$.

\end{enumerate}
\end{lemma}

\begin{proof}
Suppose first that there exists an extension \Cref{extn.for.second.adjn.lemma}. It is clear that condition \Cref{condition.fiberwise.radjt} is satisfied. To verify condition \Cref{condition.morphism.in.twocatonecartoverC.is.strict}, let us write 
\[ \begin{tikzcd}[column sep=1.5cm]
\cE
\arrow[transform canvas={yshift=0.9ex}]{r}{L}
\arrow[leftarrow, transform canvas={yshift=-0.9ex}]{r}[yshift=-0.2ex]{\bot}[swap]{R}
&
\cF
\end{tikzcd} \]
for the adjunction in $2\Cat_{1\cart/\cC}$ (leaving the functors to $\cC$ implicit). We must show that $R$ preserves cartesian 1-morphisms over $\cC$. For this, by taking pullback along an arbitrary functor $[1] \ra \cC$ it suffices to consider the case that $\cC = [1]$. Given a cartesian 1-morphism $f_0 \ra f_1$ in $\cF$ over the morphism $0 \ra 1$ in $[1]$, we must show that the 1-morphism $R(f_0) \ra R(f_1)$ in $\cE$ is also cartesian. For any object $e \in \cE_0$ we have the commutative diagram
\[ \begin{tikzcd}[row sep=0cm]
\hom_{\cE_0}(e,R(f_0))
\arrow{r}
&
\hom_\cE(e,R(f_1))
\\
\rotatebox{90}{$\simeq$}
&
\rotatebox{90}{$\simeq$}
\\
\hom_{\cF_0}(L(e),f_0)
\arrow{r}[swap]{\sim}
&
\hom_\cF(L(e),f_1)
\end{tikzcd}
~, \]
in which the vertical equivalences arise from the adjunctions $L \adj R$ and $L_0 \adj R_0$ and the lower morphism is an equivalence since $f_0 \ra f_1$ is cartesian.

Now, suppose that conditions \Cref{condition.fiberwise.radjt} and \Cref{condition.morphism.in.twocatonecartoverC.is.strict} are satisfied. We must show that there exists an extension
\begin{equation}
\label{wts.extn.over.Adjoneop.into.twocatonecartoverC}
\begin{tikzcd}
{[1]^\op}
\arrow{r}{F}
\arrow[hook]{d}
&
1\Cart_\cC
\arrow[hook]{d}
\\
\Adj^{1\op}
\arrow[dashed]{r}
&
2\Cat_{1\cart/\cC}
\end{tikzcd}
~.
\end{equation}

We claim that it suffices to assume that $\cC$ is a thin 2-category. Indeed, to show this we verify that it suffices to assume that $\cC = c_k$ is the free $(\infty,2)$-category on a $k$-morphism for $0 \leq k \leq 2$. For this, observe the monomorphism
\[
\hom_{\iota_1 2\Cat} ( \Adj^{1\op} , 2\Cat_{1\cart/\cC} )
\longhookra
\hom_{\iota_1 2\Cat} ( [1]^\op , 2\Cat_{1\cart/\cC} )
\]
in $\Spaces$, which by \Cref{thm.switching.yoga} is equivalent to a morphism
\[
\hom_{\iota_1 2\Cat} ( \cC^{1\op} , 2\Cat_{1\cocart/\Adj^{1\op}} )
\longra
\hom_{\iota_1 2\Cat} ( \cC^{1\op} , 2\Cat_{1\cocart/[1]^\op} )
\]
which is therefore also a monomorphism in $\Spaces$, functorially in $\cC \in \iota_1 2\Cat^\op$. It follows that the pullback functor
\[
2\Cat_{\cocart/\Adj^{1\op}}
\longra
2\Cat_{\cocart/[1]^\op}
\]
is a monomorphism in $\iota_1 2\Cat$. Hence, again applying \Cref{thm.switching.yoga}, there exists an extension \Cref{extn.for.second.adjn.lemma} if and only if there exists an extension after pullback along each functor $c_k \ra \cC$. (And clearly conditions \Cref{condition.fiberwise.radjt} and \Cref{condition.morphism.in.twocatonecartoverC.is.strict} hold if and only if they do after pullback along each functor $c_k \ra \cC$.) Thus, we can indeed assume that $\cC$ is thin.

Now, since $\cC$ is thin, the functor $2\Cat_{1\cart/\cC} \xra{\fgt} 2\Cat$ is 1-full. Hence, there exists an extension \Cref{extn.for.second.adjn.lemma} if and only if the composite $[1]^\op \ra 2\Cat_{1\cart/\cC} \xra{\fgt} 2\Cat$ selects a functor admitting a left adjoint that satisfies the condition of admitting a lift to $2\Cat_{1\cart/\cC}$.

Consider the cartesian unstraightening $\cE \xra{(p,q)} \cC \times [1]$ of the functor $F$ (the upper horizontal functor in diagram \Cref{wts.extn.over.Adjoneop.into.twocatonecartoverC}). By \Cref{basic.adjns.lemma}, it is equivalent to show that $\cE \xra{q} [1]$ is a 2-bicartesian fibration whose cocartesian 1-morphisms are carried to equivalences by the functor $\cE \xra{p} \cC$. By assumption, for each object $c \in \cC$, the functor $\cE_c \xra{q_c} [1]$ is a 2-bicartesian fibration. Given any $e_0 \in \cE_0 := q^{-1}(0)$, let us write $s := p(e_0) \in \cC$ and let $e_0 \ra e_1$ be a $q_s$-cocartesian 1-morphism. We claim that this 1-morphism is in fact $q$-cocartesian, which will evidently give the desired result. That is, we must show that for any object $e_1' \in \cE_1$, writing $t := p(e_1') \in \cC$, the upper functor in the commutative triangle
\begin{equation}
\label{comm.diagram.of.hom.cats.between.Eone.E.and.C.for.proof.of.radjt.in.twoCat.onecartoverC}
\begin{tikzcd}
\hom_{\cE_1} (e_1,e_1')
\arrow{rr}
\arrow{rd}
&
&
\hom_\cE(e_0,e_1')
\arrow{ld}
\\
&
\hom_\cC(s,t)
\end{tikzcd}
\end{equation}
in $\Cat$ is an equivalence. Because $\cE \xra{p} \cC$ is a 2-cartesian fibration, diagram \Cref{comm.diagram.of.hom.cats.between.Eone.E.and.C.for.proof.of.radjt.in.twoCat.onecartoverC} defines a morphism in $\coCart_{\hom_\cC(s,t)}$, so it suffices to prove that its upper horizontal functor restricts to an equivalence $\hom_{\cE_1}^\varphi(e_1,e_1') \ra \hom_\cE^\varphi(e_0,e_1')$ on fibers over an arbitrary object $\varphi \in \hom_\cC(s,t)$. Now, let $e_1'' \ra e_1'$ be a $(p,q)$-cartesian lift of the 1-morphism $(\id_1,\varphi)$. Then, we have a commutative diagram
\[ \begin{tikzcd}
\hom_{\cE_1}^\varphi(e_1,e_1')
\arrow{r}
&
\hom_\cE^\varphi(e_0,e_1')
\\
\hom_{\cE_{(1,s)}}(e_1,e_1'')
\arrow{u}[sloped, anchor=north]{\sim}
\arrow{r}[swap]{\sim}
&
\hom_{\cE_s}(e_0,e_1'')
\arrow{u}[sloped, anchor=north]{\sim}
\end{tikzcd} \]
in $\Cat$, where the vertical morphisms are equivalences since $e_1'' \ra e_1'$ is $(p,q)$-cartesian and the lower horizontal morphism is an equivalence since $e_0 \ra e_1$ is $q_s$-cocartesian.
\end{proof}

\subsection{Lax limits}
\label{subsection.lax.limits.over.twocats}

In this subsection, we define lax limits in $\Cat$ over $(\infty,2)$-categories, and we give an alternative description (\Cref{thm.defns.of.rlax.limit.over.onecats.agree}) in the case that the base is the left-laxification of an $(\infty,1)$-category in terms of its subdivision (as introduced and studied in \Cref{subsection.sd.of.posets}). For an alternate discussion of lax limits, see also \cite{GHL-fibns-lax-limits-inftytwocats}.

\begin{definition}
\label{defn.rlax.lim.of.fctr.from.a.twocat.to.Cat}
Given an $(\infty,2)$-category $\cC \in 2\Cat$ and a functor $\cC \ra \Cat$, its \bit{left-} and \bit{right-lax limits} are respectively the $(\infty,1)$-categories of sections of its cocartesian unstraightening over $\cC$ and its cartesian unstraightening over $\cC^{1\op}$.\footnote{These are also the $(\infty,2)$-categories of sections: all 2-morphisms therein are invertible.} Evidently, these define right adjoints
\[
\begin{tikzcd}[column sep=1.5cm]
\Cat
\arrow[transform canvas={yshift=0.9ex}]{r}{(-) \times \cC}
\arrow[dashed, leftarrow, transform canvas={yshift=-0.9ex}]{r}[yshift=-0.2ex]{\bot}[swap]{\lim^\llax_\cC}
&
2\Cat_{1\cocart/\cC}
\end{tikzcd}
\qquad
\text{and}
\qquad
\begin{tikzcd}[column sep=1.5cm]
\Cat
\arrow[transform canvas={yshift=0.9ex}]{r}{(-) \times \cC^{1\op}}
\arrow[dashed, leftarrow, transform canvas={yshift=-0.9ex}]{r}[yshift=-0.2ex]{\bot}[swap]{\lim^\rlax_\cC}
&
2\Cat_{1\cart/\cC^{1\op}}
\end{tikzcd}
~.
\]
\end{definition}

\begin{theorem}
\label{thm.defns.of.rlax.limit.over.onecats.agree}
Given an $\infty$-category $\cB \in \Cat$, the composite functor
\[
\loc.\coCart_\cB
\simeq
1\coCart_{\llax(\cB)}
\simeq
1\Cart_{\llax(\cB)^{1\op}}
\longhookra
2\Cat_{1\cart/\llax(\cB)^{1\op}}
\xra{\lim^\rlax_{\llax(\cB)}}
\Cat
\]
(in which the equivalences respectively follow from Theorems \ref{t11} \and \ref{thm.un.straightening.for.two.cats}) is corepresented by the object $(\sd(\cB) \xra{\max} \cB) \in \loc.\coCart_\cB$ (recall \Cref{t4}\Cref{t4.part.one}).
\end{theorem}

\begin{lemma}
\label{lemma.pull.back.sections.of.loctwocartfibns.along.rlax.nat.transfns}
Suppose we are given $(\infty,2)$-categories $\cC,\cD \in 2\Cat$, right-lax functors $F,G \in \hom_{\iota_1 2\Cat_{\rlax}}(\cD,\cC)$, and a right-lax natural transformation $F \overset{\alpha}{\laxra} G$. 
Suppose we are also given a 2-cartesian fibration $\cE \da \cC$. 
Then, we obtain a lax-commutative triangle
\[ \begin{tikzcd}
\Gamma_\cC(\cE)
\arrow{rr}{F^*}[swap, yshift=-0.4cm]{\Downarrow}
\arrow{rd}[sloped, swap]{G^*}
&
&
\Gamma_\cD(F^*\cE)
\\
&
\Gamma_\cD(G^*\cE)
\arrow{ru}[sloped, swap]{\Gamma_\cD(\alpha^*\cE)}
\end{tikzcd} \]
in $2\Cat$.
Moreover, these data are natural with respect to pullback along right-lax functors $\cD' \laxra \cD$.
\end{lemma}

\begin{proof}

Let $\cD \times [1] \overset{\beta}{\laxra} \cC$ be the right-lax functor defining $\alpha$. Note that by (the $(-)^{1\&2\op}$ version of) \Cref{lemma.for.unstraightening.yoga}, the composite functor $\beta^* \cE \ra \cD \times [1] \ra [1]$ is a 2-cartesian fibration.

Now, we have a solid commutative square
\begin{equation}
\label{diagram.with.radjts.in.lemma.for.proof.of.agreement.of.defns.of.laxlimit}
\begin{tikzcd}[column sep=1.5cm]
\Fun_{/[1]}([1] , \beta^* \cE)
\arrow[transform canvas={yshift=0.9ex}]{r}{\ev_1}
\arrow[dashed, leftarrow, transform canvas={yshift=-0.9ex}]{r}[yshift=-0.2ex]{\bot}
\arrow{d}
&
\Fun_{/[1]} ( \{1\} , \beta^* \cE)
\arrow{d}
&[-1.7cm]
\simeq
G^* \cE
\\
\Fun([1],\cD)
\arrow[transform canvas={yshift=0.9ex}]{r}{\ev_1}
\arrow[dashed, leftarrow, transform canvas={yshift=-0.9ex}]{r}[yshift=-0.2ex]{\bot}
&
\cD
\end{tikzcd}
\end{equation}
in $2\Cat$. 
Here, the upper dashed right adjoint exists by \Cref{t15}.
Moreover, 
the diagram \Cref{diagram.with.radjts.in.lemma.for.proof.of.agreement.of.defns.of.laxlimit} commutes after omitting the left adjoints (i.e.\! it satisfies the Beck--Chevalley condition)
since $\beta^\ast \cE \da \cD$ smushes $\beta^\ast \cE \da [1]$ by \Cref{lemma.for.unstraightening.yoga}.
Hence, we obtain a diagram
\[ \begin{tikzcd}[column sep=1.5cm]
\Fun_{/[1]}([1],\beta^* \cE) \underset{\Fun([1],\cD)}{\times} \cD
\arrow[transform canvas={yshift=0.9ex}]{r}{\ev_1}
\arrow[leftarrow, transform canvas={yshift=-0.9ex}]{r}[yshift=-0.2ex]{\bot}[swap]{\ev_1^R}
\arrow{d}[swap]{\ev_0}
&
G^*\cE
\\
F^* \cE
\end{tikzcd} \]
in $2\Cat_{/\cD}$, and thereafter a lax-commutative triangle
\[ \begin{tikzcd}
\Fun_{/[1]} ([1] , \beta^* \cE) \underset{\Fun([1],\cD)}{\times} \cD
\arrow{rr}[swap, yshift=-0.5cm]{\Downarrow}{\ev_0}
\arrow{rd}[sloped, swap]{\ev_1}
&
&
F^* \cE
\\
&
G^* \cE
\arrow{ru}[sloped, swap]{\ev_0 \ev_1^R}
\end{tikzcd}
\]
in $2\Cat_{/\cD}$.  By \Cref{t15}, we have an identification $\alpha^* \cE \simeq \ev_0 \ev_1^R$.
Applying $\Gamma_\cD(-)$ and precomposing with the evident functor
\[
\Gamma_\cC(\cE)
\longra
\Gamma_{\cD \times [1]}(\beta^* \cE)
\simeq
\Gamma_\cD
\left(
\Fun_{/[1]} ([1] , \beta^* \cE) \underset{\Fun([1],\cD)}{\times} \cD
\right)
\]
establishes the claim.
\end{proof}

\begin{proof}[Proof of \Cref{thm.defns.of.rlax.limit.over.onecats.agree}]
\Cref{obs.laxification.of.one.cats} gives an equivalence
\[
\llax(\cB) 
\simeq
\colim_{([n] \da \cB) \in \bDelta_{/\cB}} \llax([n])
~.
\]
Therefore, by \Cref{t4}, it suffices to consider the case that $\cB = [n]$ for some $[n] \in \bDelta$.
Throughout, we refer to the description of $\llax([n])$ given by \Cref{prop.llax.of.brax.n}.\footnote{In particular, we use without further reference that $\llax([n])^{1\op}$ is a thin 2-category.}

We now explicitly describe the image of the object $(\sd([n]) \xra{\max} [n]) \in \loc.\coCart_{[n]}$ under the composite equivalence
\[
\loc.\coCart_{[n]}
\xlongla{\sim}
1\coCart_{\llax([n])}
\simeq
1\Cart_{\llax([n])^{1\op}}
\]
(where the first (leftwards) equivalence is that of \Cref{t11}). For this, let $\coSpan(\sd([n])) \in 2\Cat$ denote the strict 2-category of cospans in $\sd([n])$ (which exists since $\sd([n])$ has pushouts, which are given by union of subsets of $[n]$). We define the 1-full subcategory
\[
\w{\sd}
:=
\w{\sd([n])}
\subseteq
\coSpan(\sd([n]))
\]
on those 1-morphisms
\begin{equation}
\label{new.generic.one.morphism.in.sd.tilde}
I
\longhookra
K
\longhookla
J
\end{equation}
from $I$ to $J$ (a cospan among subsets of $[n]$) in which the inclusion $I \hookra K$ is isomax and the inclusion $K \hookla J$ is both isomin and inert. We note that $\w{\sd}$ is a thin 2-category. Now, we claim that the desired image in $1\Cart_{\llax([n])^{1\op}}$ is given by the functor
\begin{equation}
\label{cart.unstr.of.sd}
\w{\sd}
\xra{\max}
\llax([n])^{1\op}
\end{equation}
characterized by the fact that it carries a 1-morphism \Cref{new.generic.one.morphism.in.sd.tilde} in $\w{\sd}$ to the 1-morphism in $\llax([n])^{1\op}$ corresponding to the 1-morphism
\[
\max(I)
=
\max(K)
\xla{K_{\geq \max(J)}}
\max(J)
\]
in $\llax([n])$.
Using \Cref{obs.two.cart.fibns}\Cref{obs.two.cart.fibns.one.cart.fibn.iff}, we now observe that the functor \Cref{cart.unstr.of.sd} is a 1-cartesian fibration, whose cartesian 1-morphisms are those 1-morphisms \Cref{new.generic.one.morphism.in.sd.tilde} for which $I=K$.
Next, by definition of $\w{\sd}$, for any $I,J\in \w{\sd}$, a morphism $K \hookrightarrow K'$ in $\ulhom_{\w{\sd}}(I,J)$ has the feature that $K' = K \cup K'_{\geq \max(J)}$.
Using this it is routine to verify that the functor (between posets) $\ulhom_{\w{\sd}}(I,J) \to \ulhom_{\llax([n])^{1\op}}(I,J)$ is a left fibration.

Now, in the 1-cartesian fibration \Cref{cart.unstr.of.sd} the fiber over an object $a \in \llax([n])^{1\op}$ is $\sd([n])_{\{a\}/\isomax}$ and its cartesian monodromy functors are given by concatenation. Noting that the straightening $\llax([n]) \ra \Cat$ factors through the full subcategory $\Poset \subset \Cat$ so that it is characterized by its values on the objects and generating 1-morphisms in $\llax([n])$ (i.e.\! those of the form $a < b$), the claim follows.

In order to proceed, let us observe that the functor \Cref{cart.unstr.of.sd} admits an evident section
\[
\w{\sd}
\xlongla{\sigma}
\llax([n])^{1\op}
\]
whose image is the full subcategory on the inclusions of singletons into $[n]$, which we consider as a morphism in $2\Cat_{1\cart/\llax([n])^{1\op}}$. It now suffices to show that for any 1-cartesian fibration $(\cE \da \llax([n])^{1\op}) \in 2\Cat_{1\cart/\llax([n])^{1\op}}$, the resulting restriction functor
\begin{equation}
\label{the.fctr.to.be.shown.an.equivce.for.sd.to.compute.rlaxlim}
\hom_{1\Cart_{\llax([n])^{1\op}}} ( \w{\sd} , \cE )
\xra{\sigma^*}
\Gamma_{\llax([n])^{1\op}} ( \cE )
\end{equation}
in $\Cat$ is an equivalence. We will construct an inverse. 

For this, consider the right-lax functor
\[ \begin{tikzcd}
\w{\sd}
\arrow[squiggly]{r}{\min}
&
\llax([n])^{1\op}
\end{tikzcd}
\]
characterized by the fact that it carries a 1-morphism \Cref{new.generic.one.morphism.in.sd.tilde} in $\w{\sd}$ to the 1-morphism in $\llax([n])^{1\op}$ corresponding to the 1-morphism
\[
\min(I)
\xla{K_{\leq \min(I)}}
\min(K)
=
\min(J)
\]
in $\llax([n])$. Observe that there exists a right-lax natural transformation
\begin{equation}
\label{rlax.nat.trans.from.max.to.min}
\begin{tikzcd}
\max
\arrow[squiggly]{r}
&
\min
\end{tikzcd}
\end{equation}
from $\max$ to $\min$ characterized by the fact that it carries an object $(I \subseteq [n]) \in \w{\sd}$ to the 1-morphism in $\llax([n])^{1\op}$ corresponding to the 1-morphism $\max(I) \xla{I} \min(I)$ in $\llax([n])$; 
in terms of the description of \Cref{rmk.rlax.nat.trans.has.lax.squares},
the value of the right-lax natural transformation \Cref{rlax.nat.trans.from.max.to.min} on a morphism \Cref{new.generic.one.morphism.in.sd.tilde} is depicted by the diagram
in $\llax([n])$
\begin{equation}
\label{e22}
\begin{tikzcd}[row sep=1.5cm, column sep=1.5cm]
\max(I)
\arrow[leftarrow]{d}[swap]{I}
\arrow[leftarrow, bend right=10]{rrd}[sloped, swap, pos=0.4]{I \cup K_{\leq \min(I)}}[xshift=1ex, yshift=0.5ex]{\rotatebox{15}{$\Longrightarrow$}}
&[-1.7cm]
= \max(K)
\arrow[leftarrow]{rr}{K_{\geq \max(J)}}
\arrow[leftarrow, bend left=10]{rd}[sloped]{K}
&
&[-1.7cm]
\max(J)
\arrow[leftarrow]{d}{J}
\\
\min(I)
\arrow[leftarrow]{rr}[swap]{K_{\leq \min(I)}}
&
&
\min(K) =
&
\min(J)
\end{tikzcd}
\end{equation}
in which the 2-morphism indicates the inclusion $(I \cup K_{\leq \min(I)}) \subseteq K$.

We now construct a functor
\begin{equation}
\label{inverse.of.the.fctr.to.be.shown.an.equivce.for.sd.to.compute.rlaxlim}
\hom_{1\Cart_{\llax([n])^{1\op}}} ( \w{\sd} , \cE )
\longla
\Gamma_{\llax([n])^{1\op}} ( \cE )
\end{equation}
in the opposite direction as \Cref{the.fctr.to.be.shown.an.equivce.for.sd.to.compute.rlaxlim}, which we will later show to be its inverse. By 
\Cref{t12},
we obtain a morphism $\min^* \cE \ra \max^* \cE$ in $2\Cat_{\loc.1\cart/\w{\sd}}$.
On sections, we therefore have a composite functor
\begin{equation}
\label{composite.functor.from.sections.over.llaxbraxnoneop.to.sections.of.max.pullback}
\Gamma_{\llax([n])^{1\op}} ( \cE )
\longra
\Gamma_{\w{\sd}} ( \min^* \cE )
\longra
\Gamma_{\w{\sd}} ( \max^* \cE )
~.
\end{equation}
We claim that the functor \Cref{composite.functor.from.sections.over.llaxbraxnoneop.to.sections.of.max.pullback} lands in the subcategory
\[
\hom_{1\Cart_{\llax([n])^{1\op}}} ( \w{\sd} , \cE )
\subseteq
\hom_{2\Cat_{/\llax([n])^{1\op}}} ( \w{\sd} , \cE )
\simeq
\Gamma_{\w{\sd}} ( \max^* \cE )
~;
\]
this will give our functor \Cref{inverse.of.the.fctr.to.be.shown.an.equivce.for.sd.to.compute.rlaxlim}. To see this, let $[1] \xra{\varphi} \w{\sd}$ select a $(\w{\sd} \xra{\max} \llax([n])^{1\op})$-cartesian 1-morphism. By our above description thereof, the composite (strict) functor
\[ \begin{tikzcd}
{[1]}
\arrow{r}{\varphi}
&
\w{\sd}
\arrow[squiggly]{r}{\min}
&
\llax([n])^{1\op}
\end{tikzcd} \]
is constant. 
Therefore, using \Cref{t12}\Cref{t12.part.2} the functor
\[
\Gamma_{\llax([n])^{1\op}} ( \cE )
\longra
\Gamma_{[1]}(\varphi^* \min^* \cE )
\]
factors through $\Gamma_{[1]}^\cart(\varphi^* \min^* \cE ) \subseteq \Gamma_{[1]}(\varphi^* \min^* \cE )$. Moreover, pullback along $\varphi$ of the right-lax natural transformation \Cref{rlax.nat.trans.from.max.to.min} yields a strict natural transformation (between strict functors),\footnote{Indeed, first of all, every (unital) lax functor from $[1]$ is strict.
Next, given a 1-morphism $(I \hookrightarrow K \hookleftarrow J)$ in $\w{\sd}$ it is cartesian over over $\llax([n])^{1\op}$ if $I = K$.
In this case, the inclusion $I \cup K_{\leq \min(I)} \hookrightarrow K$ appearing in \Cref{e22} is an isomorphism, hence the natural transformation is strict.
} 
so that the morphism
\[
\varphi^*\min^*\cE
\longra
\varphi^*\max^*\cE
\]
lies in $1\Cart_{[1]} \subset 2\Cat_{1\cart/[1]}$ (i.e.\! it preserves cartesian 1-morphisms). This proves the claimed factorization of the functor \Cref{composite.functor.from.sections.over.llaxbraxnoneop.to.sections.of.max.pullback}.

We now conclude by showing that the functors \Cref{the.fctr.to.be.shown.an.equivce.for.sd.to.compute.rlaxlim} and \Cref{inverse.of.the.fctr.to.be.shown.an.equivce.for.sd.to.compute.rlaxlim} are inverses.

We first show that the composite functor $\Cref{the.fctr.to.be.shown.an.equivce.for.sd.to.compute.rlaxlim} \circ \Cref{inverse.of.the.fctr.to.be.shown.an.equivce.for.sd.to.compute.rlaxlim}$ is the identity: 
using \Cref{t12}\Cref{t12.part.2}
this follows from the fact that $\Cref{rlax.nat.trans.from.max.to.min} \circ \sigma$ is the identity natural transformation from $\id_{\llax([n])^{1\op}}$ to itself.

We now show that the composite functor $\Cref{inverse.of.the.fctr.to.be.shown.an.equivce.for.sd.to.compute.rlaxlim} \circ \Cref{the.fctr.to.be.shown.an.equivce.for.sd.to.compute.rlaxlim}$ is the identity. Let us first observe that the right-lax natural transformation $\sigma \circ \max \xra{\sigma \circ \Cref{rlax.nat.trans.from.max.to.min}} \sigma \circ \min$ factors as a composite
\[ \begin{tikzcd}
\sigma \circ \max
\arrow{r}
&
\id_{\w{\sd}}
\arrow[squiggly]{r}
&
\sigma \circ \min
\end{tikzcd} \]
of right-lax natural transformations among right-lax endofunctors of $\w{\sd}$, which is determined (using that $\w{\sd}$ is thin) by the fact that its value on an object $(I \subseteq [n]) \in \w{\sd}$ is the composite
\[
\{ \max(I) \}
\xra{\{ \max(I) \} \longhookra I \xlonghookla{=} I }
I
\xra{ I \xlonghookra{=} I \longhookla \{ \min(I) \} }
\{ \min(I) \}
~.
\]
By \Cref{t12}\Cref{t12.part.1} applied to the functor $( \llax([n])^{1\op} \xra{\sigma} \w{\sd} )$, the composite
$\Cref{inverse.of.the.fctr.to.be.shown.an.equivce.for.sd.to.compute.rlaxlim} \circ \Cref{the.fctr.to.be.shown.an.equivce.for.sd.to.compute.rlaxlim}$ is the restriction to the full subcategory $\hom_{1\Cart_{\llax([n])^{1\op}}} ( \w{\sd} , \cE ) \subseteq \Gamma_{\w{\sd}} ( \max^* \cE )$ (in both the source and the target) of the composite
\begin{equation}
\label{e19}
\Gamma_{\w{\sd}} ( \max^\ast \cE )
\longrightarrow
\Gamma_{\w{\sd}} ( ( \sigma \circ \min)^\ast \max^\ast \cE )
\longra
\Gamma_{\w{\sd}} ( ( \sigma \circ \max)^\ast \max^\ast \cE )
\simeq
\Gamma_{\w{\sd}} ( \max^\ast \cE )
\end{equation}
in which the second morphism is obtained from \Cref{t12} applied to the right-lax natural transformation $ \sigma \circ \max \laxra \sigma \circ \min$ and the equivalence follows from the identification $\max \circ \sigma = \id$.
Next, observe the identification
\[
\max
\circ
\left( \sigma \circ \max 
\longra
\id_{\w{\sd}}
\right)
=
\left( \max \circ \sigma \xlongra{=} \id_{\llax([n])^{1\op}} \right) \circ \max
\]
between natural transformations.
This gives an identification of the composite \Cref{e19} with the composite
\begin{equation}
\label{composite.functor.on.sections.over.tilde.sd}
\Gamma_{\w{\sd}} ( \max^* \cE )
\longra
\Gamma_{\w{\sd}} ( (\sigma \circ \min)^* \max^* \cE )
\longra
\Gamma_{\w{\sd}} ( \max^* \cE )
~,
\end{equation}
in which the second morphism arises by applying \Cref{t12} to the right-lax natural transformation $\id_{\w{\sd}} \laxra \sigma \circ \min$.
Thus, the composite
$\Cref{inverse.of.the.fctr.to.be.shown.an.equivce.for.sd.to.compute.rlaxlim} \circ \Cref{the.fctr.to.be.shown.an.equivce.for.sd.to.compute.rlaxlim}$ can be identified as the restriction to the full subcategory $\hom_{1\Cart_{\llax([n])^{1\op}}} ( \w{\sd} , \cE ) \subseteq \Gamma_{\w{\sd}} ( \max^* \cE )$ (in both the source and the target) of the composite \Cref{composite.functor.on.sections.over.tilde.sd}.

Now, \Cref{lemma.pull.back.sections.of.loctwocartfibns.along.rlax.nat.transfns} applied to
$\id_{\w{\sd}} \laxra \sigma \circ \min$
gives a natural transformation
\begin{equation}
\label{nat.trans.from.id.to.composite.functor.on.sections.over.tilde.sd}
\id_{\Gamma_{\w{\sd}} ( \max^* \cE )}
\longra
\Cref{composite.functor.on.sections.over.tilde.sd}
~.
\end{equation}
It remains to show is an equivalence on the full subcategory $\hom_{1\Cart_{\llax([n])^{1\op}}} ( \w{\sd} , \cE ) \subseteq \Gamma_{\w{\sd}} ( \max^* \cE )$. 
For this, suppose we are given an object
\[
\varphi \in \hom_{1\Cart_{\llax([n])^{1\op}}}(\w{\sd} , \cE )
\subseteq
\Gamma_{\w{\sd}} ( \max^* \cE )
~.
\]
For an arbitrary object $(I \subseteq [n]) \in \w{\sd}$, consider the $(\w{\sd} \xra{\max} \llax([n])^{1\op})$-cartesian morphism $I \xra{I \xhookra{=} I \hookla \{\min(I)\}} \{ \min(I) \}$. Since both $\varphi$ and $\Cref{composite.functor.on.sections.over.tilde.sd}(\varphi)$ preserve cartesian 1-morphisms, it suffices to show that the natural transformation \Cref{nat.trans.from.id.to.composite.functor.on.sections.over.tilde.sd} is an equivalence on objects given by singleton subsets of $[n]$, i.e.\! those in the image of $\sigma$. Upon observing the identification
\[
\left(
\sigma 
= 
\id_{\w{\sd}} \circ \sigma  
\laxra \sigma \circ \min \circ \sigma 
\simeq 
\sigma \circ \id_{\llax([n])^{1\op}}
= 
\sigma
\right)
\simeq
\id_{\sigma}
~,
\]
this follows by the naturality of \Cref{lemma.pull.back.sections.of.loctwocartfibns.along.rlax.nat.transfns}.
\end{proof}

\bibliographystyle{amsalpha}
\bibliography{strat}{}

\newcommand{\etalchar}[1]{$^{#1}$}
\providecommand{\bysame}{\leavevmode\hbox to3em{\hrulefill}\thinspace}
\providecommand{\MR}{\relax\ifhmode\unskip\space\fi MR }
\providecommand{\MRhref}[2]{%
  \href{http://www.ams.org/mathscinet-getitem?mr=#1}{#2}
}
\providecommand{\href}[2]{#2}
\begin{thebibliography}{LMSM86}

\bibitem[ACB22]{ACB-chromfrac}
Omar Antol\'{\i}n-Camarena and Tobias Barthel, \emph{Chromatic fracture cubes},
  Equivariant topology and derived algebra, London Math. Soc. Lecture Note
  Ser., vol. 474, Cambridge Univ. Press, Cambridge, 2022, pp.~100--118.

\bibitem[AF20]{AF-fibns}
David Ayala and John Francis, \emph{Fibrations of {$\infty$}-categories}, High.
  Struct. \textbf{4} (2020), no.~1, 168--265.

\bibitem[AFR18]{AFR-fact}
David Ayala, John Francis, and Nick Rozenblyum, \emph{Factorization homology
  {I}: {H}igher categories}, Adv. Math. \textbf{333} (2018), 1042--1177.

\bibitem[AGM85]{AGM-Segal}
J.~F. Adams, J.~H. Gunawardena, and H.~Miller, \emph{The {S}egal conjecture for
  elementary abelian {$p$}-groups}, Topology \textbf{24} (1985), no.~4,
  435--460.

\bibitem[AMGRa]{AMR-cyclo}
David Ayala, Aaron Mazel-Gee, and Nick Rozenblyum, \emph{Cyclotomic spectra via
  stratifications}, available at \texttt{arXiv:1710.06416}, v2.

\bibitem[AMGRb]{AMR-mackey}
\bysame, \emph{Derived {M}ackey functors and {$\Cyclic_{p^n}$}-equivariant
  cohomology}, available at \texttt{arXiv:2105.02456}, v1.

\bibitem[AMGRc]{AMR-fact}
\bysame, \emph{Factorization homology of enriched $\infty$-categories},
  available at \texttt{arXiv:1710.06414}, v1.

\bibitem[AMGRd]{AMR-trace}
\bysame, \emph{The geometry of the cyclotomic trace}, available at
  \texttt{arXiv:1710.06409}, v2.

\bibitem[Bal02]{Balmer-pshvs}
Paul Balmer, \emph{Presheaves of triangulated categories and reconstruction of
  schemes}, Math. Ann. \textbf{324} (2002), no.~3, 557--580.

\bibitem[Bal05]{Balmer-specofprime}
\bysame, \emph{The spectrum of prime ideals in tensor triangulated categories},
  J. Reine Angew. Math. \textbf{588} (2005), 149--168.

\bibitem[Bal07]{Balmer-suppfilt}
\bysame, \emph{Supports and filtrations in algebraic geometry and modular
  representation theory}, Amer. J. Math. \textbf{129} (2007), no.~5,
  1227--1250.

\bibitem[Bal10]{Balmer-spectroix}
\bysame, \emph{Spectra, spectra, spectra---tensor triangular spectra versus
  {Z}ariski spectra of endomorphism rings}, Algebr. Geom. Topol. \textbf{10}
  (2010), no.~3, 1521--1563.

\bibitem[Bar05]{Barwick-thesis}
Clark Barwick, \emph{(infinity, n)-{C}at as a closed model category}, ProQuest
  LLC, Ann Arbor, MI, 2005, Thesis (Ph.D.)--University of Pennsylvania.

\bibitem[Bar16]{Tobi-chromcompl}
Tobias Barthel, \emph{Chromatic completion}, Proc. Amer. Math. Soc.
  \textbf{144} (2016), no.~5, 2263--2274.

\bibitem[Bar17]{Bar-Mack}
Clark Barwick, \emph{Spectral {M}ackey functors and equivariant algebraic
  {$K$}-theory ({I})}, Adv. Math. \textbf{304} (2017), 646--727.

\bibitem[BBD82]{BBD-perv}
A.~A. Be{\u{\i}}linson, J.~Bernstein, and P.~Deligne, \emph{Faisceaux pervers},
  Analysis and topology on singular spaces, {I} ({L}uminy, 1981),
  Ast\'{e}risque, vol. 100, Soc. Math. France, Paris, 1982, pp.~5--171.

\bibitem[BDS16]{BDS-GNdual}
Paul Balmer, Ivo Dell'Ambrogio, and Beren Sanders, \emph{Grothendieck-{N}eeman
  duality and the {W}irthm\"{u}ller isomorphism}, Compos. Math. \textbf{152}
  (2016), no.~8, 1740--1776.

\bibitem[Be{\u{\i}}78]{Beil-linalg}
A.~A. Be{\u{\i}}linson, \emph{Coherent sheaves on {${\bf P}^{n}$} and problems
  in linear algebra}, Funktsional. Anal. i Prilozhen. \textbf{12} (1978),
  no.~3, 68--69.

\bibitem[Be{\u{\i}}80]{Beil-ad}
\bysame, \emph{Residues and ad\`eles}, Funktsional. Anal. i Prilozhen.
  \textbf{14} (1980), no.~1, 44--45.

\bibitem[BF07]{BalmerFavi-gluing}
Paul Balmer and Giordano Favi, \emph{Gluing techniques in triangular geometry},
  Q. J. Math. \textbf{58} (2007), no.~4, 415--441.

\bibitem[BG20]{BalchGreen-adelic}
Scott Balchin and J.~P.~C. Greenlees, \emph{Adelic models of
  tensor-triangulated categories}, Adv. Math. \textbf{375} (2020), 107339, 45.

\bibitem[BGH]{BGH-exodromy}
Clark Barwick, Saul Glasman, and Peter Haine, \emph{Exodromy}, available at
  \texttt{arXiv:1807.03281}, v7.

\bibitem[BGH20]{Tobiplus-SpecofcpctLie}
Tobias Barthel, J.~P.~C. Greenlees, and Markus Hausmann, \emph{On the {B}almer
  spectrum for compact {L}ie groups}, Compos. Math. \textbf{156} (2020), no.~1,
  39--76.

\bibitem[BGN18]{BGN-dual}
Clark Barwick, Saul Glasman, and Denis Nardin, \emph{Dualizing cartesian and
  cocartesian fibrations}, Theory Appl. Categ. \textbf{33} (2018), Paper No. 4,
  67--94.

\bibitem[BGP73]{BGP-refl}
I.~N. Bern\v{s}te\u{\i}n, I.~M. Gel$'$fand, and V.~A. Ponomarev, \emph{Coxeter
  functors, and {G}abriel's theorem}, Uspehi Mat. Nauk \textbf{28} (1973),
  no.~2(170), 19--33.

\bibitem[BGT13]{BGT-K}
Andrew~J. Blumberg, David Gepner, and Gon\c{c}alo Tabuada, \emph{A universal
  characterization of higher algebraic {$K$}-theory}, Geom. Topol. \textbf{17}
  (2013), no.~2, 733--838.

\bibitem[BHN{\etalchar{+}}19]{Tobiplusplus-Specgenfabgrp}
Tobias Barthel, Markus Hausmann, Niko Naumann, Thomas Nikolaus, Justin Noel,
  and Nathaniel Stapleton, \emph{The {B}almer spectrum of the equivariant
  homotopy category of a finite abelian group}, Invent. Math. \textbf{216}
  (2019), no.~1, 215--240.

\bibitem[BK89]{BondKap-reconstrn}
A.~I. Bondal and M.~M. Kapranov, \emph{Representable functors, {S}erre
  functors, and reconstructions}, Izv. Akad. Nauk SSSR Ser. Mat. \textbf{53}
  (1989), no.~6, 1183--1205, 1337.

\bibitem[BS17]{BS-spec-SpgG}
Paul Balmer and Beren Sanders, \emph{The spectrum of the equivariant stable
  homotopy category of a finite group}, Invent. Math. \textbf{208} (2017),
  no.~1, 283--326.

\bibitem[BSP21]{BSP-unicity}
Clark Barwick and Christopher Schommer-Pries, \emph{On the unicity of the
  theory of higher categories}, J. Amer. Math. Soc. \textbf{34} (2021), no.~4,
  1011--1058.

\bibitem[Col]{Columbus-thesis}
Tobias Columbus, \emph{$2$-categorical aspects of quasi-categories}, thesis,
  {K}arlsruher {I}nstitut f\"{u}r {T}echnologie (2017), 152pp.

\bibitem[Con72]{Cond-fibns}
Fran\c{c}ois Conduch\'e, \emph{Au sujet de l'existence d'adjoints \`a droite
  aux foncteurs ``image r\'eciproque'' dans la cat\'egorie des cat\'egories},
  C. R. Acad. Sci. Paris S\'er. A-B \textbf{275} (1972), A891--A894.

\bibitem[DG02]{DwyGreen-comptors}
W.~G. Dwyer and J.~P.~C. Greenlees, \emph{Complete modules and torsion
  modules}, Amer. J. Math. \textbf{124} (2002), no.~1, 199--220.

\bibitem[DHS88]{DHS-nilp}
Ethan~S. Devinatz, Michael~J. Hopkins, and Jeffrey~H. Smith, \emph{Nilpotence
  and stable homotopy theory. {I}}, Ann. of Math. (2) \textbf{128} (1988),
  no.~2, 207--241.

\bibitem[DJW21]{DJW-refl}
Tobias Dyckerhoff, Gustavo Jasso, and Tashi Walde, \emph{Generalised {BGP}
  reflection functors via the {G}rothendieck construction}, Int. Math. Res.
  Not. IMRN (2021), no.~20, 15733--15745.

\bibitem[EH]{ElHog-diag}
Ben Elias and Matthew Hogancamp, \emph{Categorical diagonalization}, available
  at \texttt{arXiv:1707.04349}, v1.

\bibitem[Gab62]{Gabriel-thesis}
Pierre Gabriel, \emph{Des cat\'{e}gories ab\'{e}liennes}, Bull. Soc. Math.
  France \textbf{90} (1962), 323--448.

\bibitem[GH15]{GH-enr}
David Gepner and Rune Haugseng, \emph{Enriched {$\infty$}-categories via
  non-symmetric {$\infty$}-operads}, Adv. Math. \textbf{279} (2015), 575--716.

\bibitem[GHL]{GHL-fibns-lax-limits-inftytwocats}
Andrea Gagna, Yonatan Harpaz, and Edoardo Lanari, \emph{Fibrations and lax
  limits of $(\infty,2)$-categories}, available at \texttt{arXiv:2012.04537},
  v2.

\bibitem[GHL21]{GHL-gray-lax-fctrs}
\bysame, \emph{Gray tensor products and {L}ax functors of
  {$(\infty,2)$}-categories}, Adv. Math. \textbf{391} (2021), Paper No. 107986,
  32.

\bibitem[GHN17]{GHN}
David Gepner, Rune Haugseng, and Thomas Nikolaus, \emph{Lax colimits and free
  fibrations in {$\infty$}-categories}, Doc. Math. \textbf{22} (2017),
  1225--1266.

\bibitem[Gir64]{Gir-desc}
Jean Giraud, \emph{M\'ethode de la descente}, Bull. Soc. Math. France M\'em.
  \textbf{2} (1964), viii+150.

\bibitem[Glaa]{Saul-Goo}
Saul Glasman, \emph{{G}oodwillie calculus and {M}ackey functors}, available at
  \texttt{arXiv:1610.03127}, v3.

\bibitem[Glab]{Saul-strat}
\bysame, \emph{Stratified categories, geometric fixed points and a generalized
  {A}rone-{C}hing theorem}, available at \texttt{arXiv:1507.01976}, v4.

\bibitem[GM]{GM-gen}
Bertrand Guillou and J.P. May, \emph{Models of {$G$}-spectra as presheaves of
  spectra}, available at \texttt{arXiv:1110.3571}, v3.

\bibitem[GM92]{GreenMay-dual}
J.~P.~C. Greenlees and J.~P. May, \emph{Derived functors of {$I$}-adic
  completion and local homology}, J. Algebra \textbf{149} (1992), no.~2,
  438--453.

\bibitem[GM95]{GreenMay-Tate}
\bysame, \emph{Generalized {T}ate cohomology}, Mem. Amer. Math. Soc.
  \textbf{113} (1995), no.~543, viii+178.

\bibitem[Goo03]{GooCalc3}
Thomas~G. Goodwillie, \emph{Calculus. {III}. {T}aylor series}, Geom. Topol.
  \textbf{7} (2003), 645--711.

\bibitem[GR14]{GR-dgindsch}
Dennis Gaitsgory and Nick Rozenblyum, \emph{D{G} indschemes}, Perspectives in
  representation theory, Contemp. Math., vol. 610, Amer. Math. Soc.,
  Providence, RI, 2014, pp.~139--251.

\bibitem[GR17]{GR}
\bysame, \emph{A study in derived algebraic geometry. {V}ol. {I}.
  {C}orrespondences and duality}, Mathematical Surveys and Monographs, vol.
  221, American Mathematical Society, Providence, RI, 2017.

\bibitem[Gre]{Greenlees-thesis}
J.~P.~C. Greenlees, \emph{Adams spectral sequences in equivariant topology},
  thesis, {C}ambridge {U}niversity (1985), 381pp.

\bibitem[Gro17]{Groechenig-ad}
Michael Groechenig, \emph{Adelic descent theory}, Compos. Math. \textbf{153}
  (2017), no.~8, 1706--1746.

\bibitem[GS18]{GreenShip}
J.~P.~C. Greenlees and B.~Shipley, \emph{An algebraic model for rational
  torus-equivariant spectra}, J. Topol. \textbf{11} (2018), no.~3, 666--719.

\bibitem[Hai]{Haine-strat}
Peter Haine, \emph{On the homotopy theory of stratified spaces}, available at
  \texttt{arXiv:1811.01119}, v5.

\bibitem[Hau15]{Haug-rect}
Rune Haugseng, \emph{Rectification of enriched {$\infty$}-categories}, Algebr.
  Geom. Topol. \textbf{15} (2015), no.~4, 1931--1982.

\bibitem[Hau18]{rune-spans}
\bysame, \emph{Iterated spans and classical topological field theories}, Math.
  Z. \textbf{289} (2018), no.~3-4, 1427--1488.

\bibitem[HHLN]{HHLN-laxmonadjns-twovarfibns-calcmates}
Rune Haugseng, Fabian Hebestreit, Sil Linskens, and Joost Nuiten, \emph{Lax
  monoidal adjunctions, two-variable fibrations and the calculus of mates},
  available at \texttt{arXiv:2011.08808}, v2.

\bibitem[Hin20]{hinich-yoneda-enriched}
Vladimir Hinich, \emph{Yoneda lemma for enriched {$\infty$}-categories}, Adv.
  Math. \textbf{367} (2020), 107129, 119.

\bibitem[Hop87]{Hopkins-global}
Michael~J. Hopkins, \emph{Global methods in homotopy theory}, Homotopy theory
  ({D}urham, 1985), London Math. Soc. Lecture Note Ser., vol. 117, Cambridge
  Univ. Press, Cambridge, 1987, pp.~73--96.

\bibitem[HORR23]{HORR-pasting}
Philip Hackney, Viktoriya Ozornova, Emily Riehl, and Martina Rovelli, \emph{An
  {$(\infty,2)$}-categorical pasting theorem}, Trans. Amer. Math. Soc.
  \textbf{376} (2023), no.~1, 555--597.

\bibitem[HPS97]{HPS}
Mark Hovey, John~H. Palmieri, and Neil~P. Strickland, \emph{Axiomatic stable
  homotopy theory}, Mem. Amer. Math. Soc. \textbf{128} (1997), no.~610, x+114.

\bibitem[HPV]{HPV-gluing}
Benjamin Hennion, Mauro Porta, and Gabriele Vezzosi, \emph{Formal gluing along
  non-linear flags}, available at \texttt{arXiv:1607.04503}, v2.

\bibitem[HS98]{HS-nilptwo}
Michael~J. Hopkins and Jeffrey~H. Smith, \emph{Nilpotence and stable homotopy
  theory. {II}}, Ann. of Math. (2) \textbf{148} (1998), no.~1, 1--49.

\bibitem[Hub91]{Huber-ad}
A.~Huber, \emph{On the {P}arshin-{B}e\u{\i}linson ad\`eles for schemes}, Abh.
  Math. Sem. Univ. Hamburg \textbf{61} (1991), 249--273.

\bibitem[HY18]{HillYarnall-slice}
Michael~A. Hill and Carolyn Yarnall, \emph{A new formulation of the equivariant
  slice filtration with applications to {$C_p$}-slices}, Proc. Amer. Math. Soc.
  \textbf{146} (2018), no.~8, 3605--3614.

\bibitem[KR00]{KontRos-nc}
Maxim Kontsevich and Alexander~L. Rosenberg, \emph{Noncommutative smooth
  spaces}, The {G}elfand {M}athematical {S}eminars, 1996--1999, Gelfand Math.
  Sem., Birkh\"{a}user Boston, Boston, MA, 2000, pp.~85--108.

\bibitem[KS09]{KontSoi-ncgeom}
M.~Kontsevich and Y.~Soibelman, \emph{Notes on {$A_\infty$}-algebras,
  {$A_\infty$}-categories and non-commutative geometry}, Homological mirror
  symmetry, Lecture Notes in Phys., vol. 757, Springer, Berlin, 2009,
  pp.~153--219.

\bibitem[Kuz14]{Kuz-SODinAG}
Alexander Kuznetsov, \emph{Semiorthogonal decompositions in algebraic
  geometry}, Proceedings of the {I}nternational {C}ongress of
  {M}athematicians---{S}eoul 2014. {V}ol. {II}, Kyung Moon Sa, Seoul, 2014,
  pp.~635--660.

\bibitem[Lin80]{Lin-Segal}
Wen~Hsiung Lin, \emph{On conjectures of {M}ahowald, {S}egal and {S}ullivan},
  Math. Proc. Cambridge Philos. Soc. \textbf{87} (1980), no.~3, 449--458.

\bibitem[LMSM86]{LMS}
L.~G. Lewis, Jr., J.~P. May, M.~Steinberger, and J.~E. McClure,
  \emph{Equivariant stable homotopy theory}, Lecture Notes in Mathematics, vol.
  1213, Springer-Verlag, Berlin, 1986, With contributions by J. E. McClure.

\bibitem[Lur]{LurieHA}
Jacob Lurie, \emph{Higher algebra}, available from the author's website
  (version dated September 18, 2017).

\bibitem[Lur09]{LurieHTT}
\bysame, \emph{Higher topos theory}, Annals of Mathematics Studies, vol. 170,
  Princeton University Press, Princeton, NJ, 2009.

\bibitem[Man88]{Manin-qgrpsandncgeom}
Yu.~I. Manin, \emph{Quantum groups and noncommutative geometry}, Universit\'{e}
  de Montr\'{e}al, Centre de Recherches Math\'{e}matiques, Montreal, QC, 1988.

\bibitem[May96]{May-Alaska}
J.~P. May, \emph{Equivariant homotopy and cohomology theory}, CBMS Regional
  Conference Series in Mathematics, vol.~91, Published for the Conference Board
  of the Mathematical Sciences, Washington, DC; by the American Mathematical
  Society, Providence, RI, 1996, With contributions by M. Cole, G. Comeza\~na,
  S. Costenoble, A. D. Elmendorf, J. P. C. Greenlees, L. G. Lewis, Jr., R. J.
  Piacenza, G. Triantafillou, and S. Waner.

\bibitem[MG19a]{AMG-rnerves}
Aaron Mazel-Gee, \emph{The universality of the {R}ezk nerve}, Algebr. Geom.
  Topol. \textbf{19} (2019), no.~7, 3217--3260.

\bibitem[MG19b]{rnerves}
\bysame, \emph{The universality of the {R}ezk nerve}, Algebr. Geom. Topol.
  \textbf{19} (2019), no.~7, 3217--3260.

\bibitem[Mil92]{Miller-finite}
Haynes Miller, \emph{Finite localizations}, Bol. Soc. Mat. Mexicana (2)
  \textbf{37} (1992), no.~1-2, 383--389, Papers in honor of Jos\'{e} Adem
  (Spanish).

\bibitem[MM02]{ManMay-eq}
M.~A. Mandell and J.~P. May, \emph{Equivariant orthogonal spectra and
  {$S$}-modules}, Mem. Amer. Math. Soc. \textbf{159} (2002), no.~755, x+108.

\bibitem[MNN17]{MNN}
Akhil Mathew, Niko Naumann, and Justin Noel, \emph{Nilpotence and descent in
  equivariant stable homotopy theory}, Adv. Math. \textbf{305} (2017),
  994--1084.

\bibitem[Nee92]{Neeman-chrom}
Amnon Neeman, \emph{The chromatic tower for {$D(R)$}}, Topology \textbf{31}
  (1992), no.~3, 519--532, With an appendix by Marcel B\"{o}kstedt.

\bibitem[NS18]{NS}
Thomas Nikolaus and Peter Scholze, \emph{On topological cyclic homology}, Acta
  Math. \textbf{221} (2018), no.~2, 203--409.

\bibitem[Nui]{Nuiten-straightening}
Joost Nuiten, \emph{On straightening for {S}egal spaces}, available at
  \texttt{arXiv:2108.11431}, v1.

\bibitem[Par76]{Parshin-ad}
A.~N. Par\v{s}in, \emph{On the arithmetic of two-dimensional schemes. {I}.
  {D}istributions and residues}, Izv. Akad. Nauk SSSR Ser. Mat. \textbf{40}
  (1976), no.~4, 736--773, 949.

\bibitem[Pet17]{Petersen-ss}
Dan Petersen, \emph{A spectral sequence for stratified spaces and configuration
  spaces of points}, Geom. Topol. \textbf{21} (2017), no.~4, 2527--2555.

\bibitem[Rav84]{Rav-loc}
Douglas~C. Ravenel, \emph{Localization with respect to certain periodic
  homology theories}, Amer. J. Math. \textbf{106} (1984), no.~2, 351--414.

\bibitem[Rez10]{Rezk-theta}
Charles Rezk, \emph{A {C}artesian presentation of weak {$n$}-categories}, Geom.
  Topol. \textbf{14} (2010), no.~1, 521--571.

\bibitem[Ros98a]{Rosenberg-ncscheme}
Alexander~L. Rosenberg, \emph{Noncommutative schemes}, Compositio Math.
  \textbf{112} (1998), no.~1, 93--125.

\bibitem[Ros98b]{Rosenberg-SpecAbCat}
\bysame, \emph{The spectrum of abelian categories and reconstruction of
  schemes}, Rings, {H}opf algebras, and {B}rauer groups ({A}ntwerp/{B}russels,
  1996), Lecture Notes in Pure and Appl. Math., vol. 197, Dekker, New York,
  1998, pp.~257--274.

\bibitem[RV16]{RV-adjns}
Emily Riehl and Dominic Verity, \emph{Homotopy coherent adjunctions and the
  formal theory of monads}, Adv. Math. \textbf{286} (2016), 802--888.

\bibitem[Sha]{Shah-recstrat}
Jay Shah, \emph{Recollements and stratification}, available at
  \texttt{arXiv:2110.06567}, v1.

\bibitem[Ste18]{Stevenson-tour}
Greg Stevenson, \emph{A tour of support theory for triangulated categories
  through tensor triangular geometry}, Building bridges between algebra and
  topology, Adv. Courses Math. CRM Barcelona, Birkh\"{a}user/Springer, Cham,
  2018, pp.~63--101.

\bibitem[Sul05]{Sullivan-MIT}
Dennis~P. Sullivan, \emph{Geometric topology: localization, periodicity and
  {G}alois symmetry}, $K$-Monographs in Mathematics, vol.~8, Springer,
  Dordrecht, 2005, The 1970 MIT notes, Edited and with a preface by Andrew
  Ranicki.

\bibitem[Tho97]{Thomason-classification}
R.~W. Thomason, \emph{The classification of triangulated subcategories},
  Compositio Math. \textbf{105} (1997), no.~1, 1--27. \MR{1436741}

\end{thebibliography}

\end{document}